%% file: principles.tex
\documentclass{book}

\usepackage{amssymb,amsmath,latexsym}
\usepackage{amsthm}
  \makeatletter
  \def\th@plain{
    \thm@notefont{}
    \itshape
  }
  \def\th@definition{
    \thm@notefont{}
    \normalfont
  }
  \makeatother
\usepackage{mathrsfs}
  \DeclareMathAlphabet{\mathsfit}{T1}{\sfdefault}{\mddefault}{\sldefault}
  \SetMathAlphabet{\mathsfit}{bold}{T1}{\sfdefault}{\bfdefault}{\sldefault}
\usepackage{bm}
\usepackage{xcolor}
\usepackage[shortlabels]{enumitem}
\usepackage{mathtools}
\usepackage{csquotes}
\usepackage{shorttoc}
  
\usepackage{array}
\usepackage{titling}
\newcommand{\subtitle}[1]{%
  \posttitle{%
    \par\end{center}
    \begin{center}\large#1\end{center}
    \vskip0.5em}%
}
\usepackage{imakeidx}
\usepackage[totoc]{idxlayout}
  \makeindex[title=Index of Terms]
  \makeindex[
    name=symbols,
    title=Index of Symbols,
    options= -s principles.ist
  ]
\usepackage[pdftex, hidelinks]{hyperref}

\newcounter{indexcount}
\newcommand{\threedigit}[1]{\ifnum #1<100 0\fi\ifnum #1<10 0\fi#1}

\newcommand{\symindex}[1]{%
  \ifcsname\detokenize{SYM@@#1}\endcsname
    \index[symbols]{\csname\detokenize{SYM@@#1}\endcsname @#1}%
  \else
    \stepcounter{indexcount}%
    \expandafter\xdef\csname SYM@@\detokenize{#1}\endcsname{%
      \expandafter\threedigit\expandafter{\romannumeral-`Q\theindexcount}%
    }%
    \index[symbols]{\threedigit{\theindexcount}@\unexpanded{\unexpanded{#1}}}%
  \fi
}

\hypersetup{
  pdfinfo = {
    Title = {The Principles of Probability},
    Subtitle = {%
      From Formal Logic to Measure Theory to the Principle of Indifference%
    },
    Author = {Jason Swanson}
  }
}

\newtheorem{thm}{Theorem}[section]
\newtheorem{cor}[thm]{Corollary}
\newtheorem{prop}[thm]{Proposition}
\newtheorem{lemma}[thm]{Lemma}
\theoremstyle{definition}
\newtheorem{defn}[thm]{Definition}
\newtheorem{rmk}[thm]{Remark}
\newtheorem{expl}[thm]{Example}

\numberwithin{equation}{section}

\providecommand{\ang}[1]{\left\langle{#1}\right\rangle}
\providecommand{\flr}[1]{\left\lfloor{#1}\right\rfloor}

\DeclareMathOperator{\AF}{\mathrm{AF}}
\DeclareMathOperator{\ante}{ante}
\DeclareMathOperator{\opbfE}{{\bf E}}
\DeclareMathOperator{\bbP}{\mathbb{P}}
\DeclareMathOperator{\bbQ}{\mathbb{Q}}
\DeclareMathOperator{\bnd}{bnd}
\DeclareMathOperator{\bTh}{\mathbf{Th}}
\DeclareMathOperator{\card}{card}
\DeclareMathOperator{\con}{con}
\DeclareMathOperator{\dom}{dom}
\DeclareMathOperator{\free}{free}
\DeclareMathOperator{\fP}{\mathfrak{P}}
\DeclareMathOperator{\len}{len}
\DeclareMathOperator{\olbbP}{\overline{\mathbb{P}}}
\DeclareMathOperator{\olbbQ}{\overline{\mathbb{Q}}}
\DeclareMathOperator{\olnu}{\overline{\nu}}
\DeclareMathOperator{\opmu}{\mu}
\DeclareMathOperator{\opnu}{\nu}

\DeclareMathOperator{\oprho}{\rho}
\DeclareMathOperator{\rk}{rk}
\DeclareMathOperator{\Sf}{Sf}
\DeclareMathOperator{\sym}{sym}
\DeclareMathOperator{\Th}{\mathit{Th}}
\DeclareMathOperator{\var}{var}

\def\bdot{\mathbin{\bm{\cdot}}}
\def\beq{\mathrel{\bm{=}}}
\def\bin{\mathrel{\bm{\in}}}
\def\bsub{\mathrel{\bm{\subseteq}}}
\def\cao{\mathrel{\hookrightarrow}}
\def\nbeq{\mathrel{\not\bm{=}}}
\def\nbin{\mathrel{\not\bm{\in}}}
\def\nrtri{\mathrel{\ntriangleright}}
\def\ntDash{\mathrel{|\hspace{-0.48em}\not\equiv}}

\def\rtri{\mathrel{\vartriangleright}}
\def\tri{\mathbin{\triangle}}
\def\tDash{\mathrel{|\hspace{-0.48em}\equiv}}
\def\wdash{\mathrel{|\hspace{-0.48em}\sim}}

\def\dhl{\!\downharpoonleft}
\def\emp{\emptyset}
\def\ge{\geqslant}
\def\le{\leqslant}
\def\pa{\partial}
\def\tot{\leftrightarrow}

\def\ol{\overline}

\def\ts{\textstyle}
\def\ul{\underline}
\def\wh{\widehat}
\def\wt{\widetilde}

\def\al{\alpha}
\def\be{\beta}
\def\ga{\gamma}
\def\Ga{\Gamma}
\def\de{\delta}
\def\De{\Delta}
\def\ep{\varepsilon}
\def\ze{\zeta}
\def\th{\theta}
\def\ka{\kappa}
\def\la{\lambda}
\def\La{\Lambda}
\def\si{\sigma}
\def\Si{\Sigma}
\def\ph{\varphi}
\def\om{\omega}
\def\Om{\Omega}

\def\bfi{{\bf i}}
\def\bfL{{\bf L}}
\def\bfP{{\bf P}}
\def\bfr{{\bf r}}
\def\bfs{{\bf s}}
\def\bft{{\bf t}}

\def\B{\bm{B}}
\def\bemp{\bm{\emp}}

\def\bnu{\bm{\nu}}
\def\bom{\bm{\om}}
\def\bv{\bm{v}}
\def\bw{\bm{w}}
\def\bx{\bm{x}}

\def\bD{\mathbb{D}}
\def\bL{\mathbb{L}}
\def\bN{\mathbb{N}}
\def\bQ{\mathbb{Q}}
\def\bR{\mathbb{R}}
\def\bS{\mathbb{S}}
\def\bT{\mathbb{T}}
\def\bZ{\mathbb{Z}}

\def\cA{\mathcal{A}}
\def\cB{\mathcal{B}}
\def\cC{\mathcal{C}}
\def\cD{\mathcal{D}}
\def\cE{\mathcal{E}}
\def\cF{\mathcal{F}}
\def\cG{\mathcal{G}}
\def\cL{\mathcal{L}}
\def\cN{\mathcal{N}}
\def\cR{\mathcal{R}}
\def\cS{\mathcal{S}}
\def\cT{\mathcal{T}}

\def\cX{\mathcal{X}}

\def\fI{\mathfrak{I}}

\def\ch{\mathrm{ch}}

\def\fin{\mathrm{fin}}
\def\img{\mathrm{Img}}
\def\ins{\mathrm{ins}}
\def\IS{\mathrm{IS}}
\def\lm{\mathrm{lim}}
\def\midp{\mathrm{mid}}
\def\param{\mathrm{par}}
\def\rd{\mathrm{rd}}
\def\seg{\mathrm{seg}}
\def\tria{\mathrm{tri}}

\def\sD{\mathscr{D}}
\def\sP{\mathscr{P}}
\def\sQ{\mathscr{Q}}

\def\A{\mathsf{A}}
\def\AC{\mathsf{AC}}
\def\AE{\mathsf{AE}}
\def\AF{\mathsf{AF}}
\def\AI{\mathsf{AI}}
\def\AP{\mathsf{AP}}
\def\AR{\mathsf{AR}}
\def\AS{\mathsf{AS}}
\def\AU{\mathsf{AU}}
\def\E{\mathsf{E}}
\def\ISPA{\mathsf{IS}}
\def\P{\mathsf{P}}
\def\PA{\mathsf{PA}}
\def\s{\mathsf{s}}
\def\t{\mathsf{t}}

\def\ZFC{\mathsf{ZFC}}

\def\Taut{\mathsfit{Taut}}
\def\Var{\mathsfit{Var}}

\def\S{\mathtt{S}}

\begin{document}

\frontmatter

\title{The Principles of Probability}
\subtitle{%
  From Formal Logic to Measure Theory\\
  to the Principle of Indifference%
}
\author{Jason Swanson}
\date{August 2, 2024}
\maketitle

\pdfbookmark[1]{Contents}{toc}
\shorttoc{Contents}{1}

\cleardoublepage
\pdfbookmark[1]{Detailed Contents}{toc (detailed)}
\tableofcontents

\include{preface}

\mainmatter

\include{introduction}

\include{background}

\include{prop-calc}

\include{prop-semantics}

\include{pred-logic}

\include{real-ind-th}

\include{princ-indiff}

\bibliographystyle{plain}
\bibliography{principles}
\addcontentsline{toc}{chapter}{Bibliography}

\printindex
\idxlayout{columns=3}
\idxlayout{justific=standard}
\printindex[symbols]

\end{document}

%% file: preface.tex

\chapter{Preface}

In classical logic, we formalize deductive arguments. The conclusion of a
deductive argument is known with certainty, provided its premises are true.
Inductive arguments, on the other hand, are those whose conclusions are known only with some degree of plausibility. An argument in a courtroom, for example, is inductive. Its conclusion, at best, is only known ``beyond a reasonable doubt.''

We present herein a formal system of inductive logic. The system contains deductive logic as a special case. It also uses a language that allows for countable conjunctions. When we restrict our attention to deductive arguments using only finite conjunctions, we obtain ordinary first-order logic. That is, first-order logic is embedded within this system. In particular, the system is capable of expressing the usual set theory of Zermelo and Fraenkel, and as such, can express any statement of modern mathematics.

This system of inductive logic gives rise to a probability calculus that is in complete agreement with modern, mathematical probability theory. In particular, the inductive statements in the formal language of this system can be interpreted in probability spaces. Moreover, any probability space, together with any collection of random variables, can be mapped in a natural way to such an interpretive model.

Inductive logic, however, is more expressive than ordinary probability theory. There are probabilistic ideas that are expressible in this system which cannot be formulated using only probability spaces and random variables. An example of such an idea is the principle of indifference, a heuristic notion originating with Laplace. Roughly speaking, it says that if we are ``equally ignorant'' about two possibilities, then we should assign them the same probability. The principle of indifference has no rigorous formulation in ordinary probability theory. It exists only as a heuristic. Moreover, its use has a history of being problematic and prone to apparent paradoxes. In our system of inductive logic, however, we provide a rigorous formulation of this principle, and illustrate its use through a number of typical examples.

The material herein makes use of (mostly) basic facts from mathematical logic
and measure theory. We assume the reader is already familiar with the
fundamentals of measure-theoretic probability theory. On the other hand, an
effort has been made to accommodate readers with no familiarity in logic.
The logical notions that we use are presented in a way that is
mostly self-contained. Where this is not possible, explicit references to the literature are provided.

%% file: introduction.tex

\chapter{Introduction}

\begin{quote}
  Strictly speaking, all our knowledge outside mathematics and demonstrative
  logic (which is, in fact, a branch of mathematics) consists of conjectures.
  \vspace{-.1in}
  \begin{flushright}
    ---George Polya, 1954 \cite{Polya1954}
  \end{flushright}
\end{quote}

\section{Deductive vs.~inductive reasoning}

Newton's laws of motion were conjectures that Einstein showed us were false. But
Einstein's theory of relativity is also a conjecture that may one day be
overturned. Physical laws, in general, are all conjecture. Each experimental
confirmation makes them more plausible, but they can never be established with
complete certainty.

In the courtroom, the prosecuting attorney must prove the guilt of the
defendant, not with complete certainty, but only beyond a reasonable doubt. If
the law required the prosecutor to achieve complete certainty, then everyone
would be acquitted, because this would be impossible. Strictly speaking, then,
we imprison people on the basis of conjecture.

Likewise, the conclusions of the historian, the economist, the chemist, and the
medical researcher are all conjecture. None of them can establish their results
with complete certainty.

And yet, many of these results have been shown to be so plausible that no one
seriously doubts them. To borrow an example from Laplace \cite{LaPlace1814}, the
sun will rise tomorrow. I cannot say this with complete certainty, but the
degree of plausibility of this fact is so high, that my human mind cannot even
perceive the sliver of doubt that is there.

The process by which these conjectures obtain varying degrees of plausibility is
not the logic of the mathematician. A mathematical proof either establishes
complete certainty, or it fails to say anything at all. The logic of mathematics
is \emph{deductive reasoning}. The logic of everything else is \emph{inductive
reasoning}, or as Polya calls it in \cite{Polya1954}, ``plausible reasoning.''

Deductive reasoning is governed by rules. Over the course of human history, we
have uncovered and formalized these rules, and today we have complete systems of
deductive logic that codify this form of argumentation. The study of deductive
logic, or mathematical logic, is presently a mature and sophisticated
subdiscipline of mathematics that has had profound impacts in areas ranging
from computer science to philosophy to mathematics itself.

Inductive reasoning also has rules. For example, if Hypotheses $A$ implies
Hypothesis $B$, and Hypothesis $B$ is found to be true, then the plausibility of
Hypothesis $A$ increases. This rule is the basis for empirical science. If the
theory of relativity predicts something about Mercury, and the prediction is
confirmed by experiment, then the theory of relativity is made more plausible by
this discovery. In \cite{Polya1954a}, Polya calls this the ``fundamental
inductive pattern.''

But unlike its deductive counterpart, inductive reasoning has not been
formalized in any universally accepted way. Since the time of Laplace,
probability theory has seemed like the most promising candidate to formalize
inductive reasoning. For example, in probability theory it is a provable fact
that if $A$ implies $B$, and if $P(A)$ and $P(B)$ are neither $0$ nor $1$, then
$P (A \mid B) > P (A)$. This is a formalization of Polya's fundamental inductive
pattern.

Attempts to formalize inductive reasoning with probability can be traced at
least back to Boole in 1854 \cite{Boole1854}. But the recognition that we need
some kind of probabilistic logic goes all the way back to Leibniz in the 17th
century. Since then, mathematicians, philosophers, and physicists have all
contributed to this endeavor. See, for instance,
\cite{Keynes1921,Wittgenstein1922,Reichenbach1949,Carnap1950,Krauss1966,Nilsson1986,Jaynes2003}.
For a survey of the history of these efforts, see \cite{Hailperin1996}.

In the meantime, over the last 90 years, modern probability---by which we mean
measure-theoretic probability theory---has grown into a powerful and hugely
successful discipline. It started with Kolmogorov in 1933 \cite{Kolmogorov1956}
and, today, enjoys phenomenal success in all its areas of applications. Advances
have been made in physics, finance, engineering, meteorology,
telecommunications, biology, astronomy, artificial intelligence, and more. Time
has shown us that if we want to do quantified inductive reasoning, there is no
better tool than modern probability.

Perhaps, then, we should turn the effort on its head. As presented above, we
have been regarding inductive reasoning as the fundamental concept, and
probability as a tool by which to formalize it. Instead, we might consider
modern probability to be the fundamental concept, and seek out the inductive
logic that it represents. To clarify what this might mean, let us look to its
analogue in deductive logic, or more specifically, first-order logic.

\section{The two sides of logic}

First-order logic is the usual logic of sentences that involve quantifiers and
predicates. We may approach first-order logic in one of two ways. We may
consider it through a syntactic calculus, or we may consider it through
semantics and meaning. When viewed as a calculus, the central idea is the
``proof.'' A proof is a sequence of sentences whose construction follows a given
set of rules, and which terminates in the sentence that is being proved. A
sentence $\ph$ is derivable from $X$, written $X \vdash \ph$, if there exists
a proof of $\ph$ from $X$.

When viewed via semantics, the central idea is the ``structure.'' A structure is
a set with distinguished constants, functions, and relations. A group is a
structure with an identity and a group operation. A tree is a structure with a
root and an edge relation. Even a committee of senators with a chair and two
subcommittees is a structure. Formal sentences can be interpreted in a
structure, and once interpreted, the sentence is either true or false. We say
that $\ph$ is a consequence of $X$, written $X \vDash \ph$ if $\ph$ is true in
every structure where $X$ is true.

A priori, these two approaches to first-order logic have nothing to do with one
another. And yet, thanks to G\"odel's completeness theorem, we know that they
exactly coincide. That is, $X \vdash \ph$ if and only if $X \vDash \ph$. In
other words, first-order logic \emph{is} the logic of structures. By analogy, if
we want an inductive logic that \emph{is} the logic of probability, then we
should build a complete inductive logic that has modern, measure-theoretic
probability as its semantics.

That is exactly what we will do in this manuscript. We will construct inductive
logic so that, like first-order logic, it has a calculus and a semantics. The
calculus will be based on nine rules of inductive inference. It will allow us to
derive probabilities without any sample spaces or measure theory. All that
matters in the calculus are the logical relations between the sentences, and the
rules of inductive inference. On the semantic side, statements are interpreted
in what we call an ``inductive model,'' which is a probability measure on a set
of structures. We will establish completeness, showing that the calculus and the
semantics coincide. And we will show that the whole of measure-theoretic
probability theory is properly embedded in the semantic side of inductive logic.
That is, any probability space, together with a set of its random variables, can
be mapped to an inductive model in a way that gives each outcome, event, and
random variable a logical interpretation.

\section{The nature of probability}

What, then, is probability? Or more precisely, what body of ideas should the
word ``probability'' refer to? As a mathematician, it is tempting to say that
probability is simply measure-theoretic probability, the branch of mathematics
built on Kolmogorov's formalism. But to paraphrase Terence Tao \cite{Tao2015},
probability spaces are used to model probabilistic concepts. They are not the
concepts themselves. As such, and in light of the work done here, we might say
that probability is the logic of inductive reasoning. It is an abstract mode of
logical reasoning that is reflected in two parallel systems: a calculus of
inductive inference, and a semantic system of interpretation. It is within this
semantic system that measure-theoretic probability resides.

Another way to assess the nature of probability is to look at the people who
study it and ask what they actually do. In other words, what is a probabilist?
The traditional view is simple. Probability is a branch of mathematics, and a
probabilist is a mathematician who specializes in it. A probabilist is just one
of many mathematical specialists, each of whom is classified according to the
kinds of structures they study. A group theorist, for example, studies
structures satisfying the axioms of group theory, that is, groups. A
probabilist, therefore, studies structures satisfying Kolmogorov's axioms, that
is, probability spaces.

This view, however, misses something important. A probabilist typically
specializes in a particular class of random variables and stochastic processes.
They do not simply study probability spaces. A person who studies pure
probability spaces, without any random variables, would be better described as a
measure theorist. A probabilist, on the other hand, studies what we might call
``modern probability models,'' which are probability spaces equipped with a
collection of random variables. As noted above, every modern probability model
is an inductive model. Hence, probabilists study inductive models, and any given
probabilist will study a particular class of inductive models.

To clarify the situation, let us note that probability spaces are to sets as
inductive models are to structures. That is, we have four kinds of object (sets,
structures, probability spaces, and inductive models) that all stand in relation
to one another. The simplest object is the set. We may extend the notion of the
set in two directions. On the one hand, if we add a probability measure to it,
we obtain a probability space. On the other hand, if we add constants,
functions, and relations, we obtain a structure, which we use to interpret
deductive logic. An inductive model, which we use to interpret inductive logic,
can be obtained from either a structure or a probability space. By definition,
if we take a set of structures and add a probability measure, we have an
inductive model. Or we can start with a probability space and add a particular
collection of random variables. In doing so, we obtain a modern probability
model, which is embedded in the collection of inductive models.

These four kinds of objects are studied by different kinds of mathematicians.
Sets are studied by set theorists. Probability spaces, without any random
variables, are studied by measure theorists. Structures come in many varieties.
For example, graphs are studied by graph theorists. Likewise, inductive models
come in many varieties. An example is random graphs. A person who studies random
graphs would be called a probabilist.

But a probabilist who studies random graphs is as much a graph theorist as they
are a probabilist. Just as mathematicians are categorized according to the kinds
of structures they study, probabilists are categorized according to the kinds of
inductive models they study. There are probabilists who study stochastic PDEs,
random matrices, stochastic control theory, and so on. And like their
deterministic counterparts, any given probabilist will typically only specialize
in one such area. The word ``mathematician'' is an umbrella term that includes
many different areas. Likewise, the word ``probabilist'' is also an umbrella
term that includes all these same areas, but seen through the lens of
probability and inductive reasoning. The picture that emerges from this view is
that probability is not just a branch of mathematics. It is a different logical
paradigm with which to study other branches of mathematics.

\section{Potential areas of application}

As an extension of formal deductive logic, inductive logic has many
potential areas of application. Computer science, for instance, is rooted in
mathematical logic. Hence, any probabilistic extension of computer science is, in
some way, connected to a probabilistic extension of logic. It follows that
inductive logic could be relevant to any such areas. Examples might include
quantum computing and artificial intelligence.

Inductive logic could also be applicable to philosophy. It is connected to the
philosophy of science through Polya's fundamental inductive pattern. It is also
highly relevant to the philosophical interpretations of probability, and through
these, to epistemology. For instance, we can use inductive logic to formalize
the principle of indifference. This principle is the heuristic notion that we
ought to assign equal probabilities to sentences about which we are equally
ignorant. This principle is intuitively self-evident, but historically
problematic. It leads to apparent paradoxes and is without any mathematically
rigorous formulation. We will have more to say about the principle of
indifference later.

Formal deductive logic can also be used to analyze philosophical arguments. As
such, inductive logic can be used to analyze philosophical arguments that are
probabilistic in nature. Examples include the doomsday argument, the simulation
hypothesis, responses to the so-called Sleeping Beauty problem, and arguments
surrounding superintelligence and the technological singularity.

Inductive logic can be applied to mathematics itself, as well as statistics. Its
relevance to probability and Bayesian statistics is obvious. Outside of
probability, it is relevant wherever probabilistic methods are used. For
example, in graph theory and combinatorics, probabilistic methods are often used
to establish existence theorems and asymptotic results. At the foundations of
mathematics, it offers us a new tool for working with undecidable sentences.
Given a set of axioms and an undecidable sentence which they can neither prove
nor disprove, deductive logic allow us to explore the theories obtained by
either including or excluding that sentence from our axioms. With inductive
logic, we may choose to postulate a probability for the undecidable sentence,
and explore the probabilistic statements that follow from this assumption. 

There are many potential areas of applications in physics. One obvious candidate
is quantum mechanics, particularly its interpretations. The interpretations of
quantum mechanics, such as the many-worlds and the Bohmian interpretations,
cannot be decided upon through experiment, since all of their predictions are
based on the common mathematical framework of quantum mechanics. The problem of
deciding which one is correct or most useful is a philosophical problem.
Inductive logic gives us a framework for axiomatizing these interpretations, and
thereby more thoroughly analyzing their structure and consequences.

Besides quantum mechanics, inductive logic might also be applicable to
statistical mechanics. This is particularly true since statistical mechanics is
rooted in the principle of indifference. Its fundamental postulate is that, a
priori, the microstates of an isolated system occur with equal probability.
Inductive logic, therefore, through the principle of indifference, touches upon
the foundations of statistical mechanics.

The study of noisy dynamical systems is an area of physics where modern
probability has found
great success. In applications, the use of stochastic ordinary and partial
differential equations is widespread. On the other hand, we can axiomatize classical (deterministic) mechanics in
an extension of first-order logic that allows countable conjunctions. This
extended language is exactly the one in which we will build inductive logic. We
can therefore add probability and uncertainty directly into such an
axiomatization. Does doing so lead us to the modern theory of stochastic
dynamical systems? If not, how is it related to that theory? Determining this
could provide insight into both inductive logic and the modern usage of
stochastic differential equations.

\section{The principle of indifference}\label{S:intro-PoI}

The principle of indifference is the heuristic idea that if we are equally
ignorant about two propositions, then we ought to assign them the same
probability. This idea originated with Laplace, and is at the heart of what is
now called the classical interpretation of probability.
\begin{quote}
  The theory of chance consists in reducing all the events of the same kind to a
  certain number of cases equally possible, that is to say, to such as we may be
  equally undecided about in regard to their existence, and in determining the
  number of cases favorable to the event whose probability is sought. The ratio
  of this number to that of all the cases possible is the measure of this
  probability.
  \vspace{-.1in}
  \begin{flushright}
    ---Pierre-Simon Laplace, 1814 \cite{LaPlace1814}
  \end{flushright}
\end{quote}
One of the most famous descriptions of the principle is due to Keynes.
\begin{quote}
  The Principle of Indifference asserts that if there is no \emph{known} reason
  for predicating of our subject one rather than another of several
  alternatives, then relatively to such knowledge the assertions of each of
  these alternatives have an \emph{equal} probability. Thus \emph{equal}
  probabilities must be assigned to each of several arguments, if there is an
  absence of positive ground for assigning \emph{unequal} ones.\hfil\\
  \hspace*{10pt} This rule, as it stands, may lead to paradoxical and even
  contradictory conclusions.
  \vspace{-.1in}
  \begin{flushright}
    ---John Maynard Keynes, 1921 \cite{Keynes1921}
  \end{flushright}
\end{quote}
It is difficult to overstate the importance of the principle of indifference.
Not only is it at the heart of major areas of science, such as statistical
mechanics. It is also central to our understanding of what probability is.
Philosophically, it forms the basis of the classical interpretation of
probability. But beyond philosophy, it is the everyday intuition of the common
person as to why a balanced die is fair or why it is important to thoroughly
shuffle a deck of cards. And yet, as Keynes rightly points out, it has a
problematic history. Even rather elementary applications of this principle can
quickly lead to nonsensical results and apparent paradoxes. This is due, in no
small part, to the fact that, for centuries, the principle of indifference has
eluded attempts to make it rigorous. Without rigor, there are no precise
conditions that can tell us whether our attempts to use it are legitimate. A
notable formulation of the principle is given by Edwin T. Jaynes's
\cite{Jaynes2003}, in a book posthumously published in 2003. He used this
formulation as the basis of his maximum entropy principle, which plays a key
role in statistical mechanics. But even Jaynes's formulation is non-rigorous.
Moreover, there is no formulation of this principle in modern, measure-theoretic
probability theory.

Within inductive logic, however, we will be able to formulate the principle of
indifference. We will show that our formulation is a faithful representation of
the principle. It is not simply an ad hoc condition to which we affix the name.
Being mathematically rigorous, our formulation is as free from paradoxes as any
proven mathematical theorem.

Moreover, we will see that the principle of indifference cannot be formulated
using only the axioms of Kolmogorov. Its formulation requires the structure of
inductive logic, both its syntactic structure and the semantic structures
embedded in its models. As such, it exemplifies the fact that inductive logic is
strictly broader than any theory of probability that is based on measure theory
alone.

\section{A philosophical aside}

This book is, without question, a work of mathematics. It consists primarily of
definitions, theorems, and proofs, with occasional intuitive prose to tie it
together. On the other hand, it is hard not to see it as a work of philosophy,
having something to say about the interpretation of probability.

According to \cite{Hajek2019}, interpretations of probability generally address
the questions:
\begin{enumerate}[(1)]
  \item What kinds of things, metaphysically, are probabilities?
  \item What makes probability statements true or false?
\end{enumerate}
To use inductive logic to answer these questions, we must first look to the
definition of an inductive statement. An inductive statement is a triple, $(X,
\ph, p)$, where $X$ is a set of sentences (in a formal language) called the
``antecedent,'' $\ph$ is a sentence that we call the ``consequent,'' and $p$ is
a real number satisfying $0 \le p \le 1$ which we call the ``probability.''
Intuitively, we can think of $ (X, \ph, p)$ as asserting that $X$ partially
entails $\ph$, and that $p$ is the degree of this partial entailment. To answer
the first question, then, a probability is a relationship between $X$ and $\ph$.
We might interpret this relationship as being logical, evidential, or purely
subjective. Any such interpretation has no bearing on the inductive logical
properties of $(X, \ph, p)$.

Inductive logic is mostly unconcerned with the second question. Probabilities
express relative likelihoods, given a set of sentences. We take it for
granted that the sense of these likelihoods is understood. Our primary concern
is the logical relationships between inductive statements. That is, how can we
reason from hypotheses, which are themselves inductive statements, to
an inductive conclusion.

It seems, then, that inductive logic hardly qualifies as an interpretation of
probability. It does, however, make assumptions that rule out certain
interpretations. It assumes that probabilities are relationships between
sentences and that all probabilities are conditional. Inductive logic,
therefore, is incompatible with any physical interpretations of probability,
such as those based on frequencies or propensities. Beyond that, though, it
appears to leave room for a range of evidential interpretations.

\section{Constructing inductive logic}\label{S:construct-sketch}

In this section, we give a big-picture overview of how inductive logic is
constructed. It is assumed that the reader is already familiar with the basics
of measure-theoretic probability theory.

The set of sentences and formulas that we consider is one that allows for
countable conjunctions and disjunctions. This set is usually denoted in the
literature by $\cL_{\bom_1, \bom}$, though we denote it simply by $\cL$. In the
language $\cL$, there is a well-understood deductive calculus (see, for
instance, \cite{Karp1964, Keisler1971}). That is, there is a well-established
derivability relation $\vdash$, where $X \vdash \ph$ indicates that $\ph \in
\cL$ can be derived from $X \subseteq \cL$. Our first task is to extend $\vdash$
to inductive statements. We do this by defining a set of rules for inductive
inference, and writing $P \vdash (X, \ph, p)$ to mean that we can use these
rules of inference to derive $(X, \ph, p)$ from the set of inductive statements
$P$. We also define an inductive theory, which is a set that is closed under
inductive inference. Finally, we adopt familiar shorthand, writing $P(\ph \mid
X) = p$ to mean that $ (X, \ph, p) \in P$. We say that $P(\ph \mid X)$ exists if
$(X, \ph, p) \in P$ for some $p$.

Let us refer to the left-hand side of the turnstile symbol, $\vdash$, as the
premises of the derivation $P \vdash (X, \ph, p)$. From what we have described
so far, our premises can only include statements of the form $P(\ph \mid X) =
p$. We may wish, however, to include premises of the form $P(\ph \mid X) > 0$
or $P(\ph \wedge \psi \mid X) = P(\ph \mid X)P(\psi \mid X)$. To this end, we
generalize inductive derivability by defining inductive conditions, typically
denoted by calligraphic letters such as $\cC$. Inductive conditions formalize
generic assumptions we might make about an inductive theory. After defining
inductive conditions, we extend inductive derivability from $P \vdash (X, \ph,
p)$ to $\cC \vdash (X, \ph, p)$.

We then turn our attention to the semantics of inductive logic. More
specifically, we define a relation, $\vDash$, called the consequence relation.
The derivability relation, $\vdash$, is concerned only with the syntax of
sentences, formulas, and inductive statements. The consequence relation, on the
other hand, is concerned with their interpretations.

To give interpretations to sentences, we must introduce models. For us, a model,
or an inductive model, is a probability space, $(\Om, \Si, \bbP)$, where $\Om$
is a set of structures. A structure is a set, together with some distinguished
constants, functions, and relations. For example, $(\bN, 1, +, <)$ is a
structure. Structures are used to interpret the sentences in $\cL$. We write
$\om \tDash \ph$ to mean that, in the structure $\om$, the interpretation of
$\ph$ is a true statement. We can think of an inductive model as a weighted
collection of structures, where the weights represent relative likelihoods.

Suppose $\sP = (\Om, \Si, \bbP)$ is a model and $\ph$ is a sentence. Then we
define the set $\ph_\Om = \{\om \in \Om \mid \om \tDash \ph\}$. We say that
$\sP$ satisfies $\ph$, written $\sP \vDash \ph$, if $\olbbP \ph_\Om = 1$, where
$\olbbP$ is the measure-theoretic completion of $\bbP$. For sets of sentences,
we take $\sP \vDash X$ to mean that $\sP \vDash \ph$ for all $\ph \in X$. We say
that $X$ and $X'$ are semantically equivalent if, for every model $\sP$, we have
$\sP \vDash X$ if and only if $\sP \vDash X'$.

More generally, we write $\sP \vDash (X, \ph, p)$ to mean there exists a set of
sentences $Y$ and a sentence $\psi$ such that $\sP \vDash Y$, the sets $Y \cup
\{\psi\}$ and $X$ are semantically equivalent, and
\[
  \frac{\olbbP \ph_\Om \cap \psi_\Om}{\olbbP \psi_\Om} = p.
\]
For sets of inductive statements, we take $\sP \vDash P$ to mean $\sP \vDash (X,
\ph, p)$ for all $(X, \ph, p) \in P$. A set $P$ is called satisfiable if $\sP
\vDash P$ for some model $\sP$.

Having defined satisfiability, we turn to the consequence relation. We define $X
\vDash \ph$ to mean that $\sP \vDash \ph$ whenever $\sP \vDash X$, and we define
$P \vDash (X, \ph, p)$ to mean that $P$ and $(X, \ph, p)$ satisfy certain
connectivity requirements and that $\sP \vDash (X, \ph, p)$ whenever $\sP \vDash
P$. We then proceed to prove soundness and completeness, meaning that the two
relations, $\vdash$ and $\vDash$, are identical. In other words, we prove that
$X \vdash \ph$ if and only if $X \vDash \ph$, and also that $P \vdash (X, \ph,
p)$ if and only if $P \vDash (X, \ph, p)$.

\section{Inductive logic is natural}

In mathematics, a definition is neither correct nor incorrect. It simply is.
Nevertheless, we understand that there is a metamathematical sense in which a
definition can be ``right'' or ``wrong.'' Does it capture the intuitive idea it
alleges to capture? Does it ``fit'' well with existing notions? Does it provide
a sense of mathematical ``unity'' or ``elegance?''

In constructing inductive logic, we have made several choices. We chose to work
in the infinitary language, $\cL_{\bom_1, \bom}$. We chose the rules of
inductive inference that characterize the derivability relation, $\vdash$. And
we chose a particular semantic interpretation in order to arrive at the
consequence relation, $\vDash$. Were these the ``right'' choices? Are they
natural? Or are they ad hoc choices whose only purpose is to force the
conclusions we are aiming for? These are metamathematical questions, and so, for
the most part, their answers are left to the judgment of the reader. But we
certainly believe they are natural, and present here two comments which are
related to this question.

The first comment concerns the rules of inductive inference used in defining
$\vdash$. There are rules that govern probabilities $0$ and $1$, and relate them
to deductive inference. There is a continuity rule, needed only in infinite
systems. Otherwise, the two main operative rules are the familiar ones from
elementary probability: the addition rule and the multiplication rule. We make
no effort in this work to justify their adoption. Inductive logic may be seen as
simply investigating their consequences, or as implying that they are
self-evident. That said, in the paragraph after next, we will make one comment
in their defense. Beyond that, the interested reader can consult \cite{Cox1946}
for an attempt to justify these elementary principles.

The second comment concerns $\vDash$. We essentially have four interconnected
relations: deductive $\vdash$, inductive $\vdash$, deductive $\vDash$, and
inductive $\vDash$. For deductive $\vDash$, we prove both completeness ($X
\vDash \ph$ implies $X \vdash \ph$) and $\si$-compactness ($X \vDash \ph$
implies $X_0 \vDash \ph$ for some countable $X_0 \subseteq X$). As mentioned
earlier, deductive $\vdash$ is the usual relation used in $\cL_{\bom_1,
\bom}$. Deductive $\vDash$, however, is not. The usual consequence relation in
$\cL_{\bom_1, \bom}$ is the more straightforward $X \tDash \ph$, meaning $\om
\tDash X$ implies $\om \tDash \ph$. It is well-known that $\tDash$ is neither
complete, nor $\si$-compact. For this reason, one might regard $\cL_{\bom_1,
\bom}$ as being deficient in some important ways. Our primary purpose for
introducing $\vDash$ is to model inductive reasoning. As an unintended but
welcome consequence, we find that $\vDash$ corrects these deficiencies. In this
sense, then, we might view $\vDash$ as the ``right'' semantic notion for $\cL_
{\bom_1, \bom}$, and also see $\cL_{\bom_1, \bom}$ as the ``right'' language in
which to build a  probabilistic logic.

Turning this argument on its head, if deductive $\vDash$ is the ``right''
semantic relation to pair with deductive $\vdash$, and inductive $\vDash$
follows from deductive $\vDash$, then our definition of inductive $\vdash$ is
forced on us, if we want $P \vDash (X, \ph, p)$ and $P \vdash (X, \ph, p)$ to be
equivalent. In other words, we are compelled to adopt the addition and
multiplication rules, as well as all the other rules in the definition of
$\vdash$.

\section{Outline of the book}\label{S:outline}

After presenting background material in Chapter \ref{Ch:background}, we proceed
to the construction of inductive logic. In Chapters \ref{Ch:prop-calc} and
\ref{Ch:prop-models}, we follow the sketch presented in Section
\ref{S:construct-sketch} to construct a trimmed-down version of inductive logic
in a propositional language without quantifiers or variables. The propositional
version is capable of representing any probability space, but it does not
explicitly represent any random variables.

In Chapter \ref{Ch:pred-logic}, we repeat the construction for the predicate
language $\cL$. Here, we are able to establish the connection between inductive
logic and random variables. In the formal language $\cL$, constants, functions,
and relations are represented by what are called extralogical symbols. If $\s$
is an extralogical symbol and $\om$ is a structure, then we can identify $\s$
with an actual constant, function, or relation in $\om$, which we denote by
$\s^\om$. The map $\om \mapsto \s^\om$ is the starting point for connecting
inductive logic to random variables.

Almost immediately, though, we face an obstacle. In probability theory, we are
accustomed to having two very distinct kinds of objects: random variables and
constants. In $\bbP \{X > 0\}$, for example, we can be quite certain that only
thing random is $X$. In the formal language, $\cL$, however, $X$, $>$, and $0$
are just symbols. When we interpret them in a structure, we are faced with $\bbP
\{X^\om >^\om 0^\om\}$. Hence, not only might $0$ be random, the inequality
relation itself could be random!

These considerations lead us to an idea we call the relativity of randomness. To
illustrate this, consider a simple system representing a coin flip. We might
have just three extralogical symbols, $h$, $t$, and $c$, where we think of $h$
and $t$ as the sides of the coin, and $c$ as the result. To be sure, there are
models that match our intuition. That is, there are models in which $h^\om$ and
$t^\om$ are fixed, and $c^\om$ is random. But there are also models in which
$h^\om$ and $t^\om$ are random, while $c^\om$ is fixed. There are even models in
which all three are random. In general, using model isomorphisms, we can force
certain sets of extralogical symbols to be nonrandom. There are limits though.
In the coin-flip example, there are no models in which all three symbols are
fixed.

We show that in any inductive theory rich enough to contain the natural numbers,
we can use model isomorphisms to force the natural numbers to be nonrandom. We
call this the natural frame of reference.

Chapter \ref{Ch:real-ind-ths} is concerned with constructing inductive theories
that make statements about real numbers. We do this by constructing the real
numbers, in the usual way, in axiomatic set theory. We are then able to make
formal inductive statements about not only real numbers, but about any
mathematical objects whatsoever. In this context, we illustrate how all the
usual results of probability theory can be formulated in inductive logic. This
includes the law of large numbers, the central limit theorem, and conditional
expectation.

The principle of indifference is the topic of Chapter \ref{Ch:PoI}. We formulate
it as an inductive condition. To describe it, let $L$ be the set of extralogical
symbols in our language and consider a bijection $\pi: L \to L$ that preserves
the type and arity of the symbols. For instance, if $r$ is a binary relation
symbol, then so is $r^\pi$. Given a sentence $\ph$, let $\ph^\pi$ be the
sentence obtained from $\ph$ by replacing every instance of $\s$ with $\s^\pi$.
Similarly, let $X^\pi = \{\ph^\pi \mid \ph \in X\}$. We say that $X$ is
invariant under $\pi$ if $X$ and $X^\pi$ are logically equivalent.

In the deductive calculus, if we take a proof of $\ph$ from $X$ and transform
each step of that proof by $\pi$, then we obtain a proof of $\ph^\pi$ from
$X^\pi$. In other words, $X \vdash \ph$ if and only if $X^\pi \vdash \ph^\pi$.
The principle of indifference is the natural extension of this to the inductive
calculus.

Let $P$ be an inductive theory. Then $P$ satisfies the principle of indifference
if $P(\ph^\pi \mid X^\pi) = P(\ph \mid X)$. In particular, if $X$ is invariant
under $\pi$, then, given $X$, the sentences $\ph$ and $\ph^\pi$ should be
assigned the same probability.

After formulating the principle of indifference, we present several examples,
beginning with simple discrete examples and continuing to examples involving
intervals and circles. Our final example is an analysis of Bertrand's paradox, a
famous counterintuitive illustration of the principle of indifference.

%% file: background.tex

\chapter{Background}\label{Ch:background}

It is assumed that the reader is familiar with the basics of measure-theoretic
probability theory. In Section \ref{S:meas-spaces}, we introduce some basic
concepts from measure theory. The reader should already be familiar with almost
everything in that section. We introduce them only to establish notation, cite
some lesser known results, and establish new definitions that we will need
later.

Familiarity with mathematical logic would be helpful but is not required. The
concepts from logic that we use are all presented as we need them, in a mostly
self-contained way.

We will adopt the convention that $0$ is a natural number. However, we will use
the notation $\bN = \{1, 2, 3, \ldots\}$ and $\bN_0 = \{0, 1, 2, 3, \ldots\}$.
  \symindex{$\bN$}%
  \symindex{$\bN_0$}%
In other words, $\bN$ is the set of positive integers, and $\bN_0$ is the set of
natural numbers.

\section{Ordinal and cardinal numbers}

\subsection{Ordinal numbers}\label{S:ordinals}

For further details about ordinal and cardinal numbers, see any basic text on
set theory, such as \cite{Cenzer2020} or \cite{Devlin1993}.

An \emph{ordinal (number)}
  \index{ordinal}%
is a well-ordered set, $(\al, <)$, such that if $x \in \al$, then $x = \{y \in
\al: y < x\}$. If $\al$ is an ordinal and $x, y \in \al$, then $x < y$ if and
only if $x \in y$ if and only if $x \subset y$. (Here, we use $\subseteq$ for
subset and $\subset$ for proper subset.) Moreover, every $x \in \al$ is itself
an ordinal. Thus, an ordinal is a set of ordinals that is well-ordered by $\in$.

The collection of all ordinals is not a set. (If it were, then it would be an
ordinal that is an element of itself, which, as we note below, is impossible.)
It is, nonetheless, ``well-ordered'' in a certain sense. More specifically, if
$\al$ and $\be$ are ordinals, then $\al \in \be$ if and only if $\al \subset
\be$, so that $\in$ has the properties of a strict partial order. Also, for any
two ordinals, $\al$ and $\be$, exactly one of the following is true: $\al \in
\be$, $\al = \be$, or $\be \in \al$. If $\al$ and $\be$ are ordinals, we will
write $\al < \be$ for $\al \in \be$. Finally, every nonempty collection of
ordinals has a smallest element. This last fact can be made rigorous in
axiomatic set theory as a so-called ``theorem schema,'' but for our present
purposes, we do not need such a formalization.

From these facts, it follows that there is a smallest ordinal. The smallest
ordinal is $\emp$, since there is no ordinal with $\al \in \emp$. The collection
of nonempty ordinals has a smallest element, meaning there is a second smallest
ordinal. The second smallest ordinal is $\{\emp\}$, since there is no ordinal
with $\emp \subset \al \subset\{\emp\}$.

If $\al$ is an ordinal, let $s(\al)$ denote the ordinal $\al \cup \{\al\}$. We
call $s(\al)$ the successor of $\al$. Note that there is no $\be$ with $\al
\subset \be \subset s(\al)$. Also note that $s(\emp) = \{\emp\}$. From here, we
see that the third smallest ordinal is $s(\{\emp\}) = \{\emp, \{\emp\}\}$, the
fourth smallest ordinal is $\{\emp, \{\emp\}, \{\emp, \{\emp\}\}\}$, and so on.

For each $n \in \bN_0$, define the ordinal $n'$ by setting $0' = \emp$ and $(n +
1)' = s(n')$. Let $\bN_0' = \{n' \mid n \in \bN_0\}$. The map $n \mapsto n'$ is
a bijection from $\bN_0$ to $\bN_0'$ which preserves the natural ordering on
$\bN_0$. That is, $m < n$ if and only if $m' < n'$.

It is straightforward to verify that $\bN_0'$ itself is an ordinal. Since
$\bN_0'$ is an infinite set and every $\al \in \bN_0'$ is a finite set, it
follows that $\bN_0'$ is the smallest infinite ordinal. The usual notation for
$\bN_0'$ is $\bom$.
  \symindex{$\bN_0'$}%
  \symindex{$\bom$}%
We will use this notation sparingly, since it conflicts with the usual usage of
$\om$ in probability theory. The reader must generally rely on context to see
the distinction, although to make things easier, we will use bold font when
using $\bom$ to denote $\bN_0'$. As we will discuss further in Section
\ref{S:std-arith}, we will sometimes find it useful to identify $n$ with $n'$,
giving us a representation of each natural number as a set.

Every well-ordered set is order isomorphic to a unique ordinal. In particular,
no two distinct ordinals are order isomorphic. If $\al = s(\be)$ for some
ordinal $\be$, then $\al$ is called a \emph{successor ordinal}.
  \index{ordinal!successor ---}%
Otherwise, $\al$ is called a \emph{limit ordinal}.
  \index{ordinal!limit ---}%
An ordinal $\al$ is a limit ordinal if and only if $\al = \bigcup_{\xi < \al}
\xi$, and it is a successor ordinal if and only if $\al = s(\bigcup_{\xi < \al}
\xi)$. The smallest limit ordinal is $0'$ and the second smallest limit ordinal
is $\bom$.

If $\al$ is an ordinal, then an \emph{$\al$-sequence}
  \index{a_alpha-sequence@$\al$-sequence}%
is a function whose domain is $\al$, typically denoted by $\ang{x_\xi \mid \xi <
\al}$. Such a sequence is said to have length $\al$. If $\al = n'$, then these
are $n$-tuples. If $\al = \bom$, then these are ordinary sequences indexed by
$\bN_0$.

\subsection{Transfinite induction and recursion}
  \index{induction!transfinite ---}%

Let $\al$ be an ordinal and let $S$ be a set. The principle of transfinite
induction says that if $0' \in S$, and $\be \subset S$ implies $\be \in S$ for
all $\be < \al$, then $\al \subseteq S$. The principle of transfinite recursion
states that if $G: \bigcup_{\xi < \al} S^\xi \to S$, then there exists a unique
sequence $\ang{x_\xi \mid \xi < \al}$ such that $x_\be = G(\ang{x_\xi \mid \xi <
\be})$ for all $\be < \al$. In proofs that use transfinite induction or
recursion, the inductive/recursive step is typically broken into cases according
to whether $\al$ is a successor ordinal or a limit ordinal.

\subsection{Ordinal arithmetic}
  \index{ordinal!arithmetic}%

We define ordinal addition recursively by
\begin{align*}
  \al + 0' &= \al,\\
  \al + s(\be) &= s(\al + \be), \text{ and}\\
  \al + \be &= \ts{\bigcup_{\xi < \be}(\al + \xi)},
    \text{ if $\be > 0'$ is a limit ordinal}.
\end{align*}
Note that $\al + 1' = s(\al)$. Ordinal addition is associative, but not
commutative. For instance, $1' + \bom = \bom$, but $\bom + 1' > \bom$. Ordinal
subtraction on the left is always possible. That is, if $\al \le \be$, then
there exists a unique ordinal $\ga$ such that $\al + \ga = \be$. Consequently,
ordinal addition is left-cancellative, meaning that $\al + \be = \al + \ga$
implies $\be = \ga$. Moreover, it is strictly increasing in the right argument,
meaning that $\al + \be < \al + \ga$ if and only if $\be < \ga$. These facts can
be used, for instance, to prove that the function $f: \ga \to (\be + \ga)
\setminus \be$ given by $f(\xi) = \be + \xi$ is a bijection.

We define ordinal multiplication recursively by
\begin{align*}
  \al \cdot 0' &= 0',\\
  \al \cdot s(\be) &= (\al \cdot \be) + \al,\text{ and}\\
  \al \cdot \be &= \ts{\bigcup_{\xi<\be}(\al \cdot \xi)},
    \text{ if $\be > 0'$ is a limit ordinal}.
\end{align*}
Ordinal multiplication is associative, but not commutative. For instance, $2'
\cdot \bom = \bom$, but $\bom \cdot 2' = \bom + \bom > \bom$. Multiplication is
left distributive over addition, but not right distributive. For instance, $1'
\cdot \bom + 1' \cdot \bom = \bom + \bom > \bom$.

When restricted to $\bN_0'$, ordinal addition and multiplication agree with
ordinary addition and multiplication. That is, $(m + n)' = m' + n'$ and $(m
\cdot n)' = m' \cdot n'$.

\subsection{Cardinal numbers}

If $X$ and $Y$ are sets, we write $X \sim Y$ to mean there exists a bijection
$f: X \to Y$. Every set can be well-ordered, and every well-ordered set is
isomorphic to a unique ordinal. Thus, for every set $X$, there exists an
ordinal $\al$ such that $X \sim \al$. This $\al$, however, is not unique, since
a set can be well-ordered in multiple ways.

We define the \emph{cardinality}
  \index{cardinality}%
of a set $X$ to be the smallest ordinal $\al$ such that $X \sim \al$, and we
write $|X| = \al$ or $\card(X) = \al$. A \emph{cardinal (number)}
  \index{cardinal}%
is an ordinal $\al$ that is the cardinality of some set. Since $|\al + 1| =
|\al|$ whenever $\al$ is infinite, it follows that every infinite cardinal
number is a limit ordinal.

Every finite ordinal is a cardinal, and $\bom$ is a cardinal. No other countably
infinite ordinal besides $\bom$ is a cardinal. Since uncountable sets exist, we
know that uncountable cardinals exist. Let $\bom_1$ be the first uncountable
cardinal, which is also the first uncountable ordinal.
  \symindex{$\bom_1$}%

For any set $X$, we have $|\fP X| > |X|$, where $\fP X$ is the power set of $X$.
  \index{power set}%
  \symindex{$\fP$}%
Hence, for any cardinal number $\ka$, we have $|\fP \ka| > \ka$, which means
that the collection of cardinal numbers greater than $\ka$ is nonempty. Let
$\ka^+$ be the smallest cardinal greater than $\ka$. In particular, $\bom_1 =
\bom^+$. Any cardinal that is of the form $\ka^+$ for some $\ka$ is called a
\emph{successor cardinal}.
  \index{cardinal!successor ---}%
Otherwise, it is called a \emph{limit cardinal}.
  \index{cardinal!limit ---}%

Note that $\bom_1$ is the set of all countable ordinals. Every countable
sequence of countable ordinals has a countable upper bound. More specifically,
if $\al$ is an ordinal with $\al < \bom_1$, and $\ang{\be_\xi \mid \xi < \al}$
is an $\al$-sequence of ordinals with $\be_\xi < \bom_1$ for all $\xi$, then
there exists an ordinal $\ga < \bom_1$ such that $\be_\xi \le \ga$ for all
$\xi$. In fact, we may take $\ga = \bigcup_{\xi < \al} \be_\xi$. Since a
countable union of countable sets is countable, it follows that $\ga < \bom_1$.

More generally, a cardinal $\ka$ is called \emph{regular} if, whenever $\al <
\ka$ and $\ang{\be_\xi \mid \xi < \al}$ is an $\al$-sequence of ordinals with
$\be_\xi < \ka$ for all $\xi$, we have $\bigcup_{\xi < \al} \be_\xi < \ka$.
Since a finite union of finite sets is finite, $\bom$ is regular. Since a
countable union of countable sets is countable, $\bom_1$ is regular.

A cardinal number $\ka$ is called a \emph{strong limit cardinal} if, for all
cardinals $\la < \ka$, we have $|\fP \la| < \ka$. Since $|\fP \ka| \ge \ka^+$, a
strong limit cardinal is always a limit cardinal. Since the power set of a
finite set is finite, $\bom$ is a strong limit cardinal. However, the power set
of a countable set can be uncountable, so $\bom_1$ is not a strong limit
cardinal.

The cardinal $\bom$ is both regular and a strong limit cardinal. We cannot
``reach'' or ``access'' $\bom$ from the finite sets that lie below it, using the
operations of union and power set. A cardinal $\ka$ is called \emph{strongly
inaccessible}
  \index{cardinal!strongly inaccessible ---}%
if it uncountable, regular, and a strong limit cardinal.

\section{Boolean algebras}

A \emph{Boolean algebra}
  \index{Boolean!algebra}%
is a partially ordered set $(B, \le)$ such that
\begin{enumerate}[(i)]
  \item $x \vee y := \sup \{x, y\}$ exists for all $x, y \in B$,
  \item $x \wedge y := \inf \{x, y\}$ exists for all $x, y \in B$,
  \item $x \wedge (y \vee z) = (x \wedge y) \vee (x \wedge z)$ for all $x, y, z
        \in B$,
  \item $x \vee (y \wedge z) = (x \vee y) \wedge (x \vee z)$ for all $x, y, z
        \in B$,
  \item there exists $0, 1 \in B$ such that $0 \le x \le 1$ for all $x \in B$,
        and
  \item for all $x \in B$, there exists $\neg x \in B$ such that $x \wedge \neg
        x = 0$ and $x \vee \neg x = 1$.
\end{enumerate}
It can be shown that in a Boolean algebra, $\neg x$ is unique. If
\begin{enumerate}[(i)]
  \setcounter{enumi}{6}
  \item $\bigvee_{x \in C} x := \sup C$ exists for all countable $C \subseteq
        B$, and
  \item $\bigwedge_{x \in C} x := \inf C$ exists for all countable $C \subseteq
        B$,
\end{enumerate}
then $B$ is a \emph{Boolean $\si$-algebra}.
  \index{Boolean!s_sigma-algebra@$\si$-algebra}%
The smallest Boolean algebra is the \emph{degenerate} Boolean algebra with only
one element, in which $0 = 1$. The smallest nondegenerate Boolean algebra is $\B
= \{0, 1\}$
  \symindex{$\B$}%
with the usual meaning of $\le$. The Boolean algebra $\B$ is also
clearly a Boolean $\si$-algebra.

Let $B$ be a Boolean $\si$-algebra and $N \subseteq B$. Then $N$ is a
\emph{$\si$-ideal} of $B$ if $0 \in N$, $\bigvee_{x \in C} x \in N$ whenever $C
\subseteq N$ is countable, and $x \in N$ implies $y \in N$ for all $y \le x$.

If $N$ is a $\si$-ideal of $B$ and $x, y \in B$, we say $x \equiv y \mod N$ if $
(x \wedge \neg y) \vee (\neg x \wedge y) \in N$. The Boolean operations of $B$
determine Boolean operations on the set of equivalence classes, $B/N$, making
$B/N$ into a Boolean $\si$-algebra.

A \emph{Boolean measure} on a Boolean $\si$-algebra $B$ is a function $m: B \to
[0, \infty]$ such that $m (x) = 0$ if and only if $x = 0$, and $m(\bigvee_n x_n)
= \sum_n m(x_n)$ whenever $i \ne j$ implies $x_i \wedge x_j = 0$. A
\emph{Boolean measure space}
  \index{Boolean!measure space}%
is a pair $(B, m)$ where $B$ is a Boolean $\si$-algebra and $m$ is a Boolean
measure on $B$.

Let $(B, m)$ and $(B', m')$ be Boolean measure spaces. A \emph{homomorphism}
from $(B, m)$ to $(B', m')$ is a function $g: B \to B'$ such that $x \le y$
implies $g(x) \le g(y)$, and $m'(g(x)) = m(x)$. The function g is an
\emph{isomorphism} if it is a bijection. In that case, $g^{-1}$ is an
isomorphism from $(B', m')$ to $(B, m)$. If an isomorphism from $(B, m)$ to
$(B', m')$ exists, then we say that $(B, m)$ and $(B', m')$ are
\emph{isomorphic}.

\section{Measure spaces}\label{S:meas-spaces}

Let $\Om$ be a nonempty set and $\Si$ a collection of subsets of $\Om$. Then
$\Si$ is a \emph{$\si$-algebra (of sets) on $\Om$}
  \index{s_sigma-algebra@$\si$-algebra}%
if $\Si$ is nonempty and closed under complements and countable unions. In this
case, we call $(\Om, \Si)$ a \emph{measurable space}. A set $A \in \Si$ is
called a \emph{measurable set}. Note that a $\si$-algebra is a Boolean
$\si$-algebra when it is equipped with the partial order $\subseteq$.

The intersection of any family of $\si$-algebras is a $\si$-algebra. If $\cE$ is
any collection of subsets of $\Om$, then $\si(\cE)$ denotes the smallest
$\si$-algebra containing $\cE$, and is called the \emph{$\si$-algebra generated
by $\cE$}.
  \symindex{$\si(\cE)$}%
A \emph{measure} on $(\Om, \Si)$ is a function $\mu: \Si \to [0, \infty]$ such
that $\opmu \emp = 0$ and $\opmu \bigcup_1^\infty A_n = \sum_1^\infty \opmu A_n
$ whenever $\{A_n\} \subseteq \Si$ is a pairwise disjoint sequence of measurable
sets. In this case, $(\Om, \Si, \mu)$ is a \emph{measure space}.
  \index{measure space}%
A \emph{measure subspace} of $(\Om, \Si, \mu)$ is a measure space $(\Om, \Si',
\nu)$, where $\Si' \subseteq \Si$ and $\nu = \mu|_{\Si'}$. If $\opmu \Om <
\infty$, then $\mu$ is a \emph{finite measure}. If $\opmu \Om = 1$, then $\mu$
is a \emph{probability measure} and $(\Om, \Si, \mu)$ is a \emph{probability
space}. Any measure subspace of a probability one is also a \emph{probability
space}.
  \index{probability!space}%
In a probability space, it is customary to call the elements $\om \in \Om$ 
\emph{outcomes}, and the measurable sets $A \in \Si$ \emph{events}.

Let $(\Om, \Si, \mu)$ be a measure space. A \emph{$\mu$-null set} (or just a
\emph{null set})
  \index{null set}%
is a set $A \in \Si$ with $\opmu A = 0$. The collection of all null sets is
denoted by $\cN_\mu$. If $A$ and $B$ are subsets of $\Om$, then we write $A = B$
\emph{$\mu$-almost everywhere},
  \index{almost everywhere}%
if $A \tri B$ is a subset of a null set, where
\[
  A \tri B = (A \cap B^c) \cup (A^c \cap B)
\]
is the \emph{symmetric difference}. We will usually abbreviate this as $A = B$
$\mu$-a.e., or if the measure is understood, as just $A = B$ a.e. If $\mu$ is a
probability measure, then we instead write $A = B$ \emph{$\mu$-almost surely},
  \index{almost surely}%
abbreviated as $A = B$ $\mu$-a.s.~or $A = B$ a.s. We also write $A \subseteq B$
a.e.~if $A \cap B^c$ is a subset of a null set. More generally, if $f$ and $g$
are functions with domain $\Om$, then we write $f = g$ a.e.~if there exists $N
\in \cN_\mu$ such that $f(\om) = g(\om)$ for all $\om \in N^c$.

\subsection{Generating \texorpdfstring{$\si$}{sigma}-algebras}
\label{S:gen-sig-alg}

Let $\Om$ be a nonempty set and $\cE$ a collection of subsets of $\Om$. Then
$\si(\cE)$ can be constructed from $\cE$ in an iterative fashion using
transfinite recursion. Let $\cE_0 = \cE$. For an ordinal $\al < \bom_1$, let
\[
  \cE_\al' = \cE_\al \cup \{V^c \mid V \in \cE_\al\},
\]
and
\[
  \ts{
    \cE_{\al + 1} = \cE_\al' \cup \{
      \bigcap \cD \mid \cD \subseteq \cE_\al' \text{ is nonempty and countable}
    \}.
  }
\]
Here, countable means finite or countably infinite. If $\al$ is a limit ordinal,
define $\cE_\al = \bigcup_{\xi < \al} \cE_\xi$.

By transfinite induction, $\cE_\al \subseteq \si(\cE)$ for all $\al < \bom_1$,
so that $\bigcup_{\al < \bom_1} \cE_\al \subseteq \si(\cE)$. Clearly,
$\bigcup_{\al < \bom_1} \cE_\al$ is nonempty and closed under complements. Since
every countable sequence of countable ordinals has a countable upper bound, it
is also closed under countable intersections. Therefore, $\bigcup_{\al < \bom_1}
\cE_\al$ is a $\si$-algebra containing $\cE$, which gives $\si(\cE) \subseteq
\bigcup_{\al < \bom_1} \cE_\al$. This shows that $\si(\cE) = \bigcup_{\al <
\bom_1} \cE_\al$.

For each $V \in \si(\cE)$, we define the \emph{rank of $V$ (with respect to
$\cE$)},
  \index{rank of a measurable set}%
which we denote by $\rk V$, to be the smallest $\al < \bom_1$ such that $V \in
\cE_\al$. Note that $\rk V$ is always a successor ordinal.

\subsection{Complete measure spaces}

Given a measure $\mu$ on $(\Om, \Si)$, we define the associated \emph{outer
measure} by
\[
  \mu^* A = \inf\{\mu B: A \subseteq B \text{ and } B \in \Si\}.
\]
Note that $\mu^* A$ is defined for every $A \subseteq \Om$. Similarly, we
define the \emph{inner measure} by
\[
  \mu_* A = \sup\{\mu B: B \subseteq A \text{ and } B \in \Si\}.
\]
For any $A \subseteq \Om$, we have $\mu_* A \le \mu^* A$.

A \emph{negligible set}
  \index{negligible set}%
is a (not necessarily measurable) subset of a null set. A negligible set that is
also measurable is necessarily a null set. A measure space is called
\emph{complete}
  \index{measure space!complete ---}%
if every negligible set is measurable. A probability space is complete if and
only if every superset of a set of measure 1 is measurable. In a complete
measure space, if $A$ is measurable and $A = B$ a.e., then $B$ is measurable.

If $(\Om, \Si, \mu)$ is a measure space, then
\[
  \ol \Si = \{A \cup B: A \in \Si \text{ and $B$ is negligible}\}
\]
is a $\si$-algebra called the \emph{completion of $\Si$ with respect to $\mu$}.
  \index{s_sigma-algebra@$\si$-algebra!completion of a ---}%
There is a unique measure $\ol \mu$ on $(\Om, \ol \Si)$ that agrees with $\mu$
on $\Si$ and makes $(\Om, \ol \Si, \ol \mu)$ into a complete measure space. The
measure $\ol \mu$ is called the \emph{completion of $\mu$}.
  \index{measure space!completion of a ---}%
Note that a set is $\ol \mu$-null if and only if it is a subset of a $\mu$-null
set. Also, if $A \subseteq \Om$ and $\mu^* A < \infty$, then $A \in \ol \Si$ if
and only if $\mu_* A = \mu^* A$ (see, for instance, \cite[Proposition
1.5.5]{Cohn2013}).

Let $(\Om, \Si, \mu)$ be a measure space and let $A \subseteq \Om$. Then
\[
  \si(\Si \cup \{A\}) = \{(B \cap A) \cup (C \cap A^c): B, C \in \Si\}.
\]
Suppose $\mu$ is a finite measure and let $\al \ge 0$ satisfy $\mu_* A \le \al
\le \mu^* A$. Then there exists a measure $\nu$ on $(\Om, \si(\Si \cup \{A\}))$
such that $\nu|_\Si = \mu$ and $\opnu A = \al$ (see, for instance, \cite
[Exercise 1.5.12]{Cohn2013} or \cite[Theorem 1.12.14]{Bogachev2007}).

\subsection{Dynkin systems}

Let $\Om$ be a nonempty set and $\De$ a collection of subsets of $\Om$. Then
$\De$ is a \emph{Dynkin system}, or \emph{$\la$-system}, if
  \index{l_lambda-system@$\la$-system}%
  \index{Dynkin system|see {$\la$-system}}%
\begin{enumerate}[(i)]
  \item $\Om \in \De$,
  \item if $A, B \in \De$ with $A \subseteq B$, then $B \setminus A \in \De$,
        and
  \item If $\{A_n\} \subseteq \De$ with $A_n \subseteq A_{n + 1}$, then
        $\bigcup_n A_n \in \De$.
\end{enumerate}
Equivalently, one can define $\De$ to be a Dynkin system if it is nonempty and
satisfies
\begin{enumerate}[(i)$'$]
  \item if $A \in \De$, then $A^c \in \De$
  \item if $\{A_n\} \subseteq \De$ are pairwise disjoint, then $\bigcup_n A_n
        \in \De$.
\end{enumerate}
Every $\si$-algebra is a Dynkin system. Conversely, a Dynkin system that is
closed under (finite) intersections is a $\si$-algebra.

The intersection of any family of Dynkin systems is a Dynkin system.  If $\cE$
is any collection of subsets of $\Om$, then there is a smallest Dynkin system
containing $\cE$, called the \emph{Dynkin system generated by $\cE$}.

If $\cE$ is a collection of subsets of $\Om$, then $\cE$ is a
\emph{$\pi$-system}
  \index{p_pi-system@$\pi$-system}%
if $A, B \in \cE$ implies $A \cap B \in \cE$. Dynkin's $\pi$-$\la$ theorem
  \index{p_pi-lambda theorem@$\pi$-$\la$ theorem}%
states that if $\De$ is a Dynkin system, $\cE$ is a $\pi$-system, and $\cE
\subseteq \De$, then $\si(\cE) \subseteq \De$.

Let $\De$ be a Dynkin system on $\Om$ and let $B \in \De$. Define $\De|_B = \{A
\subseteq \Om: A \cap B \in \De\}$. Then $\De|_B$ is a Dynkin system on $\Om$
called the \emph{restriction of $\De$ to $B$}. If $\De$ is a $\si$-algebra, then
$\De|_B$ is a $\si$-algebra.

\subsection{Measurable functions and pushforwards}

Let $(\Om, \Si, \mu)$ be a measure space and $(S, \Ga)$ a measurable space. A
function $f: \Om \to S$ is \emph{measurable}, or \emph{$(\Si, \Ga)$-measurable},
  \index{measurable function}%
if $U \in \Ga$ implies $f^{-1}(U) \in \Si$.

If $A \subseteq \Om$, then $1_A$ denotes the \emph{indicator function} of $A$,
  \index{indicator function}%
  \symindex{$1_A$, $1_A(\om)$}
and is defined by
\[
  1_A(\om) = \begin{cases}
    1 &\text{if $\om \in A$},\\
    0 &\text{if $\om \notin A$}.
  \end{cases}
\]
The function $1_A$ can be regarded as taking values in the measurable space $
(S, \Ga)$, where $S = \{0, 1\}$ and $\Ga = \fP S$. In that case, if $A \subseteq
\Om$, then $A$ is measurable if and only if $1_A$ is measurable.

Suppose $f$ is a measurable function from a measure space $(\Om, \Si, \mu)$ to a
measurable space $(S, \Ga)$. Since $f^{-1}(\emp) = \emp$, $f^{-1}(U^c) = f^ {-1}
(U)^c$, and $f^{-1}(\bigcup_n U_n) = \bigcup_n f^{-1}(U_n)$, it follows that
$\mu \circ f^{-1}$ is a measure on $(S, \Ga)$, called the \emph{pushforward} of
$\mu$.
  \index{pushforward}%

\subsection{Measure space isomorphisms}

Let $(\Om, \Si, \mu)$ be a measure space and $\cN_\mu$ the collection of null
sets. The set $\cN_\mu$ is a $\si$-ideal of $\Si$, and $A = B$ a.e.~if and only
if $A = B \mod \cN_\mu$. The equivalence classes modulo $\cN_\mu$ are $[A]_\mu =
\{B \in \Si: A = B \text{ a.e.}\}$, and the set of equivalence classes,
$\Si/\cN_\mu$, is a Boolean $\si$-algebra. If we define $m_\mu: \Si/\cN_\mu \to
[0, \infty]$ by $m_\mu([A]_\mu) = \mu(A)$, then $m_\mu$ is well-defined and is a
Boolean measure on $\Si/\cN_\mu$. We call $(\Si/\cN_\mu, m_\mu)$ the 
\emph{Boolean measure space corresponding to $(\Om, \Si, \mu)$}. We say that two
measure spaces are \emph{isomorphic}
  \index{measure space!isomorphism}%
if their corresponding Boolean measure spaces are isomorphic.

Let $(\Om, \Si, \mu)$ and $(S, \Ga, \nu)$ be measure spaces, let $f: \Om \to S$
be measurable, and assume $\nu = \mu \circ f^{-1}$. Note that $\mu \, f^{-1}(U)
\tri f^ {-1}(V) = \nu \, U \tri V$. Hence, $U = V$ $\nu$-a.e.~if and only if $f^
{-1}(U) = f^{-1}(V)$ $\mu$-a.e. It follows, therefore, that $f$ determines an
injective function $g: \Ga/\cN_\nu \to \Si/\cN_\mu$ given by $g([U]_\nu) = [f^
{-1}(U)]_\mu$. This function is, in fact, a homomorphism from $(\Ga/\cN_\nu,
m_\nu)$ to $(\Si/\cN_\mu, m_\mu)$. Hence, $g$ is an isomorphism if and only if
it is surjective, that is, if and only if
\[
  \text{
    $[A]_\mu \in \Si/\cN_\mu$ implies
      $g([U]_\nu) = [A]_\mu$ for some $[U]_\nu \in \Ga/\cN_\nu$.
  }
\]
We can rewrite this in terms of $f$ to state the following. For two measure
spaces, $(\Om, \Si, \mu)$ and $(S, \Ga, \nu)$, to be isomorphic, it suffices
that there exists a measurable function $f: \Om \to S$ such that $\nu = \mu
\circ f^{-1}$, and
\[
  \text{
    for all $A \in \Si$,
      there exists $U \in \Ga$ such that $f^{-1}(U) = A$ $\mu$-a.e.
  }
\]
If such an $f$ exists, we say that $f$ \emph{induces an isomorphism} from $(\Om,
\Si, \mu)$ to $(S, \Ga, \nu)$. Note that in this case, $f$ also induces an
isomorphism from $(\Om, \ol \Si, \ol \mu)$ to $(S, \ol \Ga, \ol \nu)$.

Two measure spaces, $(\Om, \Si, \mu)$ and $(S, \Ga, \nu)$, are \emph{pointwise
isomorphic} if there exists a measurable bijection $f: \Om \to S$ such that $f^
{-1}$ is measurable and $\nu = \mu \circ f^{-1}$. In that case, $f$ is called a
\emph{pointwise isomorphism}. It is straightforward to verify that if $f$ is a
pointwise isomorphism from $(\Om, \Si, \mu)$ to $(S, \Ga, \nu)$, then it is also
a pointwise isomorphism from $(\Om, \ol \Si, \ol \mu)$ to $(S, \ol \Ga, \ol
\nu)$. In other words, if two measure spaces are pointwise isomorphic, then so
are their completions.

Note that pointwise isomorphic measure spaces are isomorphic, and a pointwise
isomorphism induces an isomorphism. The \emph{pointwise isomorphism class} of $
(\Om, \Si, \mu)$ is the collection of all measure spaces that are pointwise
isomorphic to $(\Om, \Si, \mu)$.

Let $(\Om, \Si, \mu)$ be a measure space and let $S$ be a set. Let $h: \Om \to
S$ be a function. Define $\Ga = \{A \subseteq S \mid h^{-1}(A) \in \Si\}$ and
define $\nu = \mu \circ h^{-1}$. Then $(S, \Ga, \nu)$ is a measure space and $h$
is a measurable function from $\Om$ to $S$. We call $(S, \Ga, \nu)$ the
\emph{measure space image of $(\Om, \Si, \mu)$ under $h$}.
  \index{measure space!image}%
If $h$ is a bijection, then $h$ is a pointwise isomorphism from $(\Om, \Si,
\mu)$ to $(S, \Ga, \nu)$.

\section{Structures}

Let $A$ be a set. If $f$ is a function with domain $A$ and $a \in A$, then the
value of $f$ at $a$ will be denoted variously by $f(a)$, $fa$ or $a^f$. Note
that $\emp$ is a function. In fact, $\emp$ is the unique function with domain
$\emp$.

If $n \in \bN$, then $A^n$ is the set of $n$-tuples, $\vec a = \ang{a_1, \ldots,
a_n}$. We let $A^0 = \{\emp\}$. If $f$ is a function with domain $A^n$, we write
$f \vec a = f(a_1, \ldots, a_n)$.

For $n \ge 1$, an \emph{$n$-ary relation} (or \emph{predicate}) is a subset of
$R \subseteq A^n$. We write $R \vec a$ for $\vec a \in R$ and $\neg R \vec a$
for $\vec a \notin R$. For $n \ge 0$, an \emph{$n$-ary operation} is a function
$f: A^n \to A$. A $0$-ary operation is a function $f: \{\emp\} \to A$, and is
uniquely determined by $c = f(\emp) \in A$. In this case, we identify $f$ with
$c$. A $0$-ary operation is also called a \emph{constant}.

An \emph{extralogical signature} is a set $L$ of symbols. Each symbol in $L$ is
called an \emph{extralogical symbol},
  \index{extralogical signature}%
  \index{extralogical symbol}%
and has both a \emph{type} and an \emph{arity}. The possible types are
\emph{relation symbols} and \emph{function (or operation) symbols}. The arity is
a nonnegative integer. Relation symbols may have an arity $n \ge 1$. Function
symbols may have an arity $n \ge 0$. A $0$-ary function symbol is also called a
\emph{constant symbol}.

When referring to symbols in $L$, we will adopt the convention that, unless
otherwise stated, $c$ will denote a constant symbol, $r$ a relation symbol, and
$f$ a function symbol with arity $n \ge 1$.

Let $L$ be an extralogical signature and let $A$ be a set. For each symbol
$\s \in L$, let $\s^\om$ be a relation such that
\begin{enumerate}[(i)]
  \item if $\s$ is an $n$-ary relation symbol, then $\s^\om \subseteq A^n$ is an
        $n$-ary relation on $A$, and
  \item if $\s$ is an $n$-ary function symbol, then $\s^\om: A^n \to A$ is an
        $n$-ary function.
\end{enumerate}
Let $L^\om = \{\s^\om \mid \s \in L\}$ and $\om = (A, L^\om)$. Then $\om$ is an
\emph{$L$-structure}. The set $A$ is called the \emph{domain} of $\om$. The
structure $\om$ is called \emph{finite} or \emph{infinite} if $A$ is finite or
infinite, respectively.
  \index{structure}%
  \index{structure!domain of a ---}%

Let $L$ be an extralogical signature, $\nu = (B, L^\nu)$ an $L$-structure, and
$A \subseteq B$. Suppose that for all functions $f^\nu \in L^\nu$, the subset
$A$ is \emph{closed} under $f^\nu$, meaning that $\vec a \in A^n$ implies $f^\nu
\vec a \in A$. Let $f^\om = f^\nu|_{A^n}$, $r^\om = r^\nu \cap A^n$, and $L^\om
= \{\s^\om \mid \s \in L\}$. The $\om = (A, L^\om)$ is a structure. We call
$\om$ a \emph{substructure} of $\nu$, and write $\om \subseteq \nu$.

Let $L$ be an extralogical signature, $\om = (A, L^\om)$ an $L$-structure,
and $L_0 \subseteq L$. Define $L^{\om_0} = \{\s^\om \mid \s \in L_0\}$ and
$\om_0 = (A, L^{\om_0})$. Then $\om_0$ is an $L_0$-structure. We call
$\om_0$ the \emph{$L_0$-reduct of $\om$}, and we call $\om$ an 
\emph{$L$-expansion of $\om_0$}.
  \index{structure!reduct of a ---}%
  \index{structure!expansion of a ---}%
  \index{reduct|see {structure}}%
  \index{expansion|see {structure}}%

\subsection{Structure homomorphisms}

Let $\om = (A, L^\om)$ and $\nu = (B, L^\nu)$ be $L$-structures and let $g: A
\to B$. We will abuse notation and also write $g: \om \to \nu$. For $\vec a \in
A^n$, we write $g \vec a$ for $\ang{ga_1, \ldots, ga_n}$. Assume that
\begin{enumerate}[(i)]
  \item $g f^\om \vec a = f^\nu g \vec a$ for all function symbols $f \in L$,
  \item $g c^\om = c^\nu$ for all constant symbols $c \in L$, and
  \item $r^\om \vec a$ implies $r^\nu g \vec a$ for all relation symbols $r \in
        L$.
\end{enumerate}
Then $g$ is a \emph{homomorphism}. A \emph{strong homomorphism} is a
homomorphism such that $r^\nu g \vec a$ implies $r^\om \vec b$ for some $\vec b
\in g^{-1} g \vec a$. An \emph{embedding} is an injective strong homomorphism,
an \emph{isomorphism} is a bijective strong homomorphism, and an
\emph{automorphism} is an isomorphism from $\om$ to $\om$. If $g$ is an
isomorphism, then $r^\om \vec a$ if and only if $r^\nu g \vec a$. We say that
$\om$ and $\nu$ are \emph{isomorphic}, written $\om \simeq \nu$, if there is an
isomorphism $g: \om \to \nu$.
 \index{structure!isomorphism}%

Let $\om = (A, L^\om)$ be a structure, let $B$ be a set with the same
cardinality as $A$, and let $g: A \to B$ be any bijection. Define the
$L$-structure $\nu$ with domain $B$ by
\begin{enumerate}[(i)$'$]
  \item $f^\nu \vec b = g f^\om g^{-1} \vec b$
        for all function symbols $f \in L$,
  \item $c^\nu = g c^\om$ for all constant symbols $c \in L$, and
  \item $r^\nu \vec b$ if and only if $r^\om g^{-1} \vec b$ for all relation
        symbols $r \in L$.
\end{enumerate}
Then $g$ is an isomorphism from $\om$ to $\nu$, and we call $\nu$ the 
\emph{isomorphic image of $\om$ under $g$}.
  \index{structure!isomorphic image of a ---}

\subsection{The standard structure of arithmetic}\label{S:std-arith}

Consider the extralogical signature $L = \{\ul 0, \S, +, \bdot, <\}$, where $\ul
0$ is a constant symbol, $\S$ is a unary function symbol, $+$ and $\bdot$ are
binary function symbols, and $<$ is a binary relation symbol.

We can define the structure $\cN = (\bN_0, L^\cN)$
  \symindex{$\cN$}%
as follows. Let $\ul 0^\cN = 0$, and let $\S^\cN$ be the successor function, $n
\mapsto n + 1$. Define $+^\cN$ and ${\bdot \,}^\cN$ to be ordinary addition and
multiplication in $\bN_0$, and $<^\cN$ to be the ordinary less--than relation on
$\bN_0$. For $+$, $\bdot$, and $<$, we will need to rely on context to
distinguish between the extralogical symbols and their ordinary meanings. To
make matters worse, we will also sometimes use $\S$ to denote the function $n
\mapsto n + 1$, so that context is also required to determine whether $\S$
denotes a symbol or a function.

The structure $\cN = (\bN_0, 0, \S, +, \bdot, <)$ is called the \emph{standard
structure of arithmetic}.
  \index{standard structure of arithmetic}%
Sometimes, we will use this phrase to refer to the same structure, but with $<$
omitted.

Now consider a different structure $\cN' = (\bN_0', L^{\cN'})$,
  \symindex{$\cN'$}%
defined as follows. Let $0^{\cN'} = \emp$. Let $\S^{\cN'} = s$, where $s(\al) =
\al \cup \{\al\}$ is the successor function on ordinals. Define $+^\cN$ and $
{\bdot \,}^\cN$ to be ordinal addition and ordinal multiplication, and let
${<^{\cN'}} = {\in}$. It can be shown that the function $n \mapsto n'$, defined
in Section \ref{S:ordinals}, is an isomorphism from $\cN$ to $\cN'$. We may
therefore sometimes identify $\cN$ and $\cN'$, thinking of the natural numbers
as being identical to the finite ordinals.

\subsection{Factor structures}

Let $\om = (A, L^\om)$ be a structure and let $\approx$ be an equivalence
relation on $A$. We write $\vec a \approx \vec b$ to mean that $a_i \approx b_i$
for all $i$. Suppose that for all function symbols $f \in L$, we have $\vec a
\approx \vec b$ implies $f^\om \vec a = f^\om \vec b$. Then $\approx$ is a
\emph{congruence (relation) in $\om$}.

Let $\om$ and $\nu$ be structures and $g: \om \to \nu$ a homomorphism. Define
${\approx_g} \subseteq A^2$ by $a \approx_g b$ if and only if $ga = gb$. Then
$\approx_g$ is a congruence in $\om$ called the \emph{kernel of $g$}. Let $A' =
A/{\approx}$ be the set of equivalence classes under $\approx$. Let $a/{\approx}$
denote the equivalence class of $a$ and write $\vec a/{\approx} = 
\ang{a_1/{\approx}, \ldots, a_n/{\approx}}$.

If $f \in L$ is a function symbol, define $f^{\om'}: (A')^n \to A'$ by $f^
{\om'}(\vec a/{\approx}) = (f^\om \vec a)/{\approx}$. If $r \in L$ is a relation
symbol, define $r^{\om'} \subseteq (A')^n$ by $r^{\om'}(\vec a/{\approx})$ if
and only if $r^\om \vec b$ for some $\vec b \approx \vec a$. It can be shown
that both $f^{\om'}$ and $r^{\om'}$ are well-defined. We also define $c^{\om'} =
c^\om/{\approx}$. Then $\om' = (A', L^{\om'})$ is a structure, denoted by
$\om/{\approx}$, and called the \emph{factor structure of $\om$ modulo
$\approx$}.

Let $\approx$ be a congruence in an $L$-structure $\om$. According to the
homomorphism theorem (see, for example, \cite[Section 2.1]{Rautenberg2010}), the
map $a \mapsto a/{\approx}$ is a strong homomorphism from $\om$ onto $\om/
{\approx}$, which we call the \emph{canonical homomorphism}.

Conversely, let $\om$ and $\nu$ be $L$-structures and $g: \om \to \nu$ a
surjective strong homomorphism. Let $\approx$ be the kernel of $g$. Let $k$ be
the canonical homomorphism from $\om \to \om/{\approx}$ and let $\iota$ denote
the map $a/{\approx} \mapsto ga$. Also according to the homomorphism theorem,
the map $\iota$ is an isomorphism from $\om/{\approx}$ to $\nu$, and $g = \iota
\circ k$.

\subsection{Direct products of structures}

Let $L$ be an extralogical signature and let $\ang{\om_i \mid i \in I}$ be an
indexed collection of $L$-structures. We let $B = \prod_{i \in I} A_i$ and adopt
the notation $a = \ang{a_i \mid i \in I} \in B$, $\vec a = \ang{a^1, \ldots,
a^n} \in B^n$, and $\vec a_i = \ang{a_i^1, \ldots, a_i^n} \in A_i^n$. For
symbols in $L$, we define relations and operations on $B$ by $r^\nu \vec a$ if
and only if $r^{\om_i} \vec a_i$ for all $i \in I$, $f^\nu \vec a = \ang{f^
{\om_i} \vec a_i \mid i \in I}$, and $c^\nu = \ang{c^{\om_i} \mid i \in I}$.
Then $\nu = (B, L^\nu)$ is a structure called the \emph{direct product} of
$\ang{\om_i \mid i \in I}$, and denoted by $\prod_{i \in I} \om_i$.

If $\om_i = \om$ for all $i \in I$, then $\nu$ is called the \emph{direct power}
and is denoted by $\om^I$. If $I = \{1, \ldots, n\}$, then we denote $\prod_{i
\in I} \om_i$ by $\om_1 \times \cdots \times \om_n$, and we denote $\om^I$ by
$\om^n$. Note that $a \mapsto \ang{a \mid i \in I}$ is an embedding from $\om$
to $\om^I$.

\section{Strings}

Let $\A$ be a nonempty set, which we will call an \emph{alphabet}. Formally, it
does not matter what the elements of $\A$ are, but in this context, we will
refer to the elements of $\A$ as \emph{symbols}. The set of \emph{(finite)
strings}
  \index{string}%
over $\A$ is the set, $\cS = \bigcup_{n = 0}^\infty \A^n$. A string in $\A^n$ is
said to have \emph{length} $n$. The unique string of length $0$ is $\emp$, and
we call this the \emph{empty string}. A string of length $1$ is called an
\emph{atomic} string.

Strings are written without brackets or commas, so that we write $\s_1 \cdots
\s_n$, rather than $\ang{\s_1, \ldots, \s_n}$. Let $\xi = \s_1 \cdots \s_n$ and
$\eta = \s_{n + 1} \cdots \s_{n + m}$ be strings of lengths $n$ and $m$,
respectively. The \emph{concatenation} of $\xi$ and $\eta$, denoted by $\xi
\eta$, is the string $\s_1 \cdots \s_{n + m}$.

If $\xi = \xi_1 \eta \xi_2$, where $\eta \ne \emp$, then $\eta$ is called a 
\emph{segment} (or \emph{substring}) of $\xi$. If $\eta \ne \xi$, then $\eta$ is
a \emph{proper} segment (or substring) of $\xi$. If $\xi_1 = \emp$, then $\eta$
is an \emph{initial} segment (or substring) of $\xi$. If $\xi_2 = \emp$, then
$\eta$ is a \emph{terminal} segment (or substring) of $\xi$.

%% file: prop-calc.tex

\chapter{Propositional Calculus}\label{Ch:prop-calc}

In this chapter, we develop a calculus of inductive inference for sentences in a
formal logical language. The language we focus on, denoted by $\cF$, is
propositional. The sentences (or formulas) of $\cF$ consist of primitive
propositional variables connected by negation and conjunction. Using negation
and conjunction, we can define other logical connectives, such as disjunction
and material implication. Our language $\cF$, however, is not the usual
propositional language one finds in basic logic textbooks. Our language is
infinitary, in the sense that we allow countable conjunctions. That is, if
$\ph_n$ is a formula of $\cF$ for every $n$, then so is $\bigwedge_1^\infty
\ph_n$.

The calculus of deductive inference in $\cF$ is well understood. The study of
infinitary languages dates back to the papers of Scott and Tarski
\cite{Scott1958, Tarski1958} and to the dissertation \cite{Karp1959} and book
\cite{Karp1964} of Carol Karp. Deductive calculus in $\cF$ is represented by a
derivability relation, $\vdash$. If $\ph$ is a formula of $\cF$ and $X$ is a set
of formulas, then $X \vdash \ph$ denotes the fact that $\ph$ can be derived from
$X$ using the rules of deductive inference. In Section \ref{S:formulas}, we
define $\cF$ and $\vdash$, and establish important facts about them. We then
present the notion of a deductive theory, which is a set of formulas that is
closed under deductive inference. An important point, noted in Remark
\ref{R:theory-first}, is that one can define deductive theories without any
reference to derivability, and then define derivability in terms of theories.
This is the approach we take in the development of our inductive calculus.

An inductive statement is a triple, $(X, \ph, p)$, where $X$ is a set of
formulas, $\ph$ is a formula, and $p \in [0, 1]$. Intuitively, an inductive
statement can be thought of as asserting that $\ph$ is partially entailed by $X$
with degree $p$. In an inductive statement, $X$ is called the antecedent, $\ph$
is called the consequent, and $p$ is called the probability. We typically use a
letter such as $P$ to denote sets of inductive statements. We write $P(\ph \mid
X) = p$ to mean that $(X, \ph, p) \in P$ and refer to $p$ as the probability of
$\ph$ given $X$. The ultimate goal of this chapter is to develop a calculus by
which we can take a given set of inductive statements $P$, and infer a new
inductive statement $(X, \ph, p)$. When such an inference is possible, we will
write $P \vdash (X, \ph, p)$, thereby extending the use of the turnstile symbol
$\vdash$ from formulas to inductive statements.

As mentioned above, we take an indirect route to defining $\vdash$. We first
define an inductive theory. Intuitively, if $P$ is an inductive theory, and if
it is possible to infer $(X, \ph, p)$ from $P$, then $(X, \ph, p)$ is already an
element of $P$. The bulk of this chapter is devoted to defining the notion of an
inductive theory. Once this is done, we define $\vdash$ in terms of inductive
theories.

In Sections \ref{S:entire} and \ref{S:closed}, we present our rules of inductive
inference. There are nine of them altogether, which we denote by (R1)--(R9).
Rules (R1)--(R4) describe the connection between inductive and deductive
inference. Rules (R5)--(R7) are the core rules of inductive inference: the
addition, multiplication, and continuity rules. Rule (R8) says that we can use
uniqueness to make inferences: if there is a unique way to assign a probability
without violating (R1)--(R7), then we may infer that probability. Rule (R9) says
that we may freely add formulas of probability one to our antecedents.

To define an inductive theory, we must first define what it means for a set of
inductive statements to be closed under the rules of inference. We do this in
tiers. An admissible set is one that is closed under (R1), an entire set is
closed under (R1)--(R7), a semi-closed set is closed under (R1)--(R8), and a
closed set is closed under (R1)--(R9).

In Section \ref{S:entire}, we define admissible and entire sets. We then prove
several theorems about entire sets, such as inclusion-exclusion, Bayes' theorem,
and countable additivity. Section \ref{S:closed} begins with the definition of
semi-closed and closed sets. We then turn our attention to a notion that has no
analogue in the deductive calculus. The set of inductive statements is much
larger than the set of formulas. As such, it is possible to have inductive
statements that are so far apart, in a certain sense, that they can never be
related to one another via the rules of inductive inference. That is, no chain
of inductive reasoning could possibly include both statements. Such statements
would be trivially incapable of producing a contradiction. Yet they are also
incapable of meaningfully contributing to a common argument. Because of this, a
closed set could potentially contain components that bear no logical connection
to one another. In Section \ref{S:closed}, we make this notion of connectivity
precise. We then define an inductive theory to be a closed set that is also
connected. Section \ref{S:closed} concludes by using the definition of an
inductive theory to define the inductive derivability relation, $P \vdash (X,
\ph, p)$.

The theorems in Section \ref{S:closed} are presented without proof, so that the
reader can see a complete overview of the development. Their proofs are
presented in Sections \ref{S:lifts} and \ref{S:ind-cond}. In Section 
\ref{S:lifts}, we prove that inductive theories exist and are well-defined. The
proof is primarily constructive, showing how to build up an inductive theory
from more basic elements. We then say that a set of inductive statements is
consistent if it can be extended to an inductive theory. In Section 
\ref{S:ind-cond}, we prove that every consistent set can be uniquely extended to
an inductive theory.

Section \ref{S:ind-cond} concludes with an important generalization of inductive
derivability. So far we have only discussed a process of inference whereby we
take a set of inductive statements, $P$, and use them to derive a new inductive
statement, $(X, \ph, p)$. In an inferential argument such as this, the inductive
statements in $P$ are our hypotheses. Since every inductive statement has the
form $P(\ph \mid X) = p$, it follows that every one of our hypotheses must have
this form as well. In Section \ref{S:ind-cond}, we allow for a broader class of
hypotheses. By introducing what we call inductive conditions, we are able to use
hypotheses such as $P(\ph \mid X) > p$ or $P(\ph \wedge \psi \mid X, \psi) = P
(\ph \mid X)$.

\section{Formulas and deductive inference}\label{S:formulas}

\subsection{Propositional formulas}

Let $PV$
  \symindex{$PV$}%
be a nonempty set whose elements we call \emph{propositional variables}.
  \index{variable!propositional ---}%
We define an alphabet, $\A = PV \cup \{\neg, \bigwedge\}$. We will define the
set of formulas so that a formula is a finite tuple, where each element in the
tuple is either a symbol from our alphabet, a formula, or a countable set of
formulas.

Let $S_0 = \{\ang{p} \mid p \in PV\}$. For an ordinal $\al<\bom_1$, let
\[
  S_\al' = S_\al \cup \{\ang{\neg, \ph} \mid \ph \in S_\al\}.
\]
When writing tuples such as these, we will typically omit the commas and angled
brackets, so that, for instance, $\neg \ph = \ang{\neg, \ph}$. In particular,
for $\bfr \in PV$, we identify $\ang{\bfr}$ and $\bfr$ so that we may write $S_0
= PV$.

We then define
\[
  S_{\al + 1} = S_\al' \cup \{
    \ang{\ts{\bigwedge}, \Phi} \mid
    	\Phi \subseteq S_\al' \text{ is nonempty and countable}
  \}.
\]
Here, countable means finite or countably infinite. As above, we will typically
write $\bigwedge \Phi$ as shorthand for ordered pairs of this type.

In the case that $\al$ is a limit ordinal, we define $S_\al = \bigcup_{\xi <
\al} S_\xi$. Finally, we define $\cF = \cF_{\bom_1} = \bigcup_{\al < \bom_1}
S_\al$.
  \symindex{$\cF$}%
Note that $S_\al \subseteq S_\be$ whenever $\al < \be$. An element $\ph \in \cF$
is called a \emph{(propositional) formula} or \emph{sentence}.
  \index{formula}%
  \index{sentence}%
A formula may also be called a \emph{deductive statement}, in contrast to
inductive statements, to be defined later. The set $\cF$ depends on the choice
of $PV$. We will rarely need to emphasize this fact, but when we do, we will
write $\cF(PV)$ instead of $\cF$.

Let $\cF_\fin$
  \symindex{$\cF_\fin$}%
denote the smallest subset of $\cF$ that satisfies
\begin{enumerate}[(i)]
  \item $PV \subseteq \cF_\fin$,
  \item if $\ph \in \cF_\fin$, then $\neg \ph \in \cF_\fin$, and
  \item if $\Phi \subseteq \cF_\fin$ is nonempty and finite, then $\bigwedge
        \Phi \in \cF_\fin$.
\end{enumerate}
Formulas in $\cF_\fin$ are said to be \emph{finitary}. The set $\cF_\fin$ is, in
fact, the set of formulas used in finitary propositional logic. Or rather, it is
one of several equivalent definitions of the finitary propositional language.
The reader can consult any introductory text on mathematical logic for the basic
properties of $\cF_\fin$ and its corresponding syntax and semantics. When
necessary, we will cite \cite{Rautenberg2010} for this purpose.

\begin{thm}[The principle of formula induction]\label{T:prop-form-ind}
    \index{induction!formula ---}%
  The set of formulas, $\cF$, is the smallest set that satisfies the
  following:
  \begin{enumerate}[(i)]
    \item $PV \subseteq \cF$,
    \item if $\ph \in \cF$, then $\neg \ph \in \cF$, and
    \item if $\Phi \subseteq \cF$ is nonempty and countable, then $\bigwedge
    			\Phi \in \cF$.
  \end{enumerate}
\end{thm}

\begin{proof}
  Property (i) follows since $PV = S_0 \subseteq \cF$. Suppose $\ph \in \cF$.
  Then there exists $\al < \bom_1$ such that $\ph \in S_\al$. Thus, $\neg \ph
  \in S_\al' \subseteq \cF$, proving property (ii). Finally, suppose $\Phi
  \subseteq \cF$ is nonempty and countable. Enumerate $\Phi$ as $\Phi =
  \{\ph_n\}_{n < \al}$, where $0 < \al \le \bom$. For each $n < \al$, choose
  $\al_n < \bom_1$ such that $\ph_n \in S_{\al_n}$. Choose $\be < \bom_1$ such
  that $\al_n \le \be$ for all $n < \al$. It follows that $\Phi$ is a nonempty,
  countable subset of $S_\be \subseteq S_\be'$, and so $\bigwedge \Phi \in S_
  {\be + 1} \subseteq \cF$, proving (iii).

  Now let $S$ be any set that satisfies these three properties. Using
  transfinite induction, it is easy to verify that $S_\al \subseteq S$ for all
  ordinals $\al$. Thus, $\cF = \bigcup_{\al < \bom_1} S_\al \subseteq S$.
\end{proof}

We may sometimes write $\bigwedge_{\ph \in \Phi} \ph$ for $\bigwedge \Phi$. If
$\Phi = \{\ph_n\}_{n \in \bN}$, we may write $\bigwedge_{n = 1}^\infty \ph_n$
for $\bigwedge \Phi$. If $\Phi \subseteq \cF$, then we use the notation $\neg
\Phi = \{\neg \ph: \ph \in \Phi\}$. As shorthand, we define
\begin{align*}
  \ts{\bigvee} \Phi &= \neg \ts{\bigwedge} \neg \Phi
    \text{ for $\Phi$ nonempty and countable},\\
  (\ph \wedge \psi) &= \ts{\bigwedge} \{\ph, \psi\},\\
  (\ph \vee \psi) &= \ts{\bigvee} \{\ph, \psi\},\\
  (\ph \to \psi) &= (\neg \ph \vee \psi), \text{ and}\\
  (\ph \tot \psi) &= (\ph \to \psi) \wedge (\psi \to \ph).
\end{align*}
We fix an arbitrary $\bfr_0 \in PV$ and define $\bot = (\bfr_0 \wedge \neg
\bfr_0)$ and $\top = \neg \bot$. The symbols, $\bot$ and $\top$, are called
\emph{falsum} and \emph{verum}, respectively. We adopt the convention that
$\bigwedge \emp = \top$ and $\bigvee \emp = \bot$.
  \index{falsum}%
  \index{verum}%
  \symindex{$\bot$ (in $\cF$)}%
  \symindex{$\top$ (in $\cF$)}%

As another form of shorthand, we may also sometimes omit outer parentheses in
formulas, and occasionally inner parentheses with the understanding that $\to$
associates right to left, other symbols associate left to right, and that
formulas obey the order of operations, $\neg$, $\wedge$, $\vee$, $\to$, $\tot$.

Finally, we may occasionally use the notation $\ph^x$, where $x$ is an element
of the Boolean algebra $\B = \{0, 1\}$, to mean $\ph^1 = \ph$ and $\ph^0 = \neg
\ph$.
  \symindex{$\ph^1, \ph^0$}%

Given $\ph \in \cF$, we define the set of \emph{subformulas} of $\ph$, denoted
by $\Sf \ph$,
  \index{subformula}%
by formula recursion. Namely, $\Sf \bfr = \{\bfr\}$ for $\bfr \in PV$, $\Sf \neg
\ph = \{\neg \ph\} \cup \Sf \ph$, and $\Sf \bigwedge \Phi = \{\bigwedge \Phi\}
\cup \bigcup_{\ph \in \Phi} \Sf \ph$. It follows by formula induction that $\Sf
\ph$ is countable for every $\ph \in \cF$. Also, the set of propositional
variables that appear in a formula $\ph$ is simply $PV \cap \Sf \ph$. In
particular, each formula $\ph$ makes use of only countably many propositional
variables.

\begin{rmk}
  The construction presented here is analogous to the one suggested in 
  \cite{Keisler1971} for formulas in infinitary predicate logic. An alternative
  construction builds formulas out of countably long sequences of symbols in an
  alphabet. In this case, one must deal precisely with the notion of
  concatenation for such strings, such as the concatenation of countably many
  strings, each of which may itself be countably long. All of this is detailed
  in \cite{Karp1964}.
\end{rmk}

\subsection{A calculus of natural deduction}\label{S:prop-nat-ded}

Given a relation $\vdash$ from $\fP \cF$ to $\cF$, we write $X \vdash Y$ to
mean $X \vdash \ph$ for all $\ph \in Y$. A comma-separated list on either side
of the turnstile, $\vdash$, refers to a union, and isolated formulas refer to
the singleton set that contains them. For example, $X, Y, \ph \vdash \psi, \ze$
means $X \cup Y \cup \{\ph\} \vdash \{\psi, \ze\}$, which means that $X \cup Y
\cup \{\ph\} \vdash \psi$ and $X \cup Y \cup \{\ph\} \vdash \psi$. Also, $
{} \vdash \ph$ is shorthand for $\emp \vdash \ph$.

We wish to define a relation $\vdash$ from $\fP \cF$ to $\cF$ such that $X
\vdash \ph$ captures what it means to say that $\ph$ can be logically deduced
from the formulas in $X$.

\begin{defn}\label{D:derivability}
  The \emph{derivability relation}
    \index{derivability relation!propositional ---}%
  is the smallest relation $\vdash$
    \symindex{$X \vdash \ph$ (in $\cF$)}%
  from $\fP \cF$ to $\cF$ such that, for all $\ph, \psi \in \cF$ and all
  countable $\Phi \subseteq \cF$, the following conditions hold:
  \begin{enumerate}[(i)]
    \item $\ph \vdash \ph$,
    \item if $X \vdash \ph$ and $X \subseteq X'$, then $X' \vdash \ph$,
    \item if $X \vdash \bigwedge \Phi$, then $X \vdash \th$ for all $\th \in
          \Phi$,
    \item if $X \vdash \th$ for all $\th \in \Phi$,
      		then $X \vdash \bigwedge \Phi$,
    \item if $X \vdash \ph$ and $X \vdash \neg\ph$, then $X \vdash \psi$, and
    \item if $X, \ph \vdash \psi$ and $X, \neg\ph \vdash \psi$,
      		then $X \vdash \psi$.
  \end{enumerate}
\end{defn}

When $X \vdash \ph$, we say that $\ph$ is \emph{(deductively) derivable} from
$X$, or that \emph{$X$ proves $\ph$}. If $X \nvdash \ph$ and $X \nvdash \neg
\ph$, then $\ph$ is \emph{undetermined by} (or \emph{deductively independent
of}) $X$. We say that $\ph$ is \emph{determined by $X$} if $\ph$ is not
undetermined by $X$.

\begin{rmk}
  Note that the intersection of any family of relations that satisfy (i)--(vi)
  also satisfies (i)--(vi). Also note that $\cF$ itself satisfies (i)--(vi).
  Thus, the derivability relation is well-defined and is equal to the
  intersection of all subsets of $\fP \cF \times \cF$ satisfying (i)--(vi).
\end{rmk}

\begin{rmk}
  An alternative but equivalent definition of the derivability relation involves
  the notion of a derivation. A derivation of $\ph$ from $X$ is a countable
  sequence $\ang{(X_\be, \ph_\be): \be \le \al}$, where $\al$ is a countable
  ordinal, $(X_\al, \ph_\al)=(X, \ph)$, and for each $\be \le \al$, the term $
  (X_\be, \ph_\be)$ is obtained from $\ang{(X_\xi, \ph_\xi): \xi < \be}$ by an
  application of one of the six rules in Definition \ref{D:derivability}. In
  this case, $X \vdash \ph$ if and only if there exists a derivation of $\ph$
  from $X$.
\end{rmk}

\begin{defn}
  The \emph{finitary derivability relation} is the smallest relation
  $\vdash_\fin$
    \symindex{$X \vdash_\fin \ph$ (in $\cF$)}%
  from $\fP \cF_\fin$ to $\cF_\fin$ such that, for all $\ph, \psi \in \cF_\fin$
  and all finite $\Phi \subseteq \cF$, conditions (i)--(vi) from Definition
  \ref{D:derivability} hold.
\end{defn}

\begin{rmk}\label{R:finitary-vs-infinitary}
  The finitary derivability relation is a typical natural-deduction calculus for
  finitary propositional logic. Clearly, ${\vdash_\fin} \subseteq {\vdash}$. As
  we will see in Proposition \ref{P:finitary-vs-infinitary}, if $X \subseteq
  \cF_\fin$, $\ph \in \cF_\fin$, and $X \vdash \ph$, then $X \vdash_\fin \ph$.
  In other words, when restricted to finitary formulas, our infinitary calculus
  cannot produce any new inferences beyond those already available with the
  finitary calculus.
\end{rmk}

The proofs of the structural rules in the following proposition are exactly the
same as in finitary propositional logic (see, for instance, \cite[Section
1.4]{Rautenberg2010}).

\begin{prop}\label{P:derivability}
  The derivability relation satisfies the following:
  \begin{enumerate}[(a)]
  	\item $X \vdash (\ph \to \psi)$ if and only if $X, \ph \vdash \psi$,
    \item if $X \vdash \ph$ and $X, \ph \vdash \psi$, then $X \vdash \psi$,
    \item if $X, \neg \ph \vdash \ph$, then $X \vdash \ph$, and
    \item if $X, \ph \vdash \neg \ph$, then $X \vdash \neg \ph$.
  \end{enumerate}
\end{prop}

In the remainder of this section, unless otherwise specified, lowercase Roman
numerals refer to Definition \ref{D:derivability} and letters refer to
Proposition \ref{P:derivability}.

\begin{lemma}\label{L:conj-subset}
  Let $\Phi, \Phi' \subseteq \cF$ be countable and $\psi \in \cF$. If
  $\Phi \subseteq \Phi'$ and $\bigwedge \Phi \vdash \psi$, then $\bigwedge \Phi'
  \vdash \psi$.
\end{lemma}

\begin{proof}
  Suppose $\Phi \subseteq \Phi'$ and $\bigwedge \Phi \vdash \psi$. By (i) and
  (iii), we have $\bigwedge \Phi' \vdash \th$ for all $\th \in \Phi$. Thus, by
  (iv), we have $\bigwedge \Phi' \vdash \bigwedge \Phi$. The result now follows
  from (b).
\end{proof}

\begin{thm}[$\si$-compactness]\label{T:sig-cpctness}
    \index{s_sigma-compactness@$\si$-compactness}%
  Let $X \subseteq \cF$ and $\ph \in \cF$. Then $X \vdash \ph$ if and only if
  there exists a countable subset $X_0 \subseteq X$ such that $X_0 \vdash \ph$.
\end{thm}

\begin{proof}
  We will actually prove that the following are equivalent:
  \begin{align}
    &X \vdash \ph,\label{sig-cpctness-1}\\
    &\text{there exists countable $X_0 \subseteq X$ such that
      $\ts{\bigwedge} X_0 \vdash \ph$, and}\label{sig-cpctness-2}\\
    &\text{there exists countable $X_0 \subseteq X$ such that
      $X_0 \vdash \ph$}.\label{sig-cpctness-3}
  \end{align}
  By (i), (ii), (iv), and (b), we have that \eqref{sig-cpctness-2} implies \eqref
  {sig-cpctness-3}, and by (ii), we have \eqref{sig-cpctness-3} implies 
  \eqref{sig-cpctness-1}.

  To prove that \eqref{sig-cpctness-1} implies \eqref{sig-cpctness-2}, we define
  $\vdash'$ so that $X \vdash' \ph$ if and only if $X \vdash \ph$ and 
  \eqref{sig-cpctness-2} holds. The proof will be complete once we show that
  (i)--(vi) still hold when $\vdash$ is replaced by $\vdash'$.

  Clearly, (i)--(iii) hold for $\vdash'$. To see that (iv)--(vi) hold for
  $\vdash'$, use Lemma \ref{L:conj-subset} and the fact that a countable union
  of countable sets is countable.
\end{proof}

\begin{prop}\label{P:conj-equiv}
  Let $\Phi \subseteq \cF$ be countable and $\psi \in \cF$. Then $\Phi \vdash
  \psi$ if and only if $\bigwedge \Phi \vdash \psi$.
\end{prop}

\begin{proof}
  By (i), (ii), and (iv), we have $\Phi \vdash \bigwedge \Phi$. Thus, by (b), it
  follows that $\bigwedge \Phi \vdash \psi$ implies $\Phi \vdash \psi$.

  Now suppose $\Phi \vdash \psi$. By \eqref{sig-cpctness-2}, there exists
  countable $\Phi' \subseteq \Phi$ such that $\bigwedge \Phi' \vdash \psi$. By
  Lemma \ref{L:conj-subset}, we have $\bigwedge \Phi \vdash \psi$.
\end{proof}

\begin{prop}\label{P:set-mono}
  If $X \vdash Y$ and $Y \vdash \ph$, then $X \vdash \ph$.
\end{prop}

\begin{proof}
  Suppose $X \vdash Y$ and $Y \vdash \ph$. By $\si$-compactness, there exists
  countable $Y_0 \subseteq Y$ such that $Y_0 \vdash \ph$. By Proposition 
  \ref{P:conj-equiv}, we have $\bigwedge Y_0 \vdash \ph$. By (iv), we have $X
  \vdash \bigwedge Y_0$. Hence, by (b), we have $X \vdash \ph$.
\end{proof}

A set $X \subseteq \cF$ is \emph{inconsistent}
  \index{inconsistent}%
if $X \vdash \ph$ for all $\ph \in \cF$; it is otherwise \emph{consistent}.
  \index{consistent}%
If $X$ is inconsistent, then $X \vdash \bot$. Conversely, by the definition of
$\bot$, and by (iii) and (v), we have that $X \vdash \bot$ implies $X$ is
inconsistent. Thus, $X$ is inconsistent if and only if $X \vdash \bot$. The
derivability relation can, in fact, be characterized in terms of consistency.

\begin{thm}\label{T:deduc-con}
  Suppose $X \subseteq \cF$ and $\ph \in \cF$. Then $X \vdash \ph$ if and only
  if $X, \neg \ph \vdash \bot$, and $X \vdash \neg \ph$ if and only if $X, \ph
  \vdash \bot$.
\end{thm}

\begin{proof}
  Suppose $X \vdash \ph$. By (i) and (ii), we have $X, \neg \ph \vdash \ph$ and
  $X, \neg \ph \vdash \neg \ph$. By (v), this implies $X, \neg \ph \vdash \bot$.
	Conversely, suppose $X, \neg \ph \vdash \bot$. Then $X \cup \{\neg \ph\}$ is
  inconsistent, so that $X, \neg \ph \vdash \ph$. By (c), we have $X
  \vdash \ph$. The proof of the second biconditional is analogous.
\end{proof}

Since $\top = \neg \bot$, the preceding theorem shows that $X \vdash \top$ if
and only if $X, \bot \vdash \bot$. Hence, by (i) and (ii), we have $X \vdash
\top$ for all $X \subseteq \cF$, which by (ii) is equivalent to ${} \vdash
\top$.

A formula $\ph$ is a \emph{tautology}
  \index{tautology}%
if ${} \vdash \ph$; it is a \emph{contradiction}
  \index{contradiction}%
if $\{\ph\}$ is inconsistent. By the preceding proposition, we see that $\ph$ is
a tautology if and only if $\neg \ph$ is a contradiction, and vice versa. The
set of tautologies is denoted by $\Taut$, or $\Taut_\cF$.
  \symindex{$\Taut_\cF$}%

\begin{prop}\label{P:sig-cpctness}
  Let $X \subseteq \cF$ and $\ph \in \cF$. Then $X \vdash \ph$ if and only if
  there exists a countable $X_0 \subseteq X$ such that $\bigwedge X_0 \to \ph
  \in \Taut$.
\end{prop}

\begin{proof}
  By $\si$-compactness, we have $X \vdash \ph$ if and only if there exists
  countable $X_0 \subseteq X$ such that $X_0 \vdash \ph$. And by Proposition 
  \ref{P:conj-equiv} and (a), we have $X_0 \vdash \ph$ if and only if $\bigwedge
  X_0 \to \ph \in \Taut$.
\end{proof}

\subsection{A Hilbert-type calculus}\label{S:Hilbert-calc}

Let $\La = \La_\cF$
  \symindex{$\La_\cF$}%
be the smallest subset of $\cF$ such that if $\ph, \psi, \ze \in \cF$ and $\Phi
\subseteq \cF$ is countable with $\ph \in \Phi$, then the following formulas are
in $\La$:
\begin{enumerate}[($\La$1)]
  \item $(\ph \to \psi \to \ze) \to (\ph \to \psi) \to \ph \to \ze$
  \item $(\ph \to \neg \psi) \to \psi \to \neg \ph$
  \item $\bigwedge \Phi \to \ph$
\end{enumerate}
The formulas in $\La$ are called \emph{axioms}.
  \index{axiom}%

We define a \emph{proof of $\ph \in \cF$ from $X \subseteq \cF$}
  \index{proof}%
as an $(\al + 1)$-sequence of formulas, $\ang{\ph_\be \mid \be \le \al}$, where
$\al$ is a countable ordinal, $\ph_\al = \ph$, and for each $\be \le \al$,
either $\ph_\be \in X \cup \La$, or there exist $i, j < \be$ such that $\ph_i =
(\ph_j \to \ph_\be)$, or there exists nonempty, countable $\Phi \subseteq
\{\ph_\xi \mid \xi < \be\}$ such that $\ph_\be = \bigwedge \Phi$. Note that if $
\ang{\ph_\be \mid \be \le \al}$ is a proof of $\ph_\al$ from $X$, then for any
$\be < \al$, it follows that $\ang{\ph_\xi \mid \xi \le \be}$ is a proof of
$\ph_\be$ from $X$. For $\ph \in \cF$ and $X \subseteq \cF$, define $X \wdash
\ph$ to mean there is a proof of $\ph$ from $X$.
  \symindex{$\wdash_\cF$}

Let $\La_\fin$ be the smallest subset of $\cF_\fin$ such that if $\ph, \psi, \ze
\in \cF_\fin$ and $\Phi \subseteq \cF_\fin$ is finite with $\ph \in \Phi$, then
($\La1$)--($\La$3) are in $\La_\fin$. A \emph{finitary proof of $\ph \in
\cF_\fin$ from $X \subseteq \cF_\fin$} is a finite sequence of formulas, $
\ang{\ph_k \mid k \le n}$, where $\ph_n = \ph$, and for each $k \le n$, either
$\ph_k \in X \cup \La_\fin$, or there exist $i, j < k$ such that $\ph_i = (\ph_j
\to \ph_k)$, or there exists nonempty, finite $\Phi \subseteq \{\ph_\ell \mid
\ell < k\}$ such that $\ph_k = \bigwedge \Phi$. For $\ph \in \cF_\fin$ and $X
\subseteq \cF_\fin$, define $X \wdash_\fin \ph$ to mean there is a finitary
proof of $\ph$ from $X$.

A finitary proof is the classical notion of proof. It is finitely long, and each
sentence in it has finite length. An infinitary proof, on the other hand, can be
infinitely long. And individual sentences in such a proof can themselves have
infinite length.

\begin{rmk}\label{R:fin-Hilbert}
  The relation $\wdash_\fin$ is a typical Hilbert-style calculus for finitary
  propositional logic. It is well-known that ${\wdash_\fin} = {\vdash_\fin}$.
  (See, for example, \cite[Theorem 1.6.6]{Rautenberg2010}.) In Theorem
  \ref{T:Hilbert=nat}, we will see that ${\wdash} = {\vdash}$. Hence, according
  to Remark \ref{R:finitary-vs-infinitary}, if $X \subseteq \cF_\fin$, $\ph \in
  \cF_\fin$, and $X \wdash \ph$, then $X \wdash_\fin \ph$. In other words,
  if we can find an infinitary proof of $\ph$ from $X$, then a finitary proof
  necessarily exists.
\end{rmk}

\begin{prop}[Induction principle for $\mathrel{|\hspace{-0.42em}\sim}$]
\label{P:Hilbert-induc}
  The relation $\wdash$ is the smallest relation from $\fP \cF$ to $\cF$
  such that if $X \subseteq \cF$, $\ph, \psi \in \cF$, and $\Phi \subseteq \cF$
  is countable, then
  \begin{enumerate}[(1)]
    \item $X \wdash \th$ for all $\th \in X \cup \La$,
    \item if $X \wdash (\ph \to \psi)$ and $X \wdash \ph$,
      		then $X \wdash \psi$, and
    \item if $X \wdash \th$ for all $\th \in \Phi$,
    		  then $X \wdash \bigwedge \Phi$.
  \end{enumerate}
\end{prop}

\begin{proof}
  The fact that $\wdash$ satisfies (1)--(3) follows from the fact that a
  countable concatenation of proofs is again a proof. (See \cite[Chapter 2]
  {Karp1964} for details on infinitary concatenation.)

  Let $\rtri$ be a relation from $\fP \cF$ to $\cF$ satisfying (1)--(3). Let
  $X \subseteq \cF$ be arbitrary. Fix a countable ordinal $\al$, and consider
  the statement,
  \begin{quote}
    for all $\ph \in \cF$, if there exists a proof of $\ph$ from $X$ with length
    $\al + 1$, then $X \rtri \ph$.
  \end{quote}
  We will prove this statement is true for all countable $\al$ by induction on
  $\al$, and this will show that $X \wdash \ph$ implies $X \rtri \ph$.

  If $\ph$ has a proof from $X$ of length 1, then it must be $\ang{\ph}$,
  implying $\ph \in X \cup \La$. Therefore, $X \rtri \ph$ by (1), and the
  statement is true for $\al=0$.

  Suppose the statement is true for all $\be < \al$, and that $\ph$ has a proof
  from $X$ of length $\al + 1$. Let $\ang{\ph_\be \mid \be \le \al}$ be such a
  proof. If $\ph \in X \cup \La$, then $X \rtri \ph$ by (1). Suppose there exists
  $i,j < \al$ such that $\ph_i = (\ph_j \to \ph)$. Then $\ang{\ph_\be \mid \be
  \le i}$ and $\ang{\ph_\be \mid \be \le j}$ are proofs of $\ph_i$ and $\ph_j$,
  respectively, each with length less than $\al + 1$. By the inductive
  hypothesis, $X \rtri \ph_i$ and $X \rtri \ph_j$, so that by (2), we have $X
  \rtri \ph$. Finally, suppose there exists $\Phi \subseteq \{\ph_\be: \be <
  \al\}$ such that $\ph = \bigwedge \Phi$. Each $\ph_\be \in \Phi$ has a proof,
  $ \ang{\ph_\xi \mid \xi \le \be}$, of length $\be + 1 < \al + 1$, so by the
  inductive hypothesis, $X \rtri \th$ for all $\th \in \Phi$. Thus, by (3), we
  have $X \rtri \ph$.
\end{proof}

\begin{thm}\label{T:Hilbert=nat}
  Let $X \subseteq \cF$ and $\ph \in \cF$. Then $X \wdash \ph$ if and only if $X
  \vdash \ph$.
\end{thm}

\begin{proof}
  In this proof, Arabic numerals will refer to Proposition 
  \ref{P:Hilbert-induc}.

  We first prove that $X \wdash \ph$ implies $X \vdash \ph$. By Proposition 
  \ref{P:Hilbert-induc}, it suffices to show that (1)--(3) hold when $\wdash$ is
  replaced by $\vdash$.

  By (iv), we have that (3) holds for $\vdash$. Suppose $X \vdash (\ph \to
  \psi)$ and $X \vdash \ph$. By (a), we have $X, \ph \vdash \psi$, so by (b), it
  follows that $X \vdash \psi$. Thus, (2) holds for $\vdash$. By (i) and (ii),
	we have $X \vdash \th$ for all $\th \in X$. It remains only to show that
  $X \vdash \ph$ whenever $\ph$ is an axiom. By (ii), it suffices to show that
  ${} \vdash \ph$ whenever $\ph$ is an axiom.

  Consider first ($\La$1). Let $Y = \{\ph \to \psi \to \ze, \ph \to \psi,
  \ph\}$. By (i) and (ii), we have $Y \vdash (\ph \to \psi)$, so that (a) yields
  $Y, \ph \vdash \psi$. But $Y \cup \{\ph\} = Y$, so $Y \vdash \psi$. We
  similarly obtain $Y \vdash (\psi \to \ze)$, so that $Y, \psi \vdash \ze$. By 
  (b), we obtain $Y \vdash \ze$. Repeated applications of (a) now yield $\ph \to
  \psi \to \ze, \ph \to \psi \vdash \ph \to \ze$, followed by $\ph \to \psi \to
  \ze \vdash (\ph \to \psi) \to \ph \to \ze$, followed by ${} \vdash (\ph \to
  \psi \to \ze) \to (\ph \to \psi) \to \ph \to \ze$.

  For ($\La2$), let $X = \{\ph \to \neg\psi, \psi\}$ and $Y = X \cup \{\ph\}$.
  By (i) and (ii), we have $Y \vdash \ph \to \neg\psi$, so that (a) yields
  $Y, \ph \vdash \neg \psi$. But $Y \cup \{\ph\} = Y$, so $Y \vdash \neg \psi$.
  On the other hand, (i) and (ii) imply $Y \vdash \psi$, so by (v), we have $Y
  \vdash \neg \ph$, or in other words, $X, \ph \vdash \neg \ph$. By (d), we have
  $X \vdash \neg \ph$. Repeated applications of (a) yield $\ph \to \neg \psi
  \vdash \psi \to \neg \ph$, followed by ${} \vdash (\ph \to \neg\psi) \to \psi
  \to \neg \ph$.

  For ($\La3$), let $\Phi \subseteq \cF$ be countable and $\ph \in \Phi$. By 
  (i), we have $\bigwedge \Phi \vdash \bigwedge \Phi$. By (iii), we obtain
  $\bigwedge \Phi \vdash \ph$. Thus, by (a), it follows that ${} \vdash
  \bigwedge \Phi \to \ph$.

  To prove that $X \vdash \ph$ implies $X \wdash \ph$, it suffices to show that
  (i)--(vi) hold when $\vdash$ is replaced by $\wdash$. The fact that (i), (ii),
  (v), and (vi) hold for $\wdash$ follows exactly as in the finitary case
  (see \cite[Section 1.6]{Rautenberg2010}, for example). We obtain (iii) and
  (iv) from ($\La$3) and (3), respectively.
\end{proof}

\subsection{Deductive theories and logical equivalence}

\begin{defn}
	A set $T \subseteq \cF$ is called a \emph{(deductive) theory}
    \index{theory!deductive ---}%
  if the following conditions hold:
	\begin{enumerate}[(i)]
		\item $\La \subseteq T$,
		\item if $(\ph \to \psi) \in T$ and $\ph \in T$, then $\psi \in T$, and
		\item if $\Phi \subseteq T$ is countable, then $\bigwedge \Phi \in T$.
	\end{enumerate}
\end{defn}

Note that the intersection of any family of theories is again a theory. Also
note that $\cF$ itself is a theory. Hence, if $X \subseteq \cF$, then we may
define \emph{the (deductive) theory generated by $X$}, denoted by $T(X)$ or
$T_X$, as the smallest theory having $X$ as a subset.
  \symindex{$T(X), T_X$ (in $\cF$)}%

\begin{thm}\label{T:theory-deduc}
	Let $X \subseteq \cF$ and $\ph \in \cF$. Then $X \vdash \ph$ if and only if
	$\ph \in T(X)$.
\end{thm}

\begin{proof}
	Suppose $X \vdash \ph$. By Theorem \ref{T:Hilbert=nat}, there exists a proof
	of $\ph$ from $X$. As in Proposition \ref{P:Hilbert-induc}, we can use
	induction on the length of the proof to show that $\ph \in T(X)$. For the
	converse, define $T' = \{\ph \in \cF: X \vdash \ph\}$. By Theorem 
	\ref{T:Hilbert=nat} and Proposition \ref{P:Hilbert-induc}, it follows that
	$T'$ is a theory with $X \subseteq T'$. Thus, $T(X) \subseteq T'$.
\end{proof}

\begin{cor}
	A set $T \subseteq \cF$ is a theory if and only if it is deductively closed,
	meaning that $T \vdash \ph$ implies $\ph \in T$.
\end{cor}

\begin{proof}
	This follows immediately by combining Theorem \ref{T:theory-deduc} with the
  fact that $T \subseteq \cF$ is a theory if and only if $T = T(T)$.
\end{proof}

\begin{rmk}\label{R:theory-first}
	Theorem \ref{T:theory-deduc} exhibits an alternative approach to defining
	derivability. One can define the theory generated by a set $X$, as we did
	above, without any reference to derivability. Then one can define
	derivability in	terms of $T(X)$. This is the approach we will take when
	considering inductive inference.
\end{rmk}

If $T$ is a theory and $S \subseteq \cF$, then $T + S$
  \symindex{$T + S$}%
denotes the theory generated by $T \cup S$. For $\ph \in \cF$, we write $T +
\ph$ for $T + \{\ph\}$. The smallest theory is $\Taut$, the largest theory is
$\cF$, and every theory $T$ satisfies $\Taut \subseteq T \subseteq \cF$. A
theory $T$ is inconsistent if and only if $T = \cF$. A theory $T$ is said to be
\emph{(deductively) complete} if it is consistent and every $\ph \in \cF$ is
determined by $T$. That is, for every $\ph \in \cF$, either $\ph \in T$ or $\neg
\ph \in T$.

Formulas $\ph$ and $\psi$ are \emph{(logically) equivalent},
  \index{equivalent}%
  \symindex{$\equiv$}%
  \symindex{$\equiv_X$}%
written $\ph \equiv \psi$, if $\ph \vdash \psi$ and $\psi \vdash \ph$. By (a),
(iii), (iv), and the shorthand definition of $\tot$, we find that $\ph \equiv
\psi$ if and only if $\ph \tot \psi \in \Taut$. Note that $\ph \in \Taut$ if and
only if $\ph \equiv \top$. Also note that if $X \vdash \ph$ and $\ph \equiv
\psi$, then $X \vdash \psi$.

More generally, if $X \subseteq \cF$, we say that $\ph$ and $\psi$ are \emph
{equivalent given $X$}, written $\ph \equiv_X \psi$, if $X, \ph \vdash \psi$ and
$X, \psi \vdash \ph$. As above, we have $\ph \equiv_X \psi$ if and only if $\ph
\tot \psi \in T(X)$. Also, $\ph \in T(X)$ if and only if $\ph \equiv_X \top$.
Note that $\equiv_\emp$ is simply $\equiv$. Also note that ${\equiv_X} \subseteq
{\equiv_{X'}}$ whenever $X \subseteq X'$.

It can be shown that $\equiv_X$ is a \emph{congruence relation}, meaning it is
an equivalence relation on $\cF$ such that
\begin{align*}
  &\text{if $\ph \equiv_X \ph'$, then $\neg \ph \equiv_X \neg \ph'$, and}\\
  &\text{if $C$ is countable and $\ph_n \equiv_X \ph_n'$ for $n \in C$, then
  	$\ts{\bigwedge_{n \in C}} \; \ph_n \equiv_X \ts{\bigwedge_{n \in C}} \;
  	\ph_n'$.}
\end{align*}
For $X, Y \subseteq \cF$, we say that $X \equiv Y$ if $X \vdash Y$ and $Y \vdash
X$. Note that $X \vdash Y$ if and only if $Y \subseteq T(X)$, which holds if and
only if $T(Y) \subseteq T(X)$. Thus, $X \equiv Y$ if and only if $T(X) = T(Y)$.
Also note that if $X \vdash \ph$ and $X \equiv Y$, then $Y \vdash \ph$.

The operations, $\neg$, $\wedge$, and $\vee$, pass in the usual way from $\cF$
to $B(X) = \cF/{\equiv_X}$, making $B(X)$ into a Boolean $\si$-algebra, called
the \emph{Lindenbaum-Tarski $\si$-algebra of $X$}. If $[\ph]_X \in B(X)$ denotes
the equivalence class of $\ph$, then $\ph \equiv_X \psi$ if and only if $[\ph]_X
= [\psi]_X$. The partial order in $B(X)$ corresponds to the derivability
relation. That is, $[\ph]_X \le [\psi]_X$ if and only if $X, \ph \vdash \psi$.
In $B(X)$, we have $0 = [\bot]_X$ and $1 = [\top]_X$.

We end this section with two items that we will need later. The first is a
piece of notation. If $T_0$ and $T_1$ are theories with $T_0 \subseteq T_1$,
then we write $[T_0, T_1]$
  \symindex{$[T_0, T_1]$}%
to denote the set of theories $T$ that satisfy $T_0 \subseteq T \subseteq T_1$.
The second is the following lemma.

\begin{lemma}\label{L:omit-psi}
  Let $T$ be a theory, $\psi \in \cF$, and $S \subseteq \cF$. Define $S' = 
  \{\psi \to \th \mid \th \in S\}$. Then $T + \psi + S = T + \psi + S'$.
\end{lemma}

\begin{proof}
  For each $\th \in S$, we have $T + \psi + S \vdash \th \vdash \psi \to \th$.
  Thus, $T + \psi + S \vdash S'$, so that $T + \psi + S \vdash T + \psi + S'$.
  Conversely, for any $\th \in S$, we have $\psi \to \th, \psi \vdash \th$, so
  that $S', \psi \vdash S$. Hence, $T + \psi + S' \vdash T + \psi + S$.
\end{proof}

\section{Inductive statements and entire sets}\label{S:entire}

Let $\cF^\IS = \fP \cF \times \cF \times [0, 1]$.
  \symindex{$\cF^\IS$}%
The elements, $(X, \ph, p)$,
  \symindex{$(X, \ph, p)$}%
of $\cF^\IS$ are called \emph{inductive statements}.
  \index{inductive statement}%
Intuitively, we interpret $(X, \ph, p)$ as asserting that $X$ partially entails
$\ph$, and that $p$ is the degree of this partial entailment. In an inductive
statement, $X$ is called the \emph{antecedent},
  \index{antecedent}%
$\ph$ is called the \emph{consequent},
  \index{consequent}%
and $p$ is called the \emph{probability}.
  \index{probability}%

The remainder of this chapter is devoted to extending the derivability relation,
$\vdash$, to inductive statements. Informally, the assertion, $Q \vdash (X, \ph,
p)$, where $Q \subseteq \cF^\IS$, means that $(X, \ph, p)$ can be derived from
$Q$, using the rules of inductive inference. We will have nine such rules. They
are:
\begin{enumerate}[(R1)]
  \item the rule of logical equivalence,
  \item the rule of logical implication,
  \item the rule of material implication,
  \item the rule of deductive transitivity,
  \item the addition rule,
  \item the multiplication rule,
  \item the continuity rule,
  \item the rule of inductive extension, and
  \item the rule of deductive extension.
\end{enumerate}
The first rule, among other things, ensures that our collective inferences form
a function mapping antecedent-consequent pairs, $(X, \ph)$, to probabilities
$p$. Rules (R2)--(R4) describe the relationship between deductive and inductive
inference. Rules (R5)--(R7) are the usual mathematical rules for working with
probabilistic assertions. And the final two rules provide a natural
``completeness'' to our inferences.

We follow the approach outlined in Remark \ref{R:theory-first}. That is, we
begin by defining an inductive theory, which will be a set of inductive
statements that is closed under the nine rules of inductive inference, and
satisfies certain connectivity requirements. This will allow us to speak of the
inductive theory generated by a set $Q \subseteq \cF^\IS$, which we denote by
$\bfP(Q)$. We then take $Q \vdash (X, \ph, p)$ to mean that $(X, \ph, p)
\in \bfP(Q)$.

The notion of being closed under the nine rules of inductive inference will be
built up in tiers. An admissible set is one that is closed under the first rule.
An entire set is closed under the first seven rules. A semi-closed set is closed
under the first eight rules. And a closed set is closed under all nine. In this
section, we focus only on entire sets.

\subsection{Seven of nine}

We now formally state the first seven of the nine rules of inductive inference.

A set $P \subseteq \cF^\IS$ is \emph{admissible}
  \index{admissible}%
if it satisfies the \emph{rule of logical equivalence}:
  \index{rule!of logical equivalence}%
\begin{enumerate}[(R1)]
  \item If $(X, \ph, p) \in P$, $X' \equiv X$, and $\ph' \equiv_X \ph$, then $ 
        (X', \ph', p) \in P$ and there is no other value $p'$ such that $(X',
        \ph', p') \in P$. 
\end{enumerate}
If $P$ is admissible, then it is a function from $\fP \cF \times \cF$ to $[0,
1]$. In this case, we write $P(\ph \mid X) = p$ to mean that $(X, \ph, p) \in
P$, and read the left-hand side, $P(\ph \mid X)$, as \emph{the probability of
$\ph$ given $X$}. We also write $X, \psi$ as shorthand for $X \cup \{\psi\}$, so
that $P(\ph \mid X, \psi)$ means $P(\ph \mid X \cup \{\psi\})$. When $X = \emp$,
we will omit it, leaving only $P(\ph)$ or $P(\ph \mid \psi)$. For admissible
$P$, if $(X, \ph, p) \in P$, then we say $P(\ph \mid X)$ \emph{exists} or is 
\emph{defined}.
  \symindex{$P(\ph \mid X) = p$}%

Note that any subset of an admissible set is also a function from $\fP \cF
\times \cF$ to $[0, 1]$. We will therefore also use the notation $P(\ph \mid X)
= p$ for subsets of admissible sets.

If $P \subseteq \cF^\IS$, we define
\[
  \ante P = \{X \subseteq \cF:
    (X, \ph, p) \in P \text{ for some $\ph \in \cF$ and $p \in [0, 1]$}
  \}.
\]
  \symindex{$\ante P$}%
That is, $X \in \ante P$ if and only if $X$ is the antecedent of some inductive
statement in $P$.

The next six rules of inductive inference are encoded in the following
definition.

\begin{defn}
  A set $P \subseteq \cF^\IS$ is \emph{entire}
    \index{entire}%
  if it is admissible and satisfies the following:
  \begin{enumerate}[(R1)]
    \setcounter{enumi}{1}
    \item (\emph{the rule of logical implication})
            \index{rule!of logical implication}%
          If $X \in \ante P$ and $X \vdash \ph$, then $P(\ph \mid X) = 1$.

    \item (\emph{the rule of material implication})
            \index{rule!of material implication}%
          If $X \in \ante P$ and $P(\psi \mid X, \ph) = 1$, then $P(\ph \to \psi
          \mid X) = 1$.

    \item (\emph{the rule of deductive transitivity})
            \index{rule!of deductive transitivity}%
          If $P(\ph \mid X) = 1$ and $\ph \vdash \psi$, then $P(\psi \mid X) =
          1$. Also, for any $X' \in \ante P$, if $X' \vdash X$ and $P(\ph \mid
          X) = 1$, then $P(\ph \mid X') = 1$.

    \item (\emph{the addition rule})
            \index{rule!the addition ---}%
            \index{addition rule|see {rule}}%
          Let $X \vdash \neg(\ph \wedge \psi)$. Consider the equation,
          \begin{equation}\label{add-rule}
            P(\ph \vee \psi \mid X) = P(\ph \mid X) + P(\psi \mid X).
          \end{equation}
          If two of the above probabilities exist, then so does the third and
          \eqref{add-rule} holds.

    \item (\emph{the multiplication rule})
            \index{rule!the multiplication ---}%
            \index{multiplication rule|see {rule}}%
          Consider the equation,
          \begin{equation}\label{mult-rule}
            P(\ph \wedge \psi \mid X) = P(\ph \mid X) P(\psi \mid X, \ph).
          \end{equation}
          If two of the above probabilities exist and are positive, then the
          third exists and \eqref{mult-rule} holds.

    \item (\emph{the continuity rule})
            \index{rule!the continuity ---}%
            \index{continuity rule|see {rule}}%
          If $P(\ph_n \mid X)$ exists and $X, \ph_n \vdash \ph_{n + 1}$ for all
          $n \in \bN$, then
          \begin{equation}\label{cont-rule}
            \ts{P(\bigvee_n \ph_n \mid X) = \lim_n P(\ph_n \mid X).}
          \end{equation}
    \end{enumerate}
\end{defn}

\begin{rmk}
  If $P$ is entire and $X \in \ante P$, then $X$ is consistent. To see this,
  suppose $X$ is inconsistent. Choose $\ph \in \cF$ such that $X \vdash \ph$ and
  $X \vdash \neg \ph$. By the rule of logical implication $P(\ph \mid X) = 1$
  and $P(\neg \ph \mid X) = 1$. Thus, by the addition rule, $P(\ph \vee \neg
  \ph) = 2$, which violates the definition of an inductive statement.
\end{rmk}

\begin{rmk}
  The first seven rules of inductive inference leave open the question of
  whether $P(\ph \mid X) = 1$ implies $X \vdash \ph$. In general, it does not, 
  but a partial converse to the rule of logical implication will be given in
  Theorem \ref{T:log-impl-iff}.
\end{rmk}

\subsection{Relative negation and certainty}\label{S:rel-neg-cert}

Given an entire set $P$ and a set $X \in \ante P$, the domain of $P(\; \cdot
\mid X)$ is not necessarily closed under conjunctions and disjunctions. It is,
however, closed under relative negation. Also, conjunctions and disjunctions
with a formula whose probability is 0 or 1 are still in the domain of $P(\;
\cdot \mid X)$. These and related facts are described in this subsection.

\begin{expl}\label{Expl:MathSEexpl}
  The possible failure of the domain of $P(\; \cdot \mid X)$ to be closed under
  conjunctions and disjunctions can be seen in the following simple example,
  using $X = \emp$. Let $PV = \{\bfr_1, \bfr_2\}$. Fix $q \in (0, 1)$ and define
  $Q \subseteq \cF^\IS$ by $Q(\bfr_1) = Q(\bfr_2) = Q(\bfr_1 \tot \bfr_2) = q$.
  In Example \ref{Expl:MathSE}, we construct an entire set $P$ such that $Q
  \subseteq P$, but in which $P (\bfr_1 \wedge \bfr_2)$ is undefined. As we will
  see in Theorem \ref{T:incl-excl} below, this also implies $P(\bfr_1 \vee
  \bfr_2)$ is undefined.
\end{expl}

\begin{prop}\label{P:rel-neg}
  Let $P$ be entire. If $P(\ph \mid X)$ and $P(\psi \mid X)$ exist and $X, \ph
  \vdash \psi$, then $P(\psi \wedge \neg \ph \mid X)$ exists and
  \begin{equation}\label{rel-neg}
    P(\psi \wedge \neg \ph \mid X) = P(\psi \mid X) - P(\ph \mid X).
  \end{equation}
  In particular, $P(\ph \mid X) \le P(\psi \mid X)$.
\end{prop}

\begin{proof}
  Let $\psi' = \psi \wedge \neg \ph$. Since $\psi \equiv_X \ph \vee \psi'$, the
  rule of logical equivalence implies $P(\ph \vee \psi' \mid X)$ exists. Since
  $\ph \wedge \psi'$ is a contradiction, the addition rule implies that $P(\psi'
  \mid X)$ exists and
  \[
    P(\ph \vee \psi' \mid X) = P(\ph \mid X) + P(\psi' \mid X),
  \]
  which gives \eqref{rel-neg}.
\end{proof}

\begin{rmk}
  The final conclusion of Proposition \ref{P:rel-neg} is referred to as the
  \emph{monotonicity property} of $P$.
\end{rmk}

\begin{cor}\label{C:rel-neg}
  Let $P$ be entire. If $P(\ph \mid X)$ exists, then $P(\neg \ph \mid X)$ exists
  and
  \begin{equation}\label{prob-neg}
    P(\neg \ph \mid X) = 1 - P(\ph \mid X).
  \end{equation}
\end{cor}

\begin{proof}
  Suppose $P(\ph \mid X)$ exists. Then $X \in \ante P$, so by the rule of
  logical implication, $P(\top \mid X) = 1$. Applying Proposition
  \ref{P:rel-neg} with $\psi = \top$, and using $\top \wedge \neg \ph \equiv
  \neg \ph$ together with the rule of logical equivalence, we obtain
  \eqref{prob-neg}.
\end{proof}

\begin{rmk}\label{R:rel-neg}
  Proposition \ref{P:rel-neg} requires neither the multiplication rule nor 
  the continuity rule in its proof. Consequently, Corollary \ref{C:rel-neg}
  also does not require them.
\end{rmk}

\begin{prop}\label{P:cert-cl-conv}
  Let $P$ be entire. If $P(\psi \mid X) = 1$ and $P(\ph \wedge \psi \mid X)$
  exists, then $P(\ph \mid X) = P(\ph \wedge \psi \mid X)$.
\end{prop}

\begin{proof}
  Since $\psi \vdash \neg \ph \vee \psi$, the rule of deductive transitivity
  gives $P(\neg \ph \vee \psi \mid X) = 1$. By Corollary \ref{C:rel-neg} and the
  rule of logical equivalence, $P(\ph \wedge \neg \psi \mid X) = 0$. Hence, the
  result follows from the addition rule and the rule of logical equivalence.
\end{proof}

\begin{lemma}\label{L:cond-exist}
  Let $P$ be entire and suppose $P(\ph \mid X)$ exists. Then $X \cup \{\ph\} \in
  \ante P$ if and only if $P(\ph \mid X) > 0$.
\end{lemma}

\begin{proof}
  Suppose $P(\ph \mid X) > 0$. Applying the multiplication rule with $\psi =
  \ph$, and using $\ph \wedge \ph \equiv \ph$ together with the rule of logical
  equivalence, we get $P(\ph \mid X, \ph) = 1$, which implies $X \cup \{\ph\}
  \in \ante P$.

  Now suppose $P(\ph \mid X) = 0$ and $X \cup \{\ph\} \in \ante P$. Then 
  \eqref{prob-neg} implies $P(\neg \ph \mid X) = 1$. Since $X, \ph \vdash X$,
  the rule of deductive transitivity gives $P(\neg \ph \mid X, \ph) = 1$. Thus,
  again by \eqref{prob-neg}, we have $P(\ph \mid X, \ph) = 0$. But this
  contradicts the rule of logical implication.
\end{proof}

\begin{prop}\label{P:certainty-closure}
  Let $P$ be entire and suppose both $P(\ph \mid X)$ and $P(\psi \mid X)$ exist.
  If $P(\ph \mid X) \in \{0, 1\}$, then both $P(\ph \vee \psi \mid X)$ and $P
  (\ph \wedge \psi \mid X)$ exist, and
  \begin{align*}
    P(\ph \vee \psi \mid X) &= \max\{P(\ph \mid X), P(\psi \mid X)\},\\
    P(\ph \wedge \psi \mid X) &= \min\{P(\ph \mid X), P(\psi \mid X)\},
  \end{align*}
\end{prop}

\begin{proof}
  By \eqref{prob-neg} and the rule of logical equivalence, it suffices to
  consider the case $\ph \wedge \psi$. Suppose $P(\ph \mid X) = 0$. By 
  \eqref{prob-neg}, we have $P(\neg \ph \mid X) = 1$. Since $\neg \ph \vdash
  \neg \ph \vee \neg \psi$, the rule of deductive transitivity gives $P(\neg \ph
  \vee \neg \psi \mid X) = 1$. By \eqref{prob-neg} and the rule of logical
  equivalence, $P(\ph \wedge \psi \mid X) = 0$, proving the claim.

  Now suppose $P(\ph \mid X) = 1$. If $P(\psi \mid X) = 0$, then we are done by
  the previous case. Assume then that $P(\psi \mid X) > 0$. By Lemma 
  \ref{L:cond-exist}, we have $X \cup \{\psi\} \in \ante P$. Since $X, \psi
  \vdash X$, the rule of deductive transitivity gives $P(\ph \mid X, \psi) = 1$.
  By the multiplication rule, $P(\psi \wedge \ph \mid X) = P(\psi \mid X)$.
\end{proof}

\begin{cor}\label{C:certainty-closure}
  Let $P$ be entire and suppose that $P(\ph_n \mid X) = 1$ for all $n \in \bN$.
  Then $P(\bigwedge_n \ph_n \mid X) = 1$.
\end{cor}

\begin{proof}
  By Proposition \ref{P:certainty-closure} and induction, we have $P
  (\bigwedge_1^n \ph_j \mid X) = 1$ for all $n$. By \eqref{prob-neg} and the
  rule of logical equivalence, $P(\bigvee_1^n \neg \ph_j \mid X) = 0$. The
  continuity rule then gives $P(\bigvee_n \neg \ph_n \mid X) = 0$, which implies
  $P(\bigwedge_n \ph_n \mid X) = 1$.
\end{proof}

For $Q \subseteq \cF^\IS$ and $\cX \subseteq \ante Q$, let
\begin{equation}\label{tau-Q-X}
  \tau(Q; \cX) = \{
    \th \in \cF \mid (X, \th, 1) \in Q \text{ for all } X \in \cX
  \}.
\end{equation}
  \symindex{$\tau(Q; \cX), \tau(Q), \tau_Q$}%
For $X \in \ante Q$, we write $\tau(Q; X)$ for $\tau(Q; \{X\})$. We also write
$\tau(Q)$, or $\tau_Q$, for $\tau(Q; \ante Q)$. Informally, $\tau(Q)$ is the set
of all formulas that are true under $Q$, regardless of the antecedent used.

\begin{prop}\label{P:tau-ded-theory}
  If $P$ is entire and $\cX \subseteq \ante P$, then $\tau(P; \cX)$ is a
  deductive theory.
\end{prop}

\begin{proof}
  Suppose $\tau(P; \cX) \vdash \ph$. By $\si$-compactness, choose countable
  $\Phi \subseteq \tau(P; \cX)$ such that $\Phi \vdash \ph$. By Corollary
  \ref{C:certainty-closure}, we have $\bigwedge \Phi \in \tau(P; \cX)$. Now let
  $X \in \cX$. Then $P(\bigwedge \Phi \mid X) = 1$ and $\bigwedge \Phi \vdash
  \ph$. By deductive transitivity, $P(\ph \mid X) = 1$. Hence, $\ph \in \tau(P;
  \cX)$, so $\tau(P; \cX)$ is a deductive theory.
\end{proof}

\subsection{Inductive vs.~deductive inference}

Rules (R1)--(R4), the rules of logical equivalence, logical implication,
material implication, and deductive transitivity, describe the relationship
between inductive and deductive inference. The next three results provide useful
generalizations of these rules.

\begin{prop}\label{P:log-equiv-gen}
  Let $P$ be entire. If $P(\ph \mid X) = p$ and $P(\ph \tot \ph' \mid X) = 1$,
  then $P(\ph' \mid X) = p$.
\end{prop}

\begin{proof}
  Since $\ph \wedge (\ph \tot \ph') \equiv \ph \wedge \ph'$, the rule of logical
  equivalence together with Proposition \ref{P:certainty-closure} imply $P(\ph
  \wedge \ph' \mid X) = p$. Also, $\ph \tot \ph' \vdash \ph \vee \neg \ph'$, so
  by the rule of deductive transitivity, $P(\ph \vee \neg \ph' \mid X) = 1$.
  Thus, \eqref{prob-neg} implies $P(\neg \ph \wedge \ph' \mid X) = 0$. Hence, by
  the addition rule and the rule of logical equivalence, $P(\ph' \mid X) = p$.
\end{proof}

\begin{prop}\label{P:add-to-root}
  Let $P$ be entire. If $P(\ph \to \psi \mid X) = 1$ and $P(\ph \mid X) > 0$,
  then $P(\psi \mid X, \ph) = 1$.
\end{prop}

\begin{proof}
  Suppose $P(\ph \to \psi \mid X) = 1$ and $P(\ph \mid X) > 0$. By
  \eqref{prob-neg} and the rule of logical equivalence, $P(\ph \wedge \neg \psi
  \mid X) = 0$. Thus \eqref{rel-neg} implies $P(\ph \wedge \psi \mid X) = P(\ph
  \mid X) > 0$. Thus, by the multiplication rule, $P(\psi \mid X, \ph) = 1$.
\end{proof}

\begin{prop}\label{P:transitivity}
  Let $P$ be entire with $X \in \ante P$ and $P(\ph \mid X, \ze) = 1$. Assume at
  least one of the following holds:
  \begin{enumerate}[(i)]
    \item $X, \ph \vdash \psi$,
    \item $P(\psi \mid X, \ph) = 1$,
    \item $P(\ph \to \psi \mid X) = 1$.
  \end{enumerate}
  Then $P(\psi \mid X, \ze) = 1$.
\end{prop}

\begin{proof}
  Suppose $X \in \ante P$ and $P (\ph \mid X, \ze) = 1$. First note that (i) is
  equivalent to $X \vdash \ph \to \psi$. Thus, by the rule of logical
  implication, (i) implies (iii). Also, by the rule of material implication,
  (ii) implies (iii). Hence, we may assume (iii) holds. In this case, deductive
  transitivity gives $P(\ph \to \psi \mid X, \ze) = 1$. By \eqref{prob-neg} and
  the rule of logical equivalence, $P(\ph \wedge \neg \psi \mid X, \ze) = 0$.
  Thus, by \eqref{rel-neg}, we have $P(\ph \wedge \psi \mid X, \ze) = 1$.
  Finally, since $\ph \wedge \psi \vdash \psi$, deductive transitivity gives
  $P(\psi \mid X,\ze) = 1$.
\end{proof}

\subsection{Generalizations of the addition rule}
  \index{rule!the addition ---}%

\begin{lemma}\label{L:fin-add}
  Let $P$ be entire and suppose $X \vdash \neg (\ph_i \wedge \ph_j)$ whenever $i
  \ne j$. If $P(\ph_i \mid X)$ exists for $1 \le i \le n$, then $P(\bigvee_{i =
  1}^n \ph_i) = \sum_{i = 1}^n P(\ph_i \mid X)$.
\end{lemma}

\begin{proof}
  The case $n = 2$ is the addition rule. Suppose the lemma holds for some
  $n = k$ and consider the case $n = k + 1$. Note that $X \vdash \neg (\ph
  \wedge \psi)$ is equivalent to $X, \psi \vdash \neg \ph$. Thus, $X, \ph_{k +
  1} \vdash \neg \ph_i$ for all $i \le k$. This implies $X, \ph_{k + 1} \vdash
  \bigwedge_1^k \neg \ph_i \equiv \neg \bigvee_1^k \ph_i$. Therefore, $X \vdash
  \neg (\bigvee_1^k \ph_i \wedge \ph_{k + 1})$, so the addition rule and the
  inductive hypothesis show that the result holds for $n = k + 1$.
\end{proof}

\begin{thm}[Inclusion-exclusion]\label{T:incl-excl}
    \index{inclusion-exclusion}%
  Let $P$ be entire and consider the equation
  \begin{equation}\label{incl-excl}
    P(\ph \vee \psi \mid X)
      = P(\ph \mid X) + P(\psi \mid X) - P(\ph \wedge \psi \mid X).
  \end{equation}
  If three of the above probabilities exist, then so does the fourth and 
  \eqref{incl-excl} holds.
\end{thm}

\begin{proof}
  Let $\ze_1 = \ph \wedge \neg \psi$ and $\ze_2 = \neg \ph \wedge \psi$. Suppose
  three of the probabilities in \eqref{incl-excl} exist. There are four possible
  cases.

  The first case is that $P(\ph \mid X)$, $P(\psi \mid X)$, and $P(\ph \wedge
  \psi \mid X)$ exist. In this case, since $\ze_1 \equiv \ph \wedge \neg (\ph
  \wedge \psi)$, Proposition \ref{P:rel-neg} implies that $P(\ze_1 \mid X)$
  exists. Similarly, $P(\ze_2 \mid X)$ exists. Thus, by the addition rule and
  the fact that $\ph \vee \psi \equiv \ze_1 \vee \ze_2 \vee (\ph \wedge \psi)$,
  we have that $P(\ph \vee \psi \mid X)$ exists.

  The second case is that $P(\ph \vee \psi \mid X)$, $P(\ph \mid X)$, and $P
  (\psi \mid X)$ exist. In this case, since $\ze_1 \equiv (\ph \vee \psi) \wedge
  \neg \psi$, Proposition \ref{P:rel-neg} implies that $P(\ze_1 \mid X)$ exists.
  Similarly, $P(\ze_2 \mid X)$ exists. By the addition rule, Proposition
  \ref{P:rel-neg}, and the fact that $\ph \wedge \psi \equiv (\ph \vee \psi)
  \wedge \neg (\ze_1 \vee \ze_2)$, we have that $P (\ph \wedge \psi \mid X)$
  exists.

  By symmetry, the third and fourth cases are covered by the assumption that
  $P(\ph \vee \psi \mid X)$, $P(\ph \mid X)$, and $P(\ph \wedge \psi \mid X)$
  exist. The argument from the second case shows that $P(\ze_2 \mid X)$ exists.
  By the addition rule and the fact that $\psi \equiv (\ph \wedge \psi) \vee
  \ze_2$, we have that $P(\psi \mid X)$ exists.

  Hence, in all cases, all four probabilities in \eqref{incl-excl} exist. By
  Lemma \ref{L:fin-add},
  \begin{align*}
    P(\ph \mid X) &= P(\ze_1 \mid X) + P(\ph \wedge \psi \mid X),\\
    P(\psi \mid X) &= P(\ze_2 \mid X) + P(\ph \wedge \psi \mid X),\text{ and}\\
    P(\ph \vee \psi \mid X)
      &= P(\ze_1 \mid X) + P(\ze_2 \mid X) + P(\ph \wedge \psi \mid X).
  \end{align*}
  Putting these together yields \eqref{incl-excl}.
\end{proof}

\begin{rmk}
  By Theorem \ref{T:incl-excl}, if $P$ is entire, then in the addition rule, it
  is not necessary that $X \vdash \neg (\ph \wedge \psi)$. It is sufficient that
  $P(\ph \wedge \psi \mid X) = 0$.
\end{rmk}

\subsection{Generalizations of the multiplication rule}
  \index{rule!the multiplication ---}%

\begin{thm}
  If $P$ is entire, then in the multiplication rule, it is not necessary that
  the two defined probabilities be positive. It is enough to assume that solving
  for the third probability does not result in dividing by zero.
\end{thm}

\begin{proof}
  First suppose $P(\ph \mid X)$ and $P(\psi \mid X, \ph)$ both exist. By Lemma
  \ref{L:cond-exist}, we have $P(\ph \mid X) > 0$. Suppose $P (\psi \mid X, \ph)
  = 0$. Then $P(\neg \psi \mid X, \ph) = 1$, by \eqref{prob-neg}, so by the
  multiplication rule, $P(\ph \wedge \neg \psi \mid X) = P(\ph \mid X)$.
  Therefore, Proposition \ref{P:rel-neg} implies $P(\ph \wedge \psi \mid X) =
  0$.

  Next, suppose $P(\ph \mid X) > 0$ and $P(\ph \wedge \psi \mid X) = 0$. Then
  Proposition \ref{P:rel-neg} implies $P(\ph \wedge \neg \psi \mid X) = P(\ph
  \mid X) > 0$, so by the multiplication rule, $P(\neg \psi \mid X, \ph) = 1$.
  Applying Proposition \ref{P:rel-neg} again gives $P(\psi \mid X, \ph) = 0$.

  Finally, suppose $P(\psi \mid X, \ph) > 0$ and $P(\ph \wedge \psi \mid X)$
  exists. By Lemma \ref{L:cond-exist}, we have $X \cup \{\ph, \psi\} \in \ante
  P$. But $X \cup \{\ph, \psi\} \equiv X \cup \{\ph \wedge \psi\}$, so by the
  rule of logical equivalence, $X \cup \{\ph \wedge \psi\} \in \ante P$. Thus,
  by Lemma \ref{L:cond-exist}, we have $P(\ph \wedge \psi \mid X) > 0$.
\end{proof}

\begin{thm}[Bayes' theorem]
    \index{Bayes' theorem}%
  If $P$ is entire, then
  \begin{equation}\label{Bayes}
    P(\ph \mid X) P(\psi \mid X, \ph) = P(\psi \mid X) P(\ph \mid X, \psi),
  \end{equation}
  provided that all four of the above probabilities exist.
\end{thm}

\begin{proof}
  Since $\ph \wedge \psi \equiv \psi \wedge \ph$, this follows immediately from 
  the multiplication rule.
\end{proof}

\subsection{Generalizations of the continuity rule}
  \index{rule!the continuity ---}%

\begin{prop}
  Let $P$ be entire. If $P(\ph_n \mid X)$ exists and $X, \ph_{n + 1} \vdash \ph_
  n$ for all $n \in \bN$, then
   \[
    \ts{P(\bigwedge_n \ph_n \mid X) = \lim_n P(\ph_n \mid X).}
  \]
\end{prop}

\begin{proof}
  Let $\psi_n = \neg \ph_n$. By the continuity rule,
   \[
    \ts{P(\bigvee_n \psi_n \mid X) = \lim_n P(\psi_n \mid X).}
  \]
  But $\bigvee_n \psi_n \equiv \neg \bigwedge_n \ph_n$, so \eqref{prob-neg}
  gives the desired result.
\end{proof}

\begin{thm}\label{T:cont-rule}
  Let $P$ be entire and assume $P(\ph_n \mid X)$ exists for all $n$. If $P(\ph_
  {n + 1} \mid X, \ph_n) = 1$ for all $n$, then
  \begin{equation}\label{cont-from-below}
    \ts{P(\bigvee_n \ph_n \mid X) = \lim_n P(\ph_n \mid X).}
  \end{equation}
  Similarly, if $P(\ph_n \mid X, \ph_{n + 1}) = 1$ for all $n$, then
  \begin{equation}\label{cont-from-above}
    \ts{P(\bigwedge_n \ph_n \mid X) = \lim_n P(\ph_n \mid X).}
  \end{equation}
\end{thm}

\begin{proof}
  Suppose $P(\ph_{n + 1} \mid X, \ph_n) = 1$ for all $n$. By Lemma 
  \ref{L:cond-exist}, we have $P(\ph_n \mid X) > 0$ for all $n$. Let $\psi_n =
  \bigvee_{j = 1}^n \ph_j$. We first show that
  \begin{align}
      &P(\psi_n \mid X) = P(\ph_n \mid X) > 0, \text{ and} \label{cont-1}\\
      &P(\ph_{n + 1} \mid X, \psi_n) = 1 \label{cont-2},
  \end{align}
  for all $n$.

  Since $\psi_1 = \ph_1$, we have that \eqref{cont-1} and \eqref{cont-2} hold
  for $n = 1$. Suppose it is true for $n - 1$. Then
  \[
    \psi_n \equiv \psi_{n - 1} \vee \ph_n
      \equiv \ph_n \vee (\psi_{n - 1} \wedge \neg \ph_n).
  \]
  But $P(\ph_n \mid X, \psi_{n - 1}) = 1$, so it follows from the rule of
  material implication and \eqref{prob-neg} that $P(\psi_{n - 1} \wedge \neg
  \ph_n \mid X) = 0$. Hence, Proposition \ref{P:certainty-closure} gives us
  \eqref{cont-1}.

  By Proposition \ref{P:transitivity} and induction, we have $P(\ph_{n + 1} \mid
  X, \ph_m) = 1$ for all $m \le n$, which gives $P(\ph_m \to \ph_{n + 1} \mid
  X) = 1$ for all $m \le n$ by the rule of material implication. Note that
  $\bigwedge_{m = 1}^n (\ph_m \to \ph_{n + 1}) \equiv \psi_n \to \ph_{n + 1}$.
  Hence, by Proposition \ref{P:certainty-closure} and induction, we have $P
  (\psi_n \to \ph_{n + 1} \mid X) = 1$. Since $P(\psi_n \mid X) > 0$,
  Proposition \ref{P:add-to-root} gives \eqref{cont-2}.

  Having established \eqref{cont-1} and \eqref{cont-2} for all $n$, observe that
  $\psi_n \vdash \psi_{n + 1}$ for all $n$. Thus, by the continuity rule,
  \[
    \ts{P(\bigvee_n \psi_n \mid X) = \lim_n P(\psi_n \mid X).}
  \]
  Using \eqref{cont-1} and the fact that $\bigvee_n \psi_n \equiv \bigvee_n
  \ph_n$, we obtain \eqref{cont-from-below}.

  For \eqref{cont-from-above}, assume $P(\ph_n \mid X, \ph_{n + 1}) = 1$ for all
  $n$. Define $\psi_n = \neg \ph_n$. By the rule of material implication, $P
  (\ph_{n + 1} \to \ph_n \mid X) = 1$ for all $n$, which gives $P(\psi_n \to
  \psi_{n + 1} \mid X) = 1$ for all $n$.

  We first suppose that $P(\ph_n \mid X) < 1$ for all $n$. Since $P(\psi_n \mid
  X) > 0$ by \eqref{prob-neg}, Proposition \ref{P:add-to-root} implies $P(\psi_
  {n + 1} \mid X, \psi_n) = 1$ for all $n$. Applying \eqref{cont-from-below} and
  \eqref{prob-neg} gives \eqref{cont-from-above}.

  Now suppose that $P(\ph_n \mid X) = 1$ for some $n$. By Proposition 
  \ref{P:transitivity}, we have $P(\ph_j \mid X) = 1$ for all $j \le n$. Let
  $n_0 = \sup\{n: P(\ph_n \mid X) = 1\}$. If $n_0 = \infty$, then $P(\ph_n \mid
  X) = 1$ for all $n$. This implies $P(\psi_n \mid X) = 0$ for all $n$ by
  Proposition \ref{P:certainty-closure}. Hence, by Corollary 
  \ref{C:certainty-closure} and \eqref{prob-neg}, we have $P(\bigvee_n \psi_n
  \mid X) = 0$, and this establishes \eqref{cont-from-above}. Assume, then, that
  $n_0 < \infty$, so that $P(\ph_n \mid X) = 1$ for all $n \le n_0$ and $P(\ph_n
  \mid X) < 1$ for all $n > n_0$. By what we have already proven,
  \[
    \ts{P(\bigwedge_{n_0 + 1}^\infty \ph_n \mid X) = \lim_n P(\ph_n \mid X).}
  \]
  By Proposition \ref{P:certainty-closure} and induction, we have $P(\bigwedge_n
  \ph_n \mid X) = P(\bigwedge_{n_0 + 1}^\infty \ph_n \mid X)$, and so 
  \eqref{cont-from-above} holds.
\end{proof}

\begin{thm}[Countable additivity]\label{T:ctbl-add}
    \index{countable additivity}%
  If $P$ is entire and $P(\ph_i \wedge \ph_j \mid X) = 0$ for all $1 \le i <
  j < \infty$, then
  \[
    \ts{P(\bigvee_n \ph_n \mid X) = \sum_n P(\ph_n \mid X).}
  \]
\end{thm}

\begin{proof}
  Let $\psi_n = \bigvee_1^n \ph_j$, so that $\psi_n \vdash \psi_{n + 1}$ and
  $\bigvee_n \psi_n \equiv \bigvee_n \ph_n$. Note that $\psi_n \wedge \ph_{n +
  1} \equiv \bigvee_1^n (\ph_j \wedge \ph_{n + 1})$. By Proposition 
  \ref{P:certainty-closure} and induction, we have $P(\psi_n \wedge \ph_{n + 1}
  \mid X) = 0$. Thus, by Theorem \ref{T:incl-excl},
  \[
    P(\psi_n \vee \ph_{n + 1} \mid X)
      = P(\psi_n \mid X) + P(\ph_{n + 1} \mid X).
  \]
  It follows by induction that $P(\psi_n \mid X) = \sum_1^n P(\ph_j \mid X)$ for
  all $n$. Letting $n \to \infty$ and applying the continuity rule completes the
  proof.
\end{proof}

\section{Closed sets and inductive derivability}\label{S:closed}

\subsection{The rule of inductive extension}

The first seven rules of inductive inference encapsulate the core of our
inductive calculus. There are, however, two important and essential supplemental
rules we must define. The first is called the ``rule of inductive extension.''
To motivate this rule, recall the situation in Example \ref{Expl:MathSEexpl}. As
mentioned therein, we will later construct an entire set $P$ in which $P
(\bfr_1) = P(\bfr_2) = P(\bfr_1 \tot \bfr_2) = q$, but $P(\bfr_1 \wedge \bfr_2)$
is undefined (see Example \ref{Expl:MathSE}). This situation is entirely
satisfactory and will not violate our rules of inference in any way. It is, in
fact, self-evident that without additional information, there is no way to
deduce a probability for $\bfr_1 \wedge \bfr_2$ based solely on the
probabilities of $\bfr_1$ and $\bfr_2$.

We can, however, decide to assign, a priori, a probability to $\bfr_1 \wedge
\bfr_2$. In other words, given a value $q'$, we may wish to consider the set
$P' = P \cup \{(\emp, \bfr_1 \wedge \bfr_2, q')\}$. Of course, we must choose
$q'$ so that our new, enlarged set $P'$ continues to conform to the seven rules
of inference we have already established. The question naturally arises: which
values of $q'$ are possible?

In a situation such as this, one of three things can occur. The first is
illustrated by the case $q = 1/4$. In this case, we show in Example
\ref{Expl:MathSEexpl-2} that there are no possible values of $q'$. In other
words, although $P$ is entire, there is no way to assign a probability to
$\bfr_1 \wedge \bfr_2$ without violating one of our first seven rules. There is,
therefore, something defective about $P$, but this flaw cannot be seen from our
first seven rules alone.

The second possibility is illustrated by the case $q = 1/2$. In this case, we
show in Proposition \ref{P:MathSE} that $q' = 1/4$ is the unique value that
works. In other words, the only way to assign a probability to $\bfr_1 \wedge
\bfr_2$ without violating one of our first seven rules is to assign it
probability $1/4$. In this case, it seems reasonable that the uniqueness of this
value ought to let us infer $(\emp, \bfr_1 \wedge \bfr_2, 1/4)$ from $P$. But
such an inference is not possible with only the first seven rules, because $P$
is already entire.

The third possibility is that there are multiple values of $q'$ that work.
Although this case does not arise in Example \ref{Expl:MathSEexpl}, it is clear
that it can arise in even simpler examples. In a case such as this, there is
nothing necessarily defective about our entire set $P$, but at the same time, we
cannot make any inference about the probability of $\bfr_1 \wedge \bfr_2$.

To describe the rule that will rectify these situations, we begin by defining a
new kind of set, which we call ``complete.'' Even after our inductive calculus
is fully developed, the process of inductive inference will not typically
produce complete sets. Rather, they represent a sort of ideal in which all
meaningfully connected probabilities have been logically determined.

\begin{defn}\label{D:complete}
    \index{complete}%
  A set $\ol P \subseteq \cF^\IS$ is \emph{complete} if it is entire and
  satisfies the following conditions:
  \begin{enumerate}[(i)]
    \item If $\ol P(\ph \mid X)$ and $\ol P(\psi \mid X)$ exist, then $\ol P(\ph
          \wedge \psi \mid X)$ exists.
    \item If $X \in \ante \ol P$ and $X \cup \{\ph\} \in \ante \ol P$, then
          $\ol P (\ph \mid X)$ exists.
  \end{enumerate}
\end{defn}

\begin{rmk}
  In general, entire sets obey neither (i) nor (ii) in the definition above.
  Example \ref{Expl:MathSEexpl} describes an entire set that violates (i). In
  Example \ref{Expl:incompl-ind-th}, we exhibit an entire set $P$ with
  $P(\bfr_1) = 1/2$ and $P (\bfr_2 \mid \bfr_3) = 1$, but with $P(\bfr_3)$
  undefined, thereby violating (ii) with $X = \emp$ and $\ph = \bfr_3$.
\end{rmk}

Having defined complete sets, we are now in a position to state our eighth rule
of inductive inference. A set which is closed under the first eight rules will
be called ``semi-closed.''

If $P \subseteq \ol P \subseteq \cF^\IS$, then $\ol P$ is called an
\emph{extension} of $P$. A complete extension will also be called a 
\emph{completion}. A set $P \subseteq \cF^\IS$ is \emph{semi-closed}
  \index{semi-closed}%
if it is entire and satisfies the \emph{rule of inductive extension}:
  \index{rule!of inductive extension}%
\begin{enumerate}[(R1)]
  \setcounter{enumi}{7}
  \item If $\ol P(\ph \mid X) = p$ for every completion $\ol P$ of $P$,
        then $P(\ph \mid X) = p$.
\end{enumerate}
Note that every complete set is semi-closed.

Every semi-closed set has a completion. To see this, suppose $P$ is semi-closed
but has no completion. Then (R8) implies $P = \cF^\IS$. But then $P$ is not
admissible, and therefore not entire, a contradiction.

This means that an entire set can fail to be semi-closed in two ways. On the one
hand, it can simply not have enough probabilities. This happens, for instance,
in Example \ref{Expl:MathSEexpl} when $q = 1/2$. In this case, our set $P$ will
not be semi-closed until we add more probabilities, including $(\emp, \bfr_1
\wedge \bfr_2, 1/4)$, as required by (R8).

On the other hand, it can fail to be semi-closed because it has no completion.
This happens, for instance, in Example \ref{Expl:MathSEexpl} when $q = 1/4$,
since in that case there is no way to assign a probability to $\bfr_1 \wedge
\bfr_2$ without breaking the entirety of the set.

The following result is an analogue of Proposition \ref{P:log-equiv-gen} for
antecedents, but it requires that $P$ be semi-closed.

\begin{prop}\label{P:log-equiv-gen-2}
  Let $P$ be semi-closed. If $P(\ph \mid X, \psi) = p$ and $P(\psi \tot \psi'
  \mid X) = 1$, then $P(\ph \mid X, \psi') = p$.
\end{prop}

\begin{proof}
  Let $P$ be semi-closed with $P(\ph \mid X, \psi) = p$ and $P(\psi \tot \psi'
  \mid X) = 1$. Let $\ol P$ be a completion of $P$. Then $\ol P(\ph \mid X,
  \psi) = p$ and $\ol P(\psi \tot \psi' \mid X) = 1$. By Definition
  \ref{D:complete}(ii), we have that $\ol P(\psi \mid X) = q$ for some $q \in
  [0, 1]$, and Lemma \ref{L:cond-exist} implies $q > 0$. By the multiplication
  rule, $\ol P(\ph \wedge \psi \mid X) = pq$.

  By Proposition \ref{P:certainty-closure}, we have $\ol P(\ph \wedge \psi
  \wedge (\psi \tot \psi') \mid X) = pq$. By the rule of logical equivalence,
  $\ol P (\ph \wedge \psi' \wedge (\psi \tot \psi') \mid X) = pq$. Proposition
  \ref{P:cert-cl-conv} then implies $\ol P(\ph \wedge \psi' \mid X)= pq$. By
  Proposition \ref{P:log-equiv-gen}, we have $\ol P(\psi' \mid X) = q$. Hence,
  by the multiplication rule, $\ol P(\ph \mid X, \psi') = p$. Since $\ol P$ was
  arbitrary, the rule of inductive extension gives $P(\ph \mid X, \psi') = p$.
\end{proof}

\subsection{The rule of deductive extension}

Our final rule is the ``rule of deductive extension.'' Informally, it says that
any antecedent can be freely expanded to include any number of formulas already
known to have probability one. A set of inductive statements that is closed
under all nine rules of inference will be called ``closed.'' More specifically,
a semi-closed set $P$ is said to be \emph{closed}
  \index{closed}%
if it satisfies the \emph{rule of deductive extension}:
  \index{rule!of deductive extension}%
\begin{enumerate}[(R1)]
  \setcounter{enumi}{8}
  \item If $S \subseteq \cF$ is nonempty
        and $P(\th \mid X) = 1$ for all $\th \in S$,
        then $X \cup S \in \ante P$
        and $P(\, \cdot \mid X, S) = P(\, \cdot \mid X)$.
\end{enumerate}
With this final definition, our rules of inductive inference are complete.

\subsection{Pre-theories}

We now wish to use these rules to define inductive derivability. Our aim is to
make sense of the statement $Q \vdash (X, \ph, p)$. Informally, we imagine $Q
\vdash (X, \ph, p)$ to mean that, starting from the inductive statements in $Q$,
we may use a sequence of applications of our nine rules of inference to derive
the inductive statement, $ (X, \ph, p)$. In keeping with the spirit of our
infinitary language $\cF$, we will imagine, when necessary, that this sequence
is at most countable. We aim to make this notion precise by using our nine rules
of inference and their related closure properties (admissible, entire,
semi-closed, and closed).

As mentioned earlier, we plan to follow the route described in Remark
\ref{R:theory-first}. That is, by analogy with deductive theories, we want to
define an ``inductive theory.'' We will then say that $Q \vdash (X, \ph, p)$ if
$(X, \ph, p)$ is an element of the smallest inductive theory containing $Q$. Our
first task, of course, is to determine exactly what ought to constitute an
inductive theory. At first glance, the answer may seem trivially obvious: an
inductive theory ought to simply be a closed set. After all, closed sets, by
definition, are those sets that are closed under all nine rules of inference, so
this would be the natural analogue of a deductive theory. We will see, however,
that closed subsets of $\cF^\IS$ can be much larger than we might initially
expect, and as such, do not fit our intuitive understanding of what an inductive
theory ought to be.

To see this, let us first focus our attention on semi-closed sets, which are
closed under (R1)--(R8). Imagine we have a starting collection of inductive
statements, $Q$. Using $Q$ together with rules (R1)--(R8), we begin making
inferences and adding new inductive statements to our collection. When we have
exhausted all inferences that are possible with (R1)--(R8), we arrive at a
finalized collection, $P_0$, which we will call a ``pre-theory.'' Our pre-theory
$P_0$ is closed under (R1)--(R8), so by definition, $P_0$ is a semi-closed set.
But $P_0$ has another important property that arbitrary semi-closed sets do not
share. The elements of $P_0$ are all ``connected'' to the elements of $Q$ in a
certain sense. They are connected via (R1)--(R8).

To clarify the nature of this connection, let us consider the effects of
(R1)--(R8) on the antecedents of $Q$. That is, imagine we use a single
application of one of the rules (R1)--(R8) to infer $(X, \ph, p)$ from $Q$. We
wish to know how $X$ is related to $\ante Q$. If we have used any of the rules
(R2), (R3), (R4), (R5), (R7), or (R8), then we know that $X \in \ante Q$. If we
used (R1), then $X \equiv Y$ for some $Y \in \ante Q$. And if we used (R6), then
either $X \in \ante Q$ or $X = Y \cup \{\ph\}$ for some $Y \in \ante Q$ and some
$\ph \in \cF$. Generally speaking, no matter which of (R1)--(R8) we used, we may
conclude that either $X \equiv Y$ or $X \equiv Y \cup \{\ph\}$ for some $Y \in
\ante Q$ and $\ph \in \cF$. More generally, if $(X, \ph, p)$ is inferred from
$Q$ via a countable sequence of applications of (R1)--(R8), then $X \equiv Y
\cup \Phi$ for some $Y \in \ante Q$ and some countable $\Phi \subseteq \cF$.

Motivated by this, we make the following definition. A set $Q \subseteq \cF^\IS$
is \emph{strongly connected}
  \index{connected!strongly ---}%
if there exists $X_0 \in \ante Q$ such that every $X \in \ante Q$ is
\emph{countably axiomatizable}
  \index{countably axiomatizable}%
over $X_0$. That is, for every $X \in \ante Q$, there exists a countable set
$\Phi \subseteq \cF$ such that $X \equiv X_0 \cup \Phi$. Note that a strongly
connected set is necessarily nonempty. A set $P_0 \subseteq \cF^\IS$ is a
\emph{pre-theory}
  \index{pre-theory}%
if it is semi-closed and strongly connected.

Strong connectivity formalizes the notion that the inductive statements in a set
can be related to one another via the calculus of (R1)--(R8). A pre-theory
represents the results of exhausting all possible inferences using (R1)--(R8).
We require not only that a pre-theory be closed under (R1)--(R8), and therefore
a semi-closed set, but also that it be strongly connected. A set which is
semi-closed but not strongly connected will not violate (R1)--(R8), but it will
include unnecessary parts and statements which can never be related to one
another by the calculus of (R1)--(R8).

\subsection{Inductive theories}\label{S:ind-theories}

We now turn our attention to (R9). The first issue to address is the interplay
between (R9) and the previous eight rules. Suppose we have exhausted all
inferences from (R1)--(R8) and obtained a pre-theory, $P_0$. We then use (R9) to
infer a new inductive statement, $(X, \ph, p)$. Is it possible that $P_0 \cup 
\{(X, \ph, p)\}$ can be extended to an even larger pre-theory? Corollary 
\ref{C:theory-defn} assures us that it cannot. In other words, we lose nothing
by requiring that all applications of (R9) take place after all applications of 
(R1)--(R8).

With this in mind, we aim to say that an ``inductive theory'' is what we obtain
by first constructing a pre-theory and then ``closing it up'' with (R9). Theorem
\ref{T:theory-defn} below shows that this closure operation is well-defined and
produces a unique result. The proof of Theorem \ref{T:theory-defn}, as well as
all the other results in the remainder of this section, will be postponed until
Sections \ref{S:lifts} and \ref{S:ind-cond}. We present here only the statements
of the results, so that we may first see an overview of the entire construction.

\begin{thm}\label{T:theory-defn}
  Every pre-theory has a unique smallest closed extension.
\end{thm}

If $P_0 \subseteq \cF^\IS$ is a pre-theory, let $\bfP(P_0)$ or $\bfP_{P_0}$
denote its smallest closed extension.

\begin{cor}\label{C:theory-defn}
  Let $P_0, P_0' \subseteq \cF^\IS$ be pre-theories. If $\bfP(P_0) = \bfP
  (P_0')$, then $P_0 = P_0'$. In particular, if $P_0$ is a pre-theory, then
  there is no pre-theory $P_0'$ such that $P_0 \subset P_0' \subseteq \bfP
  (P_0)$.
\end{cor}

Having established these results, we will be able to say that an ``inductive
theory'' is a set of the form $\bfP(P_0)$ for some pre-theory $P_0$. Before
formally defining it as such, we pause to characterize such sets in a way
analogous to our characterization of pre-theories. For this characterization, we
first define a new connectivity property.

Recall the notation established in \eqref{tau-Q-X}. We say that a set $Q
\subseteq \cF^\IS$ is \emph{connected}
  \index{connected}%
if there exists a strongly connected $\wh Q \subseteq Q$ such that for all $X
\in \ante Q$, there exists an $\wh X \in \ante \wh Q$ and a set $S \subseteq
\tau(\wh Q; \wh X)$ such that $X \equiv \wh X \cup S$. Any such $\wh Q$ will be
called a \emph{basis}
  \index{basis}%
for $Q$. In other words, a connected set is a ``lift'' of a strongly connected
set, where we lift up the antecedents by including formulas that have
probability one.

Note that if $Q$ is strongly connected, then $Q$ is connected and is its own
basis. Also note the following important difference between connectivity and
strong connectivity. Strong connectivity is a property of $\ante Q$. That is, if
$\ante Q = \ante Q'$, then $Q$ is strongly connected if and only if $Q'$ is
strongly connected. Connectivity, on the other hand, is not. Connectivity
depends not only on $\ante Q$, but also on $\{(X, \th) \mid (X, \th, 1) \in
Q\}$.

\begin{thm}\label{T:theory-char}
  Let $P \subseteq \cF^\IS$. The following are equivalent:
  \begin{enumerate}[(i)]
    \item $P = \bfP(P_0)$ for some (unique) pre-theory $P_0$.
    \item $P$ is closed and connected.
  \end{enumerate}
\end{thm}

With this last result, we can finally state the definition that is the linchpin
of our entire inductive calculus: a set $P \subseteq \cF^\IS$ is a 
\emph{inductive theory} if it is closed and connected.
  \index{theory!inductive ---}%

The intuitive interpretation of connectivity is analogous to strong
connectivity, but instead of using only (R1)--(R8), we use the whole of our
inductive calculus. That is, a connected set $P$ is one whose inductive
statements can be related to one another via the calculus. A set which is closed
but not connected will not violate the calculus, but it will have unnecessary
parts which the inductive calculus can never reach.

\subsection{Inductive derivability}\label{S:ind-derivability}

It is now straightforward to define inductive derivability. We begin by defining
a set $Q \subseteq \cF^\IS$ to be \emph{consistent}
  \index{consistent}%
if it is connected and can be extended to an inductive theory. The requirement
that a consistent set be extendable to an inductive theory ensures that it does
not violate the calculus of inductive inference. The requirement that it be
connected ensures that its statements are all logically related to one another.

Note that every pre-theory is consistent. Moreover, if $P_0$ is a pre-theory,
then $\bfP(P_0)$ is the smallest extension of $P_0$ to an inductive theory.

\begin{thm}\label{T:theory-gen-defn}
  Every consistent set has a unique smallest extension to an inductive theory.
\end{thm}

If $Q \subseteq \cF^\IS$ is consistent, let $\bfP(Q)$ or $\bfP_Q$
  \symindex{$\bfP(Q), \bfP_Q$}%
denote its smallest extension to an inductive theory. We call $\bfP_Q$ the
\emph{inductive theory generated by $Q$}. If $Q \subseteq \cF^\IS$ and $(X, \ph,
p) \in \cF^\IS$, we write $Q \vdash (X, \ph, p)$
  \symindex{$Q \vdash (X, \ph, p)$}%
to mean that $Q$ is consistent and $\bfP_Q(\ph \mid X) = p$. When the turnstile
symbol, $\vdash$, is used in this fashion, we will call it the \emph{inductive
derivability relation}.
  \index{derivability relation!propositional ---}%
When $Q \vdash (X, \ph, p)$, we say that $(X, \ph, p)$ is \emph{inductively
derivable} from $Q$, or that $Q$ \emph{proves} $(X, \ph, p)$. Note that unlike
deductive derivability, our convention is that if $Q \subseteq \cF^\IS$ is
inconsistent, then $Q$ does not prove anything.

\begin{rmk}\label{R:classic-ind-th-defn}
  If $P \subseteq \cF^\IS$ is consistent, then $P$ is an inductive theory if and
  only if $P = \bfP(P)$, which holds if and only if $\bfP(P) \subseteq P$.
  Hence, using the above definition of inductive derivability, we can say that a
  consistent set $P \subseteq \cF^\IS$ is an inductive theory if and only if $P
  \vdash (X, \ph, p)$ implies $(X, \ph, p) \in P$ for all $(X, \ph, p) \in
  \cF^\IS$.
\end{rmk}

\section{Connectivity, restrictions, and lifts}\label{S:lifts}

In this section, we prove the results in Section \ref{S:ind-theories}. We begin
with four subsections on preliminary results needed in the proofs. In the first
two subsections, we establish some basic facts about connected and strongly
connected sets, and how they relate to the rules of inference. The third
subsection looks at restrictions of sets $Q \subseteq \cF^\IS$, and examines
which closure properties are preserved under restriction. Finally, the fourth
subsection defines the ``lift'' of a pre-theory, which we denote by $\bfL(P_0)$.
After establishing these preliminaries, we gives the proofs of Theorem 
\ref{T:theory-defn}, Corollary \ref{C:theory-defn}, and Theorem 
\ref{T:theory-char}.

\subsection{Connectivity properties}

For $X, X_0 \subseteq \cF^\IS$, we write $X \cao X_0$
  \symindex{$X \cao X_0$}%
to mean that $X$ is countably axiomatizable over $X_0$. If $\cX \subseteq \fP
\cF$, we write $X \cao \cX$ to mean $X \cao X_0$ for some $X_0 \in \cX$.

\begin{prop}\label{P:ctbly-ax=one-ax}
  Let $X, X_0 \subseteq \cF$. Then $X \cao X_0$ if and only if $X \equiv X_0
  \cup \{\psi\}$ for some $\psi \in \cF$.
\end{prop}

\begin{proof}
  This follows immediately from the fact that $\Phi \equiv \bigwedge \Phi$ for
  any countable $\Phi \subseteq \cF$.
\end{proof}

\begin{prop}\label{P:root-exist}
  If $Q \subseteq \cF^\IS$ is connected, then there exists $X_0 \in \ante Q$
  such that $X \vdash X_0$ for all $X \in \ante Q$. This $X_0$ is unique in the
  sense that if $X_0'$ is another such antecedent, then $X_0 \equiv X_0'$.
\end{prop}

\begin{proof}
  Let $\wh Q$ be a basis for $Q$. Since $\wh Q$ is strongly connected, we may
  choose $X_0 \in \ante \wh Q$ such that $\wh X \cao X_0$ for all $\wh X \in
  \ante \wh Q$. Now let $X \in \ante Q$ be given. Choose $\wh X \in \ante \wh
  Q$ and $S \subseteq \tau(\wh Q; \wh X)$ such that $X \equiv \wh X \cup S$.
  Since $\wh X \cao X_0$, we have $\wh X \vdash X_0$, and therefore $X \vdash
  X_0$. Uniqueness is immediate since $X_0 \vdash X_0'$ and $X_0' \vdash X_0$ implies $X_0 \equiv X_0'$.
\end{proof}

Let $Q \subseteq \cF^\IS$ be connected. Choose $X_0$ as in Proposition 
\ref{P:root-exist} and let $T_0 = T(X_0)$. By Proposition \ref{P:root-exist},
the deductive theory $T_0$ does not depend on the choice of $X_0$. We call $T_0$
the \emph{root}
  \index{root}%
of $Q$. Note that if $Q$ is admissible, then by the rule of logical equivalence,
$T_0 \in \ante Q$.

\begin{prop}\label{P:str-conn-root-one-away}
  If $Q \subseteq \cF^\IS$ is strongly connected with root $T_0$, then for each
  $X \in \ante Q$, there exists $\psi \in \cF$ such that $T(X) = T_0 + \psi$.
\end{prop}

\begin{proof}
  Since $Q$ is strongly connected, we may choose $X_0 \in \ante Q$ such that $X
  \cao X_0$ for all $X \in \ante Q$. Hence, $X \vdash X_0$ for all $X \in \ante
  Q$. By Proposition \ref{P:root-exist}, we have $T_0 = T(X_0)$. Now let $X \in
  \ante Q$ be given. Then $X \cao X_0$, so by Proposition 
  \ref{P:ctbly-ax=one-ax}, we may choose $\psi \in \cF$ such that $X \equiv X_0
  \cup \{\psi\}$, and this gives $T(X) = T_0 + \psi$.
\end{proof}

\begin{prop}\label{P:basis-root}
  If $Q$ is connected and $\wh Q$ is a basis for $Q$, then $Q$ and $\wh Q$ have
  the same root.
\end{prop}

\begin{proof}
  Let $T_0$ be the root of $Q$ and $\wh T_0$ be the root of $\wh Q$. Then $T_0
  = T(X_0)$ for some $X_0 \in \ante Q$ and $\wh T_0 = T(\wh X_0)$ for some
  $\wh X_0 \in \ante \wh Q$. Since $\wh Q \subseteq Q$, we have $\ante \wh Q
  \subseteq \ante Q$. Hence, $\wh X_0 \in \ante Q$, so by the definition of
  $T_0$, we have $\wh X_0 \vdash X_0$.

  On the other hand, since $X_0 \in \ante Q$ and $\wh Q$ is a basis for $Q$, we
  may choose $\wh X \in \ante \wh Q$ and $S \subseteq \tau(\wh Q; \wh X)$ such
  that $X_0 \equiv \wh X \cup S$. Hence, $X_0 \vdash \wh X$. But $\wh X \in
  \ante \wh Q$, so by the definition of $\wh T_0$, we get $\wh X \vdash
  \wh X_0$. Thus, $X_0 \vdash \wh X_0$, which shows that $X_0 \equiv \wh X_0$,
  and therefore $T_0 = \wh T_0$.
\end{proof}

\subsection{Connectivity and inductive inference}

If $P$ is admissible and strongly connected with root $T_0$, and $\psi \in \cF$
satisfies $T_0 + \psi \in \ante P$, then we call $\psi$ an \emph{antecedent
formula} of $P$.
  \index{antecedent!formula}%
The set of all antecedent formulas of $P$ is denoted by $\AF(P)$.
  \symindex{$\AF(P)$}%
Note that
\[
  \ante P =  \{
    X \subseteq \cF^\IS \mid T(X) = T_0 + \psi \text{ for some } \psi \in \AF(P)
  \}.
\]
This follows from Proposition \ref{P:str-conn-root-one-away} and the rule of
logical equivalence.

For this next result, recall the notation introduced in \eqref{tau-Q-X}.

\begin{prop}\label{P:simple-tau(P)}
  If $P$ is entire and connected with root $T_0$, then $\tau(P) = \tau(P; T_0)$.
\end{prop}

\begin{proof}
  Let $\th \in \tau(P)$. Then $P(\th \mid X) = 1$ for all $X \in \ante P$. In
  particular, $P(\th \mid T_0) = 1$, so $\th \in \tau(P; T_0)$. Conversely,
  suppose $P(\th \mid T_0) = 1$ and let $X \in \ante P$ be given. Since $T_0$ is
  the root of $P$, we have $X \vdash T_0$, so by deductive transitivity, $P(\th
  \mid X) = 1$. Since $X$ was arbitrary, $\th \in \tau(P)$.
\end{proof}

For this next result, recall the interval notation introduced just prior to
Lemma \ref{L:omit-psi}.

\begin{prop}\label{P:semi-cl-conn}
  Let $P$ be semi-closed and connected with root $T_0$. If $X \in \ante P$, then
  $X \equiv T + \psi$ for some $T \in [T_0, \tau(P)]$ and some $\psi \in \cF$.
  Moreover, $\psi$ can be chosen so that $T_0 + \psi \in \ante P$.
\end{prop}

\begin{proof}
  Let $\wh Q$ be a basis for $P$. By Proposition \ref{P:basis-root}, $\wh Q$
  also has root $T_0$. Let $X \in \ante P$. Choose $\wh X \in \ante \wh Q$ and
  $S \subseteq \tau(\wh Q; \wh X)$ such that $X \equiv \wh X \cup S$. By
  Proposition \ref{P:str-conn-root-one-away}, we may choose $\psi \in \cF$, such
  that $T(\wh X) = T_0 + \psi$. Since $\wh Q \subseteq P$, we have $\wh X \in
  \ante P$. By the rule of logical equivalence, $T_0 + \psi \in \ante P$.

  Now define $S' = \{\psi \to \th \mid \th \in S\}$ and $T = T_0 + S'$. Then
  $T_0 \subseteq T$ and, by Lemma \ref{L:omit-psi}, we have $X \equiv T(\wh X) +
  S = T_0 + \psi + S = T_0 + \psi + S' = T + \psi$. Hence, it remains only to
  show that $T \subseteq \tau(P)$. By Proposition \ref{P:simple-tau(P)} and the
  fact that $T_0 \subseteq \tau(P; T_0)$, we need only show that $S' \subseteq
  \tau(P; T_0)$.

  Let $\eta \in S'$. Choose $\th \in S$ such that $\eta = \psi \to \th$. Since
  $S \subseteq \tau(\wh Q; \wh X)$, we have $\wh Q(\th \mid \wh X) = 1$. Since
  $\wh Q \subseteq P$, we have $P(\th \mid \wh X) = 1$. By the rule of logical
  equivalence, $P(\th \mid T_0, \psi) = 1$. Hence, by the rule of material
  implication, it follows that $P(\eta \mid T_0) = P(\psi \to \th \mid T_0) =
  1$, showing that $\eta \in \tau(P; T_0)$.
\end{proof}

\subsection{Restrictions}

For $Q \subseteq \cF^\IS$ and $\cX \subseteq \ante Q$, we define
\[
  Q\dhl_\cX = \{(X, \ph, p) \in Q \mid X \cao \cX\}.
\]
  \symindex{$Q \dhl_\cX$}%
Note that if $Q \subseteq Q'$, then $Q\dhl_{X_0} \subseteq Q'\dhl_{X_0}$.

\begin{thm}\label{T:restrict-inherit}
  Let $P \subseteq \cF^\IS$, $\cX \subseteq \ante P$, and define $P' =
  P\dhl_\cX$.
  \begin{enumerate}[(i)]
    \item If $P$ is admissible, then $P'$ is admissible.
    \item If $P$ is entire, then $P'$ is entire.
    \item If $P$ is complete, then $P'$ is complete.
    \item If $P$ is semi-closed, then $P'$ is semi-closed.
  \end{enumerate}
\end{thm}

\begin{proof}
  Assume $P$ is admissible. Suppose $(X, \ph , p) \in P'$, so that $X \cao \cX$
  and $(X, \ph, p) \in P$. Let $X' \equiv X$ and $\ph' \equiv_X \ph$. Since $P$
  is admissible, we have $(X', \ph', p) \in P$, and since $X' \equiv X$, it
  follows that $X' \cao \cX$. Hence, $(X', \ph', p) \in P'$. Now suppose $(X',
  \ph', p') \in P'$. Then $(X', \ph', p') \in P$, so the admissibility of $P$
  gives $p' = p$, and therefore $P'$ is admissible.

  Note that
  \begin{enumerate}[(a)]
    \item $X \in \ante P'$ if and only if $X \in \ante P$ and $X \cao \cX$,
    \item $P'(\ph \mid X) = p$ if and only if $P(\ph \mid X) = p$ and $X \cao
          \cX$, and
    \item $X \cao \cX$ implies $X \cup \{\ph\} \cao \cX$.
  \end{enumerate}
  Assume that $P$ is entire. From (a) and (b), we easily see that $P'$ satisfies
  (R2)--(R5) and (R7). From (a)--(c) we get that $P'$ satisfies (R6), so $P'$ is
  entire.

  Assume $P$ is complete. As above, (a) and (b) easily show that $P'$ satisfies
  Definition \ref{D:complete}, so that $P'$ is complete.

  Finally, assume $P$ is semi-closed. Assume every completion of $P'$ contains
  $(X, \ph, p)$. Let $\ol P$ be a completion of $P$. By (iii), $\ol P \dhl_\cX$
  is complete. Since $P \subseteq \ol P$, we have $P' = {P \dhl_\cX} \subseteq
  {\ol P \dhl_\cX}$, so that $\ol P \dhl_\cX$ is a completion of $P'$. Hence,
  $\ol P \dhl_\cX(\ph \mid X) = p$. By (b), we have $\ol P(\ph \mid X) = p$ and
  $X \cao \cX$. Since $\ol P$ was arbitrary and $P$ is semi-closed, it follows
  that $P(\ph \mid X) = p$. Since $X \cao \cX$, we get $P'(\ph \mid X) = p$, and
  $P'$ is semi-closed.
\end{proof}

\begin{cor}\label{C:restrict-inherit}
  If $P$ is complete and $X_0 \in \ante P$, then $P \dhl_{X_0}$ is complete and
  strongly connected with root $T(X_0)$.
\end{cor}

\begin{proof}
  This follows immediately from Theorem \ref{T:restrict-inherit}.
\end{proof}

\begin{cor}\label{C:restrict-ante}
  If $P$ is entire and $\ol P$ is a completion of $P$, then $\ol P \dhl_{\ante
  P}$ is also a completion of $P$.
\end{cor}

\begin{proof}
  Let $P$ be entire and let $\ol P$ be a completion of $P$. Define $P' = \ol P
  \dhl_{\ante P}$. Theorem \ref{T:restrict-inherit} implies that $P'$ is also
  complete. Suppose $P(\ph \mid X) = p$. Since $P \subseteq \ol P$, we have
  $\ol P(\ph \mid X) = p$. Since $X \in \ante P$, it follows that $P'(\ph \mid
  X) = p$, showing that $P \subseteq P'$.
\end{proof}

\begin{cor}\label{C:ind-ext-str-conn}
  Suppose $P$ is an entire set that has a completion. If $\ol P(\ph \mid X) = p$
  for all completions $\ol P$ of $P$, then $X \cao \ante P$.
\end{cor}

\begin{proof}
  Let $P$ be an entire set that has a completion. Assume $\ol P(\ph \mid X) = p$
  for all completions $\ol P$ of $P$. Choose a completion $\ol P$ of $P$. By
  Corollary \ref{C:restrict-ante}, the set $P' = \ol P \dhl_{\ante P}$ is also a
  completion of $P$. Hence, $P'(\ph \mid X) = p$. By (b) above, we have $\ol P
  (\ph \mid X) = p$ and $X \cao \ante P$.
\end{proof}

\begin{cor}\label{C:compl-restrict}
  Let $P_0$ be a pre-theory with root $T_0$ and let $\ol P_0$ be a completion of
  $P_0$. Define $P_0' = \ol P_0 \dhl_{T_0}$. Then $P_0'$ is a completion of
  $P_0$ that is also a pre-theory with root $T_0$.
\end{cor}

\begin{proof}
  By Theorem \ref{T:restrict-inherit}, the set $P_0'$ is complete and strongly
  connected with root $T_0$. Since complete sets are semi-closed, it follows
  that $P_0'$ is a pre-theory.

  Suppose $P_0(\ph \mid X) = p$. Since $P_0 \subseteq \ol P_0$, we have $\ol P_0
  (\ph \mid X) = p$. Since $P_0$ is strongly connected with root $T_0$, it
  follows that $X \cao T_0$. Thus, $P_0'(\ph \mid X) = p$, showing that $P_0
  \subseteq P_0'$.
\end{proof}

\subsection{The lift of a pre-theory}

Let $P_0$ be a pre-theory with root $T_0$. By Proposition
\ref{P:tau-ded-theory}, the set $\tau(P_0)$ is a deductive theory. Since $P_0$
is admissible, we have $T_0 \in \ante P_0$, so that by Proposition
\ref{P:simple-tau(P)} and the rule of logical implication, $T_0 \subseteq
\tau(P_0)$. Let $X \subseteq \cF$. If $T(X) = T + \psi$ for some $T \in [T_0,
\tau(P_0)]$ and some $\psi \in \AF(P_0)$, then we call $X$ a \emph{generalized
antecedent of $P_0$}.
  \index{antecedent!generalized ---}
The set of all such generalized antecedents is denoted by $\cG\cA(P_0)$.
  \symindex{$\cG\cA(P_0)$}

\begin{prop}\label{P:lift-well-defined}
  Let $P_0$ be a pre-theory with root $T_0$ and let $X \in \cG\cA(P_0)$. Suppose
  $T(X) = T + \psi = T' + \psi'$, where $T, T' \in [T_0, \tau(P_0)]$ and $\psi,
  \psi' \in \AF(P_0)$. Then $P_0(\; \cdot \mid T_0, \psi) = P_0(\; \cdot \mid
  T_0, \psi')$.
\end{prop}

\begin{proof}
  By symmetry, it suffices to show that $P_0(\ph \mid T_0, \psi) = p$ implies
  $P_0(\ph \mid T_0, \psi') = p$. Assume $P_0(\ph \mid T_0, \psi) = p$. First
  note that $T + \psi = T' + \psi'$ implies $\psi' \in T + \psi \subseteq \tau
  (P_0) + \psi$, so that $\tau(P_0), \psi \vdash \psi'$. Likewise, $\tau(P_0),
  \psi' \vdash \psi$. Hence, $\psi \tot \psi' \in \tau(P_0)$, giving $P_0(\psi
  \tot \psi' \mid T_0) = 1$. Since $\psi' \in \AF(P_0)$, we have $T_0 + \psi'
  \in \ante P_0$, so by two applications of deductive transitivity, $P_0(\psi'
  \to \psi \mid T_0, \psi') = 1$. By the rule of logical implication,
  $P_0(\psi' \mid T_0, \psi') = 1$. Applying Proposition
  \ref{P:transitivity}(iii) with $\ze = \top$ gives $P_0 (\psi \mid T_0, \psi')
  = 1$. A similar argument shows that $P_0(\psi' \mid T_0, \psi) = 1$.

  Now, by Proposition \ref{P:certainty-closure}, we have $P_0(\ph \wedge \psi'
  \mid T_0, \psi) = p$. By the multiplication rule, $P_0(\ph \mid T_0, \psi,
  \psi') = p$. A second application of the multiplication rule then gives $P_0
  (\ph \wedge \psi \mid T_0, \psi') = p$. By Proposition \ref{P:cert-cl-conv},
  it follows that $P_0(\ph \mid T_0, \psi') = p$.
\end{proof}

Let $P_0$ be a pre-theory with root $T_0$. For each $X \in \cG\cA(P_0)$, choose
$T^X \in [T_0, \tau(P_0)]$ and $\psi_X \in \AF(P_0)$ such that $T(X) = T^X +
\psi_X$. Define
\[
  \bfL(P_0) = \{
    (X, \ph, p) \mid X \in \cG\cA(P_0), P_0(\ph \mid T_0, \psi_X) = p
  \}.
\]
  \symindex{$\bfL(P_0)$}%
By Proposition \ref{P:lift-well-defined}, the definition of $\bfL(P_0)$ does not
depend on the choices of $T^X$ and $\psi_X$. We call $\bfL(P_0)$ the \emph{lift}
of $P_0$, and may sometimes denote it by $\bfL_{P_0}$.
  \index{lift}%

\begin{prop}\label{P:any-psi}
  Let $P_0$ be a pre-theory with root $T_0$. If $X \in \ante \bfL(P_0)$ and $T
  (X) = T + \psi$, where $T \in [T_0, \tau(P_0)]$ and $\psi \in \cF$, then
  $\psi \in \AF(P_0)$.
\end{prop}

\begin{proof}
  Since $X \in \ante \bfL(P_0)$, we may choose $\ph$ and $p$ such that $(X, \ph,
  p) \in \bfL(P_0)$. Then $T(X) = T^X + \psi_X$ and $P_0(\ph \mid T_0, \psi_X)
  = p$. As in the proof of Proposition \ref{P:lift-well-defined}, since $T, T^X
  \subseteq \tau(P_0)$ and $T + \psi = T^X + \psi_X$, we have $\psi \tot
  \psi_X \in \tau(P_0)$. Thus, $P_0(\psi \tot \psi_X \mid T_0) = 1$. Since $P_0$
  is semi-closed, Proposition \ref{P:log-equiv-gen-2} gives $P_0(\ph \mid T_0,
  \psi) = p$, showing that $T_0 + \psi \in \ante P_0$, and therefore $\psi \in
  \AF(P_0)$.
\end{proof}

\begin{prop}\label{P:lift-extends}
  If $P_0$ is a pre-theory, then $P_0 \subseteq \bfL(P_0)$.
\end{prop}

\begin{proof}
  Let $T_0$ be the root of $P_0$. Suppose $P_0(\ph \mid X) = p$. Since $P_0$ is
  strongly connected, we may choose $\ph \in \cF$ such that $T(X) = T_0 + \psi$.
  By the rule of logical equivalence, $P_0(\ph \mid T_0, \psi) = p$. Thus,
  $\psi \in \AF(P_0)$ and, since $T_0 \in [T_0, \tau(P_0)]$, we have $X \in
  \cG\cA(P_0)$. Therefore, $(X, \ph, p) \in \bfL(P_0)$.
\end{proof}

\begin{prop}\label{P:lift-unique}
  If $P_0$ is a pre-theory with root $T_0$, then $\bfL(P_0) \dhl_{T_0} = P_0$.
\end{prop}

\begin{proof}
  From Proposition \ref{P:lift-extends}, it follows that $P_0 = {P_0 \dhl_{T_0}}
  \subseteq {\bfL (P_0) \dhl_{T_0}}$. For the reverse, suppose $(X, \ph,
  p) \in \bfL (P_0) \dhl_ {T_0}$. Then $X \in \cG\cA(P_0)$, $T(X) = T^X +
  \psi_X$, $P_0 (\ph \mid T_0, \psi_X) = p$, and $X \cao T_0$. By Proposition 
  \ref{P:ctbly-ax=one-ax}, we may write $T(X) = T_0 + \psi$ for some $\psi
  \in \cF$. By Proposition \ref{P:any-psi}, we know that $\psi \in \AF(P_0)$.
  Therefore, by Proposition \ref{P:lift-well-defined}, we have $P_0(\ph \mid
  T_0, \psi) = p$. By the rule of logical equivalence, $P_0(\ph \mid X) = p$.
\end{proof}

\begin{lemma}\label{L:lift-admissible}
  Let $P_0$ be a pre-theory. If $X \in \ante \bfL(P_0)$ and $X \vdash \ph$, then
  $(X, \ph, 1) \in \bfL(P_0)$.
\end{lemma}

\begin{proof}
  Choose $\la$ and $p$ such that $(X, \la, p) \in \bfL(P_0)$. Then $T(X) = T^X
  + \psi_X$ and $P_0(\la \mid T_0, \psi_X) = p$. In particular, $T_0 + \psi_X
  \in \ante P_0$.

  Since $X \vdash \ph$, we have $\ph \in T(X) = T^X + \psi_X$, meaning $T^X,
  \psi_X \vdash \ph$. Choose countable $\Phi \subseteq T^X$ such that $\Phi,
  \psi_X \vdash \ph$. Since $\Phi \subseteq T^X \subseteq \tau(P_0)$, it follows
  that $P_0(\th \mid T_0) = 1$ for all $\th \in \Phi$. By Corollary 
  \ref{C:certainty-closure}, if $\ze = \bigwedge \Phi$, then $P_0(\ze \mid T_0)
  = 1$. Since $T_0 + \psi_X \in \ante P_0$, deductive transitivity gives $P_0
  (\ze \mid T_0, \psi_X) = 1$. By the rule of logical implication, $P_0(\psi_X
  \mid T_0, \psi_X) = 1$. Hence, by Proposition \ref{P:certainty-closure}, we
  obtain $P_0(\ze \wedge \psi_X \mid T_0, \psi_X) = 1$. But $\ze \wedge \psi_X
  \vdash \ph$, so by deductive transitivity, $P_0(\ph \mid T_0, \psi_X) = 1$,
  which gives $(X, \ph, 1) \in \bfL(P_0)$.
\end{proof}

\begin{prop}\label{P:lift-admissible}
  If $P_0$ is a pre-theory, then $\bfL(P_0)$ is admissible.
\end{prop}

\begin{proof}
  Suppose $(X, \ph, p) \in \bfL(P_0)$, $X' \equiv X$, and $\ph' \equiv_X \ph$.
  Then $T(X) = T^X + \psi_X$ and $P_0(\ph \mid T_0, \psi_X) = p$. Since $\ph'
  \equiv_X \ph$, we have $X \vdash \ph' \tot \ph$. From Lemma 
  \ref{L:lift-admissible}, it follows that $(X, \ph' \tot \ph, 1) \in \bfL
  (P_0)$, which gives $P_0(\ph' \tot \ph \mid T_0, \psi_X) = 1$. By Proposition
  \ref{P:log-equiv-gen}, we have $P_0(\ph' \mid T_0, \psi_X) = p$. But $X'
  \equiv X$, so $T(X') = T(X) = T^X + \psi_X$. Thus, $(X', \ph', p) \in
  \bfL(P_0)$. Finally, if $(X', \ph', p') \in \bfL(P_0)$, then $P_0(\ph' \mid
  T_0, \psi_X) = p'$. But $P_0$ is admissible, so $p' = p$.
\end{proof}

Let $P_0$ be a pre-theory with root $T_0$ and let $P = \bfL(P_0)$. By
Proposition \ref{P:lift-admissible}, we may now use the notation $P(\ph \mid X)
= p$ instead of $(X, \ph, p) \in P$. Note that by Proposition \ref{P:any-psi},
if $X \in \ante P$, $T \in [T_0, \tau(P_0)]$, $\psi \in \cF$, and $T(X) = T +
\psi$, then $P(\; \cdot \mid X) = P_0(\; \cdot \mid T_0, \psi)$.

\begin{prop}\label{P:nested-lifts}
  Let $P_0, P_0'$ be pre-theories with a common root. If $P_0 \subseteq P_0'$,
  then $\bfL(P_0) \subseteq \bfL(P_0')$.
\end{prop}

\begin{proof}
  Let $T_0$ be the common root of $P_0$ and $P_0'$. Assume $P_0 \subseteq P_0'$.
  Let $P = \bfL(P_0)$ and $P' = \bfL(P_0')$. Suppose $P(\ph \mid X) = p$. Then
  $T(X) = T^X + \psi_X$ and $P_0(\ph \mid T_0, \psi_X) = p$. Since $P_0
  \subseteq P_0'$, we have $\tau(P_0) \subseteq \tau(P_0')$, and therefore
  $[T_0, \tau(P_0)] \subseteq [T_0, \tau(P_0')]$. From $P_0 \subseteq P_0'$, it
  also follows that $\AF(P_0) \subseteq \AF(P_0')$. Hence, $T^X \in [T_0, \tau
  (P_0')]$ and $\psi_X \in \AF(P_0')$. This shows that $X \in \cG\cA(P_0')$ and
  $P'(\ph \mid X) = P_0'(\ph \mid T_0, \psi_X)$. But $P_0 \subseteq P_0'$, so
  $P_0'(\ph \mid T_0, \psi_X) = p$.
\end{proof}

\subsection{Identifying lifts with inductive theories}

We are now ready to prove Theorem \ref{T:theory-defn}. We will do this by
showing that $\bfL(P_0)$ is the smallest closed extension of $P_0$. The proof of
Theorem \ref{T:theory-defn} is broken into five pieces for greater readability.

\begin{prop}\label{P:lift-entire}
  If $P_0$ is a pre-theory, then $\bfL(P_0)$ is entire.
\end{prop}

\begin{proof}
  Let $P_0$ be a pre-theory with root $T_0$ and let $P = \bfL(P_0)$. By
  Proposition \ref{P:lift-admissible}, $P$ is admissible, and by Lemma 
  \ref{L:lift-admissible}, $P$ satisfies the rule of logical implication. It
  therefore remains only to verify (R3)--(R7).

  We begin with the rule of material implication. Suppose $X \in \ante P$ and $P
  (\ph \mid X, \psi) = 1$. Then $T(X) = T^X + \psi_X$ and $P(\, \cdot \mid X) =
  P_0(\, \cdot \mid T_0, \psi_X)$. Thus, $T(X \cup \{\psi\}) = T^X + \psi_X
  \wedge \psi$, so by Proposition \ref{P:any-psi}, we have $P_0(\ph \mid T_0,
  \psi_X \wedge \psi) = 1$. Let $T' = T_0 + \psi_X$, so that $T' + \psi = T_0 +
  \psi_X \wedge \psi$. Then $T' \in \ante P_0$ and $P_0(\ph \mid T', \psi) = 1$.
  By the rule of material implication for $P_0$, it follows that $P_0(\psi \to
  \ph \mid T') = 1$, which implies $P(\psi \to \ph \mid X) = P_0(\psi \to \ph
  \mid T_0, \psi_X) = P_0(\psi \to \ph \mid T') = 1$.

  We next verify deductive transitivity. Suppose $P(\ph \mid X) = 1$ and $\ph
  \vdash \psi$. Then $P_0(\ph \mid T_0, \psi_X) = 1$. By the deductive
  transitivity for $P_0$, we have $P_0(\psi \mid T_0, \psi_X) = 1$, which gives
  $P(\psi \mid X) = 1$. Now suppose $X' \in \ante P$, $X' \vdash X$, and $P(\ph
  \mid X) = 1$. Then $P_0(\ph \mid T_0, \psi_X) = 1$ and, since $X' \in \ante
  P$, we get $T(X') = T^{X'} + \psi_{X'}$ and $P(\; \cdot \mid X') = P_0(\;
  \cdot \mid T_0, \psi_{X'})$. Since $X' \vdash X$, we have $T^{X'}, \psi_{X'}
  \vdash \psi_X$. Choose countable $\Phi \subseteq T^{X'}$ such that $\Phi,
  \psi_ {X'} \vdash \psi_X$ and let $\ze = \bigwedge \Phi$. Then $\Phi
  \subseteq T^{X'} \subseteq \tau(P_0)$, so Corollary \ref{C:certainty-closure}
  implies $P_0(\ze \mid T_0) = 1$. Since $T_0 + \psi_{X'} \in \ante P_0$,
  deductive transitivity for $P_0$ gives $P_0(\ze \mid T_0, \psi_{X'}) = 1$. By
  Lemma \ref{L:cond-exist}, we have $T_0 + \psi_ {X'} + \ze \in \ante P_0$.
  Since $T_0 + \psi_{X'} + \ze \vdash T_0, \psi_X$, we may again apply deductive
  transitivity for $P_0$ to obtain $P_0(\ph \mid T_0, \psi_{X'}, \ze) = 1$. By
  the multiplication rule, $P_0(\ze \wedge \ph \mid T_0, \psi_{X'}) = 1$. By
  Proposition \ref{P:cert-cl-conv}, we obtain $P_0(\ph \mid T_0, \psi_{X'}) =
  P_0(\ze \wedge \ph \mid T_0, \psi_{X'}) = 1$, and this shows that $P(\ph \mid
  X') = 1$.

  To show that $P$ satisfies the addition rule, suppose $X \vdash \neg (\ph
  \wedge \psi)$ and assume two of the probabilities in \eqref{add-rule} exist.
  Then $X \in \ante P$, so that $T(X) = T^X + \psi_X$ and $P(\, \cdot \mid X) =
  P_0(\, \cdot \mid T_0, \psi_X)$. Thus, two of the probabilities in the
  following equation exist:
  \[
    P_0(\ph \vee \psi \mid T_0, \psi_X)
      = P_0(\ph \mid T_0, \psi_X) + P_0(\psi \mid T_0, \psi_X).
  \]
  By the addition rule for $P_0$, so does the third, and the above equation
  holds. Hence, all three probabilities in \eqref{add-rule} exist and 
  \eqref{add-rule} holds. The proof that $P$ satisfies the multiplication rule
  is similar.

  Finally, suppose $P(\ph_n \mid X)$ exists and $X, \ph_n \vdash \ph_{n + 1}$
  for all $n$. We want to show that \eqref{cont-rule} holds. Since $X \in \ante
  P$, we have $T(X) = T^X + \psi_X$ and $P(\; \cdot \mid X) = P_0(\; \cdot \mid
  T_0, \psi)$. Thus, $P_0(\ph_n \mid T_0, \psi)$ exists for all $n$.

  First assume $P_0(\ph_n \mid T_0, \psi) = 0$ for all $n$. Let $\psi_n =
  \bigvee_1^n \ph_j$. Then $P_0(\psi_n \mid T_0, \psi) = 0$ by Proposition
  \ref{P:certainty-closure}. Therefore, $P_0(\bigvee_n \psi_n \mid T_0, \psi) =
  0$ by the continuity rule. But $\bigvee_n \psi_n \equiv \bigvee_n \ph_n$, so
  $P_0(\bigvee_n \ph_n \mid T_0, \psi) = 0$, which implies $P (\bigvee_n \ph_n
  \mid X) = 0$, and \eqref{cont-rule} holds in this case.

  Now assume there exists $n_0$ such that $P_0(\ph_{n_0} \mid T_0, \psi) > 0$.
  By Remark \ref{R:rel-neg}, we have that $P$ satisfies Proposition
  \ref{P:rel-neg}. Thus, $P(\ph_n \mid X) \le P(\ph_ {n + 1} \mid X)$, which
  implies $P_0(\ph_n \mid T_0, \psi) \le P_0(\ph_{n + 1} \mid T_0, \psi)$.
  Hence, $P_0(\ph_n \mid T_0, \psi) > 0$ for all $n \ge n_0$. Since $X \vdash
  \ph_n \to \ph_{n + 1}$, we have $P(\ph_n \to \ph_{n + 1} \mid X) = 1$, which
  gives $P_0(\ph_n \to \ph_{n + 1} \mid T_0, \psi) = 1$. From Proposition
  \ref{P:add-to-root}, it follows that $P_0(\ph_{n + 1} \mid T_0, \psi, \ph_n) =
  1$. Therefore, by Theorem \ref{T:cont-rule}, we have $P_0(\bigvee_{n_0}^\infty
  \ph_n \mid T_0, \psi) = \lim_n P_0(\ph_n \mid T_0, \psi)$, which implies $P
  (\bigvee_{n_0}^\infty \ph_n \mid X) = \lim_n P(\ph_n \mid X)$. But $\bigvee_
  {n_0}^\infty \ph_n \equiv_X \bigvee_n \ph_n$, so \eqref{cont-rule} holds in
  this case as well.
\end{proof}

\begin{prop}\label{P:lift-complete}
  If $P_0$ is a complete pre-theory, then $\bfL(P_0)$ is complete.
\end{prop}

\begin{proof}
  Let $P_0$ be a complete pre-theory with root $T_0$ and let $P = \bfL(P_0)$.
  Proposition \ref{P:lift-entire} implies $P$ is entire. Suppose $P(\ph \mid X)
  = p$ and $P (\psi \mid X) = q$. Then $T(X) = T^X + \psi_X$, $P_0(\ph \mid T_0,
  \psi_X) = p$, and $P_0(\psi \mid T_0, \psi_X) = q$. Since $P_0$ is complete,
  $P_0(\ph \wedge \psi \mid T_0, \psi_X) = r$ for some $r$. Thus, $P(\ph \wedge
  \psi \mid X) = r$, and $P$ satisfies Definition \ref{D:complete}(i).

  Now suppose $X \in \ante P$ and $X \cup \{\ph\} \in \ante P$. Then $T(X) =
  T^X + \psi_X$, so that $T(X \cup \{\ph\}) = T(X) + \ph = T^X + \psi_X + \ph =
  T^X + \psi_X \wedge \ph$. By Proposition \ref{P:any-psi}, we have $\psi_X
  \wedge \ph \in \AF(P_0)$. Thus, $T_0 + \psi_X \wedge \ph \in \ante P_0$. Since
  $T_0 + \psi_X \wedge \ph \equiv (T_0 + \psi_X) \cup \{\ph\}$, the rule of
  logical equivalence gives $(T_0 + \psi_X) \cup \{\ph\} \in \ante P_0$. But
  $T_0 + \psi_X \in \ante P_0$ and $P_0$ is complete, so $P_0(\ph \mid T_0,
  \psi_X)$ exists. Hence, $P(\ph \mid X) = P_0(\ph \mid T_0, \psi_X)$ exists, so
  that $P$ satisfies Definition \ref{D:complete}(ii).
\end{proof}

\begin{prop}\label{P:lift-semi-closed}
  If $P_0$ is a pre-theory, then $\bfL(P_0)$ is semi-closed.
\end{prop}

\begin{proof}
  Let $P_0$ be a pre-theory with root $T_0$ and let $P = \bfL(P_0)$. Proposition
  \ref{P:lift-entire} implies $P$ is entire. Suppose $\ol P(\ph \mid X) = p$
  whenever $\ol P$ is a completion of $P$. Since $P_0$ is a pre-theory, it is
  semi-closed. It therefore has a completion. Let $\ol P_0$ be a completion of
  $P_0$. Corollary \ref{C:compl-restrict} implies $P_0' = \ol P_0 \dhl_{T_0}$ is
  a completion of $P_0$ that is also a pre-theory with root $T_0$. Let $P' =
  \bfL(P_0')$. Then Proposition \ref{P:lift-complete} implies $P'$ is complete
  and Proposition \ref{P:nested-lifts} implies $P \subseteq P'$. Hence, $P'$ is
  a completion of $P$, so by supposition, $P'(\ph \mid X) = p$.

  By Corollary \ref{C:ind-ext-str-conn}, we have $X \cao \ante P$. Choose $X'
  \in \ante P$ and $\psi' \in \cF$ such that $X \equiv X' \cup \{\psi'\}$. Then
  $T(X) = T(X') + \psi' = T^{X'} + \psi_{X'} + \psi' = T^{X'} + \psi$, where
  $\psi = \psi_{X'} \wedge \psi'$. From Proposition \ref{P:any-psi}, it follows
  that $\psi \in \AF(P_0)$. As in the proof of Proposition \ref{P:nested-lifts},
  we have $[T_0, \tau(P_0)] \subseteq [T_0, \tau(P_0')]$ and $\AF(P_0) \subseteq
  \AF(P_0')$. Hence, $T^{X'} \in [T_0, \tau(P_0')]$ and $\psi \in \AF(P_0')$,
  which implies $P_0'(\ph \mid T_0, \psi) = p$. But $P_0' \subseteq \ol P_0$, so
  $\ol P_0(\ph \mid T_0, \psi) = p$. Since $\ol P_0$ was arbitrary and $P_0$ is
  semi-closed, the rule of inductive extension gives $P_0(\ph \mid T_0, \psi) =
  p$. Hence, $P(\ph \mid X) = P(\ph \mid T^{X'}, \psi) = p$, which shows that
  $P$ satisfies the rule of induction extension and is therefore semi-closed.
\end{proof}

\begin{prop}\label{P:lift-closed}
  If $P_0$ is a pre-theory, then $\bfL(P_0)$ is a closed extension of $P_0$.
\end{prop}

\begin{proof}
  Let $P_0$ be a pre-theory with root $T_0$ and let $P = \bfL(P_0)$. Proposition
  \ref{P:lift-extends} implies $P$ is an extension of $P_0$, and Proposition
  \ref{P:lift-semi-closed} implies $P$ is semi-closed. Let $S \subseteq \cF$ be
  nonempty and assume $P(\th \mid X) = 1$ for all $\th \in S$. Then $X \in \ante
  P$ and $P_0(\th \mid T_0, \psi_X) = 1$ for all $\th \in S$. By the rule of
  material implication, $P_0(\psi_X \to \th \mid T_0) = 1$ for all $\th \in S$.
  Hence, if we define $S' = \{\psi_X \to \th \mid \th \in S\}$ and $T' = T^X +
  S'$, then Proposition \ref{P:simple-tau(P)} implies $S' \subseteq \tau(P_0)$,
  so that $T' \in [T_0, \tau(P_0)]$. By Lemma \ref{L:omit-psi}, we have $X \cup
  S \equiv T^X + \psi_X + S = T^X + \psi_X + S' = T' + \psi_X$. It therefore
  follows that $X \cup S \in \ante P$ and $P (\ph \mid X \cup S) = p$ if and
  only if $P_0(\ph \mid T_0, \psi_X) = p$, which holds if and only if $P(\ph
  \mid X) = p$. This shows that $P$ satisfies the rule of deductive extension
  and is therefore closed.
\end{proof}

\begin{proof}[Proof of Theorem \ref{T:theory-defn}]
  Let $P_0$ be a pre-theory with root $T_0$ and let $P = \bfL(P_0)$. Proposition
  \ref{P:lift-closed} shows that $P$ is a closed extension of $P_0$. Let $P'$ be
  another closed extension of $P_0$ and suppose $P(\ph \mid X) = p$. Then
  $P_0(\ph \mid T_0, \psi_X) = p$, which implies $P'(\ph \mid T_0, \psi_X) = p$.
  Since $T^X \subseteq \tau(P_0)$, we have $P_0(\th \mid T_0) = 1$ for all $\th
  \in T^X$. Deductive transitivity gives $P_0(\th \mid T_0, \psi_X) = 1$ for all
  $\th \in T^X$, which implies $P'(\th \mid T_0, \psi_X) = 1$ for all $\th \in
  T^X$. Since $P'$ is closed, the rule of deductive extension gives $P'(\ph \mid
  T_0, \psi_X, T^X) = 1$. But $T_0 \subseteq T^X$, so $T_0 +  \psi_X + T^X = T^X
  + \psi_X \equiv X$, and the rule of logical equivalence gives $P'(\ph \mid X)
  = p$.
\end{proof}

\subsection{Characterizing inductive theories}

Recall the notation $\bfP(P_0)$, established in Section \ref{S:ind-theories}.
The proof of Theorem \ref{T:theory-defn} shows that $\bfP(P_0) = \bfL(P_0)$, so
that $\bfP(P_0)$ is simply the lift of $P_0$. In particular, this gives us an
explicit construction of $\bfP(P_0)$ from $P_0$. We will use this explicit
construction to prove Corollary \ref{C:theory-defn} and Theorem
\ref{T:theory-char}. Henceforth, we will drop the notation $\bfL(P_0)$, and
write only $\bfP(P_0)$ instead.

\begin{proof}[Proof of Corollary \ref{C:theory-defn}]
  Let $P_0, P_0'$ be pre-theories with roots $T_0, T_0'$, respectively, and let
  $P = \bfP(P_0)$ and $P' = \bfP(P_0')$.

  First assume $P = P'$. Then $T_0 \in \ante P_0 \subseteq \ante P = \ante P'$.
  We may therefore choose $T' \in [T_0', \tau(P_0')]$ and $\psi' \in \AF(P_0')$
  such that $T_0 = T' + \psi'$. Then $T_0' \subseteq T' \subseteq T' + \psi' =
  T_0$. By reversing the roles of $P_0$ and $P_0'$, we also have $T_0 \subseteq
  T_0'$, showing that $T_0 = T_0'$. Thus, by Proposition \ref{P:lift-unique}, we
  have $P_0 = P \dhl_{T_0} = P' \dhl_{T_0'} = P_0'$.

  Now assume $P_0 \subset P_0' \subseteq P$. Then $P$ is a closed extension of
  $P_0'$, so Theorem \ref{T:theory-defn} implies $P' \subseteq P$. On the other
  hand, $P'$ is a closed extension of $P_0'$ and $P_0 \subset P_0'$, so $P'$ is
  a closed extension of $P_0$, giving $P \subseteq P'$. Thus, $P = P'$, so by
  the above $P_0 = P_0'$, a contradiction.
\end{proof}

\begin{prop}\label{P:lift-connected}
  If $P_0$ is a pre-theory, then $\bfP(P_0)$ is connected.
\end{prop}

\begin{proof}
  Let $P_0$ be a pre-theory with root $T_0$ and let $P = \bfP(P_0)$. We will
  show that $P_0$ is a basis for $P$. First note that $P_0$ is strongly
  connected and $P_0 \subseteq P$. Now let $X \in \ante P$ be given. Define
  $\wh X = T_0 + \psi_X$ and $S = T^X$. Then $\wh X \in \ante P_0$. Since $T^X
  \subseteq \tau(P_0)$, we have $P_0(\th \mid T_0) = 1$ for all $\th \in S$. By
  deductive transitivity, $P_0(\th \mid \wh X) = 1$ for all $\th \in S$. Hence,
  $S \subseteq \tau(P_0, \wh X)$. Lastly, from $T_0 \subseteq T^X$, it follows
  that $X \equiv T^X + \psi_X = T_0 + T^X + \psi_X \equiv \wh X \cup S$.
\end{proof}

\begin{proof}[Proof of Theorem \ref{T:theory-char}]
  By Theorem \ref{T:theory-defn} and Proposition \ref{P:lift-connected}, we have
  (i) implies (ii). For the converse, let $P$ be closed and connected. Let $T_0$
  be the root of $P$ and define $P_0 = P \dhl_{T_0}$. By Theorem
  \ref{T:restrict-inherit}, the set $P_0$ is semi-closed. By construction, $P_0$
  is strongly connected with root $T_0$. Thus, $P_0$ is a pre-theory. We will
  show that $P = \bfP(P_0)$.

  For notational simplicity, let $P' = \bfP(P_0)$. Since $P$ is a closed
  extension of $P_0$, Theorem \ref{T:theory-defn} gives $P' \subseteq P$.
  Suppose $P(\ph \mid X) = p$. By Proposition \ref{P:semi-cl-conn}, we may
  choose $T \in [T_0, \tau(P)]$ and $\psi \in \cF$ such that $X \equiv T +
  \psi$ and $T_0 + \psi \in \ante P$. Since $T \subseteq \tau(P)$, we have
  $P(\th \mid T_0) = 1$ for all $\th \in T$. By deductive transitivity, $P(\th
  \mid T_0, \psi) = 1$ for all $\th \in T$. Since $P$ is closed, the rule of
  deductive extension implies $P(\; \cdot \mid T_0, \psi) = P(\; \cdot \mid
  T_0, \psi, T)$. But $T_0 + \psi + T = T + \psi \equiv X$. Hence, $P(\; \cdot
  \mid T_0, \psi) = P(\; \cdot \mid X)$, and it therefore follows that $P(\ph
  \mid T_0, \psi) = p$. Since $P_0 = P \dhl_{T_0}$, this gives $P_0(\ph \mid
  T_0, \psi) = p$. Finally, since $P'$ is the lift of $P_0$, we have $P'(\ph
  \mid X) = P'(\ph \mid T, \psi) = P_0(\ph \mid T_0, \psi) = p$.
\end{proof}

\begin{rmk}\label{R:ante-cao}
  Note that by Theorem \ref{T:theory-char} and Proposition \ref{P:lift-unique},
  if $P$ is an inductive theory, then $P = \bfP(P_0)$, where $P_0 = P \dhl_
  {T_0}$. Also note that by Proposition \ref{P:simple-tau(P)}, we have $\tau(P)
  = \tau (P_0)$. Hence, every $X \in \ante P$ satisfies $X \cao [T_0, \tau(P)]$.
\end{rmk}

\section{Generating inductive theories}\label{S:ind-cond}

In the first half of this section, we prove Theorem \ref{T:theory-gen-defn}. One
might think that we could do this the same way we would do it for the deductive
calculus. Namely, we could try to prove that (i) the intersection of inductive
theories is an inductive theory; therefore, (ii) the intersection of all
inductive theories that contain $Q$ is the smallest inductive theory containing
$Q$. Unfortunately, as we will see in Section \ref{S:D-conseq-need2} (see Remark
\ref{R:intersect-th-fail}), it turns out that (i) is false. (Surprisingly,
though, (ii) is still true, provided we pay special attention to the root of
$Q$.) We therefore have to find a different proof method.

As it turns out, the intersection of pre-theories (with a common root) is a
pre-theory. This will be the key to our proof, but it will require several
preliminary definitions and results. When we are done proving Theorem 
\ref{T:theory-gen-defn}, we present a partial converse to the rule of logical
implication (see Theorem \ref{T:log-impl-iff}).

In the second half of this section, we generalize inductive derivability to
something that we call ``inductive conditions''. This allows us to reason with
more general statements, such as $Q(\ph \mid X) > 1/2$.

\subsection{Strongly connected equivalence}

Our first preliminary on the way to the proof of Theorem \ref{T:theory-gen-defn}
shows that, when extending a consistent set to an inductive theory, we are never
required to change roots.

\begin{prop}\label{P:cons-match-root}
  Every consistent set can be extended to an inductive theory with the same
  root. More specifically, let $Q$ be consistent with root $T_0$ and let $P$ be
  an inductive theory with $Q \subseteq P$. Then $T_0 \in \ante P$ and $P \dhl_
  {T_0}$ is a pre-theory with root $T_0$ that satisfies $Q \subseteq \bfP(P
  \dhl_{T_0}) \subseteq P$.
\end{prop}

\begin{proof}
  Let $Q$ be consistent with root $T_0$ and let $P$ be an inductive theory with
  $Q \subseteq P$. Since $T_0$ is the root of $Q$, we have $T_0 = T(X_0)$ for
  some $X_0 \in \ante Q$. But $Q \subseteq P$, so $X_0 \in \ante P$ and, by the
  rule of logical equivalence, $T_0 \in \ante P$. Let $P_0' = P \dhl_{T_0}$.
  Then $P_0'$ is strongly connected with root $T_0$ and, by Theorem 
  \ref{T:restrict-inherit}, the set $P_0'$ is semi-closed. Therefore, $P_0'$ is
  a pre-theory, and we may define $P' = \bfP(P_0')$. Note that $P'$ is an
  inductive theory with root $T_0$. Since $P_0' \subseteq P$ and $P$ is an
  inductive theory, Theorem \ref{T:theory-defn} implies $P' \subseteq P$. It
  remains only to show that $Q \subseteq P'$.

  Suppose $Q(\ph \mid X) = p$. We want to show that $P'(\ph \mid X) = p$. Since
  $P'$ is the lift of $P_0'$, we must find $T \in [T_0, \tau(P_0')]$ and $\psi
  \in \cF$ such that $X \equiv T + \psi$ and $P_0'(\ph \mid T_0, \psi) = p$. Let
  $\wh Q$ be a basis for $Q$. Choose $\wh X \in \ante \wh Q$ and $S \subseteq
  \tau(\wh Q; \wh X)$ such that that $X \equiv \wh X \cup S$. Proposition
  \ref{P:basis-root} implies that $T_0$ is the root of $\wh Q$. Hence, we may
  choose $\psi \in \cF$ such that $\wh X \equiv T_0 + \psi$. Define $S' = 
  \{\psi \to \th \mid \th \in S\}$ and $T = T_0 + S'$. By Lemma 
  \ref{L:omit-psi}, we have $X \equiv \wh X \cup S \equiv T_0 + \psi + S = T_0 +
  \psi + S' = T + \psi$, so it suffices to show $S' \subseteq \tau(P_0')$ and
  $P_0'(\ph \mid T_0, \psi) = p$.

  By Proposition \ref{P:simple-tau(P)}, in order to show that $S' \subseteq \tau
  (P_0')$, we must show that $P_0' (\psi \to \th \mid T_0) = 1$ for all $\th \in
  S$. Let $\th \in S$. Then $\wh Q(\th \mid \wh X) = 1$. But $\wh Q \subseteq Q
  \subseteq P$, so $P(\th \mid \wh X) = P(\th \mid T_0, \psi) = 1$. By the rule
  of material implication, $P(\psi \to \th \mid T_0) = 1$, which implies $P_0'
  (\psi \to \th \mid T_0) = 1$.

  Finally, since $\wh Q \subseteq Q \subseteq P$, we have $\tau(\wh Q; \wh X)
  \subseteq \tau(P; \wh X)$. Hence, $S \subseteq \tau(P; \wh X)$. Since $P$ is
  closed, deductive extension implies $P(\; \cdot \mid \wh X) = P(\;
  \cdot \mid \wh X \cup S)$. But $\wh X \cup S \equiv X$. Also, $Q \subseteq P$,
  so we have $P(\ph \mid X) = p$. Using the rule of logical equivalence, we have
  $P (\ph \mid T_0, \psi) = P (\ph \mid \wh X) = P(\ph \mid X) = p$. Therefore,
  since $P_0' = P \dhl_{T_0}$, we obtain $P_0'(\ph \mid T_0, \psi) = p$.
\end{proof}

Our next result says that connected sets are, in a certain sense, logically
unnecessary. It is enough to only consider strongly connected sets. To make this
precise, we first define what it means for two subsets of $\cF^\IS$ to be
logically equivalent.

Let $Q, Q' \subseteq \cF^\IS$ be connected. We say that $Q$ and $Q'$ are
\emph{equivalent}, written $Q \equiv Q'$,
  \index{equivalent}%
  \symindex{$\equiv$}%
if, for all inductive theories $P$, we have $Q \subseteq P$ if and only if $Q'
\subseteq P$. After proving Theorem \ref{T:theory-gen-defn}, we will be able to
speak of the inductive theory generated by $Q$. At that point, we will have the
much more natural characterization of equivalence given in Proposition
\ref{P:ind-equiv-char}.

\begin{thm}\label{T:str-conn-enough}
  If $Q$ is connected, then there exists strongly connected $Q_0$ with the same
  root as $Q$ such that $Q_0 \equiv Q$.
\end{thm}

\begin{proof}
  Let $Q$ be connected with root $T_0$ and let $\wh Q$ be a basis for $Q$. For
  each $X \in \ante Q$, choose $\wh X \in \ante \wh Q$, $S_X \subseteq \tau (\wh
  Q; \wh X)$, and $\psi_X \in \cF$ such that $X \equiv \wh X \cup S_X$ and $\wh
  X \equiv T_0 + \psi_X$. For $X \in \ante \wh Q \subseteq \ante Q$, assume we
  have chosen $\wh X = X$ and $S_X = \emp$. Define
  \[
    Q_0 = \{(T_0 + \psi_X, \ph, p) \mid (X, \ph, p) \in Q\}.
  \]
  Note that if $\wh X \in \ante \wh Q \subseteq \ante Q$, then $S_{\wh X} =
  \emp$ and $\wh X \equiv T_0 + \psi_{\wh X}$. Thus, if $(\wh X, \ph, p) \in
  \wh Q \subseteq Q$, then $(T_0 + \psi_{\wh X}, \ph, p) \in Q_0$.

  Let $P$ be an inductive theory. Assume $Q_0 \subseteq P$. Let $(X, \ph, p)
  \in Q$. Then $(T_0 + \psi_X, \ph, p) \in Q_0$, which implies $P(\ph \mid T_0,
  \psi_X) = p$. Let $\th \in S_X \subseteq \tau(\wh Q; \wh X)$. Then $(\wh X,
  \th, 1) \in \wh Q$, so that $(T_0 + \psi_{\wh X}, \th, 1) \in Q_0$, which
  implies $P(\th \mid T_0, \psi_{\wh X}) = 1$. But $T_0 + \psi_{\wh X} \equiv
  \wh X \equiv T_0 + \psi_X$, so by the rule of logical equivalence, $P(\th
  \mid T_0, \psi_X) = 1$. Since $\th$ was arbitrary, deductive extension gives
  $P(\ph \mid T_0, \psi_X, S_X) = p$. But $T_0 + \psi_X + S_X \equiv \wh X \cup
  S_X \equiv X$, so the rule of logical equivalence gives $P(\ph \mid X) = p$,
  showing that $Q \subseteq P$.

  Now assume $Q \subseteq P$. Let $(Y, \ph, p) \in Q_0$. Choose $(X, \ph, p)
  \in Q$ such that $Y = T_0 + \psi_X$. Then $P(\ph \mid X) = p$. Note that $\wh
  X \equiv Y$, so that $X \equiv Y \cup S_X$. If $S_X = \emp$, then $X \equiv
  Y$, so by the rule of logical equivalence, $P(\ph \mid Y) = p$. Assume $S_X
  \ne \emp$. Let $\th \in S_X \subseteq \tau(\wh Q; \wh X)$. Then $(\wh X, \th
  , 1) \in \wh Q$. But $\wh Q \subseteq Q \subseteq P$, so $P(\th \mid \wh X) =
  1$. Since $\wh X \equiv Y$, the rule of logical equivalence implies $P(\th
  \mid Y) = 1$. Since $\th$ was arbitrary, deductive extension gives $Y \cup
  S_X \in \ante P$ and $P(\; \cdot \mid Y) = P(\; \cdot \mid Y, S_X)$. Since $X
  \equiv Y \cup S_X$, we get $P(\ph \mid Y, S_X) = p$. Therefore, $P(\ph \mid
  Y) = p$, showing that $Q_0 \subseteq P$.
\end{proof}

\subsection{Intersections of inductive sets}

We now get to the heart of the matter, which is the intersection of subsets of
$\cF^\IS$. The next result shows that the closure properties are all preserved
under such intersections. It is straightforward to verify that strong
connectivity is also preserved. Hence, as a corollary, we find that the
intersection of pre-theories (with a common root) is a pre-theory. After
establishing that result, we give the proof of Theorem \ref{T:theory-gen-defn}.

\begin{thm}\label{T:intersect-cl}
  Let $\cC \subseteq \fP \cF^\IS$ be nonempty.
  \begin{enumerate}[(i)]
    \item If each set in $\cC$ is admissible, then $\bigcap \cC$ is admissible.
    \item If each set in $\cC$ is entire, then $\bigcap \cC$ is entire.
    \item If each set in $\cC$ is semi-closed, then $\bigcap \cC$ is
          semi-closed.
    \item If each set in $\cC$ is closed, then $\bigcap \cC$ is closed.
  \end{enumerate}
\end{thm}

\begin{proof}
  Let $P = \bigcap \cC$. Assume each set in $\cC$ is admissible. Suppose $(X,
  \ph, p) \in P$, $X' \equiv X$, and $\ph' \equiv_X \ph$. Let $P' \in \cC$, so
  that $P'(\ph \mid X) = p$. Since $P'$ is admissible, $P'(\ph' \mid X') = p$.
  Since $P'$ was arbitrary, $(X', \ph', p) \in P$. Now suppose $(X', \ph', p')
  \in P$. Choose $P' \in \cC$. Then $P'(\ph' \mid X') = p' = p$. Thus, $P$ is
  admissible.

  Now assume each set in $\cC$ is entire. Let $X \in \ante P$ and $X \vdash
  \ph$. Choose $\ph' \in \cF$ and $p \in [0, 1]$ such that $(X, \ph', p) \in P$.
  Let $P' \in \cC$. Then $P'(\ph' \mid X) = p$, so that $X \in \ante P'$. Since
  $P'$ is entire and $X \vdash \ph$, the rule of logical implication gives
  $P'(\ph \mid X) = 1$. Since $P'$ was arbitrary, we have $P(\ph \mid X) = 1$,
  showing that $P$ satisfies the rule of logical implication. Similar proofs
  show that $P$ satisfies rules (R3)--(R7), and therefore, $P$ is entire.

  Assume each set in $\cC$ is semi-closed. Suppose $\ol P(\ph \mid X) = p$ for
  every completion $\ol P$ of $P$. Let $P' \in \cC$ and let $\ol{P'}$ be a
  completion of $P'$. Since $P \subseteq P' \subseteq \ol{P'}$, it follows that
  $\ol{P'}$ is also a completion of $P$. Thus, $\ol{P'}(\ph \mid X) = p$. Since
  $\ol{P'}$ was arbitrary and $P'$ is semi-closed, we have $P'(\ph \mid X) = p$.
  Since $P'$ was arbitrary, this gives $P(\ph \mid X) = p$, and $P$ is
  semi-closed.

  Finally, assume each set in $\cC$ is closed. Suppose $S \subseteq \cF$ is
  nonempty and $P(\th \mid X) = 1$ for all $\th \in S$. That is, $\emp \ne S
  \subseteq \tau(P; X)$. Let $P' \in \cC$. Since $P \subseteq P'$, we have
  $\tau(P; X) \subseteq \tau(P'; X)$, so that $\emp \ne S \subseteq \tau(P';
  X)$. Since $P'$ is closed, deductive extension implies $X \cup S \in \ante P'$
  and $P'(\; \cdot \mid X, S) = P'(\; \cdot \mid X)$. Since $P'$ was arbitrary,
  we have
  \begin{align*}
    P(\ph \mid X, S) = p
      &\quad\text{iff}\quad P'(\ph \mid X, S) = p \text{ for all $P' \in \cC$}\\
      &\quad\text{iff}\quad P'(\ph \mid X) = p \text{ for all $P' \in \cC$}\\
      &\quad\text{iff}\quad P(\ph \mid X) = p.
  \end{align*}
  Since $S$ is a nonempty subset of $\tau(P; X)$, it follows that $\tau(P; X)$
  is nonempty, which implies $X \in \ante P$. Hence, by the above, $X \cup S \in
  \ante P$ and $P(\; \cdot \mid X, S) = P(\; \cdot \mid X)$, showing that $P$ is
  closed.
\end{proof}

\begin{cor}\label{C:intersect-cl}
  Let $\cC^0$ be a nonempty set of pre-theories with common root $T_0$ and
  define $P_0 = \bigcap \cC^0$. Then $P_0$ is a pre-theory with root $T_0$.
\end{cor}

\begin{proof}
  By Proposition \ref{T:intersect-cl}, the set $P_0$ is semi-closed. Let $X \in
  \ante P_0$. Choose $\ph \in \cF$ and $p \in [0, 1]$ such that $P_0(\ph \mid
  X) = p$. Let $P \in \cC^0$. Then $P(\ph \mid X) = p$, so that $X \in \ante P$.
  Since $P$ is strongly connected with root $T_0$, we may choose $\psi \in \cF$
  such that $X \equiv T_0 + \psi$. Since $X$ was arbitrary, $P_0$ is strongly
  connected with root $T_0$.
\end{proof}

\begin{proof}[Proof of Theorem \ref{T:theory-gen-defn}]
  Let $Q$ be consistent. Then $Q$ is connected and we may choose an inductive
  theory $P'$ such that $Q \subseteq P'$. Let $T_0$ be the root of $Q$. By
  Proposition \ref{P:cons-match-root}, we have $T_0 \in \ante P'$ and, if we
  define $P_0' = P' \dhl_{T_0}$, then $P_0'$ is a pre-theory with root $T_0$. By
  Theorem \ref{T:str-conn-enough}, we may choose strongly connected $Q_0$ with
  root $T_0$ such that $Q_0 \equiv Q$. Let $\cC^0$ be the set of all
  pre-theories with root $T_0$ that contain $Q_0$.

  Since $Q_0 \equiv Q$ and $Q \subseteq P'$, we have $Q_0 \subseteq P'$. Since
  $Q_0$ is strongly connected with root $T_0$, this gives $Q_0 \subseteq P_0'$.
  Hence, $\cC^0$ is nonempty. Let $P_0 = \bigcap \cC^0$. Corollary 
  \ref{C:intersect-cl} implies $P_0$ is a pre-theory with root $T_0$. Define $P
  = \bfP(P_0)$. Since $Q_0$ is a subset of every element of $\cC^0$, it follows
  that $Q_0 \subseteq P_0 \subseteq P$. Therefore, since $Q_0 \equiv Q$, we have
  $Q \subseteq P$.

  To show that $P$ is the smallest such inductive theory, let $P''$ be an
  arbitrary inductive theory with $Q \subseteq P''$. As above, if $P_0'' = P''
  \dhl_{T_0}$, then $P_0''$ is a pre-theory with root $T_0$ and $Q_0 \subseteq
  P_0''$. Hence, $P_0'' \in \cC^0$, so that $P_0 \subseteq P_0''$, which implies
  $P \subseteq \bfP(P_0'')$. But $P_0'' \subseteq P''$, so by Theorem 
  \ref{T:theory-defn}, we have $\bfP(P_0'') \subseteq P''$, and therefore, $P
  \subseteq P''$.
\end{proof}

\subsection{A converse to the rule of logical implication}

Having proved Theorem \ref{T:theory-gen-defn}, we now have the notation $\bfP_Q$
at our disposal. With this notation, we are able to give a more natural
characterization of the equivalence of two subsets of $\cF^\IS$. We also give a
definition that we will need later, and provide a partial converse to the rule
of logical implication.

\begin{prop}\label{P:ind-equiv-char}
  Let $Q, Q' \subseteq \cF^\IS$ be connected. Then $Q \equiv Q'$ if and only if
  either both are inconsistent or both are consistent and $\bfP_Q = \bfP_{Q'}$
\end{prop}

\begin{proof}
  Assume $Q \equiv Q'$. Then $Q \subseteq P$ if and only if $Q' \subseteq P$ for
  all inductive theories $P$. Hence, $Q$ is consistent if and only if $Q'$ is
  consistent. Suppose both are consistent. Since $Q \subseteq \bfP_Q$, we have
  $Q' \subseteq \bfP_Q$, which implies $\bfP_{Q'} \subseteq \bfP_Q$. Similarly,
  $\bfP_Q \subseteq \bfP_{Q'}$, and therefore, $\bfP_Q = \bfP_{Q'}$.

  For the converse, if both are inconsistent, then neither can be extended to an
  inductive theory, so they are vacuously equivalent. Assume both are consistent
  and $\bfP_Q = \bfP_{Q'}$. Let $P$ be an inductive theory with $Q \subseteq P$.
  Then $\bfP_Q \subseteq P$. Thus, $Q' \subseteq \bfP_{Q'} = \bfP_Q \subseteq
  P$. Similarly, $Q' \subseteq P$ implies $Q \subseteq P$, showing that $Q
  \equiv Q'$.
\end{proof}

If $Q$ is consistent, we define $T(Q) = \tau(\bfP_Q)$. We also denote $T(Q)$ by
$T_Q$.
  \symindex{$T(Q), T_Q$}%
By Proposition \ref{P:tau-ded-theory}, the set $T_Q$ is a deductive theory. We
call $T_Q$ the \emph{deductive theory determined by $Q$}.

Note that $\ph \in T_Q$ if and only if $Q \vdash (X, \ph, 1)$ for all $X \in
\ante \bfP_Q$. In other words, $T_Q$ is the set of formulas which, under
$\bfP_Q$, have probability one, regardless of the antecedent. Informally, $T_Q$
represents the deductive hypotheses that are implicit in the set $Q$.

\begin{thm}\label{T:log-impl-iff}
  Suppose $P$ is an inductive theory and let $X \in \ante P$. Then $P(\ph \mid
  X) = 1$ if and only if $X, T_P \vdash \ph$.
\end{thm}

\begin{proof}
  Let $P$ be an inductive theory and $X \in \ante P$. Then $P(\th \mid X) = 1$
  for all $\th \in T_P$. Hence, by the rule of deductive extension, $X \cup T_P
  \in \ante P$ and $P(\; \cdot \mid X) = P(\; \cdot \mid X, T_P)$.

  Suppose $X, T_P \vdash \ph$. Since $X \cup T_P \in \ante P$, the rule of
  logical implication gives $P(\ph \mid X, T_P) = 1$. By the above, $P(\ph \mid
  X) = 1$. For the converse, suppose $P(\ph \mid X) = 1$. Let $T_0$ be the root
  of $P$, so that $P = \bfP(P_0)$, where $P_0 = P \dhl_{T_0}$. We then have $X
  \equiv T^X + \psi_X$ and $P_0(\ph \mid T_0, \psi_X) = 1$. By the rule of
  material implication, $P_0(\psi_X \to \ph \mid T_0) = 1$, which implies
  $\psi_X \to \ph \in T_P$. Therefore, $X, T_P \vdash \psi_X, \psi_X \to \ph
  \vdash \ph$.
\end{proof}

\subsection{Inductive conditions}

At this point, we have proven all the results in Sections \ref{S:ind-theories}
and \ref{S:ind-derivability}. We have therefore fully established and justified
the notation $Q \vdash (X, \ph, p)$. Informally, we think of $Q \vdash (X, \ph,
p)$ as representing a process of derivation, where we take the inductive
statements in $Q$ as our hypotheses, then apply the nine rules of inductive
inference to derive $(X, \ph, p)$. But every hypothesis in $Q$ is a precise
inductive statements of the form $Q(\eta \mid Y) = q$. We are often interested
in using more general hypotheses, such as $Q(\eta \mid Y) > q$. Or we may wish
to hypothesize that $Q$ and everything it entails satisfies a certain symmetry
condition. To allow for these more general hypotheses, we define the following.

An \emph{inductive condition}
  \index{inductive condition}%
is a collection $\cC$ of inductive theories with a common root. An inductive
condition $\cC$ is said to be \emph{consistent}
  \index{consistent}%
if $\cC \ne \emp$. If each $P \in \cC$ has root $T_0$, then we call $T_0$ the 
\emph{root} of $\cC$.
  \index{root}%

For instance, $\cC$ might be a collection of inductive theories $P$, each
satisfying $P(\eta \mid Y) > q$. Or $\cC$ might be a collection where each
member satisfies a given symmetry property. If we wish to simply assume $Q$,
without any generalizations, we can also do that with an inductive condition.
Namely, if $Q$ is connected with root $T_0$, then let $\cC(Q)$, which we also
denote by $\cC_Q$, be defined as the set of inductive theories $P$ with root
$T_0$ such that $Q \subseteq P$. Then $\cC_Q$ also has root $T_0$ and $\cC_Q$ is
consistent if and only if $Q$ is consistent. Moreover, by Theorem
\ref{T:theory-gen-defn}, we have $\bfP_Q \in \cC_Q$ and $\bfP_Q \subseteq P$ for
all $P \in \cC_Q$. Hence, $\bfP_Q = \bigcap \cC_Q$. If we identify $\cC_Q$ with
$Q$, then we can say that $\cC_Q$ generates the inductive theory $\bigcap \cC_Q$.

Generalizing this to arbitrary inductive conditions is not straightforward. The
problem, as mentioned at the beginning of this section, is that the intersection
of inductive theories is not necessarily an inductive theory. What we will show,
however, is that if $\cC$ is an inductive condition with root $T_0$, then there
is a largest inductive theory with root $T_0$ contained in $\bigcap \cC$. If
$\cC = \cC_Q$, then that largest theory is $\bfP_Q$.

Let $\cC$ be an inductive condition with root $T_0$. Define $\cC^0 = \{P \dhl_
{T_0} \mid P \in \cC\}$. By Corollary \ref{C:intersect-cl}, if $\cC$ is
consistent, then $\bigcap \cC^0$ is a pre-theory with root $T_0$. Hence, may
define $\bfP(\cC) = \bfP(\bigcap \cC^0)$. We also denote $\bfP(\cC)$ by
$\bfP_\cC$,
  \symindex{$\bfP(\cC), \bfP_\cC$}%
and call this the \emph{inductive theory generated by $\cC$}. The next result
shows that the inductive theory generated by $\cC$ is indeed the largest
inductive theory contained in $\bigcap \cC$.

\begin{thm}\label{T:theory-gen-IC-defn}
  Let $\cC$ be a consistent inductive condition. Then $\bfP_\cC \subseteq
  \bigcap \cC$. Moreover, if $P$ is an inductive theory with the same root as
  $\cC$ such that $P \subseteq \bigcap \cC$, then $P \subseteq \bfP_\cC$.
\end{thm}

\begin{proof}
  Let $\cC$ be a consistent inductive condition with root $T_0$. Let $P' \in
  \cC$. Then $\bigcap \cC^0 \subseteq \bigcap \cC \subseteq P'$, which implies
  $\bfP_\cC \subseteq P'$. Since $P'$ was arbitrary, this shows $\bfP_\cC
  \subseteq \bigcap \cC$.

  Now suppose $P$ is an inductive theory with root $T_0$ such that $P \subseteq
  \bigcap \cC$. Then $P = \bfP(P_0)$, where $P_0 = P \dhl_{T_0}$. Let $P_0' \in
  \cC^0$ be arbitrary, and choose $P' \in \cC$ such that $P_0' = P' \dhl_{T_0}$.
  Since $P' \in \cC$, it follows that $P \subseteq P'$, which implies $P_0
  \subseteq P_0'$. Since $P_0'$ was arbitrary, this gives $P_0 \subseteq \bigcap
  \cC^0$. Therefore, $P = \bfP(P_0) \subseteq \bfP(\bigcap \cC^0) = \bfP_\cC$.
\end{proof}

As noted earlier, if $Q \subseteq \cF^\IS$ is consistent, then $\bfP_Q =
\bigcap \cC_Q$. Hence, by Theorem \ref{T:theory-gen-IC-defn}, we have $\bfP(Q)
= \bfP (\cC_Q)$.

If $\cC$ is an inductive condition and $\bfP_\cC \in \cC$, then we say the
condition $\cC$ is \emph{determinate}, otherwise $\cC$ is \emph{indeterminate}.
Note that if $Q \subseteq \cF^\IS$ is consistent, then $\cC_Q$ is determinate.
  \index{inductive condition!determinate ---}%
  \index{inductive condition!indeterminate ---}%

If $\cC$ is an inductive condition and $(X, \ph, p) \in \cF^\IS$, we write $\cC
\vdash (X, \ph, p)$ to mean that $\cC$ is consistent and $\bfP_\cC(\ph \mid X)
= p$.
  \index{derivability relation!propositional ---}%
  \symindex{$\cC \vdash (X, \ph, p)$}%
Since $\bfP (Q) = \bfP(\cC_Q)$, we have that $Q \vdash (X, \ph, p)$ if and only
if $\cC_Q \vdash (X, \ph, p)$, so that this new use of the turnstile symbol is
an extension of our previous use.

If $\cC$ and $\cC'$ are inductive conditions with the same root, then $\cC, \cC'
\vdash (X, \ph, p)$ means $\cC \cap \cC' \vdash (X, \ph, p)$. In particular, if
we identify $Q$ and $\cC_Q$, then $Q, \cC \vdash (X, \ph, p)$ means $\cC_Q, \cC
\vdash (X, \ph, p)$. Lastly, we use $\cC, X \vdash \ph$ as shorthand for $\cC
\vdash (X, \ph, 1)$. For example, $Q, \cC, X \vdash \ph$ means that $Q$ and
$\cC$ have the same root, $\cC_Q \cap \cC$ is consistent (that is, nonempty),
and, if $P = \bfP(\cC_Q \cap \cC)$, then $P(\ph \mid X) = 1$.

Finally, if $\cC$ is a consistent inductive condition, we define $T(\cC) = T
(\bfP_\cC)$. We also denoted $T(\cC)$ by $T_\cC$.
  \symindex{$T(\cC), T_\cC$}%

\begin{prop}\label{P:ded-th-ind-cond}
  If $\cC$ is a consistent inductive condition, then $T_\cC = \bigcap \{T_P
  \mid P \in \cC\}$.
\end{prop}

\begin{proof}
  Let $\cC$ be a consistent inductive condition with root $T_0$. Let $P_0' =
  \bigcap \cC^0$ and $P' = \bfP_\cC$, so that $P' = \bfP(P_0')$. Suppose $\th
  \in T_\cC = T_{P'}$. Then $P'(\th \mid T_0) = 1$, so by Theorem 
  \ref{T:theory-gen-IC-defn}, we have $P(\th \mid T_0) = 1$ for all $P \in \cC$.
  By Proposition \ref{P:simple-tau(P)}, we have $\th \in T_P$ for all $P \in
  \cC$, so that $\th \in \bigcap \{T_P \mid P \in \cC\}$.

  Conversely, suppose $\th \in \bigcap \{T_P \mid P \in \cC\}$. Then $P_0(\th
  \mid T_0) = 1$ for all $P_0 \in \cC^0$. Hence, $P_0'(\th \mid T_0) = 1$. As
  above, Proposition \ref{P:simple-tau(P)} gives $\th \in T_{P'}$.
\end{proof}

Recall Theorem \ref{T:log-impl-iff}, which gives a partial converse to the rule
of logical implication. The following result reformulates that in terms of our
new notation and shows that $\cC, X \vdash \ph$ can be rewritten in terms of the
classical derivability relation from Section \ref{S:formulas}.

\begin{prop}
  Suppose $\cC$ is a consistent inductive condition and let $X \in \ante
  \bfP_\cC$. Then $\cC, X \vdash \ph$ if and only if $T_\cC, X \vdash \ph$.
\end{prop}

\begin{proof}
  Let $\cC$ be consistent and let $X \in \ante \bfP_\cC$. Note that $\cC, X
  \vdash \ph$ if and only if $\bfP_\cC(\ph \mid X) = 1$. Also note that $T_\cC =
  T(\bfP_\cC)$. Hence, the result follows immediately from Theorem
  \ref{T:log-impl-iff}.
\end{proof}

We conclude this section with a result that we will need in Chapter 
\ref{Ch:prop-models}.

\begin{prop}\label{P:chop-off-root}
  Let $P$ be an inductive theory with root $T_0$ and let $T_0' \in [T_0, T_P]$.
  Let $P_0' = P \dhl_{T_0'}$. Then $P_0'$ is a pre-theory with root $T_0'$.
  Moreover, if $P' = \bfP(P_0')$, then $T_{P'} = T_P$ and $P' = P \dhl_ {[T_0',
  T_P]}$.
\end{prop}

\begin{proof}
  By Theorem \ref{T:restrict-inherit}, since $P_0'$ is strongly connected, we
  have that $P_0'$ is a pre-theory with root $T_0'$. Let $P' = \bfP(P_0')$. We
  first show that $T_{P'} = T_P$. By Proposition \ref{P:simple-tau(P)}, it
  suffices to show that $P'(\th \mid T_0') = 1$ if and only if $P(\th \mid T_0)
  = 1$. Suppose $P'(\th \mid T_0') = 1$. Then $P(\th \mid T_0') = 1$. But $T_0'
  \subseteq T_P$, so by the rule of deductive extension, we have $P(\th \mid
  T_0) = 1$. Conversely, suppose $P(\th \mid T_0) = 1$. Again by the rule of
  deductive extension, we obtain $P(\th \mid T_0') = 1$, and so $P'(\th \mid
  T_0') = 1$.

  It remains to show that $P' = P \dhl_{[T_0', T_P]}$. Since $P_0' \subseteq P$,
  Theorem \ref{T:theory-defn} implies $P' \subseteq P$. Moreover, every $X \in
  \ante P'$ satisfies $X \cao [T_0', T_{P'}] = [T_0', T_P]$. Hence, $P'
  \subseteq P \dhl_{[T_0', T_P]}$.

  Conversely, suppose $(X, \ph, p) \in P \dhl_{[T_0', T_P]}$. Then $P(\ph \mid
  X) = p$ and $X \cao [T_0', T_P]$. Write $T(X) = T + \psi$, where $T \in [T_0',
  T_P] \subseteq [T_0, T_P]$. Then $p = P(\ph \mid X) = P_0 (\ph \mid T_0,
  \psi)$. Since $P_0 \subseteq P$, we have $P(\ph \mid T_0, \psi) = p$. But $T_0
  \subseteq T_0' \subseteq \tau(P)$, so by the rule of deductive extension, we
  have $P(\ph \mid T_0', \psi) = p$, and this implies $P_0'(\ph
  \mid T_0', \psi) = p$. Since $T_P = T_{P'}$, we have $T \in [T_0', T_{P'}]$.
  Therefore $P'(\ph \mid X) = p$, and this shows $P \dhl_ {[T_0', T_P]}
  \subseteq P'$.
\end{proof}

%% file: prop-semantics.tex

\chapter{Propositional Models}\label{Ch:prop-models}

The logical relationships between deductive and inductive statements in $\cF$
and $\cF^\IS$ are described by the derivability relation $\vdash$ developed in
Chapter \ref{Ch:prop-calc}. This relation is a kind of calculus, based on a set
of inferential rules. The rules themselves depend only on the syntax of the
statements. No interpretation or meaning is given to them, and no such meaning
is necessary to describe these logical relationships.

We can, however, use meanings and interpretations to investigate the logical
relationships between statements. When we do this, we arrive at a different
relation, called the consequence relation, and denoted by $\vDash$. When we
study logical relationships using $\vdash$, we are studying the syntactics of
the logic. When we use $\vDash$, we are studying the semantics.

Meanings are assigned to statements using models. The classical type of
propositional model (which we call a strict model) is one that simply assigns a
truth value (0 or 1) to each propositional variable. Truth values then propagate
to every sentence in $\cF$ via the usual interpretations of $\neg$ and
$\bigwedge$. These strict models originated with Wittgenstein (see \cite[Satz
4.31] {Wittgenstein1922}) and can be visualized as rows in what we now call a
truth table.

In Wittgenstein's view, a strict model represents a logically possible state of
the world. Since we are interested in modeling inductive inference, we wish to
model degrees of uncertainty about the state of the world. For us, then, a model
will be a collection of strict models, together with a set of weights whose
relative magnitudes represent relative degrees of uncertainty. Without loss of
generality, we may assume these weights add up to one. In other words, we will
define a model to be a probability measure on a set of strict models.

Models are used to define the consequence relation $\vDash$. On the deductive
side, we say that $X \vDash \ph$ if, in every model that satisfies $X$, the
sentence $\ph$ is also satisfied. A completely analogous definition holds on the
inductive side when we write $P \vDash (X, \ph, p)$.

In Section \ref{S:models_ded_sem}, we define models and we define the
consequence relation for deductive statements. That is, we define what it means
to say that $X \vDash \ph$, when $X \subseteq \cF$ and $\ph \in \cF$. We prove
that, together, $\vdash$ and $\vDash$ form a sound logical system, meaning that
$X \vdash \ph$ implies $X \vDash \ph$. In other words, if we can use $X$ to
prove $\ph$, then $\ph$ is satisfied in every model of $X$. We also prove that
this logical system is complete, meaning that $X \vDash \ph$ implies $X \vdash
\ph$. In other words, if $\ph$ is satisfied in every model of $X$, then there is
a proof of $\ph$ from $X$. Together, soundness and completeness show that $X
\vdash \ph$ if and only if $X \vDash \ph$. In particular, by Theorem
\ref{T:sig-cpctness}, the consequence relation is $\si$-compact.

These results should be contrasted with the usual approach to (deductive)
semantics in $\cF$. One can define a strict consequence relation using strict
models. It is well known (see Examples \ref{Expl:Karp413} and
\ref{Expl:Karp412}) that both completeness and $\si$-compactness fail in that
case. By overlaying our semantics with a probability measure, we recover both
properties.

In Section \ref{S:ind-sem}, we extend the consequence relation to inductive
statements, and prove both the soundness and completeness of this extension.
These results finalize our description of probability as logic. Probability is a
system of inference on inductive statements. It contains classical propositional
logic as a special case and extends it in two directions: from deductive to
inductive, and from finite conjunctions to countably infinite conjunctions. It
has both semantics and a syntactic calculus. (In Chapter \ref{Ch:pred-logic},
we will repeat these constructions in a predicate language, showing that
probability, as inductive logic, also extends first-order logic in both these
directions.)

In Section \ref{S:every-prob-sp}, we address the relationship between
this logical system of probability and modern, measure-theoretic probability
theory. Modern probability has its origin in Kolmogorov's 1933 manuscript, 
\emph{Foundations of the Theory of Probability} \cite{Kolmogorov1956}. Therein,
Kolmogorov lays out what he calls the axioms of probability. Today, those axioms
take the form of a definition, namely, the definition of a probability space: a
measure space with total mass one. The foundation of modern probability,
therefore, is the probability space.

For us, the probability space is the foundation of our semantics. In Theorem
\ref{T:prob-sp-model-iso}, we show that every probability space is isomorphic to
a semantic model in our logical system. The proof of Theorem
\ref{T:prob-sp-model-iso} exhibits a natural mapping from the outcomes and
events of the given probability space to strict models and sentences,
respectively. With this result, we see that all of modern, measure-theoretic
probability is embedded in our logical system. Probability theory as we know it
today is simply the semantics of a larger system of logical reasoning. (In
Chapter \ref{Ch:pred-logic}, we will extend this embedding to include random
variables. See, for instance, Section \ref{S:RV-symb} and the table of
correspondences immediately preceding Section \ref{S:symb-func}.)

Generally speaking, there is a difference between saying that a model satisfies
a set of statements, and saying that it characterizes them. To say that it
characterizes them is to say that they are the only statements satisfied by the
model. In that case, the model is a perfect semantic reflection of the given set
of statements. Finding a model that characterizes a set of statements allows us
to see the logical structure of that set in a single semantic model.

In Theorem \ref{T:model-ind-th}, we show that models (that is, probability
spaces), are only able to characterize complete inductive theories. The
structure of an incomplete inductive theory cannot be represented by a
probability space. Proposition \ref{P:Del-P-X} tells us why. Namely, if we start
with a collection of inductive statements and conditions, and then draw all
available inferences, we are not led to a $\si$-algebra, but rather to a Dynkin
system. Consequently, we see that Dynkin systems arise naturally and organically
in the study of inductive inference. This fact may offer some insight into why
Dynkin's $\pi$-$\la$ theorem---a purely measure-theoretic result---features so
much more prominently in probability theory than it does in analysis.

In the rest of Section \ref{S:Examples1} and in Section \ref{S:Examples2}, we
present several examples, illustrating and applying many of the ideas presented
in both Chapters \ref{Ch:prop-calc} and \ref{Ch:prop-models}.

Finally, in Section \ref{S:indep}, we introduce the idea of (inductive)
independence---a purely logical and syntactic notion---and then show that it is
semantically characterized by the usual product formula from measure theory. We
then present two examples to illustrate its use.

\section{Models and deductive semantics}\label{S:models_ded_sem}

\subsection{Truth assignments}

Recall that $\B$ denotes the Boolean $\si$-algebra $\{0, 1\}$, whose partial
order is the usual $\le$. The elements $0$ and $1$ are called \emph{truth
values}. If $S$ is a set, then a function $\nu: S \to \B$ is an assignment of
truth values to the elements of $S$. The set of all such functions is denoted by
$\B^S$.
  \symindex{$\B^S$}%

Given an element $s \in S$, we define the \emph{projection} $\pi_s: \B^S \to \B$
by $\pi_s \nu = \nu s$. Let $\cB = \fP \B$. Then $\cB^S$ denotes the product
$\si$-algebra.
  \symindex{$\cB^S$}%
That is, $\cB^S$ is the smallest $\si$-algebra on $\B^S$ such that each $\pi_s$
is $(\cB^S, \cB)$-measurable. In symbols,
\[
  \cB^S = \si(\{\pi_s \mid s \in S\})
    = \si(\{\pi_s^{-1} A \mid s \in S, A \in \cB\}).
\]
A subset of $\B^S$ is called a \emph{cylinder set} if it has the form
\[
  \pi_{s_1}^{-1} A_1 \cap \cdots \cap \pi_{s_n}^{-1} A_n
\]
for some $s_1, \ldots, s_n \in S$ and $A_1, \ldots, A_n \in \cB$. Equivalent, a
cylinder set is a set of the form
\[
  \{\nu \in \B^S \mid \nu s_1 = x_1, \ldots, \nu s_n = x_n\}.
\]
for some $s_1, \ldots, s_n \in S$ and $x_1, \ldots, x_n \in \B$. The
$\si$-algebra $\cB^S$ is also generated by the collection of cylinder sets. Note
that the collection of cylinder sets is a $\pi$-system. That is, it is closed
under intersections.

If we say that a function $f: \B^S \to \B$ is measurable, we mean that it is $
(\cB^S, \cB)$-measurable. Note that $\neg f = 1 - f$ and $\bigwedge_n f_n =
\inf_n f_n$. Hence, if $f$ and $f_n$ are all measurable, then so are $\neg f$
and $\bigwedge_n f_n$. Also note that every function $f: \B^S \to \B$ has the
form $f = 1_B$ for some $B \subseteq \B^S$, and $f$ is measurable if and only if
$B \in \cB^S$.

\begin{lemma}\label{L:ctble-ary}
  Let $R \subseteq S$, let $h: \B^R \to \B$, and define $f: \B^S \to \B$ by $f
  \nu = h(\nu|_R)$. If $h$ is measurable, then $f$ is measurable.
\end{lemma}

\begin{proof}
  Let $R \subseteq S$ and define
  \[
    \Si = \{B \in \cB^R \mid \text{%
      $\nu \mapsto 1_B(\nu|_R)$ is $(\cB^S, \cB)$-measurable%
    }\}.
  \]
  It suffices to show that $\Si = \cB^R$. Since constant functions are
  measurable, we have $\emp \in \Si$. Since $1_{B^c} = \neg 1_B$, it follows
  that $\Si$ is closed under complements. Let $\{B_n\}_{n = 1}^\infty \subseteq
  \Si$ and define $B = \bigcap_n B_n$. Then $1_B = \bigwedge_n 1_{B_n}$, so that
  $B \in \Si$, and $\Si$ is closed under countable intersection. Therefore,
  $\Si$ is a $\si$-algebra.

  Now let
  \[
    B = \{\nu \in \B^R \mid \nu s_1 = x_1, \ldots, \nu s_n = x_n\}
  \]
  be a cylinder set in $\B^R$. Define $h: \B^S \to \B$ by $h \nu = 1_B(\nu|_R)$.
  Then
  \[
    h^{-1} 1 = \{\nu \in \B^S \mid \nu s_1 = x_1, \ldots, \nu s_n = x_n\}
  \]
  is a cylinder set in $\B^S$. Hence, $\Si$ contains the cylinder sets in
  $\B^R$. Since $\cB^R$ is the smallest $\si$-algebra containing the cylinder
  sets, we have $\Si = \cB^R$.
\end{proof}

Let $R \subseteq S$. A measurable function $f: \B^S \to \B$ is said to be
\emph{$R$-ary} if there exists a measurable $h:\B^{R} \to \B$ such that $f \nu =
h(\nu|_R )$ for all $\nu \in \B^S$.

\begin{cor}\label{C:ctble-ary}
  Let $R \subseteq U \subseteq S$. If $f: \B^S \to \B$ is $R$-ary, then $f$ is
  $U$-ary.
\end{cor}

\begin{proof}
  Let $f: \B^S \to \B$ be $R$-ary. Choose measurable $h: \B^R \to \B$ such that
  $f \nu = h(\nu|_R)$ for all $\nu \in \B^S$. Define $g: \B^U \to \B$ by $g \nu
  = h(\nu|_R)$. Lemma \ref{L:ctble-ary} implies $g$ is measurable. Moreover, for
  any $\nu \in \B^S$, we have $g(\nu|_U) = h((\nu|_U)|_R) = h(\nu|_R) = f \nu$.
\end{proof}

\begin{prop}\label{P:ctble-ary}
  Let $S$ be a set. Then every measurable $f: \B^S \to \B$ is $R$-ary for some
  countable $R \subseteq S$.
\end{prop}

\begin{proof}
  Let
  \[
    \Si = \{B \in \cB^S \mid \text{
      $1_B$ is $R$-ary for some countable $R \subseteq S$
    }\}.
  \]
  It suffices to show that $\Si = \cB^S$. Clearly, $1_\emp$ is $\emp$-ary, so
  that $\emp \in \Si$. Since $1_{B^c} = \neg 1_B$, it follows that $\Si$ is
  closed under complements. Let $\{B_n\}_{n = 1}^\infty \subseteq \Si$ and
  define $B = \bigcap_n B_n$. For each $n$, choose countable $R_n \subseteq S$
  such that $1_{B_n}$ is $R_n$-ary. Let $R = \bigcup_n R_n$. Then $R$ is
  countable and, by Corollary \ref{C:ctble-ary}, it follows that $1_{B_n}$ is
  $R$-ary for all $n$. Choose measurable $h_n: \B^R \to \B$ such that $1_{B_n}
  \nu = h_n(\nu \mid R)$ for all $\nu \in \B^S$, and define $h = \bigwedge_n
  h_n$. Then $h$ is measurable and
  \[
    \ts{
      1_B \nu = \bigwedge_n 1_{B_n} \nu
        = \bigwedge_n h_n(\nu|_R)
        = h(\nu|_R),
    }
  \]
  so that $B \in \Si$. Hence, $\Si$ is closed under countable intersections, and
  $\Si$ is therefore a $\si$-algebra.

  Now let
  \[
    B = \{\nu \in \B^S \mid \nu s_1 = x_1, \ldots, \nu s_n = x_n\}
  \]
  be a cylinder set in $\B^S$. Then $B$ is $R$-ary, where $R = \{s_1, \ldots,
  s_n\}$, showing that $\Si$ contains the cylinder sets. Since $\cB^S$ is the
  smallest $\si$-algebra containing the cylinder sets, we have $\Si = \cB^S$.
\end{proof}

\subsection{Strict models and Boolean functions}

A \textit{strict model}
  \index{model!strict ---}%
is a function $\om: PV \to \B$. That is, a strict model is an assignment of
truth values to each propositional formula. The set of all strict models is
$\B^{PV}$. The domain of a strict model can be uniquely extended to all of $\cF$
by formula recursion. That is, $\om \neg \ph = \neg \om \ph$ and $\om \bigwedge
\Phi = \bigwedge_ {\ph \in \Phi} \om \ph$. We call $\om \ph$ the \textit{truth
value} of $\ph$ in the strict model $\om$.

The definition of a strict model depends on the choice of propositional
variables $PV$, which in turn determine the language $\cF$. When we wish to
emphasize this fact, we will call $\om$ a \emph{strict model in $\cF$}.

A \emph{Boolean function}
  \index{Boolean!function}%
is a measurable function $f: \B^{PV} \to \B$. We say that a formula $\ph$
\emph{represents} a Boolean function $f$ if $\om \ph = f \om$ for all strict
models $\om$.

\begin{prop}\label{P:Boolean-funcs}
  Every formula represents a unique Boolean function. Conversely, every Boolean
  function is represented by a formula.
\end{prop}

\begin{proof}
  Let $\ph \in \cF$ and define $f_\ph: \B^{PV} \to \B$ by $f_\ph \om = \om \ph$.
  To show that $f_\ph$ is a Boolean function, we must show that it is
  measurable. This follows by formula induction since $f_\bfr = \pi_\bfr$, $f_
  {\neg \ph} = \neg f_\ph$, and $f_{\bigwedge \Phi} = \bigwedge_{\ph \in \Phi}
  f_\ph$.

  Now let
  \[
    \Si = \{B \in \cB^{PV} \mid \text{$1_B$ is represented by a formula}\}.
  \]
  It suffices to show that $\Si = \cB^{PV}$. If we fix $\bfr \in PV$, then
  $1_\emp$ is represented by the formula $\bfr \wedge \neg \bfr$, so $\emp \in
  \Si$. If $\ph$ represents $1_B$, then $\neg \ph$ represents $\neg 1_B = 1_
  {B^c}$, so $\Si$ is closed under complements. And if $\ph_n$ represents
  $1_{B_n}$, then $\bigwedge_n \ph_n$ represents $\bigwedge_n 1_{B_n} =
  1_{\bigcap_n B_n}$, so that $\Si$ is closed under countable intersections, and
  $\Si$ is therefore a $\si$-algebra.

  Now let
  \[
    B = \{\om \in \B^{PV} \mid \om \bfr_1 = x_1, \ldots, \om \bfr_n = x_n\}
  \]
  be a cylinder set. Recall the notation $\ph^1 = \ph$ and $\ph^0 = \neg \ph$,
  and note that $\om \ph = x$ if and only if $\om \ph^x = 1$. Hence, $1_B$ is
  represented by the formula $\bfr_1^{x_1} \wedge \cdots \wedge \bfr_n^{x_n}$,
  so that $\Si$ contains the cylinder sets, and therefore $\Si = \cB^{PV}$.
\end{proof}

By Proposition \ref{P:ctble-ary}, every Boolean function is $\Pi$-ary for some
countable set of propositional variables $\Pi \subseteq PV$.

\begin{prop}\label{P:Boolean-func-Pi-ary}
  Let $\ph \in \cF$ and let $f$ be the Boolean function that it represents. Then
  $f$ is $\Pi$-ary, where $\Pi = PV \cap \Sf \ph$ is the countable set of
  propositional variables that appear in $\ph$.
\end{prop}

\begin{proof}
  Let $\ph \in \cF$ represent $f_\ph$ and let $\Pi_\ph = PV \cap \Sf \ph$. We
  will show that $f_\ph$ is $\Pi_\ph$-ary by induction on $\ph$.

  If $\bfr \in PV$, then $\Pi_\bfr = \{\bfr\}$ and $f_\bfr = \pi_\bfr$ is $
  \{\bfr\}$-ary. Suppose $f_\ph$ is $\Pi_\ph$-ary. Then $f_{\neg \ph} = \neg
  f_\ph$ is also $\Pi_\ph$-ary. Since $\Sf \ph \subseteq \Sf \neg \ph$, we have
  $\Pi_\ph \subseteq \Pi_{\neg \ph}$. Hence, by Corollary \ref{C:ctble-ary}, it
  follows that $f_{\neg \ph}$ is $\Pi_{\neg \ph}$-ary.

  Now let $\Phi \subseteq \cF$ be countable and suppose $f_\th$ is $\Pi_\th$-ary
  for all $\th \in \Phi$. Define $\ph = \bigwedge \Phi$. Note that $\Pi_\th
  \subseteq \Pi_\ph$ for each $\th \in \Phi$. Hence, Corollary \ref{C:ctble-ary}
  implies that $f_\th$ is $\Pi_\ph$-ary for each $\th \in \Phi$, and therefore
  $f_\ph = \bigwedge_{\th \in \Phi} f_\th$ is also $\Pi_\ph$-ary.
\end{proof}

\subsection{Models and satisfiability}

An \textit{inductive model}, or simply a \emph{model},
  \index{model}%
  \index{inductive model|see {model}}%
is a probability space, $\sP = (\Om, \Si, \bbP)$, where $\Om$ is a set of strict
models.

As with strict models, the definition of a model depends on the choice of
propositional variables $PV$, which in turn determine the language $\cF$. When
we wish to emphasize this fact, we will call $\sP$ a \emph{model in $\cF$}.

If $X \subseteq \cF$ and $\om$ is a strict model, we say that $\om$
\textit{strictly satisfies} $X$, written $\om \tDash X$, if $\om \ph = 1$ for
all $\ph \in X$. We write $\om \tDash \ph$ for $\om \tDash \{\ph\}$. A set $X
\subseteq \cF$ is \textit{strictly satisfiable} if there is a strict model $\om$
such that $\om \tDash X$.
  \index{satisfiable!strictly ---}%
  \symindex{$\tDash$}%

Let $\Om$ be a set of strict models. For $\ph \in \cF$, let
\[
  \ph_\Om = \{\om \in \Om \mid \om \tDash \ph\}
\]
  \symindex{$\ph_\Om$}%
be the set of strict models in $\Om$ that strictly satisfy $\ph$. More
generally, for $X \subseteq \cF$, we define $X_\Om = \{\ph_\Om \mid \ph \in
X\}$.

The mapping $\ph \mapsto \ph_\Om$ satisfies $(\neg \ph)_\Om = \ph_\Om^c$ and $
(\bigwedge \Phi)_\Om = \bigcap_{\ph \in \Phi} \ph_\Om$. Similar relations hold
for the shorthand operators. For instance $(\bigvee \Phi)_\Om = \bigcup_{\ph \in
\Phi} \ph_\Om$, $(\ph \to \psi)_\Om = \ph_\Om^c \cup \psi_\Om$, and $(\ph \tot
\psi)_\Om = (\ph_\Om \tri \psi_\Om)^c$.

Note that if $\ph \in \cF$ represents the Boolean function $f$, then $\ph_\Om =
f^{-1} 1$. Hence, by Proposition \ref{P:Boolean-funcs}, if $\Om = \B^{PV}$,
then $\cB^ {PV} = \{\ph_\Om \mid \ph \in \cF\}$.

Let $\sP$ be a model and let $\ph \in \cF$. We say that $\sP$ \textit{satisfies}
$\ph$, written $\sP \vDash \ph$, if $\olbbP \ph_\Om = 1$, where $(\Om, \ol \Si,
\olbbP)$ is the completion of $(\Om, \Si, \bbP)$. For $X \subseteq \cF$, we
write $\sP \vDash X$ to mean $\sP \vDash \ph$ for all $\ph \in X$. Note that
$\sP \vDash \emp$ for every model $\sP$. A set $X \subseteq \cF$ is 
\textit{satisfiable} if there is a model $\sP$ such that $\sP \vDash X$. Note
that if $X \subseteq X'$ and $\sP \vDash X'$, then $\sP \vDash X$.
  \index{satisfiable}%
  \symindex{$\sP \vDash \ph$}%

\begin{prop}\label{P:sig-pre-cpct}
  Let $X \subseteq \cF$.
  \begin{enumerate}[(i)]
    \item If $X$ is strictly satisfiable, then $X$ is satisfiable.
    \item If $X$ is satisfiable and countable, then $X$ is strictly satisfiable.
  \end{enumerate}
\end{prop}

\begin{proof}
  Note that $\om \tDash X$ if and only if $\sP = (\{\om\}, \{\emp,
  \{\om\}\}, \de_ {\om}) \vDash X$, which yields (i). For (ii), suppose $X$ is
  satisfiable and countable. Let $\sP = (\Om, \Si, \bbP)$ be a model that
  satisfies $X$. Then $\olbbP \bigcap_{\ph \in X} \ph_\Om = 1$, so we may choose
  $\om \in \bigcap_{\ph \in X} \ph_\Om$, and this $\om$ strictly satisfies $X$.
\end{proof}

If $\sP$ is a model, we define
\begin{equation}\label{Th-sP-def}
  \Th \sP = \{\ph \in \cF \mid \sP \vDash \ph\}.
\end{equation}
  \symindex{$\Th \sP$ (in $\cF$)}%
As we will see in Proposition \ref{P:ThP-is-theory}, the set of formulas $\Th
\sP$ is a consistent deductive theory. Note that if $\ol \sP = (\Om, \ol \Si,
\olbbP)$ is the completion of $\sP$, then $\Th \ol \sP = \Th \sP$.

Let $\sP = (\Om, \Si, \bbP)$ be a model with completion $(\Om, \ol \Si,
\olbbP)$. Let $\ol \Si_\cF = \ol \Si \cap \cB^{PV}$ and let ${\olbbP_\cF}$ be
$\olbbP$ restricted to $\ol \Si_\cF$. Then $\ol \Si_\cF$ is a sub-$\si$-algebra
of $\ol \Si$, so that $\sP_\cF = (\Om, \ol \Si_\cF, \olbbP_\cF)$ is also a
model. For any $\ph \in \cF$, we have $\ph_\Om \in \ol \Si$ if and only if
$\ph_\Om \in \ol \Si_\cF$. Hence, all of the logical information in $\sP$ is
contained in $\sP_\cF$. This is made precise in Proposition \ref{P:prop-iso-thm}
below. We say that two models $\sP$ and $\sQ$ are \emph{isomorphic (as models)},
denoted by $\sP \simeq \sQ$, if $\sP_\cF = \sQ_\cF$.
  \index{model!isomorphic ---}%

\begin{lemma}\label{L:prop-iso-thm}
  If $\sP$ is a model, then $(\sP_\cF)_\cF = \sP_\cF$.
\end{lemma}

\begin{proof}
  Let $\sP = (\Om, \Si, \bbP)$ be a model. For notational simplicity, let $\Ga =
  \ol \Si_\cF$ and $\bbQ = \olbbP|_\Ga$, so that $\sP_\cF = (\Om, \Ga, \bbQ)$.
  We must show that $\ol \Ga_\cF = \Ga$ and $\olbbQ_\cF = \bbQ$. Since $\Ga
  \subseteq \cB^{PV}$, we have
  \[
    \Ga = \Ga \cap \cB^{PV} \subseteq \ol \Ga \cap \cB^{PV} = \ol \Ga_\cF.
  \]
  Conversely, let $A \in \ol \Ga_\cF = \ol \Ga \cap \cB^{PV}$. Since $A \in \ol
  \Ga$, we may write $A = B \cup N$, where $B \in \Ga$, $N \subseteq F \in \Ga$,
  and $\bbQ F = 0$. By the definition of $\bbQ$, this implies $\olbbP F = 0$.
  Now $\Ga = \ol \Si \cap \cB^{PV}$. Hence, $B \in \ol \Si$ and $F \in \ol \Si$.
  Since $N \subseteq F$ and $\olbbP F = 0$, we have $N \in \ol \Si$. Therefore,
  $A = B \cup N \in \ol \Si$. But $A \in \cB^{PV}$ also. Hence, $A \in \ol \Si
  \cap \cB^{PV} = \ol \Si_\cF = \Ga$, and this shows $\ol \Ga_\cF = \Ga$.

  By definition, $\olbbQ_\cF$ is $\olbbQ$ restricted to $\ol \Ga_\cF$. But $\ol
  \Ga_\cF = \Ga$, so we have that $\olbbQ_\cF = \olbbQ|_{\Ga} = \bbQ$.
\end{proof}

\begin{prop}\label{P:prop-iso-thm}
  For any model $\sP$ and any $\ph \in \cF$, we have $\sP \vDash \ph$ if and
  only if $\sP_\cF \vDash \ph$.
\end{prop}

\begin{proof}
  Suppose $\sP \vDash \ph$. Then $\ph_\Om \in \ol \Si$ and $\olbbP \ph_\Om = 1$.
  Since $\ph_\Om \in \cB^{PV}$, we have $\ph_\Om \in \ol \Si_\cF$ and
  $\olbbP_\cF \ph_\Om = 1$. For the converse, let $\Ga = \ol \Si_\cF$ and $\bbQ
  = \olbbP|_\Ga$, so that $\sP_\cF = (\Om, \Ga, \bbQ)$. Suppose $\sP_\cF \vDash
  \ph$. Then $\ph_\Om \in \ol \Ga$ and $\olbbQ \ph_\Om = 1$. By Lemma 
  \ref{L:prop-iso-thm}, we have $\ph_\Om \in \Ga$ and $\bbQ \ph_\Om = 1$. Since
  $\Ga = \ol \Si \cap \cB^{PV}$ and $\bbQ = \olbbP|_\Ga$, this gives $\ph_\Om
  \in \ol \Si$ and $\olbbP \ph_\Om = 1$, so that $\sP \vDash \ph$.
\end{proof}

\begin{rmk}
  If $\sP$ and $\sQ$ are isomorphic models, then $\sP \vDash \ph$ if and only if
  $\sQ \vDash \ph$ for all $\ph \in \cF$. This follows immediately from the
  definition of isomorphic models and Proposition \ref{P:prop-iso-thm}.
\end{rmk}

\subsection{Deductive consequence and soundness}

We say $\ph \in \cF$ is a \textit{consequence} of $X \subseteq \cF$, or that $X$
\emph{entails} $\ph$, which we denote by $X \vDash \ph$, if, for all models
$\sP$ such that $\sP \vDash X$, we have $\sP \vDash \ph$. Note that if $X$ is
not satisfiable, then it is vacuously true that $X \vDash \ph$ for all $\ph \in
\cF$.
  \index{consequence relation}%
  \symindex{$X \vDash \ph$ (in $\cF$)}%

We write $\psi \vDash \ph$ for $\{\psi\} \vDash \ph$ and ${} \vDash \ph$ for
$\emp \vDash \ph$. Note that ${} \vDash \ph$ if and only if $\sP \vDash \ph$ for
all models $\sP$, which holds if and only if $\om \tDash \ph$ for all strict
models $\om$. (If $\om \tDash \ph$ for all $\om$, then $\ph_\Om = \Om$ in every
model; conversely, if $\om \ntDash \ph$, then $\sP = (\{\om\}, \{\emp, \{\om\}\},
\de_{\om}) \nvDash \ph$.) We also write $X \vDash Y$ to mean that $X \vDash \ph$
for all $\ph \in Y$. Note here that $X\vDash Y$ if and only if $\sP \vDash X$
implies $\sP \vDash Y$ for all models $\sP$.

A logical system is sound if every formula that is derivable from $X$ is a
consequence of $X$. The following theorem shows that our notion of deductive
satisfiability yields a sound logical system, at least insofar as deductive
inference is concerned.

\begin{thm}[Deductive soundness]\label{T:soundness}
    \index{soundness!deductive ---}%
  Let $X \subseteq \cF$ and $\ph \in \cF$. If $X \vdash \ph$, then $X \vDash
  \ph$.
\end{thm}

\begin{proof}
  It suffices to show that (i)--(vi) in Definition \ref{D:derivability} still
  hold when $\vdash$ is replaced by $\vDash$. Conditions (i) and (ii) are
  trivial.

  Suppose $X \vDash \bigwedge \Phi$. Let $\sP = (\Om, \Si, \bbP) \vDash X$. Then
  $\sP \vDash \bigwedge \Phi$, which implies
  \[
    \ts{
      \big(\bigwedge \Phi\big)_\Om = \bigcap_{\th \in \Phi} \th_\Om \in \ol \Si
    }
  \]
  and $\olbbP \bigcap_{\th \in \Phi} \th_\Om = 1$. Thus, $\olbbP \bigcup_ {\th
  \in \Phi} \th_\Om^c = 0$. For each $\th \in \Phi$, we have that $\th_\Om^c$ is
  a subset of a null set. Hence, $\th_\Om^c \in \ol \Si$ and $\olbbP \th_\Om^c =
  0$, implying $\th_\Om \in \ol \Si$ and $\olbbP \th_\Om = 1$. Therefore, $\sP
  \vDash \th$, showing that $X \vDash \th$ and proving (iii). The proof of (iv)
  is similar.

  For (v), suppose $X \vDash \ph$ and $X \vDash \neg \ph$, and assume there
  exists a model $\sP$ such that $\sP \vDash X$. Then $\olbbP \ph_\Om = 1$ and
  $\olbbP \ph_\Om^c = 1$, a contradiction. Hence, $X$ is not satisfiable, and so
  it is vacuously true that $X \vDash \psi$.

  For (vi), suppose $X, \ph \vDash \psi$, $X, \neg \ph \vDash \psi$, and $X
  \nvDash \psi$. Choose a model $\sP = (\Om, \Si, \bbP)$ such that $\sP \vDash X$
  and $\sP \nvDash \psi$. If $\psi_\Om \in \ol \Si$, then $\olbbP \psi_\Om <
  1$. Suppose $\psi_\Om \notin \ol \Si$. Then $\bbP_* \psi_\Om < \bbP^* \psi_\Om
  \le 1$. In this case, there exists a measure $\bbP'$ on $(\Om, \si(\Si \cup
  \{\psi_\Om\}))$ such that $\bbP'|_\Si = \bbP$ and $\bbP' \psi_\Om = \bbP_*
  \psi_\Om$. In either case, $(\Om, \Si, \bbP)$ can be extended to a complete
  model $\sP' = (\Om, \Si', \bbP')$ in which $\psi_\Om \in \Si'$ and $\olbbP
  \psi_\Om < 1$. Therefore, $\sP' \vDash X$ and $\sP' \nvDash \psi$.

  By extending the model even further, we may assume $\ph_\Om \in \Si'$. Suppose
  $\bbP' \ph_\Om = 0$. Then $\sP' \vDash \neg \ph$, so by supposition, we have
  $\sP' \vDash \psi$, a contradiction. Hence, $\bbP' \ph_\Om > 0$, and we may
  define a probability measure $\bbQ$ on $(\Om, \Si')$ by $\bbQ = \bbP'(\, \cdot
  \mid \ph_\Om)$, and then define the model $\sQ = (\Om, \Si', \bbQ)$.

  Since $\bbQ \ph_\Om = 1$, we have $\sQ \vDash \ph$. Also, if $A \in \Si'$ and
  $\bbP' A = 1$, then $\bbQ A = 1$. Thus, since $\sP' \vDash X$, it follows that
  $\sQ \vDash X$. By supposition, then, we have $\sQ \vDash \psi$. Since
  $\psi_\Om \in \Si'$, this gives $\bbQ \psi_\Om = 1$. In other words,
  $\bbP'(\psi_\Om \mid \ph_\Om) = 1$. By reversing the roles of $\ph$ and $\neg
  \ph$, this same argument yields $\bbP'(\psi_\Om \mid \ph_\Om^c) = 1$.
  Therefore,
  \[
    \bbP' \psi_\Om
      = \bbP' \ph_\Om \, \bbP'(\psi_\Om \mid \ph_\Om)
        + \bbP' \ph_\Om^c \, \bbP'(\psi_\Om \mid \ph_\Om^c)
      = \bbP' \ph_\Om + \bbP' \ph_\Om^c = 1,
  \]
  which contradicts the fact that $\sP' \nvDash \psi$.
\end{proof}

\begin{cor}\label{C:soundness}
  If $X \subseteq \cF$ is satisfiable, then $X$ is consistent.
\end{cor}

\begin{proof}
  Suppose $X$ is inconsistent. Then $X \vdash \bot$. By Theorem 
  \ref{T:soundness}, we have $X \vDash \bot$. But $\bot_\Om = \emp$, so $\sP
  \nvDash \bot$ for all $\sP$. Hence, $X$ is not satisfiable.
\end{proof}

\begin{prop}\label{P:ThP-is-theory}
  If $\sP$ is a model, then $\Th \sP$ is a consistent deductive theory.
\end{prop}

\begin{proof}
  Let $T = \Th \sP$ and suppose $T \vdash \ph$. By Theorem \ref{T:soundness}, we
  have $T \vDash \ph$. Since $\sP \vDash T$, this implies $\sP \vDash \ph$.
  Hence, $\ph \in T$, so that $T$ is a deductive theory. Since $\emp =
  \bot_\Om$, we have $\sP \nvDash \bot$, so that $\bot \notin T$ and $T$ is
  consistent.
\end{proof}

\subsection{Karp's completeness theorem}

In this subsection, we establish that our logical system is complete, meaning
that every consequence of $X$ is derivable from $X$. Completeness is the
converse of soundness. Together, they show that the derivability and consequence
relations are identical.

In Theorem \ref{T:sig-cpctness}, we showed that $\vdash$ is $\si$-compact. In
Theorem \ref{T:compactness} below, we will show $\si$-compactness for $\vDash$,
and then use this to establish completeness in Theorem \ref{T:completeness}.

It is well-known that both $\si$-compactness and completeness fail when we adopt
the classical semantic notion of the strict model (see Example
\ref{Expl:Karp413}). In that case, only a weaker version of completeness is
available. This weaker version was proven by Karp in \cite{Karp1964}. We present
Karp's version below, and then use it to establish the full completeness theorem
for our notion of deductive satisfiability.

\begin{thm}[Karp's completeness theorem]\label{T:Karp-compl}
    \index{completeness!Karp's --- theorem}%
  For all formulas $\ph \in \cF$, we have ${} \vdash \ph$ if and only if $
  {} \vDash \ph$.
\end{thm}

\begin{proof}
  The only if direction is a consequence of Theorem \ref{T:soundness}. For the 
  if direction, we appeal to Karp's completeness theorem. In \cite[Theorem
  5.3.2]{Karp1964}, Karp proved that ${} \vdash' \ph$ if and only if $\om \tDash
  \ph$ for all strict models $\om$, where $\vdash'$ is a certain Hilbert-type
  system of deduction. As noted previously, ${} \vDash \ph$ if and only if $\om
  \tDash \ph$ for all strict models $\om$. We therefore have that ${} \vDash
  \ph$ if and only if ${} \vdash' \ph$. To complete the proof, we must verify
  that ${} \vdash' \ph$ implies ${} \vdash \ph$.

  To accomplish this, we must first describe the differences between $\vdash'$
  and $\vdash$. In Karp's system, $\to$ is a primitive symbol; for us, it is
  defined shorthand. This, however, causes no difficulties, since $(\ph \to
  \psi) \tot (\neg \ph \vee \psi)$ is a tautology in Karp's system.

  Recall from Theorem \ref{T:Hilbert=nat} that ${} \vdash \ph$ if and only if
  there is a proof of $\ph$ from the axioms $\La$. Karp's $\vdash'$ differs from
  our $\vdash$ only in the choice of the axioms; the notion of proof is the
  same. Aside from the aforementioned use of $\to$, this is the only difference
  between $\vdash'$ and $\vdash$. Hence, we need only verify that each of Karp's
  axioms can be proven in $\vdash$. The axioms of Karp that are not already
  accounted for in $\La$ are:
  \begin{enumerate}[($\La$1)]
    \setcounter{enumi}{3}
    \item $\ph \to \psi \to \ph$
    \item $(\neg \ph \to \neg \psi) \to \psi \to \ph$
    \item $\bigwedge_{\ph \in \Phi} (\psi \to \ph) \to \psi \to \bigwedge \Phi$
  \end{enumerate}
  By Theorem \ref{T:Hilbert=nat}, it suffices to prove these by natural
  deduction, which is entirely straightforward.
\end{proof}

\begin{rmk}\label{R:impl-subset}
  As a consequence of Karp's completeness theorem, we have that $\ph$ is a
  tautology if and only if $\om \tDash \ph$ for all strict models $\om$. Hence,
  in any model $\sP$, we have $\ph \vdash \psi$ implies $\ph_\Om \subseteq
  \psi_\Om$, and $\ph \equiv \psi$ implies $\ph_\Om = \psi_\Om$. If $\Om = \B^
  {PV}$ is the set of all strict models, then both of these implications are
  biconditional.
\end{rmk}

With Karp's completeness theorem, we can now prove the result that was described
in Remark \ref{R:finitary-vs-infinitary}.

\begin{prop}\label{P:finitary-vs-infinitary}
  Let $X \subseteq \cF_\fin$ and $\ph \in \cF_\fin$. If $X \vdash \ph$, then $X
  \vdash_\fin \ph$.
\end{prop}

\begin{proof}
  Let $X \subseteq \cF_\fin$ and $\ph \in \cF_\fin$. Suppose $X \vdash \ph$. The
  well-known completeness theorem from finitary propositional logic states that
  $X \vdash_\fin \ph$ if and only if $\om \tDash X$ implies $\om \tDash \ph$ for
  all strict models $\om$. (See, for instance, \cite[Theorem 1.4.6]
  {Rautenberg2010}).

  Let $\om$ be a strict model and assume that $\om \tDash X$. By Proposition
  \ref{P:sig-cpctness}, we may choose countable $X_0 \subseteq X$ such that
  $\vdash \bigwedge X_0 \to \ph$. By Karp's completeness theorem, $\vDash
  \bigwedge X_0 \to \ph$. Hence, $\om \tDash \bigwedge X_0 \to \ph$. But $\om
  \tDash X \supseteq X_0$. Therefore, $\om \tDash \ph$.
\end{proof}

\subsection{Inductive theories and Dynkin systems}\label{S:ind-th-Dynk}

We briefly pause our development to make an observation about Dynkin systems.
Let $P$ be an inductive theory and fix $X \in \ante P$. In Section
\ref{S:rel-neg-cert}, we noted that the domain of $P(\; \cdot \mid X)$ need not
be closed under conjunctions and disjunctions. We are now in a position to say
something in the positive direction about the structure of this set of formulas.

Let $\Om = \B^{PV}$ and define
\begin{equation}\label{Del-P-X}
  \De = \De(P, X) = \{\ph_\Om \mid P(\ph \mid X) \text{ exists}\}.
\end{equation}
  \symindex{$\De(P, X)$}%
Let $A \in \cB^{PV} = \{\ph_\Om \mid \ph \in \cF\}$ and choose $\ph \in \cF$
such that $A = \ph_\Om$. By the above definition, if $P(\ph \mid X)$ exists,
then $A \in \De$. Conversely, if $A \in \De$, then Remark \ref{R:impl-subset}
and the rule of logical equivalence imply that $P(\ph \mid X)$ exists. Hence,
$\De$ is an embedding of the domain of $P(\; \cdot \mid X)$ into $\cB^{PV}$. The
structure of this domain, therefore, can be understood by looking at the
structure of $\De$.

\begin{prop}\label{P:Del-P-X}
  If $P$ be an inductive theory, then $\De(P, X)$ is a Dynkin system for every
  $X \in \ante P$.
\end{prop}

\begin{proof}
  Let $P$ be an inductive theory and $X \in \ante P$. Let $\De = \De(P, X)$ be
  defined as above. By the rule of logical implication, $P(\top \mid X) = 1$.
  Hence, $\Om = \top_\Om \in \De$, and $\De$ is nonempty. Let $A \in \De$.
  Choose $\ph \in \cF$ such that $A = \ph_\Om$. Then $A^c = (\neg \ph)_\Om$. By
  Corollary \ref{C:rel-neg}, we have $A^c \in \De$. Now suppose $\{A_n\}
  \subseteq \De$ is pairwise disjoint. Choose $\ph_n \in \cF$ such that $A_n =
  (\ph_n)_\Om$. For $i \ne j$, we have $\bot_\Om = \emp = A_i \cap A_j = (\ph_i
  \wedge \ph_j)_\Om$. Hence, from Remark \ref{R:impl-subset}, it follows that
  $\ph_i \wedge \ph_j \equiv \bot$. By the rule of logical equivalence, $P(\ph_i
  \wedge \ph_j \mid X) = 0$. Therefore, Theorem \ref{T:ctbl-add} implies
  $P(\bigvee_n \ph_n \mid X)$ exists. But $(\bigvee_n \ph_n)_\Om = \bigcup_n
  A_n$, so $\bigcup_n A_n \in \De$, and $\De$ is a Dynkin system.
\end{proof}

\subsection{The full completeness theorem}

\begin{thm}[$\si$-compactness]\label{T:compactness}
    \index{s_sigma-compactness@$\si$-compactness}%
  A set $X \subseteq \cF$ is satisfiable if and only if every countable subset
  of $X$ is satisfiable.
\end{thm}

\begin{proof}
  The only if part is trivial. Suppose every countable subset of $X$ is
  satisfiable. Assume $X$ is inconsistent. Then $X \vdash \bot$. By Theorem 
  \ref{T:sig-cpctness}, there exists countable $X_0 \subseteq X$ such that $X_0
  \vdash \bot$, implying that $X_0$ is inconsistent. By Corollary 
  \ref{C:soundness}, we have that $X_0$ is not satisfiable, a contradiction.
  Hence, $X$ is consistent.

  Let $\Om$ be the set of all strict models. Let
  \[
    \Si = \{\ph_\Om \mid \ph \in T(X) \text{ or } \neg \ph \in T(X)\}.
  \]
  Then $\Si$ is a $\si$-algebra. If $A \in \Si$, choose $\ph$ such that $A =
  \ph_\Om$ and define $\bbP A = 1$ if $\ph \in T(X)$ and $0$ otherwise. By
  Remark \ref{R:impl-subset}, the function $\bbP$ is well-defined.

  Since $X$ is consistent, $\bot \notin T(X)$. Thus, $\bbP \emp = \bbP \bot_\Om =
  0$. Conversely, $\top \in T(X)$, so $\bbP \Om = \bbP \top_\Om = 1$.

  Now let $\{A_n\}_{n \in \bN} \subseteq \Si$ be pairwise disjoint, and define
  $A = \bigcup_n A_n$. For each $n$, choose $\ph_n$ such that $A_n =
  (\ph_n)_\Om$, and define $\ph = \bigvee_n \ph_n$. Note that $A = \ph_\Om$.
  Suppose $m \ne n$. Since
  \[
    (\ph_m \wedge \ph_n)_\Om = A_m \cap A_n = \emp = \bot_\Om,
  \]
  we have $\ph_m \wedge \ph_n \equiv \bot$, implying that $\ph_m \wedge \ph_n
  \notin T(X)$. Since $T(X)$ is closed under conjunctions, either $\ph_m
  \notin T(X)$ or $\ph_n \notin T(X)$. This implies that there is at most one
  $n \in \bN$ with $\bbP A_n = 1$. Therefore, $\sum_n \bbP A_n \in \{0, 1\}$ and
  \begin{align*}
    \ts{\sum \bbP A_n = 1}
      &\quad\text{iff}\quad \text{there exists $n$ such that $\bbP A_n = 1$}\\
    &\quad\text{iff}\quad \text{there exists $n$ such that $\ph_n \in T(X)$}\\
    &\quad\text{iff}\quad \ph \in T(X)\\
    &\quad\text{iff}\quad \bbP \ph_\Om = \bbP A = 1,
  \end{align*}
  showing that $\bbP$ is countably additive. Thus, $\bbP$ is a measure on $(\Om,
  \Si)$ with $\bbP \Om = 1$, and so $\sP = (\Om, \Si, \bbP)$ is a model.

  Now let $\ph \in X \subseteq T(X)$ be arbitrary. Then $\ph_\Om \in \Si$, and
  since $\ph \in T(X)$, we have $\bbP \ph_\Om = 1$, showing that $\sP \vDash X$,
  so that $X$ is satisfiable.
\end{proof}

\begin{cor}\label{C:compactness}
  A set $X \subseteq \cF$ is satisfiable if and only if $X$ is consistent.
\end{cor}

\begin{proof}
  The only if part is Corollary \ref{C:soundness}. Suppose $X$ is not
  satisfiable. By Theorem \ref{T:compactness}, there exists a countable
  subset $X_0 \subseteq X$ that is not satisfiable. By Proposition 
  \ref{P:sig-pre-cpct}, the set $X_0$ is not strictly satisfiable. Thus, $\om
  \tDash \neg \bigwedge X_0$ for all strict models $\om$, which implies ${}
  \vDash \neg \bigwedge X_0$. By Theorem \ref{T:Karp-compl}, we have ${} \vdash
  \neg \bigwedge X_0$. Thus, $X \vdash \bigwedge X_0$ and $X \vdash \neg
  \bigwedge X_0$, showing that $X$ is inconsistent.
\end{proof}

\begin{thm}[Deductive completeness]\label{T:completeness}
    \index{completeness!deductive ---}%
  For $X \subseteq \cF$ and $\ph \in \cF$, we have $X \vDash \ph$ if and only 
  if $X \vdash \ph$.
\end{thm}

\begin{proof}
  The if part is Theorem \ref{T:soundness}. Suppose $X \nvdash \ph$. Then $X
  \cup \{\neg\ph\}$ is consistent, by Theorem \ref{T:deduc-con}. Thus, $X \cup
  \{\neg\ph\}$ is satisfiable, by Corollary \ref{C:compactness}. Let $\sP$ be a
  model with $\sP \vDash X \cup \{\neg \ph\}$. Then $\sP$ is an example of a
  model with $\sP \vDash X$ and $\sP \nvDash \ph$. Thus, $X \nvDash \ph$.
\end{proof}

\section{Inductive semantics}\label{S:ind-sem}

\subsection{Inductive satisfiability}

We now define a notion of satisfiability for inductive statements. Let $\sP =
(\Om, \Si, \bbP)$ be a model with completion $\ol \sP = (\Om, \ol \Si, \olbbP)$.
Let $(X, \ph, p)$ be an inductive statement. We say that $\sP$ \emph{satisfies
$(X, \ph, p)$}, written $\sP \vDash (X, \ph, p)$, if there exists $Y \subseteq
\Th \sP$ and $\psi \in \cF$ such that $X \equiv Y \cup \{\psi\}$ and
\begin{equation}\label{cond-prob}
  \frac{\olbbP \ph_\Om \cap \psi_\Om}{\olbbP \psi_\Om} = p.
\end{equation}
Note that $\sP \vDash (X, \ph, p)$ if and only if $\ol \sP \vDash (X, \ph, p)$.
Hence, in many circumstances, we may assume without loss of generality that our
models are complete.

For $Q \subseteq \cF^\IS$, we write $\sP \vDash Q$ to mean $\sP \vDash (X, \ph,
p)$ for all $(X, \ph, p) \in Q$. A set $Q$ is \emph{satisfiable} if there exists
a model $\sP$ such that $\sP \vDash Q$.
  \index{satisfiable}%
  \symindex{$\sP \vDash_\cF (X, \ph, p)$}%

The next result shows that if $\sP \vDash (X, \ph, p)$, then \eqref{cond-prob}
will hold, regardless of how we decompose $X$ into $Y$ and $\psi$. As a
corollary, we see that for fixed $X$ and $\ph$, there can be only one $p$ such
that $\sP \vDash (X, \ph, p)$.

\begin{prop}\label{P:model-func}
  If $X \equiv Y \cup \{\psi\} \equiv Y' \cup \{\psi'\}$ and $\sP \vDash Y, Y'$,
  then $\psi_\Om = \psi'_\Om$ a.s. In particular, if $\sP \vDash (X, \ph, p)$
  and $X \equiv Y \cup \{\psi\}$, where $\sP \vDash Y$, then \eqref{cond-prob}
  holds.
\end{prop}

\begin{proof}
  Suppose $X \equiv Y \cup \{\psi\} \equiv Y' \cup \{\psi'\}$ and $\sP \vDash Y,
  Y'$. Using Theorem \ref{T:completeness}, we have $Y, \psi \vDash \psi'$, which
  implies $Y \vDash \psi \to \psi'$. Hence, $\sP \vDash \psi \to \psi'$, so that
  $\olbbP \psi_\Om \cap (\psi'_\Om)^c = 0$. Similarly, $\olbbP \psi'_\Om \cap
  \psi_\Om^c = 0$. Thus, $\olbbP \psi_\Om \tri \psi'_\Om = 0$.

  Now suppose $\sP \vDash (X, \ph, p)$ and $X \equiv Y \cup \{\psi\}$, where
  $\sP \vDash Y$. Choose $Y' \subseteq \cF$ and $\psi' \in \cF$ such that $X
  \equiv Y' \cup \{\psi'\}$, $\sP \vDash Y'$, and $\ol \bbP \ph_\Om \cap
  \psi'_\Om / \bbP \psi'_\Om = p$. By the above, $\psi_\Om = \psi'_\Om$ a.s.
  Hence, $\olbbP \psi_\Om = \olbbP \psi'_\Om$ and $\olbbP \ph_\Om \cap \psi_\Om
  = \olbbP \ph_\Om \cap \psi'_\Om$, so that \eqref{cond-prob} holds.
\end{proof}

\begin{cor}\label{C:model-func}
  Let $\sP$ be a model. If $\sP \vDash (X, \ph, p)$ and $\sP \vDash (X, \ph,
  p')$, then $p = p'$.
\end{cor}

\begin{proof}
  Suppose $\sP \vDash (X, \ph, p')$. Write $X \equiv Y \cup \{\psi\}$, where
  $\sP \vDash Y$ and $\bbP \ph_\Om \cap \psi_\Om / \bbP \psi_\Om = p'$. By
  Proposition \ref{P:model-func}, if $\sP \vDash (X, \ph, p)$, then 
  \eqref{cond-prob} holds, and so $p = p'$.
\end{proof}

\subsection{Models determine theories}

Given a model $\sP$, we define
\begin{equation}\label{bTh-sP-def}
  \bTh \sP = \{(X, \ph, p) \in \cF^\IS \mid \sP \vDash (X, \ph, p)\}.
\end{equation}
  \symindex{$\bTh \sP$}%

\begin{lemma}\label{L:model-entire}
  Let $\sP = (\Om, \Si, \bbP)$ be a model. Let $X, Y \subseteq \cF$ and $\psi
  \in \cF$. Assume $X \equiv Y \cup \{\psi\}$ and $\sP \vDash Y$. Then $\sP
  \vDash (X, \ph, 1)$ if and only if $\psi_\Om \in \ol \Si$ and $\psi_\Om
  \subseteq \ph_\Om$ a.s.
\end{lemma}

\begin{proof}
  Without loss of generality, assume $\sP$ is complete. Suppose $\sP \vDash (X,
  \ph , 1)$. Write $X \equiv Y' \cup \{\psi'\}$, where $\sP \vDash Y'$ and $\bbP
  \ph_\Om \cap \psi'_\Om / \bbP \psi'_\Om = 1$. By Proposition
  \ref{P:model-func}, we have $\psi_\Om = \psi'_\Om$ a.s. Hence, $\bbP \ph_\Om
  \cap \psi_\Om / \bbP \psi_\Om = 1$, and this gives $\bbP \ph_\Om^c \cap
  \psi_\Om = 0$. Conversely, suppose $\psi_\Om \in \Si$ and $\psi_\Om \subseteq
  \ph_\Om$ a.s. Then $\bbP \psi_\Om \cap \ph_\Om^c = 0$, which implies $\bbP
  \psi_\Om \cap \ph_\Om = \bbP \psi_\Om$, so that $\sP \vDash (X, \ph, 1)$.
\end{proof}

\begin{thm}\label{T:model-ind-th}
  If $\sP$ is a model, then $\bTh \sP$ is a complete inductive theory with root
  $\Taut$.
\end{thm}

\begin{proof}
  Let $\sP$ be a model and let $P = \bTh \sP$. Without loss of generality,
  assume $\sP$ is complete. We first show that $P$ is admissible. Suppose $(X,
  \ph, p) \in P$, $X' \equiv X$, and $\ph' \equiv_X \ph$. Choose $Y$ and $\psi$
  such that $X \equiv Y \cup \{\psi\}$, $\sP \vDash Y$, and $\bbP \ph_\Om \cap
  \psi_\Om / \bbP \psi_\Om = p$. Then $Y, \psi \vdash \ph' \tot \ph$, so that
  $Y \vdash \psi \to (\ph \tot \ph')$. But $\sP \vDash Y$, so $\bbP \psi_\Om
  \cap (\ph'_\Om \tri \ph_\Om) = 0$. But
  \[
    (\ph'_\Om \tri \ph_\Om) \cap \psi_\Om
      = (\ph'_\Om \cap \psi_\Om) \tri (\ph_\Om \cap \psi_\Om).
  \]
  Thus, $\ph'_\Om \cap \psi_\Om = \ph_\Om \cap \psi_\Om$ a.s. Since $\sP$ is
  complete, this gives $\ph'_\Om \cap \psi_\Om \in \Si$ and $\bbP \ph'_\Om \cap
  \psi_\Om = \bbP \ph_\Om \cap \psi_\Om$. Hence, $\bbP \ph'_\Om \cap \psi_\Om /
  \bbP \psi_\Om = p$. Since $X' \equiv X \equiv Y \cup \{\psi\}$, it follows
  that $(X', \ph', p) \in P$.

  Now assume $(X', \ph', p') \in P$. By Corollary \ref{C:model-func}, we have
  $p = p'$, and therefore $P$ is admissible.

  We next show that $P$ is entire. Throughout the proof of entirety, we fix $X
  \in \ante P$. Choose $(X, \ph', p') \in P$. Write $X \equiv Y \cup \{\eta\}$,
  where $\sP \vDash Y$ and $\bbP \ph_\Om \cap \eta_\Om / \bbP \eta_\Om = p'$.

  Suppose $X \vdash \ph$. Then $Y \vdash \eta \to \ph$. Hence $\sP \vDash \eta
  \to \ph$, which implies $\bbP \eta_\Om \cap \ph_\Om^c = 0$. Since $\eta_\Om
  \in \Si$, it follows that $\bbP \eta_\Om \cap \ph_\Om = \bbP \eta_\Om$, and
  therefore $P(\ph \mid X) = 1$. Thus, $P$ satisfies the rule of logical
  implication.

  Suppose $P(\psi \mid X, \ph) = 1$. Since $X \cup \{\ph\} \equiv Y \cup \{\eta
  \wedge \ph\}$, Lemma \ref{L:model-entire} gives $\eta_\Om \cap \ph_\Om
  \subseteq \psi_\Om$ a.s. Thus, $\eta_\Om = (\eta_\Om \cap \ph_\Om) \cup 
  (\eta_\Om \cap \ph_\Om^c) \subseteq \psi_\Om \cup \ph_\Om^c$ a.s. Since $(\ph
  \to \psi)_\Om = \psi_\Om \cup \ph_\Om^c$, Lemma \ref{L:model-entire} implies
  $P(\ph \to \psi \mid X) = 1$, and $P$ satisfies the rule of material
  implication.

  Suppose $P(\ph \mid X) = 1$ and $\ph \vdash \psi$. By Lemma 
  \ref{L:model-entire}, we have $\eta_\Om \subseteq \ph_\Om$ a.s. Remark 
  \ref{R:impl-subset} shows that $\ph_\Om \subseteq \psi_\Om$. Hence, $\eta_\Om
  \subseteq \psi_\Om$ a.s., so that Lemma \ref{L:model-entire} implies $P(\psi
  \mid X) = 1$. Now suppose $X' \in \ante P$ and $X' \vdash X$. Write $X' = Y'
  \cup \{\eta'\}$, where $\sP \vDash Y'$ and $\eta'_\Om \in \Si$. Then $Y',
  \eta' \vdash Y, \eta$, so that $Y' \vdash \eta' \to \eta$. Hence, $\sP \vDash
  \eta' \to \eta$, which gives $\bbP \eta'_\Om \cap \eta_\Om^c = 0$. Thus,
  $\eta'_\Om \subseteq \eta_\Om$ a.s. It follows that $\eta'_\Om \subseteq
  \ph_\Om$ a.s., so that Lemma \ref{L:model-entire} gives $P(\ph \mid X') = 1$,
  and $P$ satisfies the rule of deductive transitivity.

  Suppose $X \vdash \neg (\ph \wedge \psi)$ and two of the probabilities in
  \eqref{add-rule} exist. Then $Y \vdash \eta \to \neg (\ph \wedge \psi)$, so
  that $\bbP \eta_\Om \cap \ph_\Om \cap \psi_\Om = 0$. Let $\ph' = \ph \wedge
  \eta$ and $\psi' = \psi \wedge \eta$. Then $\bbP \ph'_\Om \cap \psi'_\Om = 0$.
  Since two of the probabilities in \eqref{add-rule} exist, two of the sets
  $\ph'_\Om \cup \psi'_\Om$, $\ph'_\Om$, and $\psi'_\Om$ are in $\Si$. Since
  $\ph'_\Om \cap \psi'_\Om$ is also in $\Si$ and $\Si$ is a $\si$-algebra, it
  follows that all three sets are in $\Si$ and $\bbP \ph'_\Om \cup \psi'_\Om =
  \bbP \ph'_\Om + \bbP \psi'_\Om$. From here, \eqref{add-rule} follows
  immediately, and $P$ satisfies the addition rule.

  By Proposition \ref{P:model-func}, we have that $P(\ph \mid X)$ exists and is
  positive if and only if $\bbP \ph_\Om \cap \eta_\Om / \bbP \eta_\Om = p$, for
  some $p > 0$. Similarly, $P(\ph \wedge \psi \mid X)$ exists and is positive if
  and only if $\bbP \ph_\Om \cap \psi_\Om \cap \eta_\Om / \bbP \eta_\Om = r$,
  for some $r > 0$. Since $X \cup \{\ph\} \equiv Y \cup \{\ph \wedge \eta\}$,
  Proposition \ref{P:model-func} also gives that $P(\psi \mid X, \ph)$ exists
  and is positive if and only if $\bbP \ph_\Om \cap \psi_\Om \cap \eta_\Om /
  \bbP \ph_\Om \cap \eta_\Om = q$, for some $q > 0$. From this, it follows that
  if two of the probabilities in \eqref{mult-rule} exist and are positive, then
  all three exist and are positive, and $pq = r$. Hence, $P$ satisfies the
  multiplication rule.

  Now suppose $P(\ph_n \mid X) = p_n$ for all $n$, and $X, \ph_n \vdash \ph_{n
  + 1}$. By Proposition \ref{P:model-func}, we have $\bbP (\ph_n)_\Om \cap
  \eta_\Om / \bbP \eta_\Om = p_n$. We also have $Y \vdash \ph_n \wedge \eta \to
  \ph_{n + 1}$, so that $\sP \vDash \ph_n \wedge \eta \to \ph_{n + 1}$, which
  gives $\bbP (\ph_n)_\Om \cap \eta_\Om \cap (\ph_{n + 1})_\Om^c = 0$. Hence, $
  (\ph_n)_\Om \cap \eta_\Om \subseteq (\ph_{n + 1})_\Om$ a.s. This gives $
  (\ph_n)_\Om \cap \eta_\Om \subseteq (\ph_{n + 1})_\Om \cap \eta_\Om$ a.s.
  Since
  \[
    \ts{
      (\bigvee_n \ph_n)_\Om \cap \eta_\Om
        = (\bigcup_n (\ph_n)_\Om) \cap \eta_\Om
        = \bigcup_n ((\ph_n)_\Om \cap \eta_\Om),
    }
  \]
  it follows that $(\bigvee_n \ph_n)_\Om \cap \eta_\Om \in \Si$ and, using
  continuity from below, we have $\bbP (\bigvee_n \ph_n)_\Om \cap \eta_\Om /
  \bbP \eta_\Om = \lim_{n \to \infty} \bbP (\ph_n)_\Om \cap \eta_\Om / \bbP
  \eta_\Om$. Therefore, $P$ satisfies the continuity rule, and $P$ is entire.

  We next show that $P$ is complete. Suppose $P(\ph \mid X)$ exists. Then we may
  write $X \equiv Y \cup \{\eta\}$, where $\sP \vDash Y$,  $\bbP \ph_\Om \cap
  \eta_\Om$ exists, and $\bbP \eta_\Om > 0$. Assume $P(\psi \mid X)$ also
  exists. From Proposition \ref{P:model-func}, it follows that $\bbP \psi_\Om
  \cap \eta_\Om$ also exists. Since $\Si$ is a $\si$-algebra, we have
  \[
    (\ph \wedge \psi)_\Om \cap \eta_\Om
      = \ph_\Om \cap \psi_\Om \cap \eta_\Om
      = (\ph_\Om \cap \eta_\Om) \cap (\psi_\Om \cap \eta_\Om) \in \Si.
  \]
  Hence, $\bbP (\ph \wedge \psi)_\Om \cap \eta_\Om$ exists, and so therefore,
  $P(\ph \wedge \psi \mid X)$ exists, showing that $P$ satisfies Definition 
  \ref{D:complete}(i).

  Now suppose $X \in \ante P$. Then we may write $X \equiv Y \cup \{\psi\}$,
  where $\sP \vDash Y$ and $\bbP \psi_\Om > 0$. Assume also that $X \cup 
  \{\ph\} \in \ante P$. Since $X \cup \{\ph\} \equiv Y \cup \{\ph \wedge
  \psi\}$, Proposition \ref{P:model-func} implies that $\bbP \ph_\Om \cap
  \psi_\Om$ exists. Hence, $P(\ph \mid X)$ exists, and so $P$ satisfies
  Definition \ref{D:complete}(ii), showing that $P$ is complete.

  Since $P$ is complete, it is therefore semi-closed. We next show that $P$ is
  closed. Assume $S \subseteq \cF$ is nonempty with $P(\th \mid X) = 1$ for all
  $\th \in S$. Then we may write $X \equiv Y \cup \{\psi\}$, where $\psi_\Om
  \subseteq \th_\Om$ a.s.~for all $\th \in S$. Let $S' = \{\psi \to \th \mid \th
  \in S\}$ and $Y' = Y \cup S'$. By Lemma \ref{L:omit-psi}, we have $X \cup S
  \equiv Y' \cup \{\psi\}$. Also, for any $\th \in S$, we have $\Om = \psi_\Om^c
  \cup \psi_\Om \subseteq \psi_\Om^c \cup \th_\Om = (\psi \to \th)_\Om$ a.s.,
  which gives $\sP \vDash \psi \to \th$. Hence, $\sP \vDash Y'$, and therefore
  $P(\; \cdot \mid X) = P(\; \cdot \mid X, S)$, showing that $P$ satisfies the
  rule of deductive extension and is thus closed.

  Finally, we show that $P$ is connected with root $\Taut$. Let $P_0 = P
  \dhl_\Taut$. Corollary \ref{C:restrict-inherit} implies that $P_0$ is a
  complete pre-theory with root $T_0$. We will show that $P_0$ is a basis for
  $P$. Let $X \in \ante P$. Write $X \equiv Y \cup \{\psi\}$, where $\sP \vDash
  Y$ and $\bbP \psi_\Om > 0$. Since $\{\psi\} = \emp \cup \{\psi\}$, we have
  $\{\psi\} \in \ante P_0$. Let $\th \in Y$ be arbitrary. Since $\sP \vDash Y$,
  we have $\bbP \th_\Om = 1$. Therefore, $P(\th \mid \psi) = 1$, which gives
  $P_0(\th \mid \psi) = 1$. This shows that $Y \subseteq \tau(P_0; \{\psi\})$,
  proving that $P_0$ is a basis for $P$.
\end{proof}

\begin{rmk}\label{R:th-of-th}
  The relationship between $\bTh \sP$ in \eqref{bTh-sP-def} and $\Th \sP$ in 
  \eqref{Th-sP-def} is that if $P = \bTh \sP$, then $T_P = \Th \sP$. To see
  this, let $P = \bTh \sP$ and note that by Proposition \ref{P:simple-tau(P)},
  we have $\th \in T_P$ if and only if $P(\th \mid \Taut) = 1$, which holds if
  and only if $\sP \vDash (\Taut, \th, 1)$. But this holds if and only if
  $\olbbP \th_\Om = 1$, and this is the definition of $\sP \vDash \th$.
\end{rmk}

\subsection{Theories determine models}

\begin{thm}\label{T:ind-th-model}
  Let $P$ be a complete inductive theory with root $T_0$. Then there exists a
  model $\sP$ such that $T_P = \Th \sP$ and $P = \bTh \sP \dhl_{[T_0, \Th
  \sP]}$. In particular, every inductive theory is satisfiable.
\end{thm}

\begin{proof}
  Let $P$ be a complete inductive theory with root $T_0$. Let $\Om = \B^{PV}$
  and let $\Si = \{\ph_\Om \mid P(\ph \mid T_0) \text{ exists}\}$. Since $\Om =
  \top_\Om$, the rule of logical implication implies $\Om \in \Si$. Since
  $\ph_\Om^c = (\neg \ph)_\Om$, Corollary \ref{C:rel-neg} implies that $\Si$ is
  closed under complements. Let $\{A_n\} \subseteq \Si$ be pairwise disjoint,
  and let $A = \bigcup_n A_n$. For each $n$, choose $\ph_n$ such that $A_n =
  (\ph_n)_\Om$, and note that $A = \ph_\Om$, where $\ph = \bigvee_n \ph_n$. Also
  note that since $\{A_n\}$ are pairwise disjoint, we have $\vdash \neg (\ph_i
  \wedge \ph_j)$ for all $1 \le i < j < \infty$. Hence, Theorem \ref{T:ctbl-add}
  implies $A \in \Si$, so that $\Si$ is closed under countable, pairwise
  disjoint unions. It follows that $\Si$ is a Dynkin system. Since $P$ is
  complete, Definition \ref{D:complete}(i) implies that $\Si$ is closed under
  intersections. Therefore, $\Si$ is a $\si$-algebra.

  By Remark \ref{R:impl-subset} and the rule of logical equivalence, we may
  define $ {\bbP}: \Si \to [0, 1]$ by $\bbP \ph_\Om = P(\ph \mid T_0)$. Note
  that $\top_\Om = \Om$ and $\bot_\Om = \emp$, so that $\bbP \Om = 1$ and $\bbP
  \emp = 0$. As above, Theorem \ref{T:ctbl-add} implies that $\bbP$ is
  countably additive, so that $\bbP$ is a probability measure on $(\Om, \Si)$.

  Let $\sP = (\Om, \Si, \bbP)$ and let $\ol \sP = (\Om, \ol \Si, \olbbP)$ be its
  completion. Let $A \in \ol \Si \cap \cB^{PV}$. Since $A \in \cB^{PV}$, we may
  choose $\ph \in \cF$ such that $A = \ph_\Om$. Since $A \in \ol \Si$, we may
  write $A = \ph_\Om = \psi_\Om \cup N$, where $P(\psi \mid T_0)$ exists, $N
  \subseteq \eta_\Om$, and $P(\eta \mid T_0) = 0$. By \eqref{prob-neg} and the
  rule of logical implication, $P(\ph \wedge \eta \mid T_0) = 0$. Hence,
  $\ph_\Om \cap \eta_\Om \in \Si$. On the other hand, $\ph_\Om \cap \eta_\Om^c
  = \psi_\Om \cap \eta_\Om^c \in \Si$. Therefore, $A = \ph_\Om \in \Si$, and
  this shows that $\ol \Si \cap \cB^{PV} \subseteq \Si$. Since the reverse
  inclusion is trivial, we have $\ol \Si \cap \cB^{PV} = \Si$.

  We next show that for any $\ph, \psi \in \cF$, we have
  \begin{equation}\label{ind-th-model}
    P_0(\ph \mid T_0, \psi) = p \quad\text{iff}\quad
      \sP \vDash (T_0 + \psi, \ph, p).
  \end{equation}
  Suppose $P_0(\ph \mid T_0, \psi) = p$. Then $P(\ph \mid T_0, \psi) = p$. By
  Definition \ref{D:complete}(ii), Lemma \ref{L:cond-exist}, and the
  multiplication rule, we have that $P(\psi \mid T_0) > 0$ and
  \begin{equation}\label{cond-prob-th}
    \frac{P(\ph \wedge \psi \mid T_0)}{P(\psi \mid T_0)} = p.
  \end{equation}
  Hence, \eqref{cond-prob} holds, so $\sP \vDash (T_0 + \psi, \ph, p)$.
  Conversely, suppose that $\sP \vDash (T_0 + \psi, \ph, p)$. Then 
  \eqref{cond-prob} holds. Since $\ol \Si \cap \cB^{PV} = \Si$, the same
  equality holds for $\bbP$ instead of $\olbbP$. Hence, $P(\psi \mid T_0) > 0$,
  $P(\ph \wedge \psi \mid T_0)$ exists, and \eqref{cond-prob-th} holds, which,
  by the multiplication rule, gives $P(\ph \mid T_0, \psi) = p$, proving
  \eqref{ind-th-model}.

  Now, Theorem \ref{T:model-ind-th} implies that $\bTh \sP$ is a complete
  inductive theory with root $\Taut$, and Remark \ref{R:th-of-th} implies $T
  (\bTh \sP) = \Th \sP$. By the rule of logical implication, $\sP \vDash T_0$.
  Hence, $T_0 \in [\Taut, \Th \sP]$. It follows from Proposition 
  \ref{P:chop-off-root} that if we define $P_0' = \bTh \sP \dhl_{T_0}$ and $P' =
  \bfP(P_0')$, then $T_{P'} = \Th \sP$ and $P' = \bTh \sP \dhl_{[T_0, \Th
  \sP]}$. By \eqref{ind-th-model}, we have $P_0 = P_0'$. Hence, $P = P'$, and so
  it follows that $T_P = \Th \sP$ and $P = \bTh \sP \dhl_{[T_0, \Th \sP]}$. This
  proves the first claim of the theorem.

  For the second claim, let $P$ be an inductive theory with root $T_0$. Let $P_0
  = P \dhl_{T_0}$, so that $P = \bfP(P_0)$. Since $P_0$ is a pre-theory, it is
  semi-closed and, therefore, has a completion. By Corollary 
  \ref{C:compl-restrict}, we may choose a completion $\ol P_0$ that is also a
  pre-theory with root $T_0$. Let $\ol P = \bfP(\ol P_0)$. By Proposition 
  \ref{P:lift-complete}, the set $\ol P$ is a complete inductive theory with
  root $T_0$ such that $P \subseteq \ol P$. As shown above, we may construct a
  model $\sP$ such that $P \subseteq \ol P \subseteq \bTh \sP$. Hence, $\sP
  \vDash P$ and $P$ is satisfiable.
\end{proof}

\subsection{Consistency and satisfiability}

Note that by deductive completeness, the notions of connectivity and strong
connectivity can be completely characterized in terms of semantics. Therefore,
the following theorem shows that consistency can also be characterized in terms
of semantics.

\begin{thm}\label{T:ind-satis-cons}
  A set $Q \subseteq \cF^\IS$ is consistent if and only if it is connected and
  satisfiable.
\end{thm}

\begin{proof}
  Suppose $Q$ is connected and satisfiable. Choose a model $\sP$ such that $\sP
  \vDash Q$. Then $Q \subseteq \bTh \sP$. Theorem \ref{T:model-ind-th} implies
  $\bTh \sP$ is an inductive theory. Hence, $Q$ can be extended to an inductive
  theory, so $Q$ is consistent. Conversely, suppose $Q$ is consistent. Choose an
  inductive theory $P$ such that $Q \subseteq P$. By Theorem
  \ref{T:ind-th-model}, there exists a model $\sP$ such that that $\sP \vDash
  P$. Since $Q \subseteq P$, we have $\sP \vDash Q$, so $Q$ is satisfiable.
\end{proof}

Recall that $Q \vdash (X, \ph, p)$ means that $Q$ is consistent and $\bfP_Q (\ph
\mid X) = p$. As such, the following proposition gives a semantic
characterization of inductive derivability in the special case that $X \cao
T_0$, where $T_0$ is the root of $Q$. As a corollary, we find that $T_Q$ also
has a semantic characterization.

\begin{prop}\label{P:sound-compl-pre-th}
  Let $Q$ be connected and satisfiable, so that $Q$ is also consistent. Let
  $T_0$ be the root of $Q$. Then the following are equivalent:
  \begin{enumerate}[(i)]
    \item $\bfP_Q (\ph \mid T_0, \psi) = p$,
    \item for all models $\sP$, if $\sP \vDash Q$ and $\sP \vDash T_0$, then
          $\sP \vDash (T_0 + \psi, \ph, p)$.
  \end{enumerate}
\end{prop}

\begin{proof}
  Suppose that $\bfP_Q(\ph \mid T_0, \psi) = p$, and let $\sP \vDash Q$ and $\sP
  \vDash T_0$. Then $Q \subseteq \bTh \sP$. By Theorems \ref{T:model-ind-th} and
  \ref{T:theory-gen-defn}, this gives $\bfP_Q \subseteq \bTh \sP$, so that $\sP
  \vDash (T_0 + \psi, \ph, p)$.

  Conversely, suppose that for all models $\sP$, if $\sP \vDash Q$ and $\sP
  \vDash T_0$, then $\sP \vDash (T_0 + \psi, \ph, p)$. Let $P_0 = {\bfP_Q \dhl_
  {T_0}}$. Let $\ol P$ be a completion of $P_0$ and define $P = \bfP(\ol P \dhl_
  {T_0})$. Since $T_0 \in \ante \ol P$, Corollary \ref{C:restrict-inherit} and
  Proposition \ref{P:lift-complete} imply that $P$ is a complete inductive
  theory with root $T_0$. By Theorem \ref{T:ind-th-model}, we may choose a model
  $\sP$ such that $T_P = \Th \sP$ and $P = \bTh \sP \dhl_{[T_0, \Th \sP]}$. Note
  that $Q \subseteq \bfP_Q = \bfP(P_0) \subseteq \bfP(\ol P \dhl_{T_0}) = P
  \subseteq \bTh \sP$. Hence, $\sP \vDash Q$. Also $T_0 \subseteq T_P = \Th
  \sP$, so that $\sP \vDash T_0$. Therefore, by assumption, $\sP \vDash (T_0 +
  \psi, \ph, p)$. But $P = \bTh \sP \dhl_{[T_0, \Th \sP]}$, so this gives
  $P(\ph \mid T_0, \psi) = p$. Also, $P = \bfP({\ol P \dhl_{T_0}})$, which
  implies $P \dhl_{T_0} = \ol P \dhl_{T_0}$. Hence, $\ol P(\ph \mid T_0, \psi) =
  p$. Since $\ol P$ was arbitrary, the rule of inductive extension yields
  $P_0(\ph \mid T_0, \psi) = p$, and therefore $\bfP_Q(\ph \mid T_0, \psi) = p$.
\end{proof}

\begin{cor}\label{C:sound-compl-pre-th}
  Let $Q$ be connected and satisfiable, so that $Q$ is also consistent. Let
  $T_0$ be the root of $Q$. Then the following are equivalent:
  \begin{enumerate}[(i)]
    \item $\th \in T_Q$,
    \item for all models $\sP$, if $\sP \vDash Q$ and $\sP \vDash T_0$, then
          $\sP \vDash \th$.
  \end{enumerate}
\end{cor}

\begin{proof}
  Assume $\th \in T_Q$, so that $\bfP_Q(\th \mid T_0) = 1$. Suppose $\sP \vDash
  Q$ and $\sP \vDash T_0$. Proposition \ref{P:sound-compl-pre-th} implies $\sP
  \vDash (T_0, \th, 1)$. Since $T_0 \equiv T_0 \cup \{\top\}$ and $\sP \vDash
  T_0$, we have $\olbbP \th_\Om \cap \top_\Om / \olbbP \top_\Om = 1$. But
  $\top_\Om = \Om$, so $\olbbP \th_\Om = 1$, which means $\sP \vDash \th$.

  Now assume that for all models $\sP$, if $\sP \vDash Q$ and $\sP \vDash T_0$,
  then $\sP \vDash \th$. As above, if $\sP \vDash T_0$ and $\sP \vDash \th$,
  then $\sP \vDash (T_0, \th, 1)$. Hence, for all models $\sP$, if $\sP \vDash
  Q$ and $\sP \vDash T_0$, then $\sP \vDash (T_0, \th, 1)$. Proposition 
  \ref{P:sound-compl-pre-th} implies $\bfP_Q(\th \mid T_0) = 1$, so that by
  Proposition \ref{P:simple-tau(P)}, we have $\th \in T_Q$.
\end{proof}

\subsection{Inductive consequence and completeness}

Having characterized $T_Q$ in terms of semantics, we are now ready to define the
inductive consequence relation.

\begin{defn}\label{D:consequence}
  Let $Q \subseteq \cF^\IS$ and $(X, \ph, p) \in \cF^\IS$. We say that $(X, \ph,
  p)$ is a \emph{consequence} of $Q$, or that $Q$ \emph{entails} $ (X, \ph, p)$,
  which we denote by $Q \vDash (X, \ph, p)$, if
  \begin{enumerate}[(i)]
    \item $Q$ is connected and satisfiable,
    \item $X \cao [T_0, T_Q]$, where $T_0$ is the root of $Q$, and
    \item $\sP \vDash Q$ implies $\sP \vDash (X, \ph, p)$, for all models
          $\sP$.
  \end{enumerate}
\end{defn}
  \index{consequence relation}%
  \symindex{$Q \vDash (X, \ph, p)$}%

Corollary \ref{C:sound-compl-pre-th} shows that $T_Q$ has an entirely semantic
characterization. Hence, Definition \ref{D:consequence} is an entirely semantic
definition.

Note that $\sP \vDash (\Taut, \top, 1)$ for all models $\sP$. However, $Q
\vdash (\Taut, \top, 1)$ only if the root of $Q$ is $\Taut$. Hence, (ii) cannot
be removed if we are to have completeness.

To simplify the verification that $Q \vDash (X, \ph, p)$, we can replace (iii)
with
\begin{enumerate}[(i)$'$]
  \setcounter{enumi}{2}
  \item $\sP \vDash Q$ implies $\sP \vDash (X, \ph, p)$ whenever $\sP$ is
        complete and $\sP \vDash T_0$.
\end{enumerate}
We show this below in Theorem \ref{T:conseq-defn-simple}, after proving two
lemmas. Let $\sQ = (\Om, \Si, \bbQ)$ be a complete model and assume that $\bbQ
\ze_\Om > 0$ for some $\ze \in \cF$. Define the probability measure $\bbP$ on $
(\Om, \Si)$ by $\bbP A = \bbQ A \cap \ze_\Om / \bbQ \ze_\Om$ and let $\sP =
(\Om, \Si, \bbP)$. Note that if $\bbQ A = 1$, then $\bbP A = 1$. Hence, $\sQ
\vDash Y$ implies $\sP \vDash Y$ for all $Y \subseteq \cF$.

\begin{lemma}\label{L:cond-compl}
  Let $\sQ$ and $\sP$ be as above and let $\ol \sP = (\Om, \ol \Si, \olbbP)$ be
  the completion of $\sP$. Then $A \in \ol \Si$ if and only if $A \cap \ze_\Om
  \in \Si$, and in this case, $\olbbP A = \bbP A \cap \ze_\Om$.
\end{lemma}

\begin{proof}
  Suppose $A \cap \ze_\Om \in \Si$. Since $\bbP \ze_\Om^c = 0$ and $A \cap
  \ze_\Om^c \subseteq \ze_\Om^c$, we have $A \cap \ze_\Om^c \in \ol \Si$. Hence,
  $A = (A \cap \ze_\Om) \cup (A \cap \ze_\Om^c) \in \ol \Si$ and $\olbbP A =
  \bbP A \cap \ze_\Om$. Conversely, suppose $A \in \ol \Si$. Write $A = B \cup
  F$, where $B \in \Si$, $F \subseteq N$, and $N \in \Si$ with $\bbP N = 0$. By
  the definition of $\bbP$, we have $\bbQ N \cap \ze_\Om = 0$. Now write $A \cap
  \ze_\Om = (B \cap \ze_\Om) \cup (F \cap \ze_\Om)$. Then $B \cap \ze_\Om \in
  \Si$. Also, since $F \cap \ze_\Om \subseteq N \cap \ze_\Om$ and $\sQ$ is
  complete, it follows that $F \cap \ze_\Om \in \Si$. Therefore, $A \cap \ze_\Om
  \in \Si$. Moreover, this shows that $\olbbP A = \bbP B$ and $\bbP A \cap
  \ze_\Om = \bbP B \cap \ze_\Om$. Since $\bbP \ze_\Om = 1$, we have $\bbP B =
  \bbP B \cap \ze_\Om$. Therefore, $\olbbP A = \bbP A \cap \ze_\Om$.
\end{proof}

\begin{lemma}\label{L:conseq-defn-simple}
  Let $\sQ$ and $\sP$ be as above. If $\sQ \vDash (X, \ph, p)$ and $\ze \in T
  (X)$, then $\sP \vDash (X, \ph, p)$.
\end{lemma}

\begin{proof}
  Suppose $\sQ \vDash (X, \ph, p)$ and $\ze \in T(X)$. Write $X \equiv Y \cup 
  \{\psi\}$, where $\sQ \vDash Y$. Then $\sP \vDash Y$ also. Since $\ze \in
  T(X)$, we have $X \equiv Y \cup \{\psi, \ze\} \equiv Y \cup \{\psi \wedge
  \ze\}$. By Proposition \ref{P:model-func} and the fact that $\sQ$ is complete,
  it follows that
  \[
    p = \frac{
      \bbQ \ph_\Om \cap \psi_\Om \cap \ze_\Om
    }{
      \bbQ \psi_\Om \cap \ze_\Om
    }
    = \frac{\bbP \ph_\Om \cap \psi_\Om}{\bbP \psi_\Om}.
  \]
  Therefore, $\sP \vDash (X, \ph, p)$.
\end{proof}

\begin{thm}\label{T:conseq-defn-simple}
  In Definition \ref{D:consequence}, we may replace (iii) with (iii)$'$.
\end{thm}

\begin{proof}
  Clearly, Definition \ref{D:consequence} implies (iii)$'$. For the converse,
  let $Q \subseteq \cF^\IS$ and $(X, \ph, p) \in \cF^\IS$. Assume (i), (ii), and
  (iii)$'$. To show that (iii) holds, let $\sQ$ be a model and assume that $\sQ
  \vDash Q$. We want to show that $\sQ \vDash (X, \ph, p)$. As noted below
  \eqref{cond-prob}, we may assume without loss of generality that $\sQ$ is
  complete.

  By Proposition \ref{P:root-exist}, we may choose $X_0$ such that $T_0 = T
  (X_0)$ and $X_0 \in \ante Q$. We may then choose $(X_0, \xi, q) \in Q$, so
  that $\sQ \vDash (X_0, \xi, q)$. Write $X_0 \equiv Y \cup \{\ze\}$, where $\sQ
  \vDash Y$ and $\bbQ \ze_\Om > 0$. Define $\sP = (\Om, \Si, \bbP)$ by $\bbP A =
  \bbQ A \cap \ze_\Om / \bbQ \ze_\Om$. Then $\sP \vDash Y$ and $\sP \vDash \ze$.
  Hence, $\sP \vDash X_0$, which gives $\sP \vDash T_0$. Let $(X', \ph', p') \in
  Q$. Then $\ze \in T_0 \subseteq T(X')$. By Lemma \ref{L:conseq-defn-simple},
  it follows that $\sP \vDash (X', \ph', p')$. This shows that $\sP \vDash Q$.
  We therefore have $\ol \sP \vDash T_0$ and $\ol \sP \vDash Q$. By (iii)$'$,
  this gives $\ol \sP \vDash (X, \ph, p)$, which implies $\sP \vDash (X, \ph,
  p)$.

  By (ii), we may write $X \equiv T + \psi$, where $T \in [T_0, T_Q]$. Since $T
  \subseteq T_Q$ and $\sQ \vDash Q$, we have $\sQ \vDash T$. It remains only to
  show that $\bbQ \ph_\Om \cap \psi_\Om / \bbQ \psi_\Om = p$. For this, note
  that $\ze \in T_0 \subseteq T(X)$, so that $X \equiv T + \psi \wedge \ze$.
  Since $\sQ \vDash T$, we also have $\sP \vDash T$. Hence, by Proposition
  \ref{P:model-func} and Lemma \ref{L:cond-compl}, it follows that
  \[
    p = \frac{
      \bbP \ph_\Om \cap \psi_\Om \cap \ze_\Om
    }{
      \bbP \psi_\Om \cap \ze_\Om
    }
    = \frac{\bbQ \ph_\Om \cap \psi_\Om}{\bbQ \psi_\Om}.
  \]
  Therefore, $\sQ \vDash (X, \ph, p)$, which shows that (iii) holds.
\end{proof}

Having defined the inductive consequence relation, we now show that it is
identical to the inductive derivability relation.

\begin{thm}[Inductive soundness and completeness]\label{T:ind-sound-compl}
    \index{soundness!inductive ---}%
    \index{completeness!inductive ---}%
  Let $Q \subseteq \cF^\IS$ and $(X, \ph, p) \in \cF^\IS$. Then $Q \vdash (X,
  \ph, p)$ if and only if $Q \vDash (X, \ph, p)$.
\end{thm}

\begin{proof}
  Suppose $Q \vdash (X, \ph, p)$. Then $Q$ is consistent and $\bfP_Q(\ph \mid X)
  = p$. By Remark \ref{R:ante-cao}, we have $X \cao [T_0, T_Q]$, where $T_0$ is
  the root of $Q$. Theorem \ref{T:ind-satis-cons} implies $Q$ is connected and
  satisfiable. Suppose $\sP \vDash Q$. Theorems \ref{T:model-ind-th} and
  \ref{T:theory-gen-defn} implies $\sP \vDash \bfP_Q$. Hence, $\sP \vDash (X,
  \ph, p)$.

  For the converse, suppose $Q \vDash (X, \ph, p)$. We need to show that
  $\bfP_Q(\ph \mid X) = p$. By Definition \ref{D:consequence}(ii), we may write
  $T(X) = T + \psi$, where $T \in [T_0, T_Q]$. Hence, it suffices to show
  $\bfP_Q(\ph \mid T_0, \psi) = p$. We will do this using Proposition
  \ref{P:sound-compl-pre-th}.

  Suppose $\sP \vDash Q$ and $\sP \vDash T_0$. Then $T_0 \subseteq \Th \sP$, so
  by Remark \ref{R:th-of-th} and Proposition \ref{P:chop-off-root}, if we define
  $P = \bTh \sP \dhl_ {[T_0, \Th \sP]}$, then $P$ is an inductive theory with
  root $T_0$, and $T_P = \Th \sP$. Corollary \ref{C:sound-compl-pre-th} gives
  $T_Q \subseteq \Th \sP$, so that $T \in [T_0, \Th \sP]$. Hence, $X \cao [T_0,
  \Th \sP]$. Moreover, Definition \ref{D:consequence}(iii) implies $(X, \ph, p)
  \in \bTh \sP$. Hence, $P(\ph \mid X) = p$. But $T(X) = T + \psi$ and $T \in 
  [T_0, T_P]$. Therefore, $P(\ph \mid T_0, \psi) = p$. This implies $\sP \vDash 
  (T_0 + \psi, \ph, p)$, so by Proposition \ref{P:sound-compl-pre-th}, we have
  $\bfP_Q(\ph \mid T_0, \psi) = p$.
\end{proof}

\begin{rmk}\label{R:classic-ind-th-char}
  According to Remark \ref{R:classic-ind-th-defn} and Theorems
  \ref{T:ind-satis-cons} and \ref{T:ind-sound-compl}, a connected and
  satisfiable set $P \subseteq \cF^\IS$ is an inductive theory if and only if $P
  \vDash (X, \ph, p)$ implies $(X, \ph, p) \in P$ for all $(X, \ph, p) \in
  \cF^\IS$.
\end{rmk}

\subsection{Differing roots}\label{S:diff-root}

Proposition \ref{P:chop-off-root} shows that if $P$ is an inductive theory with
root $T_0$, and $T_0' \in [T_0, T_P]$, then $P' = P \dhl_{[T_0', T_P]}$ is an
inductive theory with root $T_0'$. The difference between these two theories is
described entirely by the sentences in $T_0' \setminus T_0$. In $P'$, these
sentences are part of the root, which means they are part of every antecedent.
In $P$, they are part of $T_P$, which means they have probability one under
every antecedent. In either case, we are assuming such a sentence is ``true,''
in one sense or another. It may be tempting, then, to think that these two
theories are effectively the same. In fact, for any model $\sP$, if $\sP \vDash
P$, then $\sP \vDash P'$. The converse, however, is not true. The theory with
the lower root, $P$, has fewer models, as we illustrate below in Proposition
\ref{P:drop-root-fewer-models}. In other words, by placing a hypothesis in $T_P$
rather than $T_0$, we are making a semantically stronger assumption.
Intuitively, the sentences in $T_0$ are only hypothetical postulates. The
inductive statements in $P$ are all assertions about the case in which we are
given $T_0$. A sentence $\ze \in T_P \setminus T_0$ has a different status. In
that case, we have $P(\ze \mid T_0) = 1$. Hence, the inductive theory $P$ is
asserting that $\ze$ is entailed (probabilistically) by $T_0$.

To illustrate this fact, let $\sQ = (\Om, \Si, \bbQ)$ be a complete model and
assume that $\bbQ \ze_\Om \in (0, 1)$ for some $\ze \in \cF$. Define the
probability measure $\bbP$ on $ (\Om, \Si)$ by $\bbP A = \bbQ A \cap \ze_\Om /
\bbQ \ze_\Om$ and let $\sP = (\Om, \Si, \bbP)$. Let $(\Om, \ol \Si, \olbbP)$ be
the completion of $(\Om, \Si, \bbP)$. Note that $\bTh \sP$ is an inductive
theory with root $\Taut$, and $T_0 = T(\ze) \in [\Taut, \Th \sP]$. Let $P = \bTh
\sP \dhl_{[T_0, \Th \sP]}$.

\begin{prop}\label{P:drop-root-fewer-models}
  With notation as above, we have $\sQ \vDash P$, but $\sQ \nvDash \bTh \sP$.
\end{prop}

\begin{proof}
  Since $\sP \vDash (\Taut, \ze, 1)$, but $\bbQ \ze_\Om < 1$, we have $\sQ
  \nvDash \bTh \sP$. Suppose $P(\ph \mid T_0, \psi) = p$. Then $\olbbP \ph_\Om
  \cap \psi_\Om / \olbbP \psi_\Om = p$. By Lemma \ref{L:cond-compl},
  \[
    p = \frac{
      \bbP \ph_\Om \cap \psi_\Om \cap \ze_\Om
    }{
      \bbP \psi_\Om \cap \ze_\Om
    } = \frac{
      \bbQ \ph_\Om \cap \psi_\Om \cap \ze_\Om
    }{
      \bbQ \psi_\Om \cap \ze_\Om
    } = \frac{
      \bbQ \ph_\Om \cap (\psi \wedge \ze)_\Om
    }{
      \bbQ (\psi \wedge \ze)_\Om
    }.
  \]
  Since $T_0 \cup \{\psi\} \equiv \emp \cup \{\ze \wedge \psi\}$ and $\sQ \vDash
  \emp$, we have $\sQ \vDash (T_0 \cup \{\psi\}, \ph, p)$. This shows that $\sQ
  \vDash P \dhl_{T_0}$. By Proposition \ref{P:chop-off-root}, it follows that
  $\sQ \vDash P$.
\end{proof}

\subsection{The semantics of inductive conditions}

Remark \ref{R:classic-ind-th-char} gives an entirely semantic characterization
of inductive theories. Consequently, inductive conditions can also be
characterized semantically. We are therefore ready to extend the notions of
satisfiability and consequence to inductive conditions.

We say that a model $\sP$ \emph{satisfies} an inductive condition $\cC$ if $\sP
\vDash P$ for some $P \in \cC$. An inductive condition is \emph{satisfiable} if
$\sP \vDash \cC$ for some model $\sP$. The following proposition shows that this
notion of satisfiability is an extension of our previous definition.
  \index{satisfiable}%
  \symindex{$\sP \vDash \cC$}%

\begin{prop}\label{P:set-cond-satis}
  Let $Q \subseteq \cF^\IS$ be connected and let $\sP$ be a model. Then $\sP
  \vDash Q$ if and only if $\sP \vDash \cC_Q$.
\end{prop}

\begin{proof}
  Suppose $\sP \vDash \cC_Q$. Choose $P \in \cC_Q$ such that $\sP \vDash P$.
  Since $Q \subseteq P$, we have $\sP \vDash Q$. Conversely, suppose $\sP \vDash
  Q$. Theorem \ref{T:ind-satis-cons} implies $Q$ is consistent, so we may define
  $\bfP_\cC$. Since $Q \subseteq \bTh \sP$ and Theorem \ref{T:model-ind-th}
  implies $\bTh \sP$ is an inductive theory, we have $\bfP_Q \subseteq \bTh
  \sP$. That is, $\sP \vDash \bfP_Q$. But $\bfP_Q \in \cC_Q$. Hence, $\sP \vDash
  \cC_Q$.
\end{proof}

The next two results show that both the consistency of $\cC$ and the deductive
theory $T_\cC$ have semantic characterizations.

\begin{thm}\label{T:ind-satis-cons-IC}
  An inductive condition is consistent if and only if it is satisfiable.
\end{thm}

\begin{proof}
  Let $\cC$ be an inductive condition. If $\cC$ is satisfiable, then by
  definition, it is nonempty, and therefore consistent. For the converse,
  suppose $\cC$ is consistent. Choose $P \in \cC$. By Theorem 
  \ref{T:ind-th-model}, we may choose a model $\sP$ such that $\sP \vDash P$.
  Hence, $\sP \vDash \cC$, and $\cC$ is satisfiable.
\end{proof}

\begin{prop}
  Let $\cC$ be a satisfiable inductive condition, so that $\cC$ is also
  consistent. Let $T_0$ be the root of $\cC$. Then the following are equivalent:
  \begin{enumerate}[(i)]
    \item $\th \in T_\cC$,
    \item for all models $\sP$, if $\sP \vDash \cC$ and $\sP \vDash T_0$, then
          $\sP \vDash \th$.
  \end{enumerate}
\end{prop}

\begin{proof}
  Suppose $\th \in T_\cC$. Assume $\sP \vDash \cC$ and $\sP \vDash T_0$. Choose
  $P \in \cC$ such that $\sP \vDash P$. Since $T_\cC = \bigcap \{T_P \mid P \in
  \cC\}$, we have $\th \in T_P$. From Corollary \ref{C:sound-compl-pre-th}, it
  follows that $\sP \vDash \th$.

  Now suppose that for all models $\sP$, if $\sP \vDash \cC$ and $\sP \vDash
  T_0$, then $\sP \vDash \th$. Assume $\th \notin T_\cC$. Then we may choose $P
  \in \cC$ such that $\th \notin T_P$. By Corollary \ref{C:sound-compl-pre-th},
  we may choose a model $\sP$ such that $\sP \vDash P$, $\sP \vDash T_0$, and
  $\sP \nvDash \th$. Since $\sP \vDash P$ and $P \in \cC$, we have $\sP \vDash
  \cC$. Hence, by our initial supposition, $\sP \vDash \th$, a contradiction.
\end{proof}

Having characterized $T_\cC$ in terms of semantics, we can now extend the
consequence relation to inductive conditions.

\begin{defn}\label{D:consequence-IC}
  We say that $(X, \ph, p) \in \cF^\IS$ is a \emph{consequence} of an inductive
  condition $\cC$, or that $\cC$ \emph{entails} $(X, \ph, p)$, which we denote
  by $\cC \vDash (X, \ph, p)$, if
  \begin{enumerate}[(i)]
    \item $\cC$ is satisfiable,
    \item $X \cao [T_0, T_\cC]$, where $T_0$ is the root of $\cC$, and
    \item $\sP \vDash \cC$ implies $\sP \vDash (X, \ph, p)$, for all models
          $\sP$.
  \end{enumerate}
\end{defn}
  \index{consequence relation}%
  \symindex{$\cC \vDash (X, \ph, p)$}%

From Proposition \ref{P:set-cond-satis}, it follows that for any connected $Q
\subseteq \cF^\IS$, we have $Q \vDash (X, \ph, p)$ if and only if $\cC_Q \vDash 
(X, \ph, p)$. Hence, Definition \ref{D:consequence-IC} is a generalization of
Definition \ref{D:consequence}.

As with Definition \ref{D:consequence}, we cannot remove (ii). In fact, in
Section \ref{S:D-conseq-need2}, we provide an example where (i) and (iii) above
are satisfied, but (ii) fails because $X$ is too large. (See Remark 
\ref{R:D-conseq-need2}.)

\begin{thm}[Soundness and completeness for conditions]
    \index{soundness!for inductive conditions}%
    \index{completeness!for inductive conditions}%
  \label{T:ind-sound-compl-IC}
  Let $\cC$ be an inductive condition and $(X, \ph, p) \in \cF^\IS$. Then $\cC
  \vdash (X, \ph, p)$ if and only if $\cC \vDash (X, \ph, p)$.
\end{thm}

\begin{proof}
  Suppose $\cC \vdash (X, \ph, p)$. Then $\cC$ is consistent and $\bfP_\cC(\ph
  \mid X) = p$. By Remark \ref{R:ante-cao}, we have $X \cao [T_0, T_\cC]$, where
  $T_0$ is the root of $\cC$. Theorem \ref{T:ind-satis-cons-IC} implies $\cC$ is
  satisfiable. Suppose $\sP \vDash \cC$. Choose $P \in \cC$ such that $\sP
  \vDash P$. Since $\bfP_\cC \subseteq P$, we have $P(\ph \mid X) = p$. By
  Remark \ref{R:classic-ind-th-char}, this implies $P \vDash (X, \ph, p)$. But
  $T_\cC \subseteq T_P$, so $X \cao [T_0, T_P]$. Therefore, Definition
  \ref{D:consequence}(iii) gives $\sP \vDash (X, \ph, p)$.

  For the converse, suppose $\cC \vDash (X, \ph, p)$. We need to show that
  $\bfP_\cC(\ph \mid X) = p$. By Definition \ref{D:consequence-IC}(ii), we may
  write $T(X) = T + \psi$, where $T \in [T_0, T_\cC]$. Hence, it suffices to
  show that $\bfP_\cC(\ph \mid T_0, \psi) = p$.

  Let $P \in \cC$ be arbitrary. If $\sP \vDash P$, then $\sP \vDash \cC$, so by
  supposition, $\sP \vDash (X, \ph, p)$. Since $T_\cC \subseteq T_P$, we have
  $X \cao [T_0, T_P]$. Thus, $P \vDash (X, \ph, p)$. By Remark 
  \ref{R:classic-ind-th-char}, this gives $P(\ph \mid X) = p$. But $T(X) = T +
  \psi$, where $T \in [T_0, T_\cC] \subseteq [T_0, T_P]$. Therefore, $P(\ph
  \mid T_0, \psi) = p$. Since $P$ was arbitrary, it follows that $(T_0 + \psi,
  \ph, p) \in \bigcap \cC^0 \subseteq \bfP_\cC$, so $\bfP_\cC(\ph \mid T_0,
  \psi) = p$.
\end{proof}

\section{Counterexamples and resolutions I}\label{S:Examples1}

In this section, we present several examples that serve to illustrate the
necessity of rules (R8) and (R9). More specifically, entire sets, which are
closed under only (R1)--(R7), exhibit a number of pathological behaviors. These
behaviors were alluded to in Chapter \ref{Ch:prop-calc}. In this section, we
provide concrete examples.

To develop these examples, we must expand our tools for creating inductive
theories and entire sets. We do this in Sections \ref{S:every-prob-sp} and 
\ref{S:Dynkin-sp}.

\subsection{Every probability space is a model}\label{S:every-prob-sp}

A model is a particular type of probability space, namely one in which $\Om$ is
a set of strict models. Theorem \ref{T:prob-sp-model-iso} below shows that every
probability space, regardless of $\Om$, is isomorphic to a model. More
specifically, for every probability space, there is an appropriate choice of
$PV$ such that the given probability space is isomorphic to a model in
$\cF(PV)$.

Later, in Theorem \ref{T:prob-model-iso}, we will give a version of this result
in predicate logic. Theorem \ref{T:prob-model-iso} will not only be concerned
with an arbitrary probability space, but also with an arbitrary collection of
random variables on that space.

\begin{thm}\label{T:prob-sp-model-iso}
  Let $PV$ be a given set of propositional variables and let $\cF = \cF(PV)$.
  Let $(S, \Ga, \nu)$ be an arbitrary probability space. Then $(S, \Ga, \nu)$
  has a subspace that is isomorphic to a model in $\cF$. If $\card(PV) \ge \card
  (\Ga)$, then $(S, \Ga, \nu)$ itself is isomorphic to a model. In particular,
  every probability space is isomorphic to a model in $\cF(PV)$ for an
  appropriate choice of $PV$.
\end{thm}

\begin{proof}
  Let $G: PV \to \Ga$. If $\card(PV) \ge \card(\Ga)$, then take $G$ to be
  surjective. Extend $G$ to $\cF$ by $G \neg \ph = (G \ph)^c$ and $G
  \bigwedge_n \ph_n = \bigcap_n G \ph_n$. Let $\Theta = G \cF \subseteq \Ga$ and
  note that $\Theta$ is a $\si$-algebra. Hence, $(S, \Theta, \nu|_\Theta)$ is a
  measure subspace of $(S, \Ga, \nu)$. If $\card(PV) \ge \card(\Ga)$, then
  $\Theta = \Ga$. We abuse notation and simply write $\nu$ for both $\nu$ and
  $\nu|_\Theta$. We will show that $(S, \Theta, \nu)$ is isomorphic to a model
  in $\cF$. More specifically, if $\Om = \B^{PV}$ is the set of all strict
  models, and $\Si = \cB^{BV} = \{\ph_\Om \mid \ph \in \cF\}$, then we will
  construct a probability measure $\bbP$ on $(\Om, \Si)$ such that $(S, \Ga,
  \nu)$ and $(\Om, \Si, \bbP)$ are isomorphic.

  For $x \in S$, define the strict model $\om^x$ by $\om^x \tDash \bfr$ if and
  only if $x \in G \bfr$, for all $\bfr \in PV$. By formula induction, it
  follows that $\om^x \tDash \ph$ if and only if $x \in G \ph$, for all $\ph \in
  \cF$. Define $h: S \to \Om$ by $h x = \om^x$. After constructing $\bbP$, we
  will show that $h$ induces an isomorphism.

  To construct $\bbP$, we first prove that $\ph \equiv \psi$ implies $G \ph = G
  \psi$. Recall the set of axioms, $\La$, defined in Section
  \ref{S:Hilbert-calc}. It is straightforward to verify that $G \ph = S$ if $\ph
  \in \La$ is an axiom. Suppose $\ph$ is a tautology. Then there is a proof of
  $\ph$ from $\emp$. Using induction of the length of the proof, as in the proof
  of Proposition \ref{P:Hilbert-induc}, one readily verifies that $G \ph = S$.
  Now suppose $\ph \equiv \psi$. Then $\ph \tot \psi$ is a tautology. Hence $G
  \ph \tot \psi = S$. But $G \ph \tot \psi = (G \ph \tri G \psi)^c$. Therefore,
  $G \ph = G \psi$, proving the claim.

  By Remark \ref{R:impl-subset}, we have $\ph_\Om = \psi_\Om$ if and only if
  $\ph \equiv \psi$. Hence, if $\ph_\Om = \psi_\Om$, then $G \ph = G \psi$. We
  may therefore define $g: \Si \to \Theta$ by $g \ph_\Om = G \ph$. Let $A \in
  \Si$ and choose $\ph \in \cF$ such that $A = \ph_\Om = \{\om \in \Om \mid \om
  \tDash \ph\}$. Then
  \begin{align*}
    x \in h^{-1} A &\quad\text{ iff }\quad h x \in A\\
    &\quad\text{ iff }\quad \om^x \in \ph_\Om\\
    &\quad\text{ iff }\quad \om^x \tDash \ph\\
    &\quad\text{ iff }\quad x \in G \ph = g \ph_\Om = g A.
  \end{align*}
  Hence, $h^{-1} = g$, so that $h^{-1} \ph_\Om = G \ph$. In particular, this
  shows that $h$ is $(\Theta, \Si)$-measurable, so we may define ${\bbP} = \nu
  \circ h^{-1} = \nu \circ g$, making $\sP = (\Om, \Si, \bbP)$ a model.

  To verify that $h$ induces an isomorphism from $(S, \Theta, \nu)$ to $\sP$, we
  must check that for each $U \in \Theta$, there exists $A \in \Si$ such that
  $h^{-1} A = U$ $\nu$-a.s. Let $U \in \Theta$ and choose $\ph \in \cF$ such
  that $U = G \ph$. Then $A = \ph_\Om \in \Si$ and, by the above, $h^{-1} A = G
  \ph = U$.
\end{proof}

\subsection{Dynkin spaces}\label{S:Dynkin-sp}

Theorem \ref{T:model-ind-th} gives us the means to construct inductive theories.
For the first set of examples in the section, however, we need to construct
entire sets that are not inductive theories. We will do this using Dynkin
systems. More specifically, we will define what we call Dynkin spaces, a
generalization of probability spaces that use Dynkin systems instead of
$\si$-algebras. Then, analogous to Theorem \ref{T:model-ind-th}, we will use
these Dynkin spaces to construct entire sets.

\begin{defn}\label{D:Dynk-sp}
  A \emph{Dynkin space} is a triple, $(S, \De, \oprho)$, where $S$ is a nonempty
  set, $\De$ is a Dynkin system on $S$, and $\oprho: \De \to [0, 1]$ satisfies
  \begin{enumerate}[(i)]
    \item $\oprho \Om = 1$,
    \item if $A, B \in \De$ with $A \subseteq B$, then $\oprho B \setminus A =
    \oprho B - \oprho A$, and
    \item if $\{A_n\} \subseteq \De$ with $A_n \subseteq A_{n + 1}$, then
          $\oprho \bigcup A_n = \lim \oprho A_n$.
  \end{enumerate}
\end{defn}

Let $(S, \De, \oprho)$ be a Dynkin space. A set $A \in \De$ is called \emph{
(Dynkin) measurable}. A set $A \in \De$ is \emph{null} if $A \in \De$ and
$\oprho A = 0$. Note that a measurable subset of a null set is a null set. A
Dynkin space is \emph{complete} if every subset of a null set is a null set.

\begin{prop}
  If $(S, \De, \oprho)$ is a Dynkin space, then $\oprho$ is countably additive.
  That is, if $\{A_n\} \subseteq \De$ is pairwise disjoint, then $\oprho \bigcup
  A_n = \sum \oprho A_n$.
\end{prop}

\begin{proof}
  Let $A, B \in \De$ be disjoint. Since $\De$ is a Dynkin system, it is closed
  under countable, pairwise disjoint unions. Hence, $A \cup B \in \De$. By
  Definition \ref{D:Dynk-sp}(ii), we have $\oprho B = \oprho A \cup B - \oprho
  A$. Therefore, $\oprho$ is finitely additive. From Definition
  \ref{D:Dynk-sp}(iii), we obtain $\oprho \bigcup A_n = \lim_n \sum_1^n \oprho
  A_j$.
\end{proof}

\begin{prop}\label{P:Dynk-complete-reverse}
  Let $(S, \De, \oprho)$ be a complete Dynkin space. If $A \in \De$, $\oprho A =
  1$, and $A \subseteq B$, then $\oprho B = 1$.
\end{prop}

\begin{proof}
  Suppose $A \in \De$, $\oprho A = 1$, and $A \subseteq B$. Then $A^c \in \De$,
  $\oprho A^c = 0$, and $B^c \subseteq A^c$. Since $(S, \De, \oprho)$ is
  complete, this gives $\oprho B^c = 0$, which implies $\oprho B = 1$.
\end{proof}

\begin{prop}\label{P:Dynk-sp-certainty}
  Let $(S, \De, \oprho)$ be a complete Dynkin space and let $\{A_n\} \subseteq
  \De$. If $\oprho A_n = 0$ for all $n$, then $\oprho \bigcup A_n = 0$. If
  $\oprho A_n = 1$ for all $n$, then $\oprho \bigcap A_n = 1$.
\end{prop}

\begin{proof}
  Assume $\oprho A_n = 0$ for all $n$. Let $B_n = A_n \setminus \bigcup_1^{n -
  1} A_j$. Then $B_n \subseteq A_n$. Since $(S, \De, \oprho)$ is complete, we
  have $\oprho B_n = 0$. Since $\{B_n\}$ is pairwise disjoint, we also have
  $\bigcup B_n \in \De$. But $\bigcup B_n = \bigcup A_n$, so $\oprho \bigcup A_n
  = \oprho \bigcup B_n = \sum \oprho B_n = 0$. For the second claim, apply the
  first to $A_n^c$.
\end{proof}

\begin{lemma}\label{L:ThD-is-theory}
  Let $\Om = \B^{PV}$ and let $\De \subseteq \cB^{PV}$ be a Dynkin system.
  Suppose $\sD = (\Om, \De, \oprho)$ is a complete Dynkin space. Define
  \[
    \Th \sD = \{\ph \in \cF \mid \oprho \ph_\Om = 1\}.
  \]
  Then $\Th \sD$ is a consistent deductive theory.
\end{lemma}

\begin{proof}
  Suppose $\Th \sD \vdash \ph$. Choose countable $\Phi \subseteq \Th \sD$ such
  that $\vdash \bigwedge \Phi \to \ph$. Then $\om \tDash \bigwedge \Phi \to \ph$
  for all strict models $\om$. Hence, $(\bigwedge \Phi)_\Om \subseteq \ph_\Om$.
  By Proposition \ref{P:Dynk-sp-certainty}, we have $\oprho (\bigwedge
  \Phi)_\Om = \oprho \bigcap_{\th \in \Phi} \th_\Om = 1$. Proposition 
  \ref{P:Dynk-complete-reverse} then implies $\oprho \ph_\Om = 1$. Therefore,
  $\ph \in \Th \sD$ and $\Th \sD$ is a deductive theory. Finally, $\oprho
  \bot_\Om = \oprho \emp = 0$, so $\bot \notin \Th \sD$ and $\Th \sD$ is
  consistent.
\end{proof}

\begin{lemma}\label{L:entire-from-Dynkin}
  Let $\sD$ and $\Th \sD$ be as in Lemma \ref{L:ThD-is-theory}. Let $T_0
  \subseteq \Th \sD$ be a deductive theory. Define $P \subseteq \cF^\IS$ so that
  $(X, \ph, p) \in P$ if and only if there exists $\psi \in \cF$ such that $T(X)
  = T_0 + \psi$ and
  \begin{equation}\label{cond-prob-Dynk}
    \frac{\oprho \ph_\Om \cap \psi_\Om}{\oprho \psi_\Om} = p.
  \end{equation}
  Then $P$ is entire.
\end{lemma}

\begin{proof}
  Note that if $A \in \De$ and $\oprho A \tri B = 0$, then $B \in \De$ and
  $\oprho B = \oprho A$. By adapting the proofs of Proposition
  \ref{P:model-func}, Corollary \ref{C:model-func}, and Lemma
  \ref{L:model-entire}, we obtain the following results. If $T_0 + \psi = T_0 +
  \psi'$, then $\oprho \psi_\Om \tri \psi'_\Om = 0$. In particular, if $(X, \ph,
  p) \in P$ and $T(X) = T_0 + \psi$, then \eqref{cond-prob-Dynk} holds. As a
  consequence, if $(X, \ph, p) \in P$ and $(X, \ph, p') \in P$, then $p = p'$.
  Also, if $T(X) = T_0 + \psi$, then $(X, \ph, 1) \in P$ if and only if
  $\psi_\Om\in \De$ and $\oprho \psi_\Om \cap \ph_\Om^c = 0$.

  Using these results, we may adapt the first part of the proof of Theorem 
  \ref{T:model-ind-th} to show that $P$ is entire. Note that in the proof of the
  addition rule, we must use the fact that if $A \cap B \in \De$ and two of the
  sets $A \cup B$, $A$, and $B$ are in $\De$, then all three sets are in $\De$
  and $\oprho A \cup B = \oprho A + \oprho B - \oprho A \cap B$.
\end{proof}

\subsection{Entirety is not enough}

The examples in this subsection illustrate the insufficiency of entire sets as a
basis for inductive inference. 

In Example \ref{Expl:MathSE} below, we use Dynkin spaces to construct a family
of entire, strongly connected sets, indexed by $q \in (0, 1)$. Every member of
this family is an example showing that probabilities of conjunctions need not be
defined. That is, if $P$ is one of the entire sets constructed in Examples
\ref{Expl:MathSE}, then $P(\bfr_1)$ exists, $P(\bfr_2)$ exists, but $P(\bfr_1
\wedge \bfr_2)$ does not exist.

We follow up in Example \ref{Expl:MathSEexpl-2} by considering the case $q =
1/4$. In this case, we show that $P$ is inconsistent. That is, it cannot be
extended to an inductive theory. Since $P$ is entire, it exhibits no violations
whatsoever of rules (R1)--(R7). However, if we try to extend $P$ so that it is
closed under (R8), then we will inevitably create of violation of (R1)--(R7).

In Proposition \ref{P:MathSE}, we consider the case $q = 1/2$. As mentioned
above, $P(\bfr_1 \wedge \bfr_2)$ does not exist. Proposition \ref{P:MathSE} is
concerned with what happens when we try to assign a value, $q'$, to this
expression. The result is that $q' = 1/4$ is the unique value that we may choose
in order to avoid violating rules (R1)--(R7). As such, this provides an example
of the necessary use of rule (R8) to infer a probability.

\begin{expl}\label{Expl:MathSE}
  In this example, we construct an entire, strongly connected set $P$ with root
  $T_0$ such that the domain of $P (\; \cdot \mid T_0)$ is not closed under
  conjunctions.

  For $n \in \bN_0$ and $k \ge 1$, define
  \[
    d_k(n) = \flr{2^{-k + 1} n} - 2 \flr{2^{-k} n},
  \]
    \symindex{$d_k(n)$}%
  where $\flr{x}$ is the greatest integer less than or equal to $x$. Then $d_k
  (n)$ denotes the $k$-th binary digit of $n$, counting digits from the right.

  Let $PV = \{\bfr_1, \bfr_2\}$. Let $\Om = \B^{PV}$ be the set of all strict
  models, so that $\Om = \{\om_n \mid 0 \le n \le 3\}$, where $\om_n \bfr_k =
  d_k(n)$. Note that these four strict models correspond to the usual rows of a
  truth table with two propositional variables.

  Let
  \begin{align*}
    A_1 &= (\bfr_1)_\Om = \{\om_1, \om_3\}, \text{ and}\\
    A_2 &= (\bfr_2)_\Om = \{\om_2, \om_3\}.
  \end{align*}
  Note that
  \[
    (\bfr_1 \tot \bfr_2) = (A_1 \tri A_2)^c = \{\om_0, \om_3\}.
  \]
  Let $\Ga = \{A_1, A_2, (A_1 \tri A_2)^c\}$ and
  \[
    \De = \{\emp, \Om\} \cup \Ga \cup \{A^c \mid A \in \Ga\}.
  \]
  Then $\De \subseteq \cB^{PV}$ is a Dynkin system on $\Om$. Fix $q \in (0, 1)$
  and, for $A \in \Ga$, define $\oprho^q A = q$ and $\oprho^q A^c = 1 - q$.
  Together with $\oprho^q \emp = 0$ and $\oprho^q \Om = 1$, this makes $\sD = 
  (\Om, \De, \oprho)$ a complete Dynkin space.

  Let $T_0 = \Taut$ and let $P$ be the entire set defined in Lemma
  \ref{L:entire-from-Dynkin}. Note that
  \begin{align*}
    P(\bfr_1 \mid T_0) &= q,\\
    P(\bfr_2 \mid T_0) &= q,\\
    P(\bfr_1 \tot \bfr_2 \mid T_0) &= q.
  \end{align*}
  On the other hand, $(\bfr_1 \wedge \bfr_2)_\Om = \{\om_3\} \notin \De$ so that
  $P(\bfr_1 \wedge \bfr_2 \mid T_0)$ is undefined. Hence, $P$ is an entire set
  such that the domain of $P(\; \cdot \mid T_0)$ is not closed under
  conjunctions.
\end{expl}

\begin{lemma}\label{L:MathSE}
  Let $P \subseteq \cF^\IS$ be entire. Let $X \in \ante P$ and let $\cG
  \subseteq \cF$ denote the domain of $P(\; \cdot \mid X)$. Let $\bfr_1, \bfr_2
  \in PV$ and assume $\bfr_1, \bfr_2, \bfr_1 \tot \bfr_2 \in \cG$. Let
  \begin{align*}
    \ph_0 &=          \neg  \bfr_1 \wedge          \neg  \bfr_2,\\
    \ph_1 &= \phantom{\neg} \bfr_1 \wedge          \neg  \bfr_2,\\
    \ph_2 &=          \neg  \bfr_1 \wedge \phantom{\neg} \bfr_2,\\
    \ph_3 &= \phantom{\neg} \bfr_1 \wedge \phantom{\neg} \bfr_2.
  \end{align*}
  If $\ph_3 \in \cG$, then $\ph_j \in \cG$ for all $j$.
\end{lemma}

\begin{proof}
  Note that $\ph_3 \vdash \bfr_1, \bfr_2, \bfr_1 \tot \bfr_2$. Also, $\bfr_1
  \wedge \neg \ph_3 \equiv \ph_1$, $\bfr_2 \wedge \neg \ph_3 \equiv \ph_2$, and
  $(\bfr_1 \tot \bfr_2) \wedge \neg \ph_3 \equiv \ph_0$. Hence, the result
  follows from Proposition \ref{P:rel-neg} and the rule of logical equivalence.
\end{proof}

\begin{expl}\label{Expl:MathSEexpl-2}
  Let $P$ be as in Example \ref{Expl:MathSE} with $q = 1/4$. Then $P$ is
  entire but not consistent. More specifically, $P$ cannot be extended to a
  deductive theory. To see this, suppose $P'$ is an inductive theory with $P
  \subseteq P'$. Let $P'_0 = P' \dhl_{T_0}$ so that $P \subseteq P'_0$ and
  $P'_0$ is a pre-theory. Since $P_0'$ is a pre-theory, it is semi-closed and
  therefore has a completion. Let $\ol P_0$ be a completion of $P_0'$. Then $\ol
  P_0$ is also a completion of $P$. Let $\cG \subseteq \cF$ be the domain of
  $\ol P_0(\; \cdot \mid T_0)$. Definition \ref{D:complete} implies that $\cG$
  is closed under conjunctions. Thus, $\bfr_1 \wedge \bfr_2 \in \cG$, so that by
  Lemma \ref{L:MathSE}, we have $\bfr_1 \wedge \neg \bfr_2, \neg \bfr_1 \wedge
  \bfr_2, \neg \bfr_1 \wedge \neg \bfr_2 \in \cG$. From Lemma \ref{L:fin-add}
  and Proposition \ref{P:rel-neg}, it follows that
  \begin{align*}
    1 &= \ol P_0(\bfr_1 \mid T_0) + \ol P_0(\neg \bfr_1 \wedge \bfr_2 \mid T_0)
      + \ol P_0(\neg \bfr_1 \wedge \neg \bfr_2 \mid T_0)\\
    &\le \ol P_0(\bfr_1 \mid T_0) + \ol P_0(\bfr_2 \mid T_0)
      + \ol P_0(\bfr_1 \tot \bfr_2 \mid T_0)\\
    &= 3/4,
  \end{align*}
  a contradiction.

  Note that this contradiction does not depend on $\ol P_0$ being complete. It
  is enough to assume that $\ol P_0$ is an entire extension of $P$ such that
  $\ol P_0(\bfr_1 \wedge \bfr_2 \mid T_0)$ exists. Consequently, there is no
  value of $q'$ such that $P \cup \{(T_0, \bfr_1 \wedge \bfr_2, q')\}$ has an
  entire extension.
\end{expl}

\begin{prop}\label{P:MathSE}
  Let $P$ be as in Example \ref{Expl:MathSE} with $q = 1/2$. Let $q' \in [0, 1]$
  and $Q = P \cup \{(T_0, \bfr_1 \wedge \bfr_2, q')\}$. Then $Q$ has an entire
  extension if and only if $q' = 1/4$.
\end{prop}

\begin{proof}
  Suppose $P'$ is an entire set with $Q \subseteq P'$. Let $\cG$ denote the
  domain of $P'(\; \cdot \mid T_0)$. Let $\ph_j$ be defined as in Lemma
  \ref{L:MathSE}, so that $\ph_j \in \cG$ for all $j$. Let $p_j = P'(\ph_j \mid
  T_0)$ and note that $p_3 = q'$. By the addition rule,
  \[
    p_0 + p_3 = p_1 + p_3 = p_2 + p_3 = 1/2,
  \]
  which implies $p_0 + p_1 + p_2 + 3p_3 = 3/2$. But $\sum_j p_j = 1$, so
  $1 + 2p_3 = 3/2$, giving $q' = p_3 = 1/4$.

  Now assume $q' = 1/4$. Let $\sP = (\Om, \Si, \bbP)$, where $\Om = \{\om_n \mid
  0 \le n \le 3\}$ as in Example \ref{Expl:MathSE}, $\Si = \fP \Om$, and $\bbP$
  satisfies $\bbP \om_n = 1/4$ for all $n$. Then $\bTh \sP$ is an entire
  extension of $Q$.
\end{proof}

We conclude this section with an example of an inductive theory that fails to be
complete by violating Definition \ref{D:complete}(ii). That is, we construct an
inductive theory $P$ where $X \in \ante P$ and $X \cup \{\ph\} \in \ante P$ even
though $P(\ph \mid X)$ does not exist.

\begin{expl}\label{Expl:incompl-ind-th}
  Let $PV = \{\bfr_1, \bfr_2, \bfr_3\}$. Recall the notation $d_k(n)$ defined in
  Example \ref{Expl:MathSE}. Let $\Om$ be the set of all strict models, so that
  $\Om = \{\om_n: 0 \le n \le 7\}$, where $\om_n \bfr_k = d_k(n)$. Let $\Si =
  \fP \Om$. Fix $q \in (0, 1)$ and let $\bbP^q$ be the probability measure on
  $(\Om, \Si)$ determined by
  \begin{align*}
    \bbP^q \om_4 &= \bbP^q \om_5 = 0,\\
    \bbP^q \om_6 &= \bbP^q \om_7 = q/2,\\
    \bbP^q \om_n &= (1 - q)/4 \text{ for $0 \le n \le 3$}.
  \end{align*}
  Define the model $\sP^q = (\Om, \Si, \bbP^q)$ and let $P^q = \bTh \sP^q$. Then
  $P^q$ is a complete inductive theory with root $\Taut$. Note that
  \begin{align*}
    (\bfr_1)_\Om &= \{\om_1, \om_3, \om_5, \om_7\},\\
    (\bfr_2)_\Om &= \{\om_2, \om_3, \om_6, \om_7\},\\
    (\bfr_3)_\Om &= \{\om_4, \om_5, \om_6, \om_7\}.
  \end{align*}
  Hence, $P^q(\bfr_1) = 1/2$, $P^q(\bfr_2 \mid \bfr_3) = 1$, and $P^q(\bfr_3) =
  q$.

  Let $Q$ be defined by $Q(\bfr_1) = 1/2$ and $Q(\bfr_2 \mid \bfr_3) = 1$. Then
  $Q$ is strongly connected with root $\Taut$. Also, $Q \subseteq P^q$, so $Q$
  is consistent. Let $P = \bfP_Q$ be the inductive theory generated by $Q$. Then
  $P \subseteq P^q$ for all $q \in (0, 1)$. Hence, $P(\bfr_3)$ is undefined, and
  $P$ violates Definition \ref{D:complete}(ii).
\end{expl}

\section{Counterexamples and resolutions II}\label{S:Examples2}

This section contains examples related to satisfiability and the consequence
relation. In Section \ref{S:D-conseq-need2}, we construct an example where $\sP
\vDash \cC$ implies $\sP \vDash (X, \ph, p)$ for all models $\sP$, but $\cC
\nvdash (X, \ph, p)$. The failure occurs because $X$ is so large that it is not
countably axiomatizable over $[\Taut, T_\cC]$. As such, this example
demonstrates the need for Definition \ref{D:consequence-IC}(ii). It also serves
as an example of a collection of inductive theories whose intersection is not an
inductive theory, as well as an example of an indeterminate inductive condition.

In Section \ref{S:Karp-expls}, we address the issue of completeness and strict
satisfiability. In \cite{Karp1964}, Karp showed that completeness fails when we
try to use strict satisfiability as a basis for our semantics. She presented
therein two examples. In Examples \ref{Expl:Karp413} and \ref{Expl:Karp412}, we
revisit Karp's examples, demonstrate their resolution in the current context,
and show how they connect to classical probability theory.

\subsection{An unknown false statement}\label{S:D-conseq-need2}

Let $PV = \{\bfr^t \mid t \in [0, 1]\}$. The idea behind this example is the
following. We wish to build an inductive theory based on the following
assumptions. With probability $1/2$, every propositional variable is true.
Otherwise, there is exactly one $r \in PV$ that is false. But in this latter
case, we do not want to make any assumptions about which $r$ is false.

Let $T_0 = \Taut$. For $t \in [0, 1]$, let
\[
  Q(t) = \{(T_0, \bfr^t, 1/2)\} \cup \{(T_0, \bfr^s, 1) \mid s \ne t\},
\]

\begin{lemma}\label{L:D-conseq-need2}
  For each $t \in [0, 1]$, the set $Q(t)$ is consistent.
\end{lemma}

\begin{proof}
  The set $Q(t)$ is clearly strongly connected with root $T_0$. By Theorem 
  \ref{T:ind-satis-cons}, it suffices to show that $Q(t)$ is satisfiable.

  Let $\Om = \B^{PV}$ be the set of all strict models and $\Si = \cB^{PV} =
  \{\ph_\Om \mid \ph \in \cF\}$. For $A \subseteq [0, 1]$, define the strict
  model $\om_A$ by $\om_A \bfr^t = 0$ if and only if $t \in A$. Let $\om_t =
  \om_{\{t\}}$. Let $\de(w)$ be the point mass measure concentrated on $\om$.
  Define $\bbP_t = (\de(\om_t) + \de(\om_\emp))/2$ and $\sP_t = (\Om, \Si,
  \bbP_t)$. Then $\om_\emp \in r_\Om$ for all $r \in PV$, and $\om_t \in r_\Om$
  if and only if $r = \bfr^s$ for some $s \ne t$. Hence, $\olbbP_t \bfr^t = 1/2$
  and $\olbbP_t \bfr^s = 1$ for $s \ne t$. Therefore, $\sP_t \vDash Q(t)$ and
  $Q(t)$ is satisfiable.
\end{proof}

Define the consistent inductive condition $\cC = \{\bfP_{Q(t)} \mid t \in [0,
1]\}$, so that $\bfP_\cC$ is the inductive theory we were aiming to build.

\begin{prop}\label{P:D-conseq-need2}
  With notation as above, $\bfP_\cC \subset \bigcap \cC$. In particular,
  $\bfP_\cC \notin \cC$. That is, the condition $\cC$ is indeterminate.
\end{prop}

\begin{proof}
  By Theorem \ref{T:theory-gen-IC-defn}, we have $\bfP_\cC \subseteq \bigcap
  \cC$. Assume $\bfP_\cC = \bigcap \cC$.

  Note that for any inductive theory $P$, by the rule of logical implication, $X
  \in \ante P$ if and only if $(X, \top, 1) \in P$. Hence, $\ante \bfP_\cC =
  \bigcap \{\ante \bfP_{Q(t)} \mid t \in [0, 1]\}$.

  For $t \in [0, 1]$, let $S_t = \{\bfr^s \mid s \ne t\}$. Then $\bfP_{Q(t)}(r
  \mid T_0) = 1$ for all $r \in S_t$ and $\bfP_{Q(t)}(\bfr^t \mid T_0) = 1/2$.
  By Lemma \ref{L:cond-exist}, we have $\{\bfr^t\} \in \ante \bfP_{Q(t)}$. By
  the rule of deductive extension, $S_t \cup \{\bfr^t\} = PV \in \ante
  \bfP_{Q(t)}$. Since $t$ was arbitrary, this gives $PV \in \ante \bfP_\cC$.
  Hence, we may write $PV \equiv T + \psi$, where $T \in [\Taut, T_\cC]$ and
  $\psi \in \cF$.

  Let $\Om$, $\Si$, $\om_t$, $\om_\emp$, and $\sP_t$ be as in the proof of Lemma
  \ref{L:D-conseq-need2}. Let $f$ be the Boolean function that $\psi$ represent.
  Proposition \ref{P:Boolean-func-Pi-ary} implies that $f$ is $\Pi$-ary, where
  $\Pi = PV \cap \Sf \psi$ is the countable set of propositional variables that
  appear in $\psi$. Hence, we may choose a measurable $h: \B^\Pi \to \B$ such
  that, for all $\om \in \Om$, we have $\om \psi = f \om = h(\om|_\Pi)$.

  Enumerate $\Pi$ as $\Pi = \{\bfr^{t_1}, \bfr^{t_2}, \ldots\}$. Choose $t_0
  \notin \{t_1, t_2, \ldots\}$ and let $\bbP = \de(\om_ {t_0})$ and $\sP = (\Om,
  \Si, \bbP)$.

  Assume for the moment that $\sP \vDash T_\cC + \psi$. Since $T + \psi
  \subseteq T_\cC + \psi$ and $PV \equiv T + \psi$, this gives $\sP \vDash PV$.
  In particular, $\sP \vDash \bfr^{t_0}$, so that $\bfr^{t_0}_\Om \in \ol \Si$
  and $\olbbP \bfr^{t_0}_\Om = 1$. By the definition of $\bbP$, this implies
  $\om_{t_0} \in \bfr^{t_0}_\Om$, so that $\om_{t_0} \bfr^{t_0} = 1$, a
  contradiction. Therefore, $\bfP_\cC \subset \bigcap \cC$, and we are done.

  It suffices, then, to show that $\sP \vDash T_\cC + \psi$. We first show that
  $\sP \vDash T_\cC$. Let $\th \in T_\cC$ be arbitrary. By Proposition
  \ref{P:ded-th-ind-cond}, we have $\th \in T(\bfP_{Q(t)})$ for all $t \in [0,
  1]$. Hence, $\olbbP_t \th_\Om = 1$ for all $t \in [0, 1]$. This implies that
  $\om_\emp \in \th_\Om$ and $\om_t \in \th_\Om$ for all $t$. In other words,
  $\om_\emp \th = 1$ and $\om_t \th = 1$ for all $t$.

  Let $\Pi' = PV \cap \Sf \th = \{\bfr^{s_1}, \bfr^{s_2}, \ldots\}$. As above,
  we may choose measurable $g: \B^{\Pi'} \to \B$ such that, for all $\om \in
  \Om$, we have $\om \th = g(\om|_{\Pi'})$. Let $\th^0 = \bigwedge_{j = 1}^\infty
  \bfr^{s_j}$. For $n \in \bN$, let $\th^n = \neg \bfr^{s_n} \wedge
  (\bigwedge_{j \ne n} \bfr^{s_j})$. Finally, define $\th' = \bigvee_{n =
  0}^\infty \th^n$. If $t_0 \notin \{s_1, s_2, \ldots\}$, then $\om_{t_0} \bfr^
  {s_j} = 1$ for all $j \in \bN$, which implies $\om_{t_0} \th^0 = 1$. If $t_0 =
  s_n$ for some $n \in \bN$, then $\om_{t_0} \th^n = 1$. In either case, we have
  $\om_{t_0} \th' = 1$, so that $\om_{t_0} \in \th'_\Om$, and therefore, $\olbbP
  \th'_\Om = 1$.

  Now suppose $\om \in \th'_\Om$. Choose $n \in \bN_0$ such that $\om \in
  \th^n_\Om$. If $n = 0$, then $\om|_{\Pi'} = \om_\emp|_{\Pi'}$, so that $\om
  \th = \om_\emp \th = 1$ and $\om \in \th_\Om$. If $n \in \bN$, then
  ${\om|_{\Pi'}} = {\om_{s_n}|_{\Pi'}}$, so that $\om \th = \om_{s_n} \th = 1$
  and again $\om \in \th_\Om$. This shows that $\th'_\Om \subseteq \th_\Om$.
  Therefore, $\olbbP \th_\Om = 1$, so that $\sP \vDash \th$. Since $\th$ was
  arbitrary, we have $\sP \vDash T_\cC$.

  Lastly, we show that $\sP \vDash \psi$. Since $(\Om, \Si, \de(\om_\emp))
  \vDash PV$ and $PV \equiv T + \psi$, we have $\om_\emp \in \psi_\Om$, so that
  $\om_\emp \psi = 1$. Since $t_0 \notin \{t_1, t_2, \ldots\}$, we also have
  that ${\om_\emp|_\Pi} = {\om_{t_0}|_\Pi}$. Hence, $\om_{t_0} \psi = \om_\emp
  \psi = 1$, which gives $\om_{t_0} \in \psi_\Om$ and therefore $\sP \vDash
  \psi$.
\end{proof}

\begin{rmk}\label{R:intersect-th-fail}
  Since $\bfP_\cC$ is the largest inductive theory contained in $\bigcap \cC$,
  it follows that $\bigcap \cC$ is not an inductive theory. By Theorem
  \ref{T:intersect-cl}, the set $\bigcap \cC$ is closed. Hence, it must not be
  connected. In other words, the condition $\cC$ is an example of a collection
  of connected sets whose intersection is not connected.
\end{rmk}

\begin{rmk}\label{R:D-conseq-need2}
  Note that if $\sP \vDash \cC$, then $\sP \vDash \bfP_{Q(t)}$ for some $t$. But
  we also have $(PV, \top, 1) \in \bfP_{Q(t)}$ for all $t$. Hence, $\sP \vDash
  \cC$ implies $\sP \vDash (PV, \top, 1)$. On the other hand, the proof of
  Proposition \ref{P:D-conseq-need2} shows that $PV \notin \ante \bfP_\cC$.
  Therefore, this example illustrates the necessity of Definition
  \ref{D:consequence-IC}(ii).
\end{rmk}

\subsection{Karp's counterexamples}\label{S:Karp-expls}

\begin{expl}\label{Expl:Karp413}
  It is well-known that $\si$-compactness fails for strict satisfiability. That
  is, there exists $X \subseteq \cF$ such that every countable subset is
  strictly satisfiable, but $X$ itself is not strictly satisfiable. Karp gives
  an example of such an $X$ in \cite[Example 4.1.3]{Karp1964}.

  However, since strict satisfiability implies satisfiability (Proposition
  \ref{P:sig-pre-cpct}(i)) and satisfiability is $\si$-compact (Theorem
  \ref{T:compactness}), we know that any such $X$ must be satisfiable. We
  present Karp's example below, and then show that it is satisfied by one of the
  most common models in probability theory.

  Let $PV = \{\bfr_n^k \mid (n, k) \in \bN \times \{0, 1\}\}$. Let $\ze =
  \bigwedge_n (\bfr_n^0 \vee \bfr_n^1)$, and for each $f: \bN \to \{0, 1\}$, let
  $\psi_f = \neg \bigwedge_n \bfr_n^{f(n)}$. Let $X = \{\ze\} \cup \{\psi_f \mid
  f \in \{0, 1\}^\bN\}$.

  Suppose $X_0\subseteq X$ is countable. Since $\{0, 1\}^\bN$ is uncountable,
  there exists $g: \bN \to \{0, 1\}$ such that $\psi_g \notin X_0$. Define $\om$
  by $\om \bfr_n^k = 1$ if and only if $g(n) = k$. Then $\om \tDash \ze$. Note
  that $\om \tDash \psi_f$ if and only if there exists $n$ such that $\om
  \bfr_n^{f(n)} = 0$. Given $f \ne g$, we may choose $n$ such that $f(n) \ne g
  (n)$, and so for this value of $n$, we have $\om \bfr_n^{f(n)} = 0$. Hence,
  $\om \tDash \psi_f$ for all $f \ne g$, and therefore $\om \tDash X_0$. It
  follows that every countable subset of $X$ is strictly satisfiable.

  Now suppose $\om \tDash X$. Since $\om \tDash \ze$, we may choose, for each $n
  \in \bN$, a value $f(n) \in \{0, 1\}$ such that $\om \tDash \bfr_n^{f(n)}$.
  But then $\om \tDash \bigwedge_n \bfr_n^{f(n)}$, meaning $\om \ntDash \psi_f$,
  a contradiction. Therefore, $X$ is not strictly satisfiable.

  As mentioned in the beginning of this example, however, we know that there
  exists a model $\sP = (\Om, \Si, \bbP)$ such that $\sP \vDash X$. In this
  case, we can construct such a model using $\Om = \B^{PV}$ and taking $\Si =
  \cB^{PV} = \{\ph_\Om \mid \ph \in \cF\}$, thereby assigning a probability to
  every formula in $\cF$. The model is a natural one that is ubiquitous in
  probability theory. Namely, it is the one that models an i.i.d.~sequence of
  coin flips.

  Let $(S, \Ga, \nu)$ be a probability space on which we have constructed an
  i.i.d.~sequence $\ang{X_n \mid n \in \bN}$ of $\{0, 1\}$-valued random
  variables with $\nu \{X_n = 1\} = 1/2$. Define $G: PV \to \Ga$ by $G \bfr_n^k
  = \{X_n = k\}$. Let $\sP = (\Om, \Si, \bbP)$ be the model constructed in the
  proof of Theorem \ref{T:prob-sp-model-iso}. Note that
  \begin{align*}
    G \ze &= \{X_n \in \{0, 1\} \text{ for all $n$}\},\\
    G \psi_f &= \{X_n = f(n) \text{ for all $n$}\}^c.
  \end{align*}
  Hence, $\bbP \ze_\Om = \nu G \ze = 1$ and $\bbP (\psi_f)_\Om = \nu G \psi_f =
  1$ for all $f$, showing that $\sP \vDash X$.

  We will investigate this example further in Section \ref{S:Karp413}, after
  covering the topic of independence.
\end{expl}

\begin{expl}\label{Expl:Karp412}
  We present here another example of Karp's (see \cite[Example 4.1.2]
  {Karp1964}). Again we see an $X$ that demonstrates the failure of
  $\si$-compactness for strict satisfiability. And again, we know that $X$ is
  satisfiable. We could use the construction in the proof of Theorem 
  \ref{T:compactness} to build a model that satisfies $X$. In this example,
  though, we do not do that. Rather, we show that any such model has a certain
  property. Namely, in any such model, $\sP = (\Om, \Si, \bbP)$, there will be
  formulas $\ph \in \cF$ such that $\ph_\Om \notin \Si$.

  Let $I$ be an uncountable set and let $PV = \{\bfr^t_n \mid t \in I, n \in
  \bN\}$.
  For each $t \in I$, let $\ze^t = \bigvee_n \bfr^t_n$. For each $s, t \in I$ and
  $n \in \bN$, let $\psi^{s, t}_n = \neg (\bfr^s_n \wedge \bfr^t_n)$. Then define
  \[
   X = \{\ze^t \mid t \in I\}
    \cup \{\psi^{s, t}_n \mid s, t \in I, s \ne t, n \in \bN\}.
  \]
  Let $X_0 \subseteq X$ be countable. Then there is a countable set $S \subset
  I$ such that
  \[
    X_0 \subseteq \{\ze^t \mid t \in S\}
      \cup \{\psi^{s, t}_n \mid s, t \in S, s \ne t, n \in \bN\}.
  \]
  Let $t \mapsto n(t)$ be an injection from $S$ to $\bN$. Define a model $\om$
  by $\om \bfr^t_n = 1$ if and only if $t \in S$ and $n = n(t)$. Then, for each
  $t \in S$, we have $\om \tDash \bfr^t_{n(t)}$, and therefore $\om \tDash
  \ze^t$. Also, if $s, t \in S$, $s \ne t$, and $n \in \bN$, then either $n \ne
  n(s)$ or $n \ne n(t)$, implying that at least one of $\om \bfr^s_n$ and $\om
  \bfr^t_n$ is 0. Thus, $\om \tDash \psi^{s, t}_n$. Altogether, this shows that
  $\om \tDash X_0$, so that every countable subset of $X$ is strictly
  satisfiable.

  Now suppose $\om \tDash X$. For each $t \in I$, we have $\om \tDash \ze^t$. By
  the definition of $\ze^t$, we may choose $n(t) \in \bN$ such that $\om \tDash
  \bfr^t_{n(t)}$. But $I$ is uncountable and $\bN$ is countable, so there exists
  distinct $s, t \in I$ such that $n(s) = n(t)$. It follows that $\om \psi^{s,
  t}_{n(s)} = 0$, and so $\om \ntDash \psi^{s, t}_{n(s)}$, a contradiction.
  Therefore, $X$ is not strictly satisfiable.

  Again, by Proposition \ref{P:sig-pre-cpct}(i) and Theorem \ref{T:compactness},
  we know that $X$ is satisfiable, meaning there is a model $\sP = (\Om, \Si,
  \bbP)$ such that $\sP \vDash X$. In this case, however, we cannot construct
  such a model using $\Si = \cB^{PV} = \{\ph_\Om \mid \ph \in \cF\}$. In any
  model that satisfies $X$, there will exist $\ph \in \cF$ such that $\ph_\Om
  \notin \Si$. In other words, there will be formulas that are not assigned a
  probability.

  To see that this is the case, let $\sP = (\Om, \Si, \bbP)$ be a model. Assume
  that $\sP \vDash X$ and that $r_\Om \in \Si$ for all $r \in PV$. For $n, k \in
  \bN$, define
  \[
    S(n, k) = \{t \in I \mid \bbP (\bfr^t_n)_\Om \ge k^{-1}\}.
  \]
  Suppose $s, t \in I$ and $s \ne t$. Then $\psi^{s, t}_n = \neg (\bfr^s_n
  \wedge \bfr^t_n) \in X$. Thus,
  \[
    \bbP (\psi^{s, t}_n)_\Om = \bbP ((\bfr^s_n)_\Om \cap (\bfr^t_n)_\Om)^c = 1.
  \]
  In other words, $(\bfr^s_n)_\Om$ and $(\bfr^t_n)_\Om$ are pairwise disjoint,
  up to a set of measure zero. It follows that $S(n, k)$ is a finite set with at
  most $k$ elements.

  Now fix $t \in I$. Since $\ze^t = \bigvee_n \bfr^t_n \in X$, we have $1 = \bbP
  \ze^t_\Om = \bbP \bigcup_n (\bfr^t_n)_\Om$. Thus, there exists $n \in \bN$ such
  that $\bbP (\bfr^t_n)_\Om > 0$, showing that $t \in S(n, k)$ for some $n, k \in
  \bN$. In other words, $I = \bigcup_{n, k} S(n, k)$, expressing $I$ as a
  countable union of finite sets, contradicting the fact that $I$ is
  uncountable.
\end{expl}

\begin{rmk}\label{R:Karp412}
  It follows from Example \ref{Expl:Karp412} that we cannot construct an
  $\bN$-valued stochastic process $\ang{Y(t) \mid t \in I}$ such that for all $s
  \ne t$, we have $Y(s) \ne Y(t)$ a.s. (Recall that a stochastic process is
  simply an indexed collection of random variables taking values in the same
  measurable space.) To see this, suppose we have such a process, built on a
  probability space $(S, \Ga, \nu)$. Define $G: PV \to \Ga$ by $G \bfr_n^t =
  \{Y(t) = n\}$, and let $\sP = (\Om, \Si, \bbP)$ be the model constructed in
  the proof of Theorem \ref{T:prob-sp-model-iso}. Then $\sP \vDash X$ and $\Si =
  \{\ph_\Om \mid \ph \in \cF\}$, a contradiction.
\end{rmk}

\section{Independence}\label{S:indep}

In this section, we introduce the concept of (inductive) independence. It is a
purely logical concept, defined solely in terms of formulas and inductive
theories, without reference to any model. Intuitively, $\ph$ and $\psi$ are
independent given $X$ if $P(\ph \mid X)$ is unchanged by adding $\psi$ to $X$.
More generally, a sequence of formulas is independent if, whenever we take two
disjoint subsequences and, from each subsequence, create a formula using
negation and conjunction, the two resulting formulas are independent.

To make this notion precise, we will introduce the concept of a dialog set,
which describes all the formulas that can be created from a given set of
formulas. Dialog sets are the syntactic analogues of $\si$-algebras.

After defining independence, we introduce the semantic concept of ``measure
independence,'' which is nothing more than the familiar notion of independence
in a measure space, defined by the property that the measure of an intersection
factors into a product. Using this, we give a characterization of independence
in terms of satisfiability and measure independence.

Finally, we give two examples in which we use independence to build inductive
conditions, and then look at the inductive theories generated by those
conditions.

\subsection{Dialog sets}\label{S:dialog}

\begin{defn}\label{D:dialog}
    \index{dialog set}%
  A set $D \subseteq \cF$ is a \textit{dialog set} if it satisfies the
  following:
  \begin{enumerate}[(i)]
    \item $\ph \in D$ implies $\neg \ph \in D$, \label{neg-cl}
    \item $\Phi \subseteq D$ countable implies $\bigwedge \Phi \in D$, and
          \label{conj-cl}
    \item $\ph \in D$ and $\ph \equiv \psi$ implies $\psi \in D$.
          \label{eq-cl}
  \end{enumerate}
\end{defn}

Note that the $\Phi$ in \ref{conj-cl} may be empty. Hence, $\top \in D$, and by
\ref{neg-cl} and \ref{eq-cl}, also $\bot \in D$. Note further that \ref{eq-cl}
implies $D$ is closed under $\to$, $\tot$, and $\bigvee$.

The intersection of any family of dialog sets is again a dialog set. Also, the
language $\cF$ itself is a dialog set, and is the largest dialog set. If $X
\subseteq \cF$, then the \textit{dialog set generated by $X$}, denoted by
$\de(X)$, is the smallest dialog set containing $X$.
  \symindex{$\de(X)$}%
It is equal to the intersection of all dialog sets containing $X$. The smallest
dialog set is
\[
  \de(\emp) = \Taut \cup \neg \Taut.
\]
Intuitively, $\de(X)$ is the set of all formulas that can be built out the
formulas in $X$ using negation, conjunction, and logical equivalence. Note that
$\de(PV) = \cF$.

\begin{prop}\label{P:dialog-sig-alg}
  Let $\Om$ be a set of strict models. If $D \subseteq \cF$ is a dialog set,
  then $D_\Om$ is a $\si$-algebra on $\Om$. More generally, if $X \subseteq
  \cF$, then $\si(X_\Om) = \de(X)_\Om$.
\end{prop}

\begin{proof}
  Let $D \subseteq \cF$ be a dialog set. Then $\Om = \top_\Om \in D_\Om$. Let
  $A \in D_\Om$. Choose $\ph \in D$ such that $A = \ph_\Om$. Then $\neg \ph \in
  D$, so that $A^c = \ph_\Om^c = (\neg \ph)_\Om \in D_\Om$. Finally, suppose $
  \{A_n\} \subseteq D_\Om$. Choose $\ph_n \in D$ such that $A_n = (\ph_n)_\Om$.
  Then $\bigcap A_n = (\bigwedge \ph_n)_\Om \in D_\Om$, and hence, $D_\Om$ is a
  $\si$-algebra.

  Now let $X \subseteq \cF$. Then $X_\Om \subseteq \de(X)_\Om$. By the above,
  $\de(X)_\Om$ is a $\si$-algebra. Hence, $\si(X_\Om) \subseteq \de(X)_\Om$. For
  the reverse inclusion, define $D = \{\ph \in \cF \mid \ph_\Om \in \si
  (X_\Om)\}$. The set $D$ clearly satisfies \ref{neg-cl} and \ref{conj-cl} of
  Definition \ref{D:dialog}. Remark \ref{R:impl-subset} shows that it also
  satisfies \ref{eq-cl}. Thus, $D$ is a dialog set. Since $X \subseteq D$, it
  follows that $\de(X) \subseteq D$ and therefore $\de(X)_\Om \subseteq \si
  (X_\Om)$.
\end{proof}

\subsection{Independence of two formulas}

Let $P$ be an inductive theory. For $X \subseteq \cF$, let $\dom P(\; \cdot
\mid X) \subseteq \cF$ denote the domain of $P(\; \cdot \mid X)$. If $X \notin
\ante P$, then $\dom P(\; \cdot \mid X) = \emp$.

Let $\ph, \psi \in \dom P(\; \cdot \mid X)$. We say that \emph{$\ph$ is
dependent on $\psi$ given $X$ (under $P$)} if $P(\ph \mid X, \psi)$ exists and
is not equal to $P(\ph \mid X)$.
  \index{dependent}%
We say that \emph{$\ph$ is independent of $\psi$ given $X$ (under $P$)} if
either $P(\psi \mid X) = 0$ or $P(\ph \mid X, \psi) = P(\ph \mid X)$.
  \index{independent}%

Note that if $P(\psi \mid X) > 0$ and $P(\ph \mid X, \psi)$ does not exist, then
$\ph$ is not dependent on $\psi$, and $\ph$ is also not independent of $\psi$.
If $P$ is complete, then this situation cannot arise. To see this, note that if
$P$ is complete and $\ph, \psi \in \dom P(\; \cdot \mid X)$, then $P(\ph \wedge
\psi \mid X)$ exists. Hence, by the multiplication rule, if $P(\psi \mid X) >
0$, then $P(\ph \mid X, \psi)$ exists.

\begin{lemma}\label{L:dep-undec}
  If $\ph$ is dependent on $\psi$ given $X$, then both $0 < P(\ph \mid X) < 1$
  and $0 < P(\psi \mid X) < 1$.
\end{lemma}

\begin{proof}
  Let $\ph$ be dependent on $\psi$ given $X$. By Lemma \ref{L:cond-exist}, we
  have $P(\psi \mid X) > 0$. Suppose $P(\psi \mid X) = 1$. Then the
  multiplication rule and Proposition \ref{P:certainty-closure} imply $P(\ph
  \mid X) = P(\ph \mid X, \psi)$, a contradiction.

  Now suppose $P(\ph \mid X) = 1$. As above, the multiplication rule and
  Proposition \ref{P:certainty-closure} imply $P(\ph \mid X, \psi) = 1$, a
  contradiction. Finally, suppose $P(\ph \mid X) = 0$. Then $P(\neg \ph \mid X)
  = 1$. Again, this gives $P(\neg \ph \mid X, \psi) = 1$, so that
  \eqref{prob-neg} implies $P(\ph \mid X, \psi) = 0$, a contradiction.
\end{proof}

\begin{lemma}\label{L:evidence}
  If $\ph$ is dependent on $\psi$ given $X$, then both $P(\psi \mid X, \ph)$ and
  $P(\ph \mid X, \neg \psi)$ exist.
\end{lemma}

\begin{proof}
  Suppose $\ph$ is dependent on $\psi$ given $X$. Then $P(\ph \mid X)$, $P(\psi
  \mid X)$, and $P(\ph \mid X, \psi)$ exist, and $P(\ph \mid X, \psi) \ne P(\ph
  \mid X)$. Since $P(\psi \mid X)$ and $P(\ph \mid X, \psi)$ exist, the
  multiplication rule implies $P(\ph \wedge \psi \mid X)$ exists. By Lemma 
  \ref{L:dep-undec}, we have $P(\ph \mid X) > 0$. Hence, another application of
  the multiplication rule implies that $P(\psi \mid X, \ph)$ exists. From
  Proposition \ref{P:rel-neg}, it follows that $P(\ph \wedge \neg \psi \mid X)$
  exists. Lemma \ref{L:dep-undec} implies $P(\neg \psi \mid X) > 0$. Therefore,
  a final application of the multiplication rule gives the existence of $P(\ph
  \mid X, \neg \psi)$.
\end{proof}

\begin{prop}
  Suppose $\ph$ is dependent on $\psi$ given $X$. Then $P(\ph \mid X, \psi) > P
  (\ph \mid X)$ if and only if $P(\ph \mid X, \neg \psi) < P(\ph \mid X)$.
\end{prop}

\begin{proof}
  Suppose $P(\ph \mid X, \psi) > P(\ph \mid X)$. Lemma \ref{L:dep-undec} implies
  $P(\psi \mid X) > 0$, and Lemma \ref{L:evidence} implies $P(\psi \mid X, \ph)$
  exists. Thus, by \eqref{Bayes},
  \[
    P(\ph \mid X) P(\psi \mid X, \ph) > P(\psi \mid X) P(\ph \mid X),
  \]
  which implies $P(\psi \mid X, \ph) > P(\psi \mid X)$. By \eqref{prob-neg},
  this implies $P(\neg \psi \mid X, \ph) < P(\neg \psi \mid X)$. On the other
  hand, Lemma \ref{L:evidence} implies $P(\ph \mid X, \neg \psi)$ exists, so
  that \eqref{Bayes} implies
  \[
    P(\neg \psi \mid X) P(\ph \mid X, \neg \psi)
      = P(\ph \mid X) P(\neg \psi \mid X, \ph).
  \]
  As before, this implies $P(\ph \mid X, \neg \psi) < P(\ph \mid X)$. The
  analogous argument proves the converse.
\end{proof}

\begin{prop}\label{P:depen-sym}
  Let $\ph, \psi \in \cF$. Then $\ph$ is dependent on $\psi$ given $X$ if and
  only if $\psi$ is dependent on $\ph$ given $X$.
\end{prop}

\begin{proof}
  Let $\ph$ be dependent on $\psi$ given $X$. Then $P(\ph \mid X)$ and $P(\psi
  \mid X)$ exist. By Lemma \ref{L:evidence}, we have that $P(\psi \mid X, \ph)$
  exists. Suppose $P(\psi \mid X, \ph) = P(\psi \mid X)$. Lemma
  \ref{L:dep-undec} implies $P(\psi \mid X) > 0$. Thus, by \eqref{Bayes}, we
  have $P(\ph \mid X) = P(\ph \mid X, \psi)$, a contradiction. Therefore, $\psi$
  is dependent on $\ph$ given $X$. The converse follows by reversing the roles
  of $\ph$ and $\psi$.
\end{proof}

\begin{thm}\label{T:indep-prod}
  Let $\ph, \psi \in \dom P(\; \cdot \mid X)$. Then $\ph$ is independent of
  $\psi$ given $X$ if and only if
  \begin{equation}\label{indep-prod}
    P(\ph \wedge \psi \mid X) = P(\ph \mid X) P(\psi \mid X).
  \end{equation}
\end{thm}

\begin{proof}
  Suppose $P(\psi \mid X) = 0$. Then $\ph$ is independent of $\psi$ given $X$.
  Also, Proposition \ref{P:certainty-closure} implies $P(\ph \wedge \psi \mid
  X) = 0$, so that \eqref{indep-prod} holds.

  Now suppose $P(\psi \mid X) > 0$. Then $\ph$ is independent of $\psi$ given
  $X$ if and only if $P(\ph \mid X, \psi) = P(\ph \mid X)$. On the other hand,
  by the multiplication rule, \eqref{indep-prod} holds if and only if $P(\ph
  \mid X, \psi) = P(\ph \mid X)$.
\end{proof}

\begin{cor}\label{C:indep-prod}
  Let $\ph, \psi \in \dom P(\; \cdot \mid X)$. If $\ph$ is independent of $\psi$
  given $X$, then $\ph$ is independent of $\neg \psi$ given $X$.
\end{cor}

\begin{proof}
  Suppose $\ph$ is independent of $\psi$ given $X$. By Proposition 
  \ref{P:rel-neg} and Theorem \ref{T:indep-prod}, we have
  \begin{align*}
    P(\ph \wedge \neg \psi \mid X)
      &= P(\ph \mid X) - P(\ph \wedge \psi \mid X)\\
    &= P(\ph \mid X) - P(\ph \mid X)P(\psi \mid X)\\
    &= P(\ph \mid X)(1 - P(\psi \mid X))\\
    &= P(\ph \mid X)P(\neg \psi \mid X).
  \end{align*}
  Hence, Theorem \ref{T:indep-prod} implies $\ph$ is independent of $\neg \psi$
  given $X$.
\end{proof}

By Proposition \ref{P:depen-sym} and Theorem \ref{T:indep-prod}, we can alter
our terminology to say that $\ph$ and $\psi$ are dependent or independent, given
$X$. Note that by Theorem \ref{T:indep-prod} and Proposition 
\ref{P:certainty-closure}, if $\ph, \psi \in \dom P (\; \cdot \mid X)$ and
either $P(\ph \mid X) \in \{0, 1\}$ or $P (\psi \mid X) \in \{0, 1\}$, then
$\ph$ and $\psi$ are independent given $X$.

\subsection{Independence of a sequence of formulas}

Let $I$ be a set with $|I| \ge 2$ and let $\ang{\ph_i \mid i \in I}$ be an
indexed collection of formulas in $\dom P(\; \cdot \mid X)$. Such a collection
is \textit{independent given $X$ (under $P$)} if $\ph$ and $\psi$ are
independent given $X$ whenever $\ph \in \de(\{\ph_i \mid i \in I_1\})$ and $\psi
\in \de(\{\ph_i \mid i \in I_2\})$, where $I_1$ and $I_2$ are nonempty disjoint
subsets of $I$.
  \index{independent}%

\begin{prop}
  Let $\ph, \psi \in \dom P(\; \cdot \mid X)$. Then $\ph$ and $\psi$ are
  independent given $X$ if and only if $\ang{\ph, \psi}$ is independent given
  $X$.
\end{prop}

\begin{proof}
  The if direction is trivial. For the only if direction, suppose $\ph$ and
  $\psi$ are independent given $X$ and let $\ph' \in \de(\{\ph\})$ and $\psi'
  \in \de(\{\psi\})$. Note that $\de(\{\ph\})$ consists of tautologies,
  contradictions, and formulas that are equivalent to either $\ph$ or $\neg
  \ph$, and similarly for $\de(\{\psi\})$. We may assume that $0 < P(\ph' \mid
  X) < 1$ and $0 < P(\psi' \mid X) < 1$, so that neither $\ph$ nor $\psi$ is a
  tautology or contradiction.

  Clearly, $\ph'$ and $\psi'$ are independent if $\ph' \equiv \ph$ and $\psi'
  \equiv \psi$. By Corollary \ref{C:indep-prod}, $\ph'$ and $\psi'$ are
  independent if $\ph' \equiv \ph$ and $\psi' \equiv \neg \psi$. Repeated
  applications of this result cover the cases $\ph' \equiv \neg \ph$ and $\psi'
  \equiv \neg \psi$, and $\ph' \equiv \neg \ph$ and $\psi' \equiv \psi$.
\end{proof}

\begin{thm}\label{T:indep-dialog-defined}
  Let $P$ be an inductive theory and $X \in \ante P$. If $\ang{\ph_i \mid i \in
  I}$ is independent given $X$, then $\de(\{\ph_i \mid i \in I\}) \subseteq \dom
  P(\; \cdot \mid X)$.
\end{thm}

\begin{proof}
  Let $\Om = \B^{PV}$ and define $\De$ as in \eqref{Del-P-X}, so that $\De =
  Y_\Om$, where $Y = \dom P(\; \cdot \mid X)$. Proposition \ref{P:Del-P-X}
  implies that $\De$ is a Dynkin system on $\Om$. Let $U = \{(\ph_i)_\Om \mid i
  \in I\}$. Since $\ang{\ph_i \mid i \in I}$ is independent given $X$, it
  follows that $U \subseteq \De$. By Theorem \ref{T:indep-prod}, we have that
  $U$ is a $\pi$-system, that is, $U$ is closed under pairwise intersections.
  Therefore, Dynkin's $\pi$-$\la$ theorem gives $\si(U) \subseteq \De$.

  Now let $\ph \in \de(\{\ph_i \mid i \in I\})$. By Proposition
  \ref{P:dialog-sig-alg}, we have $\ph_\Om \in \si(U) \subseteq \De$. Hence, we
  may choose $\ph' \in \dom P(\; \cdot \mid X)$ such that $\ph_\Om = \ph'_\Om$.
  Since $\Om = \B^{PV}$, according to Remark \ref{R:impl-subset}, it follows
  that $\ph \equiv \ph'$. Therefore, by the rule of logical equivalence, $\ph
  \in \dom P(\; \cdot \mid X)$.
\end{proof}

\subsection{A semantic characterization of independence}\label{S:sem-indep}

Let $(S, \Ga, \nu)$ be a probability space and $I$ a set with $|I| \ge 2$. For
each $i \in I$, let $A_i \in \Ga$. Then $\ang{A_i \mid i \in I}$ is
\emph{measure independent in $(S, \Ga, \nu)$} if $\opnu \bigcap_{i \in J} A_i
= \prod_{i \in J} \opnu A_j$, whenever $J \subseteq I$ is finite.
  \index{independent!measure ---}%
Note that this is the usual definition of independence in a probability space.
Also note that we may assume without loss of generality that $|J| \ge 2$.

Let $P$ be an inductive theory. Let $\sP = (\Om, \Si, \bbP)$ be a model and
suppose $\sP \vDash P$. Let $X \in \ante P$. Write $X \equiv Y \cup \{\psi\}$,
where $\sP \vDash Y$ and $\olbbP \psi_\Om > 0$. Define the probability measure
$\olbbP_X$ on $(\Om, \ol \Si)$ by $\olbbP_X A = \olbbP A \cap \psi_\Om / \olbbP
\psi_\Om$, and let $\sP_X = (\Om, \ol \Si, \olbbP_X)$. Note that by Proposition
\ref{P:model-func}, the model $\sP_X$ does not depend on our choice of $Y$ and
$\psi$.

\begin{thm}\label{T:indep-sound-compl}
  Let $P$ be an inductive theory, $X \in \ante P$, and $I$ a set with $|I| \ge
  2$. Let $\ang{\ph_i \mid i \in I}$ be an indexed collection of formulas in
  $\dom P(\; \cdot \mid X)$. Then the following are equivalent:
  \begin{enumerate}[(i)]
    \item $\ang{\ph_i \mid i \in I}$ is independent given $X$.
    \item For any model $\sP$, if $\sP \vDash P$, then $\ang{(\ph_i)_\Om \mid i
          \in I}$ is measure independent in $\sP_X$.
    \item There exists a model $\sP$ such that $\sP \vDash P$ and $\ang{
          (\ph_i)_\Om \mid i \in I}$ is measure independent in $\sP_X$.
  \end{enumerate}
\end{thm}

\begin{proof}
  Suppose (i) holds. Assume $\sP$ is a model and $\sP \vDash P$. Let $J
  \subseteq I$ be finite with $|J| \ge 2$. Fix $k \in J$. Let $I_1 = \{k\}$ and
  $I_2 = J \setminus \{k\}$. Then $\ph_k \in \de(\{\ph_i: i \in I_1\})$ and
  $\bigwedge_{i\in I_2} \ph_i \in \de(\{\ph_i: i \in I_2\})$ are independent
  given $X$. By \eqref{indep-prod},
  \[
    \ts{
      P(\bigwedge_{i \in J} \ph_i \mid X) =
        P(\ph_k \mid X) P(\bigwedge_{i \in J \setminus \{k\}} \ph_i \mid X).
    }
  \]
  Since $\sP \vDash P$, this implies
  \[
    \ts{
      \olbbP_X \bigcap_{i \in J} (\ph_i)_\Om =
        \olbbP_X (\ph_k)_\Om
        \olbbP_X \bigcap_{i \in J \setminus \{k\}} (\ph_i)_\Om.
    }
  \]
  Iterating this argument gives $\olbbP_X \bigcap_{i \in J} (\ph_i)_\Om = \prod_
  {i \in J} \olbbP_X (\ph_i)_\Om$, so that the indexed collection $\ang{
  (\ph_i)_\Om \mid i \in I}$ is measure independent in $\sP_X$, showing that (i)
  implies (ii).

  By Theorem \ref{T:ind-th-model}, (ii) implies (iii).

  Suppose (iii) holds. For $i \in I$, let $A_i = (\ph_i)_\Om$, so that by
  hypothesis, $\ang{A_i \mid i \in I}$ is measure independent in $\sP_X$. A
  result from measure theory tells us that if $I_1$ and $I_2$ are disjoint
  subsets of $I$, then $A$ and $B$ are measure independent in $\sP_X$ whenever
  $A \in \si(\{A_i \mid i \in I_1\})$ and $B \in \si(\{A_i \mid i \in I_2\})$.

  Let $I_1$ and $I_2$ be nonempty disjoint subsets of $I$. Let $U = \{\ph_i \mid
  i \in I_1\}$ and $V = \{\ph_i \mid i \in I_2\}$. Let $\ph \in \de(U)$ and
  $\psi \in \de(V)$. By Proposition \ref{P:dialog-sig-alg}, we have $\ph_\Om \in
  \si(U_\Om)$ and $\psi_\Om \in \si(V_\Om)$, so that $\ph_\Om$ and $\psi_\Om$
  are measure independent in $\sP_X$. Hence,
  \begin{equation}\label{indep-sound-compl}
    \olbbP_X \ph_\Om \cap \psi_\Om = \olbbP_X \ph_\Om \olbbP_X \psi_\Om.
  \end{equation}
  Theorem \ref{T:indep-dialog-defined} implies $\ph$, $\psi$, and $\ph \wedge
  \psi$ are all in $\dom P(\; \cdot \mid X)$. Since $\sP \vDash P$, it follows
  that \eqref{indep-sound-compl} implies $P(\ph \wedge \psi \mid X) = P(\ph \mid
  X) P(\psi \mid X)$. By Theorem \ref{T:indep-prod}, therefore, $\ph$ and $\psi$
  are independent given $X$. Thus, (iii) implies (i).
\end{proof}

\begin{cor}\label{C:indep-sound-compl}
  Let $P$ be an inductive theory, $X \in \ante P$, and $I$ a set with $|I| \ge
  2$. Let $\ang{\ph_i \mid i \in I}$ be an indexed collection of formulas in
  $\dom P(\; \cdot \mid X)$. Then $\ang{\ph_i \mid i \in I}$ is independent
  given $X$ if and only if
  \[
    \ts{
      P(\bigwedge_{j \in J} \ph_j \mid X) = \prod_{j \in J} P(\ph_j \mid X)
    }
  \]
  for all finite $J \subseteq I$.
\end{cor}

\begin{proof}
  This follows immediately from Theorems \ref{T:indep-sound-compl} and 
  \ref{T:ind-th-model}.
\end{proof}

\subsection{Fair coin flips}\label{S:Karp413}

In this subsection, our aim is to create an inductive theory that describes an
infinite sequence of independent flips of a fair coin.

We must first construct the language in which this will be done. Let $PV =
\{\bfr_n^k \mid (n, k) \in \bN \times \{0, 1\}\}$. We interpret $\bfr_n^k$ as
representing the proposition, ``The $n$th flip of the coin lands on $k$.'' Here,
$k = 1$ represents heads and $k = 0$ represents tails.

Our inductive theory will be built on three ``axioms,'' informally stated as:
\begin{enumerate}[(1)]
  \item Each flip must land on heads or tails.
  \item On an individual flip, the probabilities of heads and tails are each
        $1/2$.
  \item The flips are independent.
\end{enumerate}
We will enforce (1) with our choice of root, $T_0$. We will enforce (2) with a
set $Q$ of inductive statements. We will enforce (3) with an inductive
condition, $\cC$.

Let $T_0 = T (\{\ze\})$, where $\ze = \bigwedge_n (\bfr_n^0 \vee \bfr_n^1)$. Let
\[
  Q = \{(T_0, \bfr_n^k, 1/2) \mid (n, k) \in \bN \times \{0, 1\}\}.
\]
Note that $Q$ is connected with root $T_0$. Define the inductive condition $\cC$
to be the set of all inductive theories with root $T_0$ such that $\langle
\bfr_n^{f(n)} \mid n \in \bN \rangle$ is independent given $T_0$ whenever $f \in
\{0, 1\}^\bN$.

Recall that $Q, \cC \vdash (X, \ph, p)$ means $\cC_Q \cap \cC \vdash (X, \ph,
p)$.

\begin{lemma}\label{L:fair-coin-flips}
  The inductive condition $\cC_Q \cap \cC$ is consistent.
\end{lemma}

\begin{proof}
  Let $\sP = (\Om, \Si, \bbP)$ be the model constructed in Example
  \ref{Expl:Karp413}. As shown in that example, $\sP \vDash T_0$. Hence, by
  Proposition \ref{P:chop-off-root}, we have that $P = \bTh \sP \dhl_{[T_0, \Th
  \sP]}$ is an inductive theory with root $T_0$. Note that 
  \[
    \bbP (\bfr_n^k)_\Om = \nu G \bfr_n^k = \nu \{X_n = k\} = 1/2.
  \]
  Hence, $P(\bfr_n^k \mid T_0) = \olbbP (\bfr_n^k)_\Om = 1/2$, so that $Q
  \subseteq P$, and therefore $P \in \cC_Q$.

  Let $f \in \{0, 1\}^\bN$. Since $(r_n^{f(n)})_\Om = \{X_n = f(n)\}$, it
  follows that $\langle (r_n^{f(n)})_\Om \mid n \in \bN \rangle$ is measure
  independent in $\sP$. Thus, Theorem \ref{T:indep-sound-compl} implies $P \in
  \cC$, and so $P \in \cC_Q \cap \cC$. Hence, $\cC_Q \cap \cC$ is nonempty, and
  therefore consistent.
\end{proof}

By Lemma \ref{L:fair-coin-flips}, we may define $\bfP_{Q, \cC} = \bfP(\cC_Q
\cap \cC)$, the inductive theory generated by $Q$ and $\cC$. Note that $Q, \cC
\vdash (X, \ph, p)$ if and only if $\bfP_{Q, \cC}(\ph \mid X) = p$. In other
words, $\bfP_{Q, \cC}$ is precisely the inductive theory we are aiming for. It
contains exactly those inductive statements that can be derived from (1)--(3).

\begin{prop}\label{P:fair-coin-flips-deter}
  The inductive condition $\cC_Q \cap \cC$ is determinate. That is, $\bfP_{Q,
  \cC} \in \cC_Q \cap \cC$.
\end{prop}

\begin{proof}
  Suppose $\sP \vDash \cC_Q \cap \cC$. Choose $P \in \cC_Q \cap \cC$ such that
  $\sP \vDash P$. Then $P$ is an inductive theory with root $T_0$ such that $Q
  \subseteq P$ and $\langle \bfr_n^{f(n)} \mid n \in \bN \rangle$ is independent
  given $T_0$ whenever $f \in \{0, 1\}^\bN$. Let $f \in \{0, 1\}^\bN$ and let
  $I \subseteq \bN$ be finite. By Corollary \ref{C:indep-sound-compl}, we have
  $P(\bigwedge_{i \in I} \bfr_i^{f(i)} \mid T_0) = 2^{-|I|}$. Hence, $\sP \vDash
  (T_0, \bigwedge_{i \in I} \bfr_i^{f(i)}, 2^{-|I|})$. By Theorem 
  \ref{T:ind-sound-compl-IC}, we have $Q, \cC \vdash (T_0, \bigwedge_{i \in I}
  \bfr_i^{f(i)}, 2^{-|I|})$. Therefore, $\bfP_{Q, \cC}(\bigwedge_{i \in I}
  \bfr_i^ {f (i)} \mid T_0) = 2^{-|I|}$. Again by Corollary 
  \ref{C:indep-sound-compl}, this shows that $\bfP_{Q, \cC} \in \cC$. Taking
  $|I| = 1$ shows $\bfP_{Q, \cC} \in \cC_Q$.
\end{proof}

\begin{prop}\label{P:fair-coin-flips-char}
  Let $P$ be the inductive theory in the proof of Lemma \ref{L:fair-coin-flips}.
  Then $P = \bfP_{Q, \cC}$.
\end{prop}

\begin{proof}
  Since $P \in \cC_Q \cap \cC$, we have $\bfP_{Q, \cC} \subseteq \bigcap \cC_Q
  \cap \cC \subseteq P$. For the reverse inclusion, since $P$ and $\bfP_{Q,
  \cC}$ both have root $T_0$, it suffices to show that ${P \dhl_{T_0}} \subseteq
  {\bfP_{Q, \cC} \dhl_{T_0}}$. For notational simplicity, let $P' = \bfP_{Q,
  \cC}$.

  We first show that $P'(\ph \mid T_0)$ exists for every $\ph \in \cF$. Let $Y
  \subseteq \cF$ be the set of all finite conjunctions of propositional
  variables. We claim that $Y \subseteq \dom P'(\; \cdot \mid T_0)$. To see
  this, let $\ph \in Y$. Suppose, for some $n$, that $\ph$ contains both
  $\bfr_n^0$ and $\bfr_n^1$. Then $\ph \vdash \bfr_n^0 \wedge \bfr_n^1$. Since
  $\bfr_n^0 \vee \bfr_n^1 \in T_0$, we have $P'(\bfr_n^0 \vee \bfr_n^1 \mid T_0)
  = 1$. Hence, Theorem \ref{T:incl-excl} implies $P'(\bfr_n^0 \wedge \bfr_n^1
  \mid T_0) = 0$, and therefore, $P'(\ph \mid T_0) = 0$. On the other hand,
  suppose that $\ph$ contains at most one of $\bfr_n^0$ and $\bfr_n^1$ for each
  $n$. Then $\ph = \bigwedge_{i \in I} \bfr_i^{f(i)}$ for some finite $I
  \subseteq \bN$ and some $f \in \{0, 1\}^\bN$. By the proof of Proposition 
  \ref{P:fair-coin-flips-deter}, we have $P'(\ph \mid T_0) = 2^{-|I|}$.
  Hence, $Y \subseteq \dom P'(\; \cdot \mid T_0)$.

  Recall that in Lemma \ref{L:fair-coin-flips}, we have $\Om = \B^{PV}$. The
  above shows that $Y_\Om \subseteq \De$, where $\De = \{\ph_\Om \mid P'(\ph
  \mid T_0) \text{ exists}\}$. Proposition \ref{P:Del-P-X} implies that $\De$ is
  a Dynkin system. Since $Y_\Om$ is closed under pairwise intersections,
  Dynkin's $\pi$-$\la$ theorem implies that $\cB^{PV} = \si(Y_\Om) \subseteq
  \De$. Hence, if $\ph \in \cF$, then $\ph_\Om \in \De$, so that $\ph_\Om =
  \ph'_\Om$ for some $\ph' \in \dom P'(\; \cdot \mid T_0)$. Remark
  \ref{R:impl-subset} gives $\ph \equiv \ph'$, so that by the rule of logical
  equivalence, $P'(\ph \mid T_0)$ exists.

  Now suppose $P(\ph \mid T_0, \psi) = p$. Then, by the definition of $P$, we
  have $\olbbP \ph_\Om \cap \psi_\Om / \olbbP \psi_\Om = p$. Since $P'(\ph
  \wedge \psi \mid T_0)$ and $P'(\psi \mid T_0)$ both exist and $\sP \vDash P'$,
  this implies that $P'(\ph \wedge \psi \mid T_0) / P'(\psi \mid T_0) = p$. From
  the multiplication rule, it follows that $P'(\ph \mid T_0, \psi) = p$.
\end{proof}

Recall $\psi_f = \neg \bigwedge_n \bfr_n^{f(n)}$ from Example 
\ref{Expl:Karp413}, where $f \in \{0, 1\}^\bN$. The function $f$ is simply a
sequence of $1$'s and $0$'s. If we interpret $f$ as a sequence of heads and
tails, then the formula $\psi_f$ represents the sentence,
\[
  \text{
    ``The pattern of heads and tails produced by the coin is not $f$.''
  }
\]
By Proposition \ref{P:fair-coin-flips-char} and Example \ref{Expl:Karp413}, we
have $\bfP_{Q, \cC}(\psi_f \mid T_0) = 1$ for every $f \in \{0, 1\}^\bN$. Hence,
$T_0, Q, \cC \vdash \psi_f$, so that $\psi_f$ is a logical consequence of
(1)--(3), and this is true for every $f \in \{0, 1\}^\bN$.

Classical intuition suggests that this is paradoxical, since the coin must
produce \emph{some} pattern. But this classical intuition is rooted in the idea
of strict satisfiability. Indeed, Example \ref{Expl:Karp413} shows that there is
no strict model that strictly satisfies both $T_0$ and every $\psi_f$. A strict
model is an assignment of truth values to every sentence. Classical intuition
thinks in terms of these truth assignments. To remove any cognitive dissonance
produced by this example, intuition must be changed so that it thinks in terms
of probability measures on truth assignments.

\subsection{Biased coin flips}

As in the previous subsection, our aim here is to create an inductive theory
that describes an infinite sequence of independent coin flips. This time,
however, we will drop the assumption that the coin is fair.

As before, we use the language built on $PV = \{\bfr_n^k \mid (n, k) \in \bN
\times \{0, 1\}\}$, where $\bfr_n^k$ represents the proposition, ``The $n$th
flip of the coin lands on $k$.''

This time, the ``axioms'' of our inductive theory will be:

\begin{enumerate}[(1)]
  \item Each flip must land on heads or tails.
  \item On an individual flip, the probabilities of heads and tails sum to $1$.
  \item Every flip has the same probability of heads, which is neither $0$ nor
        $1$.
  \item The flips are independent.
\end{enumerate}

We will enforce (1) with our choice of root, $T_0$. We will enforce (2)--(4)
with inductive conditions.

As before, let $T_0 = T (\{\ze\})$, where $\ze = \bigwedge_n (\bfr_n^0 \vee
\bfr_n^1)$. Let $\fI_{T_0}$
  \symindex{$\fI_{T_0}$}%
be the set of inductive theories with root $T_0$. Define the inductive
conditions,
  \begin{align*}
    \cC_2 &= \{
      P \in \fI_{T_0}
    \mid
      P(\bfr_n^0 \mid T_0) + P(\bfr_n^1 \mid T_0) = 1 \text{ for all $n$}
    \},\\
    \cC_q &= \{
      P \in \fI_{T_0} 
    \mid
      P(\bfr_n^1 \mid T_0) = q \text{ for all $n$}
    \}, \quad q \in (0, 1),\\
    \cC_4 &= \{
      P \in \fI_{T_0}
    \mid
      \langle \bfr_n^{f(n)} \mid n \in \bN \rangle \text{
        is independent given $T_0$ for all
      } f \in \{0, 1\}^\bN
    \}
  \end{align*}
Let $\cC_3 = \bigcup_{q \in (0, 1)} \cC_q$. Then $\cC_j$ represents assumption 
(j) for $2 \le j \le 4$. Let $\cC = \cC_2 \cap \cC_3 \cap \cC_4$.

\begin{prop}\label{P:bias-coin-flips}
  The condition $\cC$ is consistent, but indeterminate. That is, $\bfP_\cC
  \notin \cC$. More precisely, the domain of $\bfP_\cC(\; \cdot \mid T_0)$ does
  not contain any propositional variables. Hence, $\bfP_\cC \notin \cC_2$,
  $\bfP_\cC \notin \cC_4$, and $\bfP_\cC \notin \cC_q$ for any $q \in (0, 1)$.
\end{prop}

\begin{proof}
  Let $\sP$ be the model constructed in Example \ref{Expl:Karp413} and let $P =
  \bTh \sP \dhl_{[T_0, \Th \sP]}$. Then $P \in \cC_2 \cap \cC_{1/2} \cap \cC_4
  \subseteq \cC$ and $\sP \vDash P$. Hence, $\sP \vDash \cC$, so that $\cC$ is
  satisfiable and therefore consistent.

  Fix $r \in PV$ and assume $\bfP_\cC(r \mid T_0)$ exists. Let $q_0 = \bfP_\cC(r
  \mid T_0)$. Choose $q \in (0, 1)$ such that $q_0 \notin \{q, 1 - q\}$. As in
  Example \ref{Expl:Karp413}, we may construct a model $\sP^q$ such that $P^q =
  \bTh \sP^q \dhl_{[T_0, \Th \sP]} \in \cC_2 \cap \cC_q \cap \cC_4 \subseteq
  \cC$. Since $P^q \in \cC_q$, we have $P^q(r \mid T_0) \in \{q, 1 - q\}$. On
  the other hand, since $P^q \in \cC$, it follows that $\bfP_\cC \subseteq
  \bigcap \cC \subseteq P^q$, so that $P^q(r \mid T_0) = q_0$, a contradiction.
  Hence, $\bfP_\cC(r \mid T_0)$ does not exist.
\end{proof}

Proposition \ref{P:bias-coin-flips} shows that the domain of $\bfP_\cC(\; \cdot
\mid T_0)$ does not contain any propositional variables. This domain, however,
is not trivial. That is, it contains more than just tautologies and
contradictions. Recall the formulas $\psi_f = \neg \bigwedge_n \bfr_n^{f (n)}$,
where $f \in \{0, 1\}^\bN$. As in Example \ref{Expl:Karp413}, we can show that
$P(\psi_f \mid T_0) = 1$ for every $P \in \cC$. Since $\bfP_\cC \dhl_{T_0} =
\bigcap \cC^0$, it follows that $\bfP_\cC(\psi_f \mid T_0) = 1$ for every $f \in
\{0, 1\}$.

It might be tempting to think that $\bfP_\cC(\; \cdot \mid T_0)$ is entirely
deductive, in the sense that $\bfP_\cC(\ph \mid T_0) \in \{0, 1\}$ whenever
$\bfP_\cC(\ph \mid T_0)$ exists. After all, the inductive condition $\cC$ does
not specify any numerical probabilities at all. What we will show, however, is
that the opposite is true. For any $p \in (0, 1)$, there is a formula $\ph \in
\cF$ such that $\bfP_\cC(\ph \mid T_0) = p$.

The intuition behind this is the following. It is possible to simulate a fair
coin flip with a biased coin. Simply flip the coin twice. If the results match,
start over. If they do not match, use the second of the two flips as the result.
But we can do this as many times as we like. So we can simulate an
i.i.d.~sequence of fair coin flips. We can then use this sequence to simulate a
random number that is uniformly chosen from the interval $(0, 1)$. Finally, we
construct a formula which asserts that this uniform random number is less than
$p$.

\begin{prop}
  For any $p \in (0, 1)$, there exists a formula $\ph \in \cF$ such that
  $\bfP_\cC(\ph \mid T_0) = p$.
\end{prop}

\begin{proof}
  Fix $p \in (0, 1)$. Let $\Om = \B^{PV}$ be the set of all strict models and
  $\Si = \cB^{PV} = \{\ph_\Om \mid \ph \in \cF\}$. Let $\ze' = \bigwedge_n \neg
  (\bfr_n^0 \wedge \bfr_n^1)$. Define $Y_n: \Om \to \{0, 1\}$ as follows. If
  $\om \ntDash \ze \wedge \ze'$, then $Y_n(\om) = 0$. If $\om \tDash \ze \wedge
  \ze'$, then define $Y_n(\om)$ so that $\om \tDash \bfr_n^{Y_n(\om)}$. Note
  that $\{Y_n = 1\} = (\ze \wedge \ze' \wedge \bfr_n^1)_\Om \in \Si$. Hence,
  $Y_n$ is $\Si$-measurable.

  Define $\tau_0 = 0$ and
  \[
    \tau_k = \inf \{
      n > \tau_{k - 1}
    \mid
      \text{$n$ is even and $Y_n \ne Y_{n - 1}$}
    \}.
  \]
  Define $Z_k = 0$ on $\{\tau_k = \infty\}$ and $Z_k = Y_{\tau_k}$ on $
  \{\tau_k < \infty\}$, and let $U = \sum_1^\infty 2^{-k} Z_k$. Then $U$ is
  $\Si$-measurable, which implies $\{U \le p\} \in \Si$. Choose $\ph \in \cF$
  such that $\{U \le p\} = \ph_\Om$. We will show that $\bfP_\cC(\ph \mid T_0) =
  p$.

  Fix $P \in \cC$. Let $\sP$ be an arbitrary model with $\sP \vDash P$. Define
  $P' = \bTh \sP \dhl_{[T_0, \Th \sP]}$, so that $P'$ is a complete inductive
  theory with root $T_0$ such that $P \subseteq P'$. By the proof of Theorem
  \ref{T:ind-th-model}, we may construct a model $\sP' = (\Om, \Si, \bbP)$,
  where $\Om = \B^{PV}$ and $\Si = \cB^{PV}$, such that $P' = \bTh \sP' \dhl_
  {[T_0, \Th \sP']}$.

  Using $P \in \cC$, $P \subseteq P'$, and Theorem \ref{T:incl-excl}, we have
  $P' (\ze' \mid T_0) = 1$. Hence, $ \{Y_n = 1\} = (\bfr_n^1)_\Om$ $\bbP$-a.s.
  Thus, $\{Y_n\}$ are i.i.d.~in $\sP'$ with $\bbP \{Y_n = 1\} = q$, where $q$ is
  the value satisfying $P \in \cC_q$. It is a straightforward exercise to verify
  that $\{Z_k\}$ are i.i.d.~with $\bbP \{Z_k = 1\} = 1/2$. Hence, $U$ is
  uniformly distributed on $(0, 1)$, which gives $P'(\ph \mid T_0) = \bbP
  \ph_\Om = \bbP \{U \le p\} = p$.

  From the definition of $P'$, we have $\sP \vDash (T_0, \ph, p)$. Since $\sP$
  was arbitrary, Theorem \ref{T:ind-sound-compl} implies $P \vdash (T_0, \ph,
  p)$. But $P$ is an inductive theory, so this gives $P(\ph \mid T_0) = p$.
  Finally, since $P$ was arbitrary, it follows that $(T_0, \ph, p) \in \bigcap
  \cC^0 = {\bfP_\cC \dhl_ {T_0}}$, so that $\bfP_\cC(\ph \mid T_0) = p$.
\end{proof}

%% file: pred-logic.tex

\chapter{Predicate Logic}\label{Ch:pred-logic}

In this chapter, we repeat the work we did in Chapters \ref{Ch:prop-calc} and
\ref{Ch:prop-models}, but in the setting of a predicate language. Most of the
work in this chapter is devoted to the deductive side of predicate logic. The
development of inductive logic requires hardly any modification from the
propositional case.

The language we use is just like first-order logic, except it allows countable
conjunctions and disjunctions. It is typically denoted in the literature by
$\cL_{\om_1, \om}$. We will denote it simply by $\cL$.

In Section \ref{S:syntax}, we introduce the syntax of $\cL$. We define terms,
formulas, sentences, and all of their related concepts, such as subformulas,
free and bound variables, and substitutions. The set of formulas is denoted by
$\cL$. As in the propositional case, formulas are built from an alphabet of
symbols. To the propositional alphabet, we add an uncountable number of variable
symbols, and we add the logical symbol $\forall$. Unlike the propositional case,
our alphabet will not include the set $PV$. Instead, it will include a set of
symbols $L$, called the extralogical signature. The set $L$ includes constant
symbols, relation symbols, and function symbols. 

Sentences are formulas that have no free variables. For example, $x > 0$ is a
formula, whereas $\forall x \, x > 0$ is a sentence. The set of formulas is
denoted by $\cL$, and the set of sentences is denoted by $\cL^0 \subseteq \cL$.
Intuitively, a sentence says something. It can be meaningful and have a truth
value. On the other hand, a formula is ambiguous. It could mean many different
things, depending on the values assigned to its free variables. As such, it
cannot have its own truth value. Predicate logic is concerned with sentences. In
fact, both deductive and inductive theories consist exclusively of sentences. As
such, our models and our inferential calculi should deal directly with $\cL^0$.
In the inductive case, that is exactly what we will do. In the deductive case,
however, we do something different. In that case, it will be easier for us to
build models and inferential calculi for $\cL$. When we do so, we will treat
free variables as if they are constant symbols.

Section \ref{S:pred-calc} is concerned with inferential calculus in predicate
logic. The bulk of this section is devoted to a system of natural deduction for
deductive inference in $\cL$. We present this system, prove that it is
$\si$-compact, and then show how it connects to inference in $\cL^0$. We discuss
the inductive calculus, which carries over almost entirely unchanged from
Chapter \ref{Ch:prop-calc}. Finally, we give a Hilbert-type calculus for
deductive inference. Specifically, this is the calculus used by Carol Karp in
\cite{Karp1964}.

In Section \ref{S:pred-models}, we present the semantics of predicate logic. As
before, the majority of the section is given to deductive semantics. Inductive
semantics carry over from Chapter \ref{Ch:prop-models} with very little
modifications. In predicate logic, sentences are given meaning by interpreting
them in a structure. A model will be a probability measure on a set of
structures. Using this, we will define satisfiability and the consequence
relation, then prove $\si$-compactness, soundness, and completeness.

Section \ref{S:pred-models} also contains the theory of Peano arithmetic in the
infinitary setting. We present the usual theory from first-order logic, as well
as two different extensions to the infinitary language $\cL^0$. That is, we
define three theories, $\PA_\fin \subseteq \PA_- \subseteq \PA$. The theory
$\PA_\fin$ is the usual theory of Peano arithmetic in first-order logic. The
theories $\PA_-$ and $\PA$ are extensions to $\cL^0$. The first extension,
$\PA_-$, is conservative, in the sense that every sentence in $\PA_- \setminus
\PA_\fin$ is purely infinitary. In other words, in $\PA_-$, we cannot deduce any
new first-order sentences that we could not already deduce in $\PA_\fin$. This
is because $\PA_-$ and $\PA_\fin$ have the same axioms. In particular, even in
$\PA_-$, we can only do induction on finitary formulas. The second extension,
$\PA$, is stronger. There, we allow induction on infinitary formulas. In doing
so, we find that $\PA$ completely characterizes the natural numbers. That is,
every true sentence about arithmetic is provable in $\PA$.

Finally, in Section \ref{S:pred-models-RVs}, we connect inductive logic to the
measure-theoretic concept of a random variable. As described in Section
\ref{S:outline}, we will be forced to deal with an issue we call ``the
relativity of randomness.'' We will do so by introducing and discussing ``frames
of references.'' The connection between inductive logic and random variables
will allow us to show that a measure-theoretic probability model---that is, a
probability space, together with a collection of random variables---is a special
case of a model in inductive logic. In other words, measure-theoretic
probability is embedded in inductive logic. In fact, this embedding is proper.
As we will see in Example \ref{Expl:pred-Karp412}, inductive logic is capable of
expressing things that cannot be expressed in measure-theoretic probability.

\section{The syntax of predicate formulas}\label{S:syntax}

In this section, we present the predicate language $\cL$. We describe the
alphabet, and the rules for constructing terms, formulas, and sentences. In
Karp's original construction of $\cL$ (see \cite{Karp1964}), formulas are
countably long strings of symbols. Our construction follows \cite{Keisler1971}
instead, building up formulas out of sets.

\subsection{The alphabet and terms}\label{S:terms}

Let $L$ be an extralogical signature. Recall the convention that, unless
otherwise stated, $c$ denotes a constant symbol, $r$ a relation symbol, and $f$
a function symbol with arity $n \ge 1$.

Let $\Var = \{\bx_\al \mid \al < \bom_1\}$ be an uncountable set of symbols.
The symbols in $\Var$ are called \emph{individual variables}. Unless otherwise
stated, letters such as $u, v, x, y, z$ will denote distinct individual
variables.
  \symindex{$\Var$}%
  \index{variable!individual ---}%

We define an alphabet, $\A = L \cup \Var \cup \{\neg, \bigwedge, \forall,
{\beq}\}$. Let $\cS = \cS_L$ denote the set of (finite) strings over $\A$.
  \symindex{$\cS$}%
We use boldface for the symbol, $\beq$, to distinguish it from the ordinary
equal sign. For instance, $\xi = \bx_0 \beq \bx_1$ means that $\xi$ is equal to
the length-3 string, $\bx_0 \beq \bx_1$. Parentheses are not part of our
alphabet, but we may sometimes add them for readability. For example, the above
might be written as $\xi = (\bx_0 \beq \bx_1)$ to further emphasize the
distinction between $=$ and $\beq$. We may also write $\xi: \bx_0 \beq \bx_1$,
as another way to improve readability.

\begin{defn}
    \index{term}%
    \symindex{$\cT$}%
  The set of \emph{terms} in $L$, denoted by $\cT = \cT_L$, is the smallest
  subset of $\cS$ such that
  \begin{enumerate}[(i)]
    \item $\Var \subseteq \cT$,
    \item $c \in \cT$ for all constant symbols $c \in L$, and
    \item if $f \in L$ is an $n$-ary function symbol and $t_1, \ldots, t_n \in
          \cT$, then $f t_1 \cdots t_n \in \cT$.
  \end{enumerate}
\end{defn}

Unless otherwise stated, letters such as $s$ and $t$ will denote terms.
Individual variables and constant symbols are called \emph{prime terms}.
  \index{term!prime ---}%
  \index{prime term|see {term}}%
A term is called \emph{compound} if it is not prime. Terms of the form in (iii)
are called \emph{function terms}. Note that every compound term is a function
term.

Also note that terms have the same definition here as they do in first-order
logic. Hence, they have all the same properties. For instance (see \cite[Section
2.2]{Rautenberg2010}), the \emph{unique term concatenation property} says that
if $t_1 \cdots t_n = s_1 \cdots s_m$, then $n = m$ and $t_i = s_i$ for all $i$.
Also, the \emph{unique term reconstruction property} says that if $f t_1 \cdots
t_n = f s_1 \cdots s_n$, then $t_i = s_i$ for all $i$.

When convenient, we will adopt the notation $\vec t = t_1 \cdots t_n$ for a
concatenation of terms. We also adopt shorthand to improve the readability of
terms. For example, suppose $x, y, z \in \Var$, and let $+$ and $\circ$ be
binary operation symbols. Then $t = {+}xy$ is a term, and ${\circ}tz = {\circ}
{+}xyz$ is a term. This latter term is especially difficult to read, and we
would typically write it as $(x + y) \circ z$. Note that parentheses are not
symbols in our alphabet; this is simply shorthand. We also adopt this shorthand
for relations, so that ${<}xy$ would be written as $x < y$.

\begin{defn}
    \symindex{$\var t$}%
  For $t \in \cT$, the set $\var t \subseteq \Var$ is defined recursively as
  follows:
  \begin{enumerate}[(i)]
    \item if $c \in L$ is a constant symbol, then $\var c = \emp$,
    \item if $x \in \Var$, then $\var x = \{x\}$, and
    \item $\var f t_1 \cdots t_n = \var t_1 \cup \cdots \cup \var t_n$.
  \end{enumerate}
\end{defn}

Intuitively, $\var t$ is the set of individual variables occurring in $t$. By
induction on $t$, it can be shown that $\var t$ is countable for all $t \in
\cT$. If $\var t = \emp$, then we call $t$ a \emph{ground term}, or
\emph{constant term}.
  \index{term!ground ---}%
  \index{ground term|see {term}}%

\begin{defn}
    \symindex{$\sym t$}%
  For $t \in \cT$, the set $\sym t \subseteq L$ is defined recursively as
  follows:
  \begin{enumerate}[(i)]
    \item if $c \in L$ is a constant symbol, then $\sym c = \{c\}$,
    \item if $x \in \Var$, then $\sym x = \emp$, and
    \item $\sym f t_1 \cdots t_n = \{f\} \cup \sym t_1 \cup \cdots \cup \sym
          t_n$.
  \end{enumerate}
\end{defn}

Intuitively, $\sym t$ is the set of extralogical symbols occurring in $t$. Note
that $\sym t$ is also countable. In addition, we define 
\[
  \con t = \{c \in \sym t \mid \text{$c$ is a constant symbol}\},
\]
  \symindex{$\con t$}%
which denotes the (countable) set of constant symbols occurring in $t$.

\subsection{Formulas}\label{S:pred-formulas}

A string $\ph \in \cS$ is an \emph{equation} if $\ph = (s \beq t)$, where $s, t
\in \cT$.
  \index{equation}%
A string $\ph \in \cS$ is a \emph{prime (or atomic) formula} if it is an
equation or if it has the form $\ph = r t_1 \cdots t_n$, where $r \in L$ is an
$n$-ary relation symbol and $t_1, \ldots, t_n \in \cT$.
  \index{formula!prime ---}%

Note that prime formulas have the same definition here as they do in first-order
logic. Hence, they have all the same properties. For instance (see \cite[Section
2.2]{Rautenberg2010}), the \emph{unique prime formula reconstruction property}
says that if $r t_1 \cdots t_n = r s_1 \cdots s_n$, then $t_i = s_i$ for all
$i$. Also, terms do not contain the symbol $\beq$. Therefore, if $(s \beq t) = 
(s' \beq t')$, then $s = s'$ and $t = t'$.

We will define the set of formulas so that a formula is a finite tuple, where
each element in the tuple is either a symbol from our alphabet, a formula, or a
countable set of formulas.

Let $S_0$ denote the set of prime formulas. For an ordinal $\al < \bom_1$, let
\[
  S_\al' = {
    S_\al \cup \{\ang{\neg, \ph} \mid \ph \in S_\al\}
    \cup \{\ang{\forall, x, \ph} \mid x \in \Var, \ph \in S_\al\}
  }.
\]
As with strings, when writing tuples such as these, we will typically omit the
commas and angled brackets, so that, for instance, $\forall x \ph = 
\ang{\forall, x, \ph}$.

We then define
\[
  S_{\al + 1} = S_\al' \cup \{
    \ang{\ts{\bigwedge}, \Phi} \mid
      \Phi \subseteq S_\al' \text{ is nonempty and countable}
  \}.
\]
Here, countable means finite or countably infinite. As above, we will typically
write $\bigwedge \Phi$ as a shorthand for ordered pairs of this type.

In the case that $\al$ is a nonzero limit ordinal, we define $S_\al = \bigcup_
{\xi < \al} S_\xi$. Finally, we define $\cL = \cL_{\bom_1, \bom} = \bigcup_{\al
< \bom_1} S_\al$.
  \symindex{$\cL$}%
Note that $S_\al \subseteq S_\be$ whenever $\al < \be$. An element $\ph \in \cL$
is called a \emph{formula}.
  \index{formula}%
A formula $\ph$ is called a \emph{literal} if $\ph = \pi$ or $\ph = \neg \pi$
for some prime formula $\pi$.

\begin{thm}[Unique formula reconstruction property]
  If $\ph$ is a formula that is not prime, then exactly one of the following
  holds.
  \begin{enumerate}[(i)]
    \item There exists a unique $\psi \in \cL$ such that $\ph = \neg \psi$.
    \item There exists a unique $x \in \Var$ and a unique $\psi \in \cL$ such
          that $\ph = \forall x \psi$.
    \item There exists a unique $\Phi \subseteq \cL$ such that $\ph = \bigwedge
          \Phi$.
  \end{enumerate}
\end{thm}

\begin{proof}
  Let $\ph \in \cL$ and assume $\ph$ is not prime. Let $\be$ be the smallest
  ordinal such that $\ph \in S_\be$. Since $\ph$ is not prime, $\be > 0$. Since
  $\ph \notin S_\xi$ for all $\xi < \be$, it follows that $\be$ is not a limit
  ordinal. Therefore, $\be$ is a successor ordinal, and we may write $\be = \al
  + 1$. Since $\ph \notin S_\al$, we have
  \begin{multline*}
    \ph \in \{\ang{\neg, \psi} \mid \psi \in S_\al\}
      \cup \{\ang{\forall, x, \psi} \mid x \in \Var, \psi \in S_\al\}\\
    \cup \{
      \ang{\ts{\bigwedge}, \Phi} \mid
        \Phi \subseteq S_\al' \text{ is nonempty and countable}
    \}.
  \end{multline*}
  Note that the above union is a disjoint union. Hence, $\ph$ is in exactly one
  of the above three sets.
\end{proof}

Let $\cL_\fin$ denote the smallest subset of $\cL$ that satisfies
  \symindex{$\cL_\fin$}%
\begin{enumerate}[(i)]
  \item prime formulas are in $\cL_\fin$,
  \item if $\ph \in \cL_\fin$ and $x \in \Var$, then $\neg \ph \in \cL_\fin$
        and $\forall x \ph \in \cL_\fin$, and
  \item if $\Phi \subseteq \cL_\fin$ is nonempty and finite, then $\bigwedge
        \Phi \in \cL_\fin$.
\end{enumerate}
Formulas in $\cL_\fin$ are said to be \emph{finitary}. The set $\cL_\fin$ is, in
fact, the set of formulas used in first-order logic. The reader can consult any
introductory text on mathematical logic for the basic properties of $\cL_\fin$
and its corresponding syntax and semantics. When necessary, we will cite
\cite{Rautenberg2010} for this purpose.

We adopt all the same shorthand as in the propositional language $\cF$, except
for the definitions of falsum and verum, which will be given later in Section
\ref{S:taut-pred}. In addition, we also use the shorthand $\exists x \ph = \neg
\forall \neg \ph$ and $(s \nbeq t) = \neg (s \beq t)$. We may also write
$\forall x_1 x_2 \cdots x_n$ or $\forall \vec x$ instead of $\forall x_1 \forall
x_2 \cdots \forall x_n$. If $\rtri$ is a binary relation symbol, we will write
\begin{align*}
  (\forall x \rtri t) \ph &= \forall x (x \rtri t \to \ph),\\
  (\exists x \rtri t) \ph &= \exists x (x \rtri t \wedge \ph),
\end{align*}
and similarly for $\nrtri$. In Section \ref{S:subs}, after introducing
substitutions, we will give shorthand for $\exists!$, the unique existential
quantifier.

The set of formulas $\cL$ depends on the extralogical signature $L$. We may
sometimes emphasize this fact in our notation. For example, if $L = \{\circ,
e\}$, we may write $\cL = \cL\{\circ, e\}$. We may also write $\cS_\cL$ and
$\cT_\cL$ instead of $\cS_L$ and $\cT_L$.

\subsection{Formula induction and recursion}

The proof of Theorem \ref{T:prop-form-ind} carries over with minor modification
to give us the following.

\begin{thm}[The principle of formula induction]
    \index{induction!formula ---}%
  The set of formulas, $\cL$, is the smallest set that satisfies the
  following:
  \begin{enumerate}[(i)]
    \item prime formulas are in $\cL$,
    \item if $\ph \in \cL$ and $x \in \Var$, then $\neg \ph \in \cL$ and
          $\forall x \ph \in \cL$, and
    \item if $\Phi \subseteq \cL$ is nonempty and countable, then $\bigwedge
          \Phi \in \cL$.
  \end{enumerate}
\end{thm}

Given $\ph \in \cL$, we define $\Sf \ph \subseteq \cL$, the set of
\emph{subformulas} of $\ph$, by formula recursion.
  \index{subformula}%
  \symindex{$\Sf \ph$}%
Namely, $\Sf \pi = \{\pi\}$ if $\pi$ is prime, $\Sf \neg \ph = \{\neg \ph\} \cup
\Sf \ph$, $\Sf \bigwedge \Phi = \{\bigwedge \Phi\} \cup \bigcup_{\ph \in \Phi}
\Sf \ph$, and $\Sf \forall x \ph = \{\forall x \ph\} \cup \Sf \ph$. It follows
by formula induction that $\Sf \ph$ is countable for every $\ph \in \cL$.

Given $\ph \in \cL$, we define $\len \ph \in \bN \cup \{\infty\}$, which we call
the \emph{length of $\ph$}, by formula recursion.
  \index{length}%
  \symindex{$\len \ph$}%
If $\ph$ is prime, then $\ph$ is a finite, nonempty string of symbols from our
alphabet $\A$. In this case, let $\len \ph$ be the length of this string. We
then extend this by $\len \neg \ph = 1 + \len \ph$, $\len \forall x \ph = 2 +
\len \ph$, and $\len \bigwedge \Phi = 1 + \sum_{\ph \in \Phi} \len \ph$. Note
that if $\ph$ has a subformula of infinite length, then $\ph$ has infinite
length. Also note that $\ph \in \cL_\fin$ if and only if $\len \ph < \infty$.

Given $\ph \in \cL$, we define the ordinal $\rk \ph$, called the \emph{rank} of
$\ph$, by formula recursion.
  \index{rank of a formula}%
  \symindex{$\rk \ph$}%
Namely, $\rk \pi = 0$ if $\pi$ is prime, $\rk \neg \ph = \rk \ph + 1$, $\rk
\bigwedge \Phi = (\bigcup_{\ph \in \Phi} \rk \ph) + 1$, and $\rk \forall x \ph =
\rk \ph + 1$. Note that $\rk \ph = 0$ if and only if $\ph$ is prime, and $\rk
\ph$ is a successor ordinal whenever $\ph$ is not prime.

\subsection{Variables and symbols}

\begin{defn}
    \symindex{$\var \ph$}%
  For $\ph \in \cL$, the set $\var \ph \subseteq \Var$ is defined recursively as
  follows:
  \begin{enumerate}[(i)]
    \item $\var s \beq t = \var s \cup \var t$,
    \item $\var r t_1 \cdots t_n = \var t_1 \cup \cdots \cup \var t_n$,
    \item $\var \neg \ph = \var \ph$,
    \item $\var \bigwedge \Phi = \bigcup_{\ph \in \Phi} \var \ph$, and
    \item $\var \forall x \ph = \var \ph \cup \{x\}$.
  \end{enumerate}
\end{defn}

Intuitively, $\var \ph$ is the set of individual variables occurring in $\ph$.
It follows by formula induction that $\var \ph$ is countable for every $\ph \in
\cL$. In other words, even though $\Var$ is uncountable, any given formula will
only make use of countably many individual variables.

Given $\ph \in \cL$, we define $\bnd \ph \subseteq \Var$, the set of \emph{bound
variables} in $\ph$, by formula recursion.
  \index{variable!bound ---}%
  \symindex{$\bnd \ph$}%
Namely, $\bnd \pi = \emp$ if $\pi$ is prime, $\bnd \neg \ph = \bnd \ph$, $\bnd
\bigwedge \Phi = \bigcup_{\ph \in \Phi} \bnd \ph$, and $\bnd \forall x \ph =
\bnd \ph \cup \{x\}$. Intuitively, $\bnd \ph$ is the set of variables $x$ such
that the prefix $\forall x$ occurs in $\ph$. If $\bnd \ph = \emp$, then $\ph$ is
\emph{quantifier-free}.

Given $\ph \in \cL$, we define the set of \emph{free variables} in $\ph$,
denoted by $\free \ph$, by formula recursion.
  \index{variable!free ---}%
  \symindex{$\free \ph$}%
Namely, $\free \pi = \var \pi$ if $\pi$ is prime, $\free \neg \ph = \free \ph$,
$\free \bigwedge \Phi = \bigcup_{\ph \in \Phi} \free \ph$, and $\free \forall x
\ph = \free \ph \setminus \{x\}$. Intuitively, $\free \ph$ is the set of
variables in $\ph$ that are not associated with a quantifier. For $X \subseteq
\cL$, we define $\free X = \bigcup_{\ph \in X} \free \ph$.

Strictly speaking, $\bnd \ph$ and $\free \ph$ need not be disjoint. For example,
suppose
\[
  \ph = (x \le y \wedge \forall x \exists y \, x + y \beq 0).
\]
Here, $\le$ is a binary relation symbol and $0$ is a constant symbol. In this
formula, the first occurrences of $x$ and $y$ are both free, whereas the others
are bound. Hence, $\bnd \ph = \free \ph = \{x, y\}$. Once we move beyond the
syntax of formulas and establish their logical relationships, we will see that
$\ph$ is logically equivalent to
\[
  \ph' = (x \le y \wedge \forall u \exists v \, u + v \beq 0).
\]
In this formula, $\bnd \ph' = \{u, v\}$ and $\free \ph' = \{x, y\}$, so that
$\bnd \ph'$ and $\free \ph'$ are disjoint. In general, given any formula $\ph$,
there is a logically equivalent $\ph'$ with $\bnd \ph' \cap \free \ph' = \emp$.
Hence, for most purposes, we may assume that no variable is both bound and free.

A \emph{sentence}, or \emph{closed formula}, is a formula $\ph$ such that
$\free \ph = \emp$.
  \index{sentence}%
  \symindex{$\cL^0$}%
The set of sentences is denoted by $\cL^0$. The set of finitary sentences is
$\cL^0_\fin = \cL^0 \cap \cL_\fin$. Note that $\cL^0_\fin$ is the set of
sentences used in first-order logic. An \emph{open formula} is a formula that
has one or more free variables.

If $x_1, \ldots, x_n \in \Var$ are distinct, we will write $\ph = \ph(x_1,
\ldots, x_n)$ or $\ph = \ph(\vec x)$ to mean that $\free \ph \subseteq \{x_1,
\ldots, x_n\}$. Similarly, for $t \in \cT$, we write $t = t(\vec x)$ to mean
that $\var t \subseteq \{x_1, \ldots, x_n\}$.
  \symindex{$t(\vec x)$}%
  \symindex{$\ph(\vec x)$}%

\begin{defn}
    \symindex{$\sym \ph$}%
  For $\ph \in \cL$, the set $\sym \ph \subseteq L$ is defined recursively as
  follows:
  \begin{enumerate}[(i)]
    \item $\sym s \beq t = \sym s \cup \sym t$,
    \item $\sym r t_1 \cdots t_n = \{r\} \cup \sym t_1 \cup \cdots \cup \sym
          t_n$,
    \item $\sym \neg \ph = \sym \ph$,
    \item $\sym \bigwedge \Phi = \bigcup_{\ph \in \Phi} \sym \ph$, and
    \item $\sym \forall x \ph = \sym \ph$.
  \end{enumerate}
\end{defn}

Intuitively, $\sym \ph$ is the set of extralogical symbols occurring in $\ph$.
Note that $\sym \ph$, like $\var \ph$, is also countable. In addition, we define 
\[
  \con \ph = \{c \in \sym \ph \mid \text{$c$ is a constant symbol}\},
\]
  \symindex{$\con \ph$}%
which denotes the (countable) set of constant symbols occurring in $\ph$.

\subsection{Substitutions}\label{S:subs}

A \emph{substitution} is a function $\si: \Var \to \cT$.
  \index{substitution}%
Such a function can be extended to $\si: \cT \to \cT$ by $c^\si = c$ and $(f
\vec t)^\si = f t_1^\si \cdots f t_n^\si$. Given $\ph \in \cL$, we want to
define $\ph^\si \in \cL$ so that $\ph^\si$ denotes the result of substituting
for every term $t$ in $\ph$ the new term $t^\si$. We do this by transfinite
recursion on the rank of $\ph$.

Suppose $\ph \in \cL$ is prime. Then $\ph = (s \beq t)$ or $\ph = r t_1 \cdots
t_n$. In the first case, we define $\ph^\si = (s^\si \beq t^\si)$. In the second
case, we define $\ph^\si = r t_1^\si \cdots t_n^\si$. In this way, we have
defined the map $\si \mapsto \ph^\si$ for every $\ph \in \cL$ with $\rk \ph =
0$.

Let $\al$ be a nonzero ordinal and assume $\si \mapsto \ph^\si$ has been defined
for every $\ph \in \cL$ with $\rk \ph < \al$. Fix $\ph \in \cL$ with $\rk \ph =
\al$. By the unique formula reconstruction property, one of the following holds:
\begin{enumerate}[(i)]
  \item $\ph = \neg \psi$, where $\rk \psi < \al$,
  \item $\ph = \bigwedge \Phi$, where $\rk \th < \al$ for all $\th \in \Phi$, or
  \item $\ph = \forall x \psi$, where $\rk \psi < \al$.
\end{enumerate}
In the first two cases, define $\ph^\si = \neg \psi^\si$ and $\ph^\si =
\bigwedge_{\th \in \Phi} \th^\si$, respectively. In the third case, define
$\ph^\si = \forall x \psi^\tau$, where $\tau: \Var \to \cT$ is the substitution
defined by $x^\tau = x$ and $y^\tau = y^\si$ whenever $y \ne x$. By transfinite
recursion on $\rk \ph$, this defines $\si \mapsto \ph^\si$ for all $\ph \in
\cL$.
  \symindex{$\ph(t/x)$}%

If $x \in \Var$ and $t \in \cT$, we use the notation $t/x$ to denote the
substitution $\si: \Var \to \cT$ defined by $x^\si = t$ and $y^\si = y$ for $y
\ne x$. We read $t/x$ as ``$t$ for $x$'' and write $\ph^\si$ as $\ph(t/x)$. This
is extended in the natural way for $\vec x = \ang{x_1, \ldots, x_n}$ and $\vec t
= \ang{t_1, \ldots, t_n}$.

Note that, in general, $\ph(t_1 t_2 / x_1 x_2) \ne \ph(t_1/x_1)(t_2/x_2)$. For
example, if $\ph = x_1 < x_2$, then $\ph(x_2 x_1 / x_1 x_2) = x_2 < x_1$, but
$\ph(x_2/x_1)(x_1/x_2) = x_1 < x_1$.

\begin{prop}\label{P:var-sub-term}
  Let $s, t \in \cT$ and $x \in \Var$. Then
  \[
    \var s(t/x) \subseteq (\var s \setminus \{x\}) \cup \var t.
  \]
\end{prop}

\begin{proof}
  We prove this by induction on $s$. If $s = x$, then $s(t/x) = t$, so $\var s
  (t/x) = \var t$, and the result holds. If $s = y$, where $y \ne x$, or $s =
  c$, then $s(t/x) = s$ and $x \notin \var s$. Hence, $\var s(t/x) = \var s =
  \var s \setminus \{x\}$, and the result holds.

  Now suppose the result holds for $t_1, \ldots, t_n$, and let $s = f t_1
  \cdots t_n$. Let $t_i' = t_i(t/x)$, so that $\var t_i' \subseteq (\var t_i
  \setminus \{x\}) \cup \var t$ and $s(t/x) = f t_1' \cdots t_n'$. Then
  \begin{align*}
    \var s(t/x) &= \ts{\bigcup_{i = 1}^n} \var t_i'\\
    &\subseteq \ts{\bigcup_{i = 1}^n} (\var t_i \setminus \{x\}) \cup \var t\\
    &= ((\ts{\bigcup_{i = 1}^n} \var t_i) \setminus \{x\}) \cup \var t\\
    &= (\var s \setminus \{x\}) \cup \var t,
  \end{align*}
  and the result holds.
\end{proof}

\begin{prop}\label{P:free-sub-form}
  Let $\ph \in \cL$, $x \in \Var$, and $t \in \cT$. Then
  \[
    \free \ph(t/x) \subseteq (\free \ph \setminus \{x\}) \cup \var t.
  \]
\end{prop}

\begin{proof}
  We prove this by induction on $\ph$. An argument like the one in the proof of
  Proposition \ref{P:var-sub-term} covers the cases where $\ph$ is prime, $\ph
  = \neg \psi$, and $\ph = \bigwedge \Phi$. Suppose $\ph = \forall x \psi$. Then
  $\free \ph = \free \psi \setminus \{x\}$. In particular, $x \notin \free \ph$.
  Moreover, $\ph(t/x) = \ph$, so that $\free \ph(t/x) = \free \ph = \free \ph
  \setminus \{x\}$, and the result holds.

  Now suppose $\ph = \forall y \psi$, where $y \ne x$. Then $\ph(t/x) = \forall
  y \, \psi(t/x)$, and $\free \ph(t/x) = \free \psi(t/x) \setminus \{y\}$.
  Assume $z \in \free \ph (t/x)$, but $z \notin (\free \ph \setminus \{x\}) \cup
  \var t$. Then $z \in \free \psi(t/x)$ and $z \ne y$. By the inductive
  hypothesis, the result holds for $\psi$. Hence, $z \in (\free \psi \setminus
  \{x\}) \cup \var t$. But $z \notin \var t$. Thus, $z \in \free \psi$ and $z
  \ne x$. Since $\free \ph = \free \psi \setminus \{y\}$, we have $z \in \free
  \ph$. It follows that $z \in \free \ph \setminus \{x\}$, a contradiction.
\end{proof}

\begin{cor}
  If $\ph(x) \in \cL$ and $t$ is a ground term, then $\ph(t) \in \cL^0$.
\end{cor}

\begin{proof}
  Let $\ph(x) \in \cL$ and let $t$ be a ground term. Then $\free \ph \subseteq 
  \{x\}$ and $\var t = \emp$. By Proposition \ref{P:free-sub-form}, we have
  $\free \ph(t) = \emp$, so that $\ph(t)$ is a sentence.
\end{proof}

Using substitutions, we introduce the shorthand,
\[
  \exists! x \ph = {
    \exists x \ph \wedge \forall xy (\ph \wedge \ph(y/x) \to x \beq y)
  },
\]
where $y \notin \var \ph$.

\section{Predicate calculus}\label{S:pred-calc}

In this section, we define both the deductive and inductive derivability
relations. As in the propositional case, we will denote them both by $\vdash$.
We begin with the deductive case. As described in Section \ref{S:pred-theories},
the inductive case will require no modification from its presentation in Chapter
\ref{Ch:prop-calc}.

For deductive derivability, we wish to define a relation $\vdash$ from $\fP
\cL^0$ to $\cL^0$ such that $X \vdash \ph$ captures what it means to say a
sentence $\ph$ can be logically deduced from the sentences in $X$. Our aim here
is to do this through natural deduction, as we did in Section
\ref{S:prop-nat-ded} for the propositional language $\cF$. We will keep all the
rules in Definition \ref{D:derivability}, and add two rules each for $\forall$
and $\beq$. Ideally, we would like our new rules to be the following:
\begin{enumerate}[1.]
  \item if $X \vdash \forall x \ph(x)$ and $t$ is a ground term, then $X
        \vdash \ph(t)$,
  \item if $c \notin \con(X \cup \ph(x))$ and $X \vdash \ph(c)$, then $X
        \vdash \forall x \ph(x)$,
  \item $\vdash t \beq t$ for all ground terms $t$, and
  \item if $X \vdash s \beq t, \ph(s)$, then $X \vdash \ph(t)$.
\end{enumerate}
The problem with these rules is (2). We may not have enough constants in our
language to ensure there exists a $c \notin \con(X \cup \ph(x))$. As shown later
in Proposition \ref{P:add-constants}, we can always add constants to our
language without affecting derivability. But until that can be established,
these rules will not be easy to work with.

We therefore take a slightly different approach. We define the derivability
relation from $\fP \cL$ to $\cL$. That is, we allow ourselves to use open
formulas in our derivations. In an open formula, we will treat free variables
like constants. In this way, every language will effectively have an uncountable
number of constants available for use in (2). This still doesn't fully resolve
the problem, since $X$ itself can be uncountable. But in Theorem
\ref{T:pred-sig-cpctness}, we will prove $\si$-compactness, so that we need only
consider countable $X$.

Reasoning with sentences, however, is our primary concern in the bulk of what
we want to do. Both deductive and inductive theories consist entirely of
sentences. As such, after presenting our system of natural deduction and proving
$\si$-compactness, we will look at expanding our language by adding additional
constants. This will give us two ways to connect our natural deduction to
reasoning with sentences. These are presented in Propositions 
\ref{P:pred-derivability} and \ref{P:free-elim-deriv}.

We then define deductive and inductive theories, and finally finish the section
with a presentation of a Hilbert-type calculus. This is the calculus used by
Karp in \cite{Karp1964}, and we will need it in order to apply her completeness
result in the setting of predicate logic.

By allowing open formulas, however, we introduce a new problem. Suppose $\cL =
\cL\{\circ, e\}$, where $\circ$ is a binary operation symbol and $e$ is a
constant symbol. Let $\ph(x) = \exists y \, x \nbeq y$ and $t = y \circ e$. Then
$\ph(t) = \exists y \, y \circ e \nbeq y$. If we interpret these formulas in
group theory, where $e$ is the group identity and $\circ$ is the group
operation, then $\forall x \ph(x) = \forall x \exists y \, x \nbeq y$ is a true
sentence in every group that has more than one element. And yet, the sentence
$\ph (t)$ is always false in that context. Hence, $\forall x \ph(x)$ cannot
logically imply $\ph(t)$. This would not violate (1), because $t$ is not a
ground term. But when we remove that restriction on $t$, this will become a
problem. The issue is that the variable $y \in \var t$ is a bound variable in
$\ph$, so after the substitution, it becomes bound. If free variables are to be
treated as constants, then variables inside terms must become free after a
substitution. To ensure this, we will need to avoid substitutions that
``collide'' with bound variables.

\subsection{Free substitutions}

Let $\ph \in \cL$, $\ze \in \Sf \ph$, and $x \in \Var$. We say that \emph{$\ze$
is in the scope of $\forall x$ in $\ph$}
  \index{scope}%
if there exists $\psi$ such that $\forall x \psi \in \Sf \ph$ and $\ze \in \Sf
\psi$. We say that \emph{$\ze$ has a free occurrence of $x$ in $\ph$} if $x \in
\free \ze$ and $\ze$ is not in the scope of $\forall x$ in $\ph$.
  \index{free occurrence}%
Note that if $\ph' \in \Sf \ph$, $\ze \in \Sf \ph'$, and $\ze$ has a free
occurrence of $x$ in $\ph$, then $\ze$ has a free occurrence of $x$ in $\ph'$.

\begin{prop}\label{P:free-occurrence}
  If $\ze$ has a free occurrence of $x$ in $\ph$, then $x \in \free \ph$.
\end{prop}

\begin{proof}
  If $\ze = \ph$, the result is immediate, so assume $\ze \ne \ph$. Then $\ph$
  is not prime, so we may write $\ph = \neg \ph'$, $\ph = \bigwedge \Phi$, or
  $\ph = \forall y \ph'$. By induction on $\ph$, we may assume the result is
  true for $\ph'$ and for all $\th \in \Phi$.

  In the first case, $\ze \in \Sf \ph'$, so $\ze$ has a free occurrence of $x$ in
  $\ph'$. Hence, $x \in \free \ph' = \free \ph$. In the second case, $\ze \in
  \Sf \th$ for some $\th \in \Phi$. Thus, $x \in \free \th \subseteq \free \ph$.
  Similarly, in the third case, we get $x \in \free \ph'$. But $\ze$ is not in
  the scope of $\forall x$ in $\ph$, so $y \ne x$. Hence, $\free \ph = \free
  \ph'$.
\end{proof}

Let $x, y \in \Var$ and $\ph \in \cL$. We say that \emph{$y$ is not free for $x$
in $\ph$} if there exists $\ze \in \Sf \ph$ such $\ze$ is in the scope of
$\forall y$ in $\ph$ and $\ze$ has a free occurrence of $x$ in $\ph$. Otherwise,
\emph{$y$ is free for $x$ in $\ph$}. For $t \in \cT$, we way that \emph{$t$ is
free for $x$ in $\ph$} if $y$ is free for $x$ in $\ph$ for all $y \in \var t$.
More generally, a substitution \emph{$\si$ is free for $\ph$} if $x^\si$ is free
for $x$ in $\ph$, for all $x \in \Var$.
  \index{substitution!free ---}%

\begin{prop}
  If $\bnd \ph \cap (\var t \setminus \{x\}) = \emp$, then $t$ is free for $x$
  in $\ph$. In particular, $y$ is free for $x$ in $\ph$ if $y = x$ or $y \notin
  \bnd \ph$.
\end{prop}

\begin{proof}
  Suppose $t$ is not free for $x$ in $\ph$. Then there exists $y \in \var t$
  such that a free occurrence of $x$ occurs inside the scope of $\forall y$. In
  particular, we must have $y \ne x$ and $y \in \bnd \ph$, so that $y \in \bnd
  \ph \cap (\var t \setminus \{x\})$.
\end{proof}

\begin{prop}\label{P:sub-free-only}
  Let $y \notin \var \ph$ and $\ze \in \Sf \ph(y/x)$. If $\ze$ has a free
  occurrence of $y$ in $\ph(y/x)$, then $\ze$ is not in the scope of $\forall
  x$ in $\ph(y/x)$.
\end{prop}

\begin{proof}
  The proof is by induction on $\ph$ and follows the same lines as the proof of
  Proposition \ref{P:free-occurrence}.
\end{proof}

\begin{cor}\label{C:free-sub-only}
  If $y \notin \var \ph$, then $x$ is free for $y$ in $\ph(y/x)$.
\end{cor}

\begin{proof}
  Suppose $x$ is not free for $y$ in $\ph(y/x)$. Then there exists $\ze \in \Sf
  \ph(y/x)$ such that $\ze$ is in the scope of $\forall x$ in $\ph(y/x)$ and
  $\ze$ has a free occurrence of $y$ in $\ph(y/x)$. But this contradicts
  Proposition \ref{P:sub-free-only}.
\end{proof}

\subsection{Natural deduction}

\begin{defn}\label{D:pred-derivability}
  The \emph{derivability relation}, denoted by $\vdash$ or $\vdash_\cL$,
    \index{derivability relation!predicate ---}%
    \symindex{$X \vdash \ph$ (in $\cL$)}%
  is the smallest relation from $\fP \cL$ to $\cL$ satisfying (i)--(vi) in
  Definition \ref{D:derivability}, as well as the following:
  \begin{enumerate}[(i), leftmargin=27pt]
    \setcounter{enumi}{6}
    \item if  $X \vdash \forall x \ph$, then $X \vdash \ph(t/x)$ when $t$ is
          free for $x$ in $\ph$,
    \item if $x \notin \free X$ and $X \vdash \ph$,
          then $X \vdash \forall x \ph$,
    \item $\vdash t \beq t$ for all $t \in \cT$, and
    \item if $X \vdash s \beq t, \ph(s/x)$, then $X \vdash \ph(t/x)$ when $s$
          and $t$ are free for $x$ in $\ph$.
  \end{enumerate}
  Since $x$ is always free for $x$ in $\ph$, (vii) implies
  \begin{enumerate}[(i)$'$, leftmargin=27pt]
    \setcounter{enumi}{6}
    \item if $X \vdash \forall x \ph$, then $X \vdash \ph$.
  \end{enumerate}
\end{defn}

\begin{rmk}\label{R:mon}
  If $X \vdash_{\cL} \ph$ and $\cL \subseteq \cL'$, then $X \vdash_{\cL'} \ph$.
  To see this, let $\cL \subseteq \cL'$ and define ${\vdash'} = {\vdash_{\cL'}}
  \cap (\fP \cL \times \cL)$. Since $\vdash'$ satisfies (i)--(x) for $\cL$, we
  have $ {\vdash_\cL} \subseteq {\vdash'}$.
\end{rmk}

\begin{rmk}\label{R:pred-fin-vs-infin}
  The \emph{finitary derivability relation} is the smallest relation
  $\vdash_\fin$ from $\fP \cL_\fin$ to $\cL_\fin$ such that conditions (i)--(x)
  from Definition \ref{D:pred-derivability} hold, with the exception that in
  (iii) and (iv), we require $\Phi$ to be finite.
    \symindex{$X \vdash_\fin \ph$ (in $\cL$)}%
  The finitary derivability relation is a typical natural-deduction calculus for
  first-order logic. Clearly, ${\vdash_\fin} \subseteq {\vdash}$. As we will see
  in Proposition \ref{P:pred-fin-vs-infin}, if $X \subseteq \cL_\fin$, $\ph \in
  \cL_\fin$, and $X \vdash \ph$, then $X \vdash_\fin \ph$. In other words, when
  restricted to finitary formulas, infinitary calculus cannot produce any new
  inferences beyond those already available in first-order logic.
\end{rmk}

The proof of Proposition \ref{P:derivability}, which is based on (i)--(vi), is
still valid here. Throughout the rest of this chapter, unless otherwise
indicated, we will use lowercase Roman numerals refer to Definition
\ref{D:pred-derivability} and letters refer to Proposition \ref{P:derivability}.

\begin{prop}[Bound renaming]\label{P:bnd-rename}
    \index{bound renaming}%
  For any $\ph \in \cL$ and $y \notin \var \ph$, we have $\forall x \ph \vdash
  \forall y \, \ph(y/x)$ and $\forall y \, \ph(y/x) \vdash \forall x \ph$.
\end{prop}

\begin{proof}
  If $y = x$, the result follows from (i). Assume, then, that $y \ne x$. Since
  $y \notin \var \ph$, it follows from (vii) that $\forall x \ph \vdash \ph
  (y/x)$. Hence, by (viii), we have $\forall x \ph \vdash \forall y \, \ph
  (y/x)$.

  Let $\ph' = \ph(y/x)$. Corollary \ref{C:free-sub-only} implies that $x$ is
  free for $y$ in $\ph'$. Hence, by (vii), we have $\forall y \ph' \vdash \ph'
  (x/y) = \ph$. But $x \notin \free \forall y \ph'$, so (viii) implies $\forall
  y \ph' \vdash \forall x \ph$.
\end{proof}

\begin{rmk}
  Bound renaming gives us the following alternate to (viii):
  \begin{enumerate}[(i)$'$, leftmargin=3em]
    \setcounter{enumi}{7}
    \item if $y \notin \free X \cup \var \ph$ and $X \vdash \ph(y/x)$, then $X
          \vdash \forall x \ph$,
  \end{enumerate}
  To see this, suppose $y \notin \free X \cup \var \ph$ and $X \vdash \ph(y/x)$.
  Then (viii) implies $X \vdash \forall y \, \ph(y/x)$ and Proposition 
  \ref{P:bnd-rename} gives $\forall y \, \ph(y/x) \vdash \forall x \ph$.
\end{rmk}

We now prove $\si$-compactness. In the predicate case, the theorem is stronger,
in that we can not only pass to a countable subset of formulas. We can also pass
to a countable subset of extralogical symbols. We begin with the basic version,
which is the analogue of the propositional version.

\begin{prop}
  Let $X \subseteq \cL$ and $\ph \in \cL$. Then $X \vdash \ph$ if and only if
  there exists a countable subset $X_0 \subseteq X$ such that $X_0 \vdash \ph$.
\end{prop}

\begin{proof}
  As in the proof of Theorem \ref{T:sig-cpctness}, we will prove that the
  following are equivalent:
  \begin{align}
    &X \vdash \ph,\label{pred-sig-cpctness-1}\\
    &\text{there exists countable $X_0 \subseteq X$ such that
      $\ts{\bigwedge} X_0 \vdash \ph$, and}\label{pred-sig-cpctness-2}\\
    &\text{there exists countable $X_0 \subseteq X$ such that
      $X_0 \vdash \ph$}.\label{pred-sig-cpctness-3}
  \end{align}
  The proof of Theorem \ref{T:sig-cpctness} carries through to show that 
  \eqref{pred-sig-cpctness-2} implies \eqref{pred-sig-cpctness-3}, and 
  \eqref{pred-sig-cpctness-3} implies \eqref{pred-sig-cpctness-1}. Define
  $\vdash'$ so that $X \vdash' \ph$ if and only if $X \vdash \ph$ and
  \eqref{pred-sig-cpctness-2} holds. Since Lemma \ref{L:conj-subset} is still
  valid for $\vdash$, the proof of Theorem \ref{T:sig-cpctness} shows that (i)--%
  (iv) hold for $\vdash'$. It is straightforward to verify that $\vdash'$
  satisfies (vii)--(x).
\end{proof}

\begin{thm}[$\si$-compactness]\label{T:pred-sig-cpctness}
    \index{s_sigma-compactness@$\si$-compactness}%
  Let $X \subseteq \cL$ and $\ph \in \cL$. Then $X \vdash_\cL \ph$ if and only
  if there exist countable $X_0 \subseteq X$ and $L_0 \subseteq L$ such that
  $X_0 \vdash_{\cL_0} \ph$.
\end{thm}

\begin{proof}
  The if direction follows from (ii) and Remark \ref{R:mon}. For the only if
  direction, define ${\vdash'} \subseteq \fP \cL \times \cL$ by $X \vdash' \ph$
  if $X_0 \vdash_{\cL_0} \ph$ for some countable $X_0 \subseteq X$ and $L_0
  \subseteq L$. It suffices to show that $\vdash'$ satisfies (i)--(x).

  Suppose $\ph \in \cL$. Let $L_0 = \sym \ph$. Then $L_0 \subseteq L$ is
  countable, and by (i) for $\cL_0$, we have $\ph \vdash_{\cL_0} \ph$. Thus,
  $\ph \vdash' \ph$, and $\vdash'$ satisfies (i).

  Let $\Phi \subseteq \cL$ be countable and suppose $X \vdash' \th$ for all
  $\th \in \Phi$. For each $\th \in \Phi$, choose countable $X_\th \subseteq X$
  and $L_\th \subseteq L$ such that $X_\th \vdash_{\cL_\th} \th$. Let $X_0 =
  \bigcup_{\th \in \Phi} X_\th$ and $L_0 = \bigcup_{\th \in \Phi} L_\th$, both
  of which are countable. Remark \ref{R:mon} implies $X_\th \vdash_{\cL_0} \th$
  for all $\th \in \Phi$. Hence, by (ii) for $\cL_0$, we have $X_0 \vdash_
  {\cL_0} \th$ for all $\th \in \Phi$. Therefore, by (iv) for $\cL_0$, it
  follows that $X_0 \vdash_{\cL_0} \bigwedge \Phi$. Thus, $X \vdash' \bigwedge
  \Phi$, and $\vdash'$ satisfies (iv).

  The proofs of (ii), (iii), and (v)--(x) are similar.
\end{proof}

Since they are based on (i)--(vi), Propositions \ref{P:conj-equiv} and 
\ref{P:set-mono} hold here as well.

We finish this subsection with a result that we will need later. A
\emph{variable permutation} is a bijection $\pi: \Var \to \Var$.
  \index{permutation!variable ---}%
We extend a variable permutation to $\pi: \cT \to \cT$ by $c^\pi = c$ and $(f
t_1 \cdots t_n)^\pi = f t_1^\pi \cdots t_n^\pi$. For $\ph \in \cL$, we define
$\ph^\pi$ by $(s \beq t)^\pi = (s^\pi \beq t^\pi)$, $(r t_1 \cdots t_n)^\pi = r
t_1^\pi \cdots t_n^\pi$, $(\neg \ph)^\pi = \neg \ph^\pi$, $(\bigwedge \Phi)^\pi
= \bigwedge_{\ph \in \Phi} \ph^\pi$, and $(\forall x \ph)^\pi = \forall x^\pi
\ph^\pi$.

\begin{prop}\label{P:perm-vars}
  If $\pi$ is a variable permutation, then $X \vdash \ph$ if and only if $X^\pi
  \vdash \ph^\pi$.
\end{prop}

\begin{proof}
  Define $\vdash'$ by $X \vdash' \ph$ if $X \vdash \ph$ and $X^\pi \vdash
  \ph^\pi$. It is straightforward to verify that $\vdash'$ satisfies (i)--(x).
  Hence, $X \vdash \ph$ implies $X^\pi \vdash \ph^\pi$. Applying this result to
  $\pi^{-1}$ gives the converse.
\end{proof}

\subsection{Constant expansions}

We now connect our deductive system back to sentences. To do this, we will need
to expand our language by adding additional constant symbols.

Let $L$ be an extralogical signature with corresponding language $\cL$. Let $C$
be a set of constant symbols. Then $\cL C$ denotes the language corresponding to
the extralogical signature $L \cup C$.
  \symindex{$\cL C$}%
If $C = \{c\}$, then we write $\cL c$ for $\cL C$. The language $\cL C$ is
called a \emph{constant expansion} of $\cL$.
  \index{}

Let $C$ be a countable set of constant symbols. A \emph{$C$-substitution} is an
injective function $\ul\si: C \to \Var$.
  \index{C-substitution@$C$-substitution}%
Given a $C$-substitution, we extend it
to $\ul\si: \cT_{\cL C} \to \Var$ by $x^{\ul\si} = x$, $c^{\ul\si} = c$ for $c
\notin C$, and $(f t_1 \cdots t_n)^{\ul\si} = f t_1^{\ul\si} \cdots t_n^
{\ul\si}$. Define $\ph^{\ul\si}$ recursively by $(s \beq t)^{\ul\si} =
(s^{\ul\si} \beq t^ {\ul\si})$, $(r t_1 \cdots t_n)^{\ul\si} = r t_1^{\ul\si}
\cdots t_n^{\ul\si}$, $(\neg \ph)^{\ul\si} = \neg \ph^{\ul\si}$, $(\bigwedge
\Phi)^{\ul\si} = \bigwedge_{\ph \in \Phi} \ph^{\ul\si}$, and $(\forall x \ph)^
{\ul\si} = \forall x \ph^{\ul\si}$. Intuitively, $t^{\ul\si}$ and $\ph^{\ul\si}$
are obtained from $t$ and $\ph$, respectively, by replacing each occurrence of
$c \in C$ with $c^{\ul\si}$. For $X \subseteq \cL$, we let $X^{\ul\si} = \{\ph^
{\ul\si} \mid \ph \in X\}$. Note that if $X \subseteq \cL C$, then $X^{\ul\si}
\subseteq \cL$. For $V \subseteq \Var$, we say that \emph{$\ul\si$ avoids $V$}
if $c^{\ul\si} \notin V$ for all $c \in C$. Since $\Var$ is uncountable, given
any countable $V_0 \subseteq \Var$, there exists a $C$-substitution that avoids
$V_0$.  By term induction, it is easy to see that $s(t/x)^{\ul\si} =
s^{\ul\si}(t^ {\ul\si}/x)$ when $\ul\si$ avoids $x$. Using this and formula
induction, we have $\ph(t/x)^{\ul\si} = \ph^{\ul\si}(t^{\ul\si}/x)$.

If $c$ is a constant symbol and $y \in \Var$, we use the notation $y/c$ to
denote the $\{c\}$-substitution $\ul\si$ defined by $c^{\ul\si} = y$. We read
$y/c$ as ``$y$ for $c$'' and write $\ph^{\ul\si}$ as $\ph(y/c)$. In this case,
the identity $\ph(t/x)^{\ul\si} = \ph^{\ul\si}(t^{\ul\si}/x)$ becomes
$\ph(t/x)(y/c) = \ph'(t'/x)$, where $t' = t(y/c)$ and $\ph' = \ph(y/c)$. This is
extended in the natural way for $\vec c = \ang{c_1, \ldots, c_n}$ and $\vec x =
\ang{x_1, \ldots, x_n}$, provided the $c_i$'s and the $x_i$'s are distinct.

\begin{prop}\label{P:C-sub-taut}
  If $\ul\si$ and $\ul\si'$ both avoid $\var \ph$, then there is a variable
  permutation such that $\ph^{\ul\si'} = (\ph^{\ul\si})^\pi$. In particular,
  $\vdash_\cL \ph^{\ul\si}$ if and only if $\vdash_\cL \ph^{\ul\si'}$.
\end{prop}

\begin{proof}
  Since $C$ is countable and both $\ul\si$ and $\ul\si'$ are injective, we may
  choose a bijection $\pi: \Var \to \Var$ such that $\pi \circ \ul\si =
  \ul\si'$ and $x^\pi = x$ for all $x$ outside the ranges of $\ul\si$ and
  $\ul\si'$. We prove that $\ph^{\ul\si'} = (\ph^{\ul\si})^\pi$ by induction on
  $\ph$. The proof is entirely straightforward except for the inductive step
  $\ph = \forall x \psi$. For this, let $\ph = \forall x \psi$ and assume
  $\ul\si$ and $\ul\si'$ both avoid $\var \ph$. Then $\ul\si$ and $\ul\si'$ both
  avoid $\var \psi$, so by the inductive hypothesis, $\psi^{\ul\si'} =
  (\psi^{\ul\si})^\pi$. Since $\ul\si$ and $\ul\si'$ avoid $\var \ph$, we have
  $x^\pi = x$. Thus,
  \[
    (\ph^{\ul\si})^\pi = (\forall x \psi^{\ul\si})^\pi
      = \forall x^\pi (\psi^{\ul\si})^\pi
      = \forall x \psi^{\ul\si'}
      = \ph^{\ul\si'}.
  \]
  For the final result, we apply Proposition \ref{P:perm-vars}.
\end{proof}

\begin{prop}\label{P:C-sub-deriv}
  Let $X \subseteq \cL C$ and $\ph \in \cL C$. If $\ul\si$ and $\ul\si'$ are
  $C$-substitutions that both avoid $\var X \cup \{\ph\}$, then $X^{\ul\si}
  \vdash_\cL \ph^{\ul\si}$ if and only if $X^{\ul\si'} \vdash_\cL \ph^
  {\ul\si'}$.
\end{prop}

\begin{proof}
  Suppose $X^{\ul\si} \vdash_\cL \ph^{\ul\si}$. Choose countable $X_0 \subseteq
  X$ such that $X_0^{\ul\si} \vdash_\cL \ph^{\ul\si}$. Let $\psi = \bigwedge
  X_0$. Since $\psi^{\ul\si} = \bigwedge_{\th \in X_0} \th^{\ul\si}$ and $(\psi
  \to \ph)^{\ul\si} = \psi^{\ul\si} \to \ph^{\ul\si}$, we have $\vdash_\cL (\psi
  \to \ph)^{\ul\si}$. But $\ul\si$ and $\ul\si'$ both avoid $\var (\psi \to
  \ph)$. Hence, by Proposition \ref{P:C-sub-taut}, we have $\vdash_\cL (\psi \to
  \ph)^ {\ul\si'}$, which implies $X^{\ul\si'} \vdash_\cL \ph^{\ul\si'}$.
  Reversing the roles of $\ul\si$ and $\ul\si'$ gives the converse.
\end{proof}

\begin{prop}[Constant elimination]\label{P:con-elim}
  Let $C$ be a countable set of constant symbols and suppose $X \vdash_{\cL C}
  \ph$. Then there exists a countable set $X_0 \subseteq X$ such that $X_0
  \vdash_{\cL C} \ph$ and $X_0^{\ul\si} \vdash_\cL \ph^{\ul\si}$ whenever
  $\ul\si$ is a $C$-substitution that avoids $\var X_0 \cup \{\ph\}$. In
  particular, if $\ul\si$ avoids $\var X \cup \{\ph\}$, then $X^{\ul\si}
  \vdash_\cL \ph^{\ul\si}$.
\end{prop}

\begin{proof}
  Define ${\vdash'} \subseteq \fP \cL C \times \cL C$ by $X \vdash' \ph$ if $X
  \vdash_{\cL C} \ph$ and there exists countable $X_0 \subseteq X$ such that $X_0
  \vdash_{\cL C} \ph$ and $X_0^{\ul\si} \vdash_\cL \ph^{\ul\si}$ whenever
  $\ul\si$ avoids $\var X_0 \cup \{\ph\}$. It suffices to show that $\vdash'$
  satisfies (i)--(x).

  It is immediate that (i) and (ii) hold. For (iii), suppose $X \vdash'
  \bigwedge \Phi$. Choose countable $X_0 \subseteq X$ such that $X_0 \vdash_{\cL
  C} \bigwedge \Phi$ and $X_0^{\ul\si} \vdash_\cL (\bigwedge \Phi)^{\ul\si} =
  \bigwedge_{\th \in \Phi} \th^{\ul\si}$ whenever $\ul\si$ avoids $\var X_0 \cup
  \{\bigwedge \Phi\} = \var X_0 \cup \Phi$. Now fix $\th \in \Phi$. Then $X_0
  \vdash_{\cL C} \th$ and $X_0^{\ul\si} \vdash_\cL \th^{\ul\si}$ whenever
  $\ul\si$ avoids $\var X_0 \cup \Phi$. Suppose $\ul\si$ avoids $\var X_0 \cup 
  \{\th\}$. Since $\var X_0 \cup \Phi$ is countable, we may choose $\ul\si'$
  that avoids $\var X_0 \cup \Phi$. Since $\ul\si$ and $\ul\si'$ both avoid
  $\var X_0 \cup \{\th\}$, the inductive hypothesis and Proposition 
  \ref{P:C-sub-deriv} give $X_0^{\ul\si} \vdash_\cL \th^{\ul\si}$. Therefore,
  $X \vdash' \th$ and $\vdash'$ satisfies (iii). The proofs for (iv)--(vi) are
  similar.

  For (vii), suppose $X \vdash' \forall x \ph$ and let $t$ be free for $x$ in
  $\ph$. Choose countable $X_0 \subseteq X$ such that $X_0 \vdash \forall x \ph$
  and $X_0^{\ul\si} \vdash (\forall x \ph)^{\ul\si} = \forall x \ph^ {\ul\si}$
  whenever $\ul\si$ avoids $\var X \cup \{\forall x \ph\}$. Now let $\ul\si$
  avoid $\var X \cup \{\ph(t/x)\}$. As above, by Proposition
  \ref{P:C-sub-deriv}, we may assume $\ul\si$ also avoids both $\var X_0 \cup
  \{\forall x \ph\}$ and $\var t$. Then by hypothesis, $X_0^{\ul\si} \vdash 
  (\forall x \ph)^{\ul\si} = \forall x \ph^ {\ul\si}$.

  We now show that $t^{\ul\si}$ is free for $x$ in $\ph^{\ul\si}$. Suppose not.
  Then there exists $y \in \var t^{\ul\si}$ such that a free occurrence of $x$
  in $\ph^{\ul\si}$ occurs within the scope of $\forall y$ in $\ph^{\ul\si}$.
  But $\ul\si$ avoids $x$, so every occurrence of $x$ in $\ph^{\ul\si}$ is an
  occurrence of $x$ in $\ph$. Also, by the definition of $C$-substitutions, any
  occurrence of the quantifier $\forall y$ in $\ph^{\ul\si}$ is an occurrence of
  $\forall y$ in $\ph$. Hence, there is a free occurrence of $x$ in $\ph$ that
  occurs within the scope of $\forall y$ in $\ph$. Since $t$ is free for $x$ in
  $\ph$, we must have $y \notin \var t$. On the other hand, since $\forall y$
  occurs in $\ph$, we have $y \in \var \ph$. Therefore, $\ul\si$ avoids $y$. It
  follows that $y \notin \var t^{\ul\si}$, a contradiction.

  Since $t^{\ul\si}$ is free for $x$ in $\ph^{\ul\si}$ and $X_0^{\ul\si} \vdash 
  \forall x \ph^{\ul\si}$, it follows from (vii) that $X_0^{\ul\si} \vdash \ph^
  {\ul\si}(t^{\ul\si}/x)$. But $\ph^{\ul\si}(t^{\ul\si}/x) = \ph(t/x)^{\ul\si}$,
  so that $X_0^{\ul\si} \vdash \ph(t/x)^{\ul\si}$, showing that $\vdash'$
  satisfies (vii). The proofs of (viii)--(x) are the similar.
\end{proof}

\begin{prop}\label{P:add-constants}
  Let $\cL C$ be a constant expansion of $\cL$. Let $X \subseteq \cL$ and $\ph
  \in \cL$. Then $X \vdash_\cL \ph$ if and only if $X \vdash_{\cL C} \ph$.
\end{prop}

\begin{proof}
  The only if direction follows from Remark \ref{R:mon}. For the if direction,
  suppose $X \vdash_{\cL C} \ph$. By Theorem \ref{T:pred-sig-cpctness}, we may
  choose countable $X_0 \subseteq X$ and $L_0 \subseteq L \cup C$ such that $X_0
  \vdash_{\cL_0} \ph$. Let $C_0 = (L_0 \cap C) \setminus L$ denote the constant
  symbols that are in $L_0$ but not in $L$. Then $C_0$ is countable and $\cL_0
  \subseteq \cL C_0$. Hence, by Remark \ref{R:mon}, we have $X_0 \vdash_ {\cL
  C_0} \ph$. Let $\ul\si$ be a $C_0$-substitution that avoids $\var X_0 \cup 
  \{\ph\}$. Proposition \ref{P:con-elim} gives us $X_0^{\ul\si} \vdash_\cL
  \ph^{\ul\si}$. But $\sym \ph \subseteq L$, so $C_0 \cap \sym \ph = \emp$. In
  other words, none of the constants in $C_0$ appear in $\ph$. By induction
  on $\ph$, it follows that $\ph^{\ul\si} = \ph$. Similarly, $X_0^{\ul\si} =
  X_0$. Therefore, $X_0 \vdash_\cL \ph$, which gives $X \vdash_\cL \ph$.
\end{proof}

\subsection{Deduction with sentences}

Using constant expansions, we can connect derivability back to sentences in two
ways. The first is to verify the four rules given at the beginning of this
section.

\begin{prop}\label{P:pred-derivability}
  The derivability relation in $\cL$, when restricted to sentences, satisfies 
  (i)--(vi) in Definition \ref{D:derivability}, as well as the following:
  \begin{enumerate}[(i)$^0$, leftmargin=4em]
    \setcounter{enumi}{6}
    \item if $X \vdash \forall x \ph(x)$ and $t$ is a ground term, then $X
          \vdash \ph(t)$,
    \item if $c \notin \con(X \cup \ph(x))$ and $X \vdash_{\cL c} \ph(c)$, then
          $X \vdash \forall x \ph(x)$, and
    \item $\vdash t \beq t$ for all ground terms $t$,
    \item if $X \vdash s \beq t, \ph(s)$, then $X \vdash \ph(t)$.
  \end{enumerate}
\end{prop}

\begin{proof}
  Since $\vdash$ satisfies (i)--(vi) and $\cL^0$ is closed under negation and
  conjunction, $\vdash$ still satisfies (i)--(vi) when restricted to sentences.
  Suppose $X \vdash \forall x \ph(x)$ and $t$ is a ground term. Then $\ph(t)$ is
  a sentence, and (vii) implies (vii)$^0$.

  For (viii)$^0$, suppose $c \notin \con(X \cup \ph(x))$ and $X \vdash_{\cL c}
  \ph(c)$. Since $\forall x \ph(x)$ is a sentence, we only need to show that $X
  \vdash \forall x \ph$. By Theorem \ref{T:pred-sig-cpctness}, we may choose
  countable $X_0 \subseteq X$ such that $X_0 \vdash_{\cL c} \ph(c)$. Choose $y
  \notin \var X_0 \cup \{\ph\}$. Then the $\{c\}$-substitution $y/c$ avoids
  $\var X_0 \cup \{\ph (c)\}$. Hence, by Proposition \ref{P:con-elim}, we have
  $X_0(y/c) \vdash \ph (c)(y/c)$. But $\ph(c)(y/c) = \ph'(c')$, where $c' =
  c(y/c) = y$ and $\ph' = \ph(y/c) = \ph$, since $c \notin \con \ph$. Thus,
  $\ph(c)(y/c) = \ph$. Similarly, since $c \notin \con X_0$, we also have
  $X_0(y/c) = X_0$. Therefore, $X_0 \vdash \ph(y)$. Now $y \notin \free X_0 \cup
  \var \ph$. Hence, from (viii)$'$, it follows that $X_0 \vdash \forall x \ph$,
  which gives $X \vdash \forall x \ph$.

  If $t$ is a ground term, then $t \beq t$ is a sentence, so (ix) implies
  (ix)$^0$. Now suppose $X \vdash s \beq t, \ph(s)$. Then $s \beq t$ is a
  sentence, which implies $s$ and $t$ are ground terms. Thus, $\ph(t)$ is a
  sentence, and (x) implies (x)$^0$.
\end{proof}

The second way to connect derivability to sentences is to replace the free
variables in open formulas with constants. For each $x \in \Var$, choose a
distinct constant symbol $c_x$ that is not already in $L$. Let $C = \{c_x \mid x
\in \Var\}$. A \emph{free eliminator} is a substitution $\si: \Var \to \cT_{\cL
C}$ such that for all $x \in \Var$, either $x^\si = x$ or $x^\si = c_x$.
  \index{free eliminator}%
If $x^\si = c_x$ for all $x \in \Var$, then $\si$ is called a \emph{full free
eliminator}. Note that if $\si$ is a full free eliminator, then $\ph^\si$ is a
sentence for all $\ph \in \cL$.

\begin{prop}\label{P:free-elim-deriv}
  Let $X \subseteq \cL$ and $\ph \in \cL$, and let $\si$ be a free eliminator.
  Then $X \vdash_\cL \ph$ if and only if $X^\si \vdash_{\cL C} \ph^\si$.
\end{prop}

\begin{proof}
  Suppose $X \vdash_\cL \ph$. Let $V' = \{x \in \Var \mid x^\si \ne x\}$ and let
  $C' = \{c_x \mid x \in V'\}$. Then $X \vdash_{\cL C'} \ph$ by Proposition
  \ref{P:add-constants}. By Theorem \ref{T:pred-sig-cpctness}, there exists
  countable $C_0 \subseteq C'$ such that $X \vdash_{\cL C_0} \ph$. Define the
  $C_0$-substitution $\ul\si$ by $c_x^ {\ul\si} = x$, so that
  $(\psi^{\ul\si})^\si = \psi$ for all $\psi \in \cL C_0$, and $(\psi^\si)^
  {\ul\si} = \psi$ whenever $\psi \in \cL$ and $\psi^\si \in \cL C_0$. Then
  $(X^{\ul\si})^\si \vdash_{\cL C_0} (\ph^{\ul\si})^\si$. But $X \subseteq \cL$
  and $\ph \in \cL$, so $X^{\ul\si} = X$ and $\ph^ {\ul\si} = \ph$. Therefore,
  $X^\si \vdash_{\cL C_0} \ph^\si$, so Proposition \ref{P:add-constants} gives
  $X^\si \vdash_{\cL C} \ph^\si$.

  Now suppose $X^\si \vdash_{\cL C} \ph^\si$. Note that $\psi \in \cL$ implies
  $\psi^\si \in \cL C'$. Hence, Proposition \ref{P:add-constants} implies $X^\si
  \vdash_{\cL C'} \ph^\si$. By Theorem \ref{T:pred-sig-cpctness}, we may choose
  countable $X_0 \subseteq X$ and $C_0 \subseteq C'$ so that $X_0^\si \vdash_
  {\cL C_0} \ph^\si$. Let $V_0 = \{x \in V' \mid c_x \in C_0\}$ and note that
  $V_0$ is countable. By Proposition \ref{P:bnd-rename}, we may assume $\bnd
  X_0^\si \cup \{\ph^\si\}$ is disjoint from both $V_0$ and $\free X_0^\si \cup 
  \{\ph^\si\}$. We may then use Proposition \ref{P:perm-vars} to ensure $\free
  X_0^\si \cup \{\ph^\si\}$ is also disjoint from $V_0$. Now define $\ul\si$ as
  above. If $c_x \in C_0$, then $c_x^{\ul\si} = x \in V_0$. Therefore, $\ul\si$
  avoids $\var X_0^\si \cup \{\ph^\si\}$. Proposition \ref{P:con-elim} now gives
  $(X_0^\si)^{\ul\si} \vdash_\cL (\ph^\si)^{\ul\si}$. Since $\ph \in \cL$ and
  $\ph^\si \in \cL C_0$, we have $(\ph^\si)^{\ul\si} = \ph$. Similarly, $
  (X_0^\si)^{\ul\si} = X_0$. Therefore, $X_0 \vdash_\cL \ph$.
\end{proof}

\subsection{Tautologies and consistency}\label{S:taut-pred}

\begin{prop}\label{P:exists_1}
  For any $x \in \Var$, we have ${} \vdash \exists x \, x \beq x$.
\end{prop}

\begin{proof}
  Let $\ph(x) = (x \beq x)$. By (i) and (vii)$'$, we have $\forall x \neg
  \ph (x)
  \vdash \neg \ph(x)$. By (ix) and (ii), we have $\forall x \neg \ph(x) \vdash
  \ph(x)$. Hence, (v) implies $\forall x \neg \ph(x) \vdash \exists x \ph(x)$.
  Since $\exists x \ph(x) = \neg \forall x \neg \ph(x)$, we have $\neg \forall x
  \neg \ph(x) \vdash \exists x \ph(x)$, by (i). Therefore, the result follows
  from (vi).
\end{proof}

Let $\exists_1 = \exists \bx_0 \, \bx_0 \beq \bx_0$. Informally, $\exists_1$
says that at least one object exists. Justified by Proposition \ref{P:exists_1},
we define \emph{verum} and \emph{falsum}, respectively, by $\top = \exists_1$
and $\bot = \neg \top$.
  \index{verum}%
  \index{falsum}%
  \symindex{$\top$ (in $\cL$)}%
  \symindex{$\bot$ (in $\cL$)}%

More generally, for integers $n > 1$, we define
\[
  \ts{
    \exists_n = \exists \bx_0 \cdots \bx_{n - 1} \,
      \bigwedge_{i < j < n} \bx_i \not \beq \bx_j.
  }
\]
Informally, $\exists_n$ asserts that there are at least $n$ objects. Let
$\exists_{=n} = \exists_n \wedge \neg \exists_{n + 1}$, which asserts that there
are exactly $n$ objects. We also define $\exists_\infty = \bigwedge_{n =
1}^\infty \exists_n$, which say that infinitely many objects exist. This last
sentence is the only one that is not in $\cL_\fin$.

A set $X \subseteq \cL$ is \emph{inconsistent} if $X \vdash \ph$ for all $\ph
\in \cL$; it is otherwise \emph{consistent}. Note that $X$ is inconsistent if
and only if $X \vdash \bot$. Theorem \ref{T:deduc-con} is easily seen to hold
also for $\cL$.
  \index{inconsistent}%
  \index{consistent}%

A formula $\ph \in \cL$ is a \emph{tautology} if ${} \vdash \ph$; it is a
\emph{contradiction} if $\{\ph\}$ is inconsistent. Note that $\ph$ is a
tautology if and only if $\neg \ph$ is a contradiction, and vice versa.
As in Proposition \ref{P:sig-cpctness}, we have $X \vdash \ph$ if and only if
there exists a countable $X_0 \subseteq X$ such that $\vdash \bigwedge X_0 \to
\ph$.
  \index{tautology}%
  \index{contradiction}%

\subsection{Deductive and inductive theories}\label{S:pred-theories}

A set $T \subseteq \cL^0$ is a \emph{(deductive) theory} if $T \vdash \ph$
implies $\ph \in T$ for all $\ph \in \cL^0$. The intersection of any family of
theories is again a theory.
  \index{theory!deductive ---}%
Also, $\cL^0$ itself is a theory. Hence, if $X \subseteq \cL^0$, then we may
define \emph{the (deductive) theory generated by $X$}, denoted by $T(X)$ or
$T_X$, as the smallest theory having $X$ as a subset. Note that $T(X) = \{\ph
\in \cL^0 \mid X \vdash \ph\}$.
  \symindex{$T(X), T_X$ (in $\cL$)}%

We adopt the same notation for theories that we did in our propositional
calculus. The smallest theory is the set of tautological sentences, which we
denote by $\Taut$, or $\Taut_\cL$.
  \symindex{$\Taut_\cL$}%
Note that unlike the propositional case, $\Taut$ is not the set of tautologies.
A tautology $\ph$ is in $\Taut$ is and only if $\ph \in \cL^0$. The largest
theory is $\cL^0$. A theory $T$ is inconsistent if and only if $T = \cL^0$. The
definition of logical equivalence, and its associated notation, are all the same
as in the propositional calculus. Note that Lemma \ref{L:omit-psi} holds also in
$\cL^0$.

The construction of inductive derivability in predicate languages follows
exactly as it does in the propositional case. Let $\cL^\IS = \fP \cL^0 \times
\cL^0 \times [0, 1]$ denote the set of \emph{inductive statements} in $\cL$.
  \symindex{$\cL^\IS$}%
  \index{inductive statement}%
All of the results in Sections \ref{S:entire}--\ref{S:ind-cond} depend only the
fact that $\vdash_\cF$ satisfies (i)--(vi) of Definition \ref{D:derivability}.
Since Theorem \ref{P:pred-derivability} shows that $\vdash_\cL$  restricted to
$\cL^0$ also satisfies (i)--(vi), it follows that all of those results hold in
the predicate case as well, with $\cF$ replaced by $\cL^0$. We adopt all of the
notation and terminology of Sections \ref{S:entire}--\ref{S:ind-cond} to define,
in $\cL^\IS$, inductive theories, inductive conditions, inductive derivability,
and all their associated notions.

\subsection{Karp's calculus}\label{S:Karp-calc}

Karp's completeness theorem \cite[Theorem 11.4.1]{Karp1964}, which we present
later in Theorem \ref{T:pred-Karp-sent}, will be essential for us. Karp,
however, defines her deductive calculus in a different way. We present here
Karp's calculus, and show that it is equivalent to the calculus of natural
deduction that we defined earlier. We follow the presentation of her calculus
that is given in \cite[Chapter 4]{Keisler1971}.

Let $\La^-$ be the smallest subset of $\cL$ such that if $x, y \in \Var$, $t \in
\cT$,  $\ph, \psi \in \cL$, and $\Phi \subseteq \cL$ is countable with $\ph \in
\Phi$, then the following formulas are in $\La^-$:
\begin{enumerate}[($\La$1)]
  \item $(\ph \to \psi \to \ze) \to (\ph \to \psi) \to \ph \to \ze$
  \item $(\ph \to \neg \psi) \to \psi \to \neg \ph$
  \item $\bigwedge \Phi \to \ph$
  \item $\forall x \ph \to \ph(t/x)$ when $t$ is free for $x$ in $\ph$
  \item $x \beq x$
  \item $x \beq y \to y \beq x$
  \item $\ph \wedge t \beq x \to \ph(t/x)$ when $t$ is free for $x$ in $\ph$
\end{enumerate}
The set of \emph{axioms}, or \emph{logical theorems}, denoted by $\La =
\La_\cL$,
  \symindex{$\La_\cL$}%
  \index{axiom}%
is the smallest subset of $\cL$ such that
\begin{enumerate}[(I)]
  \item $\La^- \subseteq \La$,
  \item if $\psi, \psi \to \ph \in \La$, then $\ph \in \La$,
  \item if $\psi \to \ph \in \La$ and $x \notin \free \psi$, then $\psi \to
        \forall x \ph \in \La$,
  \item If $\Phi \subseteq \cL$ is countable and $\psi \to \th \in \La$ for all
        $\th \in \Phi$, then $\psi \to \bigwedge \Phi \in \La$.
\end{enumerate}
A \emph{proof of $\ph \in \cL$ from $X \subseteq \cL$}
  \index{proof}%
is an $ (\al + 1)$-sequence of formulas, $\ang{\ph_\be \mid \be \le \al}$, where
$\al$ is a countable ordinal, $\ph_\al = \ph$, and for each $\be \le \al$,
either $\ph_\be \in X \cup \La$, or there exist $i, j < \be$ such that $\ph_i =
(\ph_j \to \ph_\be)$, or there exists $\Phi \subseteq \{\ph_\xi \mid \xi <
\be\}$ such that $\ph_\be = \bigwedge \Phi$. Note that if $\ang{\ph_\be \mid \be
\le \al}$ is a proof of $\ph_\al$ from $X$, then for any $\be < \al$, it follows
that $\ang{\ph_\xi \mid \xi \le \be}$ is a proof of $\ph_\be$ from $X$. For $\ph
\in \cL$ and $X \subseteq \cL$, define $X \wdash \ph$ to mean there is a proof
of $\ph$ from $X$. Note that by (II) above and Proposition \ref{P:conj-axioms}
below, $\ph \in \La$ if and only if $\wdash \ph$.
  \symindex{$\wdash_\cL$}%

\begin{lemma}\label{L:add-ante}
  If $\ph \in \La$, then $\psi \to \ph \in \La$ for all $\psi \in \cL$.
\end{lemma}

\begin{proof}
  Let $\ph \in \La$ and $\psi \in \cL$. By ($\La$3), we have $\neg \neg \psi
  \wedge \neg \ph \to \neg \ph \in \La$. Thus, ($\La$2) and (II) give $\ph \to
  \neg (\neg \neg \psi \wedge \neg \ph) \in \La$. But unwinding our shorthand
  shows that $\neg (\neg \neg \psi \wedge \neg \ph) = \neg \psi \vee \ph = \psi
  \to \ph$. Therefore, $\ph \to \psi \to \ph \in \La$. A final application of 
  (II) yields $\psi \to \ph \in \La$.
\end{proof}

\begin{prop}\label{P:conj-axioms}
  The set $\La$ satisfies
  \begin{enumerate}[(I)$^\prime$, leftmargin=28pt]
    \setcounter{enumi}{3}
    \item If $\Phi \subseteq \La$ is countable, then $\bigwedge \Phi \in \La$.
  \end{enumerate}
\end{prop}

\begin{proof}
  Fix $\th_0 \in \Phi$. By Lemma \ref{L:add-ante}, $\th_0 \in \La$ and $\th_0
  \to \th \in \La$ for all $\th \in \Phi$. Hence, (IV) implies $\th_0 \to
  \bigwedge \Phi \in \La$, and therefore, using (II), we have $\bigwedge \Phi
  \in \La$.
\end{proof}

\begin{rmk}\label{R:prop-taut}
  Let $\cF$ be a propositional language with infinitely many propositional
  variables. Given a function $\tau: PV \to \cL$, we extend it to $\tau: \cF \to
  \cL$ recursively by $(\neg \ph)^\tau = \neg \ph^\tau$ and $(\bigwedge
  \Phi)^\tau = \bigwedge_{\ph \in \Phi} \ph^\tau$. If $\ph \in \cF_\fin$ is a
  tautology, then we call $\ph^\tau$ an \emph{instance of a finitary
  propositional tautology}. By Remark \ref{R:fin-Hilbert}, every such $\ph^\tau$
  can be derived using ($\La$1)--($\La$3), (I), (II), and (IV)$'$. Hence, every
  instance of a finitary propositional tautology is an axiom.
\end{rmk}

\begin{prop}\label{P:Hilbert-ded-thm}
  Let $\ph, \psi \in \cL$. Then $\psi \wdash \ph$ if and only if $\wdash \psi
  \to \ph$.
\end{prop}

\begin{proof}
  Suppose $\psi \wdash \ph$ and let $\ang{\ph_\be \mid \be \le \al}$ be a proof
  of $\ph$ from $\psi$. If $\al = 0$, then $\ph = \psi$ or $\ph \in \La$. If
  $\ph = \psi$, then $\wdash \ph \to \ph$ by Remark \ref{R:prop-taut}. If $\ph
  \in \La$, then $\wdash \psi \to \ph$ by Lemma \ref{L:add-ante}.

  Now assume $\al > 0$ and the result is true whenever there is a proof of
  length less than $\al$. As above, if $\ph \in \{\psi\} \cup \La$, then $\wdash
  \psi \to \ph$. Suppose $\ph_i = \ph_j \to \ph$ for some $i, j < \al$. Then
  $\wdash \psi \to \ph_j$ and $\wdash \psi \to \ph_j \to \ph$. By ($\La$1),
  \[
    \wdash (\psi \to \ph_j \to \ph) \to (\psi \to \ph_j) \to \psi \to \ph.
  \]
  With two applications of modus ponens, we obtain $\wdash \psi \to \ph$.
  Finally, suppose $\ph = \bigwedge \Phi$, where $\Phi \subseteq \{\ph_\xi
  \mid \xi < \be\}$. Then $\wdash \psi \to \th$ for all $\th \in \Phi$. From 
  (IV), it follows that $\wdash \psi \to \ph$.

  Conversely, suppose $\wdash \psi \to \ph$. Then $\psi \wdash \psi \to \ph$ and
  $\psi \wdash \psi$. Applying modus ponens gives $\psi \wdash \ph$.
\end{proof}

The proof of Proposition \ref{P:Hilbert-induc} carries through so that it also
holds in this setting.

\begin{thm}\label{T:pred-Hilbert=nat}
  Let $X \subseteq \cL$ and $\ph \in \cL$. Then $X \wdash \ph$ if and only if $X
  \vdash \ph$.
\end{thm}

\begin{proof}
  We first prove that $X \wdash \ph$ implies $X \vdash \ph$. It suffices to show
  that (1)--(3) in Proposition \ref{P:Hilbert-induc} hold when $\wdash$ is
  replaced by $\vdash$. The proof of Theorem \ref{T:Hilbert=nat} carries through
  in this case, leaving us only to show that $\vdash \ph$ whenever $\ph \in
  \La$. We prove this by induction using (I)--(IV). Our base case is (I), in
  which we prove that $\vdash \ph$ for all $\ph \in \La^-$.

  Axioms of the form ($\La$1) and ($\La$2) are covered by Remark 
  \ref{R:prop-taut}. Axioms of the form ($\La$3)--($\La$5) are covered by (iii),
  (vii), and (ix), respectively. For ($\La$6), let $s = x$, $t = y$, and $\ph =
  (x \beq z)$. By (x), we have $x \beq y, x \beq z \vdash y \beq z$, which gives
  \[
    \vdash x \beq z \to x \beq y \to y \beq z.
  \]
  From here, (viii) gives
  \[
    \vdash \forall z (x \beq z \to x \beq y \to y \beq z),
  \]
  so that by (vii),
  \[
    \vdash (x \beq z \to x \beq y \to y \beq z)(x/z)
      = (x \beq x \to x \beq y \to y \beq x).
  \]
  Hence, $x \beq x \vdash x \beq y \to y \beq x$. Combined with $\vdash x \beq
  x$, we get ($\La$6).

  Finally, for ($\La$7), suppose $t$ is free for $x$ in $\ph$. By (x), we have
  $x \beq t, \ph \vdash \ph(t/x)$. Reversing the roles of $x$ and $y$ in 
  ($\La$6) gives $\vdash y \beq x \to x \beq y$. From (viii) and (vii), it
  follows that $\vdash t \beq x \to x \beq t$. Putting these together, we obtain
  $\vdash t \beq x \wedge \ph \to \ph(t/x)$, and this completes our base case.

  For inductive step (II), suppose $\psi, \psi \to \ph \in \La$ and $\vdash
  \psi, \psi \to \ph$. Then we directly have $\vdash \ph$, and we are done. Step
  (IV) is equally straightforward. For step (III), suppose $\psi \to \ph \in
  \La$ with $x \notin \free \psi$ and $\vdash \psi \to \ph$. Then $\psi \vdash
  \ph$, so that (viii) implies $\psi \vdash \forall x \ph$, which gives $\vdash
  \psi \to \forall x \ph$.

  To prove that $X \vdash \ph$ implies $X \wdash \ph$, it suffices to show that
  (i)--(x) hold when $\vdash$ is replaced by $\wdash$. The fact that (i)--(vi)
  hold for $\wdash$ follows exactly as in the propositional case. We get (vii)
  from ($\La$4). For (viii), suppose $x \notin \free X$ and $X \wdash \ph$. Fix
  a proof of $\ph$ from $X$ and let $X_0$ be the set of $\th \in X$ that appear
  in the proof. Since proofs have countable lengths, $X_0$ is countable. Also,
  $X_0 \wdash \ph$. By ($\La$3), we have $\bigwedge X_0 \wdash \ph$. Lemma
  \ref{P:Hilbert-ded-thm} implies $\wdash \bigwedge X_0 \to \ph$, so that (III)
  gives $\wdash \bigwedge X_0 \to \forall x \ph$. Since $X \wdash \bigwedge
  X_0$, an application of modus ponens yields $X \wdash \forall x \ph$.

  From ($\La$5), (viii), and (vii), we obtain (ix). For (x), suppose $s$ and $t$
  are free for $x$ in $\ph$ and $X \wdash s \beq t, \ph(s/x)$. By ($\La$7), 
  (viii), and (vii), we have $\wdash \ph(s/x) \wedge t \beq s \to \ph(t/x)$.
  Hence, $t \beq s, \ph(s/x) \wdash \ph(t/x)$. In the same way, but using 
  ($\La$6), we get $s \beq t \wdash t \beq s$. Combining these gives $X \wdash
  \ph(t/x)$.
\end{proof}

\section{Predicate models}\label{S:pred-models}

In this section, we present the semantics of both deductive and inductive
predicate logic. We define satisfiability and consequence, and prove
$\si$-compactness, soundness, and completeness. As with the
predicate calculus, the bulk of our work will be in the deductive case. The
inductive case will require very little modification from its presentation in
Chapter \ref{Ch:prop-models}.

\subsection{Strict satisfiability}

Let $L$ be an extralogical signature and $\cL$ the set of formulas built from
$L$. We will use the two phrases, ``$L$-structure'' and ``$\cL$-structure,''
interchangeably. Let $\om = (A, L^\om)$ be an $\cL$-structure. For ground terms
$t$, we define $t^\om \in A$ recursively by $(f t_1 \cdots t_n)^\om = f^\om
(t_1^\om, \ldots, t_n^\om)$.

An \emph{assignment $v$ into $A$} is a function $v: \Var \to A$. We can extend
$v$ to a function $v_\om: \cT \to A$ by $v_\om(c) = c^\om$ and $v_\om(f t_1
\cdots t_n) = f^\om v_\om (t_1) \cdots v_\om(t_n)$. The extended function
$v_\om$ is called an \emph{assignment into $\om$}. When there is no risk of
confusion, we will omit the subscript and also write $v$ for the extended
function $v_\om$.
  \index{assignment}%

Note that if $t$ is a ground term, then $v(t) = t^\om$. More generally, if $t =
t(x_1, \ldots, x_n)$, then $v(t)$ depends only on $v(x_1), \ldots, v(x_n)$. In
this case, we will write $t^\om[\vec a]$, for $\vec a \in A^n$, as shorthand for
$v(t)$, where $v$ is any assignment satisfying $v(x_i) = a_i$.
  \symindex{$t^\om[\vec a]$}%

Given an assignment $v$, if $x \in \Var$ and $a \in A$, then we define a new
assignment $v_x^a$ by $v_x^a(x) = a$ and $v_x^a(y) = v(y)$ for $y \ne x$.

\begin{defn}\label{D:strict-sat}
  Let $\om$ be a structure and $v$ an assignment into $\om$. For $\ph \in \cL$,
  we define $\om \tDash \ph[v]$ recursively as follows:
  \begin{enumerate}[(i)]
    \item $\om \tDash (s \beq t)[v]$ if and only if $v(s) = v(t)$,
    \item $\om \tDash (r t_1 \cdots t_n)[v]$ if and only if $r^\om v(t_1) \cdots
          v(t_n)$,
    \item $\om \tDash (\neg \ph)[v]$ if and only if $\om \ntDash \ph[v]$,
    \item $\om \tDash (\bigwedge \Phi)[v]$ if and only if $\om \tDash \ph[v]$ for
          all $\ph \in \Phi$, and
    \item $\om \tDash (\forall x \ph)[v]$ if and only $\om \tDash \ph[v_x^a]$ for
          all $a \in A$.
  \end{enumerate}
  If $\om \tDash \ph[v]$, we say that \emph{$\om$ strictly satisfies $\ph$ with
  $v$}. Note that if $\ph \in \cL_\fin$, then $\tDash$ is the usual notion of
  satisfiability from first-order logic.
\end{defn}
  \index{satisfiable!strictly ---}%
  \symindex{$\om \tDash \ph[v]$}

If $\ph = \ph(x_1, \ldots, x_n)$ and $v$ and $v'$ are assignments that agree on
$x_1, \ldots, x_n$, then $\om \tDash \ph[v]$ if and only if $\om \tDash \ph
[v']$. In this case, we will write $\om \tDash \ph[\vec a]$, where $\vec a \in
A^n$, to mean that $\om \tDash \ph[v]$ for all assignments $v$ satisfying
$v(x_i) = a_i$. In particular, if $\ph$ is a sentence, then $\om \tDash \ph$
means that $\om \tDash \ph[v]$ for all assignments $v$. A formula $\ph$ is said
to be \emph{strictly satisfiable} if $\om \tDash \ph[v]$ for some structure
$\om$ and some assignment $v$.

The proofs of the following three theorems are the same as in first-order logic,
except we use Definition \ref{D:strict-sat}(iv). See, for instance, \cite
[Theorems 2.3.1, 2.3.4, and 2.3.5]{Rautenberg2010} for details.

\begin{thm}[Coincidence theorem]\label{T:coinc-thm}
  Let $\om$ and $\om'$ be structures with a common domain. Let $v$ and $v'$ be
  assignments into $\om$ and $\om'$, respectively. Let $\ph \in \cL$ and assume
  \begin{enumerate}[(i)]
    \item $v(x) = v'(x)$ for all $x \in \free \ph$, and
    \item $\s^\om = \s^{\om'}$ for all $\s \in \sym \ph$.
  \end{enumerate}
  Then $\om \tDash \ph[v]$ if and only if $\om' \tDash \ph[v']$.
\end{thm}

\begin{thm}[Invariance theorem]\label{T:invar-thm}
  Let $\om$ and $\om'$ be isomorphic $\cL$-structures and let $g: \om \to \om'$
  be an isomorphism. Let $v$ be an assignment into $\om$ and define the
  assignment $v'$ into $\om'$ by $v'(x) = g v(x)$. Then $\om \tDash \ph[v]$ if
  and only if $\om' \tDash \ph[v']$, for all $\ph \in \cL$. In particular, if
  $\ph = \ph(x_1, \ldots, x_n)$, then $\om \tDash \ph[\vec a]$ if and only if
  $\om' \tDash \ph[g \vec a]$ for all $\vec a \in A^n$, where $A$ is the domain
  of $\om$. Consequently, if $\ph \in \cL^0$ is a sentence, then $\om \tDash
  \ph$ if and only if $\om' \tDash \ph$.
\end{thm}

If $v$ is an assignment into $\om$ and $\si$ is a substitution, then $v^\si$ is
the assignment defined by $v^\si(x) = v(x^\si)$. By induction on $t$, we have
$v^\si(t) = v(t^\si)$ for all $t \in \cT$.

\begin{thm}[Substitution theorem]
  Let $v$ be an assignment into the structure $\om$. Let $\ph \in \cL$ and let
  $\si$ be a substitution. If $\si$ is free for $\ph$, then $\om \tDash
  \ph^\si[v]$ if and only if $\om \tDash \ph[v^\si]$. In particular, if $t$ is
  free for $x$ in $\ph$, then $\om \tDash \ph(t/x)[v]$ if and only if $\om
  \tDash \ph[v_x^a]$, where $a = v(t)$.
\end{thm}

\begin{rmk}
  The strict satisfiability relation is not $\si$-compact. For example, let $L =
  \{c_n \mid n \in \bN_0\}$ be a set of distinct constant symbols and let
  \[
    \ts{
      X = \{\forall x \bigvee_{n \in \bN_0} x \beq c_n\}
        \cup \{x \nbeq y \mid x, y \in \Var, x \ne y\}.
    }
  \]
  Then every countable subset of $X$ is strictly satisfiable, but $X$ is not
  satisfiable.
\end{rmk}

\subsection{Models and deductive satisfiability}

An \textit{inductive $\cL$-model}, or simply a \emph{model}, is a probability
space, $\sP = (\Om, \Si, \bbP)$, where $\Om$ is a set of $\cL$-structures.
  \index{model}%
An \emph{assignment into $\sP$} is an indexed collection $\bv = \ang{v_\om \mid
\om \in \Om}$, where $v_\om$ is an assignment into $\om$ for each $\om \in \Om$.
  \index{assignment}%
Note that $\bv$ does not depend on $\Si$ or $\bbP$. We may therefore sometimes
call $\bv$ an assignment into $\Om$.

If $\bv$ is an assignment into a model $\sP = (\Om, \Si, \bbP)$ and $\ph \in
\cL$, then we define
\[
  \ph[\bv]_\Om = \{\om \in \Om \mid \om \tDash \ph[v_\om]\}.
\]
  \symindex{$\ph[\bv]_\Om$}%
We say that \emph{$\sP$ satisfies $\ph$ with $\bv$}, denoted by $\sP \vDash
\ph[\bv]$, if $\ph[\bv]_\Om \in \ol \Si$ and $\olbbP \ph[\bv]_\Om = 1$. A set $X
\subseteq \cL$ is \emph{satisfiable} if there is a model $\sP$ and an assignment
$\bv$ into $\sP$ such that $\sP \vDash \psi[\bv]$ for all $\psi \in X$.
  \index{satisfiable}%
  \symindex{$\sP \vDash \ph[\bv]$}%

If $\ph$ is a sentence, then $\ph[\bv]_\Om$ does not depend on $\bv$. In this
case, we simply write $\ph_\Om$, and as in the propositional case, we have
\[
  \ph_\Om = \{\om \in \Om \mid \om \tDash \ph\}.
\]
We then write $\sP \vDash \ph$ to mean $\sP \vDash \ph[\bv]$ for all assignments
$\bv$, and this holds if and only if $\ph_\Om \in \ol \Si$ and $\olbbP \ph_\Om =
1$.

\begin{prop}\label{P:pred-pre-cpct}
  Let $X \subseteq \cL$.
  \begin{enumerate}[(i)]
    \item If $X$ is strictly satisfiable, then $X$ is satisfiable.
    \item If $X$ is satisfiable and countable, then $X$ is strictly satisfiable.
  \end{enumerate}
\end{prop}

\begin{proof}
  For (i), suppose $\om \tDash \psi[v]$ for all $\psi \in X$. Let $\sP =
  (\{\om\}, \{\emp, \{\om\}\}, \de_ {\om})$ and $\bv = \ang{v}$. Then $\sP
  \vDash \ph[\bv]$. For (ii), suppose $X$ is satisfiable and countable. Let $\sP
  = (\Om, \Si, \bbP)$ be a model and $\bv$ an assignment into $\sP$ such that
  $\sP \vDash \psi[\bv]$ for all $\psi \in X$. Then $\olbbP \bigcap_{\psi \in X}
  \psi[\bv]_\Om = 1$, so we may choose $\om \in \bigcap_{\psi \in X}
  \psi[\bv]_\Om$. We then have $\om \tDash \psi [v_\om]$ for all $\psi \in X$,
  so that $\om$ strictly satisfies $X$ with $v_\om$.
\end{proof}

Given a model $\sP = (\Om, \Si, \bbP)$, let
\[
  \Si_\cL = \Si \cap \{\ph[\bv]_\Om \mid \ph \in \cL \text{ and $\bv$ is an
  assignment into $\sP$}\},
\]
and let ${\bbP_\cL} = {\bbP}|_{\Si_\cL}$. Then $\Si_\cL$ is a sub-$\si$-algebra
of $\Si$, so that $ (\Om, \Si_\cL, \bbP_\cL)$ is also a model. For any $\ph \in
\cL$ and any assignment $\bv$ into $\sP$, we have $\ph[\bv]_\Om \in \Si$ if and
only if $\ph[\bv]_\Om \in \Si_\cL$. Hence, from a logical standpoint, every set
$A \in \Si \setminus \Si_\cL$ is irrelevant.

Let $\sP = (\Om, \Si, \bbP)$ and $\sQ = (\Om', \Ga, \bbQ)$ be models. We say
that $\sP$ and $\sQ$ are \emph{isomorphic (as models)}, denoted by $\sP \simeq
\sQ$, if there exists a measurable function $h: \Om \to \Om'$ such that $h$
induces an isomorphism (as measure spaces) from $(\Om, \Si_\cL, \bbP_\cL)$ to $
(\Om', \Ga_\cL, \bbQ_\cL)$, and $\om \simeq h \om$ for $\bbP$-a.e.~$\om \in
\Om$. In this case, we abuse notation and say that $h: \sP \to \sQ$ is a
\emph{model isomorphism}.
  \index{model!isomorphic ---}%

Let $h: \sP \to \sQ$ be a model isomorphism and $\bv$ and assignment into $\sP$.
An assignment $\bv'$ into $\sQ$ is called an \emph{image of $\bv$ under $h$} if,
for each $\om \in \Om$, there is a function $g_\om: \om \to h \om$ such that,
for $\bbP$-a.e.~$\om \in \Om$, the function $g_\om$ is an isomorphism and
$v'_{\om'} (x) = g_\om v_\om(x)$ for all $x \in \Var$.
  \index{assignment!image of an ---}%

An image of $\bv$ under $h$ always exists. We can construct one as follows. For
each $\om' \in \Om'$, choose $a_{\om'}$ in the domain of $\om'$. For each $\om
\in \Om$, if $\om \simeq h \om$, then let $g_\om$ be an isomorphism from $\om$
to $h \om$. Otherwise, let $g_\om$ map everything to $a_{h \om}$. We then define
$v'_\om(x) = g_\om v_\om(x)$.

\begin{lemma}\label{L:iso-thm}
  Let $h$ be a model isomorphism from $\sP = (\Om, \Si, \bbP)$ to $\sQ = (\Om',
  \Ga, \bbQ)$. Let $\bv$ be an assignment into $\sP$ and let $\bv'$ be an image
  of $\bv$ under $h$. Then $h^{-1} \ph[\bv']_{\Om'} = \ph[\bv]_\Om$, $\bbP$-a.s.
  Consequently, for all $\ph \in \cL$, we have $\ph[\bv]_\Om \in \ol \Si$ if and
  only if $\ph[\bv']_{\Om'} \in \ol \Ga$, and in this case, $\olbbQ
  \ph[\bv']_{\Om'} = \olbbP \ph[\bv]_\Om$.
\end{lemma}

\begin{proof}
  Let $\om' = h \om$. For all $\om \in \Om$, we have $\om \in h^{-1} \ph[\bv']_
  {\Om'}$ if and only if $\om' \in \ph[\bv']_{\Om'}$, which holds if and only if
  $\om' \tDash \ph[v'_{\om'}]$. By Theorem \ref{T:invar-thm}, for
  $\bbP$-a.e.~$\om$, this is equivalent to $\om \tDash \ph[v_\om]$, which holds
  if and only if $\om \in \ph[\bv]_\Om$. Hence, $h^ {-1} \ph[\bv']_{\Om'} =
  \ph[\bv]_\Om$ $\bbP$-a.e. Since $h$ also induces a measure-space isomorphism
  from $\ol \sP$ to $\ol \sQ$, we have ${\olbbQ} = {\olbbP} \circ h^{-1}$.
  Therefore, $\ph[\bv]_\Om \in \ol \Si$ if and only if $\ph[\bv']_{\Om'} \in \ol
  \Ga$, and in this case, $\olbbQ \ph[\bv']_{\Om'} = \olbbP \ph[\bv]_\Om$.
\end{proof}

\begin{thm}[Deductive isomorphism theorem]\label{T:ded-iso-thm}
    \index{isomorphism theorem!deductive ---}%
  Let $\sP$ and $\sQ$ be isomorphic $\cL$-models and let $h: \sP \to \sQ$ be a
  model isomorphism. Let $\bv$ be an assignment into $\sP$ and let $\bv'$ be an
  image of $\bv$ under $h$. Then $\sP \vDash \ph[\bv]$ if and only if $\sQ
  \vDash \ph[\bv']$, for all $\ph \in \cL$. In particular, if $\ph$ is a
  sentence, then $\sP \vDash \ph$ if and only if $\sQ \vDash \ph$.
\end{thm}

\begin{proof}
  Let $\sP = (\Om, \Si, \bbP)$ and $\sQ = (\Om', \Ga, \bbQ)$ be isomorphic. Let
  $h$, $\bv$, and $\bv'$ be as in the statement of the theorem. Suppose $\sP
  \vDash \ph[\bv]$. Then $\olbbP \ph[\bv]_\Om = 1$. By Lemma \ref{L:iso-thm}, we
  have $\olbbQ \ph[\bv']_{\Om'} = \olbbP \ph[\bv]_\Om = 1$. Therefore, $\sQ
  \vDash \ph [\bv']$. Reversing the roles of $\sP$ and $\sQ$ gives the converse.
\end{proof}

\begin{rmk}\label{R:a.s.-sure}
  Let $\sP = (\Om, \Si, \bbP)$ be a model. Let $\Om^* \in \Si$ with $\bbP \Om^*
  = 1$. Let $\Si^* = \{A \cap \Om^* \mid A \in \Si\}$ and ${\bbP^*} = {\bbP}|_
  {\Si^*}$. Then $\sP^* = (\Om^*, \Si^*, \bbP^*)$ is a model. Choose $\om_0$ in
  $\Om^*$ and define $h: \Om \to \Om^*$ by $h \om = \om$ if $\om \in \Om^*$ and
  $h \om = \om_0$ if $\om \notin \Om^*$. It is straightforward to verify that
  $h$ is measurable and induces an isomorphism (as measure spaces) from $(\Om,
  \Si_\cL, \bbP_\cL)$ to $(\Om^*, \Si^*_\cL, \bbP^*_\cL)$. Hence, $h$ is a model
  isomorphism and $\sP \simeq \sP^*$. It follows that if a given property is
  true almost surely in a model $\sP$, then we can find an isomorphic model in
  which it is true for every structure $\om$.
\end{rmk}

\subsection{Deductive consequence and soundness}

We say that $\ph \in \cL$ is a \textit{consequence} of $X \subseteq \cL$, or
that $X$ \emph{entails} $\ph$, which we denote by $X \vDash \ph$, if $\sP \vDash
\ph[\bv]$ whenever $\sP \vDash \psi[\bv]$ for all $\psi \in X$.
  \index{consequence relation}%
Note that if $X$ is not satisfiable, then it is vacuously true that $X \vDash
\ph$ for all $\ph \in \cL$. In the case $X = \emp$, we write ${} \vDash \ph$,
which means that $\sP \vDash \ph[\bv]$ for all $\sP$ and $\bv$. If $X$ and $\ph$
are sentences, then $X \vDash \ph$ if and only if $\sP \vDash X$ implies $\sP
\vDash \ph$.
  \symindex{$X \vDash \ph$ (in $\cL$)}%

\begin{rmk}\label{R:strict-conseq}
  If $X \vDash \ph$ and $\om \tDash \psi[v]$ for all $\psi \in X$, then $\om
  \tDash \ph[v]$. To see this, simply apply the above definition to $\sP = (
  \{\om\}, \{\emp, \{\om\}\}, \de_\om)$ and $\bv = \ang{v}$. In particular, if
  $X$ and $\ph$ are sentences and $\om \tDash X$, then $\om \tDash \ph$.
\end{rmk}

\begin{prop}\label{P:taut-struct}
  Let $\ph \in \cL$. Then ${} \vDash \ph$ if and only if, for all structures
  $\om$ and all assignments $v_\om$ into $\om$, we have $\om \tDash \ph[v_\om]$.
\end{prop}

\begin{proof}
  Suppose ${} \vDash \ph$. Then $\sP \vDash \ph[\bv]$ for all $\sP$ and $\bv$.
  Let $\om$ be a structure and $v_\om$ an assignment into $\om$. Define $\sP = (
  \{\om\}, \{\emp, \{\om\}\}, \de_\om)$ and $\bv = \ang{v_\om}$. Then $\sP$ is a
  model and $\bv$ is an assignment into $\sP$. By hypothesis, $\sP \vDash
  \ph[\bv]$, which means $\ph[\bv]_\Om = \{\om\}$. That is, $\om \tDash
  \ph[v_\om]$.

  Conversely, suppose $\om \tDash \ph[v_\om]$ for all structures $\om$ and all
  assignments $v_\om$ into $\om$. Let $\sP = (\Om, \Si, \bbP)$ be a model and
  $\bv$ an assignment into $\sP$. Then $\ph[\bv]_\Om = \Om$, so $\olbbP
  \ph[\bv]_\Om = 1$. Therefore, $\sP \vDash \ph[\bv]$.
\end{proof}

\begin{prop}\label{P:free-elim-conseq}
  Let $X \subseteq \cL$ and $\ph \in \cL$, and let $\si$ be a free eliminator.
  Then $X \vDash_\cL \ph$ if and only if $X^\si \vDash_{\cL C} \ph^\si$.
\end{prop}

\begin{proof}
  Suppose $X \vDash_\cL \ph$. Let $\sQ = (\Om', \Ga, \bbQ)$ be an $\cL C$-model
  and $\bv'$ an assignment into $\sQ$. Assume $\sQ \vDash \psi^\si[\bv']$ for
  all $\psi \in X$. For each $\om' \in \Om'$, let $\om$ be its $\cL$-reduct. Let
  $\Om = \{\om \mid \om' \in \Om\}$ and let $\sP = (\Om, \Si, \bbP)$ be the
  measure space image of $\sQ$ under the function $\om' \mapsto \om$. Define an
  assignment $\bv$ into $\sP$ by $v_\om(x) = v'_{\om'}(x)$ if $x^\si = x$, and
  $v_\om(x) = c_x^{\om'}$ if $x^\si = c_x$. By term induction, $v_\om(t) =
  v'_{\om'}(t^\si)$ for all $t \in \cT_\cL$. Then, by formula induction, we
  obtain $\om \tDash \psi[v_\om]$ if and only if $\om' \tDash
  \psi^\si[v'_{\om'}]$ for all $\psi \in \cL$. Hence, if $h$ denotes the
  function $\om' \mapsto \om$, then $h^{-1}\psi[\bv]_\Om = \psi^\si
  [\bv']_{\Om'}$, which gives $\olbbP \psi[\bv]_\Om = \olbbQ \psi^\si
  [\bv']_{\Om'}$ for all $\psi \in \cL$. Therefore, $\sP \vDash \psi[\bv]$ for
  all $\psi \in X$. By hypothesis, this implies $\sP \vDash \ph[\bv]$, which is
  equivalent to $\sQ \vDash \ph^\si[\bv']$, and we have $X^\si \vDash_{\cL C}
  \ph^\si$.

  For the converse, suppose $X^\si \vDash_{\cL C} \ph^\si$. Let $\sP = (\Om,
  \Si, \bbP)$ be an $\cL$-model and $\bv$ an assignment into $\sP$. Assume $\sP
  \vDash \psi[\bv]$ for all $\psi \in X$. For each $\om \in \Om$, define the
  $\cL C$-structure $\om'$ by $\s^{\om'} = \s^\om$ if $\s \in L$, and
  $c_x^{\om'} = v_\om(x)$. Define an assignment $\bv'$ into $\sP$ by
  $v'_{\om'}(x) = v_\om(x)$ for all $x \in \Var$. Then again by term and formula
  induction, we have $\om \tDash \psi[v_\om]$ if and only if $\om' \tDash
  \psi^\si[v'_{\om'}]$ for all $\psi \in \cL$, which as above yields $\sP \vDash
  \ph[\bv]$. Therefore, $X \vDash_\cL \ph$.
\end{proof}

\begin{thm}[Karp's completeness theorem]\label{T:pred-Karp-sent}
    \index{completeness!Karp's --- theorem}%
  Let $\ph \in \cL^0$ be a sentence. Then ${} \vdash \ph$ if and only if ${}
  \vDash \ph$.
\end{thm}

\begin{proof}
  Karp's completeness theorem first appears in \cite[Theorem 11.4.1]{Karp1964}.
  The version we are citing is \cite[Theorem 4.3]{Keisler1971}. There it is
  shown that if $\ph \in \cL^0$ is a sentence, then $\om \tDash \ph$ for all
  structures $\om$ if and only if $\ph \in \La'$, where $\La'$ is a set of
  logical axioms described in \cite{Keisler1971}. We claim that $\La'$ is the
  same as $\La$, the set of axioms defined in Section \ref{S:Karp-calc}. Since
  $\ph \in \La$ if and only if $\wdash \ph$, our statement of Karp's theorem
  then follows from Theorem \ref{T:pred-Hilbert=nat} and Proposition
  \ref{P:taut-struct}.

  Keisler's $\La'$ differs from $\La$ in only one way. To describe it, we
  recursively define the shorthand ${\sim} \ph$ as follows:
  \begin{enumerate}[label={}]
    \item ${\sim} \ph = \neg \ph$ if $\ph$ is prime,
    \item ${\sim} \neg \ph = \ph$,
    \item ${\sim} \bigwedge \Phi = \bigvee_{\th \in \Phi} \neg \th
          = \neg \bigwedge_{\th \in \Phi} \neg \neg \th$, and
    \item ${\sim} \forall x \ph = \exists x \neg \ph
          = \neg \forall x \neg \neg \ph$.
  \end{enumerate}
  Keisler's $\La'$ includes everything in $\La$, as well as all formulas of the
  form
  \begin{enumerate}[($\La$1)]
    \setcounter{enumi}{7}
    \item $\neg \ph \tot {\sim} \ph$
  \end{enumerate}
  To check that $\La' = \La$, we must verify that these formulas are already in
  $\La$. We can break this down according to whether $\ph$ is prime, $\ph = \neg
  \psi$, $\ph = \bigwedge \Phi$, or $\ph = \forall x \psi$. Doing this, applying
  the definition of ${\sim}$, and using Theorem \ref{T:pred-Hilbert=nat}, we must check that
  \begin{enumerate}[label={}]
    \item $\vdash \neg \ph \tot \neg \ph$,
    \item $\vdash \neg \neg \psi \tot \psi$,
    \item $\vdash \neg \bigwedge \Phi \tot \neg \bigwedge_{\th \in \Phi} \neg
          \neg \th$, and
    \item $\vdash \neg \forall x \psi \tot \neg \forall x \neg \neg \psi$.
  \end{enumerate}
  The first two are propositional tautologies. The third follows from
  $\vdash \th \tot \neg \neg \th$ and Definition \ref{D:derivability}(iii),(iv).
  The fourth follows from $\vdash \psi \tot \neg \neg \psi$ and Definition
  \ref{D:pred-derivability}(vii)$'$,(viii).
\end{proof}

\begin{thm}[Karp's theorem for formulas]\label{T:pred-Karp-form}
  For any formula $\ph \in \cL$, we have ${} \vdash \ph$ if and only if ${}
  \vDash \ph$.
\end{thm}

\begin{proof}
  Let $\si$ be a full free eliminator, so that $\ph^\si$ is a sentence. By
  Propositions \ref{P:free-elim-deriv} and \ref{P:free-elim-conseq}, we have
  $\vdash_\cL \ph$ if and only if $\vdash_{\cL C} \ph^\si$, and $\vDash_\cL \ph$
  if and only if $\vDash_{\cL C} \ph^\si$. Theorem \ref{T:pred-Karp-sent} gives
  $\vdash_{\cL C} \ph^\si$ if and only if $\vDash_{\cL C} \ph^\si$.
\end{proof}

As in the propositional case, Karp's completeness theorem allows us to prove the
result that was described in Remark \ref{R:pred-fin-vs-infin}.

\begin{prop}\label{P:pred-fin-vs-infin}
  Let $X \subseteq \cL_\fin$ and $\ph \in \cL_\fin$. If $X \vdash \ph$, then $X
  \vdash_\fin \ph$.
\end{prop}

\begin{proof}
  Let $X \subseteq \cL_\fin$ and $\ph \in \cL_\fin$. Suppose $X \vdash \ph$. The
  well-known completeness theorem from first-order logic states that $X
  \vdash_\fin \ph$ if and only if, for all structures $\om$ and all assignments
  $v$ into $\om$, we have $\om \tDash \ph[v]$ whenever $\om \tDash \psi[v]$ for
  all $\psi \in X$. (See, for instance, \cite[Theorem 3.2.7]{Rautenberg2010}).

  Let $\om$ be a structure and $v$ an assignment into $\om$. Assume that $\om
  \tDash \psi[v]$ for all $\psi \in X$. Choose countable $X_0 \subseteq X$ such
  that $\vdash \bigwedge X_0 \to \ph$. By Theorem \ref{T:pred-Karp-form}, we
  have $\vDash \bigwedge X_0 \to \ph$. Hence, $\om \tDash (\bigwedge X_0 \to
  \ph)[v]$, which means $\om \tDash (\bigwedge X_0)[v]$ implies $\om \tDash
  \ph[v]$. Since $\om \tDash \psi[v]$ for all $\psi \in X_0$, it follows that
  $\om \tDash (\bigwedge X_0)[v]$. Therefore, $\om \tDash \ph[v]$.
\end{proof}

\begin{thm}[Deductive soundness]\label{T:pred-soundness}
    \index{soundness!deductive ---}%
  Let $X \subseteq \cL$ and $\ph \in \cL$. If $X \vdash \ph$, then $X \vDash
  \ph$.
\end{thm}

\begin{proof}
  Suppose $X \vdash \ph$. Let $\sP = (\Om, \Si, \bbP)$ be a model and $\bv$ an
  assignment into $\sP$ such that $\sP \vDash \psi[\bv]$ for all $\psi \in X$.
  By Theorem \ref{T:pred-sig-cpctness}, we may choose countable $X_0 \subseteq
  X$ with $X_0 \vdash \ph$. Hence, $\vdash \ze \to \ph$, where $\ze = \bigwedge
  X_0$. By Theorem \ref{T:pred-Karp-form}, we have $\vDash \ze \to \ph$, so that
  $\sP \vDash (\ze \to \ph)[\bv]$. That is, $\olbbP (\ze \to \ph)[\bv]_\Om = 1$.
  But $\olbbP \psi[\bv]_\Om = 1$ for all $\psi \in X$ and $\ze[\bv]_\Om =
  \bigcap_{\psi \in X_0} \psi[\bv]_\Om$. Hence, $\olbbP \ze[v]_\Om = 1$, so that
  $\ze[\bv]_\Om^c$ is a null set. Since $(\ze \to \ph)[\bv]_\Om = \ze[\bv]_\Om^c
  \cup \ph[\bv]_\Om$, we have $\olbbP \ph[\bv]_\Om = \olbbP (\ze \to \ph)
  [\bv]_\Om = 1$. Therefore, $\sP \vDash \ph[\bv]$. Since $\sP$ was arbitrary,
  this shows that $X \vDash \ph$.
\end{proof}

\begin{cor}\label{C:pred-soundness}
  If $X \subseteq \cL$ is satisfiable, then $X$ is consistent. If $X$ is
  countable and consistent, then $X$ is strictly satisfiable.
\end{cor}

\begin{proof}
  Let $X \subseteq \cL$. Suppose $X$ is inconsistent. Then $X \vdash \bot$. By
  Theorem \ref{T:pred-soundness}, we have $X \vDash \bot$. But $\bot_\Om =
  \emp$, so $\sP \nvDash \bot[\bv]$ for all $\sP$ and $\bv$. Hence, $X$ is not
  satisfiable. For the second part, suppose $X$ is countable and not strictly
  satisfiable. Let $v$ be an assignment into a structure $\om$. Then $\om
  \ntDash (\bigwedge X)[v]$, which implies $\om \tDash (\neg \bigwedge X)[v]$.
  Since $\om$ and $v$ were arbitrary, we have $\vDash \neg \bigwedge X$. Theorem
  \ref{T:pred-Karp-form} then implies $\vdash \neg \bigwedge X$. Thus, $X \vdash
  \bigwedge X, \neg \bigwedge X$, so that $X$ is inconsistent.
\end{proof}

\subsection{Deductive completeness}

According to Theorem \ref{T:pred-Karp-form}, we have that $\ph$ is a tautology
if and only if $\om \tDash \ph[v]$ for all $\om$ and $v$. Hence, in any model
$\sP$, we have $\ph \vdash \psi$ implies $\ph[\bv]_\Om \subseteq \psi[\bv]_\Om$
and $\ph \equiv \psi$ implies $\ph[\bv]_\Om = \psi[\bv]_\Om$, for any assignment
$\bv$ into $\sP$.

In Remark \ref{R:impl-subset}, we saw that in the propositional case, we could
obtain a converse to the above if we took $\Om$ to be the set of all strict
models. That converse was essential to our proof of both deductive and inductive
completeness. In the predicate case, we cannot do this, since the collection of
all structures is not a set. Instead, we will use the set of structures defined
in the proof of the following proposition.

\begin{prop}\label{P:all-structures}
  There exists a set of structures $\Om$ and an assignment $\bv$ into $\Om$ such
  that $\ph[\bv]_\Om \subseteq \psi[\bv]_\Om$ implies $\ph \vdash \psi$, and
  $\ph [\bv]_\Om = \psi[\bv]_\Om$ implies $\ph \equiv \psi$. In particular, if
  $\ph$ and $\psi$ are sentences, then $\ph_\Om \subseteq \psi_\Om$ if and only
  if $\ph \vdash \psi$, and $\ph_\Om = \psi_\Om$ if and only if $\ph \equiv
  \psi$.
\end{prop}

\begin{proof}
  Let $S$ be the set of all countable, consistent subsets of $\cL$. By Corollary
  \ref{C:pred-soundness}, for each $X \in S$, we may choose a structure $\om =
  \om_X$ and an assignment $v_\om$ into $\om$ such that $\om \tDash \ze[v_\om]$
  for all $\ze \in X$. Let $\Om = \{\om_X \mid X \in S\}$ and let $\bv = 
  \ang{v_\om \mid \om \in \Om}$.

  For the first implication, let $\ph, \psi \in \cL$ and assume $\ph \nvdash
  \psi$. Then Theorem \ref{T:deduc-con} implies $X = \{\ph, \neg \psi\}$ is
  consistent, so that $X \in S$. Hence, with $\om = \om_X \in \Om$ and $v$
  defined as above, we have $\om \tDash \ph[v_\om]$ and $\om \tDash (\neg
  \psi)[v_\om]$. The latter implies $\om \ntDash \psi[v_\om]$. Thus, $\om \in
  \ph_\Om[\bv]$ and $\om \notin \psi_\Om[\bv]$, so that $\ph_\Om[\bv] \nsubseteq
  \psi_\Om[\bv]$. Reversing the roles of $\ph$ and $\psi$ gives the second
  implication.
\end{proof}

\begin{thm}[$\si$-compactness]\label{T:pred-compactness}
    \index{s_sigma-compactness@$\si$-compactness}%
  A set $X \subseteq \cL$ is satisfiable if and only if every countable subset
  of $X$ is satisfiable.
\end{thm}

\begin{proof}
  The only if part is trivial. Suppose every countable subset of $X$ is
  satisfiable. Assume $X$ is inconsistent. Then $X \vdash \bot$. By Theorem
  \ref{T:pred-sig-cpctness}, there exists countable $X_0 \subseteq X$ such that
  $X_0 \vdash \bot$, implying that $X_0$ is inconsistent. By Corollary
  \ref{C:pred-soundness}, we have that $X_0$ is not satisfiable, a
  contradiction. Hence, $X$ is consistent.

  Let $\Om$ and $\bv$ be as in Proposition \ref{P:all-structures}. Let
  \[
    \Si = \{\ph[\bv]_\Om \mid X \vdash \ph \text{ or } X \vdash \neg \ph\}.
  \]
  Then $\Si$ is a $\si$-algebra. If $A \in \Si$, choose $\ph$ such that $A =
  \ph_\Om[\bv]$. Since $X$ is consistent, we cannot have both $X \vdash \ph$ and
  $X \vdash \neg \ph$. We may therefore define $\bbP A = 1$ if $X \vdash \ph$
  and $0$ otherwise. If $A = \ph_\Om[\bv] = \psi_\Om[\bv]$, then $\ph \equiv
  \psi$, by Proposition \ref{P:all-structures}. Hence, $\bbP$ is well-defined.

  Since $X$ is consistent, $X \nvdash \bot$. Thus, $\bbP \emp = \bbP \bot_\Om =
  0$. Conversely, $X \vdash \top$, so $\bbP \Om = \bbP \top_\Om = 1$.

  Now let $\{A_n\}_{n \in \bN} \subseteq \Si$ be pairwise disjoint, and define
  $A = \bigcup_n A_n$. For each $n$, choose $\ph_n$ such that $A_n =
  \ph_n[\bv]_\Om$, and define $\ph = \bigvee_n \ph_n$. Note that $A =
  \ph[\bv]_\Om$. Suppose $m \ne n$. Since
  \[
    (\ph_m \wedge \ph_n)[\bv]_\Om = A_m \cap A_n = \emp = \bot_\Om,
  \]
  we have $\ph_m \wedge \ph_n \equiv \bot$, implying that $X \nvdash \ph_m
  \wedge \ph_n$. Therefore, either $X \nvdash \ph_m$ or $X \nvdash \ph_n$. This
  implies that there is at most one $n \in \bN$ with $\bbP A_n = 1$. Hence,
  $\sum_n \bbP A_n \in \{0, 1\}$ and
  \begin{align*}
    \ts{\sum \bbP A_n = 1}
      &\quad\text{iff}\quad \text{there exists $n$ such that $\bbP A_n = 1$}\\
    &\quad\text{iff}\quad \text{there exists $n$ such that $X \vdash \ph_n$}\\
    &\quad\text{iff}\quad X \vdash \ph\\
    &\quad\text{iff}\quad \bbP \ph[\bv]_\Om = \bbP A = 1,
  \end{align*}
  showing that $\bbP$ is countably additive. Thus, $\bbP$ is a measure on $(\Om,
  \Si)$ with $\bbP \Om = 1$, and so $\sP = (\Om, \Si, \bbP)$ is a model.

  Now let $\ph \in X$ be arbitrary. Then $X \vdash \ph$, so that $\ph[\bv]_\Om
  \in \Si$, and $\bbP \ph[\bv]_\Om = 1$. This shows that $\sP \vDash \ph[\bv]$
  for all $\ph \in X$, and $X$ is satisfiable.
\end{proof}

\begin{cor}\label{C:pred-compactness}
  A set $X \subseteq \cL$ is satisfiable if and only if $X$ is consistent.
\end{cor}

\begin{proof}
  The only if part is Corollary \ref{C:pred-soundness}. Suppose $X$ is not
  satisfiable. By Theorem \ref{T:pred-compactness}, there exists a countable
  subset $X_0 \subseteq X$ that is not satisfiable. By Proposition
  \ref{P:pred-pre-cpct}, the set $X_0$ is not strictly satisfiable. Hence, by
  Corollary \ref{C:pred-soundness}, the set $X_0$ is inconsistent, which implies
  that $X$ is inconsistent.
\end{proof}

\begin{thm}[Deductive completeness]\label{T:pred-completeness}
    \index{completeness!deductive ---}%
  For $X \subseteq \cL$ and $\ph \in \cL$, we have $X \vDash \ph$ if and only if
  $X \vdash \ph$.
\end{thm}

\begin{proof}
  The if part is Theorem \ref{T:pred-soundness}. Suppose $X \nvdash \ph$. Then
  $X \cup \{\neg\ph\}$ is consistent, by Theorem \ref{T:deduc-con}. Thus, $X
  \cup \{\neg\ph\}$ is satisfiable, by Corollary \ref{C:pred-compactness}.
  Choose $\sP$ and $\bv$ such that $\sP \vDash \psi[\bv]$ for all $\psi \in X
  \cup \{\neg \ph\}$. Then $\sP \vDash \psi[\bv]$ for all $\psi \in X$, but $\sP
  \nvDash \ph[\bv]$. Thus, $X \nvDash \ph$.
\end{proof}

\subsection{Peano arithmetic}\label{S:PA}
  \index{Peano arithmetic}%

As an example of deductive predicate logic, we present the theory of Peano
arithmetic in the infinitary setting. Let $\cL$ be a language that contains a
constant symbol $\ul 0$, a unary function symbol $\S$, and binary operation
symbols $\{+, \bdot\}$. In the language $\cL$, for each $n \in \bN$, we use the
shorthand $\ul n = \S \cdots \S \ul 0$, where $\S$ is repeated $n$ times.

Define the formulas
\begin{align*}
  &\ph_1 : \forall x \, \S x \nbeq \ul 0 &
  &\ph_2 : \forall xy (\S x \beq \S y \to x \beq y)\\
  &\ph_3 : \forall x \, x + \ul 0 \beq x &
  &\ph_4 : \forall xy \, x + \S y \beq \S(x + y)\\
  &\ph_5 : \forall x \, x \bdot \ul 0 \beq \ul 0 &
  &\ph_6 : \forall xy \, x \bdot \S y \beq x \bdot y + x
\end{align*}
For definiteness, we may assume $x = \bx_0$ and $y = \bx_1$ in the above, so
that this is a finite collection of sentences, rather than a family of formulas
indexed by $x, y \in \Var$. Note that each $\ph_i \in \cL_\fin^0$. If $\ph = \ph
(x, \vec y\,) \in \cL$, define
\[
  \ISPA(\ph) : \forall \vec y  \, (
    \ph(\ul 0/x) \wedge \forall x (\ph \to \ph(\S x / x)) \to \forall x \ph
  )
\]
Let $\ISPA = \{\ISPA(\ph) \mid \ph(x, \vec y\,) \in \cL\} \subseteq \cL^0$ and
$\ISPA_\fin = \ISPA \cap \cL_\fin^0$. Since $\ISPA(\ph)$ has finite length if
and only if $\ph$ has finite length, we have $\ISPA_\fin = \{\ISPA(\ph) \mid \ph
(x, \vec y\,) \in \cL_\fin\}$.

In first-order logic, $\La^\PA_- = \{\ph_1, \ldots, \ph_6\} \cup \ISPA_\fin$ are
the usual axioms of Peano arithmetic. The set $\ISPA_\fin$ is called the
\emph{axiom schema of induction}. We let $\PA_- = T(\La^\PA_-)$ and $\PA_\fin =
\PA_- \cap \cL^0_\fin$. By Proposition \ref{P:pred-fin-vs-infin}, we have
$\PA_\fin = \{\ph \in \cL^0_\fin \mid \La^\PA_- \vdash_\fin \ph\}$. In other
words, $\PA_\fin$ is exactly first-order Peano arithmetic.

We also define $\La^\PA = \{\ph_1, \ldots, \ph_6\} \cup \ISPA$. This differs
from $\La^\PA_-$ only in the fact that we are allowed to perform induction on
infinitary formulas. Let $\PA = T(\La^\PA)$. Then $\La^\PA_- \subseteq \La^\PA$
and $\PA_\fin \subseteq \PA_- \subseteq \PA$.
  \symindex{$\PA_\fin$}%
  \symindex{$\PA_-$}%
  \symindex{$\PA$}%

Let $\cN$ be the standard structure of arithmetic. That is, $\cN = (\bN_0, 0,
\S, +, \bdot)$, where $\S$ is the function $n \mapsto n + 1$. As usual, we will
have to rely on context to know whether $\S, +, \bdot$ are referring to objects
in the standard structure, or to symbols in the signature of $\cL$. Since $\cN
\tDash \La^\PA$, we have $\cN \tDash \PA$ by Remark \ref{R:strict-conseq}. It is
well-known that there are \emph{nonstandard structures of finitary Peano
arithmetic}. That is, there exist structures $\om$ such that $\om \tDash
\PA_\fin$ but $\om \not \simeq \cN$. As it turns out, the analogous statement is
still true for $\PA_-$ as we see below in Proposition \ref{P:models-of-PA-}. On
the other hand, Proposition \ref{P:models-of-PA} shows that it is not true for
$\PA$. In other words, $\PA$ completely characterizes the standard structure of
arithmetic, meaning that every true statement about arithmetic is provable in
$\PA$. Another way to say this, according to completeness, is that if $\ph$ is
true in the standard structure of arithmetic, then it is true in every model of
$\PA$. This is famously not the case for $\PA_-$, thanks to G\"odel's first
incompleteness theorem (see \cite[Theorem 6.5.1]{Rautenberg2010}).

\begin{prop}\label{P:models-of-PA-}
  Let $\sP = (\Om, \Si, \bbP)$ be a model. Then $\sP \vDash \PA_-$ if and only
  if $\om \tDash \PA_\fin$ for $\bbP$-a.e.~$\om \in \Om$. Consequently, $\PA_-
  \vdash \ph$ if and only if $\om \tDash \La^\PA_-$ implies $\om \tDash \ph[v]$
  for all $\om$ and all assignments $v$ into $\om$.
\end{prop}

\begin{proof}
  Suppose $\sP \vDash \PA_-$. Then $\sP \vDash \La^\PA_-$. Since $\ISPA_\fin$ is
  countable, so is $\La^\PA_-$. Hence, $\olbbP \Om^* = 1$, where $\Om^* =
  \bigcap_{\ph \in \La^\PA_-} \ph_\Om$. For every $\om \in \Om^*$, we have $\om
  \tDash \La^\PA_-$, which implies $\om \tDash \PA_\fin$.

  Conversely, suppose $\om \tDash \PA_\fin$ for $\bbP$-a.e.~$\om \in \Om$. Since
  $\La^\PA_- \subseteq \cL^0_\fin$, we have $\La^\PA_- \subseteq \PA_\fin$.
  Hence, $\om \tDash \La^\PA_-$ for $\bbP$-a.e.~$\om \in \Om$. This implies that
  $\Om^* = \Om$, $\bbP$-a.e. It follows that $\Om^* \in \ol \Si$ and $\olbbP
  \Om^* = 1$. Therefore, $\sP \vDash \La^\PA_-$, which gives $\sP \vDash \PA_-$.

  For the second claim, the only if direction follows from Theorem
  \ref{T:pred-completeness} and Remark \ref{R:strict-conseq}. For the if
  direction, suppose $\om \tDash \ph[v]$ for all $\om$ and all assignments $v$
  into $\om$. Let $\sP = (\Om, \Si, \bbP) \vDash \PA_-$ and let $\bv =
  \ang{v_\om}$ be an assignment into $\sP$. By the above, $\om \tDash \La^\PA_-$
  for $\bbP$-a.e.~$\om \in \Om$. By hypothesis, $\om \tDash \ph[v_\om]$ for
  $\bbP$-a.e.~$\om \in \Om$. Hence, $\olbbP \ph[\bv]_\Om = 1$, so that $\sP
  \vDash \ph[\bv]$. By Theorem \ref{T:pred-completeness}, this gives $\PA_-
  \vdash \ph$.
\end{proof}

\begin{prop}\label{P:models-of-PA}
  Let $\sP = (\Om, \Si, \bbP)$ be a model. Then $\sP \vDash \PA_\infty$ if and
  only if $\om \simeq \cN$ for $\bbP$-a.e.~$\om \in \Om$. Consequently, for all
  $\ph \in \cL^0$, if $\cN \tDash \ph$, then $\PA_\infty \vdash \ph$.
\end{prop}

\begin{proof}
  For the first claim, the if direction follows from the fact that $\om \simeq
  \cN$ implies $\om \tDash \PA_\infty$. For the only if direction, define the
  formula $\ph(x) = (\bigvee_{n \in \bN_0} x \beq \ul n)$. Suppose $\sP \vDash
  \PA_\infty$. Then $\sP \vDash \La^\PA \cup \{\ISPA (\ph)\}$, which is
  countable. Therefore, $\om \tDash \La^\PA \cup \{\ISPA (\ph)\}$ for
  $\bbP$-a.e.~$\om \in \Om$. Choose any such $\om$. Note that $n \mapsto \ul
  n^\om$ is an embedding of $\cN$ into $\om$.

  Clearly, $\om \tDash \ph(\ul 0 / x)$. By the definition of $\ul n$, if $a$ is
  in the domain of $\om$ and $n \in \bN_0$, then $\om \tDash (x \beq \ul n)[a]$
  implies $\om \tDash (\S x \beq \ul{n + 1}) [a]$. Hence, $\om \tDash \ph[a]$
  implies $\om \tDash \ph(\S x / x)[a]$. Since $a$ was arbitrary, we have $\om
  \tDash \forall x (\ph \to \ph(\S x / x))$. It therefore follows that $\om
  \tDash \forall x \ph$. Hence, the map $n \mapsto \ul n^\om$ is surjective, and
  so it is an isomorphism from $\cN$ to $\om$.

  Finally, let $\ph \in \cL^0$ and suppose $\cN \tDash \ph$. Let $\sP = (\Om,
  \Si, \bbP)$ be a model with $\sP \vDash \PA_\infty$. By the above result and
  Theorem \ref{T:invar-thm}, we have $\om \tDash \ph$ for $\bbP$-a.e.~$\om$.
  Hence, $\olbbP \ph_\Om = 1$, so that $\sP \vDash \ph$. By Theorem
  \ref{T:pred-completeness}, this gives $\PA_\infty \vdash \ph$.
\end{proof}

\subsection{Inductive consequence and completeness}

If $\sP$ is a model, we define $\Th \sP = \{\ph \in \cL^0 \mid \sP \vDash
\ph\}$.
  \symindex{$\Th \sP$ (in $\cL$)}%
The proof of Proposition \ref{P:ThP-is-theory} is valid here, and shows that
$\Th \sP$ is a consistent deductive theory. For $(X, \ph, p) \in \cL^\IS$, we
say that \emph{$\sP$ satisfies $(X, \ph, p)$}, denoted by $\sP \vDash (X, \ph,
p)$ if $X \equiv Y \cup \{\psi\}$ for some $Y \subseteq \Th \sP$ and some $\psi
\in \cL^0$ with $\olbbP \ph_\Om \cap \psi_\Om / \olbbP \psi_\Om = p$.
  \index{satisfiable}%
  \symindex{$\sP \vDash_\cL (X, \ph, p)$}%

As with the inductive calculus, the results in Section \ref{S:ind-sem} depend
only on deductive completeness and the fact that $\vdash_\cF$ satisfies (i)--%
(vi) of Definition \ref{D:derivability}. Hence, all of the proofs in that
section go through in the predicate case, with $\cF$ replaced by $\cL^0$,
``strict model'' replaced by ``structure,'' and $\B^{PV}$ replaced by the set
$\Om$ in Proposition \ref{P:all-structures}. We adopt all of the notation and
terminology of Section \ref{S:ind-sem} to define inductive consequence in
$\cL^\IS$, extend it to inductive conditions, and establish completeness.

Similarly, all of the results in Sections \ref{S:ind-th-Dynk} and
\ref{S:dialog}--\ref{S:sem-indep} carry through with the above three
replacements. We therefore adopt all of the notation and terminology of those
sections to define independence and its related notions.

To all of this, we add the following.

\begin{thm}[Inductive isomorphism theorem]\label{T:ind-iso-thm}
    \index{isomorphism theorem!inductive ---}%
  Let $\sP$ and $\sQ$ be isomorphic models. Then $\sP \vDash (X, \ph, p)$ if and
  only if $\sQ \vDash (X, \ph, p)$, for all $(X, \ph, p) \in \cL^\IS$.
\end{thm}

\begin{proof}
  Let $\sP = (\Om, \Si, \bbP)$ and $\sQ = (\Om', \Ga, \bbQ)$ be isomorphic and
  suppose $\sP \vDash (X, \ph, p)$. Then $X \equiv Y \cup \{\psi\}$, where $\sP
  \vDash Y$ and $\olbbP \ph_\Om \cap \psi_\Om / \olbbP \psi_\Om = p$. By Theorem
  \ref{T:ded-iso-thm}, we have $\sQ \vDash Y$. Lemma \ref{L:iso-thm} implies
  $\olbbQ \psi_{\Om'} = \olbbP \psi_\Om$ and $\olbbQ \ph_{\Om'} \cap \psi_
  {\Om'} = \olbbQ (\ph \wedge \psi)_{\Om'} = \olbbP (\ph \wedge \psi)_\Om =
  \olbbP \ph_\Om \cap \psi_\Om$. We therefore have $\olbbQ \ph_{\Om'} \cap \psi_
  {\Om'} / \olbbQ \psi_{\Om'} = p$, so that $\sQ \vDash (X, \ph, p)$.
\end{proof}

\section{Predicate models and random variables}\label{S:pred-models-RVs}

In this section, we discuss the relationship between predicate models and random
variables. Here, random variable is meant in the usual sense of
measure-theoretic probability theory. That is, a random variable is a measurable
function, defined on a probability space, taking values in a measurable space.
Our first goal will be to prove a predicate analogue of Theorem
\ref{T:prob-sp-model-iso}, which states that every probability space is
isomorphic to a propositional model. We begin by establishing the connection
between propositional and predicate models.

\begin{prop}\label{P:prop-embed}
  Every propositional model is isomorphic to a predicate model. More
  specifically, let $\cF$ be a given propositional language with propositional
  variables $PV$, and let $\sP$ be a model in $\cF$. Then there exists a
  predicate language $\cL$ and an $\cL$-model $\sQ$ such that $\sP_\cF$ and
  $\sQ$ are isomorphic as measure spaces.
\end{prop}

\begin{proof}
  Let $\sP = (\Om, \Si, \bbP)$ be a propositional model in $\cF$. Let $\al =
  |PV|$ and write $PV = \ang{\bfr_\de \mid \de < \al}$. Let $\{r_\de \mid \de <
  \al\}$ be a set of distinct unary relation symbols. Let $\rho$ be a constant
  symbol, which we call the propositional constant. Define the extralogical
  signature $L = \{\rho\} \cup \{r_\de \mid \de < \al\}$, and let $\cL$ be the
  associated predicate language.

  Given a strict propositional model $\om \in \Om$, we define the
  $\cL$-structure $\ol\om$ as follows. The domain of $\ol\om$ will be $A = \B^
  {PV}$. Let $r_\de^{\ol\om} = \{\nu \in A \mid \nu \tDash_\cF \bfr_\de\}$, and
  let $\rho^{\ol\om} = \om$. Let $\ol\Om = \{\ol\om \mid \om \in \Om\}$ and let
  $\sQ = (\ol\Om, \Ga, \bbQ)$ be the measure space image of $\sP_\cF = (\Om, \ol
  \Si_\cF, \olbbP_\cF)$ under the function $h$ mapping $\om$ to $\ol\om$. Then
  $\sQ$ is an $\cL$-model.

  Define $\tau: PV \to \cL$ by $\bfr_\de^\tau = r_\de \rho$. Extend $\tau$
  recursively to $\cF$ by $(\neg \ph)^\tau = \neg \ph^\tau$ and $(\bigwedge
  \Phi)^\tau = \bigwedge_{\ph \in \Phi} \ph^\tau$. We then have $\om \tDash_\cF
  \ph$ if and only if $\ol\om \tDash_\cL \ph^\tau$, for all $\ph \in \cF$. This
  is clear by construction when $\ph \in PV$. It then follows easily by formula
  induction on $\ph$.

  Now let $A \in \ol \Si_\cF = \ol \Si \cap \cB^{PV}$. Choose $\ph \in \cF$ such
  that $A = \ph_\Om$. Define $U = \ph^\tau_{\Om'}$. Then $\om \in A$ if and only
  if $\om \tDash_\cF \ph$, and $\om \in h^{-1} U$ if and only if $\ol\om
  \tDash_\cL \ph^\tau$. Hence, $A = h^{-1} U$, so that $U \in \Ga$. Since $A$
  was arbitrary, this shows that $h$ induces a measure-space isomorphism from
  $\sP_\cF$ to $\sQ$.
\end{proof}

\subsection{Random variables as extralogical symbols}\label{S:RV-symb}

In Theorem \ref{T:prob-sp-model-iso}, we showed that every probability space is
isomorphic to a propositional model. Conversely, every propositional model is a
probability space. In this sense, then, measure-theoretic probability theory is
exactly the semantics of propositional inductive logic. But this simple
observation misses an important point. While the propositional version of
inductive logic is capable of representing any probability space, it does not
explicitly represent any random variables.

The reason this matters is that measure-theoretic probability theory is more
than just probability spaces. The modern practitioner almost always specializes
in a particular class of random variables and stochastic processes. For this
reason, we define the following. A \emph{measure-theoretic probability model} is
a tuple, $(S, \Ga, \nu, X)$, where $(S, \Ga, \nu)$ is a probability space, $X =
\ang{X_i \mid i \in I}$ is an indexed collection of random variables, and $\Ga =
\si(\ang{X_i \mid i \in I})$. That is, for each $i \in I$, there is a measurable
space $(R_i, \Ga_i)$ such that $X_i: S \to R_i$ is measurable, and $\Ga$ is the
smallest $\si$-algebra on $S$ that contains $\{X_i \in V\}$ for every $i \in I$
and every $V \in \Ga_i$.

We aim to prove a predicate analogue of Theorem \ref{T:prob-sp-model-iso}, and
show that every measure-theoretic probability model has a natural correspondence
to an $\cL$-model, where the logical signature $L$ is directly connected to the
random variables $X$.

We construct the logical signature as follows. Let $R = \bigcup_{i \in I} R_i$.
Let $\{\ul r \mid r \in R\}$ be a set of distinct constant symbols, and $\{\ul V
\mid i \in I, V \in \Ga_i\}$ a set of distinct unary relation symbols. Let
\[
  L_R = \{\ul r \mid r \in R\} \cup \{\ul V \mid i \in I, V \in \Ga_i\},
\]
and let $\cL_R$ be the associated predicate language. Define the $L_R$-structure
$\cR = (R, L^\cR)$ by $\ul r^\cR = r$ and $\ul V^\cR = V$. Let $T_R = \{\ph \in
\cL^0 \mid \cR \tDash \ph\}$. Then $T_R$ is a deductive theory. In $\cL_R$, we
write $y \in \ul V$ as shorthand for $\ul V \, y$, and $y \notin \ul V$ as
shorthand for $\neg \ul V \, y$. Let $C = \{\ul X_i \mid i \in I\}$ be a set of
distinct constant symbols not in $L_R$, and define $L = L_R C$.

\begin{thm}\label{T:prob-model-iso}
  There exists an $\cL$-model $\sP = (\Om, \Si, \bbP)$ with $\sP \vDash T_R$,
  and a function $h: S \to \Om$ mapping $x \in S$ to $\om \in \Om$ such that
  \begin{enumerate}[(i)]
    \item $x \in \{X_i \in V\}$ if and only if $\om \tDash \ul X_i \in \ul V$,
    \item each $U \in \Ga$ can be written as $U = h^{-1} \ph_\Om$ for some $\ph
          \in \cL^0$, and
    \item $h$ induces a measure-space isomorphism from $(S, \Ga, \nu)$ to $\sP$.
  \end{enumerate}
  Consequently, if $P = \bTh \sP \dhl_{[T_R, \Th \sP]}$, then
  \begin{equation}\label{prob-model-iso}
    \ts{
      P(\bigwedge_{k = 1}^n \ul X_{i(k)} \in \ul V_k \mid T_R) =
      \nu \bigcap_{k = 1}^n \{X_{i(k)} \in V_k\},
    }
  \end{equation}
  whenever $i(1), \ldots, i(n) \in I$ and $V_k \in \Ga_{i(k)}$.
\end{thm}

\begin{proof}
  For each $x \in S$, define $\om = \om^x$ to be the $L$-expansion of $\cR$
  given by $\om^{\ul X_i} = X_i(x)$. Let $\Om = \{\om^x \mid x \in S\}$ and let
  $h: S \to \Om$ denote the map $x \mapsto \om^x$. Let $\sP = (\Om, \Si, \bbP)$
  be the measure space image of $(S, \Ga, \nu)$ under $h$. Since $\om \tDash
  T_R$ for all $\om \in \Om$, we have $\sP \vDash T_R$. By construction, we have
  $X_i(x) \in V$ if and only if $\om^x \tDash \ul X_i \in \ul V$, so (i) holds.

  For (ii), let
  \[
    \Ga' = \{
      U \in \Ga
    \mid
      U = h^{-1} \ph_\Om \text{ for some } \ph \in \cL^0
    \}.
  \]
  Since $\bigcup_n h^{-1} (\ph_n)_\Om = h^{-1} (\bigvee_n
  \ph_n)_\Om$ and $\bot_\Om = \emp$, we have that $\Ga'$ is a $\si$-algebra. Let
  $V \in \Ga_i$. Since $X_i(x) \in V$ if and only if $\om^x \tDash \ul X_i \in
  \ul V$, it follows that $\{X_i \in V\} = h^{-1} (\ul X_i \in \ul V)_\Om$.
  Therefore, $\{X_i \in V\} \in \Ga'$. Since $\Ga = \si(\ang{X_i \mid i \in
  I})$, this proves that $\Ga = \Ga'$, so (ii) holds.

  If $U \in \Ga$, $\ph \in \cL^0$, and $U = h^{-1} \ph_\Om$, then by the
  construction of $\sP$, we have $\ph_\Om \in \Si$. Therefore, (ii) implies 
  (iii).

  Finally, since $h$ also induces an isomorphism from $(S, \ol \Ga, \ol \nu)$ to
  $ (\Om, \ol \Si, \olbbP)$, we have $\olbbP = \ol \nu \circ h^{-1}$. This gives
  \[
    \ts{
      P(\bigwedge_{k = 1}^n \ul X_{i(k)} \in \ul V_k \mid T_R) =
      \olbbP \bigcap_{k = 1}^n (\ul X_{i(k)} \in \ul V_k)_\Om =
      \nu \bigcap_{k = 1}^n \{X_{i(k)} \in V_k\},
    }
  \]
  which verifies \eqref{prob-model-iso}.
\end{proof}

We excluded the case where $\si(\ang{X_i \mid i \in I})$ is a proper subset of
$\Ga$. If we wish to treat this case, we can simply add $\iota$, the identity
function on $S$, to our list of random variables. Note, however, that in this
case, there are events $U \in \Ga$ that have nothing to do with any of the
random variables $X_i$. These are analogous to propositional sentences in the
sense that they are generic assertions that lack structure. If we add $\iota$ to
our list of random variables, and $\rho = \ul \iota \in L$ is the constant
symbol that represents $\iota$, then $\rho$ is playing the same role as the
propositional constant in the proof of Proposition \ref{P:prop-embed}. We see,
then, that we are effectively treating every event $U \in \Ga \setminus
\si(\ang{X_i \mid i \in I})$ as if it were a propositional variable.

Theorem \ref{T:prob-model-iso} shows that every measure-theoretic probability
model is an inductive model. In other words, the whole of measure-theoretic
probability theory is embedded in the semantics of inductive logic. But Theorem
\ref{T:prob-model-iso} says more than just this. It exhibits a particular
embedding. The function $h$ in Theorem \ref{T:prob-model-iso} gives us a logical
interpretation for each component of a measure-theoretic probability model. With
this interpretation, we have the following correspondences.

\begin{center}
  \begin{tabular}{ l l l }
     \emph{Measure Theory} && \emph{Inductive Logic} \\ 
     \hline
     outcome && structure \\
     event && sentence \\
     set membership && strict satisfiability \\
     random variable && constant symbol
  \end{tabular}
\end{center}

\subsection{Extralogical symbols as functions}\label{S:symb-func}

An $\cL$-structure is, in fact, a function whose domain is $L$. If $\om$ is an
$\cL$-structure, then it maps each $\s \in L$ to the object $\s^\om$. Hence, if
$\sP = (\Om, \Si, \sP)$ is an $\cL$-model, then each structure $\om \in \Om$ is
a function that maps the symbol $\ul X$ to the object $\ul X^\om$. This is
exactly the opposite of what we have in measure-theoretic probability theory,
where each random variable $X$ is a function that maps the outcome $\om$ to the
object $X(\om)$.

Starting with an $\cL$-model, $\sP = (\Om, \Si, \bbP)$, we may wish to reverse
the natural direction of the mapping, and think of the extralogical symbols $\s
\in L$ as functions defined on $\Om$. We can do that as follows. If $A_\om$ is
the domain of $\om \in \Om$, then a constant symbol $c$ gives rise to the
function $X^c(\om) = c^\om$, mapping $\Om$ to $A = \bigcup_{\om \in \Om} A_\om$.
An $n$-ary relation symbol can be viewed as an indexed collection of $\{0,
1\}$-valued functions, indexed by $A^n$. Namely, for each $\vec a \in A^n$, we
have $X^r_{\vec a}(\om) = 1$ if $\vec a \in r^\om$, and $0$ otherwise. For an
$n$-ary function symbol $f$, we can add a so-called ``cemetery point'' to $A$.
Let $\pa$ be an object not in $A$. Then $f$ provides us with an indexed
collection of $A \cup \{\pa\}$-valued functions, indexed by $A^n$. That is,
\[
  X^f_ {\vec a} (\om) = \begin{cases}
    f^\om \vec a &\text{if $\vec a \in A_\om^n$},\\
    \pa &\text{if $\vec a \notin A_\om^n$}.
  \end{cases}
\]
These functions are, of course, not measurable. In fact, the set $A$ is not even
equipped with a $\si$-algebra, and there may not be a natural $\si$-algebra on
$A$ that is compatible with $\Si$. Without measurability, we cannot use the
well-established theory of random variables to analyze the functions determined
by the extralogical symbols. On the other hand, without the requirement of
measurability, we are able to model situations that are not possible with random
variables. See, for instance, Example \ref{Expl:pred-Karp412} below.

\subsection{The relativity of randomness}\label{S:rel-rand}

Let $\sP = (\Om, \Si, \bbP)$ be an $\cL$-model and let $\ul X \in L$ be a
constant symbol. If we try to think of $\om \mapsto \ul X^\om$ as a kind of
non-measurable random variable, then we run into a problem deeper than its
non-measurability. The problem we face is that \emph{every} extralogical symbol
is a non-measurable random variable. In a measure-theoretic probability model,
if we are faced with a probability of the form $\nu \{X > 0\}$, then we can be
quite certain that the only thing random is $X$. But in an inductive model, the
analogous expression is $\olbbP \{\om \in \Om \mid X^\om >^\om 0^\om\}$. Not
only can the value of $0$ vary with $\om$, the inequality relation itself can
also depend on $\om$.

This phenomenon can be seen in a very simple example. Let $L = \{\ul h, \ul t,
\ul X\}$ be a set of constant symbols and $\cL$ the associated predicate
language. We think of $\ul h$ and $\ul t$ as denoting the heads and tails sides
of a coin, and $\ul X$ the result of flipping the coin. Let $T_0 \subseteq
\cL^0$ be the deductive theory generated by the sentences, $\ul h \nbeq \ul t$
and $\ul X \beq \ul h \vee \ul X \beq \ul t$. We may think of $T_0$ as
describing our state of knowledge prior to flipping the coin. Namely, the two
sides of the coin are distinct, and the coin will not land on its edge.

Let $P$ be the inductive theory generated by
\[
  P(\ul X \beq \ul h \mid T_0) = P(\ul X \beq \ul t \mid T_0) = 1/2.
\]
This, of course, represents our assumption that the coin is fair.

Intuitively, we imagine that $\ul h$ and $\ul t$ are fixed, whereas $\ul X$ is
random. We can satisfy $P$ with a model that matches this intuition. Let $A = 
\{0, 1\}$. Define the $L$-structure $\om_0$ by $\ul h^{\om_0} = 1$, $\ul t^
{\om_0} = 0$, and $\ul X^{\om_0} = 0$. Define the $L$-structure $\om_1$ by $\ul
h^{\om_1} = 1$, $\ul t^{\om_1} = 0$, and $\ul X^{\om_1} = 1$. Let $\Om = 
\{\om_0, \om_1\}$, $\Si = \fP \Om$, and $\bbP \{\om_0\} = \bbP \{\om_1\} = 1/2$.
Then $\sP = (\Om, \Si, \bbP) \vDash P$.

Under $\sP$, the symbol $\ul h$ corresponds to the function $\om \mapsto \ul
h^\om = 1$, and the symbol $\ul t$ corresponds to the function $\om \mapsto
\ul t^\om = 0$. In other words, $\ul h$ and $\ul t$ are identified with constant
functions, and are therefore fixed. On the other hand, $\ul X$ corresponds to
$\om \mapsto \ul X^\om$, which is $0$ with probability $1/2$ and $1$ with
probability $1/2$. Hence, $\ul X$ is random.

However, we can also satisfy $P$ with a model that violates this intuition. Let
$\om_0$ be as above. Define $\om_1'$ by $\ul h^{\om_1'} = 0$, $\ul t^{\om_1'} =
1$, and $\ul X^{\om_1'} = 0$. Let $\Om' = \{\om_0, \om_1'\}$, $\Si' = \fP \Om'$,
and $\bbP' \{\om_0\} = \bbP' \{\om_1'\} = 1/2$. Then $\sP' = (\Om', \Si', \bbP')
\vDash P$. This time, however, $\ul h$ and $\ul t$ correspond to functions that
are $0$ or $1$ with equal probability, and $\ul X$ corresponds to the constant
function $0$. In this model, it is $\ul h$ and $\ul t$ that are random, while
$\ul X$ is fixed.

Since $P$ is satisfied by both $\sP$ and $\sP'$, we see that $P$ does not tell
us which terms are random and which terms are fixed. In fact, it is not even
meaningful to ask this question in $P$. The only things in $P$ which can be
random (that is, the only things which can be assigned a probability that is not
$0$ or $1$) are sentences. In order to even ask this question, we must fix a
model. And in fixing a model, we are adopting, so to speak, a point of view.
Which terms are random and which are fixed is relative to that point of view. In
$\sP$, we take the point of view that $\ul h$ and $\ul t$ are fixed, while $\ul
X$ is random. And in $\sP'$, we take the point of view that $\ul X$ is fixed,
while $\ul h$ and $\ul t$ are random. There are models in which all three are
random. In this example, however, there are no models in which all three are
fixed.

In general, then, whether a term in $P$ is random or fixed depends on our point
of view, or to borrow the language of physics, it depends on our frame of
reference.

\subsection{Frames of reference}

A \emph{frame of reference}
  \index{frame of reference}%
is a method that takes a given $\cL$-model $\sP$ and constructs a new
$\cL$-model $\sP'$ such that $\sP \simeq \sP'$. Formally, we could define a
frame of reference to be a class function that maps each $\sP$ in a certain
class of $\cL$-models to a set of $\cL$-models that are isomorphic to $\sP$.
This level of formalism, however, will not be necessary for our purposes.

Let $P$ be the inductive theory in Section \ref{S:rel-rand} that models a fair
coin flip. Proposition \ref{P:fix-constants} below gives a method of taking any
$\cL$-model $\sP$ such that $\sP \vDash P$, and constructing an isomorphic model
$\sP'$ in which the functions $\om \mapsto \ul h^\om$ and $\om \mapsto \ul
t^\om$ are constant functions. In other words, there is a frame of reference in
which $\ul h$ and $\ul t$ are fixed, and not random.

We begin by showing there is a frame of reference in which every object is an
ordinal. A model $\sP = (\Om, \Si, \bbP)$ is said to be an \emph{ordinal model}
if, for all structures $\om \in \Om$, the domain of $\om$ is an ordinal.

\begin{lemma}\label{L:ord-FOR}
  Let $\sP = (\Om, \Si, \bbP)$ be a model. For each $\om \in \Om$, let $A_\om$
  be the domain of $\om$, and let $B_\om$ be a set with $|B_\om| = |A_\om|$.
  Choose a bijection $g_\om: A_\om \to B_\om$, and let $\om'$ be the isomorphic
  image of $\om$ under $g_\om$. Let $\Om' = \{\om' \mid \om \in \Om\}$ and let
  $h: \Om \to \Om'$ denote the function $\om \mapsto \om'$. Define $\sQ = (\Om',
  \Ga, \bbQ)$ to be the measure space image of $\sP$ under $h$. Then $h$ is a
  model isomorphism from $\sP$ to $\sQ$.
\end{lemma}

\begin{proof}
  To verify that $h$ is an isomorphism from $\sP$ to $\sQ$, it suffices to check
  that $h$ induces an isomorphism as measure spaces from $ (\Om, \Si_\cL,
  \bbP_\cL)$ to $(\Om', \Ga_\cL, \bbQ_\cL)$. For this, it suffices to show that
  for all $A \in \Si_\cL$, there exists $U \in \Ga_\cL$ such that $h^{-1} U =
  A$.

  Let $A \in \Si_\cL$. Choose $\ph \in \cL$ and choose an assignment $\bv$ into
  $\sP$ such that $A = \ph[\bv]_\Om$. Define the assignment $\bv'$ into $\sQ$ by
  $v'_{\om'}(x) = g_\om v_\om(x)$, and let $U = \ph[\bv']_{\Om'}$. It now
  suffices to show that $U \in \Ga_\cL$ and $h^{-1} U = A$. But $\sQ$ is the
  measure space image of $\sP$ under $h$. Hence, if $h^{-1} U = A \in \Si_\cL
  \subseteq \Si$, then $U \in \Ga$, which implies $U \in \Ga_\cL$ by the
  definition of $\Ga_\cL$. Therefore, we need only show that $h^{-1} U = A$.

  Note that $\om \in h^{-1} U = h^{-1} \ph[\bv']_{\Om'}$ if and only if $\om'
  \tDash \ph[v'_{\om'}]$. Similarly, $\om \in A = \ph[\bv]_\Om$ if and only
  $\om \tDash \ph[v_\om]$. By Theorem \ref{T:invar-thm}, we have $\om \tDash \ph
  [v]$ if and only if $\om' \tDash \ph[v']$. Thus, $h^{-1} U = A$, and so $h$ is
  an isomorphism.
\end{proof}

\begin{prop}[Ordinal frame of reference]\label{P:ord-FOR}
  Every model is isomorphic to an ordinal model.
\end{prop}

\begin{proof}
  Let $\sP = (\Om, \Si, \bbP)$ be a model. For each $\om \in \Om$, let $A_\om$
  be the domain of $\om$. Choose an ordinal $\al_\om$ such that $|\al_\om| =
  |A_\om|$, and choose a bijection $g_\om: A_\om \to \al_\om$. Define $\sQ$ as
  in Lemma \ref{L:ord-FOR}. Then $\sP \simeq \sQ$ and $\sQ$ is an ordinal model.
\end{proof}

\begin{lemma}\label{L:fix-constants}
  Let $\al$ be an ordinal and $S \subseteq \al$. Let $\be$ and $\ga$ be ordinals
  such that $|\be| = |S|$ and $|\ga| = |\al \setminus S|$, and let $g: S \to
  \be$ be a bijection. Then $g$ can be extended to a bijection $g: \al \to \be
  + \ga$.
\end{lemma}

\begin{proof}
  Let $g: S \to \be$ be a bijection. Choose a bijection $h: \al \setminus S \to
  \ga$. Note that the function $f: \ga \to (\be + \ga) \setminus \be$ given by
  $f \xi = \be + \xi$ is a bijection. Therefore, $f \circ h: \al \setminus S \to
  (\be + \ga) \setminus \be$ is a bijection. Hence, if we define $g \xi = f h
  \xi$ for $\xi \in \al \setminus S$, then $g: \al \to \be + \ga$ is a
  bijection.
\end{proof}

\begin{prop}[Constant frame of reference]\label{P:fix-constants}
  Let $\cL$ be a predicate language with extralogical signature $L$. Let $C = 
  \{c_0, c_1, \ldots\} \subseteq L$ be a countable (possible finite) set of
  constant symbols. Let $T \subseteq \cL^0$ be a deductive theory. Assume that
  $T \vdash c_m \nbeq c_n$ for all $m \ne n$. Then for all models $\sP$ such
  that $\sP \vDash T$, there exists an ordinal model $\sP' = (\Om', \Si',
  \bbP')$ such that $\sP \simeq \sP'$ and $c_n^\om = n$ for every $\om \in
  \Om'$.
\end{prop}

\begin{proof}
  Suppose that $\sP \vDash T$. By Proposition \ref{P:ord-FOR}, we may assume
  that $\sP$ is an ordinal model. Let $\ph = (\bigwedge_{m \ne n} c_m \nbeq
  c_n)$. We then have $T \vdash \ph$, so that $\sP \vDash \ph$. Let $\om \in
  \Om$ and let $\al_\om$ denote the domain of $\Om$. We define an ordinal
  $\al_\om'$ and a bijection $g_\om: \al_\om \to \al_\om'$ as follows. If $\om
  \notin \ph_\Om$, then let $\al_\om' = \al_\om$ and let $g_\om$ be the
  identity. Suppose $\om \in \ph_\Om$. Then $\om \tDash \ph$, which means
  $c_m^\om \ne c_n^\om$ for all $m \ne n$.

  Let $\be = |C|$, so that either $\be = \{0, 1, \ldots, N\}$ or $\be = \bN_0$.
  Then $C = \{c_n \mid n \in \be\}$. Define $S = \{c_n^\om \mid n \in \be\}
  \subseteq \al_\om$. Since $c_m^\om \ne c_n^\om$ for all $m \ne n$, we have
  $|\be| = |S|$ and $n \mapsto c_n^\om$ is a bijection from $\be$ to $S$. Define
  $g_\om c_n^\om = n$, so that $g_\om: S \to \be$ is a bijection. Choose an
  ordinal $\ga$ such that $|\ga| = |\al_\om \setminus S|$ and define $\al_\om' =
  \be + \ga$. By Lemma \ref{L:fix-constants}, we may choose an extension of
  $g_\om$ to $\al_\om$ such that $g_\om: \al_\om \to \al_\om'$ is a bijection.

  Having constructed $\al_\om'$ and $g_\om$, we now define the ordinal model
  $\sQ = (\Om', \Ga, \bbQ)$ as in Lemma \ref{L:ord-FOR}, so that $\sP \simeq
  \sQ$. By Theorem \ref{T:ded-iso-thm}, we have $\sQ \vDash \ph$. Hence,
  $\olbbQ \ph_{\Om'} = 1$. Let $\om' \in \ph_{\Om'}$. Then $\om' \tDash \ph$,
  which implies $\om \tDash \ph$, since $\om' \simeq \om$. Therefore, $\om \in
  \ph_\Om$, and it follows that $c_n^{\om'} = g_\om c_n^{\om} = n$. Hence, $c_n^
  {\om'} = n$ for $\bbQ$-a.e.~$\om' \in \Om'$. By Remark \ref{R:a.s.-sure}, we
  may assume that $c_n^ {\om'} = n$ for every $\om' \in \Om'$.
\end{proof}

\subsection{The natural frame of reference}

Consider an inductive theory $P$ with root $T_0$ such that $\PA_- \subseteq
T_0$. In $P$, we may have inductive statements of the form $P(\ul X > \ul n \mid
T_0) = p$, where $>$ is the extralogical symbol defined by $\forall xy (x > y
\tot (\exists z \nbeq 0) \, x \beq y + z)$. In any model $\sP = (\Om, \Si,
\bbP)$ that satisfies $P$, we then have $\olbbP (\ul X > \ul n)_\Om = p$. But
\[
  (\ul X > \ul n)_\Om = \{\om \in \Om \mid \ul X^\om >^\om \ul n^\om\}.
\]
Hence, the very meaning of $\ul n$ and $>$ in the model $\sP$ may vary with
$\om$. If we carry with us an intuition that was built around a study of random
variables, then this situation is highly counterintuitive. We are not accustomed
to thinking of positive integers as random, let alone thinking of $>$ as a
random relation. But recall that the randomness or fixedness of these symbols is
not an inherent property of the inductive theory $P$ that we started it. It is
relative to the model we are considering.

Theorem \ref{T:fix-naturals} below, which is an immediate consequence of
Proposition \ref{P:fix-constants}, shows that for any such inductive theory $P$,
there is a frame of reference in which all the constant symbols $\ul n$ are
fixed. We call this the \emph{natural frame of reference}.
  \index{frame of reference!natural ---}%
If we replace $\PA_-$ with $\PA$, then this also fixes $>$. Indeed, in this
case, $>$ can be defined explicitly by
\[
  \ts{
    \forall xy (
      x > y \tot \bigvee_{n > m} (x \beq \ul n \wedge y \beq \ul m)
    )
  }.
\]
This is because $\PA \vdash \forall x \bigvee_{n \in \bN_0} x \beq \ul n$, which
was demonstrated in the proof of Proposition \ref{P:models-of-PA}.

To state the formal theorem, let $\cL$ be a language that contains a unary
function symbol $\S$ and constant symbols $\{\ul n \mid n \in \bN_0\}$. A
deductive theory $T \subseteq \cL^0$ is said to \emph{contain the counting
numbers} if
\begin{align*}
  T &\vdash \forall x \, \S x \nbeq \ul 0,\\
  T &\vdash \forall xy (\S x \beq \S y \to x \beq y), \text{ and}\\
  T &\vdash \ul n \beq \S \cdots \S \ul 0, \text{ for all $n \in \bN_0$}.
\end{align*}
In the last condition, the symbol $\S$ is repeated $n$ times.

\begin{thm}[Natural frame of reference]\label{T:fix-naturals}
  Let $T \subseteq \cL^0$ be a deductive theory that contains the counting
  numbers. Then for all models $\sP$ such that $\sP \vDash T$, there exists an
  ordinal model $\sP' = (\Om', \Si', \bbP')$ such that $\sP \simeq \sP'$ and
  $\ul n^\om = n$ for every $\om \in \Om'$.
\end{thm}

\begin{proof}
  Since $T$ contains the counting numbers, we have $T \vdash \ul n \nbeq \ul m$
  for all $n \ne m$. The theorem therefore follows from Proposition
  \ref{P:fix-constants}.
\end{proof}

\begin{expl}\label{Expl:pred-Karp412}
  Let $I$ be an uncountable set. Let
  \[
    L = \{\ul X_t \mid t \in I\} \cup \{\ul n \mid n \in \bN_0\}
  \]
  be a set of constant symbols, and $\cL$ the associated predicate language.
  Define $X \subseteq \cL^0$ by
  \begin{multline*}
    X = \{\ul m \nbeq \ul n \mid m, n \in \bN_0, m \ne n\}
      \cup \{\ts{\bigvee_{n \in \bN_0}} \, \ul X_t \beq \ul n \mid t \in I\}\\
    \cup \{\ul X_s \nbeq \ul X_t \mid s, t \in I, s \ne t\}.
  \end{multline*}
  Let $X_0 \subseteq X$ be countable and choose $I_0 = \{t_0, t_1, t_2,
  \ldots\}$ such that
  \begin{multline*}
    X_0 \subseteq \{\ul m \nbeq \ul n \mid m, n \in \bN_0, m \ne n\}
      \cup \{\ts{\bigvee_{n \in \bN_0}} \, \ul X_t \beq \ul n \mid t \in I_0\}\\
    \cup \{\ul X_s \nbeq \ul X_t \mid s, t \in I_0, s \ne t\}.
  \end{multline*}
  Define the $\cL$-structure $\om$ with domain $A = \bN_0$ by $\ul n^\om = n$,
  $\ul X_t^\om = n$ if $t = t_n$, and $X_t^\om = 0$ otherwise. Then $\om \tDash
  X_0$, so that $X_0$ is satisfiable, by Proposition \ref{P:pred-pre-cpct}.
  Since $X_0$ was arbitrary, Theorem \ref{T:pred-compactness} implies $X$ is
  satisfiable. Choose a model $\sP = (\Om, \Si, \bbP)$ such that $\sP \vDash X$.
  By Proposition \ref{P:fix-constants}, we may assume that $\ul n^\om = n$ for
  all $\om \in \Om$.

  Note that for each $t \in I$, we have $\sP \vDash \bigvee_{n \in \bN_0} \, \ul
  X_t \beq \ul n$. Hence, $\ul X_t^\om \in \bN_0$ for $\bbP$-a.e.~$\om \in \Om$.
  Moreover, if $T_0 = T(X)$ and $P = \bTh \sP \dhl_{[T_0, \Th \sP]}$, then
  $P(\ul X_s \ne \ul X_t \mid T_0) = 1$ for all $s \ne t$. This should be
  contrasted with the observation made in Remark \ref{R:Karp412}. Namely, there
  is no $\bN_0$-valued stochastic process $\ang{Y (t) \mid t \in I}$ such that
  $Y(s) \ne Y(t)$ a.s.~for all $s \ne t$.
\end{expl}

%% file: real-ind-th.tex

\chapter{Real inductive theories}\label{Ch:real-ind-ths}

By a ``real inductive theory,'' we mean an inductive theory that makes
statements about real numbers. If $P \subseteq \cL^\IS$ is such a theory with
root $T_0$, then $\cL$ should be capable of making statements about real
numbers, and (ideally) $T_0$ should contain all true statements about real
numbers.

One particularly straightforward way to construct such an inductive theory is to
follow the approach taken in Section \ref{S:RV-symb}. Namely, we construct a
standard structure of the real numbers, which we denote by $\cR$, and we require
$T_0$ to contain all sentences that are strictly satisfied by $\cR$. We can then
prove the analogue of Theorem \ref{T:prob-model-iso}, showing that every
collection of real-valued random variables can be represented in a natural way
inside a real inductive theory. This is done in Section \ref{S:std-R}.

There are several downsides to this approach. The first is that we cannot talk
directly about sets of real numbers. We can add them indirectly as relations in
our language, as we did in Section \ref{S:RV-symb}. But there is no intrinsic
theory of sets in this language. A second, related downside is that we cannot
talk about distinguished elements or subsets of the reals without considerable
extra effort. In particular, we cannot talk directly about integers and
rationals, and their relationships to the reals.

The primary purpose of this chapter is to present a different, more robust
approach. Namely, we will create inductive theories whose root $T_0$ contains
all of axiomatic set theory. In this way, not only can we make inductive
statements about real numbers, but also about all other objects of modern
mathematics.

After discussing definitorial extensions in Section \ref{S:def-ext}, the axioms
of set theory are presented in Section \ref{S:ZFC}. They are the usual axioms of
Zermelo-Fraenkel set theory with choice. As with Peano arithmetic in Section
\ref{S:PA}, we define multiple theories. In fact, we define four theories:
$\ZFC_\fin \subseteq \ZFC_- \subseteq \ZFC \subseteq \ZFC_+$. The first of
these, $\ZFC_\fin$, is the usual finitary set theory from first-order logic. The
others are extensions to $\cL^0$. The first extension, $\ZFC_-$, is
conservative, in the sense that every sentence in $\ZFC_- \setminus \ZFC_\fin$
is purely infinitary. In other words, in $\ZFC_-$, we cannot deduce any new
first-order sentences that we could not already deduce in $\ZFC_\fin$. This is
because $\ZFC_-$ and $\ZFC_\fin$ have the same axioms. In particular, even in
$\ZFC_-$, we can only use finitary formulas when making use of the axioms of
separation and replacement (see Section \ref{S:separation} and
\ref{S:replacement}).

The extension $\ZFC$ is stronger. There, we allow infinitary formulas in the
axiom of separation. In $\ZFC_+$, we allow infinitary formulas in both
separation and replacement. We do not spend time on $\ZFC_+$ beyond defining it.
We primarily focus on the theories $\ZFC_-$ and $\ZFC$.

In Section \ref{S:ind-th-ZFC-}, we construct the set of real numbers in $\ZFC_-$,
and build real inductive theories whose roots are required to contain $\ZFC_-$.
We then prove the analogue of Theorem \ref{T:prob-model-iso}, showing that every
collection of real-valued random variables can be represented inside such a
theory. Using $\ZFC_-$ is superior to using the standard structure of the reals.
Now, not only can we talk about distinguished sets of real numbers, we can in
fact talk about all kinds of sets, backed up by the full power of set theory.

These benefits, however, come at a price. In $\ZFC_-$, although we can define
the set of real numbers, we cannot define each individual real number. We can
explicitly define each rational number, and we can explicitly define certain
individual real numbers, such as $\pi$, $e$, and $\sqrt 2$. But it is
intuitively clear, at least in finitary set theory, that the vast majority of
real numbers elude any type of description. As such, our probabilistic
statements will only be able to mention the rationals, and a handful of
definable reals.

To elaborate on this, consider an inductive statement, $P(\ph \mid X) = p$.
There are two things to notice about this. First, it is not an element of $\cL$.
We cannot add a quantifier to the outside of this statement, except in a
metatheoretical sense. Second, the formula $\ph$ is a sentence. It cannot
contain any free variables. Hence, any mention of a real number inside $\ph$
must be done through an explicitly defined constant. Therefore, when using
$\ZFC_-$, any real number which cannot be explicitly defined in $\ZFC_-$ cannot
be used inside an inductive statement.

Another downside to using $\ZFC_-$ is that it produces a weaker analogue of
Theorem \ref{T:prob-model-iso}. In that theorem, we see a direct and intuitive
correspondence between outcomes in the measure-theoretic model, and structures
in the inductive model. This connection is lost when we do things in $\ZFC_-$.

It turns out that the right place to work is in $\ZFC$. This is done in
Section \ref{S:ind-th-ZFC}. In $\ZFC$, not only can we explicitly define
each individual real number, we can also explicitly define each individual Borel
set, and each individual measurable function. Hence, any statement we might make
in our measure-theoretic model has an explicit counterpart in $\ZFC$.
Moreover, we recover the natural correspondence between outcomes and structures.

Additionally, in $\ZFC$, we can construct a frame of reference in which the real
numbers, Borel sets, and measurable functions are all almost surely fixed, and
not random. In this sense, $\ZFC$ is home to the natural intuition of the
practicing probabilist, to whom it would never occur to think of such things as
varying with $\om$.

Adopting $\ZFC$, however, involves accepting a new axiom of set
theory---or rather, accepting an expanded version of the axiom of separation. In
Section \ref{S:Con-ZFC}, we discuss reasons why this is hardly any more
problematic that assuming that $\ZFC_\fin$ is consistent.

Finally, in Sections \ref{S:limit-thms} and \ref{S:cond-exp}, we illustrate how
the major theorems and structures of measure-theoretic probability can be
expressed using inductive logic. The examples we cover are the law of large
numbers, the central limit theorem, conditional expectation, and the general
form of the law of total probability, also known as the tower property of
conditional expectation.

\section{Definitorial extensions}\label{S:def-ext}

We often want to introduce new symbols into our language that are defined in
terms of old ones. Sometimes this can be done using shorthand. In that case, the
new symbols are not actually part of our language. They are just notational
conventions we use to talk about our language. We have seen this already with
the symbols $\exists$ and $\to$.

Sometimes, however, we want to formally augment our logical signature. For
example, in the context of a deductive theory $T$, suppose we have $T \vdash
\exists! x \ph(x)$. We may wish to introduce a constant symbol to denote the
unique object whose existence is being asserted. This is not easily done with
shorthand. In this subsection, we go over precisely how this is done, and what
effects it has on deductive and inductive derivability.

\subsection{Defining individual symbols}

\paragraph{Relation symbols.}

Let $\cL$ be a predicate language with logical signature $L$. Let $r$ be an
$n$-ary relation symbol with $r \notin L$, and let $\cL[r]$ be the language with
signature $L \cup \{r\}$. An \emph{explicit definition of $r$ in $\cL$} is a
sentence in $\cL[r]^0$ of the form $\th_r = \forall x (r \vec x \tot \de(\vec
x))$, where $\de = \de(x_1, \ldots, x_n) \in \cL$. The formula $\de$ is called a
\emph{defining formula}. We may sometimes denote $\de$ by $\de_r$, to indicate
its relationship to $r$.
  \index{explicit definition}%
  \index{formula!defining ---}%

Given $\ph \in \cL[r]$, we define $\ph^\rd \in \cL$ as follows. If $\ph$ is an
equation, then $\ph^\rd = \ph$, and if $\ph = r \vec t$, then $\ph^\rd = \de
(\vec t\,)$. We extend this recursively by $(\neg \ph)^\rd = \neg \ph^\rd$, $
(\bigwedge \Phi)^\rd = \bigwedge_{\ph \in \Phi} \ph^\rd$, and $(\forall x
\ph)^\rd = \forall x \ph^\rd$. Intuitively, $\ph$ is reduced down to $\ph^\rd$
by replacing all occurrences of $r \vec t$ by $\de(\vec t\,)$.

\paragraph{Constant symbols.}

Let $c$ be a constant symbol with $c \notin L$. Let $\cL[c]$ be the language
with signature $L \cup \{c\}$. An \emph{explicit definition of $c$ in $\cL$} is
a sentence in $\cL[c]^0$ of the form $\th_c = \forall y (y \beq c \tot \de(y))$,
where $\de = \de(y) \in \cL$. The formula $\de$ is called a \emph{defining
formula}. We may sometimes denote $\de$ by $\de_c$, to indicate its relationship
to $c$. Let $\xi_c = \exists! y \, \de_c$. Note that $\th_c \vdash \xi_c$. In
general, we will only use the definition $\th_c$ in situations where $\xi_c$
holds. Given $\ph \in \cL[c]$, we choose $z \notin \var \ph$ and define $\ph^\rd
\in \cL$ by $\ph^\rd = \exists z (\ph(z/c) \wedge \de(z))$.
  \index{explicit definition}%
  \index{formula!defining ---}%

\paragraph{Function symbols.}

Finally, let $f$ be an $n$-ary function symbol with $f \notin L$ and $n \ge 1$.
Let $\cL[f]$ be the language with signature $L \cup \{f\}$. An \emph{explicit
definition of $f$ in $\cL$} is a sentence in $\cL[f]^0$ of the form $\th_f =
\forall \vec x y (y \beq f \vec x \tot \de(\vec x, y))$, where $\de = \de(x_1,
\ldots, x_n, y) \in \cL$. The formula $\de$ is called a \emph{defining formula}.
We may sometimes denote $\de$ by $\de_f$, to indicate its relationship to $f$.
Let $\xi_f = \forall \vec x \, \exists! y \, \de_f$. Note that $\th_f \vdash
\xi_f$. In general, we will only use the definition $\th_f$ in situations where
$\xi_f$ holds.
  \index{explicit definition}%
  \index{formula!defining ---}%

Given $\ph \in \cL[f]$, we define $\ph^\rd \in \cL$ by formula recursion. First
suppose $\ph$ in prime. Then $\ph$ is a string of finite length. Here, we define
$\ph^\rd$ exactly as in first-order logic. (See, for example, \cite[Section
2.6]{Rautenberg2010}.) Namely, choose $y \notin \var \ph$. Find the leftmost
occurrence of $f$ in $\ph$, which will be followed by a unique concatenation of
terms $\vec t = t_1 \cdots t_n$, and let $\ph'$ be the prime formula obtained by
replacing $f \vec t$ with $y$. Note that $\ph = \ph'(f \vec t / y)$. We then
define $\ph_1 = \exists y (\ph' \wedge \de(\vec t, y))$. The resulting formula
$\ph_1$ has one fewer occurrence of $f$ than $\ph$. If $f$ still occurs in
$\ph_1$, then repeat the procedure to obtain $\ph_2$, and so on. Since $\ph$ has
only finitely many occurrences of $f$, this procedure will eventually terminate
in some $\ph_m$ that no longer contains $f$. We then define $\ph^\rd = \ph_m$.
We extend this definition recursively by $(\neg \ph)^\rd = \neg \ph^\rd$, $
(\bigwedge \Phi)^\rd = \bigwedge_{\ph \in \Phi} \ph^\rd$, and $(\forall x
\ph)^\rd = \forall x \ph^\rd$.

\subsection{Defining multiple symbols}\label{S:def-ext-mult}

More generally, let $M$ be a set of extralogical symbols, disjoint from $L$. Let
$\cL'$ be the language with signature $L \cup M$. For each $\s \in M$, let
$\th_\s$ be an explicit definition of $\s$ in $\cL$, and let $\Theta = \{\th_\s
\mid \s \in M\}$. Let $\xi_r = \top$ for all relation symbols $r \in M$, and let
$\Xi = \{\xi_\s \mid \s \in M\}$. Note that $\Theta \vdash \Xi$. In general, we
will only use the definitions $\Theta$ in situations where $\Xi$ holds.
  \symindex{$\Theta$}%
  \symindex{$\Xi$}%

Given $\ph \in \cL'$, we define the \emph{reduced formula}, $\ph^\rd \in \cL$,
as follows. If $\ph$ is prime, then $\sym \ph$ is finite. We may therefore
eliminate the symbols in $\sym \ph \cap M$ in a stepwise fashion as above. We
then extend this recursively by $(\neg \ph)^\rd = \neg \ph^\rd$, $ (\bigwedge
\Phi)^\rd = \bigwedge_{\ph \in \Phi} \ph^\rd$, and $(\forall x \ph)^\rd =
\forall x \ph^\rd$. More generally, for $X \subseteq \cL'$, we write $X^\rd = 
\{\ph^\rd \mid \ph \in X\}$.
  \index{formula!reduced ---}%
  \symindex{$\ph^\rd, X^\rd$}%

\subsection{Extensions and models}

Let $\om = (A, L^\om)$ be an $\cL$-structure, and define the $\cL'$-structure,
$\om' = (A, (L')^{\om'})$ as follows. First, let $\s^{\om'} = \s^\om$ whenever
$\s \in L$. If $\s = r \in M$ is a relation symbol, then we define $r^{\om'}$ by
$r^ {\om'} \vec a$ if and only if $\om \tDash \de_r[\vec a]$, where $\vec a \in
A^n$. Next, suppose $\s = c \in M$ is a constant symbol. If $\om \tDash \xi_c =
\exists! y \, \de_c$, then there exists a unique $a \in A$ such that $\om \tDash
\de_c[a]$. We then define $c^{\om'} = a$. Otherwise, if $\om \ntDash \xi_c$,
then choose $a \in A$ arbitrarily and set $c^{\om'} = a$. Lastly, suppose $\s =
f \in M$ is a function symbol. If $\om \tDash \xi_f = \forall \vec x \, \exists!
y \, \de_f$, then for each $\vec a \in A^n$, there exists a unique $b \in A$
such that $\om \tDash \de_f[\vec a, b]$. We then define $f^{\om'}(\vec a) = b$.
Otherwise, if $\om \ntDash \xi_f$, then we define $f^{\om'}$ arbitrarily. Note
that $\om'$ is constructed so that $\om' \tDash \th_\s$ whenever $\om \tDash
\xi_\s$. Conversely, note that if $\nu$ is an $\cL'$-model and $\om$ is its
$\cL$-reduct, then $\om \tDash \xi_\s$ whenever $\nu \tDash \th_\s$, and $\nu$
and $\om'$ agree on $L \cup \{\s \in M \mid \nu \tDash \th_\s\}$. Finally, given
an assignment $v_\om$ into $\om$, we define the assignment $v'_ {\om'}$ into
$\om'$ by $v'_{\om'}(x) = v(x)$ for all $x \in \Var$.

\begin{prop}\label{P:elim-struct}
  Let $\om$ be an $\cL$-structure and $\ph \in \cL'$. Assume $\om \tDash \xi_\s$
  for all $\s \in M \cap \sym \ph$. Then $\om' \tDash \ph[v'_{\om'}]$ if and
  only $\om \tDash \ph^\rd[v_\om]$, for all assignments $v_\om$.
\end{prop}

\begin{proof}
  If $\ph \in \cL$, then $\ph^\rd = \ph$. Hence, by Theorem \ref{T:coinc-thm},
  the proposition holds for all $\ph \in \cL$.

  For $\ph \in \cL'$, we first consider the case, $M \cap \sym \ph = \{r\}$. In
  this case, $\om \tDash \xi_r$. It follows that $\om' \tDash \th_r = \forall
  \vec x (r \vec x \tot \de(\vec x))$. We now prove the proposition by induction
  on $\ph$. Suppose $\ph$ is prime. Then either $\ph \in \cL$ or $\ph = r \vec
  t$. In the former case, we established that the proposition holds. In the
  latter case, we have $\ph^\rd = \de(\vec t\,)$, so the result follows from
  $\om' \tDash \th_r$. The inductive steps are straightforward. The cases $M =
  \{c\}$ and $M = \{f\}$ are similar.

  We now consider the case of general $M$. As above, the result holds if $\sym
  \ph \cap M$ contains a single element. It therefore holds whenever $\sym \ph
  \cap M$ is finite, by reducing each symbol one at a time. In particular, it
  holds for each prime $\ph$, since prime formulas are finite strings of
  symbols. The result then follows by induction on $\ph$, using the recursive
  definition of strict satisfiability.
\end{proof}

Let $\sP = (\Om, \Si, \bbP)$ be an $\cL$-model. For each $\om \in \Om$, define
the $\cL'$-structure $\om'$ as above. Let $\Om' = \{\om' \mid \om \in \Om\}$,
let $h$ denote the map $\om \mapsto \om'$, and let $\sP' = (\Om', \Ga, \bbQ)$ be
the measure space image of $\sP$ under $h$. Since $\om' \tDash \th_\s$ whenever
$\om \tDash \xi_\s$, it follows that $\sP' \vDash \th_\s$ whenever $\sP \vDash
\xi_\s$. Given an assignment $\bv$ into $\sP$, define the assignment $\bv'$ into
$\sP'$ by $v'_{\om'}(x) = v_\om(x)$ for all $x \in \Var$.

\begin{lemma}\label{L:elim-model}
  Let $\sQ = (\Om^\sQ, \Ga^\sQ, \bbQ^\sQ)$ be an $\cL'$-model such that $\sQ
  \vDash \Theta$. For each $\nu \in \Om^\sQ$, let $g \nu$ be the $L$-reduct of
  $\nu$. Define $\Om = \{g \nu \mid \nu \in \Om^\sQ\}$ and let $\sP = (\Om, \Si,
  \bbP)$ be the measure space image of $\sQ$ under $g$. Given an assignment
  $\bw$ into $\sQ$, define the assignment $\bv$ into $\sP$ by $v_{g \nu}(x) =
  w_\nu(x)$ for all $x \in \Var$. Let $\sP'$ and $\bv'$ as above. Then $\sP
  \vDash \Xi$ and, for any $\ph \in \cL'$, we have $\sQ \vDash \ph[\bw]$ if and
  only if $\sP' \vDash \ph[\bv']$.
\end{lemma}

\begin{proof}
  Let $\xi_\s \in \Xi$. Since $\xi_s \in \cL$, it follows from Theorem 
  \ref{T:coinc-thm} that $g \nu \tDash \xi_s$ if and only if $\nu \tDash \xi_s$.
  Hence, $(\xi_\s)_{\Om^\sQ} = g^{-1} (\xi_\s)_\Om$, so that $\sQ \vDash \xi_\s$
  if and only if $\sP \vDash \xi_\s$. Since $\Theta \vdash \Xi$ and $\sQ \vDash
  \Theta$, this gives $\sP \vDash \Xi$.

  Now let $M_0 = M \cap \sym \ph$ and note that $M_0$ is countable. Let
  $\Theta_0 = \{\th_\s \mid \s \in M_0\}$. Then $\sQ \vDash \bigwedge \Theta_0$,
  so that $\nu \tDash \bigwedge \Theta_0$, for a.e.~$\nu$. It follows that $\nu$
  and $\om'$ agree on $L \cup M_0$, for a.e.~$\nu$. The result now follows from
  Theorem \ref{T:coinc-thm}.
\end{proof}

\begin{prop}\label{P:elim-model}
  Let $\sP = (\Om, \Si, \bbP)$ be an $\cL$-model such that $\sP \vDash \Xi$, and
  let $\sP' = (\Om', \Ga, \bbQ)$ be as above. Let $\ph \in \cL'$ and let $\bv$
  be an assignment into $\sP$. Then $\ph[\bv']_{\Om'} \in \ol \Ga$ if and only
  if $\ph^\rd [\bv]_\Om \in \ol \Si$, and in this case, $\olbbQ \ph[\bv']_{\Om'}
  = \olbbP \ph^\rd[\bv]_\Om$. In particular, $\sP' \vDash \ph[\bv']$ if and only
  $\sP \vDash \ph^\rd[\bv]$.
\end{prop}

\begin{proof}
  Since $\sym \ph$ is countable and $\sP \vDash \Xi$, we may choose $\Om^* \in
  \Si$ such that $\bbP \Om^* = 1$ and $\om \tDash \xi_\s$, for all $\s \in M
  \cap \sym \ph$ and all $\om \in \Om^*$. Hence, by Proposition
  \ref{P:elim-struct}, we have $\om' \tDash \ph[v'_{\om'}]$ if and only $\om
  \tDash \ph^\rd[v_\om]$, for $\bbP$-a.e.~$\om \in \Om$. Therefore, $h^{-1}
  \ph[\bv']_{\Om'} = \ph^\rd[\bv]_\Om$ a.s. It follows from the definition of
  $\sP'$ that $\ph[\bv']_{\Om'} \in \ol \Ga$ if and only if $\ph^\rd [\bv]_\Om
  \in \ol \Si$, and in this case, $\olbbQ \ph[\bv']_{\Om'} =
  \olbbP \ph^\rd[\bv]_\Om$. In particular, $\olbbQ \ph[\bv']_{\Om'} = 1$ if and
  only if $\olbbP \ph^\rd[\bv]_\Om = 1$, so that $\sP' \vDash \ph[\bv']$ if and
  only $\sP \vDash \ph^\rd[\bv]$.
\end{proof}

\subsection{Deductive elimination}

The following theorem captures the exact relationship between derivability in
$\cL'$ and derivability in $\cL$. By Remark \ref{R:mon}, since $\cL \subseteq
\cL'$, we are able to simply write $\vdash$, instead of $\vdash_\cL$ and
$\vdash_{\cL'}$.

\begin{thm}[Deductive elimination theorem]\label{T:ded-elim-thm}
  Let $X \subseteq \cL'$ and $\ph \in \cL'$. Then $X, \Theta \vdash \ph$ if and
  only if $X^\rd, \Xi \vdash \ph^\rd$.
\end{thm}

\begin{proof}
  Assume $X, \Theta \vdash \ph$. Let $\sP = (\Om, \Si, \bbP)$ be an $\cL$-model
  and let $\bv$ be an assignment into $\sP$ such that $\sP \vDash \psi[\bv]$ for
  all $\psi \in X^\rd \cup \Xi$. Since $\sP \vDash \Xi$, we may define $\sP'$ as
  in Proposition \ref{P:elim-model}. We then have $\sP' \vDash \ze[\bv']$ if and
  only $\sP \vDash \ze^\rd[\bv]$ for all $\ze \in \cL'$. It follows that $\sP'
  \vDash \psi[\bv']$ for all $\psi \in X$. Since $\sP' \vDash \th_\s$ whenever
  $\sP \vDash \xi_\s$, it also follows that $\sP' \vDash \Theta$. Thus, $\sP
  \vDash \psi[\bv']$ for all $\psi \in X \cup \Theta$. Since $X, \Theta \vdash
  \ph$, we conclude that $\sP' \vDash \ph[\bv']$. One more application of
  Proposition \ref{P:elim-model} gives $\sP \vDash \ph^\rd[\bv]$. Since $\sP$
  and $\bv$ were arbitrary, this shows that $X^\rd, \Xi \vdash \ph^\rd$.

  Now suppose $X^\rd, \Xi \vdash \ph^\rd$. Let $\sQ = (\Om^\sQ, \Ga^\sQ,
  \bbQ^\sQ)$ be an $\cL'$-model and let $\bw$ be an assignment into $\sQ$ such
  that $\sQ \vDash \psi[\bw]$ for all $\psi \in X \cup \Theta$. Since $\sQ
  \vDash \Theta$, we may define $\sP$ as in Lemma \ref{L:elim-model}. We then
  have $\sP \vDash \Xi$ and $\sP' \vDash \psi[\bv']$ for all $\psi \in X \cup
  \Theta$. Proposition \ref{P:elim-model} then implies $\sP \vDash
  \psi^\rd[\bv]$ for all $\psi \in X \cup \Theta$. In particular, $\sP \vDash
  \psi[\bv]$ for all $\psi \in X^\rd$. Together with $\sP \vDash \Xi$ and
  $X^\rd, \Xi \vdash \ph^\rd$, this gives $\sP \vDash \ph^\rd[\bv]$. Another
  application of Proposition \ref{P:elim-model} gives $\sP' \vDash \ph[\bv']$.
  Therefore, by Lemma \ref{L:elim-model}, we have $\sQ \vDash \ph[\bw]$. Since
  $\sQ$ and $\bw$ were arbitrary, this shows that $X, \Theta \vdash \ph$.
\end{proof}

\begin{cor}\label{C:ded-elim-thm1}
  For any $X_1, X_2 \subseteq \cL'$, we have $X_1 \cup \Theta \equiv X_2 \cup
  \Theta$ if and only if $X_1^\rd \cup \Xi \equiv X_2^\rd \cup \Xi$.
\end{cor}

\begin{proof}
  Suppose $X_1 \cup \Theta \equiv X_2 \cup \Theta$. Then $X_1, \Theta \vdash
  X_2$. Theorem \ref{T:ded-elim-thm} implies $X_1^\rd, \Xi \vdash X_2^\rd$.
  Likewise, $X_2^\rd, \Xi \vdash X_1^\rd$. Therefore, $X_1^\rd \cup \Xi \equiv
  X_2^\rd \cup \Xi$. The proof of the converse is similar.
\end{proof}

\begin{defn}\label{D:ded-def-ext}
    \index{definitorial extension}%
  Let $T \subseteq \cL^0$ be a deductive theory, and let $\Theta$ and $\Xi$ be
  as in Section \ref{S:def-ext-mult}. We say that $\Theta$ is \emph{legitimate
  in $T$} if $\Xi \subseteq T$. If $\Theta$ is legitimate in $T$, then we define
  the deductive theory $T' \subseteq (\cL')^0$ by $T' = T + \Theta$. The
  deductive theory $T'$ is called a \emph{definitorial extension of $T$}.
\end{defn}

\begin{cor}\label{C:ded-elim-thm2}
  Let $T'$ be a definitorial extension of $T$. Then for all $\ph \in \cL'$, we
  have $T' \vdash \ph$ if and only if $T \vdash \ph^\rd$. In particular, $T$ is
  consistent if and only if $T'$ is consistent.
\end{cor}

\begin{proof}
  By Theorem \ref{T:ded-elim-thm}, we have $T' \vdash \ph$ if and only if
  $T^\rd, \Xi \vdash \ph^\rd$. But $T^\rd = T$ and $\Xi \subseteq T$. Therefore,
  $T^\rd, \Xi \vdash \ph^\rd$ if and only if $T \vdash \ph^\rd$.
\end{proof}

\subsection{Inductive elimination}

Our aim here is to prove Theorem \ref{T:ind-elim-thm} below, which is an
inductive version of Theorem \ref{T:ded-elim-thm}.

Let $P$ be an inductive theory in $\cL^\IS$ with root $T_0$, and let $T_0'$ be a
definitorial extension of $T_0$. That is, $T_0' = T_0 + \Theta$, where $\Theta$
is legitimate in $T_0$, meaning that $\Xi \subseteq T_0$.

Let $X \subseteq (\cL')^0$ and assume that $X \cao T_0'$. Choose $\psi \in
(\cL')^0$ such that $X \equiv T_0' + \psi$. Define $Q^X \subseteq (\cL')^\IS$ by
\[
  Q^X = \{(X, \ph, p) \mid P(\ph^\rd \mid T_0, \psi^\rd) = p\}.
\]
We claim that the definition of $Q^X$ does not depend on the choice $\psi$. To
see this, suppose that $\ze \in (\cL')^0$ and $X \equiv T_0' + \ze$. Then $T_0'
\vdash \psi \to \ze$. By Corollary \ref{C:ded-elim-thm2}, we have $T_0 \vdash
\psi^\rd \to \ze^\rd$, so that $T_0, \psi^\rd \vdash \ze^\rd$. Reversing the
roles of $\psi$ and $\ze$ gives $T_0, \ze^\rd \vdash \psi^\rd$. Hence, $T_0 +
\psi^\rd = T_0 + \ze^\rd$. By the rule of logical equivalence, $P(\ph^\rd \mid
T_0, \psi^\rd) = p$ if and only if $P(\ph^\rd \mid T_0, \ze^\rd) = p$, and so
$Q^X$ does not depend on $\psi$.

Let $Q = \bigcup \{Q^X \mid X \cao T_0'\}$. Then $Q$ is strongly connected with
root $T_0'$.

\begin{lemma}\label{L:ind-elim-thm1}
  With notation as above, $Q$ is satisfiable.
\end{lemma}

\begin{proof}
  Let $\sP$ be an $\cL$-model such that $\sP \vDash P$. Then $\sP \vDash T_P$
  and $\Xi \subseteq T_0 \subseteq T_P$, so that $\sP \vDash \Xi$. Define $\sP'$
  as in Proposition \ref{P:elim-model}. Since $\sP' \vDash \th_\s$ whenever $\sP
  \vDash \xi_\s$, we have $\sP' \vDash \Theta$. Also, $T_0^\rd = T_0$, so
  Proposition \ref{P:elim-model} gives $\sP' \vDash T_0$. Therefore, $\sP'
  \vDash T_0 + \Theta = T_0'$.

  Now suppose $(X, \ph, p) \in Q$. Write $X \equiv T_0' + \psi$, where $P
  (\ph^\rd \mid T_0, \psi^\rd) = p$. Then $\olbbP \ph^\rd_\Om \cap \psi^\rd_\Om
  / \olbbP \psi^\rd_\Om = p$. By Proposition \ref{P:elim-model}, we have $\olbbQ
  \ph_{\Om'} \cap \psi_{\Om'} / \olbbQ \psi_{\Om'} = p$, so that $\sP' \vDash 
  (X, \ph, p)$. Since $(X, \ph, p)$ was arbitrary, $\sP' \vDash Q$.
\end{proof}

It follows from Lemma \ref{L:ind-elim-thm1} and Theorem \ref{T:ind-satis-cons}
that $Q$ is consistent. We may therefore define $P' = \bfP_Q$. Let $P'_0 = P'
\dhl_{T_0'}$.

\begin{lemma}\label{L:ind-elim-thm2}
  With notation as above, $P'_0 = Q$.
\end{lemma}

\begin{proof}
  Since $Q \subseteq P'$ and $X \cao T_0'$ for every $X \in \ante Q$, we have $Q
  \subseteq P'_0$. Conversely, suppose that $P_0'(\ph \mid X) = p$. Write $X
  \equiv T_0' + \psi$, so that $P'(\ph \mid T_0', \psi) = p$. Let $\sP$ and
  $\sP'$ be as in the proof of Lemma \ref{L:ind-elim-thm1}. Since $\sP' \vDash
  Q$, it follows that $\sP' \vDash P'$. Therefore, $\olbbQ \ph_{\Om'} \cap
  \psi_{\Om'} / \olbbQ \psi_{\Om'} = p$. Proposition \ref{P:elim-model} then
  gives $\olbbP \ph^\rd_\Om \cap \psi^\rd_\Om / \olbbP \psi^\rd_\Om = p$, so
  that $\sP \vDash (T_0 + \psi^\rd, \ph^\rd, p)$. Since this is true for every
  $\cL$-model $\sP$ such that $\sP \vDash P$, and since $T_0 + \psi^\rd \cao
  [T_0, T_P]$, it follows from Definition \ref{D:consequence} that $P \vDash
  (T_0 + \psi^\rd, \ph^\rd, p)$. Remark \ref{R:classic-ind-th-char} therefore
  implies $P(\ph^\rd \mid T_0, \psi^\rd) = p$, so that $(X, \ph, p) \in Q$.
  Hence, $Q = P'_0$.
\end{proof}

\begin{thm}[Inductive elimination theorem]\label{T:ind-elim-thm}
  Let $P$ be an inductive theory in $\cL^\IS$ with root $T_0$, and let $T_0'$ be
  a definitorial extension of $T_0$ as in Definition \ref{D:ded-def-ext}. Then
  there exists a unique inductive theory $P' \subseteq (\cL')^\IS$ with root
  $T_0'$ such that
  \begin{equation}\label{ind-elim-thm}
    P'(\ph \mid X, \Theta) = P (\ph^\rd \mid X^\rd, \Xi),
  \end{equation}
  where either both sides exist or both sides do not.
\end{thm}

\begin{proof}
  Let $P'$ be defined as above. We first show that $P'(\ph \mid X, \Theta) = p$
  implies $P(\ph^\rd \mid X^\rd, \Xi) = p$. Suppose $P'(\ph \mid X, \Theta) =
  p$. Write $X \cup \Theta \equiv T' + \psi$, where $T' \in [T_0', T_ {P'}]$,
  $\psi \in (\cL')^0$, and $P'(\ph \mid T_0', \psi) = p$. Since $P'_0 = Q$, we
  have $Q(\ph \mid T_0', \psi) = p$. From the definition of $Q$, it follows that
  $P(\ph^\rd \mid T_0, \psi^\rd) = p$. Therefore, to show that $P (\ph^\rd \mid
  X^\rd, \Xi) = p$, it suffices to show that $X^\rd \cup \Xi \equiv T +
  \psi^\rd$ for some $T \in [T_0, T_P]$. For this, let $T = T((T')^\rd \cup
  \Xi)$. Since $\Theta \subseteq T_0' \subseteq T'$, we have $X \cup \Theta
  \equiv T' + \Theta + \psi$. Corollary \ref{C:ded-elim-thm1} therefore gives
  $X^\rd \cup \Xi \equiv (T')^\rd \cup \Xi \cup \{\psi^\rd\} \equiv T +
  \psi^\rd$. We must now show that $T_0 \subseteq T$ and $T \subseteq T_P$.

  For $T_0 \subseteq T$, note that $T_0 \subseteq T_0' \subseteq T'$. We
  therefore have $T' \vdash T_0$, so that $T', \Theta \vdash T_0$. Since
  $T_0^\rd = T_0$, Theorem \ref{T:ded-elim-thm} gives $(T')^\rd, \Xi \vdash
  T_0$, so that $T_0 \subseteq T$.

  For $T \subseteq T_P$, we first show that $T_{P'}^\rd \subseteq T_P$. Let $\ze
  \in T_{P'}$. Then $P'(\ze \mid T_0') = 1$. By Lemma \ref{L:ind-elim-thm2}, we
  have $Q(\ze \mid T_0') = 1$. From the definition of $Q$, it follows that
  $P(\ze^\rd \mid T_0) = 1$. Therefore, $\ze^\rd \in T_P$, and this shows that
  $T_{P'}^\rd \subseteq T_P$.

  Now, since $\Xi \subseteq T_0 \subseteq T_P$, we have that $\Theta$ is
  legitimate in $T_P$. Corollary \ref{C:ded-elim-thm2} therefore gives $T_P +
  \Theta \vdash \ze$ if and only if $T_P \vdash \ze^\rd$, for all $\ze \in
  \cL'$. By the above, $T_P \vdash T_{P'}^\rd$. Hence, $T_P + \Theta \vdash T_
  {P'}$. Since $T' \subseteq T_{P'}$, this implies $T_P + \Theta \vdash T'$.
  Another application of Corollary \ref{C:ded-elim-thm2} yields $T_P \vdash 
  (T')^\rd$. As previously noted, $\Xi \subseteq T_P$. Therefore, $T_P \vdash
  T$, so that $T \subseteq T_P$.

  We now show that $P(\ph^\rd \mid X^\rd, \Xi) = p$ implies $P'(\ph \mid X,
  \Theta) = p$. Suppose $P(\ph^\rd \mid X^\rd, \Xi) = p$. Write $X^\rd \cup \Xi
  \equiv T + \psi$, where $T \in [T_0, T_P]$, $\psi \in \cL^0$, and $P(\ph^\rd
  \mid T_0, \psi) = p$. By Lemma \ref{L:ind-elim-thm2}, the definition of $Q$,
  and the fact that $\psi^\rd = \psi$, we have $P'(\ph \mid T_0, \psi) = p$.
  Therefore, to show that $P'(\ph \mid X, \Theta) = p$, it suffices to show that
  $X \cup \Theta \equiv T' + \psi$ for some $T' \in [T_0', T_{P'}]$. For this,
  let $T' = T + \Theta$. Since $\Xi \subseteq T_0 \subseteq T$, we have $X^\rd
  \cup \Xi \equiv T + \Xi + \psi = T' + \psi$.

  We must now show that $T_0' \subseteq T'$ and $T' \subseteq T_{P'}$. The first
  follows easily, since $T_0' = T_0 + \Theta \subseteq T + \Theta = T'$. For the
  second, note that if $\ze \in \cL^0$ and $P(\ze \mid T_0) = 1$, then $Q(\ze
  \mid T_0') = 1$, which implies $P'(\ze \mid T_0') = 1$. Therefore, $T_P
  \subseteq T_{P'}$. Since $T \subseteq T_P$, we have $T' = T + \Theta \subseteq
  T_P + \Theta \subseteq T_ {P'} + \Theta$. But $\Theta \subseteq T_0' \subseteq
  T_{P'}$, so it follows that $T' \subseteq T_{P'}$.

  For uniqueness, let $P'' \subseteq (\cL')^0$ be another inductive theory with
  root $T_0'$ that satisfies \eqref{ind-elim-thm}, and suppose $P''(\ph \mid X)
  = p$. Then $\Theta \subseteq T_0' \subseteq T(X)$, so that $X \equiv X \cup
  \Theta$. By the rule of logical equivalence and \eqref{ind-elim-thm}, we have
  \[
    P''(\ph \mid X) = P''(\ph \mid X, \Theta) = P (\ph^\rd \mid X^\rd, \Xi)
      = P'(\ph \mid X, \Theta) = P'(\ph \mid X),
  \]
  which shows that $P'' \subseteq P'$. Reversing the roles of $P'$ and $P''$
  gives $P'' = P'$.
\end{proof}

The inductive theory $P'$ in Theorem \ref{T:ind-elim-thm} is called a
\emph{definitorial extension of $P$}. In the special case that $P$ is a complete
inductive theory, we have the following semantic characterization of $P'$.
  \index{definitorial extension}%

\begin{cor}\label{C:ind-elim-thm}
  Let $T_0 \subseteq \cL^0$ be a deductive theory. Let $\sP$ be an $\cL$-model
  with $\sP \vDash T_0$ and define $P = \bTh \sP \dhl_{[T_0, \Th P]}$. Let $P'$
  be a definitorial extension of $P$. Then $P' = \bTh \sP' \dhl_{[T_0', \Th
  \sP']}$, where $\sP'$ is defined above Lemma \ref{L:elim-model}.
\end{cor}

\begin{proof}
  By Lemma \ref{L:ind-elim-thm2} and Theorem \ref{T:theory-defn}, it suffices to
  show that $Q = \bTh \sP' \dhl_{T_0'}$. Note that $(X, \ph, p) \in Q$ if and
  only if we can write $X \equiv T_0' + \psi$, where $\sP \vDash (T_0 +
  \psi^\rd, \ph^\rd, p)$. Also note that $(X, \ph, p) \in \bTh \sP' \dhl_{T_0'}$
  if and only if we can write $X \equiv T_0' + \psi$, where $\sP' \vDash (T_0' +
  \psi, \ph, p)$. Since $P'$ is a definitorial extension, we have $\Xi \subseteq
  T_0$, which implies $\sP \vDash \Xi$. Hence, from Proposition
  \ref{P:elim-model}, it follows that $\sP \vDash (T_0 + \psi^\rd, \ph^\rd, p)$
  if and only if $\sP' \vDash (T_0' + \psi, \ph, p)$.
\end{proof}

\subsection{Primitive vs.~defined symbols}\label{S:prim-vs-def}

Let $P \subseteq \cL^\IS$ be an inductive theory with root $T_0$. The
extralogical symbols in $L$ have no formal definitions within $P$.
Syntactically, they get their meaning from their use, that is, from the
inductive statements $(X, \ph, p) \in P$ in which they appear. If $P$ and $T_0$
are generated by a set of axioms, then we might say that the symbols in $L$ are
``defined'' by these axioms. That is, they get whatever meaning they may have
from what the axioms have to say about them.

Now suppose $P'$ is a definitorial extension of $P$. Then, unlike the symbols in
$L$, the symbols in $L' \setminus L$ do have a formal definition in $P'$. For
instance, to each constant symbol $c \in L' \setminus L$, there corresponds a
defining formula $\de_c(y) \in \cL$ such that $\th_c = \forall y (y \beq c \tot
\de_c(y)) \in T_0'$. It seems, then, that the symbols in $L'$ can be divided
into two disjoint categories. Those in $L$ we might call the ``primitive
symbols,'' which lack a formal definition and get their meaning from their use.
While those is $L' \setminus L$ we might call the ``defined symbols,'' which get
their meaning by virtue of being defined in terms of the primitive symbols.

But this distinction is metatheoretical. Neither the deductive theory $T_0'$ nor
the inductive theory $P'$ can ``see'' which symbols are primitive and which are
defined. From the point of view of $P'$, a constant symbol $c \in L' \setminus
L$ is a primitive symbol that gets its meaning from the statements and sentences
in $P'$ and $T_0'$, one of which is $\th_c$. In other words, there is no real
difference between thinking of $\th_c$ as a definition of $c$ and thinking of
$\th_c$ as a new axiom that gives meaning to the primitive symbol $c$. In the
context of a given inductive theory such as $P'$, there are no defined symbols.
There are only primitive symbols.

Moreover, if we understand meaning as coming from use, which it does in formal
logical systems such as this, then the very act of defining the new symbol $c
\in L' \setminus L$ can change the meaning of the primitive symbols in $L$. More
specifically, to define $c$, we must add the new ``axiom,'' $\th_c = \forall y 
(y \beq c \tot \de_c(y))$, changing $T_0$ to $T_0'$. The formula $\de_c(y)$
undoubtedly uses symbols from $L$. Under the extension $P'$, those symbols are
now used differently from how they were used under $P$. Hence, their meanings in
$P'$ and $P$ may be different.

In practice, this latter issue may be one that we rarely, if ever, encounter.
This is because we can avoid the issue in any circumstance where the elimination
theorems apply. But we will later see in Chapter \ref{Ch:PoI} (specifically
Sections \ref{S:0<1-defn} and \ref{S:[0,1]-defn}) circumstances where the
elimination theorems do not apply. And in those circumstances, we must confront
this issue head on.

\section{Zermelo–Fraenkel set theory}\label{S:ZFC}

In this section, we present Zermelo-Fraenkel set theory in the infinitary
setting. Let $\cL$ be a language that contains a binary relation symbol $\bin$.
We use the boldface $\bin$ here to distinguish it from the usual $\in$ that is
used when discussing structures and models. As noted in Section
\ref{S:pred-formulas}, we will use the shorthand $(\forall y \bin x) \ph =
\forall y (y \bin x \to \ph)$ and $(\exists y \bin x) \ph = \exists y (y \bin x
\wedge \ph)$. More generally,
\[
  (\forall y_1 \cdots y_n \bin x) \ph =
  (\forall y_1 \bin x) \cdots (\forall y_1 \bin x) \ph,
\]
and similarly for $\exists$. We also write $x \bsub y = \forall z (z \bin x \to
z \bin y)$. Again, we use boldface $\bsub$ to distinguish it from the usual
$\subseteq$.

All of the formulas we define below are sentences. To simplify notation, we
write them as open formulas of the form $\ph(\vec x)$. It is to be understood
that this refers to the sentence $\forall \vec x \ph(\vec x)$.

\subsection{Extensionality, union, and power set}

Define the sentences
\begin{align*}
  \AE &: \forall z (z \bin x \tot z \bin y) \to x \beq y\\
  \AU &: \forall x \exists y \forall z (
    z \bin y \tot (\exists u \bin x) z \bin u
  )\\
  \AP &: \forall x \exists y \forall z (z \bin y \tot z \bsub x)
\end{align*}
These are, respectively, the axioms of extensionality, union, and power set. The
first says that two sets $x$ and $y$ are equal if they have the same elements.
The second says that, given a collection $x$ of sets, there is a set $y$ which
is the union of the sets in $x$. The third says that if $x$ is a set, then there
is a set $y$ consisting of all the subsets of $x$. We will have more to say
about these axioms later.

\subsection{Axiom schema of separation}\label{S:separation}

For $\ph(x, z, \vec u) \in \cL$ with $y \notin \free \ph$, define
\[
  \AS(\ph) : \exists y \forall z (z \bin y \tot \ph \wedge z \bin x)
\]
This is called the axiom of separation. Given a set $x$, the axiom $\AS(\ph)$
allows us to create the set $\{z \bin x \mid \ph(x, z, \vec u)\}$, which depends
on $x$ and $\vec u$. This is done as follows. Let $\ph(x, z, \vec u) \in \cL$
with $y \notin \free \ph$. It can be shown that
\[
  \AE, \AS(\ph) \vdash \forall x \vec u \; \exists! y \forall z (
    z \bin y \tot \ph \wedge z \bin x
  ).
\]
Hence, in any theory that contains $\AE$ and $\AS(\ph)$, we may make the
legitimate definitorial extension, $y \beq F x \vec u \tot \forall z (z \bin y
\tot \ph \wedge z \bin x)$. The term $F x \vec u$ is exactly the set we were
trying to create. In such a case, we use the notation $\{z \bin x \mid \ph\}$ or
$\{z \bin x \mid \ph(x, z, \vec u)\}$ as shorthand for the term $F x \vec u$.

The axiom of separation is, in fact, an axiom schema. It is one axiom for each
allowable formula $\ph$. Let
\[
  \AS = \{
    \AS(\ph) \mid \ph(x, z, \vec u) \in \cL \text{ and } y \notin \free \ph
  \},
\]
and set $\AS_\fin = \AS \cap \cL^0_\fin$. Note that $\AS(\ph) \in \cL^0_\fin$ if
and only if $\ph \in \cL_\fin$. Hence, $\AS_\fin$ can be defined just as $\AS$,
but with the requirement that $\ph \in \cL_\fin$. In other words, $\AS_\fin$ is
the usual axiom schema of separation used in first-order logic. The difference
between $\AS$ and $\AS_\fin$ is that with $\AS$, we are allowed to use
infinitary formulas $\ph$ when building new sets.

\subsection{Axiom schema of replacement}\label{S:replacement}

For $\ph(x, y, \vec z\,) \in \cL$ with $u, v \notin \free \ph$, define
\[
  \AR(\ph) : \forall x \exists! y \, \ph \to \forall u \exists v \forall y (
    y \bin v \tot (\exists x \bin u) \ph
  )
\]
This is called the axiom of replacement. Given a set $u$ and a function $F$, it
allows us to create a set of the form $\{F x \vec z \mid x \in u\}$, which
depends on $u$ and $\vec z$. The function $F$ is determined by the formula
$\ph$. This is done as follows. Let $\ph(x, y, \vec z\,) \in \cL$ with $u, v
\notin \free \ph$. Suppose $T$ is a theory such that $T \vdash \forall x \vec z
\; \exists! y \, \ph$. Also assume $\AE \in T$ and $\AR(\ph) \in T$. In $T$, we
may make the legitimate definitorial extension, $y \beq F x \vec z \tot \ph(x,
y, \vec z\,)$. Let
\[
  \psi(u, v, \vec z\,) = \forall y (y \bin v \tot (\exists x \bin u) \ph),
\]
and note that
\[
  \psi \equiv_T \forall y (
    y \bin v \tot (\exists x \bin u) \, y \beq F x \vec z\,
  ).
\]
By our hypotheses on $T$, it follows that $T \vdash \forall u \vec z \; \exists!
v \, \psi(u, v, \vec z\,)$. Hence, we may make the legitimate definitorial
extension, $v \beq G u \vec z \tot \psi(u, v, \vec z\,)$. The term $G u \vec z$
is exactly the set we were trying to create. In such a case, we use the notation
$\{F x \vec z \mid x \in u\}$ as shorthand for the term $G u \vec z$.

The axiom of replacement is, in fact, an axiom schema. It is one axiom for each
allowable formula $\ph$. Let
\[
  \AR = \{
    \AR(\ph) \mid \ph(x, y, \vec z\,) \in \cL \text{ and } u, v \notin \free \ph
  \},
\]
and set $\AR_\fin = \AR \cap \cL^0_\fin$. Note that $\AR(\ph) \in \cL^0_\fin$ if
and only if $\ph \in \cL_\fin$. Hence, $\AR_\fin$ can be defined just as $\AR$,
but with the requirement that $\ph \in \cL_\fin$. In other words, $\AR_\fin$ is
the usual axiom schema of replacement used in first-order logic. The difference
between $\AR$ and $\AR_\fin$ is that with $\AR$, we are allowed to use
infinitary formulas to construct our defined function symbol $F$.

\subsection{Definitorial extensions and shorthand}

To state the remaining axioms, it will be useful to create new symbols, both
shorthand and formally defined extralogical symbols. To this end, we define
\begin{align*}
  \ph_1(x, z) &: z \nbin x\\
  \ph_2(x, y, z_1, z_2) &: {
    (\forall u \, u \nbin x) \wedge y \beq z_1
    \vee (\exists u \, u \bin x) \wedge y \beq z_2
  }\\
  \ph_3(x, z, u) &: z \bin u
\end{align*}
Note that each $\ph_i$ is in $\cL_\fin$. Let $T$ be the deductive theory
generated by
\[
  \{\AE, \AU, \AP, \AS(\ph_1), \AS(\ph_3), \AR(\ph_2)\}.
\]
Since $\AE, \AS(\ph_1) \vdash \exists! y \forall z \, z \nbin y$, we may make
the legitimate definitorial extension, $y \beq \bemp \tot \forall z \, z \nbin
y$. Letting
\[
  \ph(x, y) = \forall z (z \bin y \tot (\exists u \bin x) z \bin u),
\]
we have $\AE, \AU \vdash \forall x \exists! y \, \ph(x, y)$, which allows the
extension $y \beq \bigcup x \tot \ph(x, y)$. And with $\ph(x, y) = \forall z (z
\bin y \tot z \bsub x)$, we have $\AE, \AP \vdash \forall x \exists! y \, \ph
(x, y)$, which allows the extension $y \beq \fP x \tot \ph(x, y)$.

We now have
\[
  \ph_2 \equiv_T (
    x \beq \bemp \wedge y \beq z_1 \vee x \nbeq \bemp \wedge y \beq z_2
  ).
\]
It can be shown that $\AE, \AS(\ph_1) \vdash \forall x \vec z \; \exists! y \,
\ph_2(x, y, \vec z)$. This allows the extension $y \beq F x \vec z \tot \ph_2
(x, y, \vec z\,)$. Hence, since $\AR(\ph_2) \in T$, we may adopt the shorthand
$\{z_1, z_2\}$ for the term $\{F x \vec z \mid x \in \fP \fP \bemp\}$. We then
write $\{z\}$ as shorthand for $\{z, z\}$. We can recursively define the
shorthand $\{z_1, \ldots, z_{n + 1}\} = \{z_1, \ldots, z_n\} \cup \{z_{n +
1}\}$.

Since $\AE, \AS(\ph_3) \in T$, we have the term $\{z \bin x \mid \ph_3\} = \{z
\in x \mid z \in u\}$. We use the shorthand $x \cap u$ to denote this term. We
also write $x \cup u$ as shorthand for $\bigcup \{x, u\}$, and $\S x$ as
shorthand for the term $x \cup \{x\}$. Finally, the ordered pair $(x, y)$ is
defined via shorthand as the term $\{\{x\}, \{x, y\}\}$. We can recursively
define the shorthand $(x_1, \ldots, x_{n + 1}) = ((x_1, \ldots, x_n), x_{n +
1})$.

\subsection{Axioms of infinity, foundation, and choice}

Now define the sentences
\begin{align*}
  \AI &: \exists u (\bemp \bin u \wedge \forall x (x \bin u \to \S x \bin u))\\
  \AF &: (\forall x \nbeq \bemp)(\exists y \bin x) x \cap y \beq \bemp\\
  \AC &: \forall u (
    \bemp \nbin u
    \wedge (\forall x y \bin u)(x \nbeq y \to x \cap y \beq \bemp)
    \to \exists z (\forall x \bin u) \exists! y (y \bin x \cap z)
  )
\end{align*}
These are, respectively, the axioms of infinity, foundation, and choice. The
axiom of infinity ensures the existence of an infinite set. The axiom of
foundation, among other things, ensure that no set can be an element of itself.
The axiom of choice asserts the existence of a set containing exactly one
element from each of a given collection of sets.

\subsection{Finitary and infinitary \texorpdfstring{$\ZFC$}{ZFC}}

We now define
\[
  \ZFC_- = T + \AS_\fin + \AR_\fin + \{\AI, \AF, \AC\},
\]
and set $\ZFC_\fin = \ZFC_- \cap \cL^0_\fin$.
  \symindex{$\ZFC_\fin$}%
  \symindex{$\ZFC_-$}%
Note that $\ZFC_- = T(\La^\ZFC_-)$, where
\[
  \La^\ZFC_- = \{\AE, \AU, \AP, \AI, \AF, \AC\} \cup \AS_\fin \cup \AR_\fin
\]
are the usual finitary axioms of set theory. Hence, from Proposition
\ref{P:pred-fin-vs-infin}, it follows that $\ZFC_\fin$ is the usual
Zermelo–Fraenkel set theory with the axiom of choice, formulated in first-order
logic.

We also define
\[
  \ZFC = T + \AS + \AR_\fin + \{\AI, \AF, \AC\},
\]
  \symindex{$\ZFC$}%
Note that $\ZFC = T(\La^\ZFC)$, where
\[
  \La^\ZFC = \{\AE, \AU, \AP, \AI, \AF\} \cup \AS \cup \AR_\fin.
\]
The difference between $\ZFC$ and $\ZFC_-$ is that, in $\ZFC$, we are allowed to
use infinitary formulas when applying the axiom schema of separation.

Finally, we define
\[
  \ZFC_+ = T + \AS + \AR + \{\AI, \AF, \AC\},
\]
  \symindex{$\ZFC_+$}%
Note that $\ZFC_+ = T(\La^\ZFC \cup \AR)$. In $\ZFC_+$, we are also allowed to
use infinitary formulas when applying the axiom schema of replacement. Also note
that $\La^\ZFC_- \subseteq \La^\ZFC$ and $\ZFC_\fin \subseteq \ZFC_- \subseteq
\ZFC \subseteq \ZFC_+$.

By the same reasoning as in the proof of Proposition \ref{P:models-of-PA-}, we
obtain the following.

\begin{prop}\label{P:models-of-ZFC-}
  Let $\sP = (\Om, \Si, \bbP)$ be a model. Then $\sP \vDash \ZFC_-$ if and only
  if $\om \tDash \ZFC_\fin$ for $\bbP$-a.e.~$\om \in \Om$. Consequently,
  $\ZFC_-$ is consistent if and only if $\ZFC_\fin$ is consistent. Moreover,
  $\ZFC_- \vdash \ph$ if and only if $\om \tDash \La^\ZFC_-$ implies $\om \tDash
  \ph[v]$ for all $\om$ and all assignments $v$ into $\om$.
\end{prop}

\subsection{Consistency of \texorpdfstring{$\ZFC$}{ZFC}}
\label{S:Con-ZFC}

In first-order logic, the consistency of $\ZFC_\fin$ is implied by a variety of
different sufficient conditions. One such condition is the following consequence
of \cite[Theorem 8.2.8]{Cenzer2020}.

\begin{thm}\label{T:Con-ZFC-fin}
  If there exists a strongly inaccessible cardinal, then $\ZFC_\fin$ is
  consistent in first-order logic.
\end{thm}

The existence of a strongly inaccessible cardinal cannot be proven in
first-order logic using the standard axioms of set theory, $\La^\ZFC_-$. If it
could, then these axioms could prove their own consistency, in violations of
G\"odel's second incompleteness theorem (see \cite[Theorem
7.3.2]{Rautenberg2010}). In situations where we want to use a strongly
inaccessible cardinal, we must assume it exists, effectively adding its
existence as a new axiom.

In the comprehensive articles, \cite{Maddy1988} and \cite{Maddy1988a}, Penelope
Maddy discusses the metamathematical arguments, both historical and present, for
accepting not only the current axioms, $\La^\ZFC$, but also additional possible
axioms, including the assumption that strongly inaccessible cardinals exist. Of
all the non-$\ZFC$ axioms discussed, this so-called \emph{Axiom of
Inaccessibles} seems to have the most support, with many historical backers,
including G\"odel. It is supported by a number of metamathematical principles,
which she calls \emph{maximize}, \emph{inexhaustibility}, \emph{uniformity},
\emph{whimsical identity}, and \emph{reflection}. This last principle, she says,
is ``probably the most universally accepted rule of thumb in higher set
theory.''

In any case, if we want to have an inductive theory whose root contains either
$\ZFC_-$ or $\ZFC$, then we will need to assume something that ensures these
theories are consistent. Of all the assumptions we could make in this regard, we
prefer the one in Theorem \ref{T:Con-ZFC-fin}. It is a very mild assumption,
compared to others we might make. It has good metamathematical support. And it
seems to produce a number of helpful results for us. Our first example of this
is Theorem \ref{T:Con-ZFC} below, which shows that this same assumption gives us
the consistency of $\ZFC$.

To prove this result, we must first see how the existence of a strongly
inaccessible cardinal implies the consistency of $\ZFC_\fin$ in Theorem
\ref{T:Con-ZFC-fin}. For this, we begin by defining the \emph{von Neumann
hierarchy}, which is a collection of sets indexed by the ordinals.
  \index{von Neumann hierarchy}%
For each ordinal $\al$, we recursively define the set $V_\al$ as follows. Let
$V_0 = \emp$. If $\al = \be + 1$, then let $V_{\be + 1} = \fP V_\be$. If $\al$
is a limit ordinal, then let $V_\al = \bigcup_{\xi < \al} V_\xi$. The sets
$V_\al$ satisfy $V_\al = \bigcup_{\be < \al} \fP V_\be$, for any ordinal $\al$.
Each $V_\al$ is a transitive set, meaning that if $A \in V_\al$ and $x \in A$,
then $x \in V_\al$. The sets $V_\al$ also satisfy the following properties:
\begin{enumerate}[(i)]
  \item $\be < \al$ implies $V_\be \in V_\al$ and $V_\be \subseteq V_\al$,
  \item $A \in V_\al$ implies $A \subseteq V_\al$,
  \item $B \subseteq V_\be$ implies $B \in V_\al$ for all $\al > \be$, and
  \item $\al \subseteq V_\al$ for all $\al$.
\end{enumerate}
Using our identification of natural numbers and ordinals, the first five sets in
the von Neumann hierarchy can be written as
\begin{align*}
  V_0 &= \emp,\\
  V_1 &= \{0\},\\
  V_2 &= \{0, 1\},\\
  V_3 &= \{0, 1, 2, \{1\}\},\\
  V_4 &= {
    \{0, 1, 2, 3, \{1\}, \{2\}, \{0, 2\}, \{1, 2\},
    \{\{1\}\}, \{0, \{1\}\}, \{0, 1, \{1\}\}, \{0, 1, 2, \{1\}\},
  }\\
  &\qquad \{1, \{1\}\}, \{2, \{1\}\}, \{0, 2, \{1\}\}, \{1, 2, \{1\}\}\}.
\end{align*}
Note that $|V_5| = 65536$ and $|V_6| = 2^{65536}$. Although these sets grow
rapidly, we have $|V_n| < \infty$ for all $n \in \bN_0$. Also note that $n
\subseteq V_n$ and $n \in V_{n + 1}$, for all $n \in \bN_0$. Since $\bN_0 =
\bom$, we have $\bN_0 \subseteq V_{\bom}$ and $\bN_0 \in V_{\bom + 1}$.

If $\ka$ is a cardinal number, then we define the $L$-structure $\bnu_\ka =
(V_\ka, \bin^{\bnu_\ka})$ by setting ${\bin^{\bnu_\ka}} = {\in}$.

\begin{lemma}\label{L:V-is-closed}
  If $\bnu_\ka \tDash \La^\ZFC_-$, then $V_\ka$ satisfies the following:
  \begin{enumerate}[(i)]
    \item $\bnu_\ka \tDash (\forall x \, x \nbin y)[b]$ if and only if $b =
          \emp$,
    \item $\bnu_\ka \tDash (x \bsub y)[a, b]$ if and only if $a \subseteq b$,
    \item if $a \subseteq b$ and $b \in V_\ka$, then $a \in V_\ka$, and
    \item if $b \in V_\ka$, then $\fP b \in V_\ka$.
  \end{enumerate}
\end{lemma}

\begin{proof}
  Let $b \in V_\ka$. Note that $\bnu_\ka \tDash (\forall x \, x \nbin y)[b]$ if
  and only if, for all $a \in V_\ka$, we have $a \notin b$. Suppose $b \ne
  \emp$. Then there is an object $a$ such that $a \in b$. Since $V_\ka$ is a
  transitive set, this implies $a \in V_\ka$, a contradiction. Hence, (i) holds.

  Next, let $a, b \in V_\ka$. Note that $\bnu_\ka \tDash (x \bsub y)[a, b]$ if
  and only if, for all $c \in V_\ka$, we have $c \in a$ implies $c \in b$.
  Suppose $a \nsubseteq b$. Then there is an object $c$ such that $c \in a$ but
  $c \notin b$. Again, since $V_\ka$ is transitive, we have $c \in V_\ka$, a
  contradiction. Therefore, (ii) holds.

  Let $a \subseteq b$ and $b \in V_\ka$. Since every infinite cardinal
  number is a limit ordinal, we have $V_\ka = \bigcup_{\be < \ka} V_\be$ and
  $V_\ka = \bigcup_{\be < \ka} \fP V_\be$. Choose $\be < \ka$ such that $b \in
  V_\be$, and note that $\be + 1 < \ka$. Then $b \subseteq V_\be$, so that $a
  \subseteq V_\be$, which implies $a \in V_{\be + 1} \subseteq V_\ka$, and (iii)
  holds.

  Finally, suppose $b \in V_\ka$. Choose $\be < \ka$ such that $b \in V_\be$.
  Then $b \subseteq V_\be$, so that every $a \subseteq b$ satisfies $a \in V_
  {\be + 1}$. Hence, $\fP b \subseteq V_{\be + 1}$, and this implies that $\fP b
  \in V_{\be + 2} \subseteq V_\ka$.
\end{proof}

\begin{thm}\label{T:Con-ZFC}
  Let $\bnu_\ka$ be as above. Then the following are equivalent:
  \begin{enumerate}[(i)]
    \item $\ka$ is strongly inaccessible,
    \item $\bnu_\ka \tDash \La^\ZFC_-$,
    \item $\bnu_\ka \tDash \La^\ZFC$
  \end{enumerate}
  In particular, if there exists a strongly inaccessible cardinal, then $\ZFC$
  is strictly satisfiable.
\end{thm}

\begin{proof}
  The equivalence of (i) and (ii) is \cite[Theorem 8.2.8]{Cenzer2020}, and (iii)
  implies (ii) follows from the fact that $\La^\ZFC_- \subseteq \La^\ZFC$. For
  the final implication, assume (ii) holds. We need to show that, for all $\ph
  \in \cL$,
  \begin{equation}\label{Con-ZFC}
    \text{
      if $\ph = \ph(x, z, \vec u)$ and $y \notin \free \ph$,
      then $\bnu_\ka \tDash \AS(\ph)$.
    }
  \end{equation}
  We will prove this by induction on $\ph$. Since prime formulas are finitary
  and $\bnu_\ka \tDash \La^\ZFC_-$, it holds whenever $\ph$ is prime.

  Suppose $\ph = \neg \psi$ and \eqref{Con-ZFC} holds for $\psi$. Assume $\ph =
  \ph(x, z, \vec u)$ and $y \notin \free \ph$. Then the same is true for $\psi$.
  Hence, $\bnu_\ka \tDash \AS(\psi)$. Let $a, d_1, \ldots, d_n \in V_\ka$. Then
  there exists $b \in V_\ka$ such that, for all $c \in V_\ka$, we have $c \in b$
  if and only if $c \in a$ and $\bnu_\ka \tDash \psi[a, c, \vec d]$. Since
  $V_\ka$ is transitive, this implies $b = \{c \in a \mid \bnu_\ka \tDash
  \psi[a, c, \vec d]\}$. By Lemma \ref{L:V-is-closed}(iii), we have $a \setminus
  b \in V_\ka$. But $a \setminus b = \{c \in a \mid \bnu_\ka \tDash (\neg
  \psi)[a, c, \vec d]\}$. Hence, $\bnu_\ka \tDash \AS(\neg \psi)$.

  Now suppose $\ph = \bigwedge \Phi$ and \eqref{Con-ZFC} holds for each $\th \in
  \Phi$, and assume that $\ph = \ph(x, z, \vec u)$ and $y \notin \free \ph$.
  Then the same is true for each $\th$, so that $\bnu_\ka \tDash \AS(\th)$. Let
  $a, d_1, \ldots, d_n \in V_\ka$. As above, $b_\th = \{c \in a \mid \bnu_\ka
  \tDash \th[a, c, \vec d]\} \in V_\ka$. Lemma \ref{L:V-is-closed}(iii) then
  gives $b = \bigcap_{\th \in \Phi} b_\th \in V_\ka$. But $b = \{c \in a \mid
  \bnu_\ka \tDash (\bigwedge \Phi)[a, c, \vec d]\}$, so that $\bnu_\ka \tDash
  \AS(\bigwedge \Phi)$.

  Finally, suppose $\ph = \forall v \psi$ and \eqref{Con-ZFC} holds for $\psi$.
  Assume $\ph = \ph(x, z, \vec u)$ and $y \notin \free \ph$. Rename $v$ to $u_{n
  + 1}$, so that $\bnu_\ka \tDash \AS(\psi)$. Let $a, d_1, \ldots, d_n \in
  V_\ka$. As above, for each $e \in V_\ka$, we have $b_e = \{c \in a \mid
  \bnu_\ka \tDash \psi[a, c, \vec d, e]\} \in V_\ka$. By Lemma
  \ref{L:V-is-closed}(iii), it follows that $b = \bigcap_{e \in V_\ka} b_e \in
  V_\ka$. But $b = \{c \in a \mid \bnu_\ka \tDash (\forall v \psi)[a, c, \vec
  d]\}$, so that $\bnu_\ka \tDash \AS(\forall v \psi)$.
\end{proof}

\section{Real inductive theories in \texorpdfstring{$\ZFC_-$}{ZFC-}}
\label{S:ind-th-ZFC-}

In this section, we show how to represent real numbers in $\ZFC_-$. We then use
this representation to construct real inductive theories in $\ZFC_-$.

\subsection{The set of natural numbers}

Let us add to the language of $\ZFC_-$ a constant symbol $\ul n$ for each $n \in
\bN_0$. We do this through the definitorial extensions, $y \beq \ul 0 \tot y
\beq \bemp$ and $y \beq \ul n \tot y \beq \S \cdots \S \bemp$, where $\S$ is
repeated $n$ times. The symbols $\ul n$ are syntactic representations of the
natural numbers. They give us a way to define each individual natural number in
$\ZFC_-$. But defining each individual natural number is not the same as
defining the set of natural numbers. We must define a set that contains the
natural numbers, and nothing but the natural numbers. One particularly
counterintuitive result in $\ZFC_-$ is that this is impossible, in the sense
expressed by \eqref{nonstd-N-1} below.

Instead, the best we can do in $\ZFC_-$ is to prove that there is a smallest set
that contains the natural numbers. But we cannot prove that this smallest set
contains only the natural numbers. For instance, it is consistent with $\ZFC_-$
to postulate that this smallest set contains an object which is greater than
every natural number. Such an object would be called a \emph{nonstandard natural
number}. Similarly, $\ZFC_-$ is consistent with the existence of nonstandard
integers, rationals, and real numbers. On the other hand, in $\ZFC$, we can
prove that there is a unique set that contains the natural numbers, and nothing
else. This is expressed in \eqref{nonstd-N-2} below, whose proof uses
essentially the same technique as in the proof of Proposition
\ref{P:models-of-PA}.

In $\ZFC_-$, we define the smallest set that contains the natural numbers as
follows. Let $\ph(u) = \ul 0 \bin u \wedge \forall x (x \bin u \to \S x \bin
u)$, so that $\AI = \exists u \, \ph(u)$, and let $\de(y) = \ph(y) \wedge
\forall z (\ph(z) \to y \bsub z)$.

\begin{prop}\label{P:defn-N-ZFC}
  With notation as above, we have $\La^\ZFC_- \vdash \exists! y \, \de(y)$.
\end{prop}

\begin{proof}
  Let $\psi(x) = \forall z (\ph(z) \to x \bin z)$. Since $\ZFC_- \vdash
  \AS(\psi)$, we may define the term $t(u) = \{x \in u \mid \psi(x)\}$, which
  satisfies $\ZFC_- \vdash \ph(u) \to \de(t(u))$. Hence, since $\ZFC_- \vdash
  \AI$ and $\AI = \exists u \, \ph(u)$, it follows that $\ZFC_- \vdash \exists y
  \, \de(y)$. For uniqueness, simply note that $\ZFC_- \vdash \de(x) \wedge
  \de(y) \to x \bsub y \wedge y \bsub x$.
\end{proof}

By Proposition \ref{P:defn-N-ZFC}, we can make the legitimate definitorial
extension $y = \ul{\bN_0} \tot \de(y)$. In words, $\ul{\bN_0}$ is the smallest
set that contains $\ul 0$ and is closed under the successor operation.

\subsection{Arithmetic operations}

Recall that the ordered pair $(x, y)$ is shorthand for the set $\{\{x\}, \{x,
y\}\}$. Note that if $\ZFC_- \vdash x \bin u, y \bin v$, then $\ZFC_- \vdash (x,
y) \bin \fP \fP (u \cup v)$. Hence, we may adopt the shorthand
\[
  u \times v = \{
    z \bin \fP \fP (u \cup v)
  \mid
    \exists x y (z \beq (x, y) \wedge x \bin u \wedge y \bin v)
  \},
\]
and
\[
  v^u = \{
    z \bin \fP (u \times v)
  \mid
    (\forall x \in u) (\exists! y \in v) (x, y) \bin z
  \}.
\]
Here, $u \times v$ is the Cartesian product of $u$ and $v$, and $v^u$ is the set
of functions from $u$ to $v$. We adopt the usual shorthand function notation,
$z(x)$. Namely, if $\ph(y) \in \cL$, then $\ph(z(x))$ is shorthand for the
sentence,
\[
  \exists u v (z \bin v^u \wedge x \bin u)
  \wedge \exists y ((x, y) \bin z \wedge \ph(y)).
\]
In other words, the sentence $\ph(z(x))$ says that $z$ is a function, $x$ is in
the domain of $z$, the object $y$ is the unique object such that $(x, y)$ is in
$z$, and $\ph(y)$ holds.

With these notions in hand, we can define addition in one of several equivalent
ways, each ending with an explicitly defined constant symbol $+$ such that
$\ZFC_- \vdash {+} \bin \ul \bN_0^{\ul \bN_0 \times \ul \bN_0}$, and which
agrees with ordinary addition on $\bN_0$. This latter fact means, specifically,
that
\[
  \text{
    $\ZFC_- \vdash (\ul m, \ul n, \ul k) \bin {+}$ if and only if $m + n = k$.
  }
\]
Note that in the above, the first instance of $+$ is an extralogical symbol and
the second denotes ordinary addition of natural numbers. Since we use the same
typographical symbol for both, we will need to rely on context to tell the
difference. We adopt the usual shorthand, writing $x + y \beq z$ to mean $(x, y,
z) \bin {+}$. We then do the same for $\bdot$ and $<$.

\subsection{Peano arithmetic and nonstandard numbers}

The axioms of Peano arithmetic can all be translated into the language of
$\ZFC_-$ by replacing each quantifier $\forall x$ with $(\forall x \bin
\ul{\bN_0})$. By an abuse of notation, we will denote these translated axioms
also by $\La^\PA_-$. We perform a similar abuse with $\La^\PA$, $\PA_-$, and
$\PA$. It can be shown that $\PA_- \subseteq \ZFC_-$, and in a similar way, that
$\PA \subseteq \ZFC$. Combined with Proposition \ref{P:models-of-PA}, this
latter fact tells us that if $\cN \tDash \ph$, then $\ZFC \vdash \ph$.

The arithmetical completeness of $\ZFC$ is intrinsically connected to the
following result. The first part of the result, \eqref{nonstd-N-1}, shows the
impossibility of proving in $\ZFC_-$ that $\ul 0, \ul 1, \ul 2, \ldots$ are the
only natural numbers. The second part, \eqref{nonstd-N-2}, shows that this is
not a problem in $\ZFC$.

\begin{prop}
  If $\ZFC_-$ is consistent, then
  \begin{equation}\label{nonstd-N-1}
    \ts{
      \ZFC_- \nvdash \exists y \forall x (
        x \bin y \tot \bigvee_{n \in \bN_0} x \beq \ul n
      ).
    }
  \end{equation}
  On the other hand,
  \begin{equation}\label{nonstd-N-2}
    \ts{
      \ZFC \vdash \forall x (
        x \bin \ul{\bN_0} \tot \bigvee_{n \in \bN_0} x \beq \ul n
      ).
    }
  \end{equation}
\end{prop}

\begin{proof}
  For notational simplicity, let $\psi(y) = \forall x (x \bin y \tot \bigvee_{n
  \in \bN_0} x \beq \ul n)$, so that we aim to show $\ZFC_- \nvdash \exists y \,
  \psi(y)$ and $\ZFC \vdash \psi(\ul{\bN_0})$.

  Suppose $\ZFC_-$ is consistent and $\ZFC_- \vdash \exists y \, \psi(y)$. Let
  $\ph(u)$ and $\de(y)$ be as in the definition of $\ul{\bN_0}$. Note that
  $\ZFC_- \vdash \psi(y) \to \ph(y)$. Also, $\ZFC_- \vdash \psi(y) \to \ph(z)
  \to y \bsub z$. Hence, $\ZFC_- \vdash \psi(y) \to \de(y)$. It therefore
  follows from $\ZFC_- \vdash \exists y \, \psi(y)$ that $\ZFC_- \vdash
  \psi(\ul{\bN_0})$. In particular, we have $\ZFC_- \vdash \ze(x)$, where
  $\ze(x) = x \bin \ul{\bN_0} \to \bigvee_{n \in \bN_0} x \beq \ul n$.

  By Proposition \ref{P:models-of-ZFC-}, since $\ZFC_-$ is consistent, we may
  find a structure $\om$ such that $\om \tDash \La^\ZFC_-$. Let $c$ be a
  constant not in $L$ and define $X \subseteq (\cL c)^0_\fin$ by $X = \La^\ZFC_-
  \cup \{c \bin \ul{\bN_0}\} \cup \{c \nbeq \ul n \mid n \in \bN_0\}$. Let $X_0
  \subseteq X$ be finite. Choose $m \in \bN_0$ such that
  \[
    X_0 \subseteq {
      \La^\ZFC_- \cup \{c \bin \ul{\bN_0}\} \cup \{c \nbeq \ul n \mid n < m\}
    }.
  \]
  Let $\om'$ be the $L c$-expansion of $\om$ with $c^{\om'} = \ul m^\om$. Then
  $\om' \tDash X_0$, so that $X_0$ is strictly satisfiable. By the finiteness
  theorem for first-order logic (see, for example, \cite[Theorem 3.3.1]
  {Rautenberg2010}), the set $X$ is strictly satisfiable. Choose an $L
  c$-structure $\nu$ such that $\nu \tDash X$, and let $\nu_0$ be its
  $L$-reduct. Let $a = c^\nu$. Then $\nu_0 \tDash \La^\ZFC_-$, $\nu_0 \tDash (x
  \bin \ul{\bN_0})[a]$, and $\nu_0 \tDash (x \nbeq \ul n)[a]$ for all $n \in
  \bN_0$. Therefore, $\nu_0 \ntDash \ze[a]$. Now let $\sP = (\{\nu_0\}, \fP
  \{\nu_0\}, \de_{\nu_0})$ and $\bv = \ang{v}$, where $v(x) = a$. Then $\sP
  \vDash \ZFC_-$ and $\sP \nvDash \ze[\bv]$. By Theorem
  \ref{T:pred-completeness}, we have $\ZFC_- \nvdash \ze(x)$, a contradiction.
  This proves \eqref{nonstd-N-1}.

  We now consider \eqref{nonstd-N-2}. By the definition of $\ul{\bN_0}$, we have
  $\ZFC_- \vdash \bigvee_{n \in \bN_0} x \beq \ul n \to x \bin \ul \bN_0$.
  Therefore, we need only prove that $\ZFC \vdash \ze(x)$. By Definition
  \ref{D:pred-derivability}(vii)$'$, it suffices to show
  \begin{equation}\label{nonstd-N-3}
    \ts{
      \ZFC \vdash \forall x \ze = (\forall x \bin \ul{\bN_0}) \, \xi(x),
    }
  \end{equation}
  where $\xi(x) = (\bigvee_{n \in \bN_0} x \beq \ul n)$.

  By $\AS(\xi)$, we may define the term $t = \{x \bin \ul{\bN_0} \mid \xi(x)\}$.
  It is straightforward to verify that $\ZFC \vdash \ph(t)$ and $\ZFC \vdash \ph
  (z) \to t \bsub z$. Hence, $\ZFC \vdash t \beq \ul{\bN_0}$. Since $t \beq
  \ul{\bN_0} \equiv_{\ZFC_-} (\forall x \bin \ul{\bN_0}) \, \xi(x)$, this gives
  \eqref{nonstd-N-3}.
\end{proof}

\subsection{Real numbers in \texorpdfstring{$\ZFC_-$}{ZFC-}}

In our construction of $\ZFC_-$ in Section \ref{S:ZFC}, we implemented a number
of definitorial extensions, so that the extralogical signature of $\cL$ contains
not only $\bin$, but many other explicitly defined symbols, such as $\bemp$,
$\bigcup$, and $\fP$. For the present treatment, we consider $\ZFC_-$ to begin
with the extralogical signature,
\[
  L = {
    \{\bin, \bemp, \S, \ul{\bN_0}, +, \bdot, <\} \cup \{\ul n \mid n \in \bN_0\}
  }.
\]
Any symbol which was defined in Section \ref{S:ZFC} but does not appear above is
considered to be reduced. We may continue to use some of those symbols, but
these uses should be considered shorthand, until otherwise specified. Note that
each symbol in $L$, except for $\bin$, is a constant symbol, and each is
explicitly defined with a finitary defining formula.

From this starting point, we can now explicitly define each integer $z \in \bZ$
in the usual way as an equivalence class of ordered pairs of natural numbers.
For example, $-5$ is explicitly defined by the formula,
\[
  \de(y) = \forall x (x \bin y \tot (\exists u v \bin \ul{\bN_0}) (
    x \beq (u, v) \wedge u + \ul 5 \beq v
  )).
\]
The set of such equivalence classes is also explicitly definable, so that we may
add the symbol $\ul \bZ$. However, as was the case with $\ul{\bN_0}$, we have
that $\ZFC_- \nvdash \forall x (x \bin \ul \bZ \tot \bigvee_{z \in \bZ} x \beq
\ul z)$, provided $\ZFC_-$ is consistent.

We can then explicitly define $+_\bZ$, $\bdot_\bZ$, and $<_\bZ$ for integers. In
this way we obtain a definitorial extension of $\ZFC_-$ with signature
\[
  L = {
    \{
      \bin, \bemp, \S, \ul{\bN_0}, \ul \bZ, +, \bdot, <, +_\bZ, \bdot_\bZ, <_\bZ
    \} \cup
    \{\ul n \mid n \in \bN_0\} \cup
    \{\ul z \mid z \in \bZ\}
  },
\]
where each symbol in $L$, except for $\bin$, is a constant symbol, and each is
explicitly defined with a finitary defining formula. We will omit the duplicates
of $+$, $\bdot$, and $<$, and leave the distinction to context. Similarly, we
will omit $\{\ul n \mid n \in \bN_0\}$, and leave to context the distinction
between the natural number $n$ and the integer $n$.

Finally, we do the same for each $q \in \bQ$ and for $\bQ$ itself, giving us the
signature,
\[
  L = {
    \{\bin, \bemp, \S, \ul{\bN_0}, \ul \bZ, \ul \bQ, +, \bdot, <\} \cup
    \{\ul q \mid q \in \bQ\}
  }.
\]
Again, everything but $\bin$ is a constant symbol with a finitary definition.

A set $B$ is a \emph{Dedekind cut} if $B$ is a nonempty, proper subset of $\bQ$
that is downward closed and has no maximum element.
  \index{Dedekind cut}%
Let
\begin{multline*}
  \ph_{DC}(u) = u \bin \fP \ul \bQ \wedge u \nbeq \bemp \wedge u \nbeq \ul \bQ\\
  \wedge (\forall x \bin u)(y \bin \ul \bQ \wedge y < x \to y \bin u)
  \wedge (\forall x \bin u)(\exists y \bin u)(x < y).
\end{multline*}
Then $\ph_{DC}(u)$ says that $u$ is a Dedekind cut. Since $\ph_{DC}$ is
finitary, we have $\ZFC_- \vdash \AS(\ph_{DC})$. Hence, if
\[
  \de(y) = \forall x (x \bin y \tot x \bin \fP \ul \bQ \wedge \ph_{DC}(x)),
\]
then $\ZFC_- \vdash \exists! y \, \de(y)$. We may therefore explicitly define
$\ul \bR$ by $y \beq \ul \bR \tot \de(y)$. We can also explicitly define
addition, multiplication, and less-than in $\ul \bR$, all with finitary
formulas. This gives us the extralogical signature,
\[
  L_- = {
    \{\bin, \bemp, \S, \ul{\bN_0}, \ul \bZ, \ul \bQ, \ul \bR, +, \bdot, <\} \cup
    \{\ul q \mid q \in \bQ\}
  }.
\]
  \symindex{$L_-$}%
As before, we have omitted the duplicate versions of $+$, $\bdot$, and $<$, and
will rely on context to understand them. Also, each symbol in $L_-$, except for
$\bin$, is an explicitly defined constant symbol with a finitary defining
formula.

\subsection{The standard real structure}\label{S:std-R}

Before constructing real inductive theories in $\ZFC_-$, we first show a
simpler, albeit more limited approach. This approach is essentially just a
special case of Theorem \ref{T:prob-model-iso}.

Let $L_\bR = \{+, \bdot, <\} \cup \{\ul r \mid r \in \bR\}$. In $\cL_\bR$, we
will write $x \le y$ as shorthand for $x < y \vee x \beq y$. Define the
$L_\bR$-structure $\cR = (\bR, L^\cR)$ by letting $+^\cR$, $\bdot^\cR$, and
$<^\cR$ denote their ordinary counterparts in $\bR$, and setting $\ul r^\cR =
r$. Define the deductive theory $T_\bR$ by $T_\bR = \{\ph \in \cL^0 \mid \cR
\tDash \ph\}$.

Let $L$ be an extralogical signature with $L_\bR \subseteq L$. If $P \subseteq
\cL^\IS$ is an inductive theory with root $T_0 \supseteq T_\bR$, then $P$ is
called a \emph{real inductive theory in $T_\bR$}.
  \index{theory!real inductive ---}%

In Theorem \ref{T:prob-model-iso}, we saw that every measure-theoretic
probability model can be represented by an inductive model in a certain
language. In Theorem \ref{T:prob-model-iso-TR} below, we show that if that
measure-theoretic model is real-valued, then it can be represented by an
inductive model in $\cL_\bR$.

Let $(S, \Ga, \nu)$ be a probability space and let $X = \ang{X_i \mid i \in I}$
be an indexed collection of real-valued random variables. That is, each $X_i$ is
a Borel-measurable function from $S$ to $\bR$. Assume $\Ga = \si(\ang{X_i \mid i
\in I})$.

Let $C = \{\ul X_i \mid i \in I\}$ be a set of distinct constant symbols not in
$L_\bR$, and define $L = L_\bR C$.

\begin{thm}\label{T:prob-model-iso-TR}
  There exists an $\cL$-model $\sP = (\Om, \Si, \bbP)$ with $\sP \vDash T_\bR$,
  and a function $h: S \to \Om$ mapping $x \in S$ to $\om \in \Om$ such that
  \begin{enumerate}[(i)]
    \item $x \in \{X_i \le r\}$ if and only if $\om \tDash (\ul X_i \le \ul r)$,
    \item each $U \in \Ga$ can be written as $U = h^{-1} \ph_\Om$ for some $\ph
          \in \cL^0$, and
    \item $h$ induces a measure-space isomorphism from $(S, \Ga, \nu)$ to $\sP$.
  \end{enumerate}
  Consequently, if $P = \bTh \sP \dhl_{[T_\bR, \Th \sP]}$, then
  \begin{equation}\label{prob-model-iso-TR}
    \ts{
      P(\bigwedge_{k = 1}^n \ul X_{i(k)} \le \ul{r_k} \mid T_\bR) =
      \nu \bigcap_{k = 1}^n \{X_{i(k)} \le r_k\},
    }
  \end{equation}
  whenever $i(1), \ldots, i(n) \in I$ and $r_k \in \bR$.
\end{thm}

\begin{proof}
  For each $x \in S$, define $\om = \om^x$ to be the $L$-expansion of $\cR$
  given by $\om^{\ul X_i} = X_i(x)$. Let $\Om = \{\om^x \mid x \in S\}$ and let
  $h: S \to \Om$ denote the map $x \mapsto \om^x$. Let $\sP = (\Om, \Si, \bbP)$
  be the measure space image of $(S, \Ga, \nu)$ under $h$. Since $\om \tDash
  T_\bR$ for all $\om \in \Om$, we have $\sP \vDash T_\bR$. By construction, we
  have $X_i(x) \le r$ if and only if $\om^x \tDash \ul X_i \le \ul r$, so (i)
  holds.

  For (ii), let
  \[
    \Ga' = \{
      U \in \Ga \mid U = h^{-1} \ph_\Om \text{ for some } \ph \in \cL^0
    \} \subseteq \Ga.
  \]
  Since $\bigcup_n h^{-1} (\ph_n)_\Om = h^{-1} (\bigvee_n \ph_n)_\Om$ and
  $\bot_\Om = \emp$, we have that $\Ga'$ is a $\si$-algebra. Let $r \in \bR$.
  Since $X_i(x) \le r$ if and only if $\om^x \tDash \ul X_i \le \ul r$, it
  follows that $\{X_i \le r\} = h^{-1} (\ul X_i \le \ul r)_\Om$. Thus, $\{X_i
  \le r\} \in \Ga'$, so that $X_i$ is $\Ga'$-measurable for all $i \in I$. Since
  $\Ga$ is the smallest $\si$-algebra with this property, we have $\Ga \subseteq
  \Ga'$. Hence, $\Ga = \Ga'$, so (ii) holds.

  If $U \in \Ga$, $\ph \in \cL^0$, and $U = h^{-1} \ph_\Om$, then by the
  construction of $\sP$, we have $\ph_\Om \in \Si$. Therefore, (ii) implies 
  (iii).

  Finally, since $h$ also induces an isomorphism from $(S, \ol \Ga, \ol \nu)$ to
  $(\Om, \ol \Si, \olbbP)$, we have $\olbbP = \ol \nu \circ h^{-1}$. This gives
  \[
    \ts{
      P(\bigwedge_{k = 1}^n \ul X_{i(k)} \le \ul{r_k} \mid T_\bR) =
      \olbbP \bigcap_{k = 1}^n (\ul X_{i(k)} \le \ul{r_k})_\Om =
      \nu \bigcap_{k = 1}^n \{X_{i(k)} \le r_k\},
    }
  \]
  which verifies \eqref{prob-model-iso-TR}.
\end{proof}

\subsection{Embedding random variables in \texorpdfstring{$\ZFC_-$}{ZFC-}}

Let $L$ be an extralogical signature with $L_- \subseteq L$. If $P \subseteq
\cL^\IS$ is an inductive theory with root $T_0 \supseteq \ZFC_-$, then $P$ is
called a \emph{real inductive theory in $\ZFC_-$}.
  \index{theory!real inductive ---}%
Note that, by definition, the root of an inductive theory is a consistent
deductive theory. Hence, the existence of a real inductive theory in $\ZFC_-$
presupposes the consistency of $\ZFC_-$.

Let $(S, \Ga, \nu)$ be a probability space and let $X = \ang{X_i \mid i \in I}$
be an indexed collection of real-valued random variables, with $\Ga = \si(
\ang{X_i \mid i \in I})$. Let $C = \{\ul X_i \mid i \in I\}$ be a set of
distinct constant symbols not in $L_-$, and define $L = L_- C$.

\begin{thm}\label{T:prob-model-iso-ZFC-}
  Assume $\ZFC_-$ is consistent. Then there exists a complete inductive theory
  $P \subseteq \cL^\IS$ with root $\ZFC_-$ such that
  \begin{equation}\label{prob-model-iso-ZFC-}
    \ts{
      P(\bigwedge_{k = 1}^n \ul X_{i(k)} \le \ul{q_k} \mid \ZFC_-) =
      \nu \bigcap_{k = 1}^n \{X_{i(k)} \le q_k\},
    }
  \end{equation}
  whenever $i(1), \ldots, i(n) \in I$ and $q_k \in \bQ$.
\end{thm}

\begin{proof}
  We will call a set $J \subseteq \bR$ a ``rational interval'' if it has one of
  the following five forms, for some $a, b \in \bQ$: $J = \emp$, $J = (a, b]$,
  $J = (-\infty, b]$, $J = (a, \infty]$, or $J = \bR$. In $\cL$, we adopt the
  following shorthand for every $a, b \in \bQ$ with $a < b$:
  \begin{enumerate}[(i)]
    \item $x \bin \ul{(a, b]} \tot \ul a < x \wedge x < \ul b \vee x \beq \ul b$,
    \item $x \bin \ul{(-\infty, b]} \tot x < \ul b \vee x \beq \ul b$, and
    \item $x \bin \ul{(a, \infty)} \tot \ul a < x$.
  \end{enumerate}
  In this way, if we adopt the shorthand $\ul \emp = \bemp$, then we may write
  $x \bin \ul J$ for every rational interval $J$.

  A ``rational cylinder'' is a set $V \subseteq \bR^I$ of the form
  \[
    V = \{y \in \bR^I \mid y_{i(1)} \in J_1, \ldots, y_{i(n)} \in J_n\},
  \]
  where $n \in \bN$, $i(1), \ldots, i(n) \in I$, and each $J_k$ is a rational
  interval. If $V$ is a rational cylinder, then we define the sentence $\ph^V
  \in \cL^0$ by
  \[
    \ts{
      \ph^V = \bigwedge_{k = 1}^n \ul X_{i(k)} \bin \ul{J_k}.
    }
  \]
  By Proposition \ref{P:models-of-ZFC-}, since $\ZFC_-$ is consistent, we may
  choose an $L_-$ structure $\om_0$ such that $\om_0 \tDash \La^\ZFC_-$. For
  each $y \in \bR^I$, let $\om = \om^y$ be the $L$-expansion of $\om_0$ given by
  $\om^{\ul X_i} = y_i$, and let $\Om = \{\om^y \mid y \in \bR^I\}$. Let
  \[
    \cE = \{\ph^V_\Om \mid \text{$V$ is a rational cylinder}\},
  \]
  and let $\Si_0$ be the set of finite, disjoint unions of sets in $\cE$. Then
  $\Si_0$ is an algebra of sets on $\Om$.

  Now define $\bbP_0: \cE \to [0, 1]$ by $\bbP_0 \ph^V_\Om = \nu \bigcap_{k =
  1}^n \{X_{i(k)} \in J_k\}$, and extend this to $\Si_0$ by finite additivity.
  Then $\bbP_0$ is a pre-measure on $(\Om, \Si_0)$ with $\bbP_0 \Om = 1$. By
  Carath\'eodory's extension theorem (see, for instance, \cite[Theorem
  1.11]{Folland1999}), there exists a unique probability measure $\bbP$ on $
  (\Om, \si(\Si_0))$ that agrees with $\bbP_0$ on $\Si_0$.

  Let $\Si = \si(\Si_0)$ and $\sP = (\Om, \Si, \bbP)$. Since $\om \tDash \ZFC_-$
  for all $\om \in \Om$, we have $\ZFC_- \subseteq \Th \sP$. We may therefore
  define $P = \bTh \sP \dhl_{[\ZFC_-, \Th \sP]}$. By construction, 
  \eqref{prob-model-iso-ZFC-} holds whenever each $q_k \in \bQ$.
\end{proof}

\section{%
  Real inductive theories in \texorpdfstring{$\ZFC$}{ZFC}%
}\label{S:ind-th-ZFC}

In this section, we show how to represent real numbers in $\ZFC$. We then use
this representation to construct real inductive theories in $\ZFC$. In stark
contrast to what happens in $\ZFC_-$, here we find that we can explicitly define
every real number, every Borel set, and every measurable function. We will also
construct a frame of reference in which each of these things is almost surely
fixed, and not random. We call this the ``real frame of reference,'' and it is
presented in three separate parts as Theorems \ref{T:real-FOR},
\ref{T:real-FOR-2}, and \ref{T:real-FOR-3}. Finally in Theorem
\ref{T:prob-model-iso-ZFC}, we show how real-valued random variables can be
naturally and directly embedding into inductive models that are based on
$\ZFC$.

\subsection{Real numbers and Borel sets}

We now work in $\ZFC$. We make all the same definitorial extensions that we did
in $\ZFC_-$, bringing us to the extralogical signature,
\[
  L = {
    \{\bin, \bemp, \S, \ul{\bN_0}, \ul \bZ, \ul \bQ, \ul \bR, +, \bdot, <\} \cup
    \{\ul q \mid q \in \bQ\}
  }.
\]
This time, however, \eqref{nonstd-N-2} holds, along with the analogous
derivability relations for $\ul \bZ$ and $\ul \bQ$. In fact, for any $B
\subseteq \bN_0$, we can use $\AS$ to construct the set $\{x \in \ul{\bN_0}
\mid \bigvee_{n \in B} x \beq \ul n\}$. In other words, we can explicitly define
every set in $\fP \bN_0$, and for each of them, the analogue of 
\eqref{nonstd-N-2} holds. The same is true for every set of integers and every
set of rationals.

Since $\ul \bR$ is defined as a set of Dedekind cuts, and a Dedekind cut is a
set of rationals, it follows that we can explicitly define each individual real
number. We may therefore add, for each $r \in \bR$, an explicitly defined
constant symbol $\ul r$, using the infinitary definition,
\[
  \ts{
    \de_r(y) = \forall x (x \bin y \tot \bigvee_{q \in B(r)} x \beq \ul q),
  }
\]
where $B(r) = \{q \in \bQ \mid q < r\}$.

We now add to $\ZFC$ an explicit, finitary definition of $\ul \cB$, which
denotes the Borel $\si$-algebra, $\cB(\bR)$, on $\bR$. Unlike with $\bN_0$,
$\bZ$, and $\bQ$, we do not expect to be able to explicitly define each
individual subset of $\bR$, since we cannot form an uncountable disjunction in
the language $\cL$. We can, however, explicitly define each Borel set.

Let $\cE \subseteq \fP \bR$ be given by $\cE = \{(-\infty, r] \mid r \in \bR\}$.
Note that $\cB(\bR) = \si(\cE)$. Recall the recursive construction of $\cB(\bR)$
from $\cE$ given in Section \ref{S:gen-sig-alg}, and recall that if $V \in \cB
(\bR)$, then $\rk V$ denotes the rank of $V$ with respect to $\cE$. For each $V
\in \cB(\bR)$, we explicitly define $\ul V$ by recursion on $\rk V$.

If $\rk V = 0$, then $V \in \cE$, so that $V = (-\infty, r]$. Define $\ul V =
\{x \bin \ul \bR \mid x \le \ul r\}$, where we adopt the shorthand $x \le y$ for
$x < y \vee x \beq y$. If $\rk V = \al$, then $\al$ is a successor ordinal, and
we may write $\al = \be + 1$ for some $\be$. If $V \in \cE_\be'$, then we may
choose $W \in \cE_\be$ such that $V = W^c$. We then define $\ul V = \{x \bin \ul
\bR \mid x \notin \ul W\}$. If $V \notin \cE_\be'$, then we may choose a
nonempty and countable $\cD \subseteq \cE_\be'$ such that $V = \bigcap \cD$. We
then define $\ul V = \{x \bin \ul \bR \mid \bigwedge_{W \in \cD} x \bin \ul
W\}$. This defines $\ul V$ for every $V \in \cB(\bR)$.

We now have the extralogical signature,
\begin{equation}\label{ZFC-signature}
  L_\ZFC = {
    \{
      \bin, \bemp, \S,
      \ul{\bN_0}, \ul \bZ, \ul \bQ, \ul \bR, \ul \cB,
      +, \bdot, <
    \} \cup
    \{\ul r \mid r \in \bR\} \cup
    \{\ul V \mid V \in \cB(\bR)\}
  }.
\end{equation}
  \symindex{$L_\ZFC$}%
As usual, we have omitted the duplicates, and each symbol in $L_\ZFC$, except
for $\bin$, is an explicitly defined constant symbol.

\subsection{The real frame of reference}

Let $L$ be an extralogical signature that contains a binary relation symbol
$\bin$. Let $\om = (A, L^\om)$ be an $L$-structure. For a given $b \in A$, we
define ${}^\om b = \{a \in A \mid a \bin^\om b\}$. If $t$ is a ground term, then
we write ${}^\om t = {}^\om \! (t^\om)$.
  \symindex{${}^\om b$, ${}^\om t$}%

\begin{thm}[Real frame of reference I]\label{T:real-FOR}
    \index{frame of reference!real ---}%
  Let $L$ be an extralogical signature such that $L_\ZFC \subseteq L$. Let $\sQ$
  be an $L$-model such that $\sQ \vDash \ZFC$. Then there exists an $L$-model
  $\sP = (\Om, \Si, \bbP)$ such that $\sQ \simeq \sP$ and
  \begin{enumerate}[(i)]
    \item $\om \tDash \ZFC_-$ for every $\om \in \Om$,
    \item $\ul q^\om = q$ for every $\om \in \Om$,
    \item ${}^\om \ul \bR \subseteq \bR$ for every $\om \in \Om$,
    \item $\ul r^\om = r$ a.s., for each $r \in \bR$, and
    \item $\ul V^\om = {}^\om \ul V = V \cap {}^\om \ul \bR$ a.s., for each $V
          \in \cB (\bR)$.
  \end{enumerate}
\end{thm}

\begin{proof}
  Let $\sQ = (\Om', \Si', \bbQ)$ be an $L$-model such that $\sQ \vDash \ZFC$. Let
  \[
    L_\ZFC' = {
    \{
      \bemp, \S, \ul{\bN_0}, \ul \bZ, \ul \bQ, \ul \bR, \ul \cB, +, \bdot, <
    \} \cup
    \{\ul q \mid q \in \bQ\}
  },
  \]
  so that $L_\ZFC'$ is countable and every $\s \in L_\ZFC'$ is an explicitly
  defined constant symbol with defining formula $\de_\s(y)$. Let
  \[
    \ts{
      \ph = \forall y (y \bin \ul \bQ \tot \bigvee_{q \in \bQ} y \beq \ul q),
    }
  \]
  and note that $\ZFC \vdash \ph$. Let
  \[
    X_0 = \La^\ZFC_- \cup \{\ph\} \cup \{
      \exists! y \, \de_\s(y) \mid \s \in L_\ZFC'
    \},
  \]
  so that $\sQ \vDash X_0$. Since $X_0$ is countable, it follows that $\nu
  \tDash X_0$ for $\bbQ$-a.e.~$\nu \in \Om'$. By Remark \ref{R:a.s.-sure}, we
  may assume that $\nu \tDash X_0$ for all $\nu \in \Om'$.

  Fix $\nu = (A_\nu, L^\nu) \in \Om'$. Let $a \in {}^\nu \ul \bR$ and define $B
  = \{q \in \bQ \mid \nu \tDash (\ul q < x)[a]\}$. Since $\nu \tDash (\forall x
  \bin \ul \bR)(\exists y \bin \ul \bQ)(y < x)$ and $\nu \tDash \ph$, it follows
  that $B \ne \emp$. Similarly, $B \ne \bQ$, $B$ is downward closed, and $B$ has
  no maximum element. That is, $B$ is a Dedekind cut. Therefore, there exists a
  unique $r \in \bR$ such that $B = B(r) = \{q \in \bQ \mid q < r\}$. Let
  $g_\nu: {}^\nu \ul \bR \to \bR$ denote the function $a \mapsto r$.

  Let $a, a' \in {}^\nu \ul \bR$ and define $B$ and $B'$ accordingly. Assume $a
  \ne a'$. Since
  \[
    \nu \tDash (\forall x y \bin \ul \bR) (x \nbeq y \to x < y \vee y < x),
  \]
  we have $a <^\nu a'$ or $a' <^\nu a$. Without loss of generality, assume it is
  the former. Since
  \[
    \nu \tDash (\forall x y \bin \ul \bR) (
      x < y \to (\exists z \bin \ul \bQ) (x < z \wedge z < y)
    ),
  \]
  and $\nu \tDash \ph$, there exists $q \in \bQ$ such that $a <^\nu \ul q^\nu$
  and $\ul q^\nu <^\nu a'$. We then have $q \in B' \setminus B$, so that $B \ne
  B'$, which implies $g_\nu a \ne g_\nu a'$. Therefore, $g_\nu$ is injective. In
  fact, $g_\nu$ is order-preserving, in the sense that $a <^\nu a'$ if and only
  if $g_\nu a < g_\nu a'$.

  Now let $b \in {}^\nu \! (\fP \ul \bR)$ and $a \in {}^\nu b$. Then $\nu
  \tDash (x \bin y \wedge y \bin \fP \ul \bR)[a, b]$. Hence, we have $\nu \tDash
  (x \bin \ul \bR)[a]$, so that $a \in {}^\nu \ul \bR$. This shows that ${}^\nu
  b \subseteq {}^\nu \ul \bR$, so that $g_\nu {}^\nu b \subseteq \bR$. Since $
  {}^\nu \ul \bR \cap {}^\nu \! (\fP \ul \bR) = \emp$, we may extend $g_\nu$ to
  be a function $g_\nu: {}^\nu \ul \bR \cup {}^\nu \! (\fP \ul \bR) \to \bR \cup
  \fP \bR$ by setting $g_\nu b = g_\nu {}^\nu b$ for all $b \in {}^\nu \! (\fP
  \ul \bR)$. As above, it follows that $g_\nu$ remains injective.

  We now extend $g_\nu$ to be a bijection from $A_\nu$ to some set $A_\om$,
  which is a superset of $g_\nu ({}^\nu \ul \bR \cup {}^\nu \! (\fP \ul \bR))$.
  Let $\om$ be the isomorphic image of $\nu$ under $g_\nu$. Note that since
  $g_\nu$ is order-preserving on ${}^\nu \ul \bR$, we have $a <^\om a'$ if and
  only if $a < a'$, whenever $a, a' \in A_\om \cap \bR$.

  Let $h$ denote the function $\nu \mapsto \om$ and let $\Om = h \, \Om'$. Let
  $\sP = (\Om, \Si, \bbP)$ be the measure space image of $\sQ$ under $h$. Then
  $h$ induces a measure-space isomorphism from $\sQ$ to $\sP$, and $\nu \simeq h
  \nu$ for all $\nu \in \Om'$. Hence, $h$ is a model isomorphism, and $\sQ
  \simeq \sP$. Since $\nu \tDash X_0$ for all $\nu \in \Om'$ and $\om \simeq
  \nu$, we have (i). Moreover, $a \in {}^\om \ul \bR$ if and only if $g_\nu^{-1}
  a \in {}^\nu \ul \bR$. Hence, ${}^\om \ul \bR = g_\nu {}^\nu \ul \bR \subseteq
  \bR$, so (iii) holds.

  Now let $r \in \bR$. Since $\sP \vDash \exists! y \, \de_r(y)$, we may choose
  $\Om^* \in \Si$ such that $\bbP \Om^* = 1$ and $\om \tDash \exists! y \, \de_r
  (y)$ for all $\om \in \Om^*$. Fix $\om \in \Om^*$. Then $\om \tDash \ul r
  \bin \ul \bR$. Hence, $\ul r^\om \in {}^\om \ul \bR = g_\nu {}^\nu \ul \bR$,
  so that
  \begin{align*}
    \ul r^\om = g_\nu g_\nu^{-1} \ul r^\om
    &= \sup \{q \in \bQ \mid \nu \tDash (\ul q < x)[g_\nu^{-1} \ul r^\om]\}\\
    &= \sup \{q \in \bQ \mid \om \tDash (\ul q < x)[\ul r^\om]\}\\
    &= \sup \{q \in \bQ \mid \om \tDash \ul q < \ul r\}.
  \end{align*}
  But by the definition of $\ul r$, we have
  \[
    \ts{
      \om \tDash (\forall x \bin \ul \bQ) (
        x < \ul r \tot \bigvee_{q \in B(r)} x \beq \ul q
      ),
    }
  \]
  where $B(r) = \{q \in \bQ \mid q < r\}$. Thus, $\om \tDash \ul q < \ul r$ if
  and only if $q \in B(r)$. Therefore, $\ul r^\om = \sup B(r) = r$, proving (iv).

  Now fix $V \in \cB(\bR)$. We will prove (v) by induction on $\rk V$. Suppose
  $\rk V = 0$. Then $V = (-\infty, r]$. Choose $\Om^* \in \Si$ such that $\bbP
  \Om^* = 1$ and, for all $\om \in \Om^*$, we have $\ul r^\om = r$ and $\om
  \tDash \exists! y \, \de_V(y)$, where $\de_V(y)$ is the defining formula for
  $\ul V$. Since $\om \tDash \ul V \bin \fP \ul \bR$, we have $\nu \tDash \ul V
  \bin \fP \ul \bR$, so that $\ul V^\nu \in {}^\nu \! (\fP \ul \bR)$. Therefore,
  $\ul V^\om = g_\nu \ul V^\nu = g_\nu {}^\nu \ul V$. On the other hand,
  \begin{align*}
    a \in {}^\om \ul V &\quad\text{iff}\quad {
      a \in A_\om \text{ and } \om \tDash (x \bin \ul V)[a]
    }\\
    &\quad\text{iff}\quad {
      g_\nu^{-1} a \in A_\nu
      \text{ and } \nu \tDash (x \bin \ul V)[g_\nu^{-1} a]
    }\\
    &\quad\text{iff}\quad g_\nu^{-1} a \in {}^\nu \ul V.
  \end{align*}
  Thus, ${}^\om \ul V = g_\nu {}^\nu \ul V = \ul V^\om$. Now, $\ul V^\om = \{x
  \bin \ul \bR \mid x \le \ul r\}^\om$. Hence,
  \[
    {}^\om \ul V
    = \{a \in {}^\om \ul \bR \mid \om \tDash (x \le \ul r)[a]\}
    = \{a \in {}^\om \ul \bR \mid a \le^\om \ul r^\om\}.
  \]
  But ${}^\om \ul \bR \subseteq \bR$ and $\ul r^\om = r \in \bR$. Therefore, $a
  \le^\om \ul r^\om$ if and only if $a \le r$, which implies $\ul V^\om = {}^\om
  \ul V = V \cap {}^\om \ul \bR$.

  Now suppose $\rk V = \al$. Since the rank of a Borel set is always a successor
  ordinal, we may write $\al = \be + 1$. Assume $V \in \cE_\be'$. Then $\ul V =
  \{x \bin \ul \bR \mid x \nbin \ul W\}$, where $W \in \cE_\be$ and $V = \bR
  \setminus W$. As above, choose $\Om^* \in \Si$ such that $\bbP \Om^* = 1$ and,
  for all $\om \in \Om^*$, we have $\om \tDash \exists! y \, \de_W(y)$ and $\om
  \tDash \exists! y \, \de_V(y)$. Then
  \[
    \ul V^\om = {}^\om \ul V = \{
      a \in {}^\om \ul \bR \mid a \notin {}^\om \ul W
    \}
    = {}^\om \ul \bR \setminus (W \cap {}^\om \ul \bR)
    = V \cap {}^\om \ul \bR,
  \]
  proving (v) in the case that $V \in \cE_\be'$. The proof in the case $V \notin
  \cE_\be'$ is similar.

  Finally, since $\ul q^\om = q$ a.s., for each $q \in \bQ$, and $\bQ$ is
  countable, it follows that $\ul q^\om = q$ for each $q \in \bQ$, a.s. By
  Remark \ref{R:a.s.-sure}, we may take $\sP$ to be such that (ii) holds.
\end{proof}

\begin{cor}\label{C:real-FOR}
  Let $r \in \bR$ and $V, V', V_n \in \cB(\bR)$.
  \begin{enumerate}[(i)]
    \item If $r \in V$, then $\ZFC \vdash \ul r \bin \ul V$.
    \item If $r \notin V$, then $\ZFC \vdash \ul r \nbin \ul V$.
    \item If $V \subseteq V'$, then $\ZFC \vdash \ul V \bsub V'$.
    \item If $V = \bigcap_{n \in \bN_0} V_n$, then $\ZFC \vdash \forall x
          (x \bin \ul V \tot \bigvee_{n \in \bN_0} x \bin \ul{V_n})$.
  \end{enumerate}
\end{cor}

\begin{proof}
  Let $r \in V$. Suppose $\sP \vDash \ZFC$. Using the real frame of reference,
  we may assume $\sP$ satisfies (i)--(v) in Theorem \ref{T:real-FOR}. Hence, for
  a.e.~$\om$, we have $\ul r^\om = r$. Since $\ul r^\om \in {}^\om \ul \bR$,
  this gives $\ul r^\om \in V \cap {}^\om \ul \bR = {}^\om \ul V$, so that $\om
  \tDash (\ul r \bin \ul V)$. Since this holds almost surely, we have $\sP
  \vDash \ul r \bin \ul V$. The proofs of (ii)--(iv) are similar.
\end{proof}

\subsection{Sequences and limits}

Let us define, in $\ZFC$, subtraction of real numbers and the absolute value
function, and adopt the usual shorthand for these functions. Let
\begin{multline*}
  \ph_\lm(u, v) = u \bin \ul \bR^{\ul \bN_0} \wedge v \bin \ul \bR\\
  \wedge {
    (\forall z \bin \ul{(0, \infty)})
    (\exists y \bin \ul \bN_0)
    (\forall x \bin \ul \bN_0)
    (x \ge y \to |u(x) - v| < z)
  }.
\end{multline*}
  \symindex{$\ph_\lm$}%
Then $\ph_\lm(u, v)$ says that $u$ is a sequence of real numbers that converges
to the real number $v$.

\begin{thm}[Real frame of reference II]\label{T:real-FOR-2}
    \index{frame of reference!real ---}%
  The model $\sP = (\Om, \Si, \bbP)$ in Theorem \ref{T:real-FOR} may be chosen
  so that for every $\om \in \Om$, we have the following:
  \begin{enumerate}[(i)]
    \item for each $a \in {}^\om \! (\ul \bR^{\ul \bN_0})$ and each $n \in
          \bN_0$, there exists a unique $a_n \in {}^\om \ul \bR \subseteq \bR$
          such that $\om \tDash (y \beq x(\ul n))[a, a_n]$, and
    \item $\om \tDash \ph_\lm[a, b]$ if and only if $b \in {}^\om \ul \bR$ and
          $a_n \to b$.
  \end{enumerate}
\end{thm}

\begin{proof}
  Note that $\ZFC \vdash \psi$, where
  \[
    \psi = \ts{
      x \bin \ul \bR^{\ul \bN_0}
      \to \bigwedge_{n \in \bN_0} (\exists! y \bin \ul \bR) y \beq x(\ul n)
    }.
  \]
  Hence, $\om \tDash \psi$ a.s. By Remark \ref{R:a.s.-sure}, we may assume $\om
  \tDash \psi$ for all $\om \in \Om$. Fix $\om \in \Om$ and suppose $a \in
  {}^\om \! (\ul \bR^{\ul \bN_0})$. Then $\om \tDash (x \bin \ul \bR^{\ul
  \bN_0})[a]$. Hence, since $\om \tDash \psi$, it follows that for each $n \in
  \bN_0$, there exists a unique $a_n \in {}^\om \ul \bR \subseteq \bR$ such that
  $\om \tDash (y \beq x(\ul n))[a, a_n]$. Note that in the definition of
  $\ph_\lm$, it suffices to consider rational $z$. Since $\ul q^\om = q$ for all
  $q \in \bQ$, it follows that $\om \tDash \ph_\lm[a, b]$ if and only if $b \in
  {}^\om \ul \bR$ and $a_n \to b$.
\end{proof}

\begin{rmk}
  Let $\sP = (\Om, \Si, \bbP)$ be a model that satisfies Theorem
  \ref{T:real-FOR-2}. In $\cL$, we can formulate a sentence $\ph$ which asserts
  that $+$ is continuous. For such a sentence, $\ZFC_- \vdash \ph$. Recall that
  $+$ was defined so that it agrees with the usual addition of rational numbers.
  That is, $\ZFC_- \vdash \ul{q_1} + \ul{q_2} \beq \ul{q_3}$ if and only if $q_1
  + q_2 = q_3$, whenever $q_i \in \bQ$. Using this, it can be shown that, for
  all $\om \in \Om$ and all $a, b, c \in {}^\om \ul \bR \subseteq \bR$, we have
  $\om \tDash (x + y \beq z) [a, b, c]$ if and only if $a + b = c$. A similar
  thing holds for $\bdot$ and $<$.
\end{rmk}

\subsection{Measurable functions}

Let $h: \bR \to \bR$ be measurable. Suppose $h$ is a simple function. That is,
the range of $h$ is a finite set $\{r_1, \ldots, r_n\}$. Let $V_j = h^{-1} 
\{r_j\}$. Then we may explicitly define $h$ in $\ZFC$ by
\[
  \ul h = \{
    (x, y) \bin \bR \times \bR
  \mid
    \ts{\bigvee_{i = 1}^n x \bin \ul{V_i} \wedge y \beq \ul{r_i}}
  \}.
\]
Suppose $h$ is not a simple function. Choose a sequence $\ang{h_n \mid n \in
\bN_0}$ of simple functions such that $h_n \to h$ pointwise. In $\ZFC$, we
define
\[
  \ul{h_\cdot} = \{
    (x, y) \bin \ul \bR \times \ul \bR^{\ul \bN_0}
  \mid
    \ts{\bigwedge_{n \bin \bN_0} y(\ul n) \beq \ul{h_n}(x)}
  \}.
\]
Then $\ul{h_\cdot}$ is an explicit definition of the function $x \mapsto 
\ang{h_n(x) \mid n \in \bN_0}$. We now define
\[
  \ul h = \{
    (x, y) \bin \ul \bR \times \ul \bR \mid \ph_\lm(\ul{h_\cdot}(x), y)
  \}.
\]

\begin{thm}[Real frame of reference III]\label{T:real-FOR-3}
    \index{frame of reference!real ---}%
  Let $h: \bR \to \bR$ be a measurable function and let $\sP = (\Om, \Si, \bbP)$
  be a model chosen according to Theorem \ref{T:real-FOR-2}. Then, for
  $\bbP$-a.e.~$\om \in \Om$, we have $\om \tDash (y \beq \ul h(x))[a, b]$ if and
  only if $a \in {}^\om \ul{\bR}$, $b \in {}^\om \ul{\bR}$, and $h(a) = b$.
\end{thm}

\begin{proof}
  First suppose $h$ is a simple function. Then $\om \tDash (y \beq \ul h(x))[a,
  b]$ if and only if there exists $i \in \{1, \ldots, n\}$ such that $a \in 
  {}^\om \ul{V_i}$ and $b = \ul{r_i}^\om$. By Theorem \ref{T:real-FOR}, we may
  choose $\Om^* \in \Si$ such that $\bbP \Om^* = 1$ and, for all $\om \in
  \Om^*$, we have ${}^\om \ul \bR \subseteq \bR$, $\ul{r_i}^\om = r_i$, and $
  {}^\om \ul{V_i} = V_i \cap {}^\om \ul{\bR}$. Hence, the conclusion of the
  theorem holds for each $\om \in \Om^*$.

  Now suppose $h$ is not a simple function. Let $\ang{h_n \mid n \in \bN_0}$ be
  the sequence of simple functions used to define $\ul h$. Choose $\Om^* \in
  \Si$ such that $\bbP \Om^* = 1$ and $\om \tDash \exists! y \, \de_\s(y)$ for
  every $\s \in \{\ul{h_n} \mid n \in \bN_0\} \cup \{\ul{h_\cdot}, \ul h\}$.
  Suppose $\om \tDash (y \beq \ul h(x))[a, b]$. Then $\om \tDash \ph_\lm(
  \ul{h_\cdot}(x), y)[a, b]$, which means $\om \tDash \ph_\lm[c, b]$, where
  $\om \tDash (u \beq \ul{h_\cdot}(x))[a, c]$. By Theorem \ref{T:real-FOR-2}(i),
  for each $n \in \bN_0$, there exists a unique $c_n \in {}^\om \ul \bR
  \subseteq \bR$ such that $\om \tDash (v \beq u(\ul n))[c, c_n]$. Since $\om
  \tDash \ph_\lm[c, b]$, Theorem \ref{T:real-FOR-2}(ii) implies $b \in {}^\om
  \ul \bR$ and $c_n \to b$. On the other hand, since $\om \tDash (u \beq 
  \ul{h_\cdot}(x))[a, c]$, it follows from the definition of $\ul{h_\cdot}$ that
  $\om \tDash (u(\ul n) \beq \ul{h_n}(x))[a, c]$. Therefore, $\om \tDash (v =
  \ul{h_n}(x))[a, c_n]$. Since the theorem holds for simple functions, we have
  $c_n = h_n(a)$. Thus, $b = \lim h_n(a) = h(a)$. The proof of the converse is
  similar.
\end{proof}

\subsection{%
  Embedding random variables in \texorpdfstring{$\ZFC$}{ZFC}%
}

Let $L$ be an extralogical signature with $L_\ZFC \subseteq L$. If $P \subseteq
\cL^\IS$ is an inductive theory with root $T_0 \supseteq \ZFC$, then $P$ is
called a \emph{real inductive theory in $\ZFC$}. As with $\ZFC_-$, the existence
of a real inductive theory in $\ZFC$ presupposes the consistency of $\ZFC$.
  \index{theory!real inductive ---}%

Let $(S, \Ga, \nu)$ be a probability space and let $X = \ang{X_i \mid i \in I}$
be an indexed collection of real-valued random variables, with $\Ga = \si(
\ang{X_i \mid i \in I})$. Let $C = \{\ul X_i \mid i \in I\}$ be a set of
distinct constant symbols not in $L_\ZFC$, and define $L = L_\ZFC C$.

Theorem \ref{T:prob-model-iso-ZFC} below shows that, if $\ZFC$ is strictly
satisfiable, then we recover the full analogue of Theorem
\ref{T:prob-model-iso}, including the natural correspondence between outcomes
and structures. Recall from Theorem \ref{T:Con-ZFC} that the existence of a
strongly inaccessible cardinal implies not only the consistency of $\ZFC_-$, but
also the strict satisfiability of $\ZFC$. Hence, the assumption in Theorem
\ref{T:prob-model-iso-ZFC} is not too far beyond the typically uncontroversial
assumption that $\ZFC_-$ is consistent.

The proof of Theorem \ref{T:prob-model-iso-ZFC} shows how the inductive model
$\sP = (\Om, \Si, \bbP)$ is built from the measure-theoretic probability model
$(S, \Ga, \nu, X)$. Each structure $\om \in \Om$ is an expansion of a single
structure $\om_0$ that strictly satisfies $\ZFC$. This implies, for example,
that $\ul r^\om$ does not depend on $\om$. The same is true for every other
symbol in $L_\ZFC$. In other words, all the familiar objects of set theory and
the real numbers are all fixed in $\sP$. The only things that vary with $\om$
are $\ul X_i^\om$. This exactly matches our intuition about measure-theoretic
models. In practice, when we work with a measure-theoretic model, we think of
the random variables as being the only things that are ``random.'' The real
numbers, their relations, the concept of set membership, and so on, are all
fixed, and do not change from one outcome to another. In a sense, then, the
practicing probabilist, in using the logic of countable unions and
intersections, and in assuming that all our familiar mathematics is fixed under
different outcomes, is operating under the implicit assumption that $\ZFC$ is
strictly satisfiable.

\begin{thm}\label{T:prob-model-iso-ZFC}
  Assume $\ZFC$ is strictly satisfiable. Then there exists an $\cL$-model $\sP =
  (\Om, \Si, \bbP)$ with $\sP \vDash \ZFC$, and a function $h: S \to \Om$
  mapping $x \in S$ to $\om \in \Om$ such that
  \begin{enumerate}[(i)]
    \item $x \in \{X_i \in V\}$ if and only if $\om \tDash \ul X_i \bin \ul V$
          for all $V \in \cB(\bR)$,
    \item each $U \in \Ga$ can be written as $U = h^{-1} \ph_\Om$ for some $\ph
          \in \cL^0$, and
    \item $h$ induces a measure-space isomorphism from $(S, \Ga, \nu)$ to $\sP$.
  \end{enumerate}
  Consequently, if $P = \bTh \sP \dhl_{[\ZFC, \Th \sP]}$, then $P$ satisfies
  \begin{equation}\label{prob-model-iso-ZFC}
    \ts{
      P(\bigwedge_{k = 1}^n \ul X_{i(k)} \bin \ul{V_k} \mid \ZFC) =
      \nu \bigcap_{k = 1}^n \{X_{i(k)} \in V_k\},
    }
  \end{equation}
  whenever $i(1), \ldots, i(n) \in I$ and $V_k \in \cB(\bR)$.
\end{thm}

\begin{proof}
  Assume there exists an $L_\ZFC$-structure $\om_0$ such that $\om_0 \tDash
  \ZFC$. For each $x \in S$, define $\om = \om^x$ to be the $L$-expansion of
  $\om_0$ given by $\ul X_i^\om = \ul r^{\om_0}$, where $r = X_i(x)$. Let $\Om =
  \{\om^x \mid x \in S\}$ and let $h: S \to \Om$ denote the map $x \mapsto
  \om^x$. Let $\sP = (\Om, \Si, \bbP)$ be the measure space image of $(S, \Ga,
  \nu)$ under $h$. Since $\om \tDash \ZFC$ for all $\om \in \Om$, we have $\sP
  \vDash \ZFC$.

  Let $x \in S$ and $V \in \cB(\bR)$, and set $r = X_i(x)$. By the definition of
  $\om^x$, we have $\om^x \tDash \ul X_i \bin \ul V$ if and only if $\om_0
  \tDash \ul r \bin \ul V$. By Corollary \ref{C:real-FOR}, we have $\ZFC \vdash
  \ul r \bin \ul V$ if $r \in V$, and $\ZFC \vdash \ul r \nbin \ul V$ if $r
  \notin V$. Hence, since $\om_0 \tDash \ZFC$, we have $\om_0 \tDash \ul r \bin
  \ul V$ if and only if $r \in V$. But $r = X_i(x)$. Hence, $x \in \{X_i \in
  V\}$ if and only if $\om^x \tDash \ul X_i \bin \ul V$, proving (i). It follows
  as in the proof of Theorem \ref{T:prob-model-iso-TR} that (ii) and (iii) hold,
  and \eqref{prob-model-iso-ZFC} holds for $P = \bTh \sP \dhl_{[\ZFC, \Th
  \sP]}$.
\end{proof}

\section{Limit theorems}\label{S:limit-thms}

In this section, we introduce the notion of Borel terms, and use them to
formulate, in inductive logic, both the law of large numbers and the central
limit theorem.

\subsection{Borel terms}

Let $P$ be a real inductive theory in $\ZFC$. That is, $P \subseteq \cL^\IS$,
where $L_\ZFC \subseteq L$, and $P$ has root $T_0 \supseteq \ZFC$. In
particular, we are assuming $\ZFC$ is consistent.

Let $X \in \ante P$. A ground term $t \in \cT$ is \emph{real given $X$} if $P(t
\bin \ul \bR \mid X) = 1$. We say that $t$ is \emph{Borel given $X$} if it is
real given $X$ and $P(t \bin \ul V \mid X)$ exists for all $V \in \cB(\bR)$. If
$t$ is Borel given $X$, then the \emph{distribution of $t$ given $X$} is the
function $\mu = \mu_{t \mid X}$ from $\cB(\bR)$ to $[0, 1]$ given by $\opmu V =
P (t \bin \ul V \mid X)$.
  \index{term!real ---}%
  \index{term!Borel ---}%
  \index{distribution}%
  \symindex{$\mu_{t \mid X}$}%

\begin{prop}\label{P:term-dist}
  If $t$ is Borel given $X$, then the distribution of $t$ is a probability
  measure on $(\bR, \cB(\bR))$.
\end{prop}

\begin{proof}
  Let $\sQ = (\Om, \Ga, \bbQ) \vDash P$. We may write $X \equiv Y \cup
  \{\psi\}$, where $\sQ \vDash Y$ and $P(\ph \mid X) = \olbbQ \ph_\Om \cap
  \psi_\Om / \olbbQ \psi_\Om$ whenever $P(\ph \mid X)$ is defined. Let $\Si =
  \ol \Ga$ and let $\bbP$ be the probability measure on $(\Om, \Si)$ given by
  $\bbP A = \olbbQ A \cap \psi_\Om / \olbbQ \psi_\Om$. Then $P(\ph \mid X) =
  \bbP \ph_\Om$ whenever $P(\ph \mid X)$ is defined. Since $\sP = (\Om, \Si,
  \bbP) \vDash \ZFC$, we may use the real frame of reference, and assume $\sP$
  satisfies (i)--(v) in Theorem \ref{T:real-FOR}.

  Thus, ${}^\om \ul \emp = \emp$ a.s. Hence, for a.e.~$\om$, we have $t^\om
  \notin {}^\om \ul \emp$, meaning $\om \ntDash t \bin \ul \emp$. This gives $(t
  \bin \ul \emp)_\Om = \emp$ a.s., so that $\opmu \emp = P(t \bin \ul \emp \mid
  X) = 0$. Similarly, since $P(t \bin \ul \bR \mid X) = 1$, we have $\om \tDash
  t \bin \ul \bR$ a.s., and it follows that $\opmu \bR = 1$.

  Let $\{V_n\} \subseteq \cB(\bR)$ be pairwise disjoint, and set $V = \bigcup_n
  V_n$. Define the sentences $\ph_n = t \bin \ul{V_n}$. Let $i \ne j$. As in the
  proof of Corollary \ref{C:real-FOR}, we have $\ZFC \vdash \forall x \neg (x
  \bin \ul{V_i} \wedge x \bin \ul{V_j})$. Hence, $\ZFC \vdash \neg (\ph_i \wedge
  \ph_j)$, which gives $P(\ph_i \wedge \ph_j \mid X) = 0$. Theorem
  \ref{T:ctbl-add} therefore implies $P(\bigvee_n \ph_n \mid X) = \sum_n P(\ph_n
  \mid X) = \sum_n \opmu V_n$. On the other hand, as above, $\ZFC \vdash t \bin
  \ul V \tot \bigvee_n \ph_n$, so that Proposition \ref{P:log-equiv-gen} gives
  $P(\bigvee_n \ph_n \mid X) = P(t \bin V \mid X) = \opmu V$.
\end{proof}

Let $t$ be Borel given $X$. We define the \emph{expected value of $t$ given $X$}
by $E[t \mid X] = \int_\bR x \, \mu_{t \mid X}(dx)$, provided this integral
exists. Note that $E[t \mid X] \in [-\infty, \infty]$. If $\int_\bR |x| \,
\mu_{t \mid X}(dx) < \infty$, then we say $t$ is \emph{integrable given $X$}. In
this case, $E[t \mid X] \in \bR$.
  \index{expected value}%
  \index{term!integrable ---}%
  \symindex{$E[t \mid X]$}%

\subsection{Jointly Borel terms}

A finite sequence of terms, $t_1, \ldots, t_n$, are \emph{jointly Borel given
$X$} if each $t_i$ is Borel given $X$ and $P(\bigwedge_{i = 1}^n t_i \bin
\ul{V_i} \mid X)$ exists, whenever $V_i \in \cB(\bR)$. If $t_1, \ldots, t_n$ are
jointly Borel, then each $t_i$ is Borel. The converse holds when $P$ is
complete, by Definition \ref{D:complete}(i). More generally, an indexed
collection of terms, $\bft = \ang{t_i \mid i \in I}$, is \emph{jointly Borel
given $X$} if each finite subsequence is.
  \index{term!jointly Borel ---}%

\begin{prop}\label{P:term-dist-Rn}
  If $\bft = (t_1, \ldots, t_n)$ are jointly Borel given $X$, then there exists
  a unique probability measure $\mu = \mu_{\bft \mid X}$ on $(\bR^n, \cB
  (\bR^n))$ such that
    \symindex{$\mu_{\bft \mid X}$}%
  \[
    \ts{
      \opmu V_1 \times \cdots \times V_n
      = P(\bigwedge_{k = 1}^n t \bin \ul{V_k} \mid X),
    }
  \]
  whenever $i(1), \ldots, i(n) \in I$ and $V_k \in \cB(\bR)$.
\end{prop}

\begin{proof}
  The existence and uniqueness of $\mu$ follows from Carath\'eodory's extension
  theorem, as in the proof of Theorem \ref{T:prob-model-iso-ZFC-}, using as our
  algebra the set of finite disjoint unions of measurable rectangles, $V_1
  \times \cdots \times V_n$.
\end{proof}

The unique Borel probability measure $\mu_{\bft \mid X}$ in Proposition
\ref{P:term-dist-Rn} is called the \emph{distribution of $\bft$ given $X$}.
  \index{distribution}%

\begin{rmk}\label{R:term-dist-Rn}
  Just as we did for $\cB(\bR)$, we can explicitly define $\ul V$ for each $V
  \in \cB(\bR^n)$ and establish a real frame of reference for $\bR^n$. If $\mu$
  is the distribution of $(t_1, \ldots, t_n)$ given $X$, then $\opmu V = P(
  (t_1, \ldots, t_n) \bin \ul V \mid X)$ for all $V \in \cB(\bR^n)$. To see
  this, let $\Ga$ be the set of $V \subseteq \bR^n$ for which it holds.
  Proposition \ref{P:term-dist-Rn} shows that $\Ga$ contains the algebra of
  finite disjoint unions of measurable rectangles, which is a $\pi$-system that
  generates $\cB(\bR^n)$. Since $P$ is an inductive theory, we have that $\Ga$
  is a $\la$-system. Therefore, by Dynkin's $\pi$-$\la$ theorem, $\cB(\bR^n)
  \subseteq \Ga$.
\end{rmk}

\subsection{Independence of terms}
  \index{term!independent ---}

Let $I$ be a set with $|I| \ge 2$ and let $\ang{t_i \mid i \in I}$ be jointly
Borel given $X$. Then $\ang{t_i \mid i \in I}$ are \emph{independent given $X$}
if $\ang{t_i \bin \ul{V_i} \mid i \in I}$ are independent given $X$, whenever
$V_i \in \cB(\bR)$. We say that $\ang{t_i \mid i \in I}$ are \emph{identically
distributed given $X$} if the distribution of $t_i$ given $X$ does not depend on
$i$. As usual, we write i.i.d.~for the phrase, ``independent and identically
distributed.''

\begin{rmk}\label{R:dist-indep}
  If $\bft = (t_1, \ldots, t_n)$ are independent given $X$, then Corollary
  \ref{C:indep-sound-compl} implies that $\mu_{\bft \mid X} = \prod_{k = 1}^n
  \mu_{t_i \mid X}$.
\end{rmk}

\begin{prop}\label{P:term-process}
  Suppose $\ang{t_i \mid i \in I}$ are jointly Borel given $X$. Then there
  exists a real-valued stochastic process, $\ang{Y_i \mid i \in I}$, defined on
  a probability space, $(S, \Ga, \nu)$, such that
  \begin{equation}\label{term-process}
    \ts{
      P(\bigwedge_{k = 1}^n t_{i(k)} \bin \ul{V_k} \mid X) =
      \opnu \bigcap_{k = 1}^n \{Y_{i(k)} \in V_k\},
    }
  \end{equation}
  whenever $i(1), \ldots, i(n) \in I$ and $V_k \in \cB(\bR)$. Moreover, if $
  \ang{t_i \mid i \in I}$ are independent given $X$, then $\ang{Y_i \mid i \in
  I}$ are independent.
\end{prop}

\begin{proof}
  Let $S = \bR^I$ and $\Ga = \bigotimes_{i \in I} \cB(\bR)$ the product
  $\si$-algebra. For each $i \in I$, let $\mu_i$ be the distribution of $t_i$.
  For each $n \in \bN$ and $\bfi = (i(1), \ldots, i(n)) \in I^n$, let $\bft
  (\bfi) = (t_{i(1)}, \ldots, t_{i(n)})$, so that $\mu_{\bft(\bfi) \mid X}$ is a
  probability measure on $(\bR^n, \cB(\bR^n))$. Using the methods in the proof
  of Proposition \ref{P:term-dist}, we have that the measures $\mu_{\bft(\bfi)
  \mid X}$ are consistent, in the sense of Kolmogorov. Hence, by Kolmogorov's
  extension theorem (see, for instance, \cite[Theorem 2.2.2] {Karatzas1991}),
  there exists a probability measure $\nu$ on $(S, \Ga)$ such that
  \begin{equation}\label{term-process-2}
    \opnu \{x \in S \mid (x_{i(1)}, \ldots, x_{i(n)}) \in A\}
    = \opmu_{\bft(\bfi) \mid X} A,
  \end{equation}
  whenever $A \in \cB(\bR^n)$. Define $Y_i: S \to \bR$ by $Y_i(x) = x_i$. Then
  $\ang{Y_i \mid i \in I}$ is a real-valued stochastic process, and
  \eqref{term-process} follows from \eqref{term-process-2}. The final claim
  about independence follows from Corollary \ref{C:indep-sound-compl}, 
  \eqref{term-process-2}, and Remark \ref{R:dist-indep}.
\end{proof}

\subsection{The law of large numbers for terms}

\begin{lemma}\label{L:LLN}
  Suppose $\ang{t_n \mid n \in \bN}$ are jointly Borel given $X$, and let $
  \ang{Y_n \mid n \in \bN}$ be as in Proposition \ref{P:term-process}. Define
  the terms $s_n = (t_1 + \cdots + t_n)/\ul n$ and the random variables $Z_n =
  (Y_1 + \cdots + Y_n)/n$. Then
  \[
    \ts{
      P(\bigwedge_{k = 1}^n s_{i(k)} \bin \ul{V_k} \mid X) =
      \opnu \bigcap_{k = 1}^n \{Z_{i(k)} \in V_k\},
    }
  \]
  whenever $i(1), \ldots, i(n) \in \bN$ and $V_k \in \cB(\bR)$.
\end{lemma}

\begin{proof}
  Let $i(1), \ldots, i(n) \in \bN$ and $V_k \in \cB(\bR)$. For each $n \in \bN$,
  define $f_n: \bR^n \to \bR$ by $f_n(x_1, \ldots, x_n) = (x_1 + \cdots +
  x_n)/n$. Let $m = \max\{i(1), \ldots, i(n)\}$ and $V_k' = f_{i(k)}^{-1} V_k
  \times \bR^{m - i(k)}$. Then $\{Z_{i(k)} \in V_k\} = \{(Y_1, \ldots, Y_m)
  \in V_k'\}$, so that $\bigcap_{k = 1}^n \{Z_{i(k)} \in V_k\} = \{(Y_1, \ldots,
  Y_m) \in V'\}$, where $V' = \bigcap_{k = 1}^n V_k'$. By the definition of
  $\nu$, we have $\opnu \bigcap_{k = 1}^n \{Z_{i(k)} \in V_k\} = \opmu V'$,
  where $\mu$ is the distribution of $(t_1, \ldots, t_m)$ given $X$. Hence, by
  Remark \ref{R:term-dist-Rn},
  \[
    \ts{
      \opnu \bigcap_{k = 1}^n \{Z_{i(k)} \in V_k\}
      = P((t_1, \ldots, t_m) \bin \ul{V'} \mid X).
    }
  \]
  But $\ZFC \vdash \bigwedge_{k = 1}^n s_{i(k)} \bin \ul{V_k} \tot (t_1, \ldots,
  t_m) \bin \ul{V'}$, so the result follows from the rule of logical implication
  and Proposition \ref{P:log-equiv-gen}.
\end{proof}

\begin{thm}[Law of large numbers]\label{T:LLN}
    \index{law of large numbers}%
  Let $P$ be a real inductive theory in $\ZFC$, and let $X \in \ante P$. Let
  $\ang{t_n \mid n \in \bN}$ be a sequence of terms that are i.i.d~given $X$.
  Assume $t_1$ is integrable and let $\mu = E[t_1 \mid X]$. Define the terms
  $s_n = (t_1 + \cdots + t_n)/\ul n$, and let
  \[
    \ts{
      s = \{
        (x, y) \bin \ul \bN_0 \times \ul \bR
      \mid
        \bigvee_{n \in \bN_0} x \beq \ul n \wedge y \beq s_n
      \}.
    }
  \]
  Then $P(\ph_\lm(s, \ul \mu) \mid X) = 1$.
\end{thm}

\begin{proof}
  By \eqref{nonstd-N-2}, we have $\ZFC_\infty \vdash s \bin \ul \bR^{\ul \bN_0}$
  and $s(\ul n) = s_n$. Since we also have $\ZFC \vdash \ul \mu \bin \ul \bR$,
  it follows that
  \[
    \ts{
      \ZFC \vdash \ph_\lm(s, \ul \mu) \tot {
        (\forall z \bin \ul{(0, \infty)})
        \bigvee_{m = 0}^\infty \bigwedge_{n = m}^\infty
        |s_n - \ul \mu| < z
      }.
    }
  \]
  In fact, if we define $\ep_\ell = 1/\ell$ for $\ell \in \bN$, then
  \[
    \ZFC \vdash \ph_\lm(s, \ul \mu) \tot \ts{
      \bigwedge_{\ell = 1}^\infty
      \bigvee_{m = 0}^\infty
      \bigwedge_{n = m}^\infty
      |s_n - \ul \mu| < \ul{\ep_\ell}.
    }
  \]
  Using the real frame of reference, we have $\ZFC \vdash |s_n - \ul \mu| <
  \ul{\ep_\ell} \tot s_n \bin \ul{V_\ell}$, where $V_\ell = (\mu - \ep_\ell, \mu
  + \ep_\ell)$. Hence, $\ZFC \vdash \ph_\lm(s, \ul \mu) \tot \psi$, where
  \[
    \psi = \ts{
      \bigwedge_{\ell = 1}^\infty
      \bigvee_{m = 0}^\infty
      \bigwedge_{k = 0}^\infty
      \bigwedge_{n = m}^{m + k}
      s_n \bin \ul{V_\ell}
    }.
  \]
  Hence, by the rule of logical implication and Proposition
  \ref{P:log-equiv-gen}, it suffices to show $P(\psi \mid X) = 1$.

  By the continuity rule,
  \[
    P(\psi \mid X) = {
      \lim_{\ell \to \infty} \lim_{m \to \infty} \lim_{k \to \infty}
      \ts{P(\bigwedge_{n = m}^{m + k} s_n \bin \ul{V_\ell} \mid X)}
    }.
  \]
  Let $Y_n$ and $Z_n$ be as in Lemma \ref{L:LLN}. Then the $Y_n$ are i.i.d.~and
  integrable, and $E^\nu [Y_1] = \mu$. By Lemma \ref{L:LLN},
  \begin{align*}
    P(\psi \mid X) &= {
      \lim_{\ell \to \infty} \lim_{m \to \infty} \lim_{k \to \infty}
      \ts{\opnu \bigcap_{n = m}^{m + k} \{Z_n \bin V_\ell\}}
    }\\
    &= \opnu \ts{
      \bigcap_{\ell = 1}^\infty \bigcup_{m = 0}^\infty \bigcap_{n = m}^\infty
      \{|Z_n - \mu| < \ep_\ell\}
    }.
  \end{align*}
  On the other hand, by the law of large numbers for random variables, $Z_n \to
  \mu$ $\nu$-a.s. Hence, $P(\psi \mid X) = 1$.
\end{proof}

\subsection{The central limit theorem for terms}

Let $t$ be Borel given $X$ and assume that $t$ is integrable. Let $r = E[t \mid
X]$. Then the term $(t - \ul r)^2 = (t - \ul r) \cdot (t - \ul r)$ is Borel
given $X$. The \emph{variance of $t$ given $X$} is defined by $V(t \mid X) =
E[(t - \ul r)^2 \mid X]$. Note that $V(t \mid X) \in [0, \infty]$.

\begin{thm}[Central limit theorem]
    \index{central limit theorem}%
  Let $P$ be a real inductive theory in $\ZFC$, and let $X \in \ante P$. Let
  $\ang{t_n \mid n \in \bN}$ be a sequence of terms that are i.i.d~given $X$.
  Assume $t_1$ is integrable. Let $\mu = E[t_1 \mid X]$ and let $s_n$ be as in
  Theorem \ref{T:LLN}. Let $\si = \sqrt{V(t_1 \mid X)}$ and assume $\si \in (0,
  \infty)$. Then
  \[
    \lim_{n \to \infty} P(\ul{\sqrt n} \bdot (s_n - \ul \mu) \le \ul r \mid X)
    = \frac 1 {\sqrt{2 \pi \si^2}} \int_{-\infty}^r {
      e^{-(x - \mu)^2/2\si^2}
    } \, dx,
  \]
  for all $r \in \bR$.
\end{thm}

\begin{proof}
  Let $Y_n$ and $Z_n$ be as in Lemma \ref{L:LLN}. As in the proof of Theorem
  \ref{T:LLN}, we have
  \[
    P(\ul{\sqrt n} \bdot (s_n - \ul \mu) \le \ul r \mid X)
    = \opnu \{\sqrt n \, (Z_n - \mu) \le r\}.
  \]
  The result therefore follows from the central limit theorem for random
  variables, applied to the sequence $\ang{Y_n \mid n \in \bN}$.
\end{proof}

\section{Probabilities of probabilities}\label{S:cond-exp}

Anyone who has played a tabletop role-playing game is familiar with the variety
of dice that are used. Not only are there the usual 6-sided dice, but the
standard collection also includes dice with 4, 8, 10, 12, and 20 sides. Imagine
taking a collection of such dice, choosing one of them at random, and rolling
it. In that case, it makes perfectly good sense to ask the question, ``What is
the probability that the probability of rolling a 1 is greater than $0.1$?''

Although this example illustrates that is sensible to talk about probabilities
of probabilities, it is a rather trivial example. More serious examples arise in
applications of inference, where probabilities are updated on the basis of a
succession of observations.

A priori, these nested probabilities seem to be a problem for inductive logic.
An inductive statement is a triple, $(X, \ph, p)$. The first two components of
the triple are part of the language $\cL$, but the triple itself is not. This
makes it impossible to construct a formal sentence in $\cL$ that contains an
inductive statement.

Measure-theoretic probability has a similar problem. There, probabilities are
neither events nor random variables, so they cannot appear inside a probability
measure. The way this is dealt with in measure theory is through the notion of
conditional expectation. A conditional expectation \emph{is}, in fact, a random
variable, so it \emph{can} appear inside a probability measure.

In this section, we construct the analogue of this for inductive logic. We will
construct conditional probability and expectation, where we condition on terms
and the resulting object is itself a term. We use this to formulate, in Theorem
\ref{T:total-prob}, the law of total probability, which is also known as the
tower rule, or the law of iterated expectation.

\subsection{Conditioning on terms}

A \emph{probability kernel from $\bR^n$ to $\bR^m$} is a function $\nu: \bR^n
\times \cB (\bR^m) \to [0, 1]$ such that $\mu(\vec r, \cdot)$ is a probability
measure for each $\vec r \in \bR^n$, and $\mu(\cdot, V)$ is a measurable
function for each $V \in \cB (\bR^m)$. If $m = n$, then we call $\mu$ a
\emph{probability kernel on $\bR^n$}. If we say that $\mu$ is a probability
kernel, or just a kernel, then we mean that $\mu$ is a probability kernel on
$\bR$.
  \index{probability!kernel}%

\begin{defn}\label{D:shrink-nice}
    \index{shrinks nicely}%
  Let $\nu$ be a probability measure on $(\bR, \cB(\bR))$. For each $n \in
  \bN_0$, let $V_n \in \cB(\bR)$, and let $r \in \bR$. Then \emph{$V_n$ shrinks
  nicely to $r$ with respect to $\nu$} if there exist $c > 0$ and $\ep_n > 0$
  such that $\ep_n \to 0$, $V_n \subseteq (r - \ep_n, r + \ep_n)$, and $\opnu
  V_n > c \opnu (r - \ep_n, r + \ep_n)$.
\end{defn}

\begin{thm}\label{T:cond-dist}
  Let $P$ be a real inductive theory in $\ZFC$, and let $X \in \ante P$. Let $s,
  t \in \cT$ be jointly Borel. Then there exists a probability kernel $\mu$ such
  that
  \begin{equation}\label{cond-dist}
    \lim_{n \to \infty} P(t \bin \ul V \mid X, s \bin \ul{V_n}) = \mu(r, V)
  \end{equation}
  for $\mu_{s \mid X}$-a.e.~$r \in \bR$, whenever $V_n$ shrinks nicely to $r$
  with respect to $\mu_{s \mid X}$. If $\wt \mu$ is another such kernel, then
  $\mu(r, \cdot) = \wt \mu(r, \cdot)$ for $\mu_{s \mid X}$-a.e.~$r \in \bR$.
\end{thm}

\begin{proof}
  By Proposition \ref{P:term-process}, we may define random variables $Y_1$ and
  $Y_2$ on a probability space $(S, \Ga, \nu)$ such that
  \[
    P(s \bin \ul V \wedge t \bin \ul V' \mid X)
    = \nu \{Y_1 \in V\} \cap \{Y_2 \in V'\}
  \]
  for all $V, V' \in \cB (\bR)$.

  Let $\opbfE^\nu$ denote expectation with respect to $\nu$, so that $\opbfE^\nu
  [Z] = \int_S Z \, d\nu$ whenever $Z$ is a real-valued random variable defined
  on $(S, \Ga, \nu)$. Then $\opbfE^\nu[1_{\{Y_2 \in V\}} \mid Y_1]$ denotes a
  version of the conditional expectation of $1_{\{Y_2 \in V\}}$ given $Y_1$.
  That is, $Z = \opbfE^\nu[1_{\{Y_2 \in V\}} \mid Y_1]$ is a
  $\si(Y_1)$-measurable random variable such that $\opbfE^\nu[1_{\{Y_2 \in V\}}
  1_B] = \opbfE^\nu[Z 1_B]$ for all $B \in \si(Y_1)$. Moreover, if $Z'$ is any
  other random variable with this property, then $Z = Z'$, $\nu$-a.s.

  There exists a probability kernel $\mu$ such that $\opbfE^\nu[1_{\{Y_2 \in
  V\}} \mid Y_1] = \mu(Y_1, V)$, $\nu$-a.s., for every $V \in \cB(\bR)$. This
  kernel is unique in the sense that if $\wt \mu$ is another such kernel, then
  there exists $N \in \Ga$ such that $\nu \{Y_1 \in N\} = 0$ and $\mu(r, V) =
  \wt \mu (r, V)$ for all $(r, V) \in N^c \times \cB(\bR)$. The kernel $\mu$ is
  called a regular conditional distribution for $Y_2$ given $Y_1$. The existence
  and uniqueness of regular conditional distributions is shown, for instance, in
  \cite[Theorem 5.3]{Kallenberg1997}.

  Let $V_n$ shrink nicely to $r$ with respect to $\mu_{s \mid X}$. Then $P(s
  \bin \ul V_n \mid X) = \opmu_{s \mid X} V_n > 0$. By Proposition
  \ref{P:term-process} and the multiplication rule,
  \[
    P(t \bin \ul V \mid X, s \bin \ul V_n)
    = \frac{\opnu \{Y_2 \in V\} \cap \{Y_1 \in V_n\}}{\opnu \{Y_1 \in V_n\}}.
  \]
  Using properties of conditional expectation and the fact that $\mu_{s \mid X}
  = \opnu \{Y_1 \bin {\cdot}\}$, we have
  \begin{align*}
    \opnu \{Y_2 \in V\} \cap \{Y_1 \in V_n\}
    &= \opbfE^\nu[1_{\{Y_1 \in V_n\}} \opbfE^\nu[1_{\{Y_2 \in V\}} \mid Y_1]]\\
    &= \opbfE^\nu[1_{\{Y_1 \in V_n\}} \mu(Y_1, V)]\\
    &= \int_{V_n} \mu(x, V) \, \mu_{s \mid X}(dx).
  \end{align*}
  Hence, by the Lebesgue differentiation theorem (see, for instance, \cite
  [Theorem 8.4.6] {Benedetto2009}),
  \[
    P(t \bin \ul V \mid X, s \bin \ul V_n)
    = \frac 1 {\mu_{s \mid X}(V_n)} \int_{V_n} \mu(x, V) \, \mu_{s \mid X}(dx)
    \to \mu(r, V)
  \]
  for $\mu_{s \mid X}$-a.e.~$r \in \bR$.

  Finally, suppose $\wt \mu$ is another probability kernel such that
  \eqref{cond-dist} holds. Define $N \subseteq \bR$ by $r \in N$ if and only if
  there exists $\ep > 0$ such that $\opmu_{s \mid X} (r - \ep, r + \ep) = 0$.
  Then $N$ is an open set, and $\opmu_{s \mid X} K = 0$ for all compact $K
  \subseteq N$. Hence, $\opmu_{s \mid X} N = 0$. Let $r \in N^c$ and $V \in
  \cB(\bR)$. Choose any $\ep_n > 0$ with $\ep_n \to 0$ and define $V_n = (r -
  \ep_n, r + \ep_n)$. Then $\opmu_{s \mid X} V_n > 0$, so $V_n$ satisfies
  Definition \ref{D:shrink-nice} with $c > 1/2$. Thus, $V_n$ shrinks nicely to
  $r$ with respect to $\mu_{s \mid X}$. By \eqref{cond-dist}, we have $\mu(r, V)
  = \wt \mu(r, V)$.
\end{proof}

\begin{cor}\label{C:cond-dist}
  Let $P$ be a real inductive theory in $\ZFC$, and let $X \in \ante P$. Let $s,
  t \in \cT$ be jointly Borel, and let $\mu$ be a probability kernel satisfying
  \eqref{cond-dist}. Let $r \in \bR$. If $P(t \bin \ul V \mid X, s \beq \ul r)$
  exists, then
  \[
    P(t \bin \ul V \mid X, s \beq \ul r) = \mu(r, V).
  \]
\end{cor}

\begin{proof}
  Choose any $\ep_n > 0$ with $\ep_n \to 0$ and let $V_n = (r - \ep_n, r +
  \ep_n)$. Suppose $P(t \bin \ul V \mid X, s \beq \ul r)$ exists. By Lemma 
  \ref{L:cond-exist}, we have $P(s \beq \ul r \mid X) > 0$. Hence, $\opmu_{s
  \mid X} \{r\} > 0$, so that $V_n$ shrinks nicely to $r$ with respect to $\mu_
  {s \mid X}$. By the multiplication and continuity rules, it follows that
  \[
    \lim_{n \to \infty} P(t \bin \ul V \mid X, s \bin \ul{V_n})
    = P(t \bin \ul V \mid X, s \beq \ul r).
  \]
  The result therefore follows from \eqref{cond-dist}.
\end{proof}

\begin{rmk}
  Theorem \ref{T:cond-dist} and Corollary \ref{C:cond-dist} can be generalized
  to the case where $s_1, \ldots, s_n, t_1, \ldots, t_m$ are jointly Borel and
  $\mu$ is a probability kernel from $\bR^n$ to $\bR^m$. In that case, let $\bfs
  = (s_1, \ldots, s_n)$, and replace \eqref{cond-dist} by
  \[
    \lim_{n \to \infty} P(
      \ts{\bigwedge_{\ell = 1}^m t_\ell \bin \ul{V_\ell}}
    \mid
      \ts{X, \bigwedge_{k = 1}^n s_k \bin \ul{V_{k, n}}}
    ) = \mu (r_1, \ldots, r_n, V_1 \times \cdots \times V_m)
  \]
  for $\mu_{\bfs \mid X}$-a.e.~$(r_1, \ldots, r_n) \in \bR^n$, whenever $V_
  {k, n}$ shrinks nicely to $r_k$ with respect to $\mu_{s_k \mid X}$.
\end{rmk}

Any kernel $\mu$ satisfying \eqref{cond-dist} is called a \emph{distribution of
$t$ given $X$ and $s$}.
  \index{distribution}%

\subsection{Versions of distributions}

For fixed $V$, a function $h(r)$ is called a \emph{version of $P(t \bin \ul V
\mid X, s \beq \ul r)$} if $h(r) = \mu(r, V)$ for some $\mu$ satisfying
\eqref{cond-dist}. Note that any two versions are equal $\mu_{s \mid X}$-a.e. We
may sometimes write $P(t \bin \ul V \mid X, s \beq \ul r) = h(r)$ to mean that
$h(r)$ is a version of $P(t \bin \ul V \mid X, s \beq \ul r)$, but it should be
remembered that if $P(s \beq \ul r \mid X) = 0$, then $P(t \bin \ul V \mid X, s
\beq \ul r)$ does not have a uniquely determined value for a fixed value of $r$.

Suppose $t$ is integrable given $X$. A function $h(r)$ is called a \emph{version
of $E[t \mid X, s \beq \ul r]$} if $h(r) = \int_\bR x \, \mu(r, dx)$ for some
$\mu$ satisfying \eqref{cond-dist}. Note that any two versions are equal $\mu_{s
\mid X}$-a.e. We may sometimes write $E[t \mid X, s \beq \ul r] = h(r)$ to mean
that $h(r)$ is a version of $E[t \mid X, s \beq \ul r]$, but we should remember
that if $P(s \beq \ul r \mid X) = 0$, then $E[t \mid X, s \beq \ul r]$ does not
have a uniquely determined value for a fixed value of $r$.

Let $\mu$ satisfy \eqref{cond-dist}. By the proof of Theorem \ref{T:cond-dist}
and properties of regular conditional distributions, we have $\int_\bR |x| \,
\mu(r, dx) < \infty$ for all $r \in \bR$. It follows that $r \mapsto \int_\bR x
\, \mu(r, dx)$ is a version of $E[t \mid X, s \beq \ul r]$ whenever $\mu$
satisfies \eqref{cond-dist}.

We may sometimes treat $P(t \bin \ul V \mid X, s \beq \ul r)$ and $E[t \mid X, s
\beq \ul r]$ as if they were functions, rather than expressions that possess
versions. In such situations, the meaning must be understood according to
context. For example, if we say that $P(t > 0 \mid X, s \beq \ul r) = 1 - e^
{-r}$, then we mean that $h(r) = 1 - e^{-r}$ is a version of $P(t > 0 \mid X, s
\beq \ul r)$. On the other hand, if we say that
\[
  P(t > 0 \mid X, s \beq \ul r) = E[t' \mid X, s \beq \ul r],
\]
then we mean that $h(r)$ is a version of $P(t > 0 \mid X, s \beq \ul r)$ if and
only if $h(r)$ is a version of $E[t' \mid X, s \beq \ul r]$. Since different
versions are equal $\mu_{s \mid X}$-a.e., such a claim can be verified by
checking it for a single version.

All of this could be made precise if we were to define $E[t \mid X, s \beq
\ul r]$ as
the equivalence class of $h(r)$, where $h(r)$ is a version, and two functions are
equivalent if they are equal $\mu_{s \mid X}$-a.e. To save ourselves from even
more notation, we avoid this approach. Moreover, it is common in
measure-theoretic probability to use the language of ``versions'' when
talking about conditional expectation.

\subsection{Indicator terms}

We have seen how to define $P(t \bin \ul V \mid X, s \beq \ul r)$. Here, we see
how to define $P(\psi \mid X, s \beq \ul r)$ for more general sentences $\psi$.
If $\psi \in \cL^0$, then we define the term $\ul 1_\psi$ by
\[
  y \beq \ul 1_\psi \tot {
    \psi \wedge y \beq \ul 1 \vee \neg \psi \wedge y \beq \ul 0
  }.
\]
If $P(\psi \mid X)$ exists, then $\ul 1_\psi$ is Borel given $X$. In fact, in
this case, $\ul 1_\psi$ is integrable and $E[\ul 1_\psi \mid X] = P(\psi \mid
X)$.

If $s$ is Borel given $X$, then $\ul 1_\psi$ and $s$ are jointly Borel if and
only if $P(\psi \wedge s \bin \ul V \mid X)$ exists for all $V \in \cB(\bR)$. In
this case, we define
\[
  P(\psi \mid X, s \beq \ul r)
  = P(\ul 1_\psi \bin \ul{\{1\}} \mid X, s \beq \ul r).
\]
That is, we say $h(r)$ is a version of $P(\psi \mid X, s \beq \ul r)$ if and
only if $h(r)$ is a version of $P(\ul 1_\psi \bin \ul{\{1\}} \mid X, s \beq \ul
r)$. Since $\ZFC \vdash \psi \tot \ul 1_\psi \bin \ul{\{1\}}$, Theorem
\ref{T:cond-dist} shows that
\[
  P(\psi \mid X, s \bin \ul{V_n}) \to P(\psi \mid X, s \beq \ul r),
\]
for $\mu_{s \mid X}$-a.e.~$r \in \bR$, whenever $V_n$ shrinks nicely to $r$ with
respect to $\mu_{s \mid X}$.

Note that
\[
  P(t \bin \ul V \mid X, s \beq \ul r) = E[\ul 1_\psi \mid X, s \beq \ul r],
\]
where $\psi = t \bin \ul V$. That is, $h(r)$ is a version of $P(t \bin \ul V
\mid X, s \beq \ul r)$ if and only if $h(r)$ is a version of $E[\ul 1_\psi \mid
X, s \beq \ul r]$. Hence, distributions given $X$ and $s$ can be characterized
entirely in terms of expectations.

The following result shows how to connect distributions given $X$ and $s$ to the
root of our inductive theory.

\begin{prop}
  Let $P$ be a real inductive theory in $\ZFC$ with root $T_0$, and let $X \in
  \ante P$. Write $X \equiv T + \psi$, where $T \in [T_0, T_P]$. Let $s$ and $t$
  be jointly Borel given $T_0$ and assume $P(\psi \wedge s \bin \ul{V_1} \wedge
  t \bin \ul{V_2} \mid T_0)$ exists for all $V_1, V_2 \in \cB(\bR)$. Also assume
  that $t$ is integrable given $T_0$. Then $E[t \mid X, s \beq \ul r] = h(r,
  1)$, where $h(r, r') = E[t \mid T_0, s \beq \ul r, \ul 1_\psi \beq \ul r']$.
\end{prop}

\begin{proof}
  Let $h(r, r')$ be a version of $E[t \mid T_0, s \beq \ul r, \ul 1_\psi \beq
  \ul r']$. Let $\bfs = (s, \ul 1_\psi)$. Then $h(r, r') = \int_\bR x \, \mu(r,
  r', dx)$ for some kernel $\mu$ from $\bR^2$ to $\bR$ such that
  \[
    \lim_{n \to \infty} P(
      t \bin \ul V \mid T_0, s \bin \ul{V_n}, \ul 1_\psi \bin \ul{V_n'}
    ) = \mu(r, r', V),
  \]
  for $\mu_{\bfs \mid T_0}$-a.e.~$(r, r') \in \bR^2$, whenever $V_n$ shrinks
  nicely to $r$ and $V_n'$ shrinks nicely to $r'$. It follows that
  \begin{align*}
    \lim_{n \to \infty} P(t \bin \ul V \mid X, s \bin \ul{V_n})
    &= \lim_{n \to \infty} P(t \bin \ul V \mid T_0, \psi, s \bin \ul{V_n})\\
    &= \lim_{n \to \infty} P(
      t \bin \ul V \mid T_0, s \bin \ul{V_n}, \ul 1_\psi \bin \ul{\{1\}}
    )\\
    &= \mu(r, 1, V),
  \end{align*}
  for $\mu_{s \mid X}$-a.e.~$r \in \bR$, whenever $V_n$ shrinks nicely to $r$.
  Hence, $h(r, 1) = \int_\bR x \, \mu(r, 1, dx)$ is a version of $E[t \mid X, s
  \beq \ul r]$.
\end{proof}

\subsection{Conditional expectation}

\stepcounter{thm}

\begin{prop}\label{P:cond-exp-versions}
  Let $P$ be a real inductive theory in $\ZFC$ with root $T_0$, and let $X \in
  \ante P$. Let $s, t \in \cT$ be jointly Borel, and assume $t$ is integrable.
  If $h(r)$ and $h'(r)$ are versions of $E[t \mid X, s \beq \ul r]$, then $P(\ul
  h(s) \beq \ul{h'}(s) \mid X) = 1$.
\end{prop}

\begin{proof}
  Let $V = \{r \in \bR \mid h(r) = h'(r)\} \in \cB(\bR)$. We first show that
  \begin{equation}\label{cond-exp-versions}
    \ZFC \vdash \ul h(s) \beq \ul{h'}(s) \tot s \bin \ul V.
  \end{equation}
  Let $\sP = (\Om, \Si, \bbP)$ be a model and assume $\sP \vDash \ZFC$ and $\sP
  \vDash \ul h(s) \beq \ul{h'}(s)$. By adopting the real frame of reference, we
  may assume $\sP$ satisfies all of the conditions in Theorems \ref{T:real-FOR},
  \ref{T:real-FOR-2}, and \ref{T:real-FOR-3}. Choose $\Om^* \in \Si$ such that
  $\bbP \Om^* = 1$ and, for every $\om \in \Om^*$, we have that Theorem
  \ref{T:real-FOR}(v) holds for $V$, Theorem \ref{T:real-FOR-3} holds for both
  $h$ and $h'$, and $\om \tDash \ul h(s) \beq \ul{h'}(s)$. Let $\om \in \Om^*$.
  Then $s^\om \in {}^\om \ul \bR$ and there exists $b \in {}^\om \ul \bR$ such
  that $b = h(s^\om)$ and $b = h'(s^\om)$. Hence, $s^\om \in V \cap {}^\om \ul
  \bR = {}^\om \ul V$, so that $\om \tDash s \bin \ul V$. Since this is true for
  every $\om \in \Om^*$, we have $\sP \vDash s \bin \ul V$. Since $\sP$ was
  arbitrary, this gives $\ZFC, \ul h(s) \beq \ul{h'} (s) \vdash s \bin \ul V$. A
  similar proof shows that $\ZFC, s \bin \ul V \vdash \ul h(s) \beq \ul{h'}(s)$,
  and this verifies \eqref{cond-exp-versions}.

  By \eqref{cond-exp-versions} and Proposition \ref{P:log-equiv-gen}, we have
    \[
      P(\ul h(s) \beq \ul{h'}(s) \mid X)
      = P(s \bin \ul V \mid X)
      = \opmu_{s \mid X} V
      = \opmu_{s \mid X} \{r \in \bR \mid h(r) = h'(r)\}.
    \]
  But $h$ and $h'$ are both versions of $E[t \mid X, s \beq \ul r]$. Therefore,
  $h = h'$, $\mu_{s \mid X}$-a.e, which shows that $P(\ul h(s) \beq \ul{h'} (s)
  \mid X) = 1$.
\end{proof}

If $h(r)$ is a version of $E[t \mid X, s \beq \ul r]$, then the term $\ul h(s)$
is called a \emph{version of $\E[t \mid X, s]$}. Similarly, if $h(r)$ is a
version of $P(t \bin \ul V \mid X, s \beq \ul r)$, then the term $\ul h(s)$ is
called a \emph{version of $\P(t \bin \ul V \mid X, s)$}. We call $\E[t \mid X,
s]$ the \emph{conditional expectation of $t$ given $X$ and $s$}.
  \index{conditional expectation}%
  \symindex{$\E[t \mid X, s]$}%
  \symindex{$\P(t \bin \ul V \mid X, s)$}%

We may sometimes treat $\E[t \mid X, s]$ as if it were a term, rather than an
expression that possesses versions. In such situations, the meaning must be
understood according to context. For example, if we say that $t' = \E[t \mid X,
s]$, then we mean that $t'$ is a version of $\E[t \mid X, s]$. On the other
hand, if we say that $E[\E[t \mid X, s] \mid X] = E[t \mid X]$, then what we
mean is that $E[t' \mid X] = E[t \mid X]$ whenever $t'$ is a version of $\E[t
\mid X, s]$. A similar convention applies to versions of $\P(t \bin \ul V \mid
X, s)$.

A version of $E[t \mid X, s \beq \ul r]$ is a measurable function, whereas a
version of $\E[t \mid X, s]$ is a term. Similarly, a version of $P(t \bin \ul V
\mid X, s \beq \ul r)$ is a measurable function, whereas a version of $\P(t \bin
\ul V \mid X, s)$ is a term. Terms can appear in inductive statements. Hence,
for example,
\begin{equation}\label{cond-exp-expl}
  P(\P(t \beq \ul 1 \mid X, s) > \ul{0.5} \mid X) = 2/3
\end{equation}
is a perfectly meaningful thing to say. It says that $(X, \ul h(s) > \ul{0.5},
2/3) \in P$ whenever $h(r)$ is a version of $P(t \beq \ul 1 \mid X, s \beq \ul
r)$. By Propositions \ref{P:cond-exp-versions} and \ref{P:log-equiv-gen}, if we
verify \eqref{cond-exp-expl} for a single version, then it is true for all
versions.

\subsection{The law of total probability}

\begin{thm}[Law of total probability]\label{T:total-prob}
    \index{law of total probability}%
  Let $P$ be a real inductive theory in $\ZFC$ and let $X \in \ante P$. Let $s,
  t \in \cT$ be jointly Borel, and assume $t$ is integrable. Then $E[t \mid X] =
  E[\E[t \mid X, s] \mid X]$.
\end{thm}

\begin{proof}
  Let $h(r)$ be a version of $E[t \mid X, s \beq \ul r]$. Then $h(r) = \int_\bR
  x \, \mu(r, dx)$ for some kernel $\mu$ satisfying \eqref{cond-dist}. Let $Y_1$
  and $Y_2$ be the random variables constructed in the proof of Theorem 
  \ref{T:cond-dist}, so that $\mu$ is a regular conditional distribution for
  $Y_2$ given $Y_1$. By properties of regular conditional distributions, we have
  $\opbfE^\nu[Y_2 \mid Y_1] = \int_\bR x \, \mu(Y_1, dx) = h(Y_1)$. Using
  properties of conditional expectations, it follows that $\opbfE^\nu[h(Y_1)] =
  \opbfE^\nu[\opbfE^\nu[Y_2 \mid Y_1]] = \opbfE^\nu[Y_2] = E[t \mid X]$. It
  therefore suffices to show that $E[s' \mid X] = \opbfE^\nu[h(Y_1)]$, where
  $s' = \ul h(s)$.

  Using the real frame of reference, it follows easily that $\ZFC \vdash s' \bin
  \ul V \tot s \bin \ul{h^{-1} V}$. Hence, $s'$ is Borel given $X$ and $\mu_{s'
  \mid X} = \mu_{s \mid X} \circ h^{-1}$, which implies
  \[
    E[s' \mid X]
    = \int_\bR x \, \mu_{s' \mid X}(dx)
    = \int_\bR h(x) \, \mu_{s \mid X}(dx).
  \]
  But $\mu_{s \mid X} = \opnu \{Y_1 \in \cdot\}$, so $E[s' \mid X] =
  \opbfE^\nu[h(Y_1)]$.
\end{proof}

Suppose $P(\psi \wedge s \bin \ul V \mid X)$ exists for each $V \in \cB(\bR)$.
In this case, we define $\P(\psi \mid X, s) = \E[\ul 1_\psi \mid X, s]$, which
means that a term $t$ is a version of $\P(\psi \mid X, s)$ if and only if it is
a version of $\E[\ul 1_\psi \mid X, s]$. Since we also have that $P(\psi \mid X)
= E [\ul 1_\psi \mid X]$, Theorem \ref{T:total-prob} gives us
\[
  P(\psi \mid X) = E[\P(\psi \mid X, s) \mid X].
\]
This special case of Theorem \ref{T:total-prob}, especially when $\mu_{s \mid
X}$ is discrete, is what is more commonly known as the law of total probability.

%% file: princ-indiff.tex

\chapter{Principle of Indifference}\label{Ch:PoI}

As discussed in Section \ref{S:intro-PoI}, the principle of indifference is the
heuristic idea that if we are equally ignorant about two statements, then we
ought to assign them the same probability. The principle dates back to Laplace
and the birth of mathematical probability. Although it is intuitively
self-evident, it has a history of producing apparent paradoxes. It has no
rigorous formulation in measure-theoretic probability theory. Hence, using
measure theory alone, we are helpless to distinguish between valid and invalid
uses of the principle.

In Section \ref{S:PoI-defn}, we give a precise formulation of the principle of
indifference in the context of inductive logic. As we will see, it is a natural
generalization of a basic principle of deductive logic. In Sections
\ref{S:basic-expls-PoI} and \ref{S:basic-expls-PoI-2}, we present several
elementary examples of the principle of indifference in action. All of these
examples are finite, in the sense that they involves models $\sP = (\Om, \Si,
\bbP)$, where $\Om$ is finite.

In the last three sections of this chapter, we consider the principle of
indifference in the context of real inductive theories. In Section
\ref{S:indiff-exch}, we show that exchangeability is a special case of
indifference. Sections \ref{S:indiff-interval} and \ref{S:indiff-plane} present
several concrete examples. Section \ref{S:indiff-interval} treats examples
involving an interval on the real line, while Section \ref{S:indiff-plane}
treats examples involving circles and disks in the plane. The final example of
Section \ref{S:indiff-plane} is the famous example of Bertrand's paradox.

\section{Formulating the principle}\label{S:PoI-defn}

The principle of indifference is an inductive generalization of a fundamental
principle of deductive logic. This deductive principle is one that we use all
the time, especially when we make assumptions ``without loss of generality.''
This principle, which we call ``deductive indifference,'' is presented in
Section \ref{S:ded-PoI}. In order to state it, we first define what we call
``signature permutations.''

\subsection{Signature permutations}

Let $L$ be a logical signature with associated predicate language $\cL$. A
\emph{signature permutation}, or \emph{$L$-permutation}, is a bijection $\pi: L
\to L$ such that $\s^\pi$ has the same type and arity as $\s$, for all $\s \in
L$.
  \index{permutation!signature ---}%
Given a signature permutation, we extend it to $\pi: \cT \to \cT$ by $x^\pi = x$
and $(f t_1 \cdots t_n)^\pi = f^\pi t_1^\pi \cdots t_n^\pi$. We then extend it
to $\pi: \cL \to \cL$ by
\begin{enumerate}[(i)]
  \item $(s \beq t)^\pi = (s^\pi \beq t^\pi)$,
  \item $(r t_1 \cdots t_n)^\pi = r^\pi t_1^\pi \cdots t_n^\pi$,
  \item $(\neg \ph)^\pi = \neg \ph^\pi$,
  \item $(\bigwedge \Phi)^\pi = \bigwedge_{\ph \in \Phi} \ph^\pi$, and
  \item $(\forall x \ph)^\pi = \forall x \ph^\pi$.
\end{enumerate}
By the unique concatenation and reconstruction properties in Sections
\ref{S:terms} and \ref{S:pred-formulas}, each of these extensions of $\pi$ is a
bijection. Moreover, by formula induction in $\cL_\fin$, we have that $\ph \in
\cL_\fin$ if and only if $\ph^\pi \in \cL_\fin$. For $X \subseteq \cL$, we write
$X^\pi = \{\ph^\pi \mid \ph \in X\}$. If $X^\pi \equiv X$, then we say that
\emph{$X$ is invariant under $\pi$}, or \emph{$\pi$-invariant}.
  \index{invariant}%

We will sometimes denote the inverse permutation, $\pi^{-1}$, by $-\pi$. Also,
when $\pi$ affects only finitely many extralogical symbols, we may use the usual
cycle notation for permutations. For example, $\pi = (c_1 \; c_2)(r_1 \; r_2 \;
r_3)$ means that $c_1^\pi = c_2$, $c_2^\pi = c_1$, $r_1^\pi = r_2$, $r_2^\pi =
r_3$, $r_3^\pi = r_1$, and $\s^\pi = \s$ for all other extralogical symbols.

By term and formula induction, $\psi \in \Sf \ph$ if and only if $\psi^\pi \in
\Sf \ph^\pi$. Also, $\var t = \var t^\pi$, and $\rk \ph$, $\var \ph$, $\bnd
\ph$, and $\free \ph$ are all unchanged by replacing $\ph$ with $\ph^\pi$. In
particular, $t$ is a ground term if and only if $t^\pi$ is a ground term, and
$\ph$ is a sentence if and only if $\ph^\pi$ is a sentence. On the other hand,
$\sym \ph^\pi = \pi(\sym \ph)$ and $\con \ph^\pi = \pi(\con \ph)$.

If $\si: \Var \to \cT$ is a substitution, then define the substitution $\si':
\Var \to \cT$ by $\si' = \pi \circ \si$, which is the substitution given by
$x^{\si'} = x^{\si \pi}$ for all $x \in \Var$. Note that this relation does not
extend to all of $\cT$. For instance $c^{\si'} = c$, but $c^ {\si \pi} = c^\pi$.

\begin{prop}
  For all $t \in \cT$ and all $\ph \in \cL$, we have $t^{\si \pi} = t^{\pi
  \si'}$ and $\ph^{\si \pi} = \ph^{\pi \si'}$. In particular, $\ph(t/x)^\pi =
  \ph^\pi(t^\pi/x)$.
\end{prop}

\begin{proof}
  We first prove $t^{\si \pi} = t^{\pi \si'}$ by term induction. Since $x^\pi =
  x$, it is true for $x \in \Var$ by the definition of $\si'$. Since constants
  are unaffected by substitutions, we have $c^{\si \pi} = c^\pi = c^{\pi \si'}$.
  Suppose it is true for $t_1, \ldots, t_n$ and let $f$ be an $n$-ary function
  symbol. Then
  \begin{multline*}
    (f t_1 \cdots t_n)^{\si \pi} = (f t_1^\si \cdots t_n^\si)^\pi
      = f^\pi t_1^{\si \pi} \cdots t_n^{\si \pi}\\
      = f^\pi t_1^{\pi \si'} \cdots t_n^{\pi \si'}
      = (f^\pi t_1^\pi \cdots t_n^\pi)^{\si'}
      = (f t_1 \cdots t_n)^{\pi \si'},
  \end{multline*}
  and it is true for $f t_1 \cdots t_n$. By term induction, it holds for all $t
  \in \cT$.

  We next prove $\ph^{\si \pi} = \ph^{\pi \si'}$ by formula induction. The proof
  that it holds for prime formulas and for formulas of the form $\ph = \neg
  \psi$ and $\ph = \bigwedge \Phi$ is similar to the above. Suppose $\ph =
  \forall x \psi$. Then $\ph^{\si \pi} = (\forall x \psi^\tau)^\pi = \forall x
  \psi^{\tau \pi}$, where $x^\tau = x$ and $y^\tau = y^\si$ for $y \ne x$. By
  the inductive hypothesis, this gives $\ph^{\si \pi} = \forall x \psi^{\pi
  \tau'}$. Using the fact that the proposition holds for terms, we have $x^
  {\tau'} = x$ and $y^{\tau'} = y^{\si'}$ for $y \ne x$. Hence, $\ph^{\si \pi} =
  (\forall x \psi^\pi)^{\si'} = \ph^{\pi \si'}$. The final assertion follows
  from the fact that if $\si = t/x$, then $\si' = t^\pi/x$.
\end{proof}

\begin{prop}
  Let $\ph \in \cL$. Then $\si$ is free for $\ph$ if and only if $\si'$ is free
  for $\ph^\pi$. In particular, $t$ is free for $x$ in $\ph$ if and only if
  $t^\pi$ is free for $x$ in $\ph^\pi$.
\end{prop}

\begin{proof}
  First note that $\ze$ is in the scope of $\forall z$ in $\ph$ if and only if
  $\ze^\pi$ is in the scope of $\forall z$ in $\ph^\pi$. Now suppose $y$ in not
  free for $x$ in $\ph$. Then there exists $\ze \in \Sf \ph$ such that $x \in
  \free \ze$, $\ze$ is not in the scope of $\forall x$ in $\ph$, and $\ze$ is in
  the scope of $\forall y$ in $\ph$. Since $\ze^\pi \in \Sf \ph^\pi$ and $\free
  \ze^\pi = \free \ze$, it follows that $y$ is not free for $x$ in $\ph^\pi$.
  The converse also holds. Hence, the second part of the proposition is true for
  $t = y$. For general $t$, simply note that $y \in t$ if and only if $y \in
  t^\pi$. The case of general $\si$ now follows since $\si$ is free for $\ph$ if
  and only if $x^\si$ is free for $x$ in $\ph$ for all $x$, and $x^{\si \pi} =
  x^{\si'}$.
\end{proof}

\subsection{Deductive indifference}\label{S:ded-PoI}

Theorem \ref{T:pf-invar} below says that the deductive derivability relation is
preserved by signature permutations. The proof is elementary and the result is
completely unsurprising. After all, the symbols that appear in a proof have no
direct relevance. It is only their relationships to one another that matters.
This principle is what we might call ``deductive indifference.'' As an example
of applying this principle, we use it to formalize the technique of assuming
something ``without loss of generality.''

\begin{thm}\label{T:pf-invar}
  Let $X \subseteq \cL$ and $\ph \in \cL$. Let $\pi$ be a signature permutation.
  Then $X^\pi \vdash \ph^\pi$ if and only if $X \vdash \ph$. In particular, if
  $X$ is invariant under $\pi$, then $X \vdash \ph^\pi$ if and only if $X \vdash
  \ph$.
\end{thm}

\begin{proof}
  Suppose $X \vdash \ph$ and let $\ang{\ph_\be \mid \be \le \al}$ be a proof of
  $\ph$ from $X$. We claim that $\langle \ph_\be^\pi \mid \be \le \al \rangle$
  is a proof of $\ph^\pi$ from $X^\pi$. Since $(\psi \to \ze)^\pi = \psi^\pi \to
  \ze^\pi$ and $(\bigwedge \Phi)^\pi = \bigwedge \Phi^\pi$, the claim will
  follow once we show that $\La^\pi = \La$.

  For this, it suffices to show that $\La \subseteq \La^\pi$, since we can
  replace $\pi$ by $-\pi$ and apply $\pi$ to both sides. Showing $\La \subseteq
  \La^\pi$ can be done by verifying (I)--(IV) in Section \ref{S:Karp-calc}, with
  $\La$ replaced by $\La^\pi$. Verifying (II)--(IV) is straightforward. To show
  that $\La^- \subseteq \La^\pi$, it is enough to show $(\La^-)^\pi \subseteq
  \La$, for the reasons given above.

  Let $\ph \in \La^-$. If $\ph$ has the form ($\La$1), then so does $\ph^\pi$,
  so that $\ph^\pi \in \La$. The same is true for ($\La$2), ($\La$3), ($\La$5),
  and ($\La$6). For ($\La$4), suppose $\ph = \forall x \psi \to \psi(t/x)$,
  where $t$ is free for $x$ in $\psi$. Then $t^\pi$ is free for $x$ in
  $\psi^\pi$, and $\ph^\pi = \forall x \psi^\pi \to \psi^\pi(t^\pi/x)$. Hence,
  $\ph^\pi \in \La$. The proof for ($\La$7) is similar.

  This shows that $X \vdash \ph$ implies $X^\pi \vdash \ph^\pi$. Using the
  result with $-\pi$ gives the converse. The second result follows from the
  first by the definition of invariance.
\end{proof}

Theorem \ref{T:pf-invar} shows that if $T \subseteq \cL^0$ is a deductive
theory, then $T^\pi$ is also. This is because if $T^\pi \vdash \ph$, then $T
\vdash \ph^{-\pi}$, so that $\ph^{-\pi} \in T$, which implies $\ph \in T^\pi$.
In fact, we have $T(X^\pi) = T(X)^\pi$ for any $X \subseteq \cL^0$.

The following corollary is a formalization of the without-loss-of-generality
proof method. It is also true in $\cL_\fin$ if we require $\Phi$ to be finite.

\begin{cor}\label{C:pf-invar}
  Let $X \subseteq \cL$ and $\ph \in \cL$. Suppose $\Phi \subseteq \cL$ is
  countable and $X \vdash \bigvee \Phi$. Fix $\th_0 \in \Phi$ and suppose that
  for each $\th \in \Phi$, there is a signature permutation $\pi$ such that
  $\th_0^\pi = \th$, $X$ is $\pi$-invariant, and $\ph$ is $\pi$-invariant. Then
  $X, \th_0 \vdash \ph$ implies $X \vdash \ph$.
\end{cor}

\begin{proof}
  Assume $X, \th_0 \vdash \ph$. Then $X \vdash \th_0 \to \ph$. But $\th_0 \to
  \ph \equiv \neg \ph \to \neg \th_0$. Hence, $X, \neg \ph \vdash \neg \th_0$.
  Let $\th \in \Phi$. Choose $\pi$ as in our hypotheses. Since $X, \neg \ph$ is
  invariant under $\pi$, it follows from Theorem \ref{T:pf-invar} that $X, \neg
  \ph \vdash \neg \th_0^\pi \equiv \neg \th$. Since $\th$ was arbitrary, $X,
  \neg \ph \vdash \bigwedge_{\th \in \Phi} \neg \th \equiv \neg \bigvee \Phi$.
  Therefore, $X, \bigvee \Phi \vdash \ph$. But $X \vdash \bigvee \Phi$, so we
  have $X \vdash \ph$.
\end{proof}

\begin{expl}
  Consider a scenario where we have three objects, and each is painted either
  red or blue. Define the extralogical signature $L = \{r, b\}$, where $r$ and
  $b$ are unary relation symbols. Let
  \[
    X = \{\exists_{=3}, \forall x ((rx \vee bx) \wedge \neg (rx \wedge bx))\}.
  \]
  The set $X$ asserts that there are three objects, each is red or blue, and
  none are both red and blue. Let us introduce the defined binary relation
  symbol $s$ by
  \[
    sxy \tot (rx \wedge ry) \vee (bx \wedge by),
  \]
  so that $sxy$ asserts that $x$ and $y$ have the same color. We want to prove
  that there must be two objects of the same color. That is, we want to show
  that $X \vdash \ph$, where
  \[
    \ph = \exists xy (x \nbeq y \wedge sxy).
  \]
  Let us add a new constant $c$ that represents an arbitrary object. By
  Proposition \ref{P:add-constants}, it suffices to show that $X \vdash_{\cL c}
  \ph$. Informally, we would like to say that, without loss of generality, we
  may assume $c$ is red. In other words, we claim that it suffices to show $X,
  rc \vdash_{\cL c} \ph$.

  This is justified by Corollary \ref{C:pf-invar}. To see this, define $\pi$ by
  $r^\pi = b$, $b^\pi = r$, and $c^\pi = c$. Then both $X$ and $\ph$ are
  $\pi$-invariant. Let $\Phi = \{rc, bc\}$. Then $X \vdash rc \vee bc$ and
  $(rc)^\pi = bc$. Hence, according to Corollary \ref{C:pf-invar}, if we can
  establish $X, rc \vdash_{\cL c} \ph$, then we may conclude $X \vdash_{\cL c}
  \ph$.
\end{expl}

\subsection{Inductive indifference}

We are now able to state the principle of indifference, which is an extension of
Theorem \ref{T:pf-invar} to the inductive setting.

\begin{defn}[The Principle of Indifference]
    \index{principle of indifference}%
  Let $P$ be an inductive theory. Suppose that, for every signature permutation
  $\pi$, we have:
  \begin{enumerate}[(R1), leftmargin=3em]
    \setcounter{enumi}{9}
    \item If $P(\ph \mid X)$ exists and $X^\pi \in \ante P$, then $P(\ph^\pi
          \mid X^\pi) = P(\ph \mid X)$.
  \end{enumerate}
  Then $P$ satisfies the \emph{principle of indifference}.
\end{defn}

If $P(\ph \mid X)$ is to be an evidentiary relationship between $X$ and $\ph$,
then the principle of indifference is a natural consistency condition to impose.
After all, there is nothing special about the symbols we choose to use. In
arithmetic, if we everywhere switch the symbols $+$ and $\cdot\,$, it is still
the same arithmetic. The symbols just have the opposite meaning. If we were
going to assign a certain probability in the original setting, then we ought to
assign that same probability after the symbols are reversed, because nothing has
actually changed.

In fact, it is such a natural requirement, it would have made sense to define it
as one of our rules of inductive inference, making it part of the definition of
an inductive theory. We did not do this for two reasons. First, the principle of
indifference is a rule solely for predicate logic. Making it a rule of inductive
inference would have created an asymmetry between propositional and predicate
logic. Second, by omitting it from the rules of inductive inference, we were be
able to prove Theorem \ref{T:prob-model-iso}, which shows that all of modern,
measure-theoretic probability is embedded in inductive logic. If we made the
principle of indifference part of the definition of an inductive theory, this
would not be the case. In other words, modern probability as we know it today,
for better or for worse, does not require us to conform to the principle of
indifference.

And so, the principle of indifference is not a required part of inductive logic.
That is, inductive theories are required to satisfy (R1)--(R9), but they are not
required to satisfy (R10). We will show, however, in the remaining sections of
this chapter, how to add this requirement to our inferences by using inductive
conditions.

Now, if $X$ is $\pi$-invariant, then according to the rule of logical
equivalence, we can reformulate the principle of indifference as $P(\ph^\pi \mid
X) = P(\ph \mid X)$. This reformulation is perhaps closer to the intuitive idea
of the principle of indifference. To say that $X$ is our antecedent is to say
that the totality of facts which we know consists of $X$, together with
everything that can be proven from $X$. In other words, $T(X)$, the deductive
theory generated by $X$, is the set of sentences that represents our knowledge.
If $X$ is $\pi$-invariant, then $T(X) = T(X^\pi)$. In other words, our knowledge
remains entirely unchanged by the permutation $\pi$. In this sense, then, we
cannot even ``see'' the permutation $\pi$. We are therefore ``indifferent''
between $\ph$ and $\ph^\pi$. Everything we know about $\ph$, we also know about
$\ph^\pi$, and vice versa. In this case, according to the principle of
indifference, we should assign them the same probability.

\subsection{Structures, models, and indifference}\label{S:PoI-models}

In the remainder of this chapter, we will present several examples of the
principle of indifference in action. In order to verify the principle of
indifference in these examples, we will need several results about how (R10)
relates to structures and models.

Let $\sP = (\Om, \Si, \bbP)$ be an $\cL$-model, and let $\pi$ be a signature
permutation. For $\om = (A, L^\om) \in \Om$, define $\om^\pi = (A, L^{\om^\pi})$
so that $(\s^\pi)^{\om^\pi} = \s^\om$ for all $\s \in L$. That is, $\om^\pi =
\om \circ \pi^{-1}$. Let $\Om^\pi = \{\om^\pi \mid \om \in \Om\}$ and let
$h_\pi: \Om \to \Om^\pi$ denote the map $\om \mapsto \om^\pi$. Let $\sP^\pi =
(\Om^\pi, \Si^\pi, \bbQ)$ be the measure space image of $\sP$ under $h_\pi$.
Note that $h_\pi$ is a bijection, and is therefore a pointwise isomorphism (as
measure spaces) from $\sP$ to $\sP^\pi$. Hence, it induces a measure-space
isomorphism from $(\Om, \ol \Si, \olbbP)$ to $ (\Om^\pi, \ol{\Si^\pi}, \olbbQ)$.
In particular, ${\olbbQ} = {\olbbP} \circ h_\pi^{-1}$.

Given an assignment $\bv$ into $\sP$, define the assignment $\bv^\pi$ into
$\sP^\pi$ by $v^\pi_{\om^\pi}(x) = v_\om(x)$ for all $x \in \Var$. By term
induction, we have $v^\pi_{\om^\pi}(t^\pi) = v_\om(t)$ for all $t \in \cT$, and
by formula induction, $\om \tDash \ph[v_\om]$ if and only if $\om^\pi \tDash
\ph^\pi[v^\pi_{\om^\pi}]$, for any $\ph \in \cL$. Hence, $\ph[\bv]_\Om =
h_\pi^{-1} \ph^\pi [\bv^\pi]_{\Om^\pi}$. Since ${\olbbQ} = {\olbbP} \circ h_\pi^
{-1}$, this gives $\olbbP \ph[\bv]_\Om = \olbbQ \ph^\pi [\bv^\pi]_{\Om^\pi}$. In
particular,
\begin{equation}\label{sem-pf-invar}
  \text{$\sP \vDash \ph[\bv]$ if and only if $\sP^\pi \vDash \ph^\pi[\bv^\pi]$},
\end{equation}
for all $\ph \in \cL$ and all assignments $\bv$ into $\sP$. Note that 
\eqref{sem-pf-invar} could be used to give a semantic proof of Theorem
\ref{T:pf-invar}.

\begin{thm}\label{T:induc-invar}
  Let $\sP$ be an $\cL$-model. Then for any $(X, \ph, p) \in \cL^\IS$, we have
  $\sP \vDash (X, \ph, p)$ if and only if $\sP^\pi \vDash (X^\pi, \ph^\pi, p)$.
\end{thm}

\begin{proof}
  Let $\sP = (\Om, \Si, \bbP)$ be an $\cL$-model and $\pi$ a signature
  permutation, so that $\sP^\pi = (\Om^\pi, \Si^\pi, \bbQ)$. Let $(X, \ph, p)
  \in \cL^\IS$. Then $X \subseteq \cL^0$ and $\ph \in \cL^0$. Suppose $\sP
  \vDash (X, \ph, p)$. Then there exists $Y \subseteq \cL^0$ and $\psi \in
  \cL^0$ such that $\sP \vDash Y$, $X \equiv Y \cup \{\psi\}$, and $\olbbP
  \ph_\Om \cap \psi_\Om / \olbbP \psi_\Om = p$. By \eqref{sem-pf-invar}, we have
  $\sP^\pi \vDash Y^\pi$. Theorem \ref{T:pf-invar} implies $X^\pi \equiv Y^\pi
  \cup \{\psi^\pi\}$. And it follows from ${\olbbQ} = {\olbbP} \circ h_\pi^{-1}$
  that $\olbbQ \ph^\pi_{\Om^\pi} \cap \psi^\pi_{\Om^\pi} / \olbbQ
  \psi^\pi_{\Om^\pi} = p$. Hence, $\sP^\pi \vDash (X^\pi, \ph^\pi, p)$. Applying
  this with $-\pi$ gives the converse.
\end{proof}

\begin{prop}
  Let $\sP$ be an $\cL$-model. Suppose that $\sP \simeq \sP^\pi$ for every
  signature permutation $\pi$. Then $\bTh \sP$ satisfies the principle of
  indifference.
\end{prop}

\begin{proof}
  Let $P = \bTh \sP$ and suppose $P(\ph \mid X) = p$. Then $\sP \vDash (X, \ph,
  p)$, so that Theorem \ref{T:induc-invar} gives $\sP^\pi \vDash (X^\pi,
  \ph^\pi, p)$. By hypothesis, $\sP \simeq \sP^\pi$. Hence, Theorem
  \ref{T:ind-iso-thm} implies $\sP \vDash (X^\pi, \ph^\pi, p)$. Therefore,
  $P(\ph^\pi \mid X^\pi) = p$.
\end{proof}

\begin{prop}
  Let $P$ be an inductive theory with root $T_0$, and let $T \in [T_0, T_P]$. If
  $P$ satisfies the principle of indifference, then so does $P \dhl_{[T, T_P]}$.
\end{prop}

\begin{proof}
  Let $P' = P \dhl_{[T, T_P]}$. Proposition \ref{P:chop-off-root} implies that
  $P'$ is an inductive theory. Suppose $P'(\ph \mid X) = p$ and $X^\pi \in
  \ante P'$. Since $P' \subseteq P$, we have $P(\ph \mid X) = p$ and $X^\pi \in
  \ante P$. By the principle of indifference for $P$, it follows that $P
  (\ph^\pi \mid X^\pi) = p$. But $X^\pi \in \ante P'$, and so we have $X^\pi
  \cao [T, T_P]$. Therefore, $P'(\ph^\pi \mid X^\pi) = p$.
\end{proof}

\begin{lemma}\label{L:mapping-T_0}
  Let $P$ be an inductive theory with root $T_0$ and $\pi$ a signature
  permutation. Let $X \subseteq \cL^0$ and assume $X, X^\pi \in \ante P$. Write
  $X \equiv T + \psi$ and $X^\pi \equiv T' + \psi'$, where $T, T' \in [T_0,
  T_P]$ and $\psi, \psi' \in \cL^0$. Suppose $\sP \vDash P$. Then, for all $\ze
  \in T_0$, we have $\psi_\Om \subseteq \ze^{-\pi}_\Om$ and $\psi'_\Om \subseteq
  \ze^\pi_\Om$, $\bbP$-a.s. In particular, if $\ze_\Om^\pi \in \ol \Si$, then
  $\olbbP \ze^\pi_\Om > 0$, and if $\ze_\Om^{-\pi} \in \ol \Si$, then $\olbbP
  \ze^{-\pi}_\Om > 0$.
\end{lemma}

\begin{proof}
  Since $(\ze^{-\pi})^\pi = \ze \in T_0 \subseteq T(X^\pi) = T(X)^\pi$, we have
  $\ze^{-\pi} \in T(X) \subseteq T_P + \psi$. Hence, $\psi \to \ze^ {-\pi} \in
  T_P$, which implies $\olbbP \psi_\Om \cap (\ze^{-\pi})_\Om^c = 0$. Therefore,
  $\psi_\Om \subseteq \ze^{-\pi}_\Om$, $\bbP$-a.s. In particular, if
  $\ze_\Om^{-\pi} \in \ol \Si$, then since $\olbbP \psi_\Om > 0$, this gives
  $\olbbP \ze^{-\pi}_\Om > 0$. Similarly, since $\ze \in T_0$, we have $\ze^\pi
  \in T(X)^\pi = T(X^\pi) \subseteq T_P + \psi'$. Hence, $\psi' \to \ze^\pi \in
  T_P$, so that $\psi'_\Om \subseteq \ze^\pi_\Om$, $\bbP$-a.s.~and $\ze_\Om^\pi
  \in \ol \Si$ implies $\olbbP \ze^\pi_\Om > 0$.
\end{proof}

Let $\sP = (\Om, \Si, \bbP)$ be a complete $\cL$-model. Let $T_0 \subseteq \Th
\sP$ and define $P = \bTh \sP \dhl_{[T_0, \Th \sP]}$. Let $\pi$ be a signature
permutation. Suppose $P(\ph \mid X) = p$ and $X^\pi \in \ante P$. Let $\ze \in
T_0$ and assume $B = \ze^\pi_\Om \in \Si$. By Lemma \ref{L:mapping-T_0}, we have
$\bbP B > 0$. We may therefore define the probability measures $\bbP_B$ on
$(\Om, \Si)$ by $\bbP_B C = \bbP C \cap B / \bbP B$. Let $\sP_B = (\Om, \Si,
\bbP_B)$. Note that $\sP_B \vDash \Th \sP$. Similarly define $\sP_{B'}$, where
we assume $B' = \ze^{-\pi}_\Om \in \ol \Si$.

\begin{prop}\label{P:mapping-T_0}
  With the notation given above, if $\sP_B \simeq \sP_{B'}^\pi$, then $P
  (\ph^\pi \mid X^\pi) = p$.
\end{prop}

\begin{proof}
  Assume $\sP_B \simeq \sP_{B'}^\pi$. Write $X \equiv T + \psi$ and $X^\pi
  \equiv T' + \psi'$, where $T, T' \in [T_0, \Th P]$ and $\psi, \psi' \in
  \cL^0$. Since $P(\ph \mid X) = p$, we have $\bbP \ph_\Om \cap \psi_\Om / \bbP
  \psi_\Om = p$. But $\psi_\Om \subseteq B'$ a.s., by Lemma \ref{L:mapping-T_0},
  so that
  \[
    p = \frac{\bbP \ph_\Om \cap \psi_\Om \cap B'}{\bbP \psi_\Om \cap B'}
      = \frac{\bbP_{B'} \ph_\Om \cap \psi_\Om}{\bbP_{B'} \psi_\Om},
  \]
  which implies $\sP_{B'} \vDash (X, \ph, p)$. By Theorems \ref{T:induc-invar}
  and \ref{T:ind-iso-thm}, it follows that $\sP_B \vDash (X^\pi, \ph^\pi, p)$.
  As above, since $\psi'_\Om \subseteq B$, this gives
  \[
    p = \frac{\bbP_B \ph^\pi_\Om \cap \psi'_\Om}{\bbP_B \psi'_\Om}
      = \frac{\bbP \ph^\pi_\Om \cap \psi'_\Om \cap B}{\bbP \psi'_\Om \cap B}
      = \frac{\bbP \ph^\pi_\Om \cap \psi'_\Om}{\bbP \psi'_\Om}.
  \]
  Therefore, $\sP \vDash (X^\pi, \ph^\pi, p)$, so that $P(\ph^\pi \mid X^\pi) =
  p$.
\end{proof}

\section{Examples with a single object}\label{S:basic-expls-PoI}

In this section, we give several elementary examples involving the principle of
indifference. In each of these examples, we are primarily concerned with the
properties of a single object.

\subsection{Either it's true or it isn't}

The most naive misapplication of the principle of indifference is illustrated by
the following invalid reasoning. Imagine we have a book whose color is unknown
to us. It might be red or might not be red. Since we have no reason to think one
way or the other, we should assign equal probabilities to both cases. Therefore,
the probability the book is red is $1/2$.

Clearly, this cannot be correct, for we could also apply it to the color black,
and then to the color blue. Since probabilities must add up to one, it cannot be
the case that all three colors have probability $1/2$.

To see formally that this argument is invalid, let $L = \{b, R\}$, where $b$ is
a constant symbol that denotes the book, and $R$ is a unary predicate symbol
that denotes the property of being red. Recall the notation, $\fI_{T_0}$, for
the set of inductive theories with root $T_0$. We will assume only the principle
of indifference. That is, we take $T_0 = \Taut$, and define the inductive
condition,
\[
  \cC = \{
    P \in \fI_\Taut
  \mid
    \text{$P(R(b) \mid \Taut)$ exists and $P$ satisfies (R10)}
  \}.
\]

\begin{prop}
  The condition $\cC$ is consistent and $\cC \nvdash (\Taut, R(b), 1/2)$.
\end{prop}

\begin{proof}
  Let $A = \{0\}$ and define $\om_0 = (A, L^{\om_0})$ by $b^{\om_0} = 0$ and
  $R^{\om_0} = \emp$. Define $\om_1$ similarly, but with $R^{\om_1} = \{0\}$.
  Let $\Om = \{\om_0, \om_1\}$ and $\Si = \fP \Om$. Fix $p \in (0, 1)$ and
  define $\sP = (\Om, \Si, \bbP)$, where $\bbP \{\om_0\} = 1 - p$ and $\bbP 
  \{\om_1\} = p$. Let $P = \bTh \sP \in \fI_\Taut$. Note that $R(b)_\Om = 
  \{\om_1\}$, so that $\bbP R(b)_\Om = p$. Thus, $P(R(b) \mid \Taut) = p$.

  We will show that $P$ satisfies (R10). Suppose that $P(\ph \mid X) = p$ and
  $X^\pi \in \ante P$. In this example, the only signature permutation is the
  identity. Hence, $P(\ph^\pi \mid X^\pi) = p$, and so $P$ satisfies (R10). This
  shows that $P \in \cC$ and therefore $\cC$ is consistent. However, since
  $P(R(b) \mid \Taut) = p$ and $p$ was arbitrary, it follows that $(\Taut, R(b),
  1/2) \notin \bfP_\cC$.
\end{proof}

\subsection{A single coin flip}\label{S:one-coin}

Let us return to the example in Section \ref{S:rel-rand}. We flip a coin with
two sides, and assume only that the sides are distinct and that the coin will
land on one of them. We also assume the principle of indifference.

Let $L = \{c, \s_0, \s_1\}$, where $c$, $\s_0$, and $\s_1$ are constant symbols.
We think of $\s_1$ and $\s_0$ as denoting the heads and tails sides of the coin,
respectively, and $c$ as denoting the result of our toss. Let $T_0$ be generated
by the sentences
\begin{align*}
  \ph_1 &: \s_0 \nbeq \s_1\\
  \ph_2 &: c \beq \s_0 \vee c \beq \s_1
\end{align*}
Let $\cC = \{P \in \fI_{T_0} \mid \text{$P(c = \s_1 \mid T_0)$ exists and $P$
satisfies (R10)}\}$.

\begin{prop}\label{P:one-coin}
  The condition $\cC$ is consistent and $\bfP_\cC(c \beq \s_1 \mid T_0) = 1/2$.
\end{prop}

\begin{proof}
  Let $A = \{0, 1\}$ and define $\om_0 = (A, L^{\om_0})$ by $\s_i^{\om_0} = i$
  and $c^{\om_0} = 0$. Define $\om_1$ similarly, but with $c^{\om_1} = 1$. Let
  $\Om = \{\om_0, \om_1\}$, $\Si = \fP \Om$, and define $\sP = (\Om, \Si,
  \bbP)$, where $\bbP \{\om_0\} = \bbP \{\om_1\} = 1/2$. Since $\sP \vDash T_0$,
  we may define $P = {\bTh \sP \dhl_{[T_0, \Th \sP]}} \in \fI_{T_0}$. Note that
  $P(c = \s_1 \mid T_0) = 1/2$.

  We will show that $P$ satisfies (R10). Suppose that $P(\ph \mid X) = p$ and
  $X^\pi \in \ante P$. First assume $c^\pi = \s_0$. Then $\s_i^\pi = c$ for some
  $i \in \{0, 1\}$. Let $i' = 1 - i$, so that $\s_{i'}^\pi = \s_1$. We will
  apply Proposition \ref{P:mapping-T_0}. Let $\ze = (\s_0 \nbeq \s_1)$. Then
  \begin{align*}
    B &= \ze^\pi_\Om = (c \nbeq \s_1)_\Om = \{\om_0\}, \text{ and}\\
    B' &= \ze^{-\pi}_\Om = (c \nbeq \s_{i'})_\Om = \{\om_i\}.
  \end{align*}
  Define $h: \sP_B \to \sP_{B'}^\pi$ by $\om_0 \mapsto \om_i^\pi$ and $\om_1
  \mapsto \om_{i'}^\pi$. Since $\bbP_B \{\om_0\} = 1$ and $\bbP_{B'}^\pi 
  \{\om_i^\pi\} = \bbP_{B'} \{\om_i\} = 1$, the function $h$ induces a measure
  space isomorphism.

  Let $g: A \to A$ be the bijection defined by $g(0) = i$ and $g(1) = i'$. Note
  that if $i = 0$, then $\pi = (c \; \s_0)$, and if $i = 1$, then $\pi = (c \;
  \s_0 \; \s_1)$. In either case, $g \circ \om_0 = \om_i \circ \pi^{-1} =
  \om_i^\pi$. Hence, $\om_i^\pi$ is the isomorphic image of $\om_0$ under $g$.
  Since $\bbP_B \{\om_0\} = 1$, we have $\om \simeq h \om$, $\bbP_B$-a.s., so
  that $h$ is a model isomorphism. Proposition \ref{P:mapping-T_0} therefore
  gives $P(\ph^\pi \mid X^\pi) = p$. The case $c^\pi = \s_1$ is similar.

  Now assume $c^\pi = c$. If $\pi$ is the identity, then $P(\ph \mid X) =
  P(\ph^\pi \mid X^\pi)$. Assume $\pi = (\s_0 \; \s_1)$. Define $h: \sP \to
  \sP^\pi$ by $\om_i \mapsto \om_{i'}^\pi$. Since $\bbP \{\om_i\} = 1/2$ and
  $\bbP^\pi \{\om_{i'}^\pi\} = \bbP \{\om_{i'}\} = 1/2$, the function $h$
  induces a measure space isomorphism. As above, if $g: A \to A$ is given by
  $g(i) = i'$, then $\om_{i'}^\pi = \om_{i'} \circ \pi^{-1} = g \circ \om_i$.
  Therefore, $\om \simeq h \om$, $\bbP$-a.s., so that $h$ is a model
  isomorphism. By Theorems \ref{T:induc-invar} and \ref{T:ind-iso-thm}, this
  gives $P(\ph^\pi \mid X^\pi) = p$. Altogether, this shows $P$ satisfies the
  principle of indifference, and so $P \in \cC$. Therefore, $\cC$ is consistent.

  For the final claim, it suffices to show that $P(c \beq \s_1 \mid T_0) = 1/2$
  whenever $P \in \cC$. Let $P \in \cC$ be given. Then $P$ is an inductive
  theory with root $T_0$, $P(c \beq \s_1 \mid T_0) = p$ for some $p$, and $P$
  satisfies the principle of indifference. Let $\pi = (\s_0 \; \s_1)$, and note
  that $T_0^\pi = T_0$. Hence, by the principle of indifference, we have $P(c
  \beq \s_0 \mid T_0) = p$. But $T_0 \vdash \neg (c \beq \s_0 \wedge c \beq
  \s_1)$, so that the addition rule gives $P(c \beq \s_0 \vee c \beq \s_1 \mid
  T_0) = 2p$. On the other hand, $T_0 \vdash c \beq \s_0 \vee c \beq \s_1$, so
  that the rule of logical implication implies $P(c \beq \s_0 \vee c \beq \s_1
  \mid T_0) = 1$. By the rule of logical equivalence, we must have $2p = 1$, or
  $p = 1/2$. Thus, $P(c \beq \s_1 \mid T_0) = 1/2$.
\end{proof}

\subsection{A single trial}\label{S:one-trial}

Here we consider a single ``trial'' that can result in either success or
failure. Intuitively, this example is the same as the single coin flip in
Section \ref{S:one-coin}. With the coin flip, the possible results (heads or
tails) were represented by objects. In this example, the possible results 
(success or failure) will be represented by predicates. For this reason, the
proof of consistency in this example will be shorter.

Let $L = \{\t, S, F\}$, where $\t$ is a constant symbol and $S$ and $F$ are
unary relation symbols. We think of $\t$ as denoting the trial, and $S$ and $F$
as denoting the properties of success and failure, respectively.

Let $T_0$ be generated by the sentence
\[
  \ph : \forall x ((Sx \vee Fx) \wedge \neg(Sx \wedge Fx))
\]
Let $\cC = \{P \in \fI_{T_0} \mid \text{$P(S(\t) \mid T_0)$ exists and $P$
satisfies (R10)}\}$.

\begin{prop}
  The condition $\cC$ is consistent and $\bfP_\cC(S(\t) \mid T_0) = 1/2$.
\end{prop}

\begin{proof}
  Let $A = \{0, 1\}$ and define $\om_0 = (A, L^{\om_0})$ by $S^{\om_0} = \{1\}$,
  $F^{\om_0} = \{0\}$, and $\t^{\om_0} = 0$. Define $\om_1$ similarly, but with
  $\t^{\om_1} = 1$. Let $\Om = \{\om_0, \om_1\}$, $\Si = \fP \Om$, and define
  $\sP = (\Om, \Si, \bbP)$, where $\bbP \{\om_0\} = \bbP \{\om_1\} = 1/2$. Since
  $\sP \vDash T_0$, we may define $P = \bTh \sP \dhl_{[T_0, \Th \sP]} \in
  \fI_{T_0}$. Note that $P(S(\t) \mid T_0) = 1/2$.

  We will show that $P$ satisfies (R10). Suppose that $P(\ph \mid X) = p$ and
  $X^\pi \in \ante P$. If $\pi$ is the identity, then $P(\ph \mid X) = P(\ph^\pi
  \mid X^\pi)$. If $\pi$ is not the identity, then $\pi = (S \; F)$. Define $h:
  \sP \to \sP^\pi$ by $\om_i \mapsto \om_{i'}^\pi$, where $i' = 1 - i$. Since
  $\bbP \{\om_i\} = 1/2$ and $\bbP^\pi \{\om_{i'}^\pi\} = \bbP \{\om_{i'}\} =
  1/2$, the function $h$ induces a measure space isomorphism. If $g: A \to A$ is
  given by $g(i) = i'$, then $\om_{i'}^\pi = \om_{i'} \circ \pi^{-1} = g \circ
  \om_i$, so that $\om_i \simeq \om_{i'}^\pi$. Therefore, $\om \simeq h \om$,
  $\bbP$-a.s., and $h$ is a model isomorphism. By Theorems \ref{T:induc-invar}
  and \ref{T:ind-iso-thm}, this gives $P(\ph^\pi \mid X^\pi) = p$. Altogether,
  this shows $P$ satisfies the principle of indifference, and so $P \in \cC$.
  Therefore, $\cC$ is consistent. The proof that $\bfP_\cC(S(\t) \mid T_0) =
  1/2$ follows as in the proof of Proposition \ref{P:one-coin}.
\end{proof}

\subsection{Success is good}\label{S:success-good}

In this example, we again consider a single trial that can result in either
success or failure. This time, however, we include the qualitative information
that success is ``good.'' This produces an asymmetry between success and
failure. We will see that, because of this asymmetry, we can no longer conclude
that the probability of success is $1/2$.

Let $L = \{\t, S, F, G\}$, where $\t$ is a constant symbol, and $S$, $F$, and
$G$ are unary relation symbol. We think of $\t$ as denoting the trial, $S$ and
$F$ as denoting the properties of success and failure, and $G$ as denoting the
property of ``goodness.''

Let $T_0$ be generated by the sentences
\begin{align*}
  \ph_1 &: \forall x ((Sx \vee Fx) \wedge \neg(Sx \wedge Fx))\\
  \ph_2 &: \forall x (Sx \to Gx)
\end{align*}
Let $\cC = \{P \in \fI_{T_0} \mid \text{$P(S(\t) \mid T_0)$ exists and $P$
satisfies (R10)}\}$.

\begin{prop}
  The condition $\cC$ is consistent and $\cC \nvdash (T_0, S(\t), 1/2)$.
\end{prop}

\begin{proof}
  Let $A = \{0, 1\}$ and define $\om_0 = (A, L^{\om_0})$ by $S^{\om_0} = G^
  {\om_0} = \{1\}$, $F^{\om_0} = \{0\}$, and $\t^{\om_0} = 0$. Define $\om_1$
  similarly, but with $\t^{\om_1} = 1$. Let $\Om = \{\om_0, \om_1\}$ and $\Si =
  \fP \Om$. Fix $p \in (0, 1)$ and define $\sP = (\Om, \Si, \bbP)$, where $\bbP
  \{\om_0\} = 1 - p$ and $\bbP \{\om_1\} = p$. Since $\sP \vDash T_0$, we may
  define $P = \bTh \sP \dhl_{[T_0, \Th \sP]} \in \fI_{T_0}$. Note that $P(S(\t)
  \mid T_0) = p$.

  We will show that $P$ satisfies (R10). Suppose that $P(\ph \mid X) = p$ and
  $X^\pi \in \ante P$. We may assume that $\pi$ is not the identity permutation.
  Since there is only one constant symbol, we have $\t^\pi = \t$. First assume
  $G^\pi = F$. Let $\ze = \forall x (S x \to G x)$. Then $\ze^\pi = \forall x
  (S^\pi x \to F x)$. By Lemma \ref{L:mapping-T_0}, we must have $\ze^\pi_\Om
  \ne \emp$, and so $S^\pi \ne S$. Therefore, $S^\pi = G$ and $F^\pi = S$. But
  this implies $\ze^{-\pi} = \forall x (F x \to S x)$, so that $\olbbP
  \ze^{-\pi}_\Om = 0$, contradicting Lemma \ref{L:mapping-T_0}. Hence, $G^\pi
  \ne F$. By reversing the roles of $\pi$ and $\pi^{-1}$ in this argument, we
  may also conclude that $F^\pi \ne G$.

  Now assume $S^\pi = F$, so that $F^\pi = S$ and $G^\pi = G$. Then $\ze^\pi =
  \forall x (F x \to G x)$, which gives $\olbbP \ze^\pi_\Om = 0$, again
  contradicting Lemma \ref{L:mapping-T_0}. Therefore, $S^\pi \ne F$. It now
  follows that $F^\pi = F$, $S^\pi = G$, and $G^\pi = S$. But this implies
  $\om_i^\pi = \om_i$, so that $\sP^\pi = \sP$. Thus, by Theorem
  \ref{T:induc-invar}, we have $P (\ph^\pi \mid X^\pi) = p$, and $P$ satisfies
  the principle of indifference. This shows that $P \in \cC$, so that $\cC$ is
  consistent. Since $p$ was arbitrary, we have $\cC \nvdash (T_0, S(\t), 1/2)$.
\end{proof}

\subsection{Goodness is independent}

In Section \ref{S:one-trial}, we considered a single trial, which could result
in success or failure. There, we used the principle of indifference to conclude
that the probability of success was $1/2$. In Section \ref{S:success-good}, we
added the assumption that success is good. By doing so, we were no longer able
to use the principle of indifference.

This may feel counterintuitive. To assert that success is good seems quite
natural, and we might not expect this to invalidate our use of the principle of
indifference. One reason we might feel this way is that we may have an ingrained
sense that the property of goodness should not affect success or failure. But if
this is a fact we wish to assume, then we must make it explicit. In this
example, we will do exactly that. We will assume that whether or not $\t$ is a
success is independent of the fact that success is a good outcome.

Let $L = \{\t, S, F, G\}$ as in Section \ref{S:success-good}. Let $T_0$ be
generated by the sentence
\[
  \ph : \forall x ((Sx \vee Fx) \wedge \neg(Sx \wedge Fx))
\]
Let $\cC$ be the set of $P \in \fI_{T_0}$ such that
\begin{enumerate}[(i)]
  \item $P(S(\t) \mid T_0)$ exists,
  \item $P$ satisfies the principle of indifference, and
  \item $S(\t)$ and $\forall x (S x \to G x)$ are independent given $T_0$.
\end{enumerate}

\begin{prop}\label{P:one-trial-indep}
  The condition $\cC$ is consistent and
  \begin{equation}\label{one-trial-indep}
    \bfP_\cC(S(\t) \mid T_0, \forall x (S x \to G x)) = 1/2.
  \end{equation}
\end{prop}

\begin{proof}
  Let $A = \{0, 1\}$ and define $\om_0 = (A, L^{\om_0})$ by $S^{\om_0}= \{1\}$,
  $F^{\om_0} = \{0\}$, $G^{\om_0} = A$, and $\t^{\om_0} = 0$. Define $\om_1$
  similarly, but with $\t^{\om_1} = 1$. Let $\Om = \{\om_0, \om_1\}$, $\Si = \fP
  \Om$, and define $\sP = (\Om, \Si, \bbP)$, where $\bbP \{\om_0\} = \bbP
  \{\om_1\} = 1/2$. Since $\sP \vDash T_0$, we may define $P = \bTh \sP
  \dhl_{[T_0, \Th \sP]} \in \fI_{T_0}$. Note that $P(S(\t) \mid T_0) = 1/2$, so
  that (i) holds. Also, $P(\forall x (Sx \to Gx) \mid T_0) = 1$, so that (iii)
  holds.

  We will show that $P$ satisfies (R10). Suppose that $P(\ph \mid X) = p$ and
  $X^\pi \in \ante P$. We may assume that $\pi$ is not the identity permutation.
  Since there is only one constant symbol, we have $\t^\pi = \t$. First assume
  $S^\pi = G$. Let $\ze = \forall x \neg (S x \wedge F x)$. Then $\ze^\pi =
  \forall x \neg (G x \wedge F^\pi x)$. Since, for all $\om \in \Om$, we have
  $G^\om = A$ and $(F^\pi)^\om \ne \emp$, this gives $\ze^\pi_\Om = \emp$,
  contradicting Lemma \ref{L:mapping-T_0}. Hence, $S^\pi \ne G$. Similarly,
  $F^\pi \ne G$. We must therefore have $\pi = (S \; F)$.

  Define $h: \sP \to \sP^\pi$ by $\om_i \mapsto \om_{i'}^\pi$, where $i' = 1 -
  i$. Since $\bbP \{\om_i\} = 1/2$ and $\bbP^\pi \{\om_{i'}^\pi\} = \bbP
  \{\om_{i'}\} = 1/2$, the function $h$ induces a measure space isomorphism. If
  $g: A \to A$ is given by $g(i) = i'$, then $\om_{i'}^\pi = \om_{i'} \circ
  \pi^{-1} = g \circ \om_i$, so that $\om_i \simeq \om_{i'}^\pi$. Therefore,
  $\om \simeq h \om$, $\bbP$-a.s., and $h$ is a model isomorphism. By Theorems
  \ref{T:induc-invar} and \ref{T:ind-iso-thm}, this gives $P(\ph^\pi \mid X^\pi)
  = p$. Altogether, this shows $P$ satisfies the principle of indifference, and
  so $P \in \cC$. Therefore, $\cC$ is consistent.

  Now let $P \in \cC$ be arbitrary. Let $\pi = (S \; F)$. Then $T_0^\pi = T_0$,
  so that $P(S(\t) \mid T_0) = P(F(\t) \mid T_0)$, which implies $P(S(\t) \mid
  T_0) = 1/2$. Since $P$ was arbitrary, we have $\bfP_\cC(S(\t) \mid T_0) =
  1/2$. Therefore, (iii) and the definition of independence give
  \eqref{one-trial-indep}.
\end{proof}

\subsection{Lowering the root}

We will take one last look at the example of the single trial. As before, let $L
= \{\t, S, F, G\}$ as in Section \ref{S:success-good}. Let $T_0$ be generated by
the sentence
\[
  \ph : \forall x ((Sx \vee Fx) \wedge \neg(Sx \wedge Fx))
\]
Let $\cC$ be the set of $P \in \fI_{T_0}$ such that
\begin{enumerate}[(i)]
  \item $P(S(\t) \mid T_0)$ exists,
  \item $P$ satisfies the principle of indifference, and
  \item $P(\forall x (S x \to G x) \mid T_0) = 1$.
\end{enumerate}
Recall that in the approach of Section \ref{S:success-good}, we could not use
the principle of indifference to determine the probability of success. The only
difference between that approach and the approach in this section is that here,
we have moved the sentence, $\forall x (S x \to G x)$, out of the root and into
$T_\cC$. As we discussed in Section \ref{S:diff-root}, this means we are making
a semantically stronger assumption about the sentence, $\forall x (S x \to G
x)$. As we will see, this stronger assumption is enough to allow us to use the
principle of indifference.

\begin{prop}
  The condition $\cC$ is consistent and $\bfP_\cC(S(\t) \mid T_0) = 1/2$.
\end{prop}

\begin{proof}
  If $P$ is the inductive condition constructed in the first part of the proof
  of Proposition \ref{P:one-trial-indep}, then $P \in \cC$. Hence, $\cC$ is
  consistent. Let $P \in \cC$ be arbitrary, and let $\pi = (S \; F)$. Then
  $T_0^\pi = T_0$, so that $P(S(\t) \mid T_0) = P(F(\t) \mid T_0)$, which
  implies $P(S(\t) \mid T_0) = 1/2$. Since $P$ was arbitrary, we have
  $\bfP_\cC(S(\t) \mid T_0) = 1/2$.
\end{proof}

\section{Examples with multiple objects}\label{S:basic-expls-PoI-2}

In this section, we give more elementary examples involving the principle of
indifference. Here, we will consider examples involving multiple objects.

\subsection{Three balls, two colors}\label{S:three-balls}

Imagine an urn containing three balls, each of which is black or white. The urn
contains at least one white ball and at least one black ball. We will use the
principle of indifference to show that every possible color combination has the
same probability. The fact that there is at least one ball of each color is
critical to this example. As we will see in Section \ref{S:two-balls}, removing
this assumption severely limits the inductive inferences that we can make.

Let $L = \{b_1, b_2, b_3, C_0, C_1\}$, where the $b_k$ are constant symbols
denoting the balls, and $C_0$ and $C_1$ are unary predicate symbols denoting the
colors black and white, respectively. Let $T_0$ be generated by the sentences
\begin{align*}
  \ze_1 &: b_1 \nbeq b_2 \wedge b_1 \nbeq b_3 \wedge b_2 \nbeq b_3\\
  \ze_2 &: \forall x ((C_0 x \vee C_1 x) \wedge \neg (C_0 x \wedge C_1 x))\\
  \ze_3 &: C_0 b_1 \vee C_0 b_2 \vee C_0 b_3\\
  \ze_4 &: C_1 b_1 \vee C_1 b_2 \vee C_1 b_3
\end{align*}
As in Example \ref{Expl:MathSE}, let $d_k(n)$ denote the $k$-th binary digit of
$n$, counting digits from the right. For $n \in \{0, \ldots, 7\}$, let
\[
  \ph_n = C_{d_1(n)} b_1 \wedge C_{d_2(n)} b_2 \wedge C_{d_3(n)} b_3.
\]
For example, $6$ has the binary representation $110$. Reading the digits right
to left, we have $0, 1, 1$. Hence, the sentence $\ph_6$ asserts that ball $b_1$
is black, ball $b_2$ is white, and ball $b_3$ is white. Let $\cC$ be the set of
$P \in \fI_{T_0}$ such that
\begin{enumerate}[(i)]
  \item $P(\ph_n \mid T_0)$ exists for $n \in \{0, \ldots, 7\}$, and
  \item $P$ satisfies the principle of indifference.
\end{enumerate}

\begin{prop}\label{P:three-balls}
  The condition $\cC$ is consistent and $\bfP_\cC(\ph_n \mid T_0) = 1/6$ for $n
  \in \{1, \ldots, 6\}$.
\end{prop}

\begin{proof}
  Let $A = \{1, 2, 3\}$. For $n \in \{0, \ldots, 7\}$, define $\om_n = (A, L^
  {\om_n})$ by $b_k^{\om_n} = k$ and $C_j^{\om_n} = \{k \mid d_k(n) = j\}$. Note
  that $C_0^{\om_n} = (C_1^{\om_n})^c$. Also note that $\om_n \tDash \ph_m$ if
  and only if $m = n$. Let $\Om = \{\om_0, \ldots, \om_7\}$, $\Si = \fP \Om$,
  and define $\sP = (\Om, \Si, \bbP)$, where $\bbP \{\om_0\} = \bbP \{\om_7\} =
  0$, and $\bbP \{\om_n\} = 1/6$ for $1 \le n \le 6$. Since $\sP \vDash T_0$, we
  may define $P = {\bTh \sP \dhl_{[T_0, \Th \sP]}} \in \fI_{T_0}$. Then $P(\ph_0
  \mid T_0) = P(\ph_7 \mid T_0) = 0$, and $P(\ph_n \mid T_0) = 1/6$ for $1 \le n
  \le 6$.

  We will show that $P$ satisfies (R10). Suppose that $P(\ph \mid X) = p$ and
  $X^\pi \in \ante P$. First assume $C_0^\pi = C_0$. Let $g: A \to A$ be the
  bijection that satisfies $b_k^{-\pi} = b_{g k}$. Let $\si$ be the permutation
  of $\{0, \ldots, 7\}$ such that $d_k(\si n) = d_{g k}(n)$. Define $h: \sP \to
  \sP^\pi$ by $\om_n \mapsto \om_{\si n}^\pi$. Since $\si 0 = 0$ and $\si 7 =
  7$, we have $\bbP^\pi \{\om_{\si n}^\pi\} = \bbP \{\om_{\si n}\} = \bbP
  \{\om_n\}$. Hence, $h$ induces a measure space isomorphism. Also, $\om_{\si
  n}^\pi = \om_{\si n} \circ \pi^{-1} = g \circ \om_n$, so that $\om_{\si n}^\pi
  \simeq \om_n$, and $h$ is a model isomorphism. By Theorems \ref{T:induc-invar}
  and \ref{T:ind-iso-thm}, this gives $P(\ph^\pi \mid X^\pi) = p$.

  Now assume $C_0^\pi = C_1$. Define $h: \sP \to \sP^\pi$ by $\om_{7 - n}
  \mapsto \om_{\si n}^\pi$. As above, $h$ induces a measure space isomorphism
  and $\om_{\si n}^\pi = \om_{\si n} \circ \pi^{-1} = g \circ \om_{7 - n}$, so
  that again, $P(\ph^\pi \mid X^\pi) = p$. This shows that $P$ satisfies (R10),
  and therefore, $\cC$ is consistent.

  Now let $P \in \cC$ be arbitrary. Let $\pi = (b_1 \; b_2)$. Then $T_0$ is
  invariant under $\pi$ and $\ph_1^\pi = \ph_2$, so by the principle of
  indifference, $P(\ph_1 \mid T_0) = P(\ph_2 \mid T_0)$. Similarly, using $\pi =
  (b_2 \; b_3)$, we have $P(\ph_2 \mid T_0) = P(\ph_4 \mid T_0)$. Now let $\pi =
  (C_0 \; C_1)$. Then $T_0^\pi = T_0$ and $\ph_n^\pi = \ph_{7 - n}$. Thus,
  $P(\ph_3 \mid T_0) = P (\ph_4 \mid T_0)$, $P(\ph_5 \mid T_0) = P(\ph_2 \mid
  T_0)$, and $P(\ph_6 \mid T_0) = P(\ph_1 \mid T_0)$. It follows that for some
  $p \in [0, 1]$, we have $P(\ph_n \mid T_0) = p$ for all $n \in \{1, \ldots,
  6\}$. But $T_0 \vdash \neg (\ph_0 \vee \ph_7)$, so $P(\ph_0 \mid T_0) = P
  (\ph_7 \mid T_0) = 0$. Therefore, $\sum_{n = 1}^6 P(\ph_n \mid T_0) = 1$,
  which implies $p = 1/6$.
\end{proof}

\subsection{Two balls, two colors}\label{S:two-balls}

Now imagine an urn containing two balls, each of which is black or white. There
are four possible color combinations. Unlike the example in Section
\ref{S:three-balls}, we will not be able to use the principle of indifference to
show that each combination has probability $1/4$. The most we can conclude is
that the probability of two whites is the same as the probability of two blacks.

Let $L = \{b_1, b_2, C_0, C_1\}$, where the $b_k$ are constant symbols denoting
the balls, and $C_0$ and $C_1$ are unary predicate symbols denoting the colors
black and white, respectively. Let $T_0$ be generated by the sentences
\begin{align*}
  \ze_1 &: b_1 \nbeq b_2\\
  \ze_2 &: \forall x ((C_0 x \vee C_1 x) \wedge \neg (C_0 x \wedge C_1 x))
\end{align*}
As in Example \ref{Expl:MathSE}, let $d_k(n)$ denote the $k$-th binary digit of
$n$, counting digits from the right. For $n \in \{0, 1, 2, 3\}$, let
\[
  \ph_n = C_{d_1(n)} b_1 \wedge C_{d_2(n)} b_2.
\]
For example, $2$ has the binary representation $10$. Reading the digits right
to left, we have $0, 1$. Hence, the sentence $\ph_2$ asserts that ball $b_1$
is black, and ball $b_2$ is white. Let $\cC$ be the set of $P \in
\fI_{T_0}$ such that
\begin{enumerate}[(i)]
  \item $P(\ph_n \mid T_0)$ exists for $n \in \{0, 1, 2, 3\}$, and
  \item $P$ satisfies the principle of indifference.
\end{enumerate}

\begin{prop}\label{P:two-balls}
  The condition $\cC$ is consistent. Moreover, for any $P \in \cC$, we have $P
  (\ph_0 \mid T_0) = P(\ph_3 \mid T_0)$ and $P(\ph_1 \mid T_0) = P(\ph_2 \mid
  T_0)$.
\end{prop}

\begin{proof}
  Let $A = \{1, 2\}$. For $n \in \{0, 1, 2, 3\}$, define $\om_n = (A, L^
  {\om_n})$ by $b_k^{\om_n} = k$ and $C_j^{\om_n} = \{k \mid d_k(n) = j\}$. Note
  that $C_0^{\om_n} = (C_1^{\om_n})^c$. Also note that $\om_n \tDash \ph_m$ if
  and only if $m = n$. Let $\Om = \{\om_0, \ldots, \om_7\}$ and $\Si = \fP \Om$.
  Fix $p \in (0, 1)$ and define $\sP = (\Om, \Si, \bbP)$, where $\bbP \{\om_0\}
  = \bbP \{\om_3\} = p/2$, and $\bbP \{\om_1\} = \bbP \{\om_2\} = (1 - p)/2$.
  Since $\sP \vDash T_0$, we may define $P = {\bTh \sP \dhl_{[T_0, \Th \sP]}}
  \in \fI_{T_0}$. Then $P(\ph_0 \mid T_0) = P(\ph_3 \mid T_0) = p/2$, and
  $P(\ph_1 \mid T_0) = P(\ph_2 \mid T_0) = (1 - p)/2$. The proof that $P$
  satisfies (R10) follows as in the proof of Proposition \ref{P:three-balls}.
  Hence, $P \in \cC$ and $\cC$ is consistent.

  Let $P \in \cC$ be arbitrary. If $\pi = (b_1 \; b_2)$, then $T_0$ is invariant
  under $\pi$ and $\ph_1^\pi = \ph_2$. Hence, by the principle of indifference,
  $P(\ph_1 \mid T_0) = P(\ph_2 \mid T_0)$. Similarly, if $\pi = (C_0 \; C_1)$,
  then $T_0^\pi = T_0$ and $\ph_0 = \ph_3$. Therefore, $P(\ph_0 \mid T_0) =
  P(\ph_3 \mid T_0)$.
\end{proof}

\begin{rmk}
  The proof of Proposition \ref{P:two-balls} shows that $\cC$ is indeterminate.
  More specifically, for any $p \in (0, 1)$, there exists $P \in \cC$ such that
  $P(\ph_0 \mid T_0) = P(\ph_3 \mid T_0) = p/2$, and $P(\ph_1 \mid T_0) =
  P(\ph_2 \mid T_0) = (1 - p)/2$. Therefore, $\bfP_\cC(\ph_n \mid T_0)$ does not
  exist for any $n \in \{0, 1, 2, 3\}$.

  This example can be generalized to any finite number of black and white balls.
  Suppose there are $N$ balls and let $\psi_m$ be the sentence which asserts
  that exactly $m$ of them are white. As above, the principle of indifference
  will be unable to tell us the probabilities of $\psi_m$. The most it can say
  is that the probability of $\psi_m$ is the same as the probability of $\psi_{N
  - m}$.
\end{rmk}

\subsection{Random numbers}\label{S:0<1}

In this example, we consider a constant that could equal either $0$ or $1$. We
will not include everything we know about the numbers $0$ and $1$, but we will
include the fact that $0 < 1$. This creates an informational asymmetry, like the
one we encountered in Section \ref{S:success-good}. As such, the principle of
indifference will not provide us with the probability that this constant is
equal to $0$.

Let $L = \{c, \ul 0, \ul 1, <\}$, where $c$, $\ul 0$, and $\ul 1$ are constant
symbols, and $<$ is a binary relation symbols. Let $T_0$ be generated by the
sentences
\begin{align*}
  \ph_1 &: \ul 0 \nbeq \ul 1\\
  \ph_2 &: \ul 0 < \ul 1\\
  \ph_3 &: c \beq \ul 0 \vee c \beq \ul 1
\end{align*}
Let $\cC = \{P \in \fI_{T_0} \mid \text{$P(c \beq \ul 0 \mid T_0)$ exists and
$P$ satisfies (R10)}\}$.

\begin{prop}\label{P:0<1}
  The condition $\cC$ is consistent and $\cC \nvdash (T_0, c \beq \ul 0, 1/2)$.
\end{prop}

\begin{proof}
  Let $A = \{0, 1\}$. Define $\om_0 = (A, L^{\om_0})$ by $\ul 0^{\om_0} = 0$,
  $\ul 1^{\om_0} = 1$, ${<}^{\om_0} = \{(0, 1)\}$, and $c^{\om_0} = 0$. Define
  $\om_1$ similarly, but with $c^{\om_1} = 1$. Let $\Om = \{\om_0, \om_1\}$ and
  $\Si = \fP \Om$. Let $p \in (0, 1)$ and define $\sP = (\Om, \Si, \bbP)$ so
  that $\bbP \{\om_0\} = p$ and $\bbP \{\om_1\} = 1 - p$. Since $\sP \vDash
  T_0$, we may define $P = \bTh \sP \dhl_{[T_0, \Th \sP]}$. Then $P \in
  \fI_{T_0}$ and $P(c \beq \ul 0 \mid T_0) = p$.

  We will show that $P$ satisfies (R10). Suppose that $P(\ph \mid X) = p$ and
  $X^\pi \in \ante P$. Note that ${<}^\pi = {<}$ for all permutations $\pi$.
  First assume that $\ul 0^\pi = \ul 1$. Then $(\ul 0 < \ul 1)^\pi = (\ul 1 <
  \ul 1^\pi)$. But $\om \ntDash (\ul 1 < \s)$ for all $\om \in \Om$ and $\s \in
  L$, so this contradicts Lemma \ref{L:mapping-T_0}. Hence, $\ul 0^\pi \ne \ul
  1$. By reversing the roles of $\pi$ and $\pi^{-1}$ in this argument, we may
  also conclude that $\ul 1^\pi \ne \ul 0$.

  It follows that if $c^\pi = c$, then $\pi$ is the identity. We may therefore
  assume that $c^\pi = \ul 0$ or $c^\pi = \ul 1$. Suppose that $c^\pi = \ul 0$,
  so that $\pi = (c \; \ul 0)$. We will apply Proposition \ref{P:mapping-T_0}
  with $\ze = (\ul 0 < \ul 1)$. In this case, $B = B' = \ze^\pi_\Om = 
  \{\om_0\}$. Hence, the function $h: \sP_B \to \sP_{B'}^\pi$ that maps $\om_i$
  to $\om_i^\pi$ it induces a measure space isomorphism. Moreover, $\om_0 =
  \om_0^\pi$, so $h$ is a model isomorphism. Proposition \ref{P:mapping-T_0}
  therefore implies $P(\ph^\pi \mid X^\pi) = p$. A similar argument gives the
  same result in the case that $c^\pi = \ul 1$. This shows that $P$ satisfies
  (R10), so that $\cC$ is consistent. Since $p$ was arbitrary, we have $\cC
  \nvdash (T_0, c \beq \ul 0, 1/2)$.
\end{proof}

\subsection{Random numbers and definitions}\label{S:0<1-defn}

In this example, we again consider a constant $c$ that could equal $0$ or $1$.
This time, however, we will expand our language by defining $d = 1 - c$. After
making this seemingly harmless addition, we will be able to infer that $c$
equals $0$ with probability $1/2$.

This result may seem counterintuitive. It feels as if introducing a defined
constant should not affect the probabilities of a pre-existing constant. But
this feeling is rooted in the intuition that $c$ is the original constant, and
$d$ is defined in terms of $c$. However, in the expanded language, it is
impossible to tell which of $c$ and $d$ is the original constant. They are
simply two constants related by $d = 1 - c$ and $c = 1 - d$. As such, when we
refer to $c$ in the expanded language, we are equally ignorant about whether $c$
is the original random number, or its inversion. Therefore, the probability that
$c$ equals $0$ in the expanded language should be the average of its
probabilities in the original language, which is $1/2$. We will return to this
idea at the end of the section, after formalizing the example.

Let $L$ and $T_0$ be as in Section \ref{S:0<1}. That is, $L = \{c, \ul 0, \ul 1,
<\}$ and $T_0$ is generated by
\begin{align*}
  \ph_1 &: \ul 0 \nbeq \ul 1\\
  \ph_2 &: \ul 0 < \ul 1\\
  \ph_3 &: c \beq \ul 0 \vee c \beq \ul 1
\end{align*}
Let $d$ be a constant symbol and define
\[
  \de(y) :
    c \beq \ul 0 \wedge y \beq \ul 1 \vee c \beq \ul 1 \wedge y \beq \ul 0
\]
Then $T_0 \vdash \xi$, where $\xi = \exists! y \, \de(y)$, so that $\th = (y
\beq d \tot \de(y))$ is legitimate in $T_0$.

Let $L' = \{c, d, \ul 0, \ul 1, <\}$ and define $T_0' = T_0 + \th \subseteq
(\cL')^0$, so that $T_0'$ is a definitorial extension as in Definition
\ref{D:ded-def-ext}. Let
\[
  \cC = \{
    P' \in \fI_{T_0'}
  \mid
    \text{$P'(c \beq \ul 0 \mid T_0')$ exists and $P'$ satisfies (R10)}\}.
\]

\begin{prop}\label{P:0<1+defn}
  The condition $\cC$ is consistent and $\bfP_\cC(c \beq \ul 0 \mid T_0') =
  1/2$.
\end{prop}

\begin{proof}
  Let $A = \{0, 1\}$. For $i \in \{0, 1\}$, define $\om_i = (A, L^{\om_i})$ by
  $\ul 0^{\om_i} = 0$, $\ul 1^{\om_i} = 1$, ${<}^{\om_i} = \{(0, 1)\}$, and $c^
  {\om_i} = i$. Let $\Om = \{\om_0, \om_1\}$, $\Si = \fP \Om$, and define $\sP =
  (\Om, \Si, \bbP)$ so that $\bbP \{\om_0\} = \bbP \{\om_1\} = 1/2$. Note that
  $\sP$ is the $\cL$-model defined in the proof of Proposition \ref{P:0<1},
  where $p = 1/2$. Also note that $\om \tDash \xi$ for all $\om \in \Om$. Since
  $\sP \vDash T_0$, we may define $P = \bTh \sP \dhl_{[T_0, \Th \sP]}$.

  Let $\om'_i = (A, (L')^{\om'_i})$ be given by $\s^{\om'_i} = \s^{\om_i}$ for
  $\s \in L$ and $d^{\om'_i} = i'$, where $i' = 1 - i$. Let $\Om' = \{\om'_0,
  \om'_1\}$, $\Ga = \fP \Si'$, and define $\sP' = (\Om', \Ga, \bbQ)$ so that
  $\bbQ \{\om'_0\} = \bbQ \{\om'_1\} = 1/2$. Note that $\sP'$ is the
  $\cL'$-model defined above Lemma \ref{L:elim-model}. By Corollary
  \ref{C:ind-elim-thm}, we have $P' = \bTh \sP' \dhl_{[T_0', \Th \sP']}$, where
  $P'$ is the definitorial extension of $P$ given in Theorem
  \ref{T:ind-elim-thm}. Note that $P' \in \fI_ {T_0'}$ and $P'(c \beq \ul 0 \mid
  T_0') = 1/2$.

  We will show that $P'$ satisfies (R10). Suppose that $P'(\ph \mid X) = p$ and
  $X^\pi \in \ante P'$. As in the proof of Proposition \ref{P:0<1}, we have
  $\ul 0^\pi \ne \ul 1$ and $\ul 1^\pi \ne \ul 0$. First assume $\pi = (c \;
  \ul 0)$. As in the proof of Proposition \ref{P:0<1}, we can use Proposition
  \ref{P:mapping-T_0} with $\ze = (\ul 0 < \ul 1)$ to conclude that $P'(\ph^\pi
  \mid X^\pi) = p$. A similar argument covers the cases where $\pi$ is $(c \;
  \ul 1)$, $(d \; \ul 0)$, or $(d \; \ul 1)$. If $\pi = (c \; d)$, then $
  (\om'_i)^\pi = \om'_{i'}$, so that $(\sP')^\pi \simeq \sP'$, which also gives
  $P'(\ph^\pi \mid X^\pi) = p$. This covers the case that $\pi$ is a
  transposition.

  Now suppose $\pi$ is a 3-cycle. Assume $\pi = (\ul 0 \; c \; d)$. Then $B =
  \ze^\pi_\Om = \{\om'_0\}$ and $B' = \ze^{-\pi}_\Om = \{\om'_1\}$. Since $
  (\om'_1)^\pi = \om'_0$, it follows from Proposition \ref{P:mapping-T_0} that
  $P'(\ph^\pi \mid X^\pi) = p$. A similar argument covers the cases where $\pi$
  is $(\ul 0 \; d \; c)$, $(\ul 1 \; c \; d)$, or $(\ul 1 \; d \; c)$.

  Finally, suppose $\pi$ affects every constant symbol. If $\pi = (c \; \ul 0) 
  (d \; \ul 1)$, then $(\om'_i)^\pi = \om'_i$, so that $(\sP')^\pi = \sP$, which
  gives $P'(\ph^\pi \mid X^\pi) = p$. A similar argument covers $\pi = (c \; \ul
  1) (d \; \ul 0)$. The remaining possibility is that $\pi$ is a 4-cycle. Assume
  $\pi = (\ul 0 \; c \; \ul 1 \; d)$. As above, we have $B = \{\om'_0\}$ and
  $B' = \{\om'_1\}$, so that Proposition \ref{P:mapping-T_0} gives $P'(\ph^\pi
  \mid X^\pi) = p$. A similar argument covers the case $\pi = (\ul 0 \; d \; \ul
  1 \; c)$. Altogether, this shows that $P'$ satisfies (R10), so that $\cC$ is
  consistent.

  Now let $P' \in \cC$ be arbitrary. Let $\pi = (c \; d)$. Then $T_0'$ is
  invariant under $\pi$. Hence, $P'(c \beq \ul 0 \mid T_0') = P'(d \beq \ul 0
  \mid T_0')$. But $d \beq \ul 0 \equiv_{T_0'} c \beq \ul 1$. Therefore, $P'(c
  \beq \ul 0 \mid T_0') = P'(c \beq \ul 1 \mid T_0')$, which gives $P'(c \beq
  \ul 0 \mid T_0') = 1/2$.
\end{proof}

This example can be generalized to a constant $c$ that could equal $0$, $1$, or
$2$. That is, let $L = \{c, \ul 0, \ul 1, \ul 2\}$ and let $T_0$ be generated by
\begin{align*}
  \ph_1 &: \ul 0 \nbeq \ul 1 \wedge \ul 0 \nbeq \ul 2 \wedge \ul 1 \nbeq \ul 2\\
  \ph_2 &: \ul 0 < \ul 1 \wedge \ul 0 < \ul 2 \wedge \ul 1 < \ul 2\\
  \ph_3 &: c \beq \ul 0 \vee c \beq \ul 1 \vee c \beq \ul 2
\end{align*}
As in Section \ref{S:0<1}, we cannot use the principle of indifference to
determine $P(c \beq \ul n \mid T_0)$. But we can create a definitorial expansion
$T_0'$ using
\[
  \th_d = (y \beq d \tot {
    c \beq \ul 0 \wedge y \beq \ul 1 \vee
    c \beq \ul 1 \wedge y \beq \ul 0 \vee
    c \beq \ul 2 \wedge y \beq \ul 2
  }).
\]
Informally, $d = f(c)$, where $f$ interchanges $0$ and $1$. As above, we could
then use the principle of indifference to conclude that $P'(c \beq \ul 0 \mid
T_0') = P'(c \beq \ul 1 \mid T_0')$.

It is tempting to think we can iterate this process. That is, suppose we create
a definitorial expansion $T_0''$ of $T_0'$ using
\[
  \th_e = (y \beq e \tot {
    c \beq \ul 0 \wedge y \beq \ul 2 \vee
    c \beq \ul 1 \wedge y \beq \ul 1 \vee
    c \beq \ul 2 \wedge y \beq \ul 0
  }).
\]
Informally, $e = g(c)$, where $g$ interchanges $0$ and $2$. We might now expect
that the principle of indifference gives $P''(c \beq \ul n \mid T_0'') = 1/3$.
But it does not. In fact, even our previous inference is no longer valid. That
is, we can no longer even conclude that $P''(c \beq \ul 0 \mid T_0'') = P''(c
\beq \ul 1 \mid T_0'')$. This is because $T_0''$ is no longer invariant under
$\pi = (c \; d)$. In particular, $T_0'' \vdash d \beq \ul 0 \to e \beq \ul 1$,
but $T_0'' \nvdash c \beq \ul 0 \to e \beq \ul 1$.

\begin{rmk}
  This example shows that definitorial extensions do not preserve the principle
  of indifference. It is possible for $P$ to satisfy (R10), but for its
  definitorial extension $P'$ to not satisfy it. The converse is also possible.
  In other words, Theorem \ref{T:ind-elim-thm} is another theorem that would
  fail if we included (R10) in the definition of an inductive theory.
\end{rmk}

The juxtaposition of Propositions \ref{P:0<1} and \ref{P:0<1+defn} may seem
counterintuitive. On the one hand, in $\cL$, we cannot infer the probability
that $c \beq \ul 0$. On the other hand, by passing to $\cL'$, whose only
difference is that it includes the defined constant $d$, we are suddenly able to
infer that $c \beq \ul 0$ with probability $1/2$. It seems that we must have
added some new information by passing to $\cL'$. But clearly we did not. It is
the nature of a definitorial extension that it adds no new logical information.
The explanation is not that we have added new information. Rather, we have
altered the very meaning of $c$, in the way described in Section 
\ref{S:prim-vs-def}.

It is tempting to think that a constant symbol such as $c$ stands for some
object. From that point of view, it must stand for the same object in both $\cL$
and $\cL'$. And in that case, it makes no sense to say that we altered the
meaning of $c$. But syntactically, $c$ does not stand for anything. Standing for
an object is a semantic notion. What $c$ stands for is relative to the model we
are using, and even then, it can vary from structure to structure within that
model. Syntactically speaking, $c$ is not denoting an object. Rather, it is a
primitive symbol that gains its meaning from the deductive and inductive facts
that use it.

In $\cL'$, we have changed those facts from $T_0$ to $T_0'$. The meaning of $c$
in $T_0'$ is not necessarily the same as in $T_0$. To see this more clearly,
simply rename $c$ and $d$ in $\cL'$ to $c'$ and $d'$. It is then no longer
surprising that $P'(c' \beq \ul 0 \mid T_0') = 1/2$. After all, in $T_0'$, it is
impossible to tell which of $c'$ and $d'$ is the original constant from $\cL$,
and which of them is defined in terms of that original constant. We are
indifferent about those two possibilities. Therefore, $P'(c' \beq \ul 0 \mid
T_0')$ should be the average of $P(c \beq \ul 0 \mid T_0)$ and $P(c \beq \ul 1
\mid T_0)$, which is $1/2$. By analogy with measure-theoretic probability, it is
as if we started with a $\{0, 1\}$-valued random variable $X$, defined $Y = 1 -
X$, and then let $(X', Y')$ be a random permutation of $(X, Y)$. In that case,
if $(X', Y') = (X, Y)$ and $(X', Y') = (Y, X)$ are equally likely, then $P(X' =
0) = 1/2$, regardless of the distribution of $X$.

\section{Indifference and exchangeability}\label{S:indiff-exch}

For the remainder of this chapter, we will look at how the principle of
indifference relates to real inductive theories. For simplicity, and to match
the intuition of measure-theoretic probability models, we will take the approach
presented in Theorem \ref{T:prob-model-iso-ZFC}. Namely, we will operate under
the standing assumption that $\ZFC$ is strictly satisfiable. This assumption
will be in effect for the remainder of this chapter.

In this short section, we show that, in the context of a measure-theoretic
probability model, exchangeability is a special case of the principle of
indifference. This result is given below in Theorem \ref{T:exchangeable}. In
order to prove it, we first establish Lemma \ref{L:fixed-perm}, which we be
useful several times throughout the remainder of this chapter.

\subsection{Permutations of real inductive theories}

Let $P \subseteq \cL^\IS$ be a real inductive theory in $\ZFC$. Let $\pi$ be an
$L$-permutation such that ${\bin^\pi} = {\bin}$. Define $\pi': L \to L$ by
\[
  \s^{\pi'} = \begin{cases}
    \s &\text{if $\s \in L_\ZFC$},\\
    \s^\pi &\text{if $\s \notin L_\ZFC$ and $\s^\pi \notin L_\ZFC$, and}\\
    (\s^\pi)^\pi &\text{if $\s \notin L_\ZFC$ and $\s^\pi \in L_\ZFC$.}
  \end{cases}
\]

\begin{lemma}\label{L:fixed-perm}
  Suppose $X, X^\pi \in \ante P$. Then the function $\pi'$ is an $L$-permutation
  that fixes $L_\ZFC$. Moreover, if $X^{\pi'} \in \ante P$ and $\pi'$
  satisfies (R10), then $\pi$ satisfies (R10).
\end{lemma}

\begin{proof}
  Assume $X, X^\pi \in \ante P$. By construction, the function $\pi'$ fixes
  $L_\ZFC$ and preserves the type and arity of extralogical symbols. The fact
  that $\pi'$ is a bijection is a consequence of the following:
  \begin{equation}\label{fixed-perm}
    \text{
      if $\s \in L_\ZFC$ and $\s^\pi \ne \s$,
      then $\s^{-\pi} \notin L_\ZFC$ and $\s^\pi \notin L_\ZFC$.
    }
  \end{equation}
  To see this, let $\s \in L_\ZFC$ with $\s^\pi \ne \s$. Since ${\bin^\pi} =
  {\bin}$, we have $\s \ne {\bin}$. Hence, $\s$ is an explicitly defined
  constant symbol. Let $\de(y)$ be its defining formula, and let $\de^\rd(y)$ be
  its reduction to $\cL\{{\bin}\}$, so that the only extralogical symbol in
  $\de^\rd(y)$ is $\bin$. Let $\ze = \forall y (y \beq \s \tot \de^\rd(y))$, so
  that $\ZFC \vdash \ze$. By Lemma \ref{L:mapping-T_0}, if $\sP \vDash P$, then
  $\psi_\Om \subseteq \ze^{-\pi}_\Om$ a.s., which implies $\sP \vDash (X,
  \ze^{-\pi}, 1)$. Hence, $P(\ze^{-\pi} \mid X) = 1$. It follows that $P (\ze
  \wedge \ze^{-\pi} \mid X) = 1$. But $\ze^{-\pi} = \forall y (y \beq \s^ {-\pi}
  \tot \de^\rd(y))$, so that $\ZFC \vdash \ze \wedge \ze^{-\pi} \tot \s \beq
  \s^{-\pi}$. Therefore $P(\s \beq \s^{-\pi} \mid X) = 1$. A similar argument
  shows that $P(\ze^\pi \mid X^\pi) = 1$ and $P(\s \beq \s^\pi \mid X^\pi) = 1$.
  Finally, note that if $\s, \s' \in L_\ZFC \setminus \{{\bin}\}$, then $\ZFC
  \vdash \s \nbeq \s'$. Hence, $\s^{-\pi} \notin L_\ZFC$ and $\s^\pi \notin
  L_\ZFC$.

  Now assume $X^{\pi'} \in \ante P$ and $\pi'$ satisfies (R10). Using the facts
  that ${\bin^\pi} = {\bin}$, every $\s \in L_\ZFC \setminus \{{\bin}\}$ is a
  constant symbol, and $P(\s \beq \s^\pi \mid X^\pi) = 1$ for all $\s \in
  L_\ZFC \setminus \{{\bin}\}$, it follows by term induction and formula
  induction that
  \begin{equation}\label{fixed-perm-1}
    P(\th^{\pi'} \tot \th^\pi \mid X^\pi) = 1 \text{ for all $\th \in \cL^0$.}
  \end{equation}
  A similar argument shows that $P(\th^{-\pi'} \tot \th^{-\pi} \mid X) = 1$ for
  all $\th \in \cL^0$. Applying this to $\th^\pi$ gives $P((\th^\pi)^{-\pi'}
  \tot \th \mid X) = 1$. But $\pi'$ satisfies (R10), so this gives
  \begin{equation}\label{fixed-perm-2}
    P(\th^\pi \tot \th^{\pi'} \mid X^{\pi'}) = 1 \text{
      for all $\th \in \cL^0$.
    }
  \end{equation}
  To show that $\pi$ satisfies (R10), suppose that $P(\ph \mid X) = p$. Since
  $\pi'$ satisfies (R10), we have $P(\ph^{\pi'} \mid X^{\pi'}) = p$. Also,
  \eqref{fixed-perm-2} gives $P(\ph^\pi \tot \ph^{\pi'} \mid X^{\pi'}) = 1$.
  Hence, Proposition \ref{P:log-equiv-gen} implies $P(\ph^\pi \mid X^{\pi'}) =
  p$.

  Let $\th \in X^\pi$. Then $(\th^{-\pi})^{\pi'} \in X^{\pi'}$, so that $P((\th^
  {-\pi})^{\pi'} \mid X^{\pi'}) = 1$. As it was for $\ph$ above, this gives $P(
  (\th^{-\pi})^\pi \mid X^{\pi'}) = 1$. But $(\th^{-\pi})^\pi = \th$. Hence, $P
  (\th \mid X^{\pi'}) = 1$ for all $\th \in X^\pi$. By the rule of deductive
  extension, $P(\,\cdot \mid X^{\pi'}, X^\pi) = P(\,\cdot \mid X^{\pi'})$. In
  particular, $P(\ph^\pi \mid X^{\pi'}, X^\pi) = p$. On the other hand, the
  analogous argument using \eqref{fixed-perm-1} shows that $P(\th \mid X^\pi) =
  1$ for all $\th \in X^{\pi'}$. The rule of deductive extension therefore also
  shows that $P(\,\cdot \mid X^{\pi'}, X^\pi) = P(\,\cdot \mid X^\pi)$. Hence,
  $P(\ph^\pi \mid X^\pi) = p$.
\end{proof}

\subsection{Exchangeability}

Let $\ang{X_i \mid i \in I}$ be a collection of real-valued random variables
defined on a probability space, $(S, \Ga, \nu)$. We say that $\ang{X_i \mid i
\in I}$ are \emph{exchangeable} if the distribution of $(X_{i(1)}, \ldots,
X_{i(n)})$ is unchanged by a finite permutation of $I$.
  \index{exchangeable}%
More specifically, let $\si: I \to I$ be a bijection with $\si(i) = i$ for all
but finitely many $i$. Then
\begin{equation}\label{exchangeable}
  \ts{
    \nu \bigcap_{k = 1}^n \{X_{i(k)} \in V_k\} =
    \nu \bigcap_{k = 1}^n \{X_{\si(i(k))} \in V_k\},
  }
\end{equation}
for all choices of $n \in \bN$, $i(k) \in I$, and $V_k \in \cB(\bR)$.

Let $\ang{X_i \mid i \in I}$ be real-valued random variables on $(S, \Ga, \nu)$.
Without loss of generality, we may assume $S = \bR^I$, $\Ga = \bigotimes_{i \in
I} \cB(\bR)$, and $X_i(x) = x_i$. Recall our standing assumption that $\ZFC$ is
strictly satisfiable. Let $\sP = (\Om, \Si, \bbP)$ be the model constructed in
the proof of Theorem \ref{T:prob-model-iso-ZFC}, and let $P = \bTh \sP \dhl_{ 
[\ZFC, \Th \sP]}$.

\begin{thm}\label{T:exchangeable}
  With notation as above, the inductive theory $P$ satisfies the principle of
  indifference if and only if $\ang{X_i \mid i \in I}$ are exchangeable.
\end{thm}

\begin{proof}
  Assume $P$ satisfies the principle of indifference. Let $\si: I \to I$ be a
  bijection with $\si(i) = i$ for all but finitely many $i$. Let $\pi$ be the
  signature permutation that fixes everything in $L_\ZFC$, and maps $\ul X_i$ to
  $\ul X_{\si(i)}$. Then $\psi^\pi = \psi$ for all $\psi \in \cL_\ZFC$. In
  particular, $\ZFC^\pi = \ZFC$. Let $\ph = \bigwedge_{k = 1}^n \ul X_{i(k)}
  \bin \ul{V_k}$. Then $\ph^\pi = \bigwedge_{k = 1}^n \ul X_{\si(i (k))} \bin
  \ul{V_k}$. By the principle of indifference,
  \[
    \ts{
      P(\bigwedge_{k = 1}^n \ul X_{i(k)} \bin \ul{V_k} \mid \ZFC)
      = P(\bigwedge_{k = 1}^n \ul X_{\si(i(k))} \bin \ul{V_k} \mid \ZFC).
    }
  \]
  Therefore, by \eqref{prob-model-iso-ZFC}, we have \eqref{exchangeable}.

  Now assume $X$ is exchangeable. Let $\pi$ be an $L$-permutation and suppose
  $P(\ph \mid Y) = p$. Since $\bin$ is the only binary operation symbol in $L$,
  we have ${\bin^\pi} = {\bin}$. By Lemma \ref{L:fixed-perm}, it suffices to
  assume $\pi$ fixes $L_\ZFC$, and to show that both $Y^\pi \in \ante P$ and
  $P(\ph^\pi \mid Y^\pi) = p$.

  Since $\pi$ fixes $L_\ZFC$, it only affects $C = \{\ul X_i \mid i \in I\}$.
  Let $\si: I \to I$ be the bijection defined by $\ul X_i^\pi = \ul X_{\si(i)}$
  and define the bijection $g: S \to S$ by $(gx)_{\si(i)} = x_i$. Note that both
  $g$ and $g^{-1}$ are measurable. We claim that $\nu = \nu \circ g^{-1}$. To
  verify this, it suffices to check that $\opnu B = \opnu g^{-1} B$, when $B$ is
  a cylinder set of the form
  \[
    \ts{B = \bigcap_{k = 1}^n \{x \in S \mid x_{\si(i(k))} \in V_k\}.}
  \]
  In this case, we have
  \begin{align*}
    g^{-1} B &= \ts{
      \bigcap_{k = 1}^n \{x \in S \mid (g x)_{\si(i(k))} \in V_k\}
    }\\
    &= \ts{\bigcap_{k = 1}^n \{x \in S \mid x_{i(k)} \in V_k\}},
  \end{align*}
  so that $\opnu B = \opnu g^{-1} B$ follows from \eqref{exchangeable}. Hence,
  $\nu = \nu \circ g^{-1}$, so that $g$ is a pointwise isomorphism from $(S,
  \Ga, \nu)$ to itself.

  Let $h: S \to \Om$ be the function in the proof of Theorem
  \ref{T:prob-model-iso-ZFC} that maps $x \in S$ to $\om^x \in \Om$. Let
  $\sP^\pi = (\Om^\pi, \Si^\pi, \bbQ)$ and let $h_\pi: \Om \to \Om^\pi$ be the
  function that maps $\om$ to $\om^\pi$. Then $\sP^\pi$ is the measure space
  image of $ (S, \Ga, \nu)$ under $h_\pi \circ h$. But $h_\pi \circ h = h \circ
  g$ and $g$ is a pointwise isomorphism, so $\sP^\pi$ is the measure space image
  of $(S, \Ga, \nu)$ under $h$. By the definition of $\sP$, this means $\sP^\pi
  = \sP$.

  Now, since $P(\ph \mid Y) = p$ exists, we have $\sP \vDash (Y, \ph, p)$. By
  Theorem \ref{T:induc-invar} and $\sP = \sP^\pi$, it follows that $\sP \vDash
  (Y^\pi, \ph^\pi, p)$. Therefore, $Y^\pi \in \ante P$ $P(\ph^\pi \mid Y^\pi) =
  p$.
\end{proof}

\section{Examples on an interval}\label{S:indiff-interval}

In this section, we present examples of the principle of indifference that
involve an interval on the real line.

\subsection{The interval \texorpdfstring{$[0, 1]$}{[0, 1]}}\label{S:[0,1]}

In our first example, we have a real number $c$, about which we know only that
$c \in [0, 1]$. We then ask what the principle of indifference has to say about
the distribution of $c$. At first glance, we might expect the principle to
assign $c$ a uniform distribution, based on the fact that we are somehow
``equally ignorant'' about where $c$ lies in the interval $[0, 1]$. A little
further thought, however, quickly reveals that this cannot be the case. The
principle of indifference requires an informational symmetry, encoded in the
permutation $\pi$. In the coin flip of Section \ref{S:one-coin}, for example, we
obtained the probability $1/2$ by interchanging the symbols for heads and tails.
We could do this because our background information, $T_0$, was symmetric with
respect to this interchange.

In this case, however, our background information, $T_0$, will contain $\ZFC$,
and the individual numbers in $[0, 1]$ are most certainly not interchangeable
with respect to $\ZFC$. For instance, in $\ZFC$, we know that $\ul 0$ is the
additive identity and $\ul 1$ is the multiplicative identity. That sentence is
no longer true if we interchange $\ul 0$ and $\ul 1$. To say that $c \beq \ul 0$
is a qualitatively different assertion than to say that $c \beq \ul 1$.

In this way, the present example is less like the coin flip of Section
\ref{S:one-coin} and more like the random number (either $\ul 0$ or $\ul 1$) in
Section \ref{S:0<1}. In that example, the principle of indifference did not
narrow down the probabilities at all. Every possible value for the probability
of $\ul 0$ was consistent with the principle of indifference. Similarly, here,
every possible distribution on $c$ will be consistent with it.

To state this result, let $L_\ZFC$ be the logical signature of $\ZFC$, given by
\eqref{ZFC-signature}. Let $c$ be a constant symbol not in $L_\ZFC$, let $L =
L_\ZFC \cup \{c\}$, and define $T_0 = \ZFC + c \bin \ul{[0, 1]}$. Let $\cC$ be
the inductive condition consisting of all inductive theories $P \subseteq
\cL^\IS$ with root $T_0$ such that
\begin{enumerate}[(i)]
  \item $c$ is Borel given $T_0$, and
  \item $P$ satisfies the principle of indifference.
\end{enumerate}
Let $\nu$ be an arbitrary Borel probability measure on $\bR$. Fix an
$L_\ZFC$-structure $\om_0$ such that $\om_0 \tDash \ZFC$. For each $r
\in \bR$, define $\om = \om^r$ to be the $L$-expansion of $\om_0$ given by
$c^\om = \ul r^{\om_0}$. Let $\Om = \{\om^r \mid r \in \bR\}$ and let $h: \bR
\to \Om$ denote the map $r \mapsto \om^r$. Let $\sP_\nu = (\Om, \Si, \bbP)$ be
the measure space image of $(\bR, \cB(\bR), \nu)$ under the function $h$.

If $\opnu [0, 1] = 1$, then $\sP_\nu \vDash T_0$, so may define the complete
inductive theory $P_\nu = \bTh \sP_\nu \dhl_{[T_0, \Th \sP_\nu]}$. By Theorem
\ref{T:prob-model-iso-ZFC}, we have $P(c \bin \ul V \mid T_0) = \opnu V$ for
all $V \in \cB(\bR)$. In particular, $c$ is Borel given $T_0$ under $P_\nu$.

\begin{prop}
  For any such $\nu$, we have $P_\nu \in \cC$.
\end{prop}

\begin{proof}
  Suppose $P_\nu(\ph \mid X)$ exists and let $\pi$ be an $L$-permutation. As in
  the proof of Theorem \ref{T:exchangeable}, we may assume the permutation $\pi$
  fixes all of $L_\ZFC$. Since there is only one symbol in $L \setminus
  L_\ZFC$, the permutation $\pi$ is the identity. Therefore, it is trivially
  the case that $P_\nu(\ph^\pi \mid X^\pi) = P_\nu(\ph \mid X)$.
\end{proof}

\subsection{A point on a rod}\label{S:pt-on-rod}

\subsubsection{Introduction}

In this example, we consider a rigid rod of some unspecified length, and we let
$c$ be a point that lies somewhere along the length of the rod. We then ask
what, if anything, the principle of indifference has to say about the
distribution of the point $c$.

A common approach to a problem like this would be to replace the rod with the
interval $[0, 1]$. If we do that, then we have our answer, according to the
preceding example: the principle of indifference says nothing. But we should ask
ourselves why we feel justified in replacing the rod with $[0, 1]$. When we do
so, we are introducing qualitative differences between points on the rod that
were not there originally. In other words, we are making assumptions that are
not indicated by the statement of the problem.

Instead of \emph{replacing} the rod with $[0, 1]$, we should \emph{represent} it
by $[0, 1]$. For instance, we could replace the rod with a smooth manifold with
boundary, and give it a Riemannian metric that indicates it has no curvature. If
we did that, then we would see informational symmetries that are not there when
we simply replace the rod with $[0, 1]$.

Taking this approach would require us to formulate, in $\ZFC$, a number of new
and complicated definitions. While this is certainly doable, we will avoid these
complications by simply replacing our rod with a subset $M$ of the real line,
and assuming that $M$ can be parameterized by some affine linear function on
$[0, 1]$. We then let $c$ be an element of $M$.

To talk about the distribution of $c$, we need to have names for subsets of $M$.
For this, we will let $c_0$ and $c_1$ be the endpoints of $M$, arbitrarily
labeled, and name the Borel subsets of $M$ relative to $c_0$. That is, if $B$ is
a Borel subset of $[0, 1]$, then $B_*$ will be the image of $B$ when $[0, 1]$ is
mapped to $M$ in a way that sends $0$ to $c_0$ and $1$ to $c_1$.

\subsubsection{Notation in \texorpdfstring{$\ZFC$}{ZFC}}

To make this precise, we first establish some new shorthand in $\cL_\ZFC$, the
language of $\ZFC$. Let $\de(u, v, y) \in \cL_\ZFC$ be given by
\begin{multline*}
  \de(u, v, y) = (u \nbin \ul \bR \vee v \nbin \ul \bR) \wedge y \beq \ul \emp\\
  \vee {
    u \bin \ul \bR
    \wedge v \bin \ul \bR
    \wedge y \bin \ul \bR^{\ul \bR}
    \wedge (\forall x \bin \ul \bR)(y(x) \beq u \bdot x + v)
  }.
\end{multline*}
Then $\ZFC \vdash \forall u v \exists! y \, \de(u, v, y)$. Hence, we could
explicitly define the function symbol $F$ by $y \beq Fuv \tot \de (u, v, y)$,
and then let $f_{uv}$ be shorthand for the term $Fuv$. We do not, in fact, add
the symbol $F$ to our extralogical signature, but instead regard both $F$ and
$f_{uv}$ as shorthand. We also adopt the shorthand, $\dom(f)$ and $f^\img (z)$,
given by
\begin{align*}
  \dom(f) &= \{x \bin \ul \bR \mid (\exists y \bin \ul \bR) (x, y) \bin f\},\\
  f^\img(z) &= \{f(x) \mid x \bin z \cap \dom(f)\}.
\end{align*}

\subsubsection{Extralogical symbols and assumptions}

To talk about our rod, we add new extralogical symbols to $L_\ZFC$, the
signature of $\ZFC$. Let $\cE = \{B \in \cB([0, 1]) \mid \emp \subset B \subset
[0, 1]\}$ and let
\[
  C = \{M, c, c_0, c_1\} \cup \{B_* \mid B \in \cE\}
\]
be a set of distinct constant symbols not in $L_\ZFC$. Let $L = L_\ZFC C$.

In the language $\cL$, we want to build $T_0$, the assumptions we will be making
in our setup of the problem. Using $\ph_M(u, v) = (u \nbeq \ul 0 \wedge M \beq
f_{uv}^\img(\ul{[0, 1]}))$, define the following sentences in $\cL$:
\begin{align*}
  \ph_1 &: (\exists uv \bin \ul \bR) \, \ph_M(u, v)\\
  \ph_2 &: c \bin M\\
  \ph_3 &: (\exists uv \bin \ul \bR) (
    \ph_M(u, v) \wedge c_0 \beq v \wedge c_1 \beq u + v
  )\\
\intertext{Also, for any $B \in \cE$, define}
  \ph_B(x) &: (\exists uv \bin \ul \bR) (
    \ph_M(u, v) \wedge c_0 \beq v \wedge x \beq f_{uv}^\img(\ul B)
  )
\end{align*}
Then let $T_0 = \ZFC + \{\ph_1, \ph_2, \ph_3\} \cup \{\ph_B(B_*) \mid B \in
\cE\}$.

The sentence $\ph_1$ says that $M$ is a ``rod.'' That is, $M$ can be
parameterized by the interval $[0, 1]$. Another way to think of this is that $M$
is a closed interval, but its location, length, and orientation are left
unspecified.

The sentence $\ph_2$ says that $c$ is an element of $M$. This is the only
information about $c$ contained in $T_0$.

The sentence $\ph_3$ says that $c_0$ and $c_1$ are the two distinct endpoints of
$M$. But it does not specify which is the left endpoint and which is the right.
Thinking of $M$ as a rod, it does not even make sense to ask which end is left
and which is right, for the orientation of the rod is arbitrary. It therefore
makes sense that we would not include such information in $T_0$.

The sentence $\ph_B(B_*)$ defines the symbol $B_*$ so that $B_*$ is the
representative of $B$ in $M$, relative to $c_0$. For instance, if $B = [0,
1/2]$, then $B_*$ is the subset of $M$ that extends from $c_0$ to the midpoint
of $M$.

\subsubsection{Inductive hypotheses and conclusion}

Now let $\cC$ be the inductive condition consisting of all inductive theories $P
\subseteq \cL^\IS$ with root $T_0$ such that
\begin{enumerate}[(i)]
  \item $P(c \bin B_* \mid T_0)$ exists for all $B \in \cE$, and
  \item $P$ satisfies the principle of indifference.
\end{enumerate}

\begin{thm}\label{T:stick}
  The condition $\cC$ is consistent. Moreover, if $P \in \cC$, then
  \begin{equation}\label{stick}
    P(c \bin B_* \mid T_0) = P(c \bin (\rho B)_* \mid T_0),
  \end{equation}
  for all $B \in \cB([0, 1])$, where $\rho: [0, 1] \to [0, 1]$ is given by
  $\rho r = 1 - r$.
\end{thm}

The proof of Theorem \ref{T:stick} will be given at the end of this subsection.
Proving consistency is what will require the most work. This consistency proof
will show, in fact, that \eqref{stick} is the only restriction on the
distribution of $c$. In particular, the principle of indifference does not
require the distribution to be uniform. This is because the only points on the
rod that are interchangeable with respect to $T_0$ are points that are
equidistant from the ends. To say that $c$ lies twice as far from one endpoint
than the other is a qualitatively different statement than saying that $c$ lies
at the midpoint. These two locations are not symmetric with respect to
everything we know about the rod.

\subsubsection{A model for the rod}

Our proof of consistency will utilize an $\cL$-model, $\sP = (\Om, \Si, \bbP)$,
which we construct as follows. Let $\nu_0$ be a probability measure on $(\bR,
\cB(\bR))$ such that $\opnu_0 [0, 1] = 1$ and $\nu_0$ is continuous. That is,
$\opnu_0 \{r\} = 0$ for all $r \in \bR$. Let $S = \bR \times \bR \times \{0,
1\}$ and $\Ga = \cB(\bR) \otimes \cB (\bR) \otimes \fP \{0, 1\}$, and define the
probability measure $\nu$ on $(S, \Ga)$ by
\[
  \opnu B \times B' \times \{n\} = (1/2)(\opnu B)(\opnu B').
\]
Let $\om_0$ be an $L_\ZFC$-structure such that $\om_0 \tDash \ZFC$. If $x = (r,
t, n) \in S$, then let $\om = \om^x$ be the $L$-expansion of $\om_0$ defined by
$M^\om = \ul{[t, t + 1]}^{\om_0}$, $c^\om = \ul{t + r}^{\om_0}$,
\[
  c_0^\om = \begin{cases}
    \ul t^{\om_0} &\text{if $n = 0$},\\
    \ul{t + 1}^{\om_0} &\text{if $n = 1$},
  \end{cases} \qquad
  c_1^\om = \begin{cases}
    \ul{t + 1}^{\om_0} &\text{if $n = 0$},\\
    \ul t^{\om_0} &\text{if $n = 1$},
  \end{cases}
\]
and
\[
  B_*^\om = \begin{cases}
    \ul{\tau_t B}^{\om_0} &\text{if $n = 0$},\\
    \ul{\tau_t \rho B}^{\om_0} &\text{if $n = 1$},
  \end{cases}
\]
where $\tau_t: \bR \to \bR$ and $\rho: \bR \to \bR$ are defined by $\tau_t r
= t + r$ and $\rho r = 1 - r$.

Let $\Om = \{\om^x \mid x \in S\}$, let $h: S \to \Om$ denote the function $x
\mapsto \om^x$, and let $\sP$ be the measure space image of $(S, \Ga, \nu)$
under $h$.

\begin{lemma}\label{L:stick-1}
  With notation as above, we have $\sP \vDash T_0$.
\end{lemma}

\begin{proof}
  Since each $\om \in \Om$ is an extension of $\om_0$, we have $\sP \vDash
  \ZFC$. Note that $\om \tDash \ph_M[\ul 1^{\om_0}, \ul t^{\om_0}]$ for all $\om
  \in \Om$. Hence, $(\ph_1)_\Om = \Om$, so that $\sP \vDash \ph_1$. Also,
  $h^{-1} (\ph_2)_\Om = [0, 1] \times \bR \times \{0, 1\}$, so that $\olbbP
  (\ph_2)_\Om = \nu_0 [0, 1] = 1$. Thus, $\sP \vDash \ph_2$.

  Note that
  \begin{align}
    \om_0 &\tDash (x \beq f_{uv}^\img(\ul B))[
      \ul{\tau_t B}^{\om_0}, \ul 1^{\om_0}, \ul t^{\om_0}
    ], \label{stick-1} \\
    \om_0 &\tDash (x \beq f_{uv}^\img(\ul B))[
      \ul{\tau_t \rho B}^{\om_0}, \ul{-1}^{\om_0}, \ul{t + 1}^{\om_0}
    ], \label{stick-2}
  \end{align}
  for all $B \in \cE$. Fix $B \in \cE$. Let $x = (r, t, 0)$ and $\om = \om^x$.
  Then $\om \tDash \ph_M[\ul 1^{\om_0}, \ul t^{\om_0}]$, $\om \tDash c_0 \beq
  \ul t$, and $\om \tDash c_1 \beq \ul 1 + \ul t$. Also, by \eqref{stick-1} and
  the definition of $B_*^\om$, we have $\om \tDash (B_* \beq f_{uv}^\img(\ul B))
  [\ul 1^{\om_0}, \ul t^{\om_0}]$. It therefore follows that $\om \tDash \ph_3$
  and $\om \tDash \ph_B(B_*)$. Similarly, if $x = (r, t, 1)$ and $\om = \om^x$,
  then $\om \tDash \ph_M[\ul{-1}^{\om_0}, \ul{t + 1}^{\om_0}]$, $\om \tDash c_0
  \beq \ul{t + 1}$, and $\om \tDash c_1 \beq \ul{-1} + \ul{t + 1}$. This time
  using \eqref{stick-2} and the definition of $B_*^\om$, we have $\om \tDash
  (B_* \beq f_{uv}^\img(\ul B))[\ul{-1}^{\om_0}, \ul{t + 1}^{\om_0}]$.
  Therefore, in this case also, we obtain $\om \tDash \ph_3$ and $\om \tDash
  \ph_B(B_*)$.
\end{proof}

\subsubsection{Narrowing down the permutations}

By Lemma \ref{L:stick-1}, we may define the inductive theory $P = \bTh \sP \dhl_
{[T_0, \Th \sP]}$. We will show that $\cC$ is consistent by showing that $P \in
\cC$. The difficult part of proving this is showing that $P$ satisfies the
principle of indifference. In preparation for this, we first prove two lemmas
that narrow down the possible permutations that need to be checked.

\begin{lemma}\label{L:stick-2}
  If $P(\ph \mid X) = p$ and $X^\pi \in \ante P$, then $\s^\pi = \s$ for all $\s
  \in L_\ZFC$ and $\pi \{c, c_0, c_1\} = \{c, c_0, c_1\}$.
\end{lemma}

\begin{proof}
  Assume $M^\pi \in L_\ZFC$. Let $\s = M^\pi$. By the argument following
  \eqref{fixed-perm}, we have $P(\s \beq M \mid X) = 1$. Note that
  \[
    h^{-1} (\s \beq M)_\Om = \{
      (r, t, n) \in S \mid \s^{\om_0} \beq \ul{[t, t + 1]}^{\om_0}
    \}.
  \]
  Since $t \ne t'$ implies $\ZFC \vdash \ul{[t, t + 1]} \nbeq \ul{[t', t' +
  1]}$, and since $\om_0 \tDash \ZFC$, there can be at most one $t \in \bR$ such
  that $\s^{\om_0} \beq \ul{[t, t + 1]}^{\om_0}$. Hence, we may choose $t_0 \in
  \bR$ such that $h^{-1} (\s \beq M)_\Om \subseteq \bR \times \{t_0\} \times
  \{0, 1\}$. Since $\opnu_0 \{t_0\} = 0$, this implies $\olbbP (\s \beq M)_\Om =
  \olnu h^{-1} (\s \beq M)_\Om = 0$. Therefore, $P(\s \beq M \mid T_0) = 0$. By
  \eqref{prob-neg} and deductive transitivity, $P(\s \beq M \mid X) = 0$, a
  contradiction. This shows that $M^\pi \notin L_\ZFC$. Similar arguments show
  that $\s^\pi \notin L_\ZFC$ for all $\s \in C$. That is, if $\s \notin
  L_\ZFC$, then $\s^\pi \notin L_\ZFC$. By contraposition, if $\s \in L_\ZFC$,
  then $\s^{-\pi} \in L_\ZFC$. Hence, by \eqref{fixed-perm}, we have $\s^\pi =
  \s$ for all $\s \in L_\ZFC$.

  Now assume $c^\pi \notin \{c, c_0, c_1\}$. By the above, $c^\pi \notin
  L_\ZFC$. Hence, either $c^\pi = M$ or $c^\pi = B_*$ for some $B \in \cE$. In
  either case, $\om \tDash c^\pi \nbin \ul \bR$ for all $\om \in \Om$.
  Therefore, $ (c^\pi \bin \ul \bR)_\Om = \emp$. Now, note that $T_0 \vdash c
  \bin \ul \bR$. Let $\de(y) \in \cL\{{\bin}\}$ be a reduced defining formula
  for $\ul \bR$, so that $\ZFC_\infty \vdash \forall y (y \beq \ul \bR \tot
  \de(y))$. Then $\ze \in T_0$, where $\ze = \exists y (\de(y) \wedge c \bin
  y)$. We then have $\ze^\pi = \exists y (\de(y) \wedge c^\pi \bin y)$, so that
  $T_0 \vdash \ze^\pi \tot c^\pi \bin \ul \bR$. It follows that $\ze^\pi_\Om =
  (c^\pi \bin \ul \bR)_\Om = \emp$, $\bbP$-a.s. Thus, $\olbbP \ze^\pi_\Om = 0$,
  contradicting Lemma \ref{L:mapping-T_0}. This shows that $c^\pi \in \{c, c_0,
  c_1\}$. Similar arguments show that $c_0^\pi \in \{c, c_0, c_1\}$ and $c_1^\pi
  \in \{c, c_0, c_1\}$. We therefore have $\pi \{c, c_0, c_1\} = \{c, c_0,
  c_1\}$.
\end{proof}

\begin{rmk}\label{R:stick}
  The above proof relies on the fact that $\nu_0$ is a continuous measure. Since
  we are only concerned with the relative positioning of $M$ and $c$, we are
  free to randomize the location of $M$. By choosing a $\nu_0$ which is
  continuous, we are randomizing $M$ so that it has probability $0$ of sitting
  at any fixed location. This allows us to narrow down the possible permutations
  in the principle of indifference, and thereby simplify our proofs. Without
  this construction, the results still hold, but the proofs would be more
  complicated. Namely, we would need to consider the possible that $M$ and $\ul
  B$ are interchanged for some interval $B$.

  In Section \ref{S:one-coin}, we saw how to deal with this extra complication
  in the finite setting. There, we needed to deal with the possibility that the
  symbol for the result of the coin toss, $c$, was interchanged with one of the
  symbols for heads and tails.
\end{rmk}

\begin{lemma}\label{L:stick-3}
  Let $P(\ph \mid X) = p$ and $X^\pi \in \ante P$. Assume $\pi$ is not the
  identity. Then $\s^\pi = \s$ for all $\s \in L$, except for the following:
  \begin{enumerate}[(a)]
    \item $c_0^\pi = c_1$,
    \item $c_1^\pi = c_0$, and
    \item $B_*^\pi = (\rho B)_*$ for all $B \in \cE$.
  \end{enumerate}
\end{lemma}

\begin{proof}
  By Lemma \ref{L:stick-2}, we have $\s^\pi = \s$ for all $\s \in L_\ZFC$. We
  first show that $M^\pi = M$. Let $\s = M^\pi$ and assume $\s \ne M$. By Lemma
  \ref{L:stick-2}, we have $\s = B_*$ for some $B \in \cE$. Let $\ze = B_*
  \bsub M \wedge B_* \nbeq M \in T_0$. By Lemma \ref{L:stick-2}, we must have
  $M^{-\pi} = B'_*$ for some $B' \in \cE$. Then $\ze^{-\pi} = M \bsub B'_*
  \wedge M \nbeq B_*'$. But $\om \tDash \neg M \bsub B'_*$ for all $\om \in
  \Om$. Therefore, $\ze_\Om^{-\pi} = \emp$, contradicting Lemma 
  \ref{L:mapping-T_0}. Hence, $M^\pi = M$.

  We next show that $c^\pi = c$. Let $\s = c^\pi$ and assume $\s \ne c$. By
  Lemma \ref{L:stick-2}, we have $\s \in \{c_0, c_1\}$. Suppose $\s = c_0$. Let
  \[
    \ph(x, y) = (\exists uv \bin \ul \bR)(
      u \nbeq \ul 0 \wedge y \beq f_{uv}^\img(\ul{[0, 1]}) \wedge (
        x \beq u \vee x \beq u + v
      )
    ).
  \]
  Then $\ph(x, y)$ says that $x$ is an endpoint of $y$. Therefore, $\ze =
  \ph^\rd(c_0, M) \in T_0$. Since $\ze^{-\pi} = \ph^\rd(c, M)$, we have
  \begin{align*}
    h^{-1} \ze_\Om^{-\pi} &= \{(r, t, n) \in S \mid r = 0 \text{ or } r = 1\}\\
    &= \{0\} \times \bR \times \{0, 1\} \cup \{1\} \times \bR \times \{0, 1\}.
  \end{align*}
  Since $\nu_0$ is continuous, this gives $\olbbP \ze_\Om^{-\pi} = \olnu h^{-1}
  \ze_\Om^{-\pi} = 0$, contradicting Lemma \ref{L:mapping-T_0}. Hence, $\s \ne
  c_0$. A similar argument shows $\s \ne c_1$. Thus, $c^\pi = c$.

  We next show that $c_0^\pi = c_1$. Assume not. Then, by Lemma \ref{L:stick-2},
  we have $c_0^\pi = c_0$ and $c_1^\pi = c_1$. We will show that $B_*^\pi = B_*$
  for all $B \in \cE$, contradicting the assumption that $\pi$ is not the
  identity permutation. Let $B \in \cE$ and let $\s = B_*^\pi$. Since $M^\pi =
  M$, Lemma \ref{L:stick-2} implies $\s = B'_*$ for some $B' \in \cE$. Note that
  $\ph_B^\rd(x) \in \cL\{\bin, M, c_0\}$. Let $\ze = \ph_B^\rd(B_*) \in T_0$.
  Since $M^\pi = M$ and $c_0^\pi = c_0$, we have $\ze^\pi = \ph_B^\rd(B'_*)$.
  Lemma \ref{L:mapping-T_0} implies $\ze_\Om^\pi \ne \emp$. Hence, we may choose
  $x = (r, t, n) \in S$ such that $\om \tDash \ze^\pi$, where $\om = \om^x$.
  Therefore, there exists $a$ and $b$ in the domain of $\om$ such that $\om
  \tDash \ph_M[a, b]$, $c_0^\om = b$, and $\om \tDash (B'_* \beq f_{uv}^\img
  (\ul B))[a, b]$. By the construction of $\om$, we have
  \[
    \om \tDash (\forall uv \bin \ul \bR)(
      \ph_M(u, v) \to u \beq \ul 1 \vee u \beq \ul{-1}
    ).
  \]
  Hence, $a \in \{\ul 1^{\om_0}, \ul{-1}^{\om_0}\}$.

  Suppose $n = 0$. Then $b = c_0^\om = \ul t^{\om_0}$. If $a = \ul{-1}^{\om_0}$,
  then
  \[
    \om \tDash (M \beq f_{uv}^\img(\ul{[0, 1]}))[
      \ul{-1}^{\om_0}, \ul t^{\om_0}
    ],
  \]
  which implies $M^\om = \ul{[t - 1, t]}^{\om_0}$. But $M^\om = \ul{[t, t + 1]}^
  {\om_0}$, and so we have $a = \ul 1^{\om_0}$. Thus, $\om \tDash (B'_* \beq
  f_{uv}^\img(\ul B))[\ul 1^{\om_0}, \ul t^{\om_0}]$. But $\om_0 \tDash (
  \ul{\tau_t} \beq f_{uv})[\ul 1^{\om_0}, \ul t^{\om_0}]$ and also $\om_0 \tDash
  \ul{\tau_t}^\img(\ul B) = \ul{\tau_t B}$. Hence, $\om \tDash B'_* \beq 
  \ul{\tau_t B}$, which gives $(B'_*)^\om = \ul{\tau_t B}^{\om_0}$. On the other
  hand, by the definition of $\om$, and since $n = 0$, we have $(B'_*)^\om = 
  (\ul{\tau_t B'})^{\om_0}$. Therefore, $(\ul{\tau_t B'})^{\om_0} = \ul{\tau_t
  B}^{\om_0}$, which implies $\tau_t B' = \tau_t B$, so that $B' = B$. Hence,
  $\s = B'_* = B_*$. A similar proof in the case $n = 1$ also yields $\s = B_*$.
  Thus, $B_*^\pi = B_*$, completing the proof that $c_0^\pi = c_1$.

  By Lemma \ref{L:stick-2}, we must have $c_1^\pi = c_0$, so that both (a) and
  (b) hold. For (c), let $B \in \cE$ and let $\s = B_*^\pi$. As above, $\s =
  B'_*$ for some $B' \in \cE$. Also as above, let $\ze = \ph_B^\rd(B_*) \in
  T_0$. Then $\om \tDash \ze^\pi$ if and only if $\om \tDash \ph_B(B_*)^\pi$,
  and
  \[
    \ph_B(B_*)^\pi = (\exists uv \bin \ul \bR) (
      \ph_M(u, v) \wedge c_1 \beq v \wedge B'_* \beq f_{uv}^\img(\ul B)
    ).
  \]
  The above argument, with $c_1$ in place of $c_0$, shows that $B' = \rho B$.
\end{proof}

\subsubsection{Proof of main result}

With all of the above preparation, we are now ready to prove Theorem
\ref{T:stick}.

\begin{proof}[Proof of Theorem \ref{T:stick}]
  Let $P$ be the inductive theory defined above. Since
  \[
    (c \bin B_*)_\Om = {
      B \times \bR \times \{0\} \cup \rho B \times \bR \times \{1\}
    },
  \]
  we have $P(c \bin B_* \mid T_0) = \olbbP (c \bin B_*)_\Om = (\opnu_0 B +
  \opnu_0 \rho B)/2$. Hence, $P$ satisfies (i) in the definition of $\cC$.
  Suppose $P(\ph \mid X) = p$ and $X^\pi \in \ante P$. If $\pi$ is the identity,
  then it is trivially the case that $P(\ph^\pi \mid X^\pi) = p$. Assume $\pi$
  is not the identity. Define $g: S \to S$ by $g(r, t, n) = (r, t, 1 - n)$. Then
  $g$ is a pointwise isomorphism from $(S, \Ga, \nu)$ to itself. Recall that $h$
  is the function that maps $x = (r, t, n)$ to $\om^x$. Recall also the function
  $h_\pi$, used in Section \ref{S:PoI-models} in the construction of $\sP^\pi$.
  By Lemma \ref{L:stick-3}, we have $h_\pi = h \circ g \circ h^{-1}$. Thus,
  $\Om^\pi = \Om$ and
  \[
    {\bbP} \circ h_\pi^{-1} = {\bbP} \circ h \circ g \circ h^{-1}
    = {\opnu} \circ g \circ h^{-1}
    = {\opnu} \circ h^{-1}
    = {\bbP},
  \]
  so that $\sP^\pi = \sP$. By Theorem \ref{T:induc-invar}, we have $\sP \vDash
  (X^\pi, \ph^\pi, p)$, so that $P(\ph^\pi \mid X^\pi) = p$. This shows that $P$
  satisfies (ii), and hence, $P \in \cC$. Therefore, $\cC$ is consistent.

  Now let $P \in \cC$. Let $\pi$ be the permutation described in Lemma
  \ref{L:stick-3}. Since $\pi$ fixes $L_\ZFC$, we have $\ZFC^\pi = \ZFC$. Since
  $M^\pi = M$ and $c^\pi = c$, we have $\ph_1^\pi = \ph_1$ and $\ph_2^\pi =
  \ph_2$. Now,
  \begin{multline*}
    \ZFC, u \bin \ul \bR, v \bin \ul \bR
    \vdash (\ph_M(u, v) \wedge c_0 \beq v \wedge c_1 \beq u + v)\\
    \tot (\ph_M(-u, u + v) \wedge c_1 \beq u + v \wedge c_0 \beq -u + (u + v)).
  \end{multline*}
  Therefore, $\ZFC_\infty \vdash \ph_3^\pi \tot \ph_3$. Similarly,
  \begin{multline*}
    \ZFC, u \bin \ul \bR, v \bin \ul \bR, \ph_3
    \vdash (\ph_M(u, v) \wedge c_0 \beq v \wedge B_* \beq f_{uv}^\img(\ul B))\\
    \tot (
      \ph_M(-u, u + v)
      \wedge c_1 \beq u + v
      \wedge (\rho B)_* \beq f_{-u, u + v}^\img (\ul B)
    ).
  \end{multline*}
  Therefore, $\ZFC, \ph_3 \vdash \ph_B(B_*)^\pi \tot \ph_B(B_*)$.
  Altogether, this implies $T_0^\pi = T_0$, so that
  \[
    P(c \bin B_* \mid T_0)
    = P((c \bin B_*)^\pi \mid T_0)
    = P(c \bin (\rho B)_* \mid T_0),
  \]
  by the principle of indifference.
\end{proof}

\subsection{Adding a defined constant}\label{S:[0,1]-defn}

Let us return to example of Section \ref{S:[0,1]}, in which we have a constant
$c$, about which we only know that $c \in [0, 1]$. We saw that in this case, the
principle of indifference has nothing to say about the distribution of $c$.

From here, let us proceed as in Section \ref{S:0<1-defn}. That is, let us expand
our language by defining $d = 1 - c$. After making this seemingly harmless
addition, we turn our attention back to $c$, and ask again what the principle of
indifference has to say. This time, we find the same result that we obtained in
\eqref{stick}. That is, the distribution of $c$ must be symmetric under the
reflection $r \mapsto 1 - r$.

However, as in Section \ref{S:0<1-defn}, this result is misleading. In adding
$d$ to our language, we are also compelled to expand $T_0$ to a larger theory
$T_0'$ which includes the definition of $d$. In the expanded theory $T_0'$, the
implicit meaning of $c$ has changed. In $T_0'$, we can no longer determine which
of $c$ and $d$ is the original number, and which is its reflection. Just as in
Section \ref{S:0<1-defn}, it is as if $c$ is equally likely to be the original
as it is to be the reflection. It is no surprise, then, that $c$ must have a
symmetric distribution.

In other words, by adding $d$, we have changed the problem. We no longer have a
single unknown number in $[0, 1]$. We now have two unknown numbers that are
reflections of one another, and we are unable to definitively identify the
original. We have simultaneously added information (by adding a second number)
and lost information (by losing track of which was the original). So although
\eqref{stick} still holds, the situation is not the same. In this case, 
\eqref{stick} is answering a different question.

To make this precise, let $c$ and $d$ be constant symbols not in $L_\ZFC$ and
let $L = L_\ZFC \cup \{c, d\}$. Let $f(r) = 1 - r$ and define
\[
  T_0 = \ZFC + \{c \bin \ul{[0, 1]}, d \beq \ul f(c)\}.
\]
Note that we did not bother to express things in terms of a definitorial
extension. It is the case, however, that $T_0$ is a definitorial extension of
$\ZFC + c \bin [0, 1]$, and $d$ is defined by $\forall y (y \beq d \tot
y \beq \ul f(c))$.

Let $\cC$ be the inductive condition consisting of all inductive theories $P
\subseteq \cL^\IS$ with root $T_0$ such that
\begin{enumerate}[(i)]
  \item $c$ is Borel given $T_0$, and
  \item $P$ satisfies the principle of indifference.
\end{enumerate}

\begin{prop}
  The condition $\cC$ is consistent. Moreover, if $P \in \cC$, then
  \begin{equation}\label{stick-defn}
    P(c \bin \ul V \mid T_0) = P(c \bin \ul f^\img(\ul V) \mid T_0),
  \end{equation}
  for all $V \in \cB([0, 1])$.
\end{prop}

\begin{proof}
  Let $S = [0, 1]$ and $\Ga = \cB([0, 1])$. Let $\nu$ be a probability measure
  on $(S, \Ga)$ such that $\opnu \rho V = \opnu V$ for all $V \in \Ga$, and
  $\nu$ is continuous. That is, $\opnu \{r\} = 0$ for all $r \in S$. Let $\om_0$
  be an $L_\ZFC$-structure such that $\om_0 \tDash \ZFC$. For each $r \in S$,
  let $\om = \om^r$ be the $L$-expansion of $\om_0$ given by $c^\om = \ul
  r^{\om_0}$ and $d^\om = \ul{1 - r}^{\om_0}$. Let $\Om = \{\om^r \mid r \in
  S\}$, let $h: S \to \Om$ denote the function $r \mapsto \om^r$, and let $\sP =
  (\Om, \Si, \bbP)$ be the measure space image of $(S, \Ga, \nu)$ under $h$.
  Then $\sP \vDash T_0$, so we may define the inductive theory $P = \bTh \sP
  \dhl_{[T_0, \Th \sP]}$.

  By construction, $c$ is Borel given $T_0$. Suppose $P(\ph \mid X) = p$ and
  $X^\pi \in \ante P$. If $\pi$ is the identity, then it is trivially the case
  that $P(\ph^\pi \mid X^\pi) = p$. Assume, then, that $\pi$ is not the
  identity.

  Let $\s = c^\pi$. Assume $\s \in L_\ZFC$. By the argument following
  \eqref{fixed-perm}, we have $P(\s \beq c \mid X) = 1$. Note that $h^{-1} (\s
  \beq c)_\Om = \{r \in S \mid \s^{\om_0} = \ul r^{\om_0}\}$. Since $r \ne r'$
  implies $\ZFC \vdash \ul r \nbeq \ul{r'}$, and since $\om_0 \tDash \ZFC$,
  there can be at most one $r \in \bR$ such that $\s^{\om_0} = \ul r^{\om_0}$.
  Since $\nu$ is continuous, this implies $P(\s \beq c \mid T_0) = \olnu h^{-1}
  (\s \beq c)_\Om = 0$, which contradicts $P(\s \beq c \mid X) = 1$. This shows
  that $\s \notin L_\ZFC$. Similarly, $d^\pi \notin L_\ZFC$. It follows from
  \eqref{fixed-perm} that $\s^\pi = \s$ for all $\s \in L^\infty$. Also, since
  $\pi$ is not the identity, we have $c^\pi = d$ and $d^\pi = c$.

  Now define $g: S \to S$ by $gr = 1 - r$. Then $g$ is a pointwise isomorphism
  from $(S, \Ga, \nu)$ to itself and $h_\pi = h \circ g \circ h^{-1}$. As in the
  proof of Theorem \ref{T:stick}, it follows that $\sP^\pi = \sP$. Hence, by
  Theorem \ref{T:induc-invar}, we have $\sP \vDash (X^\pi, \ph^\pi, p)$, so that
  $P(\ph^\pi \mid X^\pi) = p$. This shows that $P$ satisfies the principle of
  indifference, and hence, $P \in \cC$. Therefore, $\cC$ is consistent.

  Now let $P \in \cC$. Let $\pi$ be the permutation described above. Then
  $T_0^\pi = T_0$, so that
  \[
    P(c \bin \ul V \mid T_0)
    = P((c \bin \ul V)^\pi \mid T_0)
    = P(d \bin \ul V \mid T_0),
  \]
  by the principle of indifference. But $T_0 \vdash d \bin \ul V \tot c \bin
  \ul f^\img(\ul V)$. Hence, \eqref{stick-defn} follows from the rule of logical
  implication and Proposition \ref{P:log-equiv-gen}.
\end{proof}

There is nothing special about the function $f(r) = 1 - r$ in this example.
What is essential is that $f$ is measurable and $f \circ f = \iota$, where
$\iota$ is the identity. For example, let
\[
  f(x) = \begin{cases}
    1 - 2x &\text{if $0 \le x \le \frac13$},\\
    \frac12 - \frac12 x &\text{if $\frac13 < x \le 1$}.
  \end{cases}
\]
If we change the definition of $d$ from $d \beq 1 - c$ to $d \beq \ul f(c)$,
then we can adapt the above proof so that we obtain \eqref{stick-defn} in this
case as well. However, the resulting distribution of $c$ would no longer be
symmetric under reflection. Instead, it would be symmetric under $f$. In
particular, after defining $d \beq \ul f(c)$, the principle of indifference
would tell us that
\[
  P(c \bin \ul{[0, 1/3]} \mid T_0) = P(c \bin \ul{[1/3, 1]} \mid T_0) = 1/2.
\]
This is decidedly different from the situation obtained when $f(r) = 1 - r$. We
can therefore see clearly that the act of defining $d$ is not a harmless one. As
described above and in Section \ref{S:0<1-defn}, when we introduce $d$ via
definition, we change the background assumptions in $T_0$, which in turn changes
the meaning of $c$. The principle of indifference, therefore, can produce
different distributions for $c$, depending on the definition of $d$.

\section{Examples in the plane}\label{S:indiff-plane}

In this section, we present examples of the principle of indifference that are
situated in the Euclidean plane, $\bR^2$.

\subsection{A point on a circle}\label{S:pt-on-circ}

In our first example in this section, we consider a circle of some unspecified
diameter, and we let $c$ be a point that lies somewhere along the circumference
of the circle. As with a point on a rod in Section \ref{S:pt-on-rod}, we do not
simply want to replace the circle in the statement of the problem with the unit
circle, $\bS^1 = \{(x, y) \in \bR^2 \mid x^2 + y^2 = 1\}$. Instead, we want to
replace the circle with a subset $M \subseteq \bR^2$ that can be parameterized
in the usual way by the radian angles in $[0, 2\pi)$.

To make this precise, define $e: \bR \times \bR^2 \to \bR^2$ by
\[
  e(\th, r) = \begin{pmatrix}
    \cos \th & -\sin \th\\
    \sin \th & \phantom{-}\cos \th
  \end{pmatrix} \begin{pmatrix}
    r_1 \\ r_2
  \end{pmatrix},
\]
so that $e$ rotates $r$ counterclockwise by an angle of $\th$ radians. For $t
\in \bR^2$ and $B \subseteq \bR^2$, let $t + B = \{t + r \mid r \in B\}$. For
$\th \in \bR$, let $[\th] = \th - 2 \pi \flr{\th/2 \pi}$, where $\flr{\th/2
\pi}$ is the greatest integer less than or equal to $\th/2 \pi$. Then $[\th] \in
[0, 2 \pi)$ and $e ([\th], r) = e(\th, r)$ for all $r \in \bR^2$.

In $\ZFC$, we use $\ul \bR^2$ as shorthand for $\ul \bR \times \ul \bR$. We
extend $+$ in $\ZFC$ so that it also denotes vector addition in $\ul \bR^2$, and
we extend $\bdot$ so that it also denotes scalar multiplication. As we did for
$\bR$ and $\cB(\bR)$, we add an explicitly defined symbol $\ul r$ for each $r
\in \bR^2$, and an explicitly defined symbol $\ul V$ for each $V \in \cB (\ul
\bR^2)$. If $h: \bR^m \to \bR^n$ is Borel measurable, where $m, n \in \{1, 2\}$,
then we can explicitly define $\ul h$ in $\ZFC$.

We now create our extralogical signature $L$ and our root $T_0$ as we did in
Section \ref{S:pt-on-rod}. Define $\de(u, v, w, y) \in \cL_\ZFC$ by
\begin{multline*}
  \de(u, v, w, y) = (
    u \nbin \ul \bR \vee v \nbin \ul \bR^2 \vee w \nbin \ul \bR
  ) \wedge y \beq \ul \emp\\
  \vee {
    u \bin \ul \bR
    \wedge v \bin \ul \bR^2
    \wedge w \bin \ul \bR
    \wedge y \bin (\ul \bR^2)^{\ul \bR^2}
  }\\
  \wedge (\forall x \bin \ul \bR)(y(x) \beq u \bdot \ul e(w, x) + v).
\end{multline*}
Then $\ZFC \vdash \forall u v w \exists! y \, \de(u, v, w, y)$. Hence, we could
explicitly define the function symbol $F$ by $y \beq Fuvw \tot \de (u, v, w,
y)$, and then let $f_{uvw}$ be shorthand for the term $Fuvw$. We do not, in
fact, add the symbol $F$ to our extralogical signature, but instead regard both
$F$ and $f_{uvw}$ as shorthand. We also adopt the shorthand, $\dom(f)$ and
$f^\img (z)$, given by
\begin{align*}
  \dom(f) &= \{
    x \bin \ul \bR^2 \mid (\exists y \bin \ul \bR^2) (x, y) \bin f
  \},\\
  f^\img(z) &= \{f(x) \mid x \bin z \cap \dom(f)\}.
\end{align*}
Let $\cE = \{B \in \cB(\bS^1) \mid \emp \subset B \subset \bS^1\}$ and let
\[
  C = \{M, c\} \cup \{B_* \mid B \in \cE\}
\]
be a set of distinct constant symbols not in $L_\ZFC$. Let $L = L_\ZFC C$.

Let $\ph_c = c \bin M$ and
\[
  \ph_M(u, v, w) = {
    u > \ul 0
    \wedge w \bin \ul{[0, 2 \pi)}
    \wedge M \beq f_{uvw}^\img(\ul{\bS^1}))
  }.
\]
For $B \in \cE$, define
\begin{multline*}
  \ph_B = (\exists uw \bin \ul \bR) (\exists v \bin \ul \bR^2)\\
  (
    \ph_M(u, v, w)
    \wedge \{(1, 0)\}_* = f_{uvw}^\img(\ul{\{(1, 0)\}})
    \wedge B_* \beq f_{uvw}^\img(\ul B)
  ),
\end{multline*}
and let $T_0 = \ZFC + \{\ph_c\} \cup \{\ph_B \mid B \in \cE\}$.

In this presentation, we have streamlined our construction of $T_0$, compared to
what was done in the rod example of Section \ref{S:pt-on-rod}. For instance, we
do not have a separate sentence which says that $M$ is a circle. Rather, that
fact is contained in each sentence $\ph_B$. The sentence $\ph_c$ says that $c$
is a point on the circle $M$. And the sentences $\ph_B$ say that $B_*$ is a
Borel subset of $M$ that is geometrically related to $\{(1, 0)\}_*$ in the same
way that $B$ is related to $\{(1, 0)\}$. In the rod example, we first named our
endpoints and then named our Borel sets relative to them. Here, we are first
naming a point to serve as the initial and terminal point of a parameterization,
and then naming our Borel sets relative to that point.

Now let $\cC$ be the inductive condition consisting of all inductive theories $P
\subseteq \cL^\IS$ with root $T_0$ such that
\begin{enumerate}[(i)]
  \item $P(c \bin B_* \mid T_0)$ exists for all $B \in \cE$, and
  \item $P$ satisfies the principle of indifference.
\end{enumerate}

\begin{prop}\label{P:circle}
  The condition $\cC$ is consistent and, for every $B \in \cE$, we have
  $\bfP_\cC (c \bin B_* \mid T_0) = m(B)$, where $m$ is the uniform measure on
  $\bS^1$.
\end{prop}

The proof of Proposition \ref{P:circle} will come at the end of this subsection.
The proof follows the same lines as the proof of Theorem \ref{T:stick}. We first
prove consistency by building an $\cL$-model, $\sP = (\Om, \Si, \bbP)$, as
follows. Let $\nu_0$ be a probability measure on $(\bS^1, \cB(\bS^1))$ such that
$\nu_0$ is continuous. That is, $\opnu_0 \{r\} = 0$ for all $r \in \bS^1$. Let
$S = \bS^1 \times \bS^1 \times [0, 2 \pi)$, $\Ga = \cB(S)$, and $\nu = \nu_0
\times \nu_0 \times m_0$, where $m_0$ is the uniform measure on $[0, 2 \pi)$.

Let $\om_0$ be an $L_\ZFC$-structure such that $\om_0 \tDash \ZFC$. If $x = (r,
t, \th) \in S$, then let $\om = \om^x$ be the $L$-expansion of $\om_0$ defined
by $M^\om = \ul{t + \bS^1}^{\om_0}$, $c^\om = \ul{t + r}^{\om_0}$, and $B_*^\om
= \ul{t + e(\th, B)}^{\om_0}$.

Let $\Om = \{\om^x \mid x \in S\}$, let $h: S \to \Om$ denote the function $x
\mapsto \om^x$, and let $\sP$ be the measure space image of $(S, \Ga, \nu)$
under $h$.

By a proof similar to that of Lemma \ref{L:stick-1}, we have $\sP \vDash T_0$.
We may therefore define the inductive theory $P = \bTh \sP \dhl_{[T_0, \Th
\sP]}$.

\begin{lemma}\label{L:circle}
  Let $P(\ph \mid X) = p$ and $X^\pi \in \ante P$. Then $\s^\pi = \s$ for all
  $\s \in L_\ZFC \cup \{M, c\}$ and there exists $\th_0 \in [0, 2 \pi)$ such
  that $B_*^\pi = e(\th_0, B)_*$ for all $B \in \cE$.
\end{lemma}

\begin{proof}
  Applying the methods used in the proof of Lemma \ref{L:stick-2} and the first
  part of the proof of Lemma \ref{L:stick-3}, we obtain that $\s^\pi = \s$ for
  all $\s \in L_\ZFC \cup \{M, c\}$. Hence, we must have $\{(1, 0)\}_*^\pi =
  B^0_*$ for some $B^0 \in \cE$. Let $\ze = (\exists! y \bin \ul \bR^2) (y \bin
  \{(1, 0)\}_*) \in T_0$. By Lemma \ref{L:mapping-T_0}, the set $\ze_\Om^\pi$ is
  nonempty. Choose $\om \in \ze_\Om^\pi$. Then $\om \tDash (\exists! y \bin \ul
  \bR^2) (y \bin B^0_*)$. Write $\om = \om^x$, where $x = (r, t, \th) \in S$.
  Then there exists a unique $b$ in the domain of $\om_0$ such that $b \bin^
  {\om_0} (\ul \bR^2)^{\om_0}$ and $b \bin^{\om_0} (B^0_*)^\om = \ul{t + e(\th,
  B^0)}^{\om_0}$. Thus, $\om_0 \tDash (\exists! y \bin \ul \bR^2) (y \bin \ul{t
  + e(\th, B^0)})$. This implies $|B^0| = 1$, so that we may write $B^0 = 
  \{r_0\}$ for some $r_0 \in \bS^1$. We therefore have $\{(1, 0)\}_*^\pi = 
  \{r_0\}_*$.

  Now let $B \in \cE$ be arbitrary. Then $B_*^\pi = B'_*$ for some $B' \in \cE$.
  By Lemma \ref{L:mapping-T_0}, we may choose $x = (r, t, \th) \in S$ such that
  $\om = \om^x \tDash \ph_B^\pi$. Note that
  \begin{multline*}
    \ph_B^\pi \equiv_\ZFC {
      (\exists! uw \bin \ul \bR) (\exists! v \bin \ul \bR^2)
    }\\
    (
      \ph_M(u, v, w)
      \wedge \{r_0\}_* = f_{uvw}^\img(\ul{\{(1, 0)\}})
      \wedge B'_* \beq f_{uv}^\img(\ul B)
    ).
  \end{multline*}
  Choose $\th_0$ such that $r_0 = e(\th_0, (1, 0))$. Then
  \[
    \{r_0\}_*^\om
    = \ul{t + e(\th, \{r_0\})}^{\om_0}
    = \ul{t + e([\th_0 + \th], \{(1, 0)\})}^{\om_0},
  \]
  so that
  \[
    \om \tDash (\{r_0\}_* = f_{uvw}^\img(\ul{\{(1, 0)\}}))[
      \ul 1^{\om_0}, \ul t^{\om_0}, \ul{[\th_0 + \th]}^{\om_0}
    ].
  \]
  Since $\om \tDash \ph_B^\pi$, we must have
  \[
    \om \tDash (B'_* \beq f_{uv}^\img(\ul B))[
      \ul 1^{\om_0}, \ul t^{\om_0}, \ul{[\th_0 + \th]}^{\om_0}
    ],
  \]
  which implies
  \[
    \ul{t + e(\th, B')}^{\om_0} = \ul{t + e([\th_0 + \th], B)}^{\om_0},
  \]
  and therefore $B' = e(\th_0, B)$.
\end{proof}

\begin{proof}[Proof of Proposition \ref{P:circle}]
  Let $P$ be the inductive theory defined above. If $x = (r, t, \th)$, then
  \begin{align*}
    \om^x \tDash c \bin B_* &\quad\text{iff}\quad {
      \ul{t + r}^{\om_0} \bin^{\om_0} \ul{t + e(\th, B)}^{\om_0}
    }\\
    &\quad\text{iff}\quad \om_0 \tDash \ul{t + r} \bin \ul{t + e(\th, B)}\\
    &\quad\text{iff}\quad t + r \in t + e(\th, B)\\
    &\quad\text{iff}\quad r \in e(\th, B).
  \end{align*}
  Thus, $h^{-1} (c \bin B_*)_\Om = \{(r, t, \th) \mid r \in e(\th, B)\}$. Note
  that $r \in e(\th, B)$ if and only if $e(-\th, r) \in B$. Hence,
  \begin{align*}
    \opnu h^{-1}(c \bin B_*)_\Om &= \int_S 1_{h^{-1}(c \bin B_*)_\Om} \, d\nu\\
    &= \int_{\bS^1} \int_0^{2\pi} 1_B(e(-\th, r)) \, m_0(d\th) \, \nu_0(dr)\\
    &= \int_{\bS^1} m(B) \, \nu_0(dr) = m(B).
  \end{align*}
  Since ${\bbP} = {\opnu} \circ h^{-1}$, it follows that $P$ satisfies (i) in
  the definition of $\cC$.

  Suppose $P(\ph \mid X) = p$ and $X^\pi \in \ante P$. Let $\th_0$ be as in
  Lemma \ref{L:circle} and define $g: S \to S$ by $g(r, t, \th) = (r, t, [\th_0
  + \th])$. We may use $g$ as in the proof of Theorem \ref{T:stick} to show that
  $\sP^\pi = \sP$, so that Theorem \ref{T:induc-invar} yields $P(\ph^\pi \mid
  X^\pi) = p$. This shows that $P$ satisfies (ii), and hence, $P \in \cC$.  
  Therefore, $\cC$ is consistent.

  Now let $P \in \cC$ be arbitrary. Given $\th_0 \in \bR$, let $\pi$ be the
  $L$-permutation satisfying $\s^\pi = \s$ for all $\s \in L_\ZFC \cup \{M,
  c\}$ and $B_*^\pi = e(\th_0, B)_*$ for all $B \in \cE$. Then $T_0^\pi = T_0$,
  so by the principle of indifference, we have
  \begin{equation}\label{circle}
    P(c \bin B_* \mid T_0) = P(c \bin e(\th_0, B)_* \mid T_0)
  \end{equation}
  for all $B \in \cE$. Let us adopt the shorthand notation, $\emp_* = \ul \emp$
  and $\bS^1_* = M$. Since $P(c \bin B_* \mid T_0)$ exists for all $B \in \cE$,
  we may define $m': \cB(\bS^1) \to [0, 1]$ by $m'(B) = P(c \bin B_* \mid T_0)$.
  Note that if $B, B' \in \cB(\bS^1)$ and $B \cap B' = \emp$, then $T_0 \vdash
  \neg (c \bin B_* \wedge c \bin B'_*)$. Hence, Theorem \ref{T:ctbl-add} implies
  that $m'$ is a probability measure on $(\bS^1, \cB(\bS^1))$. By
  \eqref{circle}, we have $m' (B) = m'(e(\th_0, B))$ for all $B \in \cB(\bS^1)$
  and all $\th_0 \in \bR$. This implies $m' = m$, so that $P(c \bin B_* \mid
  T_0) = m(B)$. Since $P$ was arbitrary, $\bfP_\cC(c \bin B_* \mid T_0) = m(B)$.
\end{proof}

\subsection{Bertrand's paradox}
  \index{Bertrand's paradox}%

\subsubsection{Introduction}

In 1888, Joseph Bertrand posed the following problem (see \cite{Bertrand1888}).
Consider an equilateral triangle inscribed in a circle. Let $\ell$ be a chord of
the circle, chosen at random. What is the probability that the chord is longer
than a side of the triangle?

It is considered a ``paradox'' because Bertrand presented three different
solutions, all purporting to use the principle of indifference, that gave three
different answers: $1/3$, $1/2$, and $1/4$. Of course, this is only
``paradoxical'' if we have the prior expectation that the principle of
indifference ought to produce a unique answer. We have already seen, however,
that this is not always the case. The principle of indifference is a tool that
can narrow down the possible distributions in certain circumstances, but it does
not necessarily determine for us a unique distribution. In the example of the
rod from Section \ref{S:pt-on-rod}, for instance, the principle tells us that
the distribution must be symmetric under reflection. But beyond that, it leaves
open a whole range of possibilities.

Something similar happens with Bertrand's chord. We will show below that,
according to the principle of indifference, the distribution of the chord must
be rotationally invariant. But beyond that, it has nothing more to say. Hence,
all three of Bertrand's solutions (which are all rotationally invariant) are
consistent with the principle of indifference. But so are many distributions
that Bertrand did not consider. In fact, in Theorem \ref{T:Bertrand}, we show
that for any $p \in [0, 1]$, it is consistent with the principle of indifference
to say that the answer is $p$.

\subsubsection{Notation in \texorpdfstring{$\ZFC$}{ZFC}}

To precisely formulate the problem, we first establish some new notation and
shorthand. Let $\bD = \{(x, y) \in \bR^2 \mid x^2 + y^2 \le 1\}$. Using the
notation of Section \ref{S:pt-on-circ}, define
\[
  \ph_\bD^\param(x, u, v, w) : {
    u > \ul 0 \wedge w \bin \ul{[0, 2\pi)} \wedge x \beq f_{uvw}^\img(\ul \bD)
  }
\]
Then $\ph_\bD^\param(x, u, v, w)$ says that $x$ is a disk in the plane and that
$u$, $v$, and $w$ are the constants in a parameterization of $x$. Also define
\[
  \ph_\bD(x) : {
    (\exists uw \bin \ul \bR) (\exists v \bin \ul \bR^2)
    \ph_\bD^\param(x, u, v, w)
  }
\]
Then $\ph_\bD(x)$ simply says that $x$ is a disk in the plane.

Let $\ph_\tria(x)$ be a formula which says that $x$ is a nondegenerate
equilateral triangle in the plane. The exact details of this formula are not
relevant for our purposes and will be omitted. In what follows, we will
similarly omit the details of other formulas whose descriptions are given only
verbally.

Let $\ph_\seg(x)$ be a formula which says that $x$ is a line segment in the
plane whose length is positive and finite. Let $\ph_\ins(x, y) = \ph_\tria(x)
\wedge \ph_\bD(y) \wedge \ze(x, y)$, where $\ze(x, y)$ is a formula which says
that the triangle $x$ is inscribed in the circle that is the boundary of $y$.
Similarly, let $\ph_\ch(x, y) = \ph_\seg(x) \wedge \ph_\bD(y) \wedge \ze(x, y)$,
where $\ze(x, y)$ is a formula which says that the endpoints of the line segment
$x$ lie on the boundary of the disk $y$.

Define
\[
  \de_{\len}(x, y) : {
    \neg \ph_\tria(x) \wedge \neg \ph_\seg(x) \wedge y \beq \ul \emp
    \vee \ph_\tria(x) \wedge \ze(x, y)
    \vee \ph_\seg(x) \wedge \ze'(x, y)
  }
\]
Here, $\ze(x, y)$ is a formula which says that $y$ is the length of each side of
the equilateral triangle $x$, and $\ze'(x, y)$ is a formula which says that $y$
is the length of the line segment $x$. Then $\ZFC \vdash \forall x
\exists! y \, \de_{\len}(x, y)$. We could therefore define the function symbol
$F$ by $y \beq Fx \tot \de_{\len}(x, y)$ and let $\len(x)$ be shorthand for the
term $Fx$. We do not actually add $F$ to our extralogical signature, and instead
leave both $F$ and $\len(x)$ as shorthand. In this way, $\len(x)$ is a function
informally described by
\[
  \len(x) = \begin{cases}
    \text{the length of $x$} &\text{if $x$ is a line segment},\\
    \text{the length of a side of $x$}
      &\text{if $x$ is an equilateral triangle},\\
    \emp &\text{otherwise}.
  \end{cases}
\]

\subsubsection{A first pass at setting up the problem}

Let $C' = \{D, \tau, \ell\}$ and $L' = L_\ZFC C'$. Define the deductive theory
$T_0' \subseteq (\cL')^0$ by
\[
  T_0' = \ZFC + \ph_\bD(D) + \ph_\ins(\tau, D) + \ph_\ch(\ell, D).
\]
Then $T_0'$ includes all facts in $\ZFC$ together with the following three
assumptions: $D$ is a disk in the plane, $\tau$ is an equilateral triangle
inscribed in $D$, and $\ell$ is a chord of $D$.

Define the inductive condition $\cC'$ as the set of inductive theories $P
\subseteq (\cL')^\IS$ with root $T_0'$ such that
\begin{enumerate}[(i)]
  \item $P(\len(\ell) > \len(\tau) \mid T_0')$ exists, and
  \item $P$ satisfies the principle of indifference.
\end{enumerate}
Let $P \subseteq (\cL')^\IS$ be any inductive theory with root $T_0'$. Suppose
$P(\ph \mid X) = p$ and $X^\pi \in \ante P$. As in Remark \ref{R:stick}, since
we are only concerned with the relative positions of $D$, $\tau$, and $\ell$, we
may assume that $P(D \beq \ul B \mid T_0') = 0$ for all $B \in \cB(\bR^2)$. We
may also make similar assumptions for $\tau$ and $\ell$. It then follows as in
the proof of Lemma \ref{L:circle} that $\pi$ must be the identity permutation,
so that $P(\ph^\pi \mid X^\pi) = p$. In other words, every inductive theory in
$(\cL')^\IS$ satisfies the principle indifference. This means that the principle
of indifference has nothing to say in this setting. It offers no restrictions,
so that $P(\len(\ell) > \len(\tau) \mid T_0')$ could be anything we like.

This, however, is misleading. We have omitted a critical assumption. Namely, we
failed to interpret the fact that the chord is ``chosen at random.'' We will
interpret this additional assumption as simply saying that the location of the
chord has a probability distribution. We leave the exact nature of this
distribution unspecified. To formulate this additional assumption, we must
expand our extralogical signature and our root.

\subsubsection{The complete setup and conclusion}

Let $\cE = \{B \in \cB(\bD) \mid \emp \subset B \subset \bD\}$ and let
\[
  C = C' \cup \{B_* \mid B \in \cE\}
  = \{D, \tau, \ell\} \cup \{B_* \mid B \in \cE\}.
\]
Let $L = L_\ZFC C$. For each $B \in \cE$, define
\begin{multline*}
  \ph_B(x) : (\exists uw \bin \ul \bR) (\exists v \bin \ul \bR^2)\\
  (
    \ph_\bD^\param(D, u, v, w)
    \wedge \{(1, 0)\}_* \beq f_{uvw}^\img(\ul{\{(1, 0)\}})
    \wedge B_* \beq f_{uvw}^\img(\ul B)
  )
\end{multline*}
Then let
\begin{align*}
  T_0 &= T_0' + \{\ph_B(B_*) \mid B \in \cE\}\\
  &= {
    \ZFC + \ph_\bD(D) + \ph_\ins(\tau, D) + \ph_\ch(\ell, D)
    + \{\ph_B(B_*) \mid B \in \cE\}
  }.
\end{align*}
Our root, $T_0$, says that all the facts in $\ZFC$ hold. It also says that $D$
is a disk in the plane, $\tau$ is an equilateral triangle inscribed in $D$, and
$\ell$ is a chord of $D$. Regarding the subsets of $D$, it says that $\{(1,
0)\}_*$ is a singleton set on the boundary of $D$ that serves as a fixed point
of reference. The sets $B_*$ are then Borel subsets of $D$ that are
geometrically related to $\{(1, 0)\}_*$ in the same manner as $B$ is related to
$\{(1, 0)\}$.

Define
\[
  \de_\midp(x, y) : {
    \neg \ph_\seg(x) \wedge y \beq \ul \emp \vee \ph_\seg(x) \wedge \ze(x, y)
  }
\]
Here, $\ze(x, y)$ is a formula which says that $y$ is the midpoint of the line
segment $x$. Then $T_0 \vdash \exists! y \, \de_\midp(\ell, y)$. We could
therefore define the constant symbol $c$ by $y \beq c \tot \de_\midp(\ell, y)$.
We do not actually add $c$ to our extralogical signature, but instead leave it
as shorthand. With this construction, $c$ denotes the midpoint of the segment
$\ell$. We can therefore talk about the location of $\ell$ using the sentence
$c \bin B_*$, where $B \in \cE$.

Let
\[
  B^{1/2} = \{(x, y) \in \bR^2 \mid \sqrt{x^2 + y^2} < 1/2\} \in \cE.
\]
Then $T_0 \vdash \len(\ell) > \len(\tau) \tot c \bin B_*^{1/2}$. Hence, for any
inductive theory $P \subseteq \cL^\IS$ with root $T_0$, we have
\begin{equation}\label{len-midp}
  P(\len(\ell) > \len(\tau) \mid T_0) = P(c \bin (B^{1/2})_* \mid T_0),
\end{equation}
by Proposition \ref{P:log-equiv-gen}.

Define the inductive condition $\cC$ as the set of inductive theories $P
\subseteq \cL^\IS$ with root $T_0$ such that
\begin{enumerate}[(i)]
  \item $P(c \bin B_* \mid T_0)$ exists for all $B \in \cE$, and
  \item $P$ satisfies the principle of indifference.
\end{enumerate}
Our main result now follows. For this, recall the notation $e(\th, r)$ from
Section \ref{S:pt-on-circ}.

\begin{thm}\label{T:Bertrand}
  The inductive condition $\cC$ is consistent and every $P \in \cC$ satisfies
  \begin{equation}\label{Bertrand}
    P(c \bin B_* \mid T_0) = P(c \bin e(\th, B)_* \mid T_0)
  \end{equation}
  for all $B \in \cE$ and all $\th \in \bR$. Moreover, for every $p \in [0, 1]$,
  there exists $P \in \cC$ such that
  \[
    P(\len(\ell) > \len(\tau) \mid T_0) = p.
  \]
\end{thm}

\begin{proof}
  We begin by proving consistency. Let $\bT \subseteq \bR^2$ be the equilateral
  triangle inscribed in $\bD$ with one corner situated at $(1, 0)$. Given $r \in
  \bD$ with $|r| < 1$, let $\bL_r$ be the chord of $\bD$ whose midpoint is $r$.

  Let $\nu_0$ be a probability measure on $(\bD, \cB(\bD))$ such that $\nu_0$ is
  continuous. That is, $\opnu_0 \{r\} = 0$ for all $r \in \bD$. Let $S = \bD
  \times \bD \times [0, 2 \pi)$, $\Ga = \cB(S)$, and $\nu = \nu_0 \times \nu_0
  \times m_0$, where $m_0$ is the uniform measure on $[0, 2 \pi)$.

  Let $\om_0$ be an $L_\ZFC$-structure such that $\om_0 \tDash \ZFC$. If $x =
  (r, t, \th) \in S$, then let $\om = \om^x$ be the $L$-expansion of $\om_0$
  defined by $D^\om = \ul{t + \bD}^{\om_0}$, $\tau^\om = \ul{t + \bT}^{\om_0}$,
  $\ell^\om = \ul{t + \bL_r}^{\om_0}$, and $B_*^\om = \ul{t + e(\th,
  B)}^{\om_0}$.

  Let $\Om = \{\om^x \mid x \in S\}$, let $h: S \to \Om$ denote the function $x
  \mapsto \om^x$, and let $\sP = (\Om, \Si, \bbP)$ be the measure space image of
  $(S, \Ga, \nu)$ under $h$. By a proof similar to that of Lemma
  \ref{L:stick-1}, we have $\sP \vDash T_0$. We may therefore define the
  inductive theory $P = \bTh \sP \dhl_{[T_0, \Th \sP]}$.

  For each $x = (r, t, \th) \in S$ and $\om = \om^x$, it follows that $\om
  \tDash c \bin B_*$ if and only if $r \in e(\th, B)$. Thus, $h^{-1} (c \bin
  B_*)_\Om = \{(r, t, \th) \mid r \in e(\th, B)\}$. Since $e$ is measurable,
  $P(c \in B_* \mid T_0) = \bbP (c \bin B_*)_\Om = \opnu h^{-1} (c \bin
  B_*)_\Om$ exists, so that $P$ satisfies (i) in the definition of $\cC$.

  Suppose $P(\ph \mid X) = p$ and $X^\pi \in \ante P$. Using methods like those
  in the proofs of Lemmas \ref{L:stick-2} and \ref{L:stick-3}, it follows that
  $\s^\pi = \s$ for all $\s \in L_\ZFC \cup \{D, \tau, \ell\}$, and there
  exists $\th_0 \in [0, 2\pi)$ such that $B_*^\pi = e(\th_0, B)_*$ for all $B
  \in \cE$. Define $g: S \to S$ by $g(r, t, \th) = (r, t, [\th_0 + \th])$. As is
  the proofs of Theorem \ref{T:stick} and Proposition \ref{P:circle}, we can use
  $g$ to show that $\sP^\pi = \sP$. It therefore follows that $P(\ph^\pi \mid
  X^\pi) = p$, so that $P$ satisfies the principle of indifference. Hence, $P
  \in \cC$ and $\cC$ is consistent.

  Now let $P \in \cC$ be given. Fix $\th_0 \in \bR$ and $B \in \cE$. Let $\pi$
  be the $L$-permutation such that $\s^\pi = \s$ for all $\s \in L_\ZFC \cup
  \{D, \tau, \ell\}$, and $B_*^\pi = e(\th_0, B)_*$ for all $B \in \cE$. Then
  $T_0^\pi = T_0$, so \eqref{Bertrand} follows immediately from the principle of
  indifference.

  Finally, let $p \in [0, 1]$. By \eqref{len-midp}, it suffices to show that
  there exists $P \in \cC$ such that $P(c \bin (B^{1/2})_* \mid T_0) = p$. Let
  $\nu_0$ be a continuous probability measure on $(\bD, \cB(\bD))$ such that
  $\nu_0(B^{1/2}) = p$. Construct $P$ as in the first part of this proof. Since
  $r \in e(\th, B)$ if and only if $e(-\th, r) \in B$, it follows that
  \begin{align*}
    P(c \bin (B^{1/2})_* \mid T_0) &= \opnu h^{-1} (c \bin (B^{1/2})_*)_\Om\\
    &= \int_\bD \int_0^{2\pi} {
      1_{B^{1/2}} (e(-\th, r)) \, m_0(d\th) \, \nu_0(dr)
    }.
  \end{align*}
  But $1_{B^{1/2}} (e(-\th, r)) = 1_{B^{1/2}} (r)$, so $P(c \bin (B^{1/2})_*
  \mid T_0) = \nu_0(B^{1/2}) = p$.
\end{proof}